\documentclass[11pt]{amsart}
\usepackage{geometry}
\usepackage{setspace}
\usepackage{enumerate}

\usepackage{amsmath,amssymb,graphicx,amscd}
\usepackage{stmaryrd} 
\usepackage{mathtools}
\usepackage{color}
\usepackage{enumitem}
\usepackage{xcolor}
\usepackage{float}
\usepackage[numbers]{natbib}
\usepackage[all]{xy}
\usepackage{amsthm}
\usepackage{amsfonts}
\usepackage{tikz}
\usepackage{tikz-cd}
\usepackage{mathtools}
\usepackage{fancyhdr}
\usepackage{hyperref}
\hypersetup{colorlinks=true,citecolor=brown,linkcolor=cyan!100,urlcolor=black,filecolor=black}
\numberwithin{figure}{section}
\numberwithin{equation}{section}
\allowdisplaybreaks
\numberwithin{equation}{section}
\usepackage{imakeidx}
\makeindex[name=notation,title={Index of notations}]

\makeindex

\newtheorem{theorem}{Theorem}[section]
\newtheorem{proposition}[theorem]{Proposition}
\newtheorem{corollary}[theorem]{Corollary}
\newtheorem{lemma}[theorem]{Lemma}

\theoremstyle{definition}
\newtheorem{remark}[theorem]{Remark}
\newtheorem{example}[theorem]{Example}
\newtheorem{definition}[theorem]{Definition}
\newtheorem{convention}[theorem]{Convention}
\newtheorem{notation}[theorem]{Notation}

\newcommand{\hooklongrightarrow}{\lhook\joinrel\longrightarrow}

\definecolor{ASTRAL}{RGB}{46,116,181}
\newcommand{\overbar}[1]{\mkern 1.5mu\overline{\mkern-1.5mu#1\mkern-1.5mu}\mkern 1.5mu}

\newcommand{\myencircle}[1]{\tikz \node[anchor=south west, draw,circle, inner sep=-0.5, minimum size=2.5mm]{#1};}

\author{Benjamin Enriquez}
\address{IRMA (UMR 7501) ET DÉPARTEMENT DE MATHÉMATIQUES, UNIVERSITÉ DE STRASBOURG, 7 RUE RENÉ DESCARTES, 67084 STRASBOURG (FRANCE)}
\email{b.enriquez@math.unistra.fr}

\author{Anderson Vera}
\address{Center for Geometry and Physics, Institute for Basic Science (POSTECH Campus)
79, Jigok-ro 127 beon-gil, Nam-gu, Pohang-si, Gyeongsangbuk-do, Korea 37673}
\email{andersonvera@ibs.re.kr}

\date{\today}
\title{A survey on the Le-Murakami-Ohtsuki invariant for closed 3-manifolds}

\begin{document}

\begin{abstract} We review the original approach to the \emph{Le-Murakami-Ohtsuki (LMO) invariant} of closed 3-manifolds (as opposed to the later approach based on the Aarhus integral). Following the ideas of surgery presentation, we introduce a class of combinatorial structures, called  \emph{Kirby structures}, which we prove to yield multiplicative 3-manifold invariants. We illustrate this with the Reshetikhin-Turaev invariants. We then introduce a class of combinatorial structures, called \emph{pre-LMO structures}, and prove that they give rise to Kirby structures. We show how the Kontsevich integral can be used to construct a pre-LMO structure. This yields two families of multiplicative 3-manifolds invariants $\{\Omega_n^{\mathfrak{c}}\}_{n\geq 1}$ and $\{\Omega_n^{\mathfrak{d}}\}_{n\geq 1}$. We review the elimination of redundant information in the latter family, leading to the construction of the LMO invariant. We also provide uniqueness results of some aspects of the LMO construction. The family of invariants $\{\Omega_n^{\mathfrak{c}}\}_{n\geq 1}$ is not discussed explicitly in the literature; whereas $\Omega_n^{\mathfrak{c}}$ enables one to recover $\Omega_n^{\mathfrak{d}}$ for any $n \geq 1$, we show that these invariants coincide for $n = 1$.
\end{abstract}

\maketitle

\tableofcontents

\section*{Introduction}

A \emph{knot} (resp. \emph{link}) is the isotopy class of an embedding of the circle $S^1$ (resp. several copies of $S^1$) into the 3-dimensional sphere  $S^3$. The study of these objects is known as \emph{knot theory}. This theory has relations with many other areas, in particular, it is closely related with the study of the topology of  3-manifolds. (From now on, we use the word 3-manifold to mean compact, connected oriented 3-manifold.) One instance of these relations is the following: starting with a link, we can remove a neighborhood (solid torus) of each  of its components  and paste them again using some parameters (\emph{framing}). This gives rise to a new 3-manifold, the framed link used in the process is known as a \emph{surgery presentation} of the 3-manifold. Some remarkable facts are that every 3-manifold admits a surgery presentation (Lickorish-Wallace Theorem) and that the links underlying  two surgery presentations of the same 3-manifold can be related by a finite sequence of standard modifications called  \emph{Kirby moves} (Kirby Theorem), see \cite{Lick,Wal,Kir,Lu}.

Until the 80s, the main approach to construct invariants for knots and 3-manifolds was based on tools of algebraic topology (e.g. homotopy, homology, covering spaces, etc).  The close relationship between knots and 3-manifolds was not really used to obtain invariants of the latter using invariants of the former. A new perspective arose in 1984 with the introduction of a new knot invariant: the \emph{Jones polynomial}  \cite{Jon}. This  was followed by  the discovery of a great variety of new knot invariants following Jones' ideas and interpretations of it \cite{FYHLMO, PT, Kau1, Tur1988, Tur1989, RT1}, notably HOMFLY-PT and Kauffman polynomials, which established connections with other areas of mathematics: quantum groups, Lie algebras, mathematical physics, etc.  This was the starting point of the so-called \emph{quantum invariants} of knots. 
 Around the same time,  Witten  \cite{Wit} gave an interpretation of the Jones polynomial using physics  ideas which allowed him to  suggest  the existence of  new invariants of 3-manifolds.

One of the possible mathematical constructions of the invariants proposed by  Witten was first obtained by  N. Reshetikhin and V. Turaev using representations of quantum groups \cite{RT}. They  showed that the above-mentioned surgery description of 3-manifolds, based in the Lickorish-Wallace and Kirby theorems, enables one to build up 3-manifold invariants out of certain quantum knot invariants (see the surveys \cite{LickBook,PS,Tur}). These constructions were generalized to categorical settings in a series of works, see for instance \cite{Tur}.

Around the 90s, Vassiliev \cite{Vas} proposed another approach for  studying knots, based on the study of the topology of the space of all knots and  leading to the introduction of an algorithm for constructing knot invariants.  In  \cite{BirLin}, it was shown that the invariants obtained in  \cite{Vas} satisfy the vanishing of certain alternating sums, and it was proved that the set of invariants satisfying this vanishing condition contains the information of the Jones, HOMFLY-PT and Kauffman polynomials. In the independent work \cite{Gou1, Gou2}, the same set of invariants  was introduced, and its relation with the HOMFLY-PT polynomial  was also obtained. The invariants of this set are called \emph{finite-type invariants} in \cite{Gou1}, and commonly referred to as \emph{Vassiliev} or \emph{Vassiliev-Goussarov} invariants.

In  \cite{K} (see also \cite{BN2} where details are worked out), an explicit formula for a knot invariant is given, which is called the \emph{Kontsevich integral}; this invariant takes its values in a degree completed vector space $\mathcal{A}^{\mathrm{chords}}= \prod_{n\geq 0}\mathcal{A}^{\mathrm{chords}}_n$ of \emph{chord diagrams} on a circle. The Kontsevich integral is a \emph{universal Vassiliev invariant}, that is,  for any $n \geq 0$ and any $\mathbb{Q}$-valued Vassiliev invariant $V$ of degree $n$ there exists a linear map $\varphi_V: \mathcal{A}^{\mathrm{chords}}_{\leq n}\to \mathbb{Q}$ such that $V$ is the  composition of the Kontsevich integral with the map $\mathcal{A}^{\mathrm{chords}}\xrightarrow{\mathrm{proj}}\mathcal{A}^{\mathrm{chords}}_{\leq n}\xrightarrow{\varphi_V}\mathbb{Q}$.  The Kontsevich integral is also a \emph{universal quantum invariant}, as any such invariant can be recovered by the choice of a \emph{weight system} (constructed from the same data used to construct the quantum invariant: a semi-simple Lie algebra and one of its finite dimensional representations).

As mentioned above, the Reshetikhin-Turaev 3-manifold invariants are constructed using surgery presentations and quantum invariants of knots. 
Then the question arose of the similar  construction of 3-manifold invariants, replacing the quantum invariants of knots by the Kontsevich integral.  
This question was solved in a series of works. A key step toward a solution is the understanding of the behavior of the Kontsevich integral under Kirby moves;  this was accomplished in   \cite{LeMu97}, where Le and Murakami introduced a normalization $\check{Z}$ of the Kontsevich integral and described its behavior under the second Kirby move.  In \cite{LMMO95, LMMO99}, a quotient $\mathcal A^{\mathrm{chords}}/(P^2+\mathrm{CO}\mathcal A+L^{<2})$ of $\mathcal A^{\mathrm{chords}}$ was constructed\footnote{In  \cite{LMMO95, LMMO99}, the relations $P^2$ and $\mathrm{CO}\mathcal A$ are called $3\mathrm{T}$-relation and $\mathrm{OI}$-relation, respectively. The relation  $L^{<2}$ consists of the $0$-vanishing and $1$-vanishing relations in  \cite{LMMO95, LMMO99}.}, and the corresponding reduction of $\check{Z}$ was shown to be invariant under the second Kirby move. Then  a normalization of this reduction, using the signature of the linking matrix,  gives rise  to a 3-manifold invariant valued in $\mathcal A^{\mathrm{chords}}/P^2+\mathrm{CO}\mathcal A+L^{<2}$.

 In \cite{LMO98}, this construction was extended to more general quotients  ${\mathcal A}^{\mathrm{chords}}/P^{n+1}+\mathrm{CO}\mathcal A+L^{<2n}+(\mathrm{deg}>n)$  giving rise, for each positive integer $n$, to a 3-manifold invariant  taking values in a quotient of a space of \emph{Jacobi diagrams} (trivalent graphs with additional structure up to some local relations), usually denoted by $\mathcal A(\emptyset)/(\mathrm{deg}>n)$. These invariants were then reassembled into a single invariant, the \emph{Le-Murakami-Ohtsuki} invariant, taking their values in the space $\mathcal A(\emptyset)$. These works are surveyed in  \cite{Ohts,Le:notes}.

To sum up, the construction in \cite{LMO98}  of the LMO invariant proceeds according to the following steps, where we use the notation of the present work: 
\begin{enumerate}
\item Construction of a family of invariants $\{\Omega^{\mathfrak{e}}_n\}_{n\geq 1}$
\begin{enumerate}
\item description of the fibers of the surgery map $\{$framed oriented links in $S^3\}\to\{$3-manifolds$\}$  in terms of the Kirby moves KI and KII; 
\item study of the behavior of a normalized version of the Kontsevich invariant of links with respect to the second Kirby move;
\item using (b) and a normalization utilizing the signature of the linking matrix to construct a map $\Omega^{\mathfrak{e}}_n:
\{$framed oriented  links in $S^3\}\to\mathcal{A}_0/(\mathrm{deg}>n)$ for each integer $n>0$, invariant under the Kirby moves, and which is therefore an invariant of 3-manifolds;
\end{enumerate}
\item Elimination of redundant information
\begin{enumerate}
\item establishing comparison properties of the various $\Omega^{\mathfrak{e}}_n$ for $n>0$, based on coproduct operations of Jacobi diagrams; 
\item constructing the LMO invariant $Z^{\mathrm{LMO}}$ by putting together the  information provided by each $\Omega^{\mathfrak{e}}_n$, 
$n>0$ which is new with respect to the $\Omega^{\mathfrak{e}}_k$ with $k<n$.   
\end{enumerate}
\end{enumerate}

There is an  alternative approach to this construction, based on a diagrammatic version of the Gaussian integral (called \emph{Aarhus integral}), proposed by the authors of 
\cite{BNGRTI, BNGRT, BNGRTIII}.  One of the main differences with the original approach is that the Aarhus integral gives the LMO  in one fell swoop, that is, without going through the construction of an invariant at each degree. This approach only works for rational homology spheres, which is not a weakness since the LMO invariant is particularly interesting for those types of 3-manifolds, as it is known that the restriction of the LMO invariant to non-$\mathbb{Q}$-homology spheres can be expressed using classical invariants.

A recurring idea in quantum topology is the identification of the set of homeomorphism classes of closed 3-manifolds with the set of endomorphisms of the empty surface in a category of cobordisms between surfaces. An extension of a given invariant of closed  3-manifolds to a functor with source a category of cobordisms, which recovers it by restriction to the endomorphism set of the empty surface, is called a TQFT (topological quantum field theory) extending this invariant. TQFT extensions of the Reshetikhin-Turaev invariants were constructed in \cite{RT2,TurBook}. Similarly, a kind of TQFT extension of the LMO invariant was first obtained in  \cite{MuOht} using the original LMO techniques. In \cite{CHM} a TQFT extension of the LMO invariant was constructed  using the Aarhus integral, recovering the previous extension from~\cite{ChLe}.

The purpose of the present paper is to review the original definition of the LMO invariant. We clarify some of the combinatorial backgrounds of the construction of 3-manifolds invariants using surgery presentation, mainly of the  LMO construction, either by relating it with known constructions (Brauer algebras) or  introducing appropriate notions (Kirby monoid/structures). We also give a unified and detailed account of some parts of the construction, filling in some gaps. We also provide uniqueness results of some aspects of the LMO construction (see \S\ref{sec:3:7}).  Let us describe in detail the contents of each section. 

In \S\ref{sec::1}, we make explicit the combinatorial structure underlying  the theory of tangles, which we show to be naturally related to that of Brauer category \cite{Leh}. In \S\ref{sec::2}, we formalize the application of the Lickorish-Wallace and Kirby (LW-K) theorems to the construction of 3-manifold invariants, via the introduction of the notion of \emph{Kirby structure}; these invariants are multiplicative with respect to the connected sum of 3-manifolds. This is illustrated with the Kirby structures giving rise to Reshetikhin-Turaev invariants. We then introduce the notion of \emph{pre-LMO structure}, which naturally yields a Kirby structure. In \S\ref{sec:2}, following literature, we show how the Kontsevich integral and its compatibilities with certain operations (change of orientation, cabling) give rise to a pre-LMO structure.   In \S\ref{ch4}, following the ideas of the source papers, we consider a sequence of maps between different quotient spaces of Jacobi diagrams,  which yields two families of multiplicative 3-manifold invariants $\{\Omega_n^{\mathfrak{c}}\}_{n\geq 1}$ and $\{\Omega_n^{\mathfrak{d}}\}_{n\geq 1}$, see Theorem~\ref{r2022-04-12theinvariants}. The invariance property of  $\Omega_n^{\mathfrak{c}}$ is new and implies   the invariance property for  $\Omega_n^{\mathfrak{d}}$.

In \S\ref{sec::5} we construct relations between the various $\Omega_n^{\mathfrak{c}}$  as well as between the  various $\Omega_n^{\mathfrak{d}}$. We show that the space $\mathfrak{d}(n)$ where $\Omega_n^{\mathfrak{d}}$ takes its values is isomorphic to $\mathfrak{e}(n):=\mathcal{A}_0/(\mathrm{deg}>n)$, we denote by 
$\Omega_n^{\mathfrak{e}}$ the image of  $\Omega_n^{\mathfrak{d}}$  by this isomorphism. The resulting relations between the  $\{\Omega_n^{\mathfrak{e}}\}_{n\geq 1}$  account for overlap of the information contained in the various $\Omega_n^{\mathfrak{e}}$. The LMO invariant is constructed out of the family $\{\Omega_n^{\mathfrak{e}}\}_{n\geq 1}$  by eliminating the redundant information.

For more details, we refer to the local introductions of each section.

\subsection*{Acknowledgements} The research of B.E. has been partially funded by ANR grant “Project
HighAGT ANR20-CE40-0016”. The research of A.V.  is supported by the grant IBS-R003-Y002. We thank sincerely   Jun Murakami and  Yuta Nozaki and for helpful and stimulating discussions.

\subsection*{Notations} The following is a list of common notations used throughout this work. Additionally, we refer readers to the index of notations at the end of the paper.

\begin{itemize}

\item For an integer $l\geq$ we denote by $[\![1,l]\!]$ the set of integers $k$ such that $1\leq k \leq l$. In particular , $[\![1,0]\!]=\emptyset$.

\item For a finite set $X$ we denote by $|X|$ its cardinality; by $\mathfrak{S}_X$ the symmetric group on $X$ and  by $\mathrm{FPFI}(X)$ the set of fixed-point free involutions of $X$.

\item  For $V$ a $\mathbb{C}$-vector space and a subset $S\subset V$, we denote by $\mathrm{Vec}_{\mathbb{C}}S$ the subspace of $V$ generated by $S$, i.e., the smallest subspace of $V$ containing $S$. If $V$ is  complete graded and the subset $S$ consists of homogeneous elements, then $\widehat{\mathrm{Vec}}_{\mathbb{C}}S$ denotes the complete graded subspace of $V$ generated by~$S$.

\item For a set $X$, we denote by $\mathcal P(X)$ its power set, i.e., the set of all subsets of $X$. 
 
\item For $X$ a set, we denote by $\mathbb{C}X$ the free $\mathbb{C}$-vector space on $X$.

\item $\emptyset$ empty word. 

\item $\varnothing$ , $\vec{\varnothing}$ empty Brauer diagram, empty tangle, empty link.

\item Let $\mathcal Set_f$ (resp. $\mathcal Set_f^*$) be the category whose objects are (resp. pointed) finite sets and set of morphisms given by (resp. pointed) bijections.

\item Let $\mathcal Vect$ denote the category of $\mathbb{C}$-vector spaces.
\end{itemize}

\section{Partitions and (oriented) Brauer diagrams}\label{sec::1}

Consider a \emph{tangle} in a cube $[-1,1]^3$, that is, an isotopy class of a proper embedding $T:P\to[-1,1]^3$ of a compact 1-manifold $P$ such that $i(\partial P)$ is uniformly distributed in $\{0\}\times (-1,1)\times \{\pm 1\}$. The abstract underlying compact 1-manifold $P$ is sometimes referred as \emph{skeleton} or \emph{pattern} of the tangle (it does not depend on the isotopy class) and will be called here its \emph{skeletization}. One can define operations of  composition and of tensor product on tangles. Similar operations can be defined on skeleta, and they are compatible with skeletization. The \emph{Kontsevich integral} is an invariant of parenthesized tangles which associates to  such a tangle an element of a complete vector space constructed by glueing Jacobi diagrams to its skeletization, and is  compatible with the tangle operations of composition and tensor product. One of its main features is that it is \emph{universal} for both quantum and Vassiliev invariants, in the sense that any quantum invariant (an invariant associated to a semi-simple Lie algebra and one of its finite dimension representations) or any rational-valued Vassiliev invariant, can be recovered from the Kontsevich integral \cite{K,BN,LM1,AltD,Ca,CDM,Ohts,JacksonMoffat}.

A skeleton $P$ can be decomposed as the disjoint union of ``circle" and ``segment" parts: $P=P_{\mathrm{cir}} \sqcup P_{\mathrm{seg}}$, where $P_{\mathrm{cir}} $ consists of the connected components of $P$ with empty boundary (homeomorphic to circles) and $P_{\mathrm{seg}}$ consists of those connected components with non-empty boundary (homeomorphic to closed intervals). The operations of composition and tensor product on skeleta can be expressed as follows 
$$(P \circ Q)_{\mathrm{seg}}=P_{\mathrm{seg}} \circ Q_{\mathrm{seg}}, \qquad   (P \circ Q)_{\mathrm{cir}}=P_{\mathrm{cir}} \sqcup Q_{\mathrm{cir}} \sqcup \mathrm{Cir}(P_{\mathrm{seg}},Q_{\mathrm{seg}}),$$
$$(P \otimes Q)_{\mathrm{seg}}=P_{\mathrm{seg}} \otimes Q_{\mathrm{seg}}, \qquad   (P \otimes Q)_{\mathrm{cir}}=P_{\mathrm{seg}} \sqcup Q_{\mathrm{cir}},$$ where in the right-hand sides, $\circ$ and $\otimes$ are operations on skeleta with empty circle parts, and $\mathrm{Cir}(P_{\mathrm{seg}},Q_{\mathrm{seg}})$ is a finite set of circle components attached to a pair $(P_{\mathrm{seg}},Q_{\mathrm{seg}})$ of skeleta with empty circle parts.

For each $\delta$ in a unital commutative ring $R$, there is a \emph{Brauer category} $\mathcal B(\delta)$, see  \cite{Leh}. Its objects are non-negative integers, and for non-negative integers $n$ and $m$, the set of morphisms from $n$ to $m$ is the $R$-module freely generated by the partitions of $[\![1,n]\!] \sqcup [\![1,m]\!]$ consisting of subsets of cardinality $2$. The operations $\circ$, $\otimes$ and $\mathrm{Cir}(-,-)$ on skeleta with empty circle parts are closely related with $\mathcal B(\delta)$. Namely, the set of morphisms in $\mathcal B(\delta)$ can be identified with the $R$-module freely generated by the set of skeleta with empty circle 
parts, the tensor product in $\mathcal B(\delta)$ then coincides with $\otimes$, and for $P$ and $Q$ skeleta with empty circle parts, the composition $\circ_{\delta}$ in $\mathcal B(\delta)$ is given by 
$P \circ_{\delta} Q:= \delta^{|{\mathrm{Cir}}(P,Q)|}P \circ Q$. In particular, 
$P \circ Q=P \circ_1 Q$, so that the category of skeleta with empty circle parts equipped with $(\circ, \otimes)$ coincides with the category $\mathcal B(1)$, which will be denoted $\underline{\mathcal Br}$ in the present paper.   This relation with $\mathcal B(1)$ motivates   our treatment of the notion of skeleton, in which the segment and circle parts are given different statuses: the segments of the skeleta will be identified with morphisms in  $\underline{\mathcal{B}r}$, while the circle parts will just be viewed as finite sets.  

The section is organized as follows. In  \S\ref{sec:1-1} we introduce  the category  of partitions $\underline{\mathcal{P}art}$ and in \S\ref{sec:1-2} we stablish a classification of finite  $\mathbb{Z}_2*\mathbb{Z}_2$-sets.  The definition of the category $\underline{\mathcal Br}$ of Brauer diagrams and of the operation $\mathrm{Cir} : \underline{\mathcal{B}r}\times \underline{\mathcal{B}r} \to  \mathcal{S}et_f$ is introduced in  \S\ref{sec:1-3};  based on the results of  \S\ref{sec:1-1} and \S\ref{sec:1-2}, we give a proof of the associativity of the composition of $\underline{\mathcal Br}$, as well as of cocycle properties of $\mathrm{Cir}$, which are generally understood in the literature.  In \S\ref{sec:1-4}, we study various operations on these objects (tensor product, doubling). We introduce the notion of  oriented Brauer diagrams in \S\ref{sec:1-5} and study its operations and properties in  \S\ref{sec:1-6} and \S\ref{sec:1-7}. Finally, in \S\ref{sec:1-8}, we explain how the abstract notions introduced in the section can be encoded in the familiar diagrams.

\subsection{A category  \texorpdfstring{$\underline{\mathcal{P}art}$}{Part}  of partitions}\label{sec:1-1}

For a set $X$ we denote by $\mathcal P(X)$\index[notation]{P(X)@$\mathcal P(X)$} its power set, i.e., the set of all subsets of $X$. 

\begin{definition} Let $X$ be a set, the set $\mathrm{Part}(X)$\index[notation]{Part(X)@$\mathrm{Part}(X)$} of  partitions of $X$ is the subset of $\mathcal P(X)$ consisting of all $A$ such that
\begin{itemize}
\item[$(a)$] for all $\xi \in A$, $\xi \neq \emptyset$ 
\item[$(b)$] for all $\xi \neq \eta \in A$, $\xi \cap \eta=\emptyset$
\item[$(c)$] $X=\bigcup_{\xi \in A}\xi$.
\end{itemize}
The set $\mathrm{Part}(X)$ is endowed with a lattice structure with partial order $\leq$  given by: for $A, B\in \mathrm{Part}(X)$, we say $A \leq B$ if  for all $\omega\in A$ there exists $\xi\in B$ such that $\omega \subset \xi$.

For $A, B\in \mathrm{Part}(X)$, the \emph{join} $A\vee B$ of $A$ and $B$ is defined as $\mathrm{sup}\{A,B\}$, i.e., the least partition of $X$ which is greater or equal than $A$ and $B$. 
\end{definition}

The join operation is associative, i.e., $(A\vee B)\vee C = A\vee (B\vee C)$. 

\begin{definition} Let $X$ and $Y$ be sets and $f:X\to Y$ be a map. 
\begin{itemize}
\item[$(a)$] Define a map $f_*:\mathrm{Part}(X)\to \mathrm{Part}(Y)$ as follows: for  $A\in\mathrm{Part}(X)$,  $f_*A$ is the least partition of $Y$ such that for all $\omega\in A$ there exists $\xi\in f_*A$ with $f(\omega)\subset \xi$.  Explicitly, $y\equiv y'$ mod $f_*A$ iff $y=y'$ or there exists 
$x_1,\ldots,x_{2n+1}$ in $X$ such that $y=f(x_1)$, $y'=f(x_{2n+1})$ and 
$x_{2i+1}\equiv x_{2i+2}$ (mod $A$) for $i=0,\ldots,n-1$ and $f(x_{2i})=f(x_{2i+1})$
for $i=1,\ldots,n$. 

Notice that $ \{ \{y\} \ | \ y\in Y \setminus  f(X)\} \subset f_*A$ for any $A\in \mathrm{Part}(X)$.

When $f$ is injective, we have $$f_*A=\bigsqcup_{\omega\in A}\{f(\omega)\}\sqcup \bigsqcup_{y\in Y\setminus f(X)}\{\{y\}\},$$
explicitly,  $y\equiv y'$ mod $f_*A$ iff $y=y'$ or there exist 
$x,x'\in X$ such that $x\equiv x'$ (mod~$A$) and $y=fx$, $y'=fx'$. 
\item[$(b)$] Define a map $f^*:\mathrm{Part}(Y)\to \mathrm{Part}(X)$ as follows: for  $B\in\mathrm{Part}(Y)$,  $f^*B$ is the least partition of $X$ such that for all $\xi\in B$ there exists $\omega\in f^*B$ with $f^{-1}(\xi)\subset \omega$.  Explicitly, $x\equiv x'$ mod $f^*B$ iff $fx\equiv fx'$ mod $B$. 
\end{itemize}
\end{definition}

If $f$ is injective, then $f^*f_*A=A$. In general, we have $A\subset f^*f_*A$ and $f^*f_*A=A\vee P_f$, where $P_f\in\mathrm{Part}(X)$ is the partition induced by $f$. For sets $X$, $Y$ and $Z$ and maps $f:X\to Y$ and $g:Y\to Z$ one checks that $(g\circ f)^*=f^*\circ g^*$ and $(g\circ f)_ *=g_*\circ f_*$.

\begin{notation}\label{notation0} Let $S_1, S_2, \ldots S_n$ be sets. For $i_1,\ldots i_k\in [\![1,n]\!]$ with $i_1<\cdots <i_k$ we denote by $\mathrm{can}_{S_{i_1}\cdots S_{i_k}}^{S_1\cdots S_n}: S_{i_1}\sqcup \cdots \sqcup S_{i_k}\to S_1\sqcup \cdots \sqcup S_n$\index[notation]{can@$\mathrm{can}_{S_{i_1}\cdots S_{i_k}}^{S_1\cdots S_n}$} the canonical inclusion of $S_{i_1}\sqcup \cdots \sqcup S_{i_k}$ into $S_1\sqcup \cdots \sqcup S_n$.

\end{notation}

\begin{definition} Let $X$ be a set and $X'\subset X$. For $A\in\mathrm{Part}(X)$ define the set, denoted $A_{|X'}$ and called \emph{restriction of $A$ to $X'$}, by $$A_{|X'}:=\{a\in A \ | \  a\subset X'\}\subset A.$$
\end{definition}
Notice that in general $A_{|X'}$ is not a partition of $X'$.

\begin{definition}\label{def:2022-09-13-defPart} Let $X$, $Y$ and $Z$ be sets. \begin{itemize}
\item[$(a)$]  Define $\underline{\mathcal Part}(X,Y):=\mathrm{Part}(X \sqcup Y)$\index[notation]{Part(X,Y)@$\underline{\mathcal Part}(X,Y)$}.  

\item[$(b)$] Define a composition map
$$\underline{\mathcal Part}(X,Y)\times\underline{\mathcal Part}(Y,Z)\to\underline{\mathcal Part}(X,Z),$$
by
$$(A,B)\mapsto B \circ A:=(\mathrm{can}^{XYZ}_{XZ})^*\Big((\mathrm{can}^{XYZ}_{XY})_*(A)\vee(\mathrm{can}^{XYZ}_{YZ})_*(B)\Big).$$
\item[$(c)$] For $A\in\underline{\mathcal Part}(X,Y)$ and $B\in\underline{\mathcal Part}(Y,Z)$, define the set $\mathrm{Cir}(B,A)$\index[notation]{Cir(B,A)@$\mathrm{Cir}(B,A)$} by
$$\mathrm{Cir}(B,A):=\Big( (\mathrm{can}_{XY}^{XYZ})_*(A) \vee (\mathrm{can}_{YZ}^{XYZ})_*(B)\Big)_{|Y}.$$
\end{itemize}
\end{definition}

\begin{lemma}\label{r:r2-2022-09-01} Let $X$ and $Y$ be sets and $Y'\subset Y$. Then for any $U\in\mathrm{Part}(X\sqcup Y)$ there is a bijection $\big((\mathrm{can}_Y^{XY})^*(U)\big)_{|Y'}\simeq U_{|X\sqcup Y'}\setminus U_{|X}$.
\end{lemma}

\begin{proof}
By definition we have
\begin{equation}\label{equ:equ1-2022-08-29}
\big((\mathrm{can}^{XY}_Y)^*(U)\big)_{|Y'}=\{b\in\mathcal{P}(Y) \ | \ \emptyset\not= b \subset Y' \text{ and } \exists a\in U \text{ such that } b=a\cap Y \text{ and } b\subset Y'\}.
\end{equation}
Besides, 
\begin{equation*}
\begin{split}
U_{|X\sqcup Y'}&=\{a\in U \ | \ a\subset  X\sqcup Y'\}\\
 & = \{a\in U \ | \ a\subset X\} \sqcup \{a\in U  \ |\ a\subset X\sqcup Y' \text{ and } a\cap Y'\not= \emptyset\} \\
 & = (U_{|X}) \sqcup \{a\in U  \ |\ a\subset X\sqcup Y' \text{ and } a\cap Y'\not= \emptyset\}.
\end{split}
\end{equation*}
Therefore,
\begin{equation}\label{equ:equ2-2022-08-29}
U_{|X\sqcup Y'}\setminus  U_{|_X}= \{a\in U  \ |\ a\subset X\sqcup Y' \text{ and } a\cap Y'\not= \emptyset\}.
\end{equation}

Define $\varphi: (U_{|X\sqcup Y'}\setminus  U_{|_X})\rightarrow \big((\mathrm{can}^{XY}_Y)^*(U)\big)_{|Y'}$ by $\varphi(a):=a\cap Y$ for any $a\in U_{|X\sqcup Y'}\setminus  U_{|_X}$.

Let us show that $\varphi$ is well-defined. If $a\in (U_{|X\sqcup Y'}\setminus  U_{|_X})$, then $\varphi(a)=a \cap Y \in \mathcal P(Y)$. Moreover, $\varphi(a) \supset a \cap Y' \neq \emptyset$ by \eqref{equ:equ2-2022-08-29}. Therefore $\varphi(a) \neq \emptyset$. Also $\varphi(a)=a \cap Y \subset Y'$ since $a \subset X \sqcup Y'$ (see \eqref{equ:equ2-2022-08-29}). Finally, $\varphi(a)=a \cap Y$ with $a \in U$. All this implies that $\varphi(a) \in \big((\mathrm{can}^{XY}_Y)^*(U)\big)_{|Y'}$. Besides, if $b\in \big((\mathrm{can}^{XY}_Y)^*(U)\big)_{|Y'}$, then by \eqref{equ:equ1-2022-08-29}, there exists $a\in U$ with $b=a\cap Y$. Let us show that $a\in(U_{|X\sqcup Y'}\setminus  U_{|_X})$. By assumption on $b$, one has $b \subset Y'$, which together with $b=a \cap Y$ implies $a \subset X \sqcup Y'$. Moreover, the same facts imply $b=a \cap Y'$; since $b \neq \emptyset$ by~\eqref{equ:equ1-2022-08-29}, then  $a \cap Y' \neq \emptyset$. It follows that $a\in(U_{|X\sqcup Y'}\setminus  U_{|_X})$, therefore $b=a \cap Y=\varphi(a)$, i.e., $\varphi$ is surjective.   Finally, if $a,a'\in  U_{|X\sqcup Y'}\setminus  U_{|_X}$ are such that $\varphi(a)=\varphi(a')$, then $\emptyset\not=\varphi(a)=\varphi(a')\subset a\cap a'$. As $U$ is a partition, this implies $a=a'$, i.e., $\varphi$ is injective. 
\end{proof}

\begin{lemma}\label{equ:2022-08-22-equ2} Let $X$, $Y$, $Z$ and $T$ be sets. Then for any $U\in\mathrm{Part}(X\sqcup Y\sqcup Z)$ and any $V\in\mathrm{Part}(Z\sqcup T)$ we have
\begin{equation*}
(\mathrm{can}^{XZT}_{XZ})_*\Big((\mathrm{can}^{XYZ}_{XZ})^*(U)\Big) \vee (\mathrm{can}^{XZT}_{ZT})_*(V) = (\mathrm{can}^{XYZT}_{XZT})^*\Big((\mathrm{can}^{XYZT}_{XYZ})_*(U) \vee (\mathrm{can}^{XYZT}_{ZT})_*(V)\Big).
\end{equation*}
\end{lemma}

\begin{proof}
In general,  for any set $W$ and any $U'\in\mathrm{Part}(W\sqcup Z)$ and $V\in\mathrm{Part}(Z\sqcup T)$ the partition
$(\mathrm{can}_{WZ}^{WZT})_*(U')\vee (\mathrm{can}^{WZT}_{ZT})_*(V)\in\mathrm{Part}(W\sqcup Z\sqcup T)$ can be described explicitly as follows.  Let $\sim_{U'}$ (resp. $\sim_V$) denote the equivalence relation on $W\sqcup Z$ (resp. $Z\sqcup T$) induced by the partition $U'$ (resp. $V$). Define a relation $\sim$ on $W\sqcup Z\sqcup T$ by:
\begin{itemize}
\item if $w,w'\in W$, then $w\sim w'$ if and only if there exist $n\geq 0$ and $z_1,\ldots, z_{2n}\in Z$ such that $w\sim_{U'} z_1 \sim_{V} z_2\sim_{U'}\cdots \sim_{V}z_{2n}\sim_{U'} w'$,

\item if $z,z'\in Z$, then $z\sim z'$ if and only if there exist $n\geq 0$ and $z_1,\ldots, z_{2n}\in Z$ such that $z\sim_{U'} z_1 \sim_{V} z_2\sim_{U'}\cdots \sim_{V}z_{2n}\sim_{U'} z'$,

\item if $t,t'\in T$, then $t\sim t'$ if and only if there exist $n\geq 0$ and $z_1,\ldots, z_{2n}\in Z$ such that $t\sim_{V} z_1 \sim_{U'} z_2\sim_{V}\cdots \sim_{U'}z_{2n}\sim_{V} t'$,

\item if $w\in W$, $z\in Z$, then $w\sim z$ if and only if there exist $n\geq 0$ and $z_1,\ldots, z_{2n+1}\in Z$ such that $w\sim_{U'} z_1 \sim_{V} z_2\sim_{U'}\cdots \sim_{U'}z_{2n+1}\sim_{V} z$,

\item if $w\in W$, $t\in T$, then $w\sim t$ if and only if there exist $n\geq 0$ and $z_1,\ldots, z_{2n+1}\in Z$ such that $w\sim_{U'} z_1 \sim_{V} z_2\sim_{U'}\cdots \sim_{U'}z_{2n+1}\sim_{V} t$,

\item if $z\in Z$, $t\in T$, then $t\sim z$ if and only if there exist $n\geq 0$ and $z_1,\ldots, z_{2n+1}\in Z$ such that $t\sim_{V} z_1 \sim_{U'} z_2\sim_{V}\cdots \sim_{V}z_{2n+1}\sim_{U'} z$.

\end{itemize}
One checks that $\sim$ is an equivalence relation. Let $P\in \mathrm{Part}(W\sqcup Z\sqcup T)$  be the partition associated to $\sim$. We have $(\mathrm{can}_{WZ}^{WZT})_*(U')\leq P$ and  $(\mathrm{can}^{WZT}_{ZT})_*(V)\leq P$. One also checks that $P\leq (\mathrm{can}_{WZ}^{WZT})_*(U')\vee (\mathrm{can}^{WZT}_{ZT})_*(V)\in\mathrm{Part}(W\sqcup Z\sqcup T)$ and therefore $P = (\mathrm{can}_{WZ}^{WZT})_*(U')\vee (\mathrm{can}^{WZT}_{ZT})_*(V)\in\mathrm{Part}(W\sqcup Z\sqcup T)$.

Let us come back to our particular case ($W:=X\sqcup Y$). Notice that $a,a'\in X\sqcup Z$ are such that $a\sim_{(\mathrm{can}_{XZ}^{XYZ})^*(U)} a'$ if and only if $a\sim_{U} a'$  we then see that the explicit description of both sides of the equality in the statement of the lemma is the same. 
\end{proof}

\begin{proposition}\label{r:2022-09-13-associativityinPart} Let $X$, $Y$, $Z$ and $T$ be sets. Let $A\in\underline{\mathcal Part}(X,Y)$, $B\in\underline{\mathcal Part}(Y,Z)$, $C\in\underline{\mathcal Part}(Z,T)$. Then $C \circ (B\circ A) = (C\circ B) \circ A$. In particular, $(\underline{\mathcal Part} , \circ)$ is a category, we call it \emph{category of partitions}.
\end{proposition}

\begin{proof}
We show that both $C \circ (B\circ A)$ and $(C\circ B) \circ A$ can  be expressed as
\begin{equation}\label{equ:2022-08-22-equ1}
(\mathrm{can}^{XYZT}_{XT})^*\Big[(\mathrm{can}^{XYZT}_{XY})_*(A) \vee (\mathrm{can}^{XYZT}_{YZ})_*(B) \vee (\mathrm{can}^{XYZT}_{ZT})_*(C)\Big].
\end{equation}
We have
\begin{equation}\label{equ:equ2022-09-13}
\begin{split}
C\circ (B\circ A) & = (\mathrm{can}^{XZT}_{XT})^*\big((\mathrm{can}^{XZT}_{XZ})_*(B\circ A)\vee(\mathrm{can}^{XZT}_{ZT})_*(C)\big)\\
& = (\mathrm{can}^{XZT}_{XT})^*\Big[(\mathrm{can}^{XZT}_{XZ})_*\Big((\mathrm{can}^{XYZ}_{XZ})^*\big((\mathrm{can}^{XYZ}_{XY})_*(A) \vee(\mathrm{can}^{XYZ}_{YZ})_*(B)\big)\Big)\\
& \qquad \qquad \qquad \qquad \qquad \vee (\mathrm{can}^{XZT}_{ZT})_*(C)\Big],
\end{split}
\end{equation}
where the first equality is by the definition by of $C \circ(B \circ A)$ and the second one by the definition of $B\circ A$.

Besides, \eqref{equ:2022-08-22-equ1} can be expressed as
\begin{equation}\label{equ:equ2-2022-09-13} 
\begin{split}
(\mathrm{can}^{XYZT}_{XT})^*&\Big[(\mathrm{can}^{XYZT}_{XY})_*(A) \vee (\mathrm{can}^{XYZT}_{YZ})_*(B) \vee (\mathrm{can}^{XYZT}_{ZT})_*(C)\Big] \\
&=(\mathrm{can}^{XYZT}_{XT})^*\Big[(\mathrm{can}^{XYZT}_{XYZ})_*\Big((\mathrm{can}^{XYZ}_{XY})_*(A)\vee (\mathrm{can}^{XYZ}_{YZ})_*(B)\Big) \vee (\mathrm{can}^{XYZT}_{ZT})_*(C)\Big]\\
& = \big((\mathrm{can}^{XZT}_{XT})^*\circ  (\mathrm{can}^{XYZT}_{XZT})^*\big)\Big[(\mathrm{can}^{XYZT}_{XYZ})_*\Big((\mathrm{can}^{XYZ}_{XY})_*(A)\vee (\mathrm{can}^{XYZ}_{YZ})_*(B)\Big)\\
& \qquad \qquad \qquad \qquad \qquad \qquad \qquad \qquad \vee (\mathrm{can}^{XYZT}_{ZT})_*(C)\Big].
\end{split}
\end{equation}
In the first equality we have used the identities  $\mathrm{can}^{XYZT}_{XY} =  \mathrm{can}^{XYZT}_{XYZ}\circ \mathrm{can}^{XYZ}_{XY}$ and $\mathrm{can}^{XYZT}_{YZ} = \mathrm{can}^{XYZT}_{XYZ}\circ \mathrm{can}^{XYZ}_{YZ}$ and the fact that $(\mathrm{can}^{XYZT}_{XYZ})_*$ is  join-preserving. In the second equality we have used the identity $\mathrm{can}_{XT}^{XYZT}= \mathrm{can}^{XYZT}_{XZT}\circ \mathrm{can}^{XZT}_{XT}$.

Applying $(\mathrm{can}_{XT}^{XZT})^*$ to equality in Lemma~\ref{equ:2022-08-22-equ2} with $U:=(\mathrm{can}_{XY}^{XYZ})_*(A) \vee (\mathrm{can}_{YZ}^{XYZ})_*(B)$ and $V:=C$ we obtain that the last term in \eqref{equ:equ2022-09-13} is equal to the last term in \eqref{equ:equ2-2022-09-13}. That is, $C\circ (B\circ A)$ can be expressed as in  \eqref{equ:2022-08-22-equ1}. A similar reasoning allows us to show that $(C\circ B)\circ A$ is also expressed as in \eqref{equ:2022-08-22-equ1}. Therefore $C\circ (B\circ A) = (C\circ B)\circ A$.
\end{proof}

\begin{lemma}\label{r:r1-2022-09-01} Let $X$, $Y$, $Z$ and $T$ be sets. Let $U\in\mathrm{Part}(X\sqcup Y\sqcup Z)$ and $C\in\mathrm{Part}(Z\sqcup T)$.  Then 
\begin{itemize}
\item [$(a)$] the assignment  $\psi: U_{|Y}\to (\mathrm{can}_{XYZ}^{XYZT})_*(U) \vee (\mathrm{can}_{ZT}^{XYZT})_*(C)$ taking $a\in U_{|Y}$ to the unique class of $(\mathrm{can}_{XYZ}^{XYZT})_*(U) \vee (\mathrm{can}_{ZT}^{XYZT})_*(C)$ containing  $\mathrm{can}_{XYZ}^{XYZT}(a)$, denoted $\widetilde{a}$, is well-defined and takes values in $\Big((\mathrm{can}_{XYZ}^{XYZT})_*(U) \vee (\mathrm{can}_{ZT}^{XYZT})_*(C)\Big)_{|Y}$,
\item[$(b)$] the induced map  
$$\psi:  U_{|Y}\longrightarrow \Big((\mathrm{can}_{XYZ}^{XYZT})_*(U) \vee (\mathrm{can}_{ZT}^{XYZT})_*(C)\Big)_{|Y}$$
by $(a)$ is a bijection.
\end{itemize}
\end{lemma}

\begin{proof} Let denote by $\sim_{P}$, $\sim_U$ and $\sim_C$ the equivalence relations on $X\sqcup Y\sqcup Z\sqcup T$, on $X\sqcup Y\sqcup Z$  and on $Z\sqcup T$ induced by $P:= (\mathrm{can}_{XYZ}^{XYZT})_*(U) \vee (\mathrm{can}_{ZT}^{XYZT})_*(C)$,  by $U$ and by $C$, respectively. We also recall from the proof of Lemma~\ref{equ:2022-08-22-equ2} that (taking $W:=X\sqcup Y$) there is an explicit description of $\sim_P$.

$(a)$ If $a\in U_{|Y}$ then $\mathrm{can}_{XYZ}^{XYZT}(a)\subset X\sqcup Y\sqcup Z \sqcup T$. Moreover, $\mathrm{can}_{XYZ}^{XYZT}(a)$ is an equivalence class of $(\mathrm{can}_{XYZ}^{XYZT})_*(U)$ and therefore there exists a unique class $\widetilde{a}$ in   $(\mathrm{can}_{XYZ}^{XYZT})_*(U) \vee (\mathrm{can}_{ZT}^{XYZT})_*(C)$ containing $\mathrm{can}_{XYZ}^{XYZT}(a)$. Let us show that $\widetilde{a}\subset Y$, i.e., $\widetilde{a}\in \Big((\mathrm{can}_{XYZ}^{XYZT})_*(U) \vee (\mathrm{can}_{ZT}^{XYZT})_*(C)\Big)_{|Y}$.  Assuming the contrary and using the explicit description of $\sim_P$ we have:
\begin{itemize}
\item If there exists $\widetilde{\alpha}\in \widetilde{a}\cap X$, then for every $\alpha\in a$ we have $\alpha\sim_P \widetilde{\alpha}$. Hence, $\alpha\sim_U \widetilde{\alpha}$. That is, $a\cap X\not= \emptyset$ which contradicts $a\in U_{|Y}$.
\item If there exists $\widetilde{\alpha}\in \widetilde{a}\cap Z$, then for every $\alpha\in a$ we have $\alpha\sim_P \widetilde{\alpha}$. Therefore, there exists $n\geq 0$ and $z_1,\ldots, z_{2n}\in Z$ such that  $\alpha\sim_U  z_1\sim_C z_2\sim_U\cdots \sim_C z_{2n}\sim_U\widetilde{\alpha}$. Hence, $\widetilde{\alpha}$ or $z_1$ belong to $a\cap Z$ which contradicts $a\in U_{|Y}$.

\item If there exists $\widetilde{\alpha}\in \widetilde{a}\cap T$, then for every $\alpha\in a$ we have $\alpha\sim_P \widetilde{\alpha}$. Therefore, there exists $n\geq 0$ and $z_1,\ldots, z_{2n+1}\in Z$ such that  $\alpha\sim_U  z_1\sim_C z_2\sim_U\cdots \sim_U z_{2n+1}\sim_C\widetilde{\alpha}$. Hence,  $z_1\in a\cap Z$ which contradicts $a\in U_{|Y}$.
\end{itemize}
To sum up, $\widetilde{a}\subset Y$.

$(b)$ If $a\in P_{|Y}$, then $a\in P$ and $a\subset Y$. Since $(\mathrm{can}_{XYZ}^{XYZT})_*(U)\leq P$, there exists $A'\subset U$ such that $a=\sqcup_{a'\in A'}a'$. Moreover, for all $a'\in A'$ we have $a'\subset Y$. Hence, $A'\subset U_{|Y}$.

Let $a',a''\in A'$, then for all $\alpha'\in a'\subset Y$ and for all $\alpha''\in a''\subset Y$ we have $\alpha'\sim_P \alpha''$. By the explicit description of $\sim_P$ we then have $\alpha'\sim_U \alpha''$ or there exists $n\geq 0$ and $z_1,\ldots z_{2n}\in Z$ such that $\alpha'\sim_Uz_1\sim_C z_2\sim_U\cdots \sim_C z_{2n}\sim_U \alpha''$. The latter case is discarded because it contradicts $a'\subset Y$. We conclude $a'=a''$, therefore $a\in U_{|Y}$ and $A'=\{a\}$; then  $\psi(a)=a$ which shows the surjectivity of $\psi$.  The same argument shows that $\psi(a)=\mathrm{can}_{XYZ}^{XYZT}(a)$ which using the injectivity of $\mathrm{can}_{XYZ}^{XYZT}$ implies the injectivity of $\psi$.
\end{proof}

\begin{proposition}\label{r:2022-09-13-bijofcirclesPart} Let $X$, $Y$, $Z$ and $T$ be sets. Let $A\in\underline{\mathcal Part}(X,Y)$, $B\in\underline{\mathcal Part}(Y,Z)$, $C\in\underline{\mathcal Part}(Z,T)$. Then there is a bijection
$$\mathrm{bij}_{C,B,A}: \mathrm{Cir}(C,B)\sqcup \mathrm{Cir}(C\circ B, A) \longrightarrow \mathrm{Cir}(B,A)\sqcup \mathrm{Cir}(C, B\circ A).$$
\end{proposition}
\begin{proof}
Let $\Upsilon:=(\mathrm{can}_{XY}^{XYZT})_*(A) \vee (\mathrm{can}_{YZ}^{XYZT})_*(B) \vee (\mathrm{can}_{ZT}^{XYZT})_*(C)\in \mathrm{Part}(X\sqcup Y\sqcup Z\sqcup T)$.
Using the equalities $(\mathrm{can}_{XY}^{XYZT})_* = (\mathrm{can}_{XYZ}^{XYZT}\circ \mathrm{can}_{XY}^{XYZ})_* = (\mathrm{can}_{XYZ}^{XYZT})_*\circ (\mathrm{can}_{XY}^{XYZ})_*$ and  $(\mathrm{can}_{YZ}^{XYZT})_* = (\mathrm{can}_{XYZ}^{XYZT}\circ \mathrm{can}_{YZ}^{XYZ})_* = (\mathrm{can}_{XYZ}^{XYZT})_*\circ (\mathrm{can}_{YZ}^{XYZ})_*$ together with Lemma~\ref{r:r1-2022-09-01} (with $U:=(\mathrm{can}_{XY}^{XYZ})_*(A) \vee (\mathrm{can}_{YZ}^{XYZ})_*(B)$) we have
\begin{equation}\label{equ:2022-09-01-equ2}
\begin{split}
\mathrm{Cir}(B,A)= U_{|Y} &\simeq \Big[(\mathrm{can}_{XYZ}^{XYZT})_*(U)\vee  \mathrm{can}_{ZT}^{XYZT}(C)\Big]_{|Y}\\
 & = \Big[(\mathrm{can}_{XYZ}^{XYZT})_*\Big((\mathrm{can}_{XY}^{XYZ})_*(A) \vee (\mathrm{can}_{YZ}^{XYZ})_*(B)\Big)\vee  (\mathrm{can}_{ZT}^{XYZT})_*(C)\big]_{|Y}\\
 & = \Big[(\mathrm{can}_{XY}^{XYZT})_*(A) \vee (\mathrm{can}_{YZ}^{XYZT})_*(B) \vee (\mathrm{can}_{ZT}^{XYZT})_*(C)\Big]_{|Y}\\
 & = \Upsilon_{|Y}.
\end{split}
\end{equation}

Besides, notice that for any $V\in\mathrm{Part}(X\sqcup Y\sqcup Z)$ we have
\begin{equation}\label{equ:2022-09-01equ1}
\begin{split}
\Big((\mathrm{can}_{XYZ}^{XYZT})_*(V) \vee &(\mathrm{can}_{ZT}^{XYZT})_*(C)\Big)_{Y\sqcup Z} \setminus \Big((\mathrm{can}_{XYZ}^{XYZT})_*(V) \vee (\mathrm{can}_{ZT}^{XYZT})_*(C)\Big)_{Y} \\
& \simeq  \Big[(\mathrm{can}_{XZT}^{XYZT})^*\Big((\mathrm{can}_{XYZ}^{XYZT})_*(V) \vee (\mathrm{can}_{ZT}^{XYZT})_*(C)\Big)\Big]_{|Z}\\
& = \Big[(\mathrm{can}^{XZT}_{XZ})_*\Big((\mathrm{can}^{XYZ}_{XZ})^*(V)\Big) \vee (\mathrm{can}^{XZT}_{ZT})_*(C) \Big]_{|Z}
\end{split}
\end{equation}
where the isomorphism is given by Lemma~\ref{r:r2-2022-09-01} and the equality is given by Lemma~\ref{equ:2022-08-22-equ2}.

Considering the instance of \eqref{equ:2022-09-01equ1} with $V:= (\mathrm{can}_{XY}^{XYZ})_*(A) \vee (\mathrm{can}_{YZ}^{XYZ})_*(B)$ and using the equalities $(\mathrm{can}_{XY}^{XYZT})_* = (\mathrm{can}_{XYZ}^{XYZT}\circ \mathrm{can}_{XY}^{XYZ})_* = (\mathrm{can}_{XYZ}^{XYZT})_*\circ (\mathrm{can}_{XY}^{XYZ})_*$ and  $(\mathrm{can}_{YZ}^{XYZT})_* = (\mathrm{can}_{XYZ}^{XYZT}\circ \mathrm{can}_{YZ}^{XYZ})_* = (\mathrm{can}_{XYZ}^{XYZT})_*\circ (\mathrm{can}_{YZ}^{XYZ})_*$ we obtain
\begin{equation*}
\begin{split}
\Upsilon_{|Y\sqcup Z} \setminus \Upsilon_{|Y} &\simeq \Big[(\mathrm{can}^{XZT}_{XZ})_*\Big((\mathrm{can}^{XYZ}_{XZ})^*((\mathrm{can}_{XY}^{XYZ})_*(A) \vee (\mathrm{can}_{YZ}^{XYZ})_*(B))\Big) \vee (\mathrm{can}^{XZT}_{ZT})_*(C) \Big]_{|Z}\\
 & =  \Big[(\mathrm{can}^{XZT}_{XZ})_*(B\circ A) \vee (\mathrm{can}^{XZT}_{ZT})_*(C)\Big]_{|Z} \\
 & = \mathrm{Cir}(C, B\circ A).
\end{split}
\end{equation*}
Combining the above isomorphism with that in \eqref{equ:2022-09-01-equ2} we get 
$$\Upsilon_{|Y\sqcup Z}\simeq \mathrm{Cir}(B,A)\sqcup \mathrm{Cir}(C, B\circ A).$$
Similarly, we show 
$$\Upsilon_{|Y\sqcup Z}\simeq \mathrm{Cir}(C,B)\sqcup \mathrm{Cir}(C\circ B, A),$$
which finishes the proof.
\end{proof}

\begin{remark}\label{rem:pentagone-gen} By using similar arguments to the ones used in the proof of Proposition~\ref{r:2022-09-13-bijofcirclesPart}, we can show that for any composable objects $A,B,C,D$ in $\underline{\mathcal Part}$ the diagram

\bigskip

\begin{tikzpicture}[commutative diagrams/every diagram]
\tikzstyle{every node}=[minimum width=4cm]
\node (P0) at (4,0) {$\mathrm{Cir}(C,B)\sqcup \mathrm{Cir}(D, (C B))\sqcup \mathrm{Cir}(D (C  B), A)$};
\node (P1) at (0,-2) {$\mathrm{Cir}(D,C)\sqcup \mathrm{Cir}(D C, B)\sqcup \mathrm{Cir}((D  C) B, A)$} ;
\node (P2) at (7,-4) {$\mathrm{Cir}(C,B)\sqcup \mathrm{Cir}(C B, A)\sqcup \mathrm{Cir}(D, (C  B) A)$};
\node (P3) at (0,-6) {$\mathrm{Cir}(D,C)\sqcup \mathrm{Cir}(B, A)\sqcup \mathrm{Cir}(D C, B A)$};
\node (P4) at (4,-8) {$\mathrm{Cir}(B,A)\sqcup \mathrm{Cir}(C, B  A)\sqcup \mathrm{Cir}(D, C  (B  A))$};

\draw[->] (P1.north) -- (P0.255)  node[midway,left] {$\mathrm{bij}_{D,C,B}\sqcup\mathrm{Id}_{\mathrm{Cir}((D C) B, A)}\qquad $};
\draw[->] (P0.295) -- (P2.north) node[midway, right] {$\quad \mathrm{Id}_{\mathrm{Cir}(C,B)} \sqcup \mathrm{bij}_{D,CB,A}$};
\draw[->] (P2.south) -- (P4.75)node[midway, right] {$\quad\mathrm{bij}_{C,B,A}\sqcup\mathrm{Id}_{\mathrm{Cir}(D, (C  B) A)} $};
\draw[->] (P1.south) -- (P3.north)  node[midway,left] {$\mathrm{Id}_{\mathrm{Cir}(D, C)}\sqcup\mathrm{bij}_{DC,B,A}\quad $};
\draw[->] (P3.south) -- (P4.105)node[midway,left] {$\big(\mathrm{Id}_{\mathrm{Cir}(B, A)}\sqcup\mathrm{bij}_{D,C,BA}\big)\circ t \quad \quad$};
\end{tikzpicture}

\bigskip

\noindent commutes, where $t:=t_{(\mathrm{Cir}(D, C), \mathrm{Cir}(B, A))}\sqcup \mathrm{Id}_{\mathrm{Cir}(DC, BA)}: \mathrm{Cir}(D, C)\sqcup\mathrm{Cir}(B, A) \to \mathrm{Cir}(B, A)\sqcup \mathrm{Cir}(D, C)$ is the canonical bijection. By simplicity we have omitted $\circ$ from the notation in the above diagram.
\end{remark}

\begin{remark} In \cite{Del}, Deligne introduces a $\mathbb{Z}[T]$-linear tensor category $\mathrm{Rep}_1(S_T)$ with a particular object $X$.  For $U$ and $V$ finite sets, $\mathrm{Rep}_1(S_T)(X^{\otimes U}, X^{\otimes V})$ is a free $\mathbb{Z}[T]$-module with basis $([P])_{P \in \mathrm{Part}(U \sqcup V)}$ (\cite[\S8.2]{Del}). For $U$, $V$ and $W$ finite sets and $P\in \mathrm{Part}(U \sqcup V)$ and $Q \in \mathrm{Part}(V \sqcup W)$, one has $[Q] \circ [P]=T^{|\mathrm{Cir}(P,Q)|}\langle Q \circ P\rangle_{U\sqcup W}$, where $\langle Q \circ P\rangle_{U\sqcup W}$ coincides with the partition $P\circ Q$ defined in Definition~\ref{def:2022-09-13-defPart}$(b)$. The associativity of the composition in $\mathrm{Rep}_1(S_T)$ is equivalent to the equalities $(R \circ Q) \circ P = R \circ (Q \circ P)$ (see Proposition~\ref{r:2022-09-13-associativityinPart}) and $|\mathrm{Cir}(RQ,P)|+|\mathrm{Cir}(R,Q)|=|\mathrm{Cir}(R,QP)|+|\mathrm{Cir}(Q,P)|$ (which is a consequence of Proposition~\ref{r:2022-09-13-bijofcirclesPart}). See also \cite{ComOs} for more details on the category $\mathrm{Rep}_1(S_T)$.
\end{remark}

\subsection{Classification of finite  \texorpdfstring{$\mathbb{Z}_2*\mathbb{Z}_2$}{Z2*Z2}--sets}\label{sec:1-2}

\begin{notation}\label{notation1} Let $X_1, \ldots, X_n$ be sets. The image of $x\in X_i$  under the injection $X_i \subset X_1 \sqcup \cdots \sqcup X_n$, i.e.,  under the map $\mathrm{can}_{X_i}^{X_1\cdots X_n}: X_i\to X_1\sqcup \cdots \sqcup X_n$ (see Notation~\ref{notation0}), is denoted by using $i-1$ bars over $x$. For instance $X_1\ni x \mapsto x \in X_1\sqcup X_2\sqcup X_3$ and $X_2\ni x \mapsto {\bar{x}} \in X_1\sqcup X_2\sqcup X_3$ and $X_3\ni x \mapsto \bar{\bar{x}} \in  X_1\sqcup X_2\sqcup X_3$
\end{notation}

Consider the group $G:=\mathbb{Z}\rtimes \mathbb{Z}_2=\langle t, c \ | \  c^2=1, \ ct=t^{-1}c \rangle$. For $N\in\mathbb{Z}_{\geq 1}$ define the cofinite subgroups $A_N$, $B_N$ and $C_N$ of $\mathbb{Z}\rtimes \mathbb{Z}_2$ generated by $\{t^N\}$, $\{t^N,c\}$ and $\{t^N, ct\}$, respectively.

\begin{lemma}\label{r:cofinitesubgroupsofG} Define the map
$$F:\mathbb{Z}_{\geq 1}\sqcup (2\mathbb{Z}_{\geq 0}+1)\sqcup (2\mathbb{Z}_{\geq 0}\times \{0,1\})\longrightarrow \{\text{cofinite subgroups of } G\}
$$
by
$F(N):= A_N$, $F(\overbar{2N+1}):= B_{2N+1}$, $F\left(\overbar{\overbar{(2N,0)}}\right):= B_{2N}$ and  $F\left(\overbar{\overbar{(2N,1)}}\right):= C_{2N}$. Then $F$ induces a bijection
 $$[F]:\mathbb{Z}_{\geq 1}\sqcup (2\mathbb{Z}_{\geq 0}+1)\sqcup (2\mathbb{Z}_{\geq 0}\times \{0,1\})\longrightarrow \dfrac{\{\text{cofinite subgroups of } G\}
}{\text{conjugation}}.$$
\end{lemma}

\begin{proof} For $x\in (2\mathbb{Z}_{\geq 0}+1)\sqcup (2\mathbb{Z}_{\geq 0}\times \{0,1\})$ we set $|x|:=x$ if $x\in 2\mathbb{Z}_{\geq 0}+1$ and $|x|:=y$ if $x=(y,\epsilon)\in 2\mathbb{Z}_{\geq 0}\times\{0,1\}$. 

Let us first show the injectivity of $[F]$. Let $z,z'\in \mathbb{Z}_{\geq 1}\sqcup (2\mathbb{Z}_{\geq 0}+1)\sqcup (2\mathbb{Z}_{\geq 0}\times \{0,1\})$ such that $[F]z=[F]z'$. Then $\mathrm{Im}\big(F(z)\subset G\to \mathbb{Z}_2\big)=\mathrm{Im}\big(F(z')\subset G\to \mathbb{Z}_2\big)$ is equal to $\mathbb{Z}_2$ or $\{0\}$. In the first case, $z,z'\in \mathbb{Z}_{\geq 0}$ and so $2z=|G/F(z)|=|G/F(z')|=2z'$. Consequently, $z=z'$. In the second case, $z,z'\in (2\mathbb{Z}_{\geq 0}+1)\sqcup (2\mathbb{Z}_{\geq 0}\times \{0,1\})$ and we have $|z|=|G/F(z)|=|G/F(z')|=|z'|$. Therefore, $z=z'$ or there exists $N\in\mathbb{Z}_{\geq 0}$ with $z=(2N,0)$ and $z'=(2N,1)$. We claim that the latter option is not possible. Indeed, we have $F(2N,0)\cap\{t^kc \ | \ k\in\mathbb{Z}\}=B_{2N}\cap\{t^kc \ | \ k\in\mathbb{Z}\}\subset \{t^kc \ | \ k\in 2\mathbb{Z}\}$. Besides, for any $g\in G$ we have $(gF(2N,1)g^{-1})\cap\{t^kc \ | \ k\in\mathbb{Z}\}=(gC_{2N}g^{-1})\cap\{t^kc \ | \ k\in\mathbb{Z}\}\subset \{t^kc \ | \ k\in 2\mathbb{Z}+1\}$. In consequence, $[F](2N,0)\not = [F](2n,1)$. We then conclude that $z=z'$, i.e., $[F]$ is injective.

We proceed to show the surjectivity of $[F]$.    Let $H$ be a cofinite subgroup of $G$. Suppose that $\mathrm{Im}(H\subset G \to \mathbb{Z}_2)= \{0\}$, then $H\leq \mathrm{ker}(G\to\mathbb{Z}_2)= \mathbb{Z}$ and the map $\mathbb{Z}/H \to G/H$ is injective. Hence, $H$ is a cofinite subgroup of $\mathbb{Z}$, i.e., there exists $N\in\mathbb{Z}_{\geq 1}$ such that $H=\langle t^N\rangle = A_N = F(N)$.

Suppose now that $\mathrm{Im}(H\subset G \to \mathbb{Z}_2)= \mathbb{Z}_2$. The map $\mathbb{Z}/H\cap \mathbb{Z} \to G/H$ is injective and therefore $H\cap \mathbb{Z}$ is a cofinite subgroup of $\mathbb{Z}$, i.e., there exists $M\in\mathbb{Z}_{\geq 1}$ such that $H\cap \mathbb{Z}=\langle t^M\rangle$. Besides, the group $\langle t^kc \ | \ k\in\mathbb{Z} \rangle$ is endowed with a free and transitive action of $\langle t\rangle = \mathbb{Z}$ which induces a free and transitive action of $H\cap \mathbb{Z}$ on $H\cap \langle t^kc \ | \ k\in\mathbb{Z} \rangle$. Therefore, there exists $a\in \mathbb{Z}$ such that $H\cap \langle t^kc \ | \ k\in\mathbb{Z} \rangle=\{t^{a+lM}c\ | \ l\in\mathbb{Z}\}$. Hence, $H=\{t^{lM}\ |\ l\in\mathbb{Z}\}\sqcup \{t^{a+lM}c\ | l\in\mathbb{Z}\ \}$. We  have the following possibilities: 
\begin{itemize}
\item[$(a)$] the integers $M$ and $a$ are both even. Let $b\in\mathbb{Z}$ such that $a=2b$, then $H=t^bB_Mt^{-b}$, i.e., $[H]=[F](\overbar{\overbar{(M,0)}})$.
\item[$(b)$] the integer $M$ is even and $a$ is odd. Let $b\in\mathbb{Z}$ such that $a=2b+1$, then $H=t^bC_Mt^{-b}$, i.e., $[H]=[F](\overbar{\overbar{(M,1)}})$.
\item[$(c)$] the integer $M$ is odd. Then  $H=t^dC_Mt^{-db}$, where $d=a/2$ if $a$ is even and $d=(a+M)/2$ is $a$ is odd. We then have $[H]=[F](\overbar{M})$.
\end{itemize}
\end{proof}

\begin{proposition}\label{r:classificationofGsets} Let $\widetilde{G}:=\mathbb{Z}_2*\mathbb{Z}_2=\langle \alpha, \beta \ |\ \alpha^2=\beta^2=1 \rangle$. Let $X$ be a finite $\widetilde{G}$-set and for $\gamma\in\{\alpha,\beta\}$ define $X^\gamma:=\{x\in X  \ | \  \gamma x=x\}$. Then $(|X^\alpha|,|X^\beta|)\in\{(0,0), (2,0), (1,1), (0,2)\}$ and therefore  $|X^\alpha| + |X^\beta|\in\{0,2\}$.
\end{proposition}

\begin{proof}
We use the isomorphism $\widetilde{G}\simeq G=\mathbb{Z}\rtimes \mathbb{Z}_2=\langle t,c \ | \  c^2=1, \ ct =t^{-1}c \rangle$ given by $\alpha \beta\leftrightarrow  t$ and $\beta \leftrightarrow c$. Besides, recall that there is a bijection
$$\dfrac{\{\text{finite } G\text{--sets}\}}{\text{isomorphism}}\simeq \dfrac{\{\text{cofinite subgroups of } G\}}{\text{conjugation}}.$$
By Lemma~\ref{r:cofinitesubgroupsofG} the cofinite subgroups of $G$ are classified in the three families $A_N$, $B_N$ and $C_N$.  Let us consider case by case. If $X=G/A_N$, then $X=\{\bar{1},\bar{t},\ldots, \overbar{t^{N-1}}\}\sqcup \{\bar{c}, \overbar{ct},\ldots, \overbar{ct^{N-1}}\}$ and $c$ acts on $X$ by $c\cdot \overbar{t^m}=\overbar{ct^{m}}$ and $c\cdot \overbar{ct^{m}}=\overbar{t^{m}}$. Hence, $X^c=\emptyset$. Similarly, $tc$ acts on $X$ by $tc\cdot \overbar{t^m}=\overbar{ct^{m-1}}$ and $tc \cdot \overbar{ct^{m}}=\overbar{t^{m-1}}$. Hence, $X^{tc}=\emptyset$. Therefore, $(|X^\alpha|, |X^\beta|) = (|X^{tc}|, |X^{c}|) = (0,0)$.

If $X=G/B_N$, then $X=\{\overbar{1}, \overbar{t},\ldots,\overbar{t^{N-1}}\}$ and  $c$ acts on $X$ by $c\cdot\overbar{t^m}=\overbar{t^{-m}}$. Hence, $X^\beta\simeq X^c\simeq \{x\in\mathbb{Z}_N \ | \ 2x=0\}$. Similarly, $tc$ acts on $X$ by $tc\cdot \overbar{t^m} = \overbar{t^{-m+1}}$. Hence, $X^\alpha\simeq X^{tc}\simeq\{x\in\mathbb{Z}_{N} \ | \ 2x=1\}$. Then, if $N$ is even we have $(|X^\alpha|, |X^{\beta}|) = (0,2)$ and if $N$ is odd  we have $(|X^\alpha|, |X^{\beta}|) = (1,1)$.

Finally, if $X=G/C_N$, then $X=\{\overbar{1}, \overbar{t},\ldots,\overbar{t^{N-1}},\overbar{c}\}$ and  $c$ acts on $X$ by $c\cdot \overbar{c}= \overbar{1}$ and $c\cdot\overbar{t^m}=\overbar{t^{1-m}}$. Hence, $X^{\beta}\simeq X^c\simeq \{x\in\mathbb{Z}_N \ | \ 2x=1\}$. Similarly, $tc$ acts on $X$ by $tc\cdot \overbar{c}= \overbar{t}$, and $tc\cdot \overbar{t^m} = \overbar{t^{2-m}}$. Hence, $X^{\alpha}\simeq X^{tc}\simeq\{x\in\mathbb{Z}_{N} \ | \ 2(m-1)=0\}$. Then, if $N$ is even we have $(|X^{\alpha}|, |X^{\beta}|) = (2,0)$ and if $N$ is odd  we have $(|X^{\alpha}|, |X^{\beta}|) = (1,1)$.
\end{proof}

\subsection{The category  \texorpdfstring{$\underline{\mathcal Br}$}{Br} of Brauer diagrams}\label{sec:1-3}

\begin{definition}  Let $X$, $Y$ be sets. Define the set $\underline{\mathcal Br}'(X,Y) \subset \underline{\mathcal Part}(X,Y)$ to be the subset of $\underline{\mathcal Part}(X,Y)=\mathrm{Part}(X \sqcup Y)$\index[notation]{Br'(X,Y)@$\underline{\mathcal Br}'(X,Y)$} consisting of the partitions where all the constituents are of  cardinality~$2$. 
\end{definition}

For a finite set $X$ we denote by $\mathrm{FPFI}(X)$\index[notation]{FPFI(X)@$\mathrm{FPFI}(X)$} the set of fixed-point free involutions of $X$. Therefore, by definition, we have  a bijection $\underline{\mathcal Br}'(X,Y)\simeq\mathrm{FPFI}(X\sqcup Y)$. For $A\in\mathcal \underline{\mathcal Br}'(X,Y)$, we denote by $\sigma_A$ the corresponding element in $\mathrm{FPFI}(X\sqcup Y)$.

\begin{definition}\label{def:concatenationofperm} Let $X$ and $Y$ be sets. For any pair of permutations  $\sigma\in\mathfrak{S}_X$ and $\tau\in\mathfrak{S}_Y$ we define their \emph{concatenation} $\sigma*\tau\in\mathfrak{S}_{X\sqcup Y}$ by $\sigma*\tau(a):=\sigma(a)$ if $a\in X$ and $\sigma*\tau(\bar{a}):=\overline{\tau(a)}$ if $a\in Y$. Here we use Notation~\ref{notation1}.
\end{definition}

\begin{theorem}\label{r:2022-09-27-1} Let $X$, $Y$, and $Z$ be sets. If $A\in \underline{\mathcal Br}'(X,Y)\subset \underline{\mathcal Part}(X,Y)$ and $B\in \underline{\mathcal Br}'(Y,Z)\subset \underline{\mathcal Part}(X,Z)$ then the element $B\circ A\in\underline{\mathcal Part}(X,Z)$ belongs to $\underline{\mathcal Br}'(X,Z)$. 
\end{theorem}
\begin{proof}
Let $A\in \underline{\mathcal Br}'(X,Y)$ and $B\in\underline{\mathcal Br}'(Y,Z)$. 
Let $\alpha:=\sigma_A*\mathrm{Id}_{Z}\in\mathfrak{S}_{X\sqcup Y\sqcup Z}$ and $\beta:=\mathrm{Id}_{X}*\sigma_B\in\mathfrak{S}_{X\sqcup Y\sqcup Z}$, where $\sigma_A$ and $\sigma_B$ are as indicated before Definition~\ref{def:concatenationofperm}. By definition, we have
$$\frac{X\sqcup Y\sqcup Z}{\langle \alpha,\beta\rangle} = (\mathrm{can}_{XY}^{XYZ})_*(A) \vee (\mathrm{can}_{YZ}^{XYZ})_*(B), \quad \mathrm{Cir}(B,A) = \left\{ \omega\in \frac{X\sqcup Y\sqcup Z}{\langle \alpha,\beta\rangle} \  \Big | \  \omega \subset Y\right\}.$$

Besides, if $\omega \in \big((X \sqcup Y \sqcup Z)/\langle \alpha, \beta\rangle\big) \setminus \mathrm{Cir}(B,A)$, then $\omega \cap (X \sqcup Z) \neq \emptyset$. It follows that $\{\omega \ | \ \omega \in \big((X \sqcup Y \sqcup Z)/\langle \alpha,\beta\rangle\big) \setminus \mathrm{Cir}(B,A)\}$ is a partition of $X \sqcup Z \sqcup (Y\setminus \cup_{\omega \in \mathrm{Cir}(B,A)}\omega)$. If $U' \subset U$ is an inclusion of sets, and if $\mathcal E$ is a partition of $U$ such that for any $x \in \mathcal E$, $x \cap U' \neq \emptyset$, then the restriction of the map $\mathcal P(U) \to \mathcal P(U')$, $x \mapsto x \cap U'$ to $\mathcal E$ is injective, and if we denote by $\mathcal E':=\{x \cap U' \ | \ x \in \mathcal E\}$ the image of this map, then the map $\mathcal E \to \mathcal E'$ induced by $x\mapsto x \cap U'$ is a bijection; moreover, $\mathcal E'$ is a partition of $U'$. It follows that $\{\omega \cap (X \sqcup Z) \ |\  \omega \in \big((X \sqcup Y \sqcup Z)/\langle \alpha, \beta\rangle\big) \setminus \mathrm{Cir}(B,A)\}$ is a partition of $X \sqcup Z$. One checks it to coincide with $B \circ A$, therefore 
$$ B\circ A = \left\{ \omega\cap(X\sqcup Z)  \Big | \ \omega\in \frac{X\sqcup Y\sqcup Z}{\langle \alpha,\beta\rangle}\setminus \mathrm{Cir}(B,A)\right\}\in \mathrm{Part}(X\sqcup Z).$$

The combination of the disjoint union decomposition $(X \sqcup Y \sqcup Z)/\langle \alpha,\beta\rangle =\big(\big((X \sqcup Y \sqcup Z)/\langle \alpha, \beta \rangle\big) \setminus \mathrm{Cir}(B,A)\big)\sqcup \mathrm{Cir}(B,A)$ and  of the bijection $\big((X \sqcup Y \sqcup Z)/\langle \alpha,\beta \rangle\big) \setminus \mathrm{Cir}(B,A) \to B \circ A$ gives rise to a natural bijection 
\begin{equation}\label{eq:toto}
\dfrac{X \sqcup Y \sqcup Z}{\langle \alpha, \beta\rangle} \simeq  (B \circ A) \sqcup  \mathrm{Cir}(B,A).
\end{equation}

Let us show that for any $\varphi \in B\circ A$, we have $|\varphi| = 2$ which shows that $B\circ A\in \underline{\mathcal Br}'(X,Z)$. Indeed, if $\omega\in \frac{X\sqcup Y\sqcup Z}{\langle \alpha,\beta\rangle}$, then $\omega^\alpha = \omega \cap X$ and $\omega^\beta = \omega\cap Z$, here we are using the notations from Proposition~\ref{r:classificationofGsets}. By this proposition, the number $|\omega\cap (X\sqcup Z)|= |\omega^\alpha| + |\omega^\beta|$ belongs to $\{0,2\}$. If $\varphi:=\omega\cap (X\sqcup Z)\in B\circ A$, then $\omega\not\in \mathrm{Cir}(B,A)$, that is, $\omega\not\subset Y$. Therefore, $|\varphi|=|\omega\cap (X\sqcup Z)|>0$. Hence, $|\varphi|=2$, which finishes the proof.
\end{proof}

Theorem~\ref{r:2022-09-27-1} allows to define the category $\underline{\mathcal Br}'$ of \emph{abstract Brauer diagrams} with same objects as the category $\underline{\mathcal Part}$ of partitions. 

\begin{corollary}\label{Skprimeisasubcategory} The category $\underline{\mathcal Br}'$ of abstract Brauer diagrams  is a subcategory of the category $\underline{\mathcal Part}$ of partitions.
\end{corollary}

\begin{definition}\label{def:defofSk}
Define a category  $\underline{\mathcal{B}r}$\index[notation]{Br@$\underline{\mathcal{B}r}$}  of \emph{Brauer diagrams} as follows. The set  of objects $\mathrm{Ob}(\underline{\mathcal Br})$ of $\underline{\mathcal Br}$ is~$\mathbb{Z}_{\geq 0}$. For $p,q\in \mathbb{Z}_{\geq 0}$ the set of morphisms $\underline{\mathcal Br}(p,q)$ from $p$ to $q$ is given by $\underline{\mathcal Br}(p,q):=\underline{\mathcal Br}'([\![1,p]\!],[\![1,q]\!])$. In particular, for any $p,q,r\in\mathbb{Z}_{\geq 0}$ there is a map $\mathrm{Cir}: \underline{\mathcal Br}(p,q)\times \underline{\mathcal Br}(q,r)\to \mathcal{S}ets$ satisfying the properties from Proposition~\ref{r:2022-09-13-bijofcirclesPart} and Remark~\ref{rem:pentagone-gen}.
\end{definition}

\begin{definition} 
\begin{itemize}
\item[$(a)$] The set $\mathcal Br$\index[notation]{Br@$\mathcal Br$} of \emph{Brauer diagrams} is defined as $\bigsqcup_{p,q\in\mathbb{Z}_{\geq 0}} \underline{\mathcal Br}(p,q)$, i.e.,
\begin{equation}
\mathcal Br = \bigsqcup_{p,q\in\mathbb{Z}_{\geq 0}}\underline{\mathcal Br}(p,q) \simeq \big\{(p,q,\sigma) \ | \ p,q\in\mathbb{Z}_{\geq 0}, \ \sigma\in\mathrm{FPFI}\big([\![1,p]\!]\sqcup [\![1,q]\!]\big)\big\}.
\end{equation}
\item[$(b)$] Define the \emph{source} and \emph{target} maps $\mathrm{s,t}:\mathcal Br \to \mathbb{Z}_{\geq 0}$ by $\mathrm{s}(p,q,\sigma):= p$ and $\mathrm{t}(p,q,\sigma):= q$.
\item[$(c)$] Let $\mathcal Br \times_{\mathbb Z} \mathcal Br:=\{(a,b) \in \mathcal Br^2 \ | \ \mathrm{t}(a)=\mathrm{s}(b)\}$ and  $\circ_{\mathcal Br}:\mathcal Br \times_{\mathbb Z} \mathcal Br \to \mathcal Br$, $(a,b)\mapsto b \circ_{\mathcal Br} a$\index[notation]{b\circ  a@$b \circ_{\mathcal Br} a$} be the map corresponding to the composition in $\underline{\mathcal Br}$. 
\end{itemize}
By convention, $\underline{\mathcal Br}(0,0)$ consists in one element denoted $\varnothing$\index[notation]{\varnothing @$\varnothing$}.
\end{definition}

\begin{definition}\label{def:piceroforBr} For $(p,q,\sigma)\in \mathcal Br$, we define its set of \emph{connected components}, denoted  $\pi_0(p,q,\sigma)$\index[notation]{\pi_0(p,q,\sigma)@$\pi_0(p,q,\sigma)$}, by $\pi_0(p,q,\sigma):=([\![1,p]\!] \sqcup[ \![1,q]\!])/\langle\sigma\rangle$.
\end{definition}

The map $(p,q,\sigma) \mapsto \pi_0(p,q,\sigma)$ sets up a bijection underline $\underline{\mathcal Br}(p,q) \to \underline {\mathcal Br}'([\![1,p]\!],[\![1,q]\!])$ for any $p,q \in\mathbb{Z}_{\geq 0}$. Moreover, the collection of these maps is compatible with compositions on both sides. In particular, equation \eqref{eq:toto} implies that for any $(p,q,\sigma)$ and $(q,r,\tau)$ in $\mathcal Br$, one has 
\begin{equation}\label{eq:toto2}
[\![1,p]\!] \sqcup [\![1,q]\!] \sqcup [\![1,r]\!]/\langle \alpha,\beta \rangle \simeq \pi_0\big((q,r,\tau)\circ(p,q,\sigma)\big)\sqcup \mathrm{Cir} ((q,r,\tau),(p,q,\sigma))
\end{equation}
where $\alpha$, $\beta$ are the involutions of  $[\![1,p]\!] \sqcup [\![1,q]\!] \sqcup [\![1,r]\!]$ given by $\sigma*\mathrm{Id}_{[\![1,r]\!]}$ and $\mathrm{Id}_{[\![1,p]\!]}*\tau$.

\subsection{Operations on Brauer diagrams}\label{sec:1-4}

\subsubsection{Tensor product of Brauer diagrams}

\begin{definition}\label{def:tensorproductinSk}
Let $p,p',q,q'\in\mathbb{Z}_{\geq 0}$ and $(p,q,\sigma), (p',q',\sigma')\in\mathcal Br$. Consider the identification $[\![1,p+p']\!]\simeq [\![1,p]\!]\sqcup [\![1,p']\!]$ given by $[\![1,p]\!]\ni x\mapsto x$ and $[\![1,p']\!]\ni x\mapsto x+p$. Similarly, consider the identification $[\![{1},{q+q'}]\!]\simeq [\![{1},{q}]\!]\sqcup [\![{1},{q}']\!]$. Therefore,
\begin{equation}\label{equ:identificationofintervalfortensor}
[\![1,p+p']\!]\sqcup [\![{1},{q+q'}]\!]\simeq [\![1,p]\!]\sqcup [\![1,p']\!]\sqcup [\![{1},{q}]\!]\sqcup [\![{1},{q}']\!]\simeq [\![1,p]\!]\sqcup [\![{1},{q}]\!]\sqcup [\![1,p']\!]\sqcup [\![{1},{q}']\!].
\end{equation}
Then $\sigma*\sigma'$ (see Definition~\ref{def:concatenationofperm}) defines a fixed-point free involution of $[\![1,p+p']\!]\sqcup [\![{1},{q+q'}]\!]$. The Brauer diagram $(p+p',q+q',\sigma*\sigma')$ is called the \emph{tensor product} of $(p,q,\sigma)$ and $(p',q',\sigma')$ and it is denoted by  $(p,q,\sigma)\otimes(p',q',\sigma')$.
\end{definition}

The following is straightforward.

\begin{lemma} The  tensor product from Definition~\ref{def:tensorproductinSk} endows the category $\underline{\mathcal Br}$ with a monoidal structure. The unit object is $0 \in \mathbb{Z}_{\geq 0}$ and its identity morphism is $\varnothing$. 
\end{lemma}

It follows from the axioms of  a monoidal category that for any pair $(P_i,P'_i)$ $(i=1,2)$ of pairs of composable morphisms in $\underline{\mathcal Br}$, one has 
\begin{equation}\label{eq:comptensorandcompositioninBr}
(P'_1 \otimes P'_2) \circ (P_1 \otimes P_2)=(P'_1 \circ P_1) \otimes (P'_2 \circ P_2).
\end{equation}

\begin{lemma}\label{r:bijectioncomptensorandcompinBr}  Let $P_i,P_i'$ composable morphisms in $\underline{\mathcal Br}$ for $i=1,2$. Then there is a bijection 
$$\mathrm{bij}_{P'_1,P_1,P'_2,P_2} : \mathrm{Cir}(P'_1,P_1) \sqcup \mathrm{Cir}(P'_2,P_2)\longrightarrow \mathrm{Cir}(P'_1 \otimes P'_2,P_1 \otimes P_2).$$
\end{lemma}
\begin{proof}
Let $\varphi:[\![1, |\mathrm{t}(P_1)|]\!]\sqcup[\![1,|\mathrm{t}(P_2)|]\!]\to [\![1, |\mathrm{t}(P_1)| + |\mathrm{t}(P_2)|]\!]$ be the canonical bijection taking $x\in [\![1, |\mathrm{t}(P_1)|]\!]$ to $x$ and $x\in [\![1, |\mathrm{t}(P_2)|]\!]$ to $x+ |\mathrm{t}(P_1)|$. One checks that the induced map $$\mathrm{Cir}(P'_1,P_1) \sqcup \mathrm{Cir}(P'_2,P_2)\longrightarrow \mathrm{Cir}(P'_1 \otimes P'_2,P_1 \otimes P_2)$$
taking $\omega \in \mathrm{Cir}(P'_1,P_1) \sqcup \mathrm{Cir}(P'_2,P_2)$ to $\varphi(\omega)$ is a bijection.
\end{proof}

\subsubsection{Doubling operation for Brauer diagrams}

\begin{definition}\label{def:doublingforskeleta} Let $(p,q,\sigma)\in\mathcal Br$ and $A\subset \pi_0(p,q,\sigma)$. Set $p':=p + |\sqcup_{\omega\in A}\omega\cap [\![1,p]\!]|$ and $q':=q + |\sqcup_{\omega\in A}\omega\cap [\![{1},{q}]\!]|$.
\begin{itemize}
\item[$(a)$]  Let $\mathrm{pr}_{\mathrm{s}}: [\![1,p']\!]\to [\![1,p]\!]$ be the only no-decreasing  map such that $|\mathrm{pr}_{\mathrm{s}}^{-1}(x)|=\begin{cases} 1 & \text{ if } x\notin A\cap [\![1,p]\!]\\
2 & \text{ if } x\in A\cap [\![1,p]\!]\end{cases}$

Similarly we define $\mathrm{pr}_{\mathrm{t}}: [\![{1},{q}']\!]\to [\![{1},{q}]\!]$. Set $\mathrm{pr}:= \mathrm{pr}_{\mathrm{s}}\sqcup \mathrm{pr}_{\mathrm{t}}: [\![1,p']\!]\sqcup [\![{1},{q}']\!]\to [\![1,p]\!]\sqcup [\![{1},{q}]\!]$.
\item[$(b)$] Let $\sigma'\in\mathrm{FPFI}\big([\![1,p']\!]\sqcup [\![1,q']\!]\big)$ be the only involution such that
\begin{itemize}
\item[$\bullet$] $\sigma\circ \mathrm{pr} = \mathrm{pr}\circ \sigma'$. This implies that for any $x\subset [\![1,p]\!]\sqcup [\![{1},{q}]\!]$ there is a map $\sigma'_x:\mathrm{pr}^{-1}(x)\to \mathrm{pr}^{-1}(\sigma(x))$,
\item[$\bullet$] for any $a\in A$, the map $\sigma'_a$ is a bijection which is decreasing if $a\subset [\![1,p]\!]$ or $a\subset [\![{1},{q}]\!]$ and increasing otherwise.
\end{itemize}
\item[$(c)$] The Brauer diagram $(p',q',\sigma')\in\mathcal Br$ is called the \emph{doubling of $(p,q,\sigma)$ along $A$}, and we denote it by $\mathrm{dbl}_{\mathcal Br}((p,q,\sigma),A)$ or $(p,q,\sigma)^A$.
\end{itemize}

\end{definition}

\subsection{Oriented Brauer diagrams}\label{sec:1-5}

\begin{definition}\label{def:vecBr}
\begin{itemize}
\item[$(a)$] Let $p,q\geq 0$ be integers. The set   $\underline{\vec{\mathcal Br}}(p,q)$ of \emph{oriented Brauer diagrams from~$p$ to~$q$} is  defined as the set of tuples $\big((p,q,\sigma), B\big)$ where $(p,q,\sigma)\in \mathcal Br$ and $B\subset [\![1,p]\!]\sqcup [\![{1},{q}]\!]$ such that $\sigma_{|B}: B\to \big([\![1,p]\!]\sqcup [\![{1},{q}]\!]\big)/\langle\sigma\rangle=\pi_0(p,q,\sigma)$ is a bijection. The elements of $B$ are called the \emph{beginning points}  of the diagram $((p,q,\sigma),B)$ and the elements of $E:=\sigma(B)$ its \emph{ending points}.  

\item[$(b)$] The set $\vec{\mathcal Br}$\index[notation]{Br2@$\vec{\mathcal Br}$} of \emph{oriented Brauer diagrams} is defined as $\bigsqcup_{p,q\in\mathbb{Z}_{\geq 0}} \underline{\vec{\mathcal Br}}(p,q)$, i.e.,
\begin{equation*}
\vec{\mathcal Br} = \Big\{\big((p,q,\sigma), B\big) \ \mid \  (p,q,\sigma)\in\mathcal Br, \ B\subset [\![1,p]\!]\sqcup [\![{1},{q}]\!] \text{ such that } B\simeq \big([\![1,p]\!]\sqcup [\![{1},{q}]\!]\big)/\sigma\Big\}.
\end{equation*}
\item[$(c)$] For an oriented Brauer diagram $\big((p,q,\sigma), B\big)$ we define the set of its \emph{connected components}, denoted $\pi_0\big((p,q,\sigma), B\big)$\index[notation]{\pi_0\big((p,q,\sigma), B\big)@$\pi_0\big((p,q,\sigma), B\big)$}, as the set $\pi_0(p,q,\sigma)$.
\end{itemize}
We denote by $\vec{\varnothing}$\index[notation]{\varnothing2@$\vec{\varnothing}$}  the unique element of ${\vec{ \mathcal Br}}(0,0)$.
\end{definition}

If $\big((p,q,\sigma), B\big)\in \vec{\mathcal{B}r}$ and $E:=\sigma(B)$, then $B\sqcup E = [\![1,p]\!]\sqcup [\![{1},{q}]\!]$ and $\sigma_{|B}: B \to E$ is a bijection.

\begin{definition}\label{def:beginningandendofadiagram}
For $x:=((p,q,\sigma),B)$ an oriented Brauer diagram, we denote by $\mathrm{b}_x : \pi_0(x) \to B$ and $\mathrm{e}_x : \pi_0(P) \to E$ the maps such that for any $u \in \pi_0(x)$, one has $u=\{\mathrm{b}_x(u),\mathrm{e}_x(u)\}$.
\end{definition}

The maps $\mathrm{b}_x$ and $\mathrm{e}_x$ are bijections, and the bijection $\sigma_{|B} : B \to E$ is the composition $\mathrm{e}_x \circ \mathrm{b}_x^{-1}$.

\begin{definition}\label{def:compositioninOSk} Two elements 
 $\big((p,q,\sigma), B\big), \big((p',q',\sigma'), B'\big)\in  \vec{\mathcal Br}$ are \emph{composable} if $q=p'$ and $\sigma(B)\cap [\![1,q]\!] = B' \cap [\![1,q]\!]$ (notice that this is equivalent to  $B\cap [\![1,q]\!] = \sigma'(B') \cap [\![1,q]\!]$). In such a case, the \emph{composition} of $\big((p,q,\sigma), B\big)$ and $\big((p',q',\sigma'), B'\big)$, denoted $\big((p',q',\sigma'), B'\big)\circ_{\vec{\mathcal Br}} \big((p,q,\sigma), B\big)$, is defined by
$$\big((p',q',\sigma'), B'\big)\circ_{\vec{\mathcal Br}} \big((p,q,\sigma), B\big):= \big((p',q',\sigma')\circ_{\mathcal Br} (p,q,\sigma), (B\cap[\![1,p]\!]) \sqcup (B'\cap[\![1,q']\!])\big).$$ 
\end{definition}


\begin{proposition}  The operation $\circ_{\vec{\mathcal Br}}$ is associative.
\end{proposition}

\begin{proof}
Let $a:=\big((p,q,\alpha),A\big), b:=\big((q,r,\beta),B\big), c:=\big((r,s,\gamma),C\big)\in\vec{\mathcal Br}$. The equality $c\circ_{\vec{\mathcal Br}}(b \circ_{\vec{\mathcal Br}} a) =(c\circ_{\vec{\mathcal Br}} b) \circ_{\vec{\mathcal Br}} a$ follows from Corollary~\ref{Skprimeisasubcategory}, the definition of $\circ_{\vec{\mathcal Br}}$ and the equalities of sets
\begin{equation*}
\begin{split}
\Big[\big((A\cap [\![1,p]\!])\sqcup (B\cap [\![1,r]\!])\big)\cap [\![1,p]\!]\Big]\sqcup (C\cap [\![1,s]\!])&= (A\cap [\![1,p]\!])\sqcup (C \cap [\![1,s]\!])\\
&= (A\cap [\![1,p]\!])\sqcup \Big[\big((B\cap [\![1,q]\!])\sqcup (C\cap [\![1,s]\!])\big)\cap [\![1,s]\!]\Big].
\end{split}
\end{equation*}
\end{proof}

\subsection{Operations on oriented Brauer diagrams}\label{sec:1-6}

\subsubsection{Tensor product of oriented Brauer diagrams}
\begin{definition}\label{def:tensorproductinOSk} Let $\big((p,q,\sigma),B\big),\big((p',q',\sigma'),B'\big)$ be oriented Brauer diagrams. The \emph{tensor product} of $\big((p,q,\sigma),B\big)$ and $\big((p',q',\sigma'),B'\big)$, denoted $\big((p,q,\sigma),B\big)\otimes\big((p',q',\sigma'),B'\big)$, is defined by
$$\big((p,q,\sigma),B\big)\otimes\big((p',q',\sigma'),B'\big):= \big((p,q,\sigma)\otimes (p',q',\sigma'),B\sqcup B'\big)\in\vec{\mathcal Br}(p+p, q+q'),$$
where we use the identification~\eqref{equ:identificationofintervalfortensor} to see $B$ and $B'$ as subsets of $[\![1,p+p']\!]\sqcup [\![{1},{q+q'}]\!]$.
\end{definition}

\subsubsection{Change of orientation for oriented Brauer diagrams}

\begin{definition}\label{def:changeoforientarionvecBr} Let $\big((p,q,\sigma),B\big)\in \vec{\mathcal Br}$. For $b\in B$, we define an oriented Brauer diagram  by 
$$\mathrm{co}_{\vec{\mathcal Br}}\big(\big((p,q,\sigma),B)\big),b\big):= \big((p,q,\sigma), (B\setminus \{b\})\cup \{\sigma(b)\}\big).$$ 
We say that $\mathrm{co}_{\vec{\mathcal Br}}\big(\big((p,q,\sigma),B)\big),b\big)$\index[notation]{co_Br@$\mathrm{co}_{\vec{\mathcal Br}}$} is obtained from $\big((p,q,\sigma),B\big)$ by \emph{change of orientation of} $b$.
\end{definition}

\subsubsection{Doubling operation for oriented Brauer diagrams}

\begin{definition}\label{def:doublingforvecBr} Let $\big((p,q,\sigma),B\big)\in \vec{\mathcal Br}$ and $A\subset \pi_0\big((p,q,\sigma),B\big)$.  Let $p':=p + |\sqcup_{\omega\in A}\omega\cap [\![1,p]\!]|$ and $q':=q + |\sqcup_{\omega\in A}\omega\cap [\![{1},{q}]\!]|$ and $\mathrm{pr}:[\![1,p']\!]\sqcup [\![{1},{q}']\!]\to [\![1,p]\!]\sqcup [\![1,q]\!]$ as in Definition~\ref{def:doublingforskeleta}$(a)$. The \emph{doubling of $\big((p,q,\sigma),B\big)$ along $A$}, denoted $\mathrm{dbl}_{\vec{\mathcal Br}}\big(\big((p,q,\sigma),B\big),A\big)$ or $\big((p,q,\sigma),B\big)^A$, \index[notation]{dbl_Br@$\mathrm{dbl}_{\vec{\mathcal Br}}$} is the oriented Brauer diagram defined by 
$$\mathrm{dbl}_{\vec{\mathcal Br}}\big(\big((p,q,\sigma),B\big),A\big) := \big(\mathrm{dbl}_{\mathcal Br}((p,q,\sigma),A), \mathrm{pr}^{-1}(B)\big)\in\vec{\mathcal Br}.$$
\end{definition}

\subsection{Properties of operations of oriented Brauer diagrams}\label{sec:1-7}

\begin{definition}
Let $\{+,-\}^*$\index[notation]{\{+,-\}^*@$\{+,-\}^*$} be the free associative monoid over $\{+,-\}$, i.e., the set of words in these letters. It can be 
identified with $\sqcup_{n\geq 0}\{+,-\}^n$, i.e., with the set of pairs $(n,f)$, where $n\in\mathbb{Z}_{\geq 0}$ and $f$ is a map $[\![1,n]\!]\to\{+,-\}$. For $w\in \{+,-\}^*$, we denote by $|w|$ its \emph{length}: $|w|=n\in\mathbb{Z}_{\geq 0}$ if and only if $w$ is identified with a pair $(n,f)$ as above. The empty word (corresponding to $n=0$) will be denoted $\emptyset$.
\end{definition}

\begin{definition}\label{def:sourcetargetorientedskeleta}
Define the \emph{source} and \emph{target} maps $\mathrm{s,t}: \vec{\mathcal Br}\longrightarrow \{+,-\}^*$
by  $\mathrm{s}\big((p,q,\sigma),B\big)^{-1}(+)=\sigma(B) \cap [\![1,p]\!]$  and $\mathrm{s}\big((p,q,\sigma),B\big)^{-1}(-) = B \cap [\![1,p]\!]$; and   $\mathrm{t}\big((p,q,\sigma),B\big)^{-1}(+)=B \cap [\![\bar{1},\bar{q}]\!]$   and $\mathrm{t}\big((p,q,\sigma),B\big)^{-1}(-)=\sigma(B) \cap [\![\bar{1},\bar{q}]\!]$.
\end{definition} 
 
 Notice that two oriented Brauer diagrams $a$ and $b$ are composable (according to Definition~\ref{def:compositioninOSk}) if an only if $\mathrm{t}(a)=\mathrm{s}(b)$.

\begin{definition} The monoidal category $\underline{\vec{\mathcal Br}}$\index[notation]{Br5@$\underline{\vec{\mathcal Br}}$} of oriented Brauer diagrams is defined as follows. The set of objects of $\underline{\vec{\mathcal Br}}$ is $\{+,-\}^*$. For~$u,v\in \{+,-\}^*$, the set of morphisms $\underline{\vec{\mathcal Br}}(u,v)$ is the set oriented Brauer diagrams $a$ such that $\mathrm{s}(a)=u$ and $\mathrm{t}(a)=v$.

For $u,v,w\in\{+,-\}^*$ the composition 
$$\underline{\vec{\mathcal Br}}(u,v)\times \underline{\vec{\mathcal Br}}(v,w) \longrightarrow \underline{\vec{\mathcal Br}}(v,w)$$
is given by $\circ_{\vec{\mathcal Br}}$ as in Definition~\ref{def:compositioninOSk}. The identity $\mathrm{Id}_u\in\underline{\vec{\mathcal Br}}(u,u)$ of $u\in\{+,-\}^*$ is the tuple $(\mathrm{Id}_{|u|}, B)$
where $\mathrm{Id}_{|u|}\in\underline{\mathcal Br}(|u|, |u|)$ is the identity of $|u|$ and $B= \{x\in[\![1,|u|]\!] \ | \ u(x)=-\}\sqcup \{{x}\in[\![{1},{|u|}]\!] \ | \ u(x)=+\}$.

At the level of objects, the tensor product is given by the concatenation of words:  $u\otimes v:=uv$ for any $u,v\in\{+,-\}^*$. For $u,u',v,v'\in\{+,-\}^*$, the tensor product map at the level of morphisms
$$\underline{\vec{\mathcal Br}}(u,v)\times \underline{\vec{\mathcal Br}}(u',v') \longrightarrow \underline{\vec{\mathcal Br}}(uv,u'v')$$
is as in Definition~\ref{def:tensorproductinOSk}. The unit object of $\underline{\vec{\mathcal Br}}$ is $0 \in \mathbb{Z}_{\geq 0}$ and its identity morphism is $\vec{\varnothing}$. 
\end{definition}

In particular, we have the oriented version of \eqref{eq:comptensorandcompositioninBr}:
\begin{equation}\label{eq:comptensorandcompositioninOBr}
(P'_1 \otimes P'_2) \circ (P_1 \otimes P_2)=(P'_1 \circ P_1) \otimes (P'_2 \circ P_2)
\end{equation}
for any $P_i,P_i'$  $(i=1,2)$ composable morphisms in $\underline{\vec{\mathcal Br}}$.

\begin{lemma} There is a canonical monoidal functor $\underline{\vec{\mathcal Br}} \to \underline{\mathcal Br}$ which takes $u\in\{+,-\}^*$ to its length $|u|\in\mathbb{Z}_{\geq 0}$ and  it takes a morphism $((p,q,\sigma),B)$ in $\underline{\vec{\mathcal Br}}$ to the morphism $(p,q,\sigma)$ in $\underline{\mathcal Br}$. 
\end{lemma}
\begin{proof}
The result follows from the definitions of the categories. 
\end{proof}

\begin{definition}\label{def:CirforvecBr} Let $u,v,w\in\{+,-\}^*$. Define the map $\mathrm{Cir}:\vec{\underline{\mathcal Br}}(u,v)\times \vec{\underline{\mathcal Br}}(v,w)\to \mathcal Sets$ as the composition
$$ \vec{\underline{\mathcal Br}}(u,v)\times \vec{\underline{\mathcal Br}}(v,w)\longrightarrow{\underline{\mathcal Br}}(u,v)\times {\underline{\mathcal Br}}(v,w)\xrightarrow{\mathrm{Cir}} \mathcal Sets,$$
where the last map was defined in Definition~\ref{def:defofSk}.
\end{definition}

Lemma~\ref{r:bijectioncomptensorandcompinBr} generalizes naturally for oriented Brauer diagrams.

\begin{lemma}\label{r:bijectioncomptensorandcompinOBr}  Let $P_i,P_i'$ composable morphisms in $\underline{\vec{\mathcal Br}}$ for $i=1,2$. Then there is a bijection 
$$\mathrm{bij}_{P'_1,P_1,P'_2,P_2} : \mathrm{Cir}(P'_1,P_1) \sqcup \mathrm{Cir}(P'_2,P_2)\longrightarrow \mathrm{Cir}(P'_1 \otimes P'_2,P_1 \otimes P_2).$$
\end{lemma}

\begin{proof} The result follows from Lemma~\ref{r:bijectioncomptensorandcompinBr}
\end{proof}

\begin{definition}\label{def:pi1:pi2:24022020}
For $w\in\{+,-\}^*$ and $A \subset [\![1,|w|]\!]$, define the word $\mathrm{dbl}_{\{+,-\}^*}(w,A)$ in $\{+,-\}^*$ (also denoted $w^A$)  \index[notation]{dbl_{+,-}@$\mathrm{dbl}_{\{+,-\}^*}$} by $\mathrm{dbl}_{\{+,-\}^*}(w,A):=(|w|+|A|,w \circ \pi_{|w|,A})$ where $\pi_{|w|,A} : [\![1,|w|+|A|]\!] \to [\![1,|w|]\!]$ is the only non-decreasing {surjective} map such that the fibers of the elements of $A$
(resp. $[\![1,|w|]\!]\setminus A$) have cardinality~$2$ (resp.~$1$). 
\end{definition}

\begin{lemma}\label{lemma:1:18:03032020} Let $P:=((p,q,\sigma),B) \in \vec{\mathcal Br}$ and $A \subset \pi_0(P)$ and $a,b\in B$. Then
\begin{itemize}
\item[$(a)$] $\mathrm{s}(\mathrm{dbl}_{\vec{\mathcal Br}}(P,A)) = \mathrm{dbl}_{\{+,-\}^*}(\mathrm{s}(P),\tilde{A} \cap [\![1,p]\!])$ and  $\mathrm{t}(\mathrm{dbl}_{\vec{\mathcal Br}}(P,A)) = \mathrm{dbl}_{\{+,-\}^*}(\mathrm{t}(P),\tilde{A} \cap [\![1,q]\!])$, where $\tilde{A}$ is the preimage of $A$ by the projection $[\![1,p]\!] \sqcup [\![1,q]\!]\!\to \pi_0(P)$.

\item[$(b)$] $\pi_0(\mathrm{co}_{\vec{\mathcal Br}}(P,a))\simeq \pi_0(P)$ and $\mathrm{co}_{\vec{\mathcal Br}}(\mathrm{co}_{\vec{\mathcal Br}}(P,b),a) = \mathrm{co}_{\vec{\mathcal Br}}(\mathrm{co}_{\vec{\mathcal Br}}(P,a),b)$ and $\mathrm{co}_{\vec{\mathcal Br}}(\mathrm{co}_{\vec{\mathcal Br}}(P,a),a)=P$. 
\item[$(c)$] for any $P'\in\vec{\mathcal Br}$ we have
$\mathrm{dbl}_{\vec{\mathcal Br}}(P\otimes P', A) = \mathrm{dbl}_{\vec{\mathcal Br}}(P, A)\otimes P'$, where in the left-hand  side  we identify $A$ with a subset of $\pi_0(P\otimes P')$, still denoted by $A$, via de inclusion $\pi_0(P)\subset \pi_0(P\otimes P')$.
\end{itemize} 
\end{lemma}

\begin{proof} Direct verification.
\end{proof}

\begin{lemma}\label{r:relPi0}
\begin{itemize}
\item[$(a)$]  For any $P,P' \in \vec{\mathcal Br}$, one has $\pi_0(P \otimes P') \simeq \pi_0(P) \sqcup \pi_0(P')$.
\item[$(b)$]  For any $P \in \vec{\mathcal Br}$ and $A \subset \pi_0(P)$, one has a natural map $\mathrm{proj}^{\pi_0}_{P,A} : \pi_0(\mathrm{dbl}_{\vec{\mathcal Br}}(P,A))\to \pi_0(P)$.
\item[$(c)$] For any composable pair $(P,P')$ of elements of $\vec{ \mathcal Br}$ there is a natural map 
$\mathrm{proj}^{\pi_0}_{P,P'} : \pi_0(P) \sqcup \pi_0(P') \to \pi_0(P' \circ P) \sqcup \mathrm{Cir}(P',P) $.
\end{itemize}
\end{lemma}

\begin{proof}
It is enough to show the statements for elements in $\mathcal Br$.
$(a)$ Let $P=(p,q,\sigma), P'=(p',q',\sigma')\in \mathcal Br$. By Definition~\ref{def:piceroforBr}, \ref{def:tensorproductinSk} and~\ref{def:concatenationofperm} we have
$$\pi_0(P\otimes P') = \dfrac{[\![1,p+p']\!]\sqcup [\![1,q+q']\!]}{\langle \sigma*\sigma'\rangle} \simeq \dfrac{[\![1,p]\!]\sqcup [\![1,q]\!]}{\langle \sigma\rangle}\sqcup \dfrac{[\![1,p']\!]\sqcup [\![1,q']\!]}{\langle \sigma'\rangle} =\pi_0(P)\sqcup \pi_0(P').$$

$(b)$ Let $P=(p,q,\sigma)\in\mathcal Br$ and  $A\subset \pi_0(P)$. Let $(p',q',\sigma')=\mathrm{dbl}_{\mathcal Br}((p,q,\sigma), A)$ and $\mathrm{pr}:[\![1,p']\!]\sqcup [\![1,q']\!]\to [\![1,p]\!]\sqcup [\![1,q]\!]$ as in Definition~\ref{def:doublingforvecBr}. By definition of $\sigma'$ the map $[\![1,p']\!]\sqcup [\![1,q']\!]\xrightarrow{\mathrm{pr}} [\![1,p]\!]\sqcup [\![1,q]\!]\xrightarrow{\text{proj}}\pi_0(P)$ factors to the quotient $\big([\![1,p']\!]\sqcup [\![1,q']\!])/\langle \sigma'\rangle$ giving the map $\mathrm{proj}^{\pi_0}_{P,A} : \pi_0(\mathrm{dbl}_{\vec{\mathcal Br}}(P,A))\to \pi_0(P)$.

$(c)$    Let $P=(p,q,\sigma), P'=(q,r,\sigma')\in \mathcal Br$ a pair of composable Brauer diagrams. The projection map 
$$\big([\![1,p]\!]\sqcup [\![1,q]\!]\big)\sqcup [\![1,q]\!]\sqcup [\![1,r]\!]\longrightarrow \dfrac{[\![1,p]\!]\sqcup [\![1,q]\!]\sqcup [\![1,r]\!]}{\langle \sigma*\mathrm{Id}_{[\![1,r]\!]}, \mathrm{Id}_{[\![1,p]\!]}*\sigma'\rangle} \simeq \pi_0(P'\circ P)\sqcup \mathrm{Cir}(P',P),$$
see \eqref{eq:toto2} for the last identification, factorizes in the map 
$\mathrm{proj}^{\pi_0}_{P,P'} : \pi_0(P) \sqcup \pi_0(P') \to \pi_0(P' \circ P) \sqcup \mathrm{Cir}(P',P) $.
\end{proof}

\begin{definition} Let $\pm:\{+,-\}\to \{+,-\}$ be the non-identity involution, i.e., $\pm(+)=-$ and $\pm(-)=+$.
\end{definition}

\begin{remark} Let $\mathfrak{B}$ be the set of tuples $(u,v,\sigma)$ where $u,v\in\{+,-\}^*$  and  $\sigma\in \mathrm{FPFI}\big([\![1,|u|]\!]\sqcup [\![1,|v|]\!]\big)$ such that 
\begin{itemize} 
\item $u(\sigma(x))=\pm (u(x))$ if $x,\sigma(x)\in[\![1,|u|]\!]$,
\item $v(\sigma(x))=u(x)$ if $x\in[\![1,|u|]\!]$ and $\sigma(x)\in[\![1,|v|]\!]$,
\item $v(\sigma(x))=\pm (v(x))$ if $x,\sigma(x)\in[\![1,|v|]\!]$,

\item $u(\sigma(x))=v(x)$ if $x\in[\![1,|v|]\!]$ and $\sigma(x)\in[\![1,|u|]\!]$.
\end{itemize}
There is a unique bijection $\vec{\mathcal Br} \to \mathfrak{B}$ given by  $\big((p,q,\sigma),B\big)\mapsto\big(\mathrm{s}\big((p,q,\sigma),B\big),\mathrm{t}\big((p,q,\sigma),B\big), \sigma\big)$ for any $\big((p,q,\sigma),B\big)\in\vec{\mathcal Br}$,  where $\mathrm{s}$ and $\mathrm{t}$ are as in Definition~\ref{def:sourcetargetorientedskeleta}. Indeed, one checks that the map $(u,v,\sigma)\mapsto \big((|u|, |v|, \sigma), B\big)$ where $B:=\{x\in[\![1,|u|]\!] \ | \  u(x)=-\}\sqcup \{x\in[\![1,|v|]\!] \ | \  v(x)=+\}$ defines an inverse.
\end{remark}

\subsection{Schematic representation of Brauer diagrams, oriented Brauer diagrams and their operations}\label{sec:1-8}\label{sec:2-8}

We represent a Brauer diagram $(p,q,\sigma)\in \mathcal Br$  in the square $[-1,1]^2$ as follows. Consider $p$ points labelled by $1,\ldots, p$   uniformly distributed along $[-1,1]\times \{0\}$. Similarly consider~$q$ points labelled $\bar{1},\ldots, \bar{q}$  uniformly distributed along $[-1,1]\times \{1\}$. For any $i\in [\![1,p]\!]\sqcup [\![{1},{q}]\!]$ we join $i$ with $\sigma(i)$ using a path $\overline{i\sigma(i)} = \overline{\sigma(i)i}$ contained in $[-1,1]^2$ and such that its intersection with $\partial [-1,1]^2$ consists only of $i$ and $\sigma(i)$. We suppose that $\bigcap_{i}\overline{i\sigma(i)}$ is a finite set and that all the intersections are transversal.  We consider such a diagram up to homotopy. Reciprocally, such a diagram gives rise to a Brauer diagram. See Figure~\ref{figureEx1_14}$(a)$ and $(b)$ for some examples. 
\begin{figure}[ht!]
										\centering
                        \includegraphics[scale=0.8]{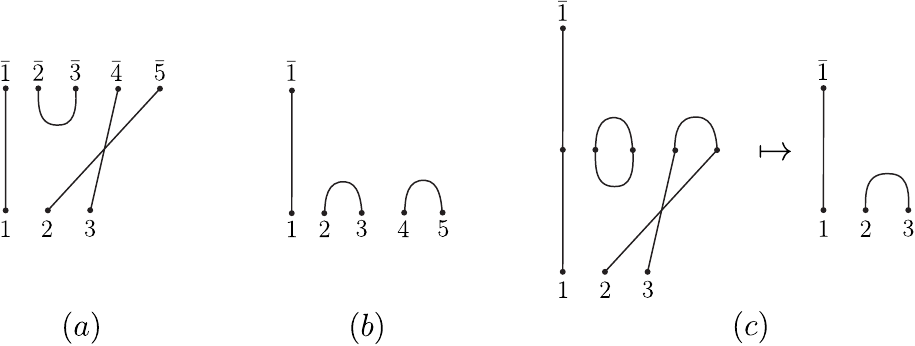}										
\caption{$(a)$ Schematic representation of the Brauer diagram $(3,5,\sigma)$ where $\sigma\in\mathrm{FPFI}\big([\![1,3]\!]\sqcup [\![{1},{5}]\!]\big)$ is given by $\sigma(1)=\bar{1}$, $\sigma(2)=\bar{5}$, $\sigma(3)=\bar{4}$, $\sigma(\bar{2})=\bar{3}$. $(b)$  Schematic representation of the Brauer diagram $(5,1,\tau)$ where $\tau\in\mathrm{FPFI}\big([\![1,5]\!]\sqcup [\![{1},{1}]\!]\big)$ is given by $\tau(1)=\bar{1}$, $\tau(2)=3$, $\tau(4)=5$.   $(c)$  Schematic  representation of the composed Brauer diagram $(5,1,\tau)\circ_{\mathcal Br} (3,5,\sigma)=(3,1,\rho)$ where $\rho\in\mathrm{FPFI}\big([\![1,3]\!]\sqcup [\![{1},{1}]\!]\big)$ is given by $\rho(1)=\bar{1}$ and $\rho(2)=3$.}
\label{figureEx1_14} 								
\end{figure}

In the case of an oriented Brauer diagram $((p,q,\sigma), B)$ we orient the path $\overline{b\sigma(b)}$ from $b$ to $\sigma(b)$ for~$b\in B$. 

For an oriented Brauer diagram $a:=((p,q,\sigma), B)$ we can read the source and target words in $\{+,-\}^*$ in its schematic representation using the following convention: for $i\in[\![1,p]\!]$ we have that $\mathrm{s}(a)(i)=+$  if the orientation of the schematic representation points toward $i$ and $\mathrm{s}(a)(i)=-$ otherwise; and for $i\in[\![{1},{q}]\!]$ we have  $\mathrm{t}(a)(\bar{i})=-$  if the orientation of the schematic representation points towards $\bar{i}$ and $\mathrm{t}(a)(\bar{i})=+$ otherwise. See Figure~\ref{figureEx1_14b} $(a)$ and $(b)$ for some examples.
\begin{figure}[ht!]
										\centering
                        \includegraphics[scale=0.8]{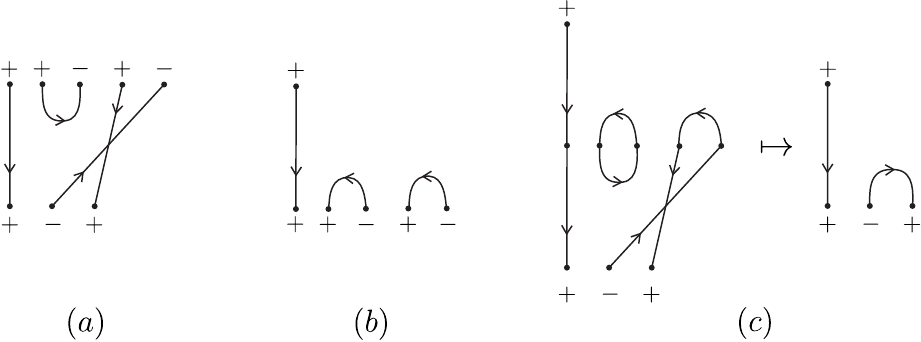}										
\caption{$(a)$ Schematic representation of the oriented Brauer diagram $((3,5,\sigma), B)$ with $\sigma(1)=\bar{1}$, $\sigma(2)=\bar{5}$, $\sigma(3)=\bar{4}$, $\sigma(\bar{2})=\bar{3}$ and $B=\{\bar{1},\bar{2}, \bar{4},2\}$. $(b)$ Schematic representation of the oriented Brauer diagram  $((5,1,\tau), B')$ where $\tau\in\mathrm{FPFI}\big([\![1,5]\!]\sqcup [\![{1},{1}]\!]\big)$ is given by $\tau(1)=\bar{1}$, $\tau(2)=3$, $\tau(4)=5$ and $B'=\{\bar{1},3, 5\}$.   $(c)$  Schematic representation of the composed oriented Brauer diagram $((5,1,\tau),B')\circ_{\vec{\mathcal Br}} ((3,5,\sigma),B)=((3,1,\rho), B'')$ where $\rho\in\mathrm{FPFI}\big([\![1,3]\!]\sqcup [\![{1},{1}]\!]\big)$ is given by $\rho(1)=\bar{1}$ and $\rho(2)=3$ and $B''=\{\bar{1},2\}$.}
\label{figureEx1_14b} 								
\end{figure}

The tensor product is described schematically by horizontal concatenation of the schematic representations and horizontal rescaling. The composition is described schematically  by vertical concatenation of the schematic representations (read from bottom to top), vertical rescaling and deletion of the circle components, see Figure~\ref{figureEx1_14}$(c)$ Figure~\ref{figureEx1_14b}$(c)$ for examples of composition in $\underline{\mathcal Br}$ and $\underline{\vec{\mathcal Br}}$, respectively.

The canonical functor $\underline{\vec{\mathcal Br}}\to \underline{\mathcal Br}$ is described schematically by forgetting the orientation of the paths. For instance, the image under this functor of the oriented Brauer diagrams in Figure~\ref{figureEx1_14b} is the Brauer diagrams of Figure~\ref{figureEx1_14}.

The Brauer diagram $\mathrm{dbl}_{\vec{\mathcal Br}}\big(\big((p,q,\sigma),B\big),A\big)$ obtained by a doubling operation   is described schematically by the parallel doubling of the paths with endpoints in $A$. We represent the paths in $A$ by using thickened lines. The orientation of the new paths after the doubling is compatible with that of the paths before the doubling. In Figure~\ref{figureEx1_19}$(a)$ and $(b)$ we shown an example of doubling for an oriented Brauer diagram.

\begin{figure}[ht!]
										\centering
                        \includegraphics[scale=0.8]{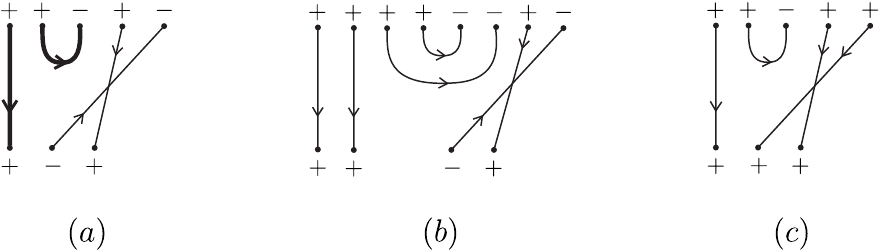}
												
\caption{$(a)$ Schematic representation of the oriented Brauer diagram $((3,5,\sigma), B)$ with $\sigma(1)=\bar{1}$, $\sigma(2)=\bar{5}$, $\sigma(3)=\bar{4}$, $\sigma(\bar{2})=\bar{3}$ and $B=\{\bar{1},\bar{2}, \bar{4},2\}$ and we use thickened lines to represent $A=\{\{1,\bar{1}\}, \{\bar{2},\bar{3}\}\}\subset \pi_0((3,5,\sigma), B)_{\mathrm{seg}}$. $(b)$ Schematic representation of the oriented Brauer diagram $\mathrm{dbl}_{\vec{\mathcal Br}}\big(\big((3,5,\sigma), B\big),A\big)$. $(c)$~schematic representation of $\mathrm{co}_{\vec{\mathcal Br}}\big(\big((3,5,\sigma), B\big),2\big)$.}
\label{figureEx1_19} 								
\end{figure}

The change of orientation operation is described schematically by changing the orientation of the corresponding path, see Figure~\ref{figureEx1_19}$(a)$ and $(c)$ for an example.

\section{Links, tangles, KII map, Kirby bialgebras, pre-LMO structures and invariants of 3-manifolds}\label{sec::2}

We start this section by recalling the description  of 3-manifolds based on surgery along framed oriented links in $S^3$ and the Lickorish-Wallace and Kirby theorems (\S\ref{sect:1:2jan2020}). In \S\ref{sec:3.1} we introduce basic definitions about the set $\vec{\mathcal{T}}$ of framed oriented tangles, its categorical interpretation and its relation with oriented Brauer diagrams. In \S\ref{sect:4:2jan2020} we define  an orientation-reversal operation and the set $\mathrm{CO}\vec{\mathcal{T}}=\sqcup_{k\geq 0}\mathrm{CO}\vec{\mathcal{T}}_k$.  In \S\ref{sect:3:2jan2020} we introduce the doubling map for oriented framed tangles. In \S\ref{sec:new2.5} we define the KII map and the set $\mathrm{KII}\vec{\mathcal T}$. In \S\ref{sec:3.5} we give an alternative description of the set $\mathrm{KII}\vec{\mathcal{T}}$ as well as a description of it using the category of parenthesized framed oriented tangles $\mathcal Pa\vec{\mathcal{T}}$.  In \S\ref{sec:3.7} we use the sets $\mathrm{CO}\vec{\mathcal T}$ and $\mathrm{KII}\vec{\mathcal{T}}$ to define the Kirby monoid which allows to describe the monoid of $3$-manifolds by reinterpretating the Lickorish-Wallace and Kirby theorems. We  introduce the notions of Kirby bialgebra, (semi)Kirby structures and their relation with the construction of invariants of $3$-manifolds. Finally in \S\ref{sec:3.8} we define the notion of pre-LMO structure and we show that it gives rise to a semi-Kirby structure.

\subsection{Presentation of \texorpdfstring{$3$}{3}-manifolds based on links in \texorpdfstring{$S^3$}{S}}\label{sect:1:2jan2020}

\subsubsection{Framed oriented links in \texorpdfstring{$S^3$}{S}}

\begin{definition} 
\begin{itemize}
\item[$(a)$]  An \emph{oriented link} of $l$ components  in $S^3$ is an embedding of $l$ disjoint copies of $S^1$ in $S^3$, 
$\coprod_{i=1}^l{S^1}\hookrightarrow S^3$. We denote the image of the embedding by $L=K_1\cup\cdots\cup K_l$. See 
Figure \ref{fig:example:link}.
\begin{figure}[H]
\centering
\includegraphics[width=35mm]{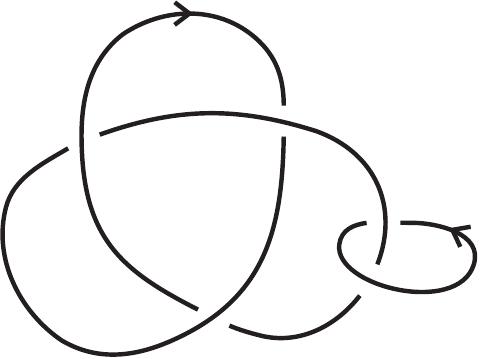}
\caption{Oriented link of two components.}\label{fig:example:link}
\end{figure}
\item[$(b)$] A \emph{framed oriented  link} $L=(L,f)$ is an oriented link $L=K_1\cup\cdots\cup K_l$ with a  \emph{framing} $f=(f_1,\ldots, f_l)$, where $f_i$ is the isotopy class of a section of the projection $\partial N(K_i)\rightarrow K_i$, where $N(K_i)$ is a tubular neighborhood of $K_i$. We orient $f_i$ by using the orientation of $K_i$.
See Figure~\ref{fig:example:framed:link}. 
\begin{figure}[H]
\centering
\includegraphics[width=100mm]{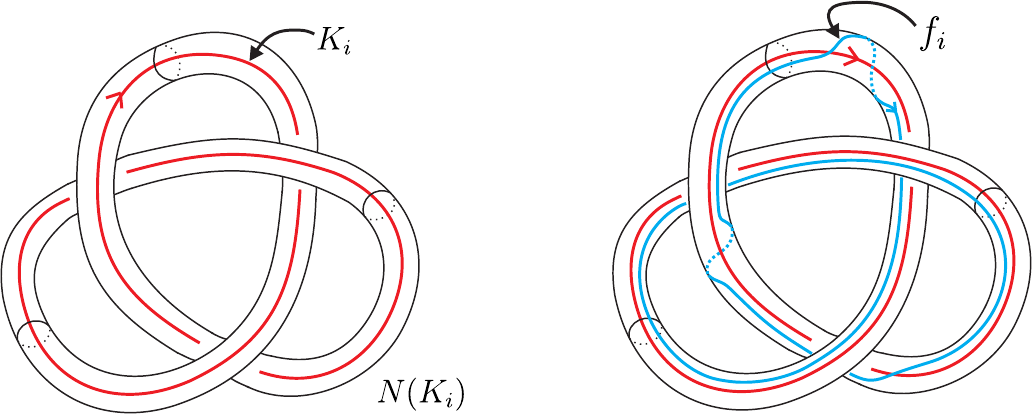}
\caption{Tubular neighborhood and framing of the component $K_i$.}\label{fig:example:framed:link}
\end{figure}
\end{itemize}
\end{definition}

We are interested in framed oriented links up to isotopy. We use the following convention.
\begin{convention}\label{conventionlinks1} If $L$ is a framed oriented link, we denote by $\mathbf{L}$ its class modulo isotopy. Reciprocally, if $\mathbf{L}$ is the isotopy class of a framed oriented link, we write $L$  to denote one of its representatives.
\end{convention}

\subsubsection{Diagrammatic description of framed oriented  links}\label{sect:112}

An {\it oriented link diagram} is a planar diagram which may be obtained as a regular projection of an oriented link. Such a diagram gives rise to a framed oriented  link by the convention that the framing vector at a given point is contained in the plane of the figure, perpendicular to 
the tangent vector at this point, and such that the pair {(tangent vector, framing vector) gives the counterclockwise orientation of the plane.}  

By the Reidemeister theorem (see \cite{BuZie}), (oriented) links modulo isotopy can be coded by (oriented) link diagrams modulo planar isotopy 
and moves denoted RI, RII and RIII. It then follows (see for instance \cite[Thm~1.8]{Ohts}) that framed oriented links modulo isotopy are coded by 
planar diagrams modulo planar isotopy and moves RI', RII and RIII shown in Figure \ref{fig:reid}. 
\begin{figure}[H]
\centering
\includegraphics[scale=0.65]{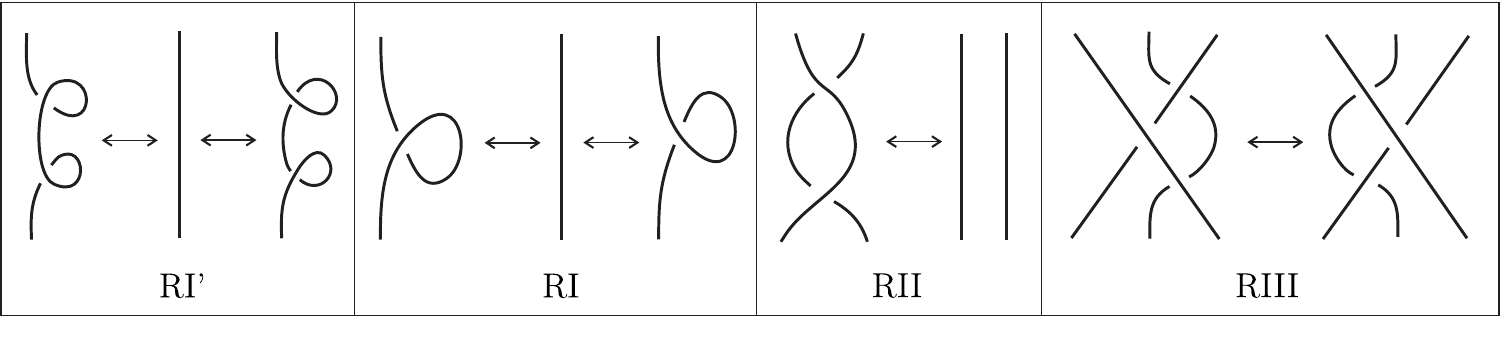}
\caption{Reidemeister moves. The orientation is arbitrary.}\label{fig:reid}
\end{figure}

Summarizing, one gets  a commutative diagram 
$$
\xymatrix{
\dfrac{\{\text{oriented link diagrams}\}}{\text{planar isotopy, RI', RII, RIII}}\ar[d]\ar^-{\sim}[r]&\dfrac{\{\text{framed oriented links in }S^3\}}{\text{isotopy}}\ar[d] \\ 
\dfrac{\{\text{oriented link diagrams}\}}{\text{planar isotopy, RI, RII, RIII}}\ar^-{\sim}[r]& 
\dfrac{\{\text{oriented links in }S^3\}}{\text{isotopy}}. 
} 
$$
The right vertical map decomposes as a disjoint union according to the number $l$ of connected components. For each $l\geq1$, the group 
$\mathbb Z^l$ acts simply on the source of this map and the quotient identifies with its target.

\subsubsection{Linking number}\label{sec:linkingnumber}

Let $K$ and $K'$ be oriented knots with empty intersection, define the \emph{linking number} $\mathrm{Lk}(K,K')\in \mathbb{Z}$\index[notation]{Lk@$\mathrm{Lk}(K,K')$}  of $K$ and $K'$ as follows. Consider a diagram of $K\sqcup K'$, denote by $\mathrm{cr}(K,K')$ the set of crossings of the diagram. For each $c\in\mathrm{cr}(K,K')$ associate a sign $\mathrm{sgn}(c)\in\{+1,-1\}$ according to the rule shown in Figure~\ref{fig:linkingnumber}.
\begin{figure}[H]
\centering
\includegraphics[scale=0.65]{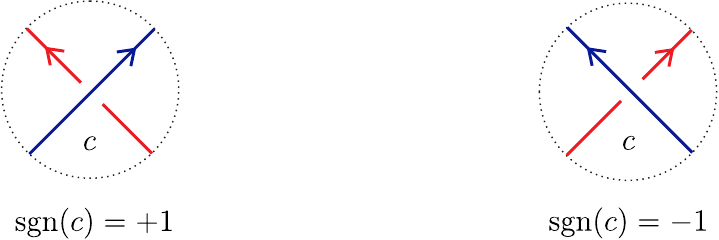}
\caption{Sign of crossings.}\label{fig:linkingnumber}
\end{figure}

Set $$\mathrm{Lk}(K,K')= \frac{1}{2}\sum_{c\in\mathrm{cr}(K,K')}\mathrm{sgn}(c).$$

It can be checked easily that this number is well-defined, i.e., it does not depend on the diagram of $K\sqcup K'$. By definition $\mathrm{Lk}(K,K')= \mathrm{Lk}(K',K)$. For more details and for several equivalent definitions of the linking number we refer to~\cite[Ch.~5]{Rolf}.

For $L$ a framed oriented link, $\pi_0(L)$ is the set of its connected components. For $K \in \pi_0(L)$, define $\mathrm{Lk}(K):=\mathrm{Lk}(K,f_K)$ where $f_K$ is the oriented knot obtaining from $K$ by pushing it along its framing. Then a symmetric matrix $\mathrm{Lk}(L)$ indexed by $\pi_0(L)$ and with integer coefficients is defined by $\mathrm{Lk}(L)_{K,K'}:=\mathrm{Lk}(K,K')$ and $\mathrm{Lk}(L)_{K,K}:=\mathrm{Lk}(K)$.  We call $\mathrm{Lk}(L)$ the \emph{linking matrix} of $L$. This matrix is an isotopy invariant of the framed oriented link $L$, therefore we can also write $\mathrm{Lk}(\mathbf{L})$.

\subsubsection{The surgery map}

Let us recall the \emph{Dehn surgery map} 
\begin{equation}\label{dehn:surg:map}
\begin{array}{rcl}
\dfrac{\{\text{framed oriented  links in }S^3\}}{\text{isotopy}} & \longrightarrow &\dfrac{\{\text{closed connected oriented 3-manifolds}\}}{\text{orientation-preserving homeomorphism}} \\
& & \\
L=(L,f)&\longmapsto&S^3_L
\end{array}
\end{equation}

\noindent (see e.g.  \cite[Ch. 12]{LickBook} and \cite[Ch. 9]{Rolf}). 
Start with a framed oriented  link $L=(L,f)=(K_1\cup\cdots\cup K_l,(f_1,\ldots, f_l))$ in $S^3$. Let $N(L)=\cup_{i=1}^l N(K_i)$ be 
a tubular neighborhood of $L$. 
For $i\in[\![1,l]\!]$, we equip $\partial N(K_i)$ with the orientation induced from that of $S^3\setminus \mathrm{int}(N(K_i))$ by using the convention 
of the first normal vector\footnote{\raggedright For $p\in\partial N(K_i)$, the orientation of $S^3$ is a choice of an element of 
$\pi_0\mathrm{Iso}(\mathbb R,\Lambda^3(T_p S^3))$, therefore in $\pi_0\mathrm{Iso}(T_p\partial N(K_i)^\perp,
\Lambda^2(T_p\partial N(K_i)))$, where $T_p\partial N(K_i)^\perp\subset T_p^* S^3$. An element in 
$\pi_0\mathrm{Iso}(\mathbb R,T_p\partial N(K_i)^\perp)$ is chosen by taking~1 to a linear form which is positive on vectors pointing 
outside $S^3\setminus N(K_i)$ (so inside $N(K_i)$). By composition, this gives an element in $\pi_0\mathrm{Iso}(\mathbb R,\Lambda^2(T_p\partial N(K_i)))$, therefore an orientation of $\partial N(K_i)$.}. We identify $f_i$ with a map $S^1\to\partial N(K_i)$. We also denote by $m_i:S^1\to\partial N(K_i)$ the meridian map, oriented in such a way that (tangent vector of $f_i$, tangent vector of $m_i$) gives the opposite orientation of $\partial N(K_i)$. 

Recall that $\mathbb D^2\times S^1$ has boundary $ S^1\times S^1$. Let $\ell,m:
 S^1\to S^1\times  S^1$ the maps $x\mapsto (x,*)$ and $x\mapsto(*,x)$, where $*$ is a fixed point in $ S^1$. 
For $i\in[\![1,l]\!]$, there is a unique (up to isotopy) homeomorphism
$$
h_i : \partial(\mathbb D^2\times S^1)\simeq S^1\times S^1 { \longrightarrow } \partial N(K_i), 
$$
such that $h_i\circ \ell=m_i$ and $h_i\circ m=f_i$ (see Figure \ref{fig:att:map}). One then sets 
$$
 S^3_L:=\big( S^3 \setminus \cup_{i=1}^l \mathrm{int}(N(K_i))\big)\bigcup_{(h_1,\ldots,h_l)}
\big(\sqcup_{i=1}^l(\mathbb D^2\times S^1)\big). 
$$
\begin{figure}[ht]
\centering
\includegraphics[width=110mm]{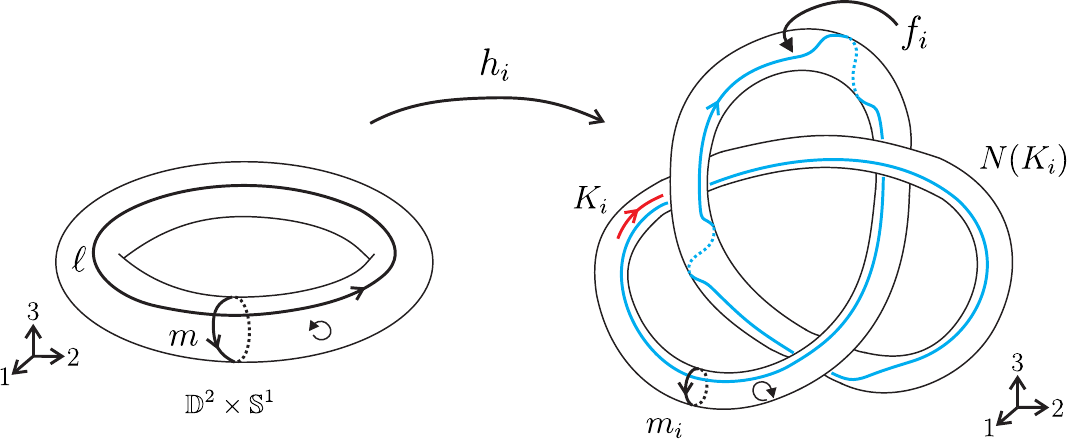}
\caption{Attaching map.}\label{fig:att:map}
\end{figure}

The obtained compact oriented $3$-manifold $S^3_L$ does not depend on the isotopy class of $L$, therefore we can also write $S^3_{\mathbf{L}}$.

\subsubsection{Surjectivity and fibers of the surgery map}

According to the Lickorish-Wallace theorem (\cite{Lick,Wal}), the Dehn surgery map (\ref{dehn:surg:map}) is surjective. 

By contrast, this map is far from being injective. The following theorem allows one to relate the different links giving rise to the same 3-manifold.

\begin{theorem}[\cite{Kir}] \label{thm:kirby}
 Let ${L}$ and ${L}'$ be framed oriented  links in $S^3$. The 3-manifolds $S_{{L}}^3$ and $S_{{L}'}^3$ are homeomorphic if and only if ${L}$ and ${L}'$ are related by a finite sequence of the following moves:

0) Change of orientation of some components.  

1)  \emph{Kirby I (KI) move}. Starting with a framed oriented link $L$, perform a disjoint union with the $\pm1$-framed unknot (see Figure \ref{figuresunknots}). \index[notation]{U^\pm@$U^\pm$}

2)  \emph{Kirby I' (KI') move}. Inverse operation to KI move.  

3)  \emph{Kirby II (KII) move}. Starting with a framed oriented link $L$, select two different components~$L_a$ and~$L_b$. Take a parallel 
copy of~$L_b$ following its framing and label such copy by~$L_{b'}$. Then, choose a way of performing an orientation-compatible band 
sum  between $L_{b'}$ and $L_a$. This results in a new framed oriented link $L'$ (some examples are shown in Figure~\ref{figuraKI10_ell}). 
\end{theorem}

\begin{figure}[ht]
										\centering
                        \includegraphics[scale=0.8]{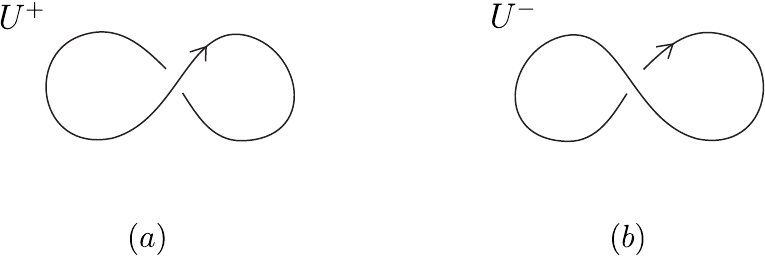}
\caption{The $\pm1$-framed unknots $U^\pm$.}
\label{figuresunknots} 								
\end{figure}

\begin{remark} One checks that $U^+$ and $U^-$ are invariant under the operation of change of orientation.
\end{remark} 

\begin{figure}[ht]
										\centering
                        \includegraphics[scale=0.9]{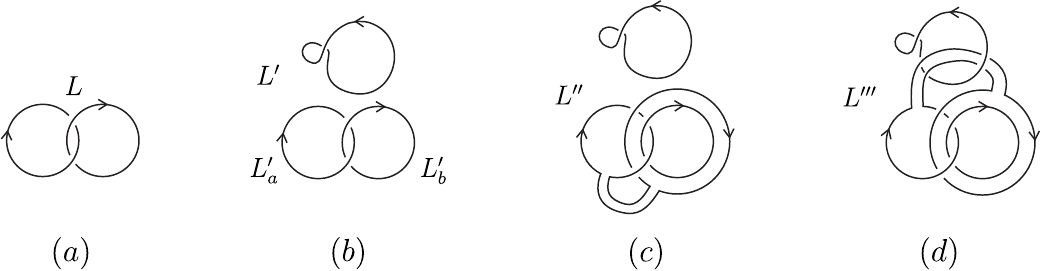}
												
\caption{The link $L'$ is obtained from $L$ by a KI move using a $+1$-framed unknot. Both $L''$ and $L'''$ are obtained from $L'$ by a KII move. These moves are based on different ways of performing band sums, which results in $L''$ and $L'''$ being non-isotopic.}
\label{figuraKI10_ell} 								
\end{figure}

Notice that Theorem~\ref{thm:kirby} can be stated replacing framed oriented links by isotopy classes of framed oriented links.

\begin{remark} By the change of orientation move we can ignore the orientation of the components of a surgery presentation of a 3-manifold.
\end{remark}
\begin{example}
Theorem \ref{thm:kirby} can be used to show combinatorially that the 3-manifolds $S_{H}^3$ and 
$S^3$ are homeomorphic, where $H$ is the $0$-framed Hopf link. That is, by means of Kirby moves we can go from the link $H$ to the empty link. See Figure \ref{figure:Example-Kirby-Calculus}.
\begin{figure}[ht]
										\centering
                        \includegraphics[scale=0.9]{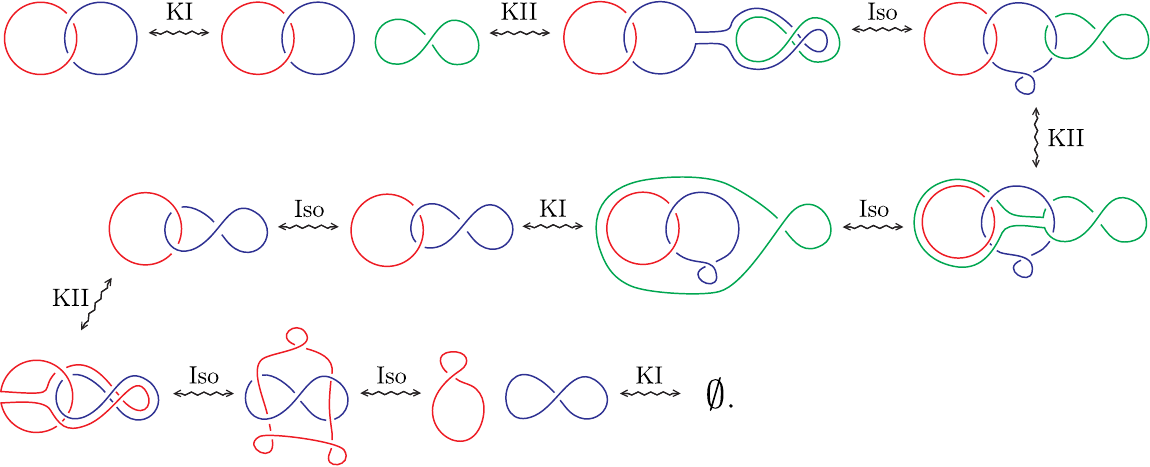}
\caption{Combinatorial proof of the homeomorphism between $S_{H}^3$ and $S^3$. In this figure by ``Iso'' we mean planar isotopies or Reidemeister moves.}		
\label{figure:Example-Kirby-Calculus}						
\end{figure}
\end{example}

\begin{definition}\label{def:disjointandconnectedcum}
$(a)$ Let  $L$ and $L'$ be framed oriented links in $S^3$. We define the \emph{disjoint union} of $L$ and $L'$, denoted by $L\dot{\sqcup} L'$, as the framed oriented link in $S^3$ consisting of $L$ and $L'$ and such that there exists an embedded $3$-ball $B\subset S^3$ with $L\in\mathrm{int}(B)$ and $L'\cap B=\emptyset$. This operation does not depend on the isotopy classes of $L$ and $L'$, therefore we can also write $\mathbf{L}\dot{\sqcup}\mathbf{L}'$.

$(b)$ Let $Y$ and $Y'$ be closed connected oriented $3$-manifolds.  We define the \emph{connected sum} of $Y$ and $Y'$,   denoted by $Y\# Y'$, as the closed connected oriented $3$-manifold $(Y\setminus \mathrm{int}(B))\sqcup_{\partial B\sim \partial B'} (Y'\setminus \mathrm{int}(B'))$ obtained from $Y$ and $Y'$ by removing the interior of embedded $3$-balls $B\subset Y$, $B'\subset Y'$ and identifying the corresponding boundaries with an orientation-compatible gluing map.
\end{definition}

The following is straightforward, see for instance \cite{Hat:notes, LickBook}.

\begin{lemma}\label{def:monoid3man} 
$(a)$ The disjoint union operation $\dot{\sqcup}$ endows the set of isotopy classes of framed oriented links in $S^3$  with a commutative monoid structure with the (isotopy class of the) empty link as unit element. 

$(b)$ The set   $3\text{-}\mathrm{Mfds}$\index[notation]{3Mfds@$3\text{-}\mathrm{Mfds}$}  of homeomorphism classes of closed connected oriented $3$-manifolds has a structure of commutative monoid with product $\#$ given by the connected sum and  unit element the homeomorphism class of~$S^3$.

$(c)$  The surgery map~\eqref{dehn:surg:map} defines a monoid homomorphism
\begin{equation}\label{eq:monoidmorphismlinksto3man}
\left(\dfrac{\{\text{framed oriented links in } S^3\}}{\text{isotopy}}, \dot{\sqcup}\right) \longrightarrow (3\text{-}\mathrm{Mfds}, \# ),
\end{equation}
that is,  if $Y=S^3_L$ and $Y'=S^3_{L'}$ for some framed oriented links $L$ and $L'$ in $S^3$, then $Y\#Y' = S^3_{L\dot{\sqcup} L'}$. More formally, $[Y\#Y']=S^3_{\mathbf{L}\dot{\sqcup}\mathbf{L}'}$, where  $[Y\#Y']$ denotes the homeomorphism class of $Y\#Y'$.
\end{lemma}

\subsection{Framed oriented tangles: basic definitions}\label{sec:3.1}

In this subsection we introduce the basic terminology about framed oriented tangles as well as some of their operations.

\subsubsection{Framed oriented tangles: first operations} 

\begin{definition}\label{def:framedtangle}

$(a)$  A \emph{tangle} $T$ is  a compact $1$-manifold  properly embedded in the cube $[-1,1]^3\subset\mathbb{R}^3$, 
such that the boundary $\partial T$ (the \emph{endpoints} of $T$) is uniformly distributed along $\{0\}\times (-1,1)\times\{\pm 1\}$.

$(b)$  A tangle $T$ is \emph{oriented} if each of its components is endowed with an orientation.

$(c)$ A \emph{framed tangle} is a tangle $T$ such that each component is endowed with a framing which takes the value $(0,1,0)$   at each endpoint $t\in\partial T$.

$(d)$  A \emph{framed oriented tangle} is an oriented  tangle $T$ such that each component is endowed with a framing which takes the value $\pm(0,1,0)$,  the sign being determined  as follows: if $v$ is the unit tangent vector  and $n_v$ is the unit normal vector at an endpoint, then the tuple $(n_v,v,(1,0,0))$ is a positive frame for $\mathbb{R}^3$.

 We consider framed (oriented) tangles up to isotopy. We denote by $\mathcal{T}$\index[notation]{T1@$\mathcal{T}$} and (resp. $\vec{\mathcal{T}}$\index[notation]{T2@$\vec{\mathcal{T}}$}) the set of framed (resp. framed oriented) tangles  up to isotopy.

Notice that there is a canonical map $\vec{\mathcal T}\to \mathcal T$ which forgets the orientation. 
\end{definition}

\begin{convention}\label{conventionfortangles} If $T$ is a framed (oriented) tangle, we denote by $\mathbf{T}$ its class modulo isotopy.  Similarly, if $\mathbf{T}$ is the isotopy class of a framed (oriented) tangle, we denote by $T$ one of its representatives.
\end{convention}

\begin{definition}  Let $T$ be a framed  tangle.
\begin{itemize}
\item[$(a)$] Denote by $\partial_b T$ (resp. $\partial_t T$) the subset of $\partial T$ contained in $\{0\}\times [-1,1]
\times\{-1\}$ (resp. in $\{0\}\times [-1,1]\times\{+1\}$) and call it the \emph{bottom} (resp. \emph{top}) 
boundary of $T$. The embedding of $\partial T$ into $\{0\}\times [-1,1]\times \{\pm 1\}$ allows to equip $\partial_bT$ and $\partial_tT$ with linear orders. We use these orders to  identify the elements in $\partial_bT$ with the set $[\![1,|\partial_b T|]\!]$ and the elements in $\partial_tT$ with the set $[\![{1},{|\partial_t T|}]\!]$.

\item[$(b)$] Define $\sigma_T\in\mathrm{FPFI}\big([\![1,|\partial_b T|]\!]\sqcup [\![{1},{|\partial_t T|}]\!]\big)$ by $\sigma(a)=b$ if and only if there exists $C\in\pi_0(T)$ such that $\partial C=\{a,b\}$. 

\item[$(c)$] If moreover $T$ is oriented, we define $B_T\subset [\![1,|\partial_b T|]\!]\sqcup [\![{1},{|\partial_t T|}]\!]$  to be the image  under the identifications $\partial_b T\simeq [\![{1},{|\partial_b T|}]\!]$ and  $\partial_t T\simeq [\![{1},{|\partial_t T|}]\!]$ of the beginning points of $\partial T$, i.e., the points where the orientation of $T$ is outgoing.
\end{itemize}
\end{definition}

\begin{definition}\label{definitionsigmaT}
\begin{itemize}
\item[$(a)$] Let $T$ be a framed tangle. Define \emph{the underlying  Brauer diagram of $T$}, denoted~$\mathrm{br}(T)$, by
$\mathrm{br}(T):=(|\partial_b T|, |\partial_t T|, \sigma_T)\in \mathcal Br$.
\item[$(b)$] Let $T$ be a framed oriented tangle. Define \emph{the underlying  oriented Brauer diagram of $T$}, denoted ~$\vec{\mathrm{br}}(T)$, by $ \vec{\mathrm{br}}(T):=\big((|\partial_b T|, |\partial_t T|, \sigma_T), B_T\big)\in\vec{\mathcal Br}$.
\end{itemize}
\end{definition}

\begin{lemma}\label{sourceandtargetfortangles}
There are maps $\mathrm{br}:\mathcal{T}\to \mathcal Br$\index[notation]{br@$\mathrm{br}$} and $\vec{\mathrm{br}}:\vec{\mathcal{T}}\to \vec{\mathcal Br}$\index[notation]{br2@$\vec{\mathrm{br}}$} taking $\mathbf{T}\in \mathcal{T}$ to $(|\partial_bT|,|\partial_tT|,\sigma_T)$ and  $\mathbf{T}\in \vec{\mathcal{T}}$ to $\big((|\partial_bT|,|\partial_tT|,\sigma_T), B_T\big)$, where $T$ is a representative of $\mathbf{T}$. We also denote by $\mathrm{br}$ the composite map $\vec{\mathcal T}\to \mathcal T\xrightarrow{\mathrm{br}}\mathcal Br$. 
\end{lemma}

\begin{proof}
For a framed tangle $T$, the involution $\sigma_T$ is invariant under isotopy. Moreover if $T$ is oriented, the set $B_T$ is also invariant under isotopy. That is, the operations $\mathrm{br}$ and $\vec{\mathrm{br}}$  are invariant under isotopy.
\end{proof}

\begin{notation} Let $\mathbf{T}\in\vec{\mathcal T}$. We write $\sigma_{\mathbf{T}}$ for the involution $\sigma_T$ where $T$ is a representative of $\mathbf{T}$.
\end{notation}

\begin{definition}\label{sourceandtargetfortangleswords} Define  maps $\mathrm{s},\mathrm{t}:\vec{\mathcal T}{\to}\{+,-\}^*$  as the compositions $\vec{\mathcal T}\xrightarrow{\vec{\mathrm{br}}}\vec{\mathcal Br}\xrightarrow{\mathrm{s},\mathrm{t}}\{+,-\}^*$.
\end{definition}

\begin{definition}
\begin{itemize}
\item[$(a)$] Let $p,q\in\mathbb{Z}_{\geq 0}$. Define the set $\underline{\mathcal{T}}(p,q)$ by 
$$\underline{\mathcal{T}}(p,q):=\{T \ | \ T \text{ is a framed tangle with } |\partial_b T|=p \text{ and } |\partial_t T|=q \}/\text{isotopy}.$$
\item[$(b)$] Let $u,v\in \{+,-\}^*$. Define the set $\underline{\vec{\mathcal{T}}}(u,v)$ by 
$$\underline{\vec{\mathcal{T}}}(u,v):=\{\mathbf{T} \ | \ \mathbf{T} \text{ is an isotopy class of a framed oriented tangle with } \mathrm{s}(\mathbf{T})=u \text{ and } \mathrm{t}(\mathbf{T})=v\}.$$
\end{itemize}
\end{definition}

Notice that $\mathcal{T}=\bigsqcup_{p,q\geq 0}\underline{\mathcal{T}}(p,q)$ and $\vec{\mathcal{T}}=\bigsqcup_{u,v\in\{+,-\}^*}\underline{\vec{\mathcal{T}}}(u,v)$. The maps from Lemma~\ref{sourceandtargetfortangles} restrict to maps $\mathrm{br}:\underline{\mathcal{T}}(p,q)\to \underline{\mathcal Br}(p,q)$ for any $p,q\in\mathbb{Z}_{\geq 0}$ and to $\vec{\mathrm{br}}:\vec{\underline{\mathcal{T}}}(u,v)\to \vec{\underline{\mathcal Br}}(u,v)$ for any $u,v\in\{+,-\}^*$.

\begin{definition}\label{defpicircandpiseg} Let $T$ be a framed (oriented) tangle. We denote by $\pi_0(T)_\mathrm{cir}$ (resp. $\pi_0(T)_\mathrm{seg}$) the subset of $\pi_0(T)$ consisting of components homeomorphic to the circle $S^1$ (resp. to the segment $[0,1]$).
\end{definition}

\begin{lemma}\label{r:2022-10-24-2} 
\begin{itemize}
\item[$(a)$] Let $T$ and $T'$ be  framed (oriented) tangles. If $T$ and $T'$ are isotopic, then  $\pi_0(T)=\pi_0(T')$,  $\pi_0(T)_\mathrm{cir}=\pi_0(T')_\mathrm{cir}$ and $\pi_0(T)_\mathrm{seg}=\pi_0(T')_\mathrm{seg}$. We then denote such sets by $\pi_0(\mathbf{T})$,  $\pi_0(\mathbf{T})_{\mathrm{cir}}$ and $\pi_0(\mathbf{T})_\mathrm{seg}$ respectively.

\item[$(b)$] For $\mathbf{T}\in\vec{\mathcal T}$ (or $\mathbf{T}\in{\mathcal T}$), the set $\pi_0(\mathbf{T})_\mathrm{seg}$ is canonically identified with the set $\pi_0(\mathrm{br}(\mathbf{T}))$, see Definition~\ref{def:piceroforBr}.
\end{itemize}
\end{lemma}
\begin{proof} $(a)$ follows  straightforwardly. $(b)$ follows from Lemma~\ref{sourceandtargetfortangles}$(a)$ and the definition of $\pi_0(P)$ for a Brauer diagram $P$,  see Definition~\ref{def:piceroforBr}.
\end{proof}

\begin{lemma}\label{r2022-07-05}\begin{itemize}
\item[$(a)$] Let $u,v,z\in \{+,-\}^*$ and $\mathbf{T}_1,\mathbf{T}_2\in\vec{\mathcal T}$  with $\mathrm{s}(\mathbf{T}_1)=u$, $\mathrm{t}(\mathbf{T}_1)=v=\mathrm{s}(\mathbf{T}_2)$, $\mathrm{t}(\mathbf{T}_2)=z$. Let $T_1$ and $T_2$ be representatives of $\mathbf{T}_1$ and $\mathbf{T}_2$ respectively. Let $T_2\circ T_1$ be the framed oriented tangle obtained from the union of $T_1$ and $T_2$ by identifying $\partial_t T_1$ with $\partial_b T_2$ and by rescaling the vertical union of two cubes. The  isotopy class of $T_2\circ T_1$ does not depend on the representatives $T_1$ and $T_2$, we then denote it by ${\mathbf{T}}_2\circ \mathbf{T}_1$. We have $\mathrm{s}(\mathbf{T}_2\circ \mathbf{T}_1)=u$ and $\mathrm{t}(\mathbf{T}_2\circ \mathbf{T}_1)=z$.

\item[$(b)$] Let $u,v,u',v'\in \{+,-\}^*$ and $\mathbf{T}_1,\mathbf{T}_2\in\vec{\mathcal T}$ with  $\mathrm{s}(\mathbf{T}_1)=u$, $\mathrm{t}(\mathbf{T}_1)=v$, $\mathrm{s}(\mathbf{T}_2)=u'$, $\mathrm{t}(\mathbf{T}_2)=v'$. Let $T_1$ and $T_2$ be representatives of $\mathbf{T}_1$ and $\mathbf{T}_2$ respectively.  Let $T_1\otimes T_2$ be the framed oriented tangle obtained by horizontal juxtaposition and by rescaling the horizontal union of two cubes. The  isotopy class of $T_1\otimes T_2$ does not depend on the representatives $T_1$ and $T_2$, we then denote it by ${\mathbf{T}}_1\otimes \mathbf{T}_2$. We have $\mathrm{s}(\mathbf{T}_1\otimes \mathbf{T}_2)=u\cdot u'$ and $\mathrm{t}(\mathbf{T}_1\otimes \mathbf{T}_2)=v\cdot v'$.
\end{itemize}
Similar statements hold for the set of non-oriented framed tangles $\mathcal T$.
\end{lemma}

\begin{proof} The proof follows straightforwardly.
\end{proof}

\subsubsection{Properties of operations on framed oriented tangles}

\begin{lemma}\label{r:2022-10-24-03}
\begin{itemize}
\item[$(a)$] For any $\mathbf{T}, \mathbf{T}'\in \vec{\mathcal T}$ we have $$\vec{\mathrm{br}}({\mathbf{T}}\otimes \mathbf{T}') = \vec{\mathrm{br}}(\mathbf{T})\otimes \vec{\mathrm{br}}(\mathbf{T}') \qquad \text{and} \qquad \pi_0(\mathbf{T}\otimes \mathbf{T}')_{\mathrm{cir}} \simeq \pi_0(\mathbf{T})_{\mathrm{cir}}\sqcup \pi_0(\mathbf{T}')_{\mathrm{cir}}.$$
\item[$(b)$] For any composable $\mathbf{T}, \mathbf{T}'\in\vec{\mathcal T}$ we have $$\vec{\mathrm{br}}(\mathbf{T}'\circ \mathbf{T}) = \vec{\mathrm{br}}(\mathbf{T}')\circ_{\vec{\mathcal Br}} \vec{\mathrm{br}}(\mathbf{T}) \qquad \text{and} \qquad \pi_0(\mathbf{T}'\circ \mathbf{T})_{\mathrm{cir}} \simeq \pi_0(\mathbf{T}')_{\mathrm{cir}}\sqcup \pi_0(\mathbf{T})_{\mathrm{cir}}\sqcup \mathrm{Cir}(\vec{\mathrm{br}}(\mathbf{T}'),\vec{\mathrm{br}}(\mathbf{T})),$$
\noindent where the set $\mathrm{Cir}(\vec{\mathrm{br}}(\mathbf{T}'),\vec{\mathrm{br}}(\mathbf{T}))$ is as in Definition~\ref{def:CirforvecBr}.
\end{itemize}
Similar identities hold for isotopy classes of non-oriented framed tangles. 
\end{lemma}

\begin{proof}
The identities clearly hold for framed (oriented) tangles (before taking the quotient by isotopy). Then the result follows from Lemmas~\ref{sourceandtargetfortangles} and~\ref{r:2022-10-24-2}.
\end{proof}

\begin{definition}\label{def1.3} 
The monoidal category of framed oriented tangles $\vec{\underline{\mathcal T}}$\index[notation]{T3@$\vec{\underline{\mathcal T}}$} is defined as follows. The set of objects of $\vec{\underline{\mathcal T}}$ is the free associative monoid $\{+,-\}^*$. For $u,v\in \{+,-\}^*$, the set of morphisms $\vec{\underline{\mathcal T}}(u,v)$\index[notation]{T4@$\vec{\underline{\mathcal T}}(u,v)$} is the set of isotopy classes of framed oriented tangles $\mathbf{T}$ such that  $\mathrm{s}(\mathbf{T})=u$ and $\mathrm{t}(\mathbf{T})=v$. Let  $z\in \{+,-\}^*$,  $\mathbf{T}_1\in\underline{\vec{\mathcal T}}(u,v)$, $\mathbf{T}_2\in\underline{\vec{\mathcal T}}(v,z)$. The composition ${\mathbf{T}}_2\circ \mathbf{T}_1\in \underline{\vec{\mathcal T}}(u,z) $ is defined in terms of representatives as in Lemma~\ref{r2022-07-05}$(a)$.
The monoidal structure is also defined in terms of representatives  as in Lemma~\ref{r2022-07-05}$(b)$.
\end{definition}

\begin{lemma}\label{lemma:1:19:2jan2020} 
There is a monoidal functor $\underline{\vec{\mathrm{br}}}:\underline{\vec{\mathcal T}}\to\underline{\vec{\mathcal Br}}$\index[notation]{br3@$\underline{\vec{\mathrm{br}}}$} given by the identity on objects and by  the map $\vec{\mathrm{br}}:{\vec{\mathcal T}}\to {\vec{\mathcal Br}}$ from Lemma~\ref{sourceandtargetfortangles} at the level of morphisms. 
\end{lemma}

\begin{proof} The result follows from the combination of Lemmas~\ref{sourceandtargetfortangles}, \ref{r2022-07-05} and~\ref{r:2022-10-24-03}.
\end{proof}

\subsubsection{The monoid $\underline{\vec{\mathcal{T}}}(\emptyset,\emptyset)$, 3-manifolds and linking matrices}\label{sec:1.2_T-monoid}

\begin{lemma}\label{lemma:identif} We have the following.
\begin{itemize}
\item[$(a)$] The set $\underline{\vec{\mathcal T}}(\emptyset,\emptyset)$\index[notation]{T6@$\vec{\underline{\mathcal T}}(\emptyset,\emptyset)$}  has two structures of commutative monoid given by composition and~$\otimes$. These two structures coincide with the disjoint union operation from Definition~\ref{def:disjointandconnectedcum}$(a)$, therefore we will write $(\underline{\vec{\mathcal{T}}}(\emptyset,\emptyset), \dot{\sqcup})$ for this monoid.

\item[$(b)$] The map $\underline{\vec{\mathcal T}}(\emptyset,\emptyset)\to \mathbb{Z}_{\geq 0}$ which takes  $\mathbf{L}\in \underline{\vec{\mathcal T}}(\emptyset,\emptyset)$  to $|\pi_0(\mathbf{L})_{\mathrm{cir}}|$ is a monoid homomorphism.

\item[$(c)$] $\underline{\vec{\mathcal T}}(\emptyset,\emptyset)$ is a $\mathbb{Z}_{\geq 0}$-graded monoid with decomposition $\underline{\vec{\mathcal T}}(\emptyset,\emptyset)=\bigsqcup_{k \geq 0}\underline{\vec{\mathcal T}}(\emptyset,\emptyset)_k$, where $\underline{\vec{\mathcal T}}(\emptyset,\emptyset)_k$
is the preimage of $k$ under the map in $(b)$.

\item[$(d)$] There is monoid isomorphism 
\begin{equation}\label{eq:2022-04-18monoidhomo}
\left(\dfrac{\{\text{framed oriented  links in } S^3\}}{\text{isotopy}}, \dot{\sqcup}\right)\simeq (\underline{\vec{\mathcal T}}(\emptyset,\emptyset),\dot{\sqcup}).
\end{equation}
\end{itemize}
\end{lemma}

\begin{proof}
$(a)$ follows straightforwardly. $(b)$ follows from the identity $\pi_0(\mathbf{L} \dot{\sqcup} \mathbf{L}')=\pi_0(\mathbf{L}) \sqcup \pi_0(\mathbf{L}')$ for any~$\mathbf{L}, \mathbf{L}'\in \underline{\vec{\mathcal T}}(\emptyset,\emptyset)$. The statement in $(c)$ follows from $(b)$ and $(d)$ is straightforward. 
\end{proof}

\begin{definition}\label{def:tanglesTwz-k} For $w,z\in\{+,-\}^*$ and $k\in\mathbb{Z}_{\geq 0}$ define the set  $\underline{\vec{\mathcal{T}}}(w,z)_k$    by
 $$\underline{\vec{\mathcal{T}}}(w,z)_k=\{\mathbf{T}\in \underline{\vec{\mathcal{T}}}(w,z) \ | \ |\pi_0(\mathbf{T})_{\mathrm{cir}}| = k\}.$$

\end{definition}

One checks that $\underline{\vec{\mathcal T}}(\emptyset,\emptyset)_k = \vec{\mathcal T}(\vec{\varnothing})_k$, where 
$\{\vec{\varnothing}\} = \underline{\vec{\mathcal Br}}(\emptyset,\emptyset)=\underline{\vec{\mathcal Br}}(0,0)$, see Definition~\ref{def:vecBr} and Lemma~\ref{lemma:identif}.

\begin{remark} There is a bijection between the nerve  $N(\underline{\vec{\mathcal T}})_1$ of the category $\underline{\vec{\mathcal T}}$ and the set $\vec{\mathcal T}$.
\end{remark}

The composition of \eqref{eq:monoidmorphismlinksto3man} and \eqref{eq:2022-04-18monoidhomo} gives rise to the monoid homomorphism
\begin{equation}\label{dehn:surg:map:monoid}
(\underline{\vec{\mathcal T}}(\emptyset,\emptyset),\dot{\sqcup})\longrightarrow (3\text{-}\mathrm{Mfds},\#).
\end{equation}
in terms of which the Lickorish-Wallace theorem \cite{Lick,Wal} (surjectivity) and Kirby theorem (Theorem~\ref{thm:kirby})~\cite{Kir} (description of the kernel) will be reformulated in Proposition~\ref{categorical:kirby}.

For $k\geq 0$ we denote by $\mathrm{Sym}_k(\mathbb{Z})$\index[notation]{Sym_k@$\mathrm{Sym}_k(\mathbb{Z})$} the set of $k\times k$  symmetric matrices. We denote by $\mathrm{Sym}_k(\mathbb{Z})/\mathfrak{S}_k$ the set of classes of conjugation by the symmetric group~$\mathfrak{S}_k$. Then $\bigsqcup_{k\geq 0} \mathrm{Sym}_k(\mathbb{Z})/\mathfrak{S}_k$ is a graded monoid with product induced by the direct sum of matrices. 

Recall from \S~\ref{sec:linkingnumber} that for a framed oriented link with $k$ components ${L}$, its linking matrix  $\mathrm{Lk}({L})$ is a symmetric matrix of size $k\times k$.

\begin{lemma} There is a graded monoid homomorphism
$$\mathrm{Lk}: \underline{\vec{\mathcal{T}}}(\emptyset,\emptyset)\longrightarrow \bigsqcup_{k\geq 0} \mathrm{Sym}_k(\mathbb{Z})/\mathfrak{S}_k$$
taking the isotopy class $\mathbf{L}$ of a framed oriented link  to  the conjugacy class of the linking matrix $\mathrm{Lk}({L})$, where $L$ is a representative of $\mathbf{L}$.
\end{lemma}
\begin{proof}
The result follows from a direct verification.
\end{proof}

\subsubsection{Diagrams of framed (oriented) tangles}

An {\it oriented tangle diagram} is a planar diagram which may be obtained by regular projection  from  an oriented tangle. 
Conversely, such a diagram gives rise to an oriented tangle, which can be equipped with a framing by the convention (compatible with that of 
\S\ref{sect:112}) that the framing vector at each point is normal to the diagram, contained in the plane of the figure, and oriented to the right 
(resp. left) at the endpoints of the diagram marked $+$ (resp. $-$). 
See Figure \ref{figuraKI1_ell} for a diagram of a framed oriented tangle with its source and target.

\begin{figure}[ht!]
										\centering
                        \includegraphics[scale=0.9]{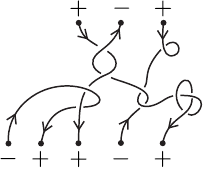}
					
\caption{Oriented tangle diagram representing an element of $\vec{\mathcal T}$ with source $-++-+$ and target $+-+$. 
}
\label{figuraKI1_ell} 								
\end{figure}

As in the case of  isotopy classes framed oriented links, for working with isotopy classes of framed oriented  tangles it is enough to work with their diagrams up to Reidemeister moves, i.e., we have the following bijection:
$$
\dfrac{\{\text{oriented tangle diagrams with blackboard framing}\}}{\text{planar isotopy, RI', RII, RIII}}\simeq \dfrac{\{\text{framed oriented  tangles}\}}{\text{isotopy}}.
$$

\subsection{Orientation-reversal operation for framed oriented tangles}\label{sect:4:2jan2020}

In this subsection we define an orientation-reversal operation for framed oriented tangles and we study its compatibility with the similar operation for oriented Brauer diagrams. We then introduce the set $\mathrm{CO}\vec{\mathcal{T}}=\sqcup_{k\geq 0}\mathrm{CO}\vec{\mathcal{T}}_k$.

Recall from Lemma~\ref{r:2022-10-24-2} that for $\mathbf{T}\in\vec{\mathcal T}$ we have $\pi_0(\mathbf{T})=\pi_0(\mathbf{T})_{\mathrm{cir}}\sqcup \pi_0(\mathbf{T})_{\mathrm{seg}}$ and $\pi_0(\mathbf{T})_{\mathrm{seg}}$ identifies canonically with $\pi_0(\mathrm{br}(\mathbf{T}))$.

\begin{definition}  Let $T$ be a framed oriented tangle and $c \in\pi_0({T})$. Let $\mathrm{co}_{\vec{\mathcal T}}(T,c)$\index[notation]{co_T@$\mathrm{co}_{\vec{\mathcal T}}$}  be the framed oriented tangle obtained from  $T$  by changing the orientation of $c$ as well as its framing into their opposites and leaving all other orientations and framings unchanged.
\end{definition}

\begin{lemma}\label{r:changeoforientationtangles}  Let $\mathbf{T}\in\vec{\mathcal T}$ and $T$ be a representative of $\mathbf{T}$. For any $c\in\pi_0(\mathbf{T})=\pi_0(T)$ (Lemma~\ref{r:2022-10-24-2}) the isotopy class of the framed oriented tangle $\mathrm{co}_{\vec{\mathcal T}}(T,c)$ does not depend on the representative $T$ of $\mathbf{T}$. We denote it by  $\mathrm{co}_{\vec{\mathcal T}}(\mathbf{T},c)$.
\end{lemma}
\begin{proof}
The result follows straightforwardly. 
\end{proof}

\begin{lemma}\label{r:2022-12-05-01} Let $\mathbf{T} \in \vec{\mathcal T}$ and $c \in \pi_0(\mathbf{T})$. 
\begin{itemize}
\item[$(a)$]  We have $\pi_0(\mathrm{co}_{\vec{\mathcal T}}(\mathbf{T},c))_{\mathrm{cir}}  \simeq \pi_0(\mathbf{T})_{\mathrm{cir}}$.

\item[$(b)$] If $c \in\pi_0(\mathbf{T})_{\mathrm{cir}}$, then $\vec{\mathrm{br}}(\mathrm{co}_{\vec{\mathcal T}}(\mathbf{T},c)) = \vec{\mathrm{br}}(\mathbf{T})$.

\item[$(c)$] If $c \in \pi_0(\mathbf{T})_{\mathrm{seg}}$, then $\vec{\mathrm{br}}(\mathrm{co}_{\vec{\mathcal T}}(\mathbf{T},c)) = \mathrm{co}_{\vec{\mathcal Br}}(\vec{\mathrm{br}}(\mathbf{T}),c)$. Here we use the identification  $\pi_0(\mathbf{T})_{\mathrm{seg}}\simeq\pi_0(\mathrm{br}(\mathbf{T}))$ from Lemma~\ref{r:2022-10-24-2} (see Definition~\ref{def:changeoforientarionvecBr} for $\mathrm{co}_{\vec{\mathcal Br}}$).
\end{itemize}
\end{lemma}

\begin{proof}
The identities hold for framed oriented tangles (before taking the quotient by isotopy). Then the result follows from Lemmas~\ref{sourceandtargetfortangles}, \ref{r:2022-10-24-2}, and \ref{r:changeoforientationtangles}. 
\end{proof}

Recall the decomposition 
$\underline{\vec{\mathcal T}}(\emptyset,\emptyset)=\sqcup_{k\geq0}{\underline{\vec{\mathcal T}}}(\emptyset,\emptyset)_k$, see Lemma~\ref{lemma:identif}.

\begin{lemma} Let $k\geq 0$ and $\mathbf{T}\in {\underline{\vec{\mathcal T}}}(\emptyset,\emptyset)_k$, then $\mathrm{co}_{\vec{\mathcal T}}(\mathbf{T},c)\in {\underline{\vec{\mathcal T}}}(\emptyset,\emptyset)_k$ 
for any $c\in \pi_0(\mathbf{T})$.
\end{lemma}

\begin{proof} The result follows from Lemma~\ref{r:changeoforientationtangles} by using representatives.
\end{proof}

\subsubsection{The subset $\mathrm{CO}\vec{\mathcal{T}}\subset\underline{\vec{\mathcal{T}}}(\emptyset,\emptyset)\times \underline{\vec{\mathcal{T}}}(\emptyset,\emptyset)$}

It is clear that if $\mathbf{T}\in\underline{\vec{\mathcal T}}(\emptyset,\emptyset)_k$  (for $k\in\mathbb{Z}_{\geq 1}$) and $c \in \pi_0(\mathbf{T})$ then $\mathrm{co}_{\vec{\mathcal T}}(\mathbf{T},c) \in \underline{\vec{\mathcal T}}(\emptyset,\emptyset)_k$.

\begin{definition}\label{def:espaceCOT} For $k\in\mathbb{Z}_{\geq 1}$ define the subset $\mathrm{CO}\vec{\mathcal{T}}_k\subset \underline{\vec{\mathcal{T}}}(\emptyset,\emptyset)_k\times \underline{\vec{\mathcal{T}}}(\emptyset,\emptyset)_k$\index[notation]{COT_k@$\mathrm{CO}\vec{\mathcal{T}}_k$} as
\begin{equation}\label{eq:espaceCOT-k}
\mathrm{CO}\vec{\mathcal{T}}_k =\big\{(\mathbf{T}, \mathrm{co}_{\vec{\mathcal{T}}}(\mathbf{T},c)) \ | \ \mathbf{T}\in \underline{\vec{\mathcal{T}}}(\emptyset,\emptyset)_k, \ c\in\pi_0(\mathbf{T})\big\}.
\end{equation}
We set $\mathrm{CO}\vec{\mathcal{T}}_0=\emptyset$ and
\begin{equation}\label{eq:espaceCOT}
\mathrm{CO}\vec{\mathcal{T}}  := \bigsqcup_{k\geq 0} \mathrm{CO}\vec{\mathcal{T}}_k
 = \big\{(\mathbf{T}, \mathrm{co}_{\vec{\mathcal{T}}}(\mathbf{T},c)) \ | \ \mathbf{T}\in \underline{\vec{\mathcal{T}}}(\emptyset,\emptyset), \ c\in\pi_0(\mathbf{T})\big\} \subset \underline{\vec{\mathcal{T}}}(\emptyset,\emptyset)\times \underline{\vec{\mathcal{T}}}(\emptyset,\emptyset).
\end{equation}
\index[notation]{COT@$\mathrm{CO}\vec{\mathcal{T}}$}
\end{definition}

\subsection{Doubling operations for framed oriented tangles}\label{sect:3:2jan2020}

In this subsection we introduce  a doubling operation  for framed oriented tangles and study its compatibility with the similar operation for oriented Brauer diagrams.

\begin{definition}\label{def:doublinginT} Let $T$ be a framed (oriented) tangle and $X\subset \pi_0(T)_{\mathrm{seg}}$. Let $\mathrm{dbl}_{\vec{\mathcal T}}(T,X)$ (also denoted~$T^X$) \index[notation]{dbl_T@$\mathrm{dbl}_{\vec{\mathcal T}}$} be the framed (oriented) tangle obtained from $T$ by doubling each of the connected components of~$T$ belonging to~$X$ along 
its framing.
\end{definition}

\begin{lemma}\label{lem120} Let $\mathbf{T}\in\vec{\mathcal T}$ and $T$ be a representative of $\mathbf{T}$. For any $X\subset \pi_0(\mathbf{T})_{\mathrm{seg}} = \pi_0({T})_{\mathrm{seg}}$ (Lemma~\ref{r:2022-10-24-2}) the isotopy class of the framed oriented tangle $\mathrm{dbl}_{\vec{\mathcal T}}(T,X)$ does not depend on the representative $T$ of $\mathbf{T}$. We denote it by $\mathrm{dbl}_{\vec{\mathcal T}}(\mathbf{T},X)$.
\end{lemma}

\begin{proof}
The result is a direct verification.
\end{proof}

\begin{lemma}\label{r:2022-10-31-01} Let $\mathbf{T} \in \vec{\mathcal T}$ and $X \subset \pi_0(\mathbf{T})_{\mathrm{seg}}$. 
\begin{itemize}
\item[$(a)$]   $\pi_0(\mathrm{dbl}_{\vec{\mathcal T}}(\mathbf{T},X))_{\mathrm{cir}}  \simeq \pi_0(\mathbf{T})_{\mathrm{cir}}$.

\item[$(b)$]  $\vec{\mathrm{br}}(\mathrm{dbl}_{\vec{\mathcal T}}(\mathbf{T},X)) = \mathrm{dbl}_{\vec{\mathcal Br}}(\vec{\mathrm{br}}(\mathbf{T}),X)$. Here we use the identification  $\pi_0(\mathbf{T})_{\mathrm{seg}}\simeq\pi_0(\mathrm{br}(\mathbf{T}))$ from Lemma~\ref{r:2022-10-24-2} (see Definition~see Definitions~\ref{def:doublingforvecBr} and \ref{def:doublingforskeleta} for $\mathrm{dbl}_{\vec{\mathcal Br}}$).
\end{itemize}
\end{lemma}

\begin{proof}
The identities hold for framed oriented tangles (before taking the quotient by isotopy). Then the result follows from Lemmas~\ref{sourceandtargetfortangles}, \ref{r:2022-10-24-2}, and \ref{r:2022-10-31-01}. 
\end{proof}

\begin{example}\label{example:doubl:T} 
Figure~\ref{figureEx1_25}$(a)$ shows a representative of $\mathbf{T}\in \vec{\mathcal T}$ and $X\subset \pi_0(\mathbf{T})_{\mathrm{seg}}$ indicated using thick lines.  Figure~\ref{figureEx1_25}$(b)$ shows a representative of ${\mathrm{dbl}_{\vec{\mathcal T}}}(\mathbf{T},X)\in\vec{\mathcal T}$. The image in $\vec{\mathcal Br}$ of these elements under the map $\vec{\mathrm{br}}$ are shown in  Figure~\ref{figureEx1_19}$(a)$ and $(b)$ respectively.

\begin{figure}[ht!]
										\centering
                        \includegraphics[scale=0.85]{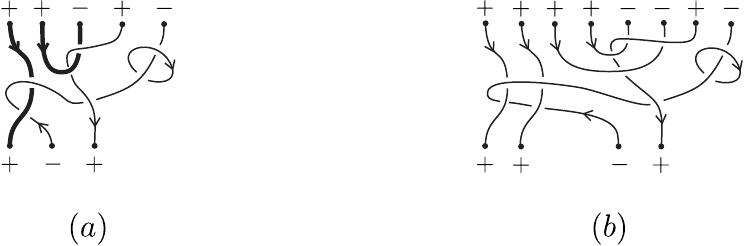}
												
\caption{$(a)$ A representative of an element $(\mathbf{T},X)$ and $(b)$ a representative of ${\mathrm{dbl}_{\vec{\mathcal T}}}(\mathbf{T},X)$.}
\label{figureEx1_25} 								
\end{figure}
\end{example}

\subsection{The space KII}\label{sec:new2.5}

In this subsection we use the doubling operation from the previous subsection and \emph{external tangles} to define a subset $\mathrm{KII}\vec{\mathcal{T}}\subset \underline{\vec{\mathcal{T}}}(\emptyset,\emptyset)\times \underline{\vec{\mathcal{T}}}(\emptyset,\emptyset)$.

\subsubsection{External tangles}
Consider $k$ ordered points $d_1,\cdots, d_k\in \{0\}\times [-1,1]
 \subset [-1,1]^2$ distributed uniformly. Let $b_i \in [-1,1]^2\times\{-1\}$ (resp. $t_i \in [-1,1]^2\times\{1\}$) denote the point in $\partial([-1,1]^3)$ given by $d_i\times\left\{-1\right\}$ (resp. $d_i\times\left\{1\right\}$).

Denote by $[-1,1]^3_{\mathrm{ext}}$ the compact oriented 3-manifold ${S}^3 \setminus \mathrm{int} ([-1,1]^3)$.

\begin{definition}\label{def:externaltangle} $(a)$ An \emph{external  oriented tangle $T$ in $[-1,1]^3_{\mathrm{ext}}$ with $k\geq 1$ non-closed components} is a compact oriented 1-manifold with boundary  $T$ properly embedded into $[-1,1]^3_{\mathrm{ext}}$ such that 
\begin{itemize}
\item $\partial T\subset [-1,1]^2\times\left\{\pm 1\right\}$,
\item $\partial T\cap \left([-1,1]^2\times\left\{ -1\right\}\right) = \{b_1,\cdots,b_k\}$,
\item $\partial T\cap \left([-1,1]^2\times\left\{ 1\right\}\right)  = \{t_1,\cdots,t_k\}$
\item There are exactly $k$ non-closed connected component $T_1,\cdots T_k$  of $T$ such that $\partial T_i = \{b_i,d_i\}$ for each $1\leq i\leq k$.
\end{itemize}
$(b)$ An \emph{external  framed oriented tangle  in $[-1,1]^3_{\mathrm{ext}}$  with $k$ non-closed components} is an external  oriented tangle $T$ in $[-1,1]^3_{\mathrm{ext}}$  with $k$ non-closed components such that each connected component is endowed with a framing whose value at the endpoints of $\partial T$ is $\pm(0,1,0)$,  the sign being determined  as follows: if $v$ is the unit tangent vector and $n_v$  the unit normal vector at an endpoint, then the tuple $(n_v,v,(1,0,0))$ is a positive frame for $\mathbb{R}^3$.

\end{definition}

As in the case of tangles, external (framed, oriented tangles can be represented in two dimensions by considering regular projections on the plane $(\{-1\}\times\mathbb{R}^2)\setminus \mathrm{int}(\{-1\}\times[-1,1]^2)$ giving rise to diagrams and the framing can be encoded using the blackboard framing convention,  see Figure~\ref{figure2posarxiv}.

\begin{figure}[ht!]
			\centering
            \includegraphics[scale=0.8]{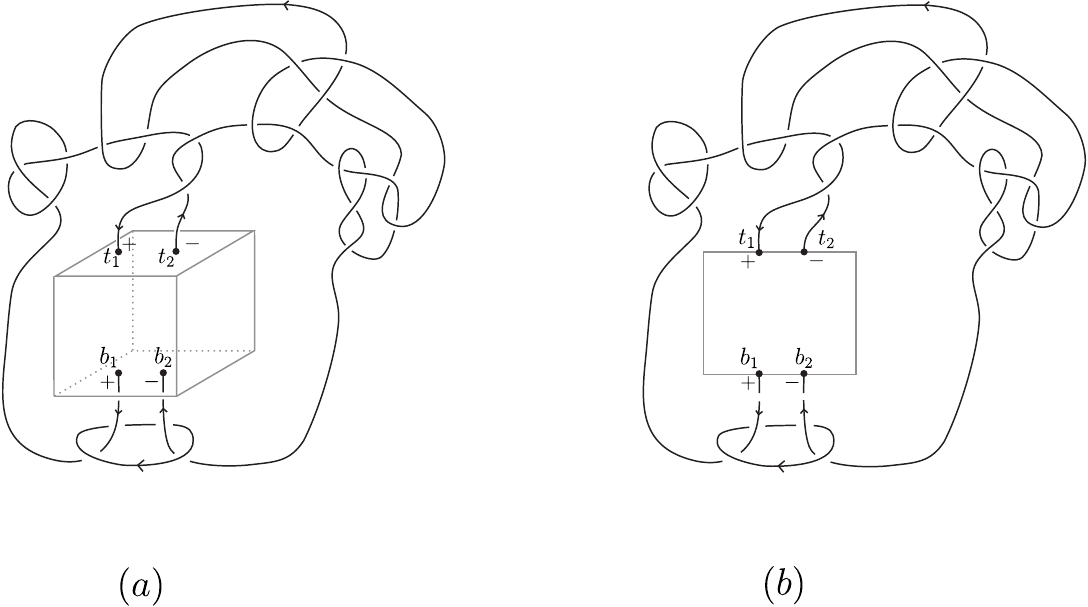}			
            \caption{$(a)$ An example of an external oriented tangle with $2$ non-closed components. $(b)$ a diagram of the external tangle in $(a)$.}
\label{figure2posarxiv} 								
\end{figure}


 \begin{definition} $(a)$ Let ${\vec{\mathcal T}}(+-,+-)^{\mathrm{ext}}$ be the set of isotopy classes (relative to the boundary) of external framed oriented tangles $T$ with $2$ non-closed components in $[-1,1]^3_{\mathrm{ext}}$ such that the orientations of $T$ at $b_1$ and $t_2$ are  outgoing and at $b_2$ and $t_1$ are incoming. See Figure~\ref{figure2posarxiv} for an example.  

 $(b)$ Let  ${\vec{\mathcal T}}(+--,+--)^{\mathrm{ext}}$ be the set of isotopy classes (relative to the boundary) of external framed oriented tangles $T$ with $3$ non-closed components in $[-1,1]^3_{\mathrm{ext}}$  such that the orientations of $T$ at $b_1$, $t_2$ and $t_3$ are outgoing and at $b_2$, $b_3$ and $t_1$ are  incoming. 
 \end{definition}

\subsubsection{Doubling maps for external tangles}
Definition~\ref{def:doublinginT} can be applied in the context of external tangles as follows. 

\begin{definition}\label{def:doublinginext} Define a map 
\begin{equation}
\mathrm{dbl}^{+-}: {\vec{\mathcal T}}(+-,+-)^{\mathrm{ext}} \longrightarrow {\vec{\mathcal T}}(+--+,+--)^{\mathrm{ext}}
\end{equation}
as follows.    Let $\mathbf{T}\in{\vec{\mathcal T}}(+-,+-)^{\mathrm{ext}}$ and let $T$ be one of its representatives. Let $T_2$ be the connected component of $T$ such that $\partial T_2=\{b_2, t_2\}$. Let $T'$ be the external framed oriented tangle with $3$ non-closed components in $[-1,1]^3_{\mathrm{ext}}$ obtained from $T$ by doubling the connected component $T_2$ along its framing and endowing it with the same orientation and framing as those of $T_2$. The isotopy class of $T'$ in $[-1,1]^3_{\mathrm{ext}}$ does not depend on the isotopy class of $T$ and it belongs to ${\vec{\mathcal T}}(+--,+--)^{\mathrm{ext}}$, we set $\mathrm{dbl}^{+-}(\mathbf{T})=\mathbf{T}'$.
\end{definition}

\subsubsection{From external tangles to  links and KII maps}

\begin{definition}\label{def:pairingPu} For $u\in\{+-, +--\}$ define the map
\begin{equation}
    P_u:  {\vec{\mathcal T}}(u,u)^{\mathrm{ext}}\times \underline{\vec{\mathcal T}}(u,u) \longrightarrow \underline{\vec{\mathcal T}}(\emptyset,\emptyset)
\end{equation}
as follows. Let $(\mathbf{T}, \mathbf{S})\in {\vec{\mathcal T}}(u,u)^{\mathrm{ext}}\times \underline{\vec{\mathcal T}}(u,u)$ and let $T$ and $S$ be representatives of $\mathbf{T}$ and $\mathbf{S}$, respectively. Define $P_u(\mathbf{T}, \mathbf{S})$ as the isotopy class in ${S}^3$ of the framed oriented link obtained by gluing the pairs $([-1,1]^3_{\mathrm{ext}}, T\subset  [-1,1]^3_{\mathrm{ext}})$   and $([-1,1]^3, S\subset [-1,1]^3)$ identically along their boundaries. See Figure~\ref{figure1posarxiv} for a schematic explanation.  
\end{definition}

We use a particular case of the above maps as follows.

\begin{definition}\label{def_KIIMaps}

$(a)$ Define the map 
\begin{equation}
    \mathrm{KIImap}^1:  \underline{\vec{\mathcal T}}(+-,+-)^{\mathrm{ext}} \longrightarrow \underline{\vec{\mathcal T}}(\emptyset,\emptyset) 
\end{equation}
    by $\mathrm{KIImap}^1(\mathbf{T}) = P_{+-}(\mathbf{T}, \mathrm{Id}_{+-})$, where $\mathrm{Id}_{+-}\in \underline{\vec{\mathcal T}}(+-,+-)$ is the identity morphism, see Figure~\ref{figure1posarxiv}$(a)$. 

$(b)$ Define the map 
\begin{equation}
    \mathrm{KIImap}^2:  \underline{\vec{\mathcal T}}(+-,+-)^{\mathrm{ext}} \longrightarrow \underline{\vec{\mathcal T}}(\emptyset,\emptyset) 
\end{equation}
    by $\mathrm{KIImap}^2(\mathbf{T}) = P_{+--}(\mathrm{dbl}^{+-}(\mathbf{T}), e_{+--})$, where $e_{+--}\in \underline{\vec{\mathcal T}}(+--,+--)$ is the isotopy class of the blackboard framed oriented tangle in $[-1,1]^3$ shown in red in Figure~\ref{figure1posarxiv}$(b)$, this map is illustrated in the same figure.   
\end{definition}

\begin{figure}[ht!]
			\centering
            \includegraphics[scale=0.8]{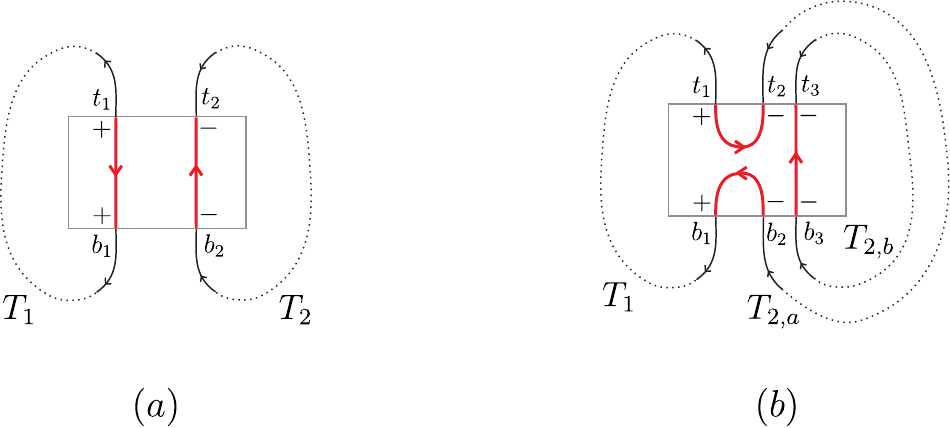}			
            \caption{Schematic representation of the maps $ \mathrm{KIImap}^1$ in $(a)$ and  $\mathrm{KIImap}^2$ in  $(b)$. In both figures there are circle components which are not shown.}
\label{figure1posarxiv} 								
\end{figure}

\subsubsection{The space $\mathrm{KII}\vec{\mathcal{T}}\subset\underline{\vec{\mathcal{T}}}(\emptyset,\emptyset)\times \underline{\vec{\mathcal{T}}}(\emptyset,\emptyset)$ and the Kirby II move}

\begin{definition}\label{def:espaceKIIT} Define the subset $\mathrm{KII}\vec{\mathcal{T}}\subset \underline{\vec{\mathcal{T}}}(\emptyset,\emptyset)\times \underline{\vec{\mathcal{T}}}(\emptyset,\emptyset)$ as follows.
\begin{equation}\label{eq:espaceKIIT}
\mathrm{KII}\vec{\mathcal{T}} =\big\{(\mathrm{KIIMap}^1(\mathbf{T}), \mathrm{KIIMap}^2(\mathbf{T})) \ | \  \mathbf{T}\in {\vec{\mathcal{T}}}(+-,+-)^{\mathrm{ext}}\big\}.
\end{equation}
\end{definition}

The set $\mathrm{KII}\vec{\mathcal{T}}\subset \underline{\vec{\mathcal{T}}}(\emptyset,\emptyset)\times \underline{\vec{\mathcal{T}}}(\emptyset,\emptyset)$ allows to describe the Kirby II move in Theorem~\ref{thm:kirby} as follows.

\begin{proposition} Let $L$ be a framed oriented link in $S^3$ and let $L'$ be obtained from $L$ by a Kirby II move as described in Theorem~\ref{thm:kirby}3). Then there exists an external framed oriented tangle $\mathbf{T}\in {\vec{\mathcal{T}}}(+-,+-)^{\mathrm{ext}}$ such that $\mathbf{L} = \mathrm{KIIMap}^1(\mathbf{T})$ and $\mathbf{L}' = \mathrm{KIIMap}^2(\mathbf{T})$, that is, $(\mathbf{L}, \mathbf{L}')\in \mathrm{KII}\vec{\mathcal{T}}$.
\end{proposition}

\begin{proof} Let $L_a$ and $L_b$ the two different components of $L$ as in Theorem~\ref{thm:kirby}. Let $S_1=\{0\}\times\{-1/3\}\times [-1,1]\subset [-1,1]^3$ and $S_2=\{0\}\times\{1/3\}\times [-1,1]\subset [-1,1]^3$. 
Up to isotopy we can suppose that $L\cap [-1,1]^3 = S_1\sqcup S_2$ and $L_a \cap [-1,1]^3 = S_1$ and   $L_b \cap [-1,1]^3 = S_2$. Also, up to isotopy, we can suppose that the orientation of $S_1$ (resp. $S_2$) is downwards (resp. upwards).  Set $T= L\setminus \mathrm{int} (S_1\sqcup S_2) \subset [-1,1]^3_{\mathrm{ext}}$, then the isotopy class $\mathbf{T}$ of $T$ satisfies the conditions stated in the proposition.
\end{proof}

\subsubsection{Decomposition of $\mathrm{KII}\vec{\mathcal{T}}$}

In this subsection we show that the set $\mathrm{KII}\vec{\mathcal T}$ decomposes as a disjoint union $\sqcup_{k\geq 0}\mathrm{KII}\vec{\mathcal T}_k$ such that  $\mathrm{KII}\vec{\mathcal T}_k\subset \underline{\vec{\mathcal T}}(\emptyset,\emptyset)_k\times \underline{\vec{\mathcal T}}(\emptyset,\emptyset)_k$.


\begin{definition}\label{def:espaceTP}

$(a)$  Let $P\in\vec{\mathcal Br}$.
Define the subset $\vec{\mathcal T}(P)$ of $\vec{\mathcal T}$ \index[notation]{T(P)@$\vec{\mathcal T}(P)$} as the preimage of $P$ under the map $\vec{\mathrm{br}}:\vec{\mathcal T}\to \vec{\mathcal Br}$ (see Lemma~\ref{sourceandtargetfortangles}).

 $(b)$ Let $P\in\vec{\mathcal Br}$. For $k\in\mathbb{Z}_{\geq 0}$ define the subset $\vec{\mathcal T}(P)_k$\index[notation]{T(P)_k@$\vec{\mathcal T}(P)_k$} of $\vec{\mathcal T}(P)$ as the preimage of $k$ under the map $\vec{\mathcal T}(P)\to \mathbb{Z}_{\geq 0}$ taking $\mathbf{T}\in \vec{\mathcal T}(P)$ to $|\pi_0(\mathbf{T})_{\mathrm{cir}}|$. We have $\vec{\mathcal T}(P)= \sqcup_{k\geq 0} \vec{\mathcal T}(P)_k$. 

$(b)$ For $w,z\in\{+,-\}^*$ and $k\in\mathbb{Z}_{\geq 0}$ define the set   ${\vec{\mathcal{T}}}(w,z)_k^{\mathrm{ext}}$  by
 $${\vec{\mathcal{T}}}(w,z)_k^{\mathrm{ext}}=\{\mathbf{T}\in {\vec{\mathcal{T}}}(w,z)^{\mathrm{ext}} \ | \ |\pi_0(\mathbf{T})_{\mathrm{cir}}| = k\}.$$

\end{definition}

\begin{lemma}\label{lemma_KImapsrestriction}  The maps ${\mathrm{KIIMap}}^i$ (for $i=1,2$) from Definition~\ref{def_KIIMaps} restrict to well-defined maps 
\begin{equation}\label{eq:def:tildeKII_1maprestriction}
\mathrm{KIIMap}^i_{k}:{\vec{\mathcal T}}(+-,+-)^{\mathrm{ext}}_{k-2}\longrightarrow \underline{\vec{\mathcal{T}}}(\emptyset,\emptyset)_k 
\end{equation}
for any $k\in\mathbb{Z}_{\geq 2}$. 
\end{lemma}

\begin{proof}

Let $\mathbf{T}\in {\vec{\mathcal T}}(+-,+-)^{\mathrm{ext}}_{k-2}$ and let $T$ be one of its representatives. Let $T_i\in\pi_0(T)$ such that $\partial T_i=\{b_i,t_i\}$ (for $i=1,2$). Then $\pi_0(T) = \{T_1, T_2\}\sqcup \pi_0(T)_{\mathrm{cir}}$. 

A representative $L$ of ${\mathrm{KIIMap}}^1(\mathbf{T})=P_{+-}(\mathbf{T},\mathrm{Id}_{+-})$ is obtained from $T$ by inserting in $[-1,1]^3$ a representative of the trivial tangle $\mathrm{Id}_{+-}\in \underline{\vec{\mathcal T}}(+-,+-)$, this operation creates two new circle components, one arising from the non-closed component $T_1$ together with the first vertical segment of $\mathrm{Id}_{+-}$ and  the other one from  $T_2$ together with the second vertical segment of $\mathrm{Id}_{+-}$,  see Figure~\ref{figure1posarxiv}$(a)$. So we have $|\pi_0(L)|=|\pi_0(L)_{\mathrm{cir}}| = 2 + |\pi_0(T)_{\mathrm{cir}}| = k$. It follows that $|\pi_0({\mathrm{KIIMap}}^1(\mathbf{T}))| = |\pi_0({\mathrm{KIIMap}}^1(\mathbf{T}))_{\mathrm{cir}}| = k$, which proves the assertion for ${\mathrm{KIIMap}}^1$.

Let $T'\subset [-1,1]^3_{\mathrm{ext}}$ be the representative of  $\mathrm{dbl}^{+-}(\mathbf{T})$ obtained from $T$ as described in Definition~\ref{def:doublinginext}. This external tangle has three non-closed connected components: $T_1, T_{2,a}, T_{2,b}$, where $T_{2,a}, T_{2,b}$ are the two parallel copies coming from $T_2$. A representative of $L'$ of ${\mathrm{KIIMap}}^2(\mathbf{T})=P_{+-}(\mathbf{T},e_{+--})$ is obtained from $T'$ by inserting in $[-1,1]^3$ a representative of $e_{+--}$, this operation creates two  circle components  from $T_1, T_{2,a}, T_{2,b}$, one coming from the non-closed component $T_{2,b}$ together with the vertical segment  of $e_{+--}$ and the other one from the non-closed components $T_1$, $T_{2,a}$ and the cup and cap components of  $e_{+--}$, see Figure~\ref{figure1posarxiv}$(b)$. Hence $|\pi_0(L')|=|\pi_0(L')_{\mathrm{cir}}| = 2 + |\pi_0(T)_{\mathrm{cir}}| = k$. Therefore $|\pi_0({\mathrm{KIIMap}}^2(\mathbf{T}))| = |\pi_0({\mathrm{KIIMap}}^2(\mathbf{T}))_{\mathrm{cir}}| = k$, which proves the assertion for ${\mathrm{KIIMap}}^2$.

\end{proof}

\begin{definition}\label{def:espaceKIIT-k} For $k\in\mathbb{Z}_{\geq 2}$, define the subset $\mathrm{KII}\vec{\mathcal{T}}_k\subset \underline{\vec{\mathcal{T}}}(\emptyset,\emptyset)_k\times \underline{\vec{\mathcal{T}}}(\emptyset,\emptyset)_k$ by
\begin{equation}\label{eq:espaceKIIT-k}
\mathrm{KII}\vec{\mathcal{T}}_k := \mathrm{Im}\big((\mathrm{KIIMap}^1_{k},\mathrm{KIIMap}^2_{k})\big).
\end{equation}
We set $\mathrm{KII}\vec{\mathcal{T}}_0= \mathrm{KII}\vec{\mathcal{T}}_1 = \emptyset$.
\end{definition}

\begin{proposition} The set $\mathrm{KII}\vec{\mathcal{T}}$ from Definition~\ref{def:espaceKIIT}  decomposes as
\begin{equation}\label{eq:espaceKIITdecomposed}
\mathrm{KII}\vec{\mathcal{T}} =\bigsqcup_{k\in\mathbb{Z}_{\geq 0}} \mathrm{KII}\vec{\mathcal{T}}_k.
\end{equation}
\end{proposition}

\begin{proof}

We have
\begin{equation*}
\begin{split}
\mathrm{KII}\vec{\mathcal{T}} & = \big\{(\mathrm{KIIMap}^{1}(\mathbf{T}), \mathrm{KIIMap}^{2}(\mathbf{T})) \ | \  \mathbf{T}\in {\vec{\mathcal{T}}}(+-,+-)^{\mathrm{ext}}\big\}\\
& =\bigsqcup_{k\in\mathbb{Z}_{\geq 0}} \big\{(\mathrm{KIIMap}^{1}(\mathbf{T}), \mathrm{KIIMap}^{2}(\mathbf{T})) \ | \  \mathbf{T}\in {\vec{\mathcal{T}}}(+-,+-)^{\mathrm{ext}}_k\big\}\\
& =  \bigsqcup_{k\in\mathbb{Z}_{\geq 0}}\mathrm{KII}\vec{\mathcal{T}}_k,
\end{split}
\end{equation*}
in the second equality we use Lemma~\ref{lemma_KImapsrestriction} and  the third one  follows from  Definition~\ref{def:espaceKIIT-k}.
\end{proof}

\subsection{Alternative description of the image \texorpdfstring{$\mathrm{KII}{\vec{\mathcal T}}$}{KIIT} of the \texorpdfstring{$\mathrm{KII}$}{KII} map}\label{sec:3.5}

In Definition~\ref{def:espaceKIIT}, the set $\mathrm{KII}\vec{\mathcal{T}}$  is defined using the map $(\mathrm{KIIMap}^1,\mathrm{KIIMap}^2)$. In this subsection, we give alternative descriptions of this set, first in terms of the category~$\underline{\vec{\mathcal T}}$ (Proposition~\ref{firstequivKIIT}),  then in terms of the category $\underline{\mathcal{P}a\vec{\mathcal{T}}}$ of parenthesized tangles (Proposition~\ref{secondequivKIIT}).  The first alternative description is based on maps $\widetilde{\mathrm{KIIMap}}^i$ for $i=1,2)$ (see \S\ref{sec:1-6-1}). A bijection from ${\vec{\mathcal T}}(\downarrow \uparrow)$ (the source of  the maps $\widetilde{\mathrm{KIIMap}}^i$) to  ${\vec{\mathcal T}}(+-,+-)^{\mathrm{ext}}$ (the source of the maps $\mathrm{KIIMap}^i$) is constructed in \S\ref{sec:1-6-4}. This allows to obtain the first alternative description of $\mathrm{KII}\vec{\mathcal{T}}$. In \S\ref{sec:1-6-5}, we introduce the category $\underline{\mathcal{P}a\vec{\mathcal{T}}}$ and in \S\ref{sec:1-6-6}, we obtain the description of $\mathrm{KII}\vec{\mathcal{T}}$  in terms of maps $\widetilde{\mathrm{KIIMap}}_{\underline{\mathcal{P}a\vec{\mathcal{T}}}}^i \quad (i=1,2)$ involving elementary generators of this category.

\subsubsection{Maps $\widetilde{\mathrm{KIIMap}}^{i}:\vec{\mathcal{T}}(\downarrow \uparrow)\longrightarrow \underline{\vec{\mathcal{T}}}(\emptyset,\emptyset)$  for $i=1,2$.}\label{sec:1-6-1}

\begin{notation}\label{notationarrows}
From now on we use the following notation:  $\downarrow \ :=\mathrm{Id}_+ \in \underline{\vec{\mathcal Br}}(+,+)$, $\uparrow \ :=\mathrm{Id}_- \in \underline{\vec{\mathcal Br}}(-,-)$. Then for  $\epsilon_1,\ldots, \epsilon_n \in \{+,-\}$, one has $\mathrm{Id}_{\epsilon_1\cdots \epsilon_n} = \mathrm{ar}_1 \otimes \cdots \otimes \mathrm{ar}_n\in \underline{\vec{\mathcal Br}}(\epsilon_1\cdots \epsilon_n, \epsilon_1\cdots \epsilon_n)$ (denoted $\mathrm{ar}_1  \cdots  \mathrm{ar}_n$) where $\mathrm{ar}_i \in \{\downarrow,\uparrow\}$ is given by $\mathrm{ar}_i= \ \downarrow$ if $\epsilon_i=+$ and $\mathrm{ar}_i= \ \uparrow$ if $\epsilon_i=-$. For instance, $\downarrow \uparrow \ =\mathrm{Id}_{+-}\in \underline{\vec{\mathcal Br}}(+-,+-)$ and  $\uparrow \downarrow\  =\mathrm{Id}_{-+}\in\underline{\vec{\mathcal Br}}(-+,-+)$.
\end{notation}

\begin{definition}\label{def_KIIMap1}
Define  the map $\widetilde{\mathrm{KIIMap}}^1:{\vec{\mathcal{T}}}(\downarrow\uparrow)\longrightarrow \underline{\vec{\mathcal{T}}}(\emptyset,\emptyset)$ \index[notation]{KIIMap4@$\widetilde{\mathrm{KIIMap}}^1$} as the restriction of the map
\begin{equation}
\widetilde{\mathrm{KIIMap}}^1:\underline{\vec{\mathcal{T}}}(+-,+-)\longrightarrow \underline{\vec{\mathcal{T}}}(\emptyset,\emptyset) 
\end{equation}
given by 
\begin{equation}\label{eq:def:tildeKII_1}
\widetilde{\mathrm{KIIMap}}^1(\mathbf{S})= \mathrm{cap}^{\mathcal T}_{+-}\circ (\mathrm{Id}_{+}\otimes \mathrm{cap}^{\mathcal T}_{-+}\otimes \mathrm{Id}_{-})\circ (\mathrm{Id}_{+-}\otimes \mathbf{S})\circ (\mathrm{Id}_{+}\otimes \mathrm{cup}^{\mathcal T}_{-+}\otimes \mathrm{Id}_{-})\circ \mathrm{cup}^{\mathcal T}_{+-}
\end{equation}
for $\mathbf{S}\in\underline{\vec{\mathcal{T}}}(+-,+-)$.  See Figure~\ref{figureEx1_41_abstract} for a schematic representation of $\widetilde{\mathrm{KIIMap}}^1$ using representatives.   This ma is well-defined since $\underline{\vec{\mathcal T}}$ is a category.
\begin{figure}[ht!]
			\centering
            \includegraphics[scale=0.8]{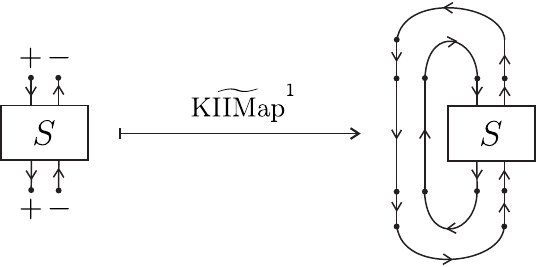}					\caption{Schematic representation of the map $\widetilde{\mathrm{KIIMap}}^1$ using representatives.}
\label{figureEx1_41_abstract} 								
\end{figure}
\end{definition}

%

\begin{lemma}\label{r:20221115-01} Let $\mathbf{S}\in\vec{\mathcal T}(\downarrow\uparrow)\subset \underline{\vec{\mathcal T}}(+-,+-)$. Let $c_{\downarrow}\in\pi_0(\mathbf{S})_{\mathrm{seg}}$ be the connected component corresponding to $\downarrow$ in $\vec{\mathrm{br}}(\mathbf{S})=\downarrow\uparrow$. Then the expression
\begin{equation}\label{tildeKIImap20221113}
\widetilde{\mathrm{KIIMap}}^2(\mathbf{S}): =(\mathrm{cap}^{\mathcal T}_{-+}\otimes \mathrm{cap}^{\mathcal T}_{+-})\circ (\mathrm{Id}_{-}\otimes \mathrm{dbl}_{\vec{\mathcal{T}}}(\mathbf{S}, c_{\downarrow}))\circ(\mathrm{cup}^{\mathcal T}_{-+}\otimes \mathrm{cup}^{\mathcal T}_{+-})
\end{equation}
 \index[notation]{KIIMap6@$\widetilde{\mathrm{KIIMap}}^2$} is well-defined and belongs to $\underline{\vec{\mathcal T}}(\emptyset,\emptyset)$. See Figure~\ref{figureEx1_42_abstract} for a schematic representation of $\widetilde{\mathrm{KIIMap}}^2(\mathbf{S})$ using representatives. 
\begin{figure}[ht!]
			\centering
            \includegraphics[scale=0.8]{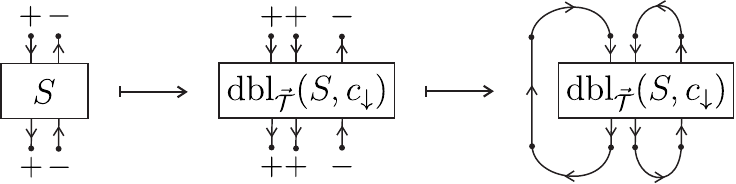}					\caption{Schematic representation of the map $\widetilde{\mathrm{KIIMap}}^2$ using representatives.}
\label{figureEx1_42_abstract} 								
\end{figure}
\end{lemma}

\begin{proof} Let $\mathbf{S}\in\vec{\mathcal T}(\downarrow\uparrow)\subset \underline{\vec{\mathcal T}}(+-,+-)$. By Lemma~\ref{r:2022-10-31-01}$(b)$, we have $\mathrm{dbl}_{\vec{\mathcal T}}(\mathbf{S},c_{\downarrow})\in\vec{\mathcal T}(\downarrow\downarrow\uparrow)\subset \underline{\vec{\mathcal T}}(++-,++-)$, then the result follows since $\underline{\vec{\mathcal T}}$ is a monoidal category.
\end{proof}

\begin{definition}\label{altKIIlema2} Define the map
\begin{equation}\label{eq:def:tildeKII_2map}
\widetilde{\mathrm{KIIMap}}^2:{\vec{\mathcal T}}(\downarrow\uparrow)\longrightarrow \underline{\vec{\mathcal T}}(\emptyset,\emptyset) 
\end{equation}
by mapping $\mathbf{S}\in\vec{\mathcal T}(\downarrow\uparrow)$ to  $\widetilde{\mathrm{KIIMap}}^2(\mathbf{S})\in\underline{\vec{\mathcal T}}(\emptyset,\emptyset)$  given as in Lemma~\ref{r:20221115-01}.
\end{definition}
%


%
\subsubsection{Relation between the two versions of the KII maps}\label{sec:1-6-4}

\begin{lemma} Let
\begin{equation}
    \varphi: {\vec{\mathcal T}}(\downarrow\uparrow) \longrightarrow {\vec{\mathcal T}}(+-,+-)^{\mathrm{ext}} \qquad \text{ and }  \qquad  \psi:  {\vec{\mathcal T}}(\downarrow\downarrow\uparrow) \longrightarrow {\vec{\mathcal T}}(+--,+--)^{\mathrm{ext}}
\end{equation}
be the assignations given as shown in Figure~\ref{figure0924posarxiv}$(a)$ and $(b)$ respectively in terms of diagram representatives. Then $\varphi$ and $\psi$ are well-defined maps and they are bijections.

\begin{figure}[ht!]
			\centering
            \includegraphics[scale=0.8]{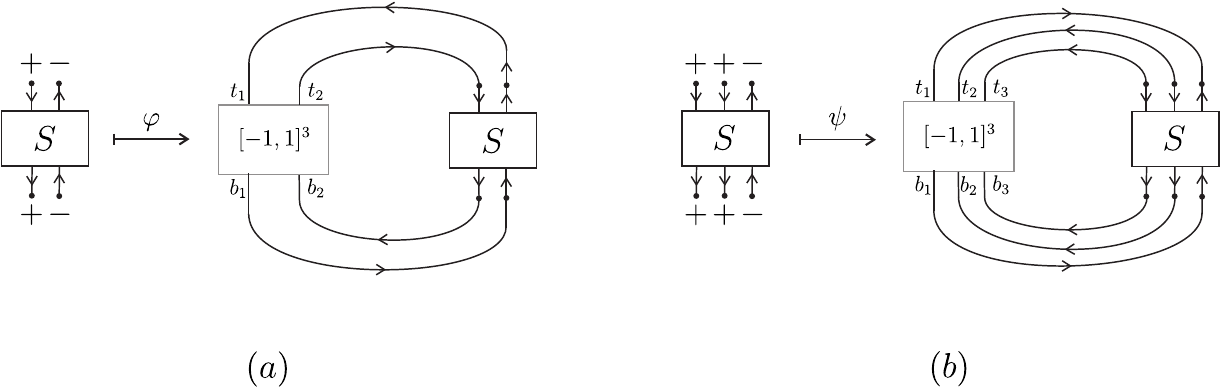} \caption{Schematic representation of the maps   $\varphi: {\vec{\mathcal T}}(\downarrow\uparrow) \rightarrow {\vec{\mathcal T}}(+-,+-)^{\mathrm{ext}}$  and $\psi:  {\vec{\mathcal T}}(\downarrow\downarrow\uparrow) \rightarrow {\vec{\mathcal T}}(+--,+--)^{\mathrm{ext}}$ using representatives.}
\label{figure0924posarxiv} 								
\end{figure}

\end{lemma}

\begin{proof} If $\mathbf{S}$ belongs to  ${\vec{\mathcal T}}(\downarrow\uparrow)$ (resp. ${\vec{\mathcal T}}(\downarrow\downarrow\uparrow)$) and $S$ is a diagram representative $\mathbf{S}$ then  the embedded framed oriented compact 1-manifold  represented by the diagram shown in Figure~\ref{figure0924posarxiv}$(a)$ (resp. $(b)$) is a diagram of an external framed oriented tangle in $[-1,1]^3_{\mathrm{ext}}$ according to Definition~\ref{def:externaltangle}, in particular, notice the compatibility of framings (see Definition~\ref{def:framedtangle}). It follows that its isotopy class $\varphi(\mathbf{S})$ (resp. $\psi(\mathbf{S}$) belongs to  ${\vec{\mathcal T}}(+-,+-)^{\mathrm{ext}}$ (resp. ${\vec{\mathcal T}}(+--,+--)^{\mathrm{ext}}$). This show that the maps are well-defined. The bijectivity property follows form the homeomorphism $S^3\setminus \mathrm{int}([-1,1]^3)\simeq [-1,1]^3$.

\end{proof}

\begin{lemma}\label{relKII_tildeKII} The following diagrams are commutative
\begin{equation*}
\xymatrix{
\ar^-{\widetilde{\mathrm{KIIMap}}^1}[rrr]\ar[d]_{\varphi}^{\simeq} {\vec{\mathcal T}(\downarrow\uparrow)}
& & & \underline{\vec{\mathcal T}}(\emptyset, \emptyset) \\ 
\ar[rrru]_-{\mathrm{KIIMap}^1}{\vec{\mathcal T}(+-,+-)^{\mathrm{ext}}}
& & & \\ 
} \text{\quad and \quad }  \xymatrix{
\ar^-{\widetilde{\mathrm{KIIMap}}^2}[rrr]\ar[d]_{\varphi}^{\simeq} {\vec{\mathcal T}(\downarrow\uparrow)}
& & & \underline{\vec{\mathcal T}}(\emptyset, \emptyset) \\ 
\ar[rrru]_-{\mathrm{KIIMap}^2}{\vec{\mathcal T}(+-,+-)^{\mathrm{ext}}}
& & & \\ 
}
\end{equation*}
\end{lemma}

\begin{proof}
The commutativity of the first (resp. second) diagram is illustrated in Figure~\ref{figure0924-2posarxiv}  (resp. Figure~\ref{figure0924-3posarxiv}) using rerepresentatives.

\begin{figure}[ht!]
			\centering
            \includegraphics[scale=0.8]{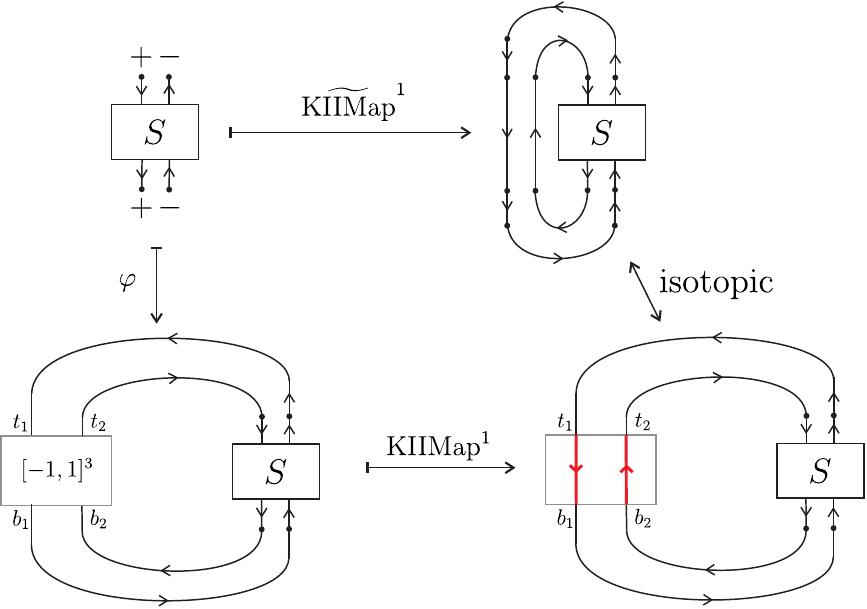} \caption{Relation between $\widetilde{\mathrm{KIIMap}}^1$ and ${\mathrm{KIIMap}}^1$.}
\label{figure0924-2posarxiv} 								
\end{figure}

\begin{figure}[ht!]
			\centering
            \includegraphics[scale=0.8]{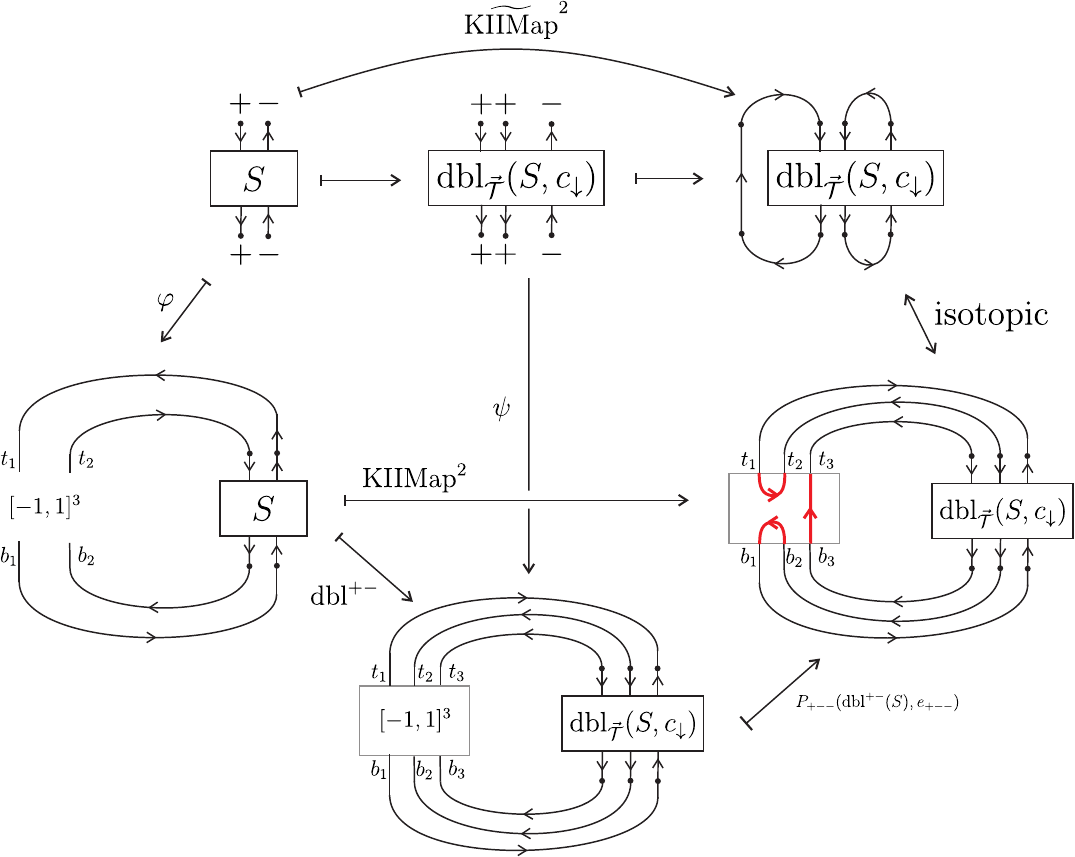} \caption{Relation between $\widetilde{\mathrm{KIIMap}}^2$ and ${\mathrm{KIIMap}}^2$.}
\label{figure0924-3posarxiv} 								
\end{figure}

\end{proof}

\begin{proposition}\label{lemmaA2021.07.08}\label{firstequivKIIT} The set $\mathrm{KII}\vec{\mathcal T}$ from Definition~\ref{def:espaceKIIT}$(b)$ is given by
$$\mathrm{KII}\vec{\mathcal T} = \mathrm{Im}\big(\widetilde{\mathrm{KIIMap}}^1, \widetilde{\mathrm{KIIMap}}^2\big).$$

\end{proposition}

\begin{proof}
This result follows directly from Lemma~\ref{relKII_tildeKII}.
\end{proof}

\subsubsection{The monoidal category $\underline{\mathcal Pa\vec{\mathcal T}}$ of parenthesized framed oriented tangles}\label{sec:1-6-5}

Denote by $\{+,-\}^{(*)}$\index[notation]{{+,-}(*)@$\{+,-\}^{(*)}$}  the {\it free magma} in letters in $\{+,-\}$; elements of $\{+,-\}^{(*)}$ are words in $+,-$ together with 
a parenthesization
(for example $(+-)+\in \{+,-\}^{(*)}$). The operation of forgetting parenthesizations defines a map $f:\{+,-\}^{(*)}\to \{+,-\}^*$.

\begin{definition}\begin{itemize}
\item[$(a)$] Define the set $\mathcal Pa\vec{\mathcal T}$\index[notation]{Pa\vec{\mathcal T}@$\mathcal Pa\vec{\mathcal T}$} of \emph{isotopy classes of parenthesized framed oriented tangles} by
$$\mathcal Pa\vec{\mathcal T}:=\big\{(\mathbf{T},p,p') \ | \  \mathbf{T} \in \vec{\mathcal T}  \text{ and } p \text{ and } p' \text{ are parenthesizations of } \mathrm{s}(\mathbf{T}) \text{ and }   \mathrm{s}(\mathbf{T}) \text{ respectively}\big\}.$$ 
\item[$(b)$] Define maps $\mathrm{s},\mathrm{t}:\mathcal Pa\vec{\mathcal T}\to \{+,-\}^{(*)}$ by $\mathrm{s}(\mathbf{T},p,p'):= (\mathrm{s}(T),p)$ and $\mathrm{t}(\mathbf{T},p,p'):= (\mathrm{t}(T),p')$ for any $(\mathbf{T},p,p')\in \mathcal Pa\vec{\mathcal T}$.
\end{itemize}
\end{definition}

An element in ${\mathcal Pa\vec{\mathcal T}}$ can be represented by an oriented tangle diagram together with an indication of the 
parenthesizations for its source and target. In Figure~\ref{figuraKI1_ell:par}, we show a representative of an element of~${\mathcal Pa\vec{\mathcal T}}$ obtained by endowing with parenthesizations the source and target of the 
element in~${\vec{\mathcal T}}$ represented by the tangle diagram from Figure~\ref{figuraKI1_ell}.  

\begin{figure}[ht!]
										\centering
                        \includegraphics[scale=0.9]{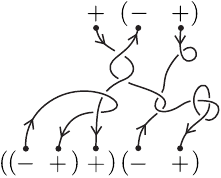}
					
\caption{Representative of an element in ${\mathcal Pa\vec{\mathcal T}}$ with source $((-+)+)(-+)$ and target $+(-+)$.}
\label{figuraKI1_ell:par} 								
\end{figure}

We now define the doubling operation and change of orientation operation for elements of ${\mathcal{P}a\vec{\mathcal T}}$.

\begin{definition}\label{def:DPaT}\begin{itemize}
\item[$(a)$]  Let $(\mathbf{T},p,p') \in \mathcal Pa\vec{\mathcal T}$ and $X \subset \pi_0(\mathbf{T})_{\mathrm{seg}}$, define the element $\mathrm{dbl}_{\mathcal Pa\vec{\mathcal T}}((\mathbf{T},p,p'),X)$\index[notation]{dbl_Pa\vec{\mathcal T}@$\mathrm{dbl}_{\mathcal Pa\vec{\mathcal T}}$} in $\mathcal Pa\vec{\mathcal T}$ by  $\mathrm{dbl}_{\mathcal Pa\vec{\mathcal T}}((\mathbf{T},p,p'),X):=\big(\mathrm{dbl}_{\vec{\mathcal T}},X), p^X, (p')^X\big)$, where $\mathrm{dbl}_{\vec{\mathcal T}}(\mathbf{T},X)=\mathbf{T}^X$ is as in Lemma \ref{lem120} and $p^X$ (resp.~$(p')^X$) is the parenthesization of {$\mathrm{s}(\mathbf{T}^X)$ (resp., $\mathrm{t}(\mathbf{T}^X)$)}  induced by $p$ (resp. $p'$), the map  $\pi_{\mathrm{s}(\mathbf{T}),X\cap\partial_bT}:[\![1,|\mathrm{s}(\mathbf{T}^X)|]\!]\to[\![1,|\mathrm{s}(\mathbf{T})|]\!]$ 
(resp. $\pi_{\mathrm{t}(\mathbf{T}),X\cap\partial_tT}:[\![1,|\mathrm{t}(\mathbf{T}^X)|]\!]\to[\![1,|\mathrm{t}(\mathbf{T})|]\!]$)  and the condition that the letters located in the positions given by a fiber of 
$\pi_{\mathrm{s}(\mathbf{T}),X\cap\partial_bT}$ (resp. $\pi_{\mathrm{t}(\mathbf{T}),X\cap\partial_tT}$) of cardinality~2 do not contain any parenthesis between them. 
\item[$(b)$]  Let $(\mathbf{T},p,p') \in \mathcal Pa\vec{\mathcal T}$ and $a \in \pi_0(\mathbf{T})$, define the element $\mathrm{co}_{\mathcal Pa\vec{\mathcal T}}((\mathbf{T},p,p'),a)\in \mathcal Pa\vec{\mathcal T}$\index[notation]{co_Pa\vec{\mathcal T}@$\mathrm{co}_{\mathcal Pa\vec{\mathcal T}}$} by  $$\mathrm{co}_{\mathcal Pa\vec{\mathcal T}}((\mathbf{T},p,p'),a):=\big(\mathrm{co}_{\vec{\mathcal T}}(\mathbf{T},a), p, p'\big),$$ where $\mathrm{co}_{\vec{\mathcal T}}(\mathbf{T},a)$ is as in Lemma \ref{r:changeoforientationtangles}.
\end{itemize}
\end{definition}

Note that by Lemmas~\ref{lemma:1:19:2jan2020}, \ref{r:2022-10-31-01}$(b)$ and \ref{lemma:1:18:03032020}$(a)$, one has $\mathrm{s}(\mathbf{T}^X)=\mathrm{s}(\mathbf{T})^{X\cap \partial_bT}$, $\mathrm{t}(\mathbf{T}^X)=\mathrm{t}(\mathbf{T})^{X\cap \partial_tT}$.

\begin{definition}
The monoidal category of isotopy classes of parenthesized framed oriented tangles 
$\underline{\mathcal Pa\vec{\mathcal T}}$\index[notation]{Pa\vec{\mathcal T}2@$\underline{\mathcal Pa\vec{\mathcal T}}$} is defined as follows. Its set of objects is the free magma 
$\{+,-\}^{(*)}$ in~$\{+,-\}$. For $u,v\in \{+,-\}^{(*)}$, one has $\underline{\mathcal Pa\vec{\mathcal T}}(u,v):=\underline{\vec{\mathcal T}}(f(u),f(v))$, where $f:\{+,-\}^{(*)}\to \{+,-\}^*$ is the forget-parenthesization map. The composition 
in $\underline{\mathcal Pa\vec{\mathcal T}}$ is defined using the composition in $\underline{\vec{\mathcal T}}$, and the tensor product in  $\underline{\mathcal Pa\vec{\mathcal T}}$ 
is defined using the tensor product in $\underline{\vec{\mathcal T}}$ and the compatibility of $f:\{+,-\}^{(*)}\to \{+,-\}^*$ with the products on both sides. 
\end{definition}

\subsubsection{A description of the subset $\mathrm{KII}\vec{\mathcal T}\subset\underline{\vec{\mathcal T}}(\emptyset,\emptyset)\times \underline{\vec{\mathcal T}}(\emptyset,\emptyset)$ using the category $\underline{\mathcal{P}a\vec{\mathcal T}}$}\label{par:tangles:19fev2020}\label{sec:1-6-6}

Define  the set $\mathcal{P}a\vec{\mathcal T}(\downarrow\uparrow)\subset \underline{\mathcal{P}a\vec{\mathcal T}}((+-),(+-))$\index[notation]{Pa\vec{\mathcal T}(\downarrow\uparrow)@$\mathcal{P}a\vec{\mathcal T}(\downarrow\uparrow)$}
as the  image of $\vec{\mathcal T}(\downarrow\uparrow)\subset \underline{\vec{\mathcal T}}((+-),(+-))$ under the canonical identification $\underline{\vec{\mathcal T}}((+-),(+-))\simeq \underline{\mathcal{P}a\vec{\mathcal T}}((+-),(+-))$. 

\begin{definition} For $i=1,2$, define the maps
\begin{equation}\label{eq:tildeKIImap+-20210804}
\widetilde{\mathrm{KIIMap}}^i_{\underline{\mathcal{P}a\vec{\mathcal T}}}: {\mathcal{P}a\vec{\mathcal T}}(\downarrow \uparrow)\longrightarrow \underline{\vec{\mathcal T}}(\emptyset,\emptyset) 
\end{equation}\index[notation]{KIIMapPaT+-@$\widetilde{\mathrm{KIIMap}}^i_{\underline{\mathcal{P}a\vec{\mathcal T}}}$}  as the composition of the maps $\widetilde{\mathrm{KIIMap}}^i$ from Definition~\ref{def_KIIMap1} and Definition~\ref{altKIIlema2}, with the identification ${\mathcal{P}a\vec{\mathcal T}}(\downarrow \uparrow)\simeq {\vec{\mathcal T}}(\downarrow \uparrow)$. 
\end{definition}

We have therefore the following.
\begin{proposition}\label{secondequivKIIT} The subset $\mathrm{KII}\vec{\mathcal T}\subset \underline{\vec{\mathcal T}}(\emptyset,\emptyset)\times \underline{\vec{\mathcal T}}(\emptyset,\emptyset)$, from Definition~\ref{def:espaceKIIT} can be described as
\begin{equation}
\mathrm{KII}\vec{\mathcal T} =\mathrm{Im}\big((\widetilde{\mathrm{KIIMap}}_{\underline{\mathcal{P}a\vec{\mathcal T}}}^1,\widetilde{\mathrm{KIIMap}}_{\underline{\mathcal{P}a\vec{\mathcal T}}}^2)\big).
\end{equation}
\end{proposition}

We now express the maps $\widetilde{\mathrm{KIIMap}}^i_{\underline{\mathcal{P}a\vec{\mathcal T}}}$ for $i=1,2$ in terms of morphisms in the category~$\underline{\mathcal{P}a\vec{\mathcal T}}$.

\begin{definition}\label{gencupcap_1} Define the elements 
\begin{equation*}
\begin{split}
\mathrm{cup}^{\mathcal T}_{(+-)(+-)}\in\underline{\mathcal{P}a\vec{\mathcal T}}(\emptyset,(+-)(+-)) \quad \quad \text{and} \quad \quad \mathrm{cap}^{\mathcal T}_{(+-)(+-)}\in\underline{\mathcal{P}a\vec{\mathcal T}}((+-)(+-),\emptyset)
\end{split}
\end{equation*}\index[notation]{cup^{\mathcal T}_{(+-)(+-)}@$\mathrm{cup}^{\mathcal T}_{(+-)(+-)}$}\index[notation]{cap^{\mathcal T}_{(+-)(+-)}@$\mathrm{cap}^{\mathcal T}_{(+-)(+-)}$}
as the images of
\begin{equation}
\begin{split}
(\mathrm{Id}_+\otimes\mathrm{cup}^{\mathcal T}_{-+}\otimes\mathrm{Id}_{-})\circ \mathrm{cup}^{\mathcal T}_{+-}\in\underline{\vec{\mathcal T}}(\emptyset,+-+-) &  \quad \text{and} \quad  \mathrm{cap}^{\mathcal T}_{+-}\circ(\mathrm{Id}_+\otimes\mathrm{cap}^{\mathcal T}_{-+}\otimes\mathrm{Id}_{-})\in\underline{\vec{\mathcal T}}(+-+-,\emptyset)
\end{split}
\end{equation}
under the isomorphisms 
\begin{equation}
\begin{split}
\underline{\vec{\mathcal T}}(\emptyset,+-+-)\simeq\underline{\mathcal{P}a\vec{\mathcal T}}(\emptyset,(+-)(+-)) & \quad \text{and} \quad \underline{\vec{\mathcal T}}(+-+-,\emptyset)\simeq\underline{\mathcal{P}a\vec{\mathcal T}}((+-)(+-),\emptyset).
\end{split}
\end{equation}
\end{definition}

\begin{lemma}\label{lemma_B120210709}
\begin{itemize}
\item[$(a)$] There is a map $\underline{\mathcal{P}a\vec{\mathcal T}}((+-),(+-))\longrightarrow \underline{\mathcal{P}a\vec{\mathcal T}}(\emptyset,\emptyset)$ given by


\begin{equation}\label{eq:def:tildeKII_1pat}
\mathbf{S}\longmapsto\mathrm{cap}^{\mathcal T}_{(+-)(+-)}\circ (\mathrm{Id}_{(+-)}\otimes \mathbf{S})\circ  \mathrm{cup}^{\mathcal T}_{(+-)(+-)}
\end{equation}
for $\mathbf{S}\in\underline{\mathcal{P}a\vec{\mathcal T}}((+-),(+-))$. 

\item[$(b)$] The restriction of the composition of the above map with the identification $\underline{\mathcal{P}a\vec{\mathcal T}}(\emptyset,\emptyset) \simeq \underline{\vec{\mathcal T}}(\emptyset,\emptyset)$ to $\mathcal{P}a\vec{\mathcal T}(\downarrow\uparrow)$ is  are equal to the map \eqref{eq:tildeKIImap+-20210804} with $i=1$. 
\end{itemize}
\end{lemma}

\begin{proof}
The part $(a)$ of the lemma is obvious. For $(b)$, notice that elements obtained from formulas \eqref{eq:def:tildeKII_1} and \eqref{eq:def:tildeKII_1pat} are identical. Therefore, the maps are equal. 
\end{proof}

\begin{definition}\label{gencupcap_2} Define the elements 
\begin{equation*}
\begin{split}
\mathrm{cup}^{\mathcal T}_{-((++)-)}\in\underline{\mathcal{P}a\vec{\mathcal T}}(\emptyset,-((++)-)), & \quad \quad  \text{and} \quad \mathrm{cap}^{\mathcal T}_{-((++)-)}\in\underline{\mathcal{P}a\vec{\mathcal T}}(-((++)-),\emptyset)
\end{split}
\end{equation*}\index[notation]{cup^{\mathcal T}_{-((++)-)}@$\mathrm{cup}^{\mathcal T}_{-((++)-)}$}\index[notation]{cap^{\mathcal T}_{-((++)-)}@$\mathrm{cap}^{\mathcal T}_{-((++)-)}$}
as the images of
\begin{equation*}
\begin{split}
\mathrm{cup}^{\mathcal T}_{-+}\otimes \mathrm{cup}^{\mathcal T}_{+-}\in\underline{\vec{\mathcal T}}(\emptyset,-++-), & \quad \quad \text{and} \quad \quad \mathrm{cap}^{\mathcal T}_{-+}\otimes\mathrm{cap}^{\mathcal T}_{+-}\in\underline{\vec{\mathcal T}}(-++-,\emptyset)
\end{split}
\end{equation*}
under the isomorphisms 
\begin{equation*}
\begin{split}
\underline{\vec{\mathcal T}}(\emptyset,-++-)\simeq\underline{\mathcal{P}a\vec{\mathcal T}}(\emptyset,-((++)-)), & \quad \quad \quad \text{and} \quad \quad \underline{\vec{\mathcal T}}(-++-,\emptyset)\simeq\underline{\mathcal{P}a\vec{\mathcal T}}(-((++)-),\emptyset).
\end{split}
\end{equation*}
\end{definition}

\begin{lemma}\label{def:tildeKIImap20210715}
The map $\widetilde{\mathrm{KIIMap}}^2_{\underline{\mathcal{P}a\vec{\mathcal T}}}$  from \eqref{eq:tildeKIImap+-20210804} with $i=2$ is given by
\begin{equation}\label{tildeKIImapPaT}
\mathbf{S}\longmapsto\mathrm{cap}^{\mathcal T}_{-((++-)}\circ (\mathrm{Id}_{-}\otimes \mathrm{dbl}_{1}(\mathbf{S}))\circ \mathrm{cup}^{\mathcal T}_{-((++)-)}\in\underline{\vec{\mathcal T}}(\emptyset,\emptyset),
\end{equation}
where $\mathrm{dbl}_1$ denotes the doubling operation $\mathrm{dbl}_{{\mathcal{P}a\vec{\mathcal T}}}$ applied to the elements of ${\mathcal{P}a\vec{\mathcal T}}(\downarrow\uparrow)$  with respect to the connected component corresponding to~$\downarrow$.  
\end{lemma}

\begin{proof}
The result follows because the tangles obtained from formulas \eqref{tildeKIImap20221113} and \eqref{tildeKIImapPaT} are identical. 
\end{proof}

\subsection{Kirby bialgebra, (semi)Kirby structures, and invariants of 3-manifolds}\label{sec:3.7}

The purpose of this subsection is to express the surgery presentation of 3-manifolds in terms of some algebraic structures and to apply this to the construction of invariants of 3-manifolds.  In~\S\ref{sec:1-7-1}, we describe the Lickorish-Wallace and Kirby theorems in terms of a monoid obtained from the category of tangles, which we call the Kirby monoid.  In~\S\ref{sec:1-7-2}, we apply this to the construction of invariants of $3$-manifolds. Linear versions of these constructions are done in the next subsections: in \S\ref{sec:1-7-3}, the description of~\S\ref{sec:1-7-1} is reformulated in terms of the bialgebra attached to the Kirby monoid (Kirby bialgebra) and the construction of~\S\ref{sec:1-7-2} is reformulated in terms of the notions of (semi)Kirby structures.

\subsubsection{A description of the monoid of homeomorphism classes of 3-manifolds in terms of tangles}\label{sect:1:5:2jan2020}\label{sec:1-7-1}

\begin{definition}\label{def:equivalence-S-over-R}
Let $S$ be a set and $R \subset S\times S$, define $\langle R \rangle$\index[notation]{\langle R \rangle@$\langle R \rangle$} to be the equivalence relation on $S$ defined by: $s \sim s'$ if and only if there exists $k \geq 1$ and $(s_0,\ldots,s_k) \in S^{k+1}$ such that $(s_i,s_{i+1}) \in R \cup R^{\mathrm{op}}$ for $i=0,\ldots,k-1$ with $s_0=s$ and $s_{k}=s'$. Here  $R^\mathrm{op}\subset S\times S$ is the image of $R$ under the map which permutes the two factors in $S\times S$.    The quotient of $S$ by this relation is then denoted $S/\langle R \rangle$.
\end{definition} 

\begin{definition} Let $M$ be a commutative monoid and $R \subset M \times M$. We say that $R$ is \emph{compatible with the monoid structure} if for any $(a,b) \in R$ and any $c \in M$ we have  $(ac,bc) \in R$.
\end{definition}
 
\begin{lemma}\label{lemma-quootient-M-R} Let $M$ be a commutative monoid and $R \subset M \times M$ compatible with the monoid structure then $M/\langle R \rangle$ has the structure of a commutative monoid and the canonical map $M\to M/\langle R\rangle $ is a monoid homomorphism. 
\end{lemma}

\begin{proof}
The result follows straightforwardly.
\end{proof}

Recall from Lemma~\ref{lemma:identif} the graded monoid structure of  $\underline{\vec{\mathcal T}}(\emptyset,\emptyset)$.

\begin{lemma}\label{compatibilityKIITCOTwithmonoidstructure} The subsets  $\mathrm{KII}\vec{\mathcal T} \subset \underline{\vec{\mathcal T}}(\emptyset,\emptyset)\times \underline{\vec{\mathcal T}}(\emptyset,\emptyset)$ and $\mathrm{CO}\vec{\mathcal T} \subset \underline{\vec{\mathcal T}}(\emptyset,\emptyset)\times \underline{\vec{\mathcal T}}(\emptyset,\emptyset)$ are compatible with the monoid structure of $\underline{\vec{\mathcal T}}(\emptyset,\emptyset)$, where $\mathrm{KII}\vec{\mathcal T}$ and $\mathrm{CO}\vec{\mathcal T}$ are as in Definitions~\ref{def:espaceCOT} and~\ref{def:espaceKIIT}, respectively. 
\end{lemma}

\begin{proof}
The result follows from the definitions of these subsets.
\end{proof}

\begin{corollary}\label{the-monoid-X} The set $\mathrm{Kir}:=\underline{\vec{\mathcal T}}(\emptyset,\emptyset)/\langle\mathrm{KII}{\vec{\mathcal T}} \cup \mathrm{CO}{\vec{\mathcal T}}\rangle$\index[notation]{Kir@$\mathrm{Kir}$}  is a commutative monoid,  which we call the \emph{Kirby monoid}. Moreover, $\mathrm{Kir}$ is $\mathbb{Z}_{\geq 0}$-graded, i.e., $\mathrm{Kir} = \bigsqcup_{k\geq 0} \mathrm{Kir}_k$, where 
\begin{equation} 
\mathrm{Kir}_k:=\dfrac{\underline{\vec{\mathcal T}}(\emptyset,\emptyset)_k}{\langle\mathrm{KII}{\vec{\mathcal T}}_k \cup \mathrm{CO}{\vec{\mathcal T}}_k\rangle},
\end{equation}\index[notation]{Kir_k@$\mathrm{Kir}_k$} 
where $\mathrm{KII}{\vec{\mathcal T}}_k$ and $\mathrm{CO}{\vec{\mathcal T}}_k$ are as in Definitions~\ref{def:espaceKIIT-k} and \ref{def:espaceCOT}, respectively.
\end{corollary}

\begin{proof} The result follows from Lemmas~\ref{lemma-quootient-M-R} and~\ref{compatibilityKIITCOTwithmonoidstructure}.
\end{proof}

\begin{lemma}
If $(M,*)$ is a commutative monoid and $m_1,\ldots,m_k \in M$, then the subset $R \subset M \times M$ consisting of pairs $(m,m')$ such that there exist $(a_1,...,a_k)$, $(a'_1,...,a'_k) \in \mathbb{Z}_{\geq 0}^k$ with $m*m_1^{a_1}*...*m_k^{a_k}=m'*m_1^{a'_1}*\ldots*m_k^{a'_k}$ is compatible with the monoid structure. We denote by $M/\mathrm{Eq}(m_1,...,m_k)$ the corresponding quotient monoid given by Lemma~\ref{lemma-quootient-M-R}. 
\end{lemma}

\begin{proof} The result follows straightforwardly. 
\end{proof}

\begin{proposition}\label{categorical:kirby} 
The monoid homomorphism  $\underline{\vec{\mathcal T}}(\emptyset,\emptyset)\to 3\text{-}\mathrm{Mfds}$ from \eqref{dehn:surg:map:monoid} induces an isomorphism of commutative monoids 
$$
\frac{\mathrm{Kir}}{\mathrm{Eq}(U^+,U^-)} \simeq 3\text{-}\mathrm{Mfds},
$$\index[notation]{Eq(U^+,U^-)@$\mathrm{Eq}(U^+,U^-)$}  
where $\mathrm{Kir}$ is the Kirby monoid (see Cororally~\ref{the-monoid-X}) and $U^\pm\in\underline{\vec{\mathcal T}}(\emptyset,\emptyset)_1$ are the $\pm 1$-framed unknots shown Figure \ref{figuresunknots}.
\end{proposition}

\begin{proof}
If follows from the Lickorish-Wallace theorem (\cite{Lick,Wal}) that $\underline{\vec{\mathcal T}}(\emptyset,\emptyset)\to 3\text{-}\mathrm{Mfds}$  is surjective. It follows from  Kirby's theorem 
(see Theorem \ref{thm:kirby}) that this homomorphism factors as a composition of monoid homomorphisms $\underline{\vec{\mathcal T}}(\emptyset,\emptyset)\to  \mathrm{Kir} \to \mathrm{Kir}/\mathrm{Eq}(U^+,U^-) \to 3\text{-}\mathrm{Mfds}$  among which the last morphism $\mathrm{Kir}/\mathrm{Eq}(U^+,U^-) \to 3\text{-}\mathrm{Mfds}$ is injective. All this implies that $\mathrm{Kir}/\mathrm{Eq}(U^+,U^-) \to 3\text{-}\mathrm{Mfds}$ is a monoid isomorphism. 
\end{proof}

\subsubsection{Algebraic construction of invariants of 3-manifolds}\label{sec:1-7-2}

\begin{definition}\label{map:pos-neg-eigenvalues-lkmatrix}
Let $\sigma_{\pm},\sigma_0:\underline{\vec{\mathcal T}}(\emptyset,\emptyset)\to\mathbb{Z}_{\geq 0}$ be the maps which take $x$ to $\sigma_\pm(x)$ and $\sigma_0(x)$, where 
$\sigma_+(x)$ (resp. $\sigma_-(x)$ and $\sigma_0(x)$) is the number of positive (resp. negative and zero) eigenvalues of the linking matrix of $x\in \underline{\vec{\mathcal T}}(\emptyset,\emptyset)$.
\end{definition}

\begin{lemma}\label{sec:1.8lemmalinking}
The maps $\sigma_{\pm},\sigma_0:\underline{\vec{\mathcal T}}(\emptyset,\emptyset)\to\mathbb{Z}_{\geq 0}$ are morphisms of commutative monoids which factor through  morphisms  $\mathrm{Kir}\to\mathbb{Z}_{\geq 0}$, where $\mathrm{Kir}$ is the Kirby monoid (see Corollary~\ref{the-monoid-X}).  Moreover, $\sigma_{\epsilon}(U^{\epsilon'}))=\delta_{\epsilon,\epsilon'}$ for $\epsilon,\epsilon'\in\{+,-\}$.
\end{lemma}

\begin{proof} The linking matrix of a representative $T$ of $\mathbf{T}\in \underline{\vec{\mathcal T}}(\emptyset,\emptyset)_k$ is a symmetric matrix in $\mathrm{Mat}_k(\mathbb Z)$, which undergoes a transformation 
$m\mapsto \ ^tPmP$ with $P\in\mathrm{GL}_k(\mathbb Z)$ both if $T$ undergoes a KII move (see~\cite[\S3.3]{Sav} for more details) or 
a change of orientation (in that case $P$ is diagonal). Therefore, the number of positive (resp. negative and zero) eigenvalues does not change under the equivalence relation~$\langle \mathrm{KII}\vec{\mathcal T}\cup \mathrm{CO}\vec{\mathcal T}\rangle$ induced by $\mathrm{KII}\vec{\mathcal T}\cup \mathrm{CO}\vec{\mathcal T}$. Therefore, they factor through morphisms $\mathrm{Kir}\to\mathbb{Z}_{\geq 0}$. The last statement of the lemma is obvious.
\end{proof}

\begin{lemma} The monoid morphism $\sigma_0$ factors through $\mathrm{Eq}(U^+,U^-)$, and therefore defines a monoid homomorphism $\sigma_0: \mathrm{Kir}/\mathrm{Eq}(U^+,U^-)\to\mathbb{Z}_{\geq 0}$ which fits in the following commutative diagram
$$
\xymatrix{
\dfrac{\mathrm{Kir}}{\mathrm{Eq}(U^+,U^-)}\ar^-{\sigma_0}[r]\ar_-{\simeq}[d] & \mathbb{Z}_{\geq 0}\\
3\text{-}\mathrm{Mfds}\ar[ru]&  
}
$$
where the vertical map is the surgery map  \eqref{dehn:surg:map} and the diagonal map  takes the homeomorphism class of  a $3$-manifold $Y$ to $\mathrm{dim}_{\mathbb{Q}}(H_1(Y,\mathbb{Q}))$. One has therefore $\sigma_0(L)=\mathrm{dim}_{\mathbb{Q}}(H_1(S^3_L,\mathbb{Q}))$ for any framed oriented link~$L$.
\end{lemma}

\begin{proof}
For any framed oriented link $L$ we have $\mathrm{Lk}(L\dot{\sqcup} U^\pm) =\mathrm{Lk}(L)\oplus\{\pm1\}$, therefore $\sigma_0(L\dot{\sqcup} U^\pm) =\sigma_0(L)$ and hence $\sigma_0$ factors through $\mathrm{Eq}(U^+,U^-)$. Besides,  there is a simply-connected $4$-manifold $W_L$ with $\partial W_L = S^3_L$ such that $\mathrm{dim}_{\mathbb{Q}}(H_2(W_L;\mathbb{Q})) = \mathrm{dim}_{\mathbb{Q}}(H_2(W_L, S^3_L;\mathbb{Q}))  =|\pi_0(L)|$
and basis of these modules such that the matrix of the morphism $j_*$ in the short exact sequence 
$$0\longrightarrow H_2(S^3_L;\mathbb{Q})\xrightarrow{\  i_* \ } H_2(W_L;\mathbb{Q})\xrightarrow{\ j_* \ }H_2(W_L,S^3_L;\mathbb{Q})\xrightarrow{\ \delta_*\ } H_1(S^3_L;\mathbb{Q})\longrightarrow 0$$
is given by $\mathrm{Lk}(L)$, see for instance \cite[Theorems 6.1~and~6.2]{Sav}. Hence, $\mathrm{coker}(\mathrm{Lk}(L))\simeq H_1(S^3_L;\mathbb{Q})$. Since $\mathrm{Lk}(L)$ is a symmetric matrix, it is diagonalisable. By definition, the multiplicity of the eigenvalue $0$ is $\sigma_0(L)$, therefore  $\mathrm{dim}(\mathrm{coker}(\mathrm{Lk}(L))=\sigma_0(L)$. Hence,  $\sigma_0(L)=\mathrm{dim}_{\mathbb{Q}}(H_1(S^3_L,\mathbb{Q}))$. All the described arguments do not depend on the isotopy class of $L$.
\end{proof}

\begin{proposition}\label{r1-2022-02-16}\label{cor:invt:3-vars:alg-version1} Let $(M,*)$ be a commutative monoid and $f:\mathrm{Kir}\to M$ be a monoid morphism such that $f(U^\pm)$ are invertible in $M$. Then the map $\widetilde{f}: \mathrm{Kir}\to M$ defined by $$\widetilde{f}(x):=(f(U^+))^{-\sigma_+(x)}
*(f(U^-))^{-\sigma_-(x)}*f(x)$$ is a monoid morphism, which moreover factors through a monoid morphism $\mathrm{Kir}/\mathrm{Eq}(U^+,U^-)\to M$  which by composition with the monoid isomorphism of Proposition~\ref{categorical:kirby} gives rise to a monoid homomorphism $3\text{-}\mathrm{Mfds} \to M$.
\end{proposition}

\begin{proof} The results follows from Lemma~\ref{sec:1.8lemmalinking} and the definition of $\mathrm{Kir}/\mathrm{Eq}(U^+,U^-)$.
\end{proof}

\begin{example}\label{example1} To an  integer $r\geq 3$ and a primitive $4r$-th root of unity, one associates the map $f_0:  \underline{\vec{\mathcal T}}(\emptyset,\emptyset) \to \mathbb{C}$ defined as follows. Let $L\in \underline{\vec{\mathcal T}}(\emptyset,\emptyset)$ and let $D$ be a link diagram of $L$ and $\bar{D}$ the link diagram obtained from $D$ by forgetting the orientations. Let $f_0(L) := \langle \bar{D}^{\omega}\rangle$,
where $\langle \bar{D}^{\omega}\rangle$ is as in \cite[\S8.2]{Ohts} or \cite[Chapter 13]{LickBook}. Then $f_0$ is multiplicative (Kauffman bracket is multiplicative) and factors through a multiplicative map $f: \mathrm{Kir} \to \mathbb{C}$ (see \cite[Lemma 13.5]{LickBook} or \cite[Theorem 8.11]{Ohts}). Moreover, $f(U^{\pm})\not = 0$. Therefore, by Proposition~\ref{r1-2022-02-16}, it induces a monoid homomorphism 
\begin{equation}\label{RT}
\mathrm{RT}_r: 3\text{-Mfds}\to \mathbb{C},
\end{equation}
which coincides with the  \emph{Reshetikhin-Turaev} invariant as defined in  \cite[\S8.2]{Ohts}.
\end{example}

\subsubsection{Kirby bialgebra, (semi)Kirby structures, and invariants of 3-manifolds}\label{sect:1:6:2jan2020}\label{sec:1-7-3}

Let $S$ be a set, then $\mathbb{Z}S$ has a coalgebra structure defined by the condition that any element of $S$ is group-like and with trivial counit. 

Let $S$ be a set and $R \subset S\times S$. Consider the map $R\to \mathbb{Z}S$ given by $(s,s')\mapsto s-s'$. It extends to a morphism of abelian groups $\mathbb{Z}R\to \mathbb{Z}S$. The cokernel of this map is $\mathbb{Z}S/(R)$, where $(R) \subset \mathbb{Z}S$\index[notation]{(R)@$(R)$} is its image. It is equipped with a coalgebra structure, and the projection $\mathbb{Z}S\to \mathbb{Z}S/(R)$ is a morphism of coalgebras. Moreover, the map $S \to \mathbb{Z}S/(R)$ induces a map 
\begin{equation}\label{equ:the-map-S/R-to-ZZS/R}
\frac{S}{\langle R \rangle} \longrightarrow \frac{\mathbb{Z} S}{(R)}
\end{equation} 

If $M$ is a commutative monoid, then $\mathbb{Z}M$ has a bialgebra structure. If moreover $R \subset M\times M$ is compatible with the monoid structure, then $\mathbb{Z}M/(R)$ is a quotient bialgebra of $\mathbb{Z}M$. The map $M \to \mathbb{Z} M/(R)$ is a monoid morphism, which induces a monoid morphism
\begin{equation}\label{equ:the-map-M/R-to-ZZM/R}
\frac{M}{\langle R \rangle} \longrightarrow \frac{\mathbb{Z}M}{(R)}.
\end{equation} 

\begin{lemma}\label{r:2022-04-18-abstractiso}  The map~\eqref{equ:the-map-S/R-to-ZZS/R} induces an isomorphism of coalegbras $\mathbb{Z}(S/\langle R \rangle) \simeq \mathbb{Z}S/(R)$. Similarly, the map~\eqref{equ:the-map-M/R-to-ZZM/R} induces an isomorphism of bialgebras $\mathbb{Z}(M/\langle R\rangle) \to \mathbb{Z}M/(R)$.
\end{lemma}
\begin{proof}
The result follows straightforwardly.
\end{proof}

\begin{definition}\label{def:Kir-and-Kir-k} We call the bialgebra $\mathfrak{Kir}:=\mathbb{Z}\underline{\vec{\mathcal T}}(\emptyset,\emptyset)/(\mathrm{KII}\vec{\mathcal T} \cup \mathrm{CO}\vec{\mathcal T})$\index[notation]{Kir2@$\mathfrak{Kir}$}    the \emph{Kirby bialgebra}. It is a commutative, associative, cocommutative and coassociative bialgebra. As an algebra $\mathfrak{Kir}$ is $\mathbb{Z}_{\geq 0}$-graded
$$\mathfrak{Kir}=\bigoplus_{k\geq 0} \mathfrak{Kir}_k,$$
where  
$$\mathfrak{Kir}_k=\frac{\mathbb{Z}\underline{\vec{\mathcal T}}(\emptyset,\emptyset)_k}{(\mathrm{KII}\vec{\mathcal T}_k\cup\mathrm{CO}\vec{\mathcal T}_k)} = \frac{\mathbb{Z}\underline{\vec{\mathcal T}}(\emptyset,\emptyset)_k}{(\mathrm{KII}\vec{\mathcal T}_k) + (\mathrm{CO}\vec{\mathcal T}_k)}.$$\index[notation]{Kir2_2@$\mathfrak{Kir}_k$} 
Recall that $\mathrm{KII}\vec{\mathcal T}_0 = \mathrm{KII}\vec{\mathcal T}_1 = \mathrm{CO}\vec{\mathcal T}_0=\emptyset$, therefore $(\mathrm{KII}\vec{\mathcal T}_0) = (\mathrm{KII}\vec{\mathcal T}_1) = (\mathrm{CO}\vec{\mathcal T}_0)=\{0\}$. We denote the coproduct in $\mathfrak{Kir}$ by $\Delta_{\mathfrak{Kir}}$.\index[notation]{\Delta_Kir@$\Delta_{\mathfrak{Kir}}$} 
\end{definition}

\begin{lemma} There is a  graded isomorphism of bialgebras $\mathfrak{Kir}\simeq \mathbb{Z} \mathrm{Kir}$, where $\mathrm{Kir}$ is the Kirby monoid (see Corollary~\ref{the-monoid-X}).
\end{lemma}

\begin{proof}
The result is an instance of Lemma~\ref{r:2022-04-18-abstractiso}.
\end{proof}

\begin{definition}\label{def:strB}\label{def:kirby} A \emph{semi-Kirby structure} is a pair $(A,\mu)$ where $\mu : \mathfrak{Kir}\to A$ is an algebra morphism. A semi-Kirby structure  $(A,\mu)$ is called a \emph{Kirby structure} if the elements $\mu(U^\pm) \in A$ are invertible.
\end{definition}

\begin{example}\label{example2} Consider the pair $(\mathbb{C}, \mu_r)$, where $\mu_r:\mathfrak{Kir}\to \mathbb{C}$ is the algebra homomorphism induced by  the map  $f: \mathrm{Kir} \to \mathbb{C}$ from Example~\ref{example1}  (using Lemma~\ref{r:2022-04-18-abstractiso}). Then $(\mathbb{C},\mu_r)$ is a Kirby structure.
\end{example}

\begin{definition}\label{def:catSKSandKS} We define the category of semi-Kirby structures, denoted $\mathcal{SKS}$\index[notation]{SKS@$\mathcal{SKS}$}, as  follows.  The objects of  $\mathcal{SKS}$ are semi-Kirby structures. If $(A,\mu),(B,\nu)\in\mathrm{Ob}(\mathcal{SKS})$, the set of morphisms $\mathcal{SKS}((A,\mu),(B,\nu))$ consists of algebra homomorphism $\phi: A \to B$ such that $\nu=\phi \circ \mu$. The composition in $\mathcal{SKS}$ is induced by the composition of algebra homomorphisms. We denote by $\mathcal{KS}$\index[notation]{KS@$\mathcal{KS}$}  the full subcategory of $\mathcal{SKS}$ of Kirby structures. 
\end{definition}

\begin{lemma}\label{sec:1.8lemma1}\label{sec:1.8lemma3bis} $(a)$ If $(A,\mu) \in \mathrm{Ob}(\mathcal{SKS})$ and $f : A \to B$ is an algebra homomorphism, then $(B,f \circ \mu)$ is in $\mathrm{Ob}(\mathcal{SKS})$, and $f$ induces a morphism  $(A,\mu) \to (B,f \circ \mu)$ in $\mathcal{SKS}$. 

$(b)$ If $(A,\mu) \in \mathrm{Ob}(\mathcal{KS})$ and $f : A\to B$ is an algebra homomorphism, then $(B, f \circ \mu)$ is in $\mathrm{Ob}(\mathcal{KS})$. 

In both cases, we call $(B,f \circ \mu)$ the \emph{push-forward} of $(A,\mu)$  by $f$.
\end{lemma}

\begin{proof}
This follows straightforwardly.
\end{proof}

\begin{lemma}\label{tensorofSKS} A tensor structure is defined on $\mathcal{SKS}$ by $(A,\mu) \otimes (B,\nu):=(A \otimes B,(\mu \otimes \nu) \circ \Delta_{\mathfrak{Kir}})$ at the level of objects and induced by the tensor product of algebra homomorphisms at the level of morphisms. Moreover, $\mathcal{KS}$ is a tensor subcategory of $\mathcal{SKS}$. 
\end{lemma}
\begin{proof}
The first claim is straightforward and the fact that $\mathcal{KS}$ is a tensor subcategory follows from group-likeness of $U^\pm$ in the bialgebra $\mathfrak{Kir}$. 
\end{proof}

\begin{definition}\label{def:mult-inv}
An \emph{algebra-valued invariant} of (closed, oriented) $3$-manifolds is a pair $(A,i)$ where $A$ is an algebra and $i : 3\text{-}\mathrm{Mfds}\to (A,\cdot)$ is a monoid morphism, where $(A,\cdot)$  is the underlying multiplicative monoid of $A$. 
\end{definition}

\begin{definition}\label{def:cat-of-invts}
We define the category $\mathcal{I}nv$\index[notation]{Inv@$\mathcal{I}nv$} of algebra-valued invariants of $3$-manifolds as follows. The objects are algebra-valued invariants of 3-manifolds $(A,i)$ as in Definition~\ref{def:mult-inv}. If $(A,i),(B,j)\in\mathrm{Ob}(\mathcal{I}nv)$, the set of morphisms $\mathcal{I}nv((A,i),(B,j))$ consists of algebra homomorphisms $\phi : A\to B$ such that $j=\phi \circ i$. 
\end{definition}

\begin{lemma}
The category $\mathcal{I}nv$ is a monoidal category with tensor product given by $(A,i) \otimes (B,j):=(A \otimes B ,i \otimes j)$ at the level of objects and induced by the tensor product of algebra homomorphisms at the level of morphisms. 
\end{lemma}
\begin{proof}
This follows straightforwardly.
\end{proof}

\begin{proposition}\label{cor:invt:3-vars:alg} 
There is a monoidal functor $\mathcal{KS} \to \mathcal{I}nv$, taking a Kirby structure $(A,\mu)$ to the algebra-valued invariant of $3$-manifolds $(A,Z_{(A,\mu)})$, where 
$Z_{(A,\mu)}:3\text{-}\mathrm{Mfds}\to (A,\cdot)$  takes the homeomorphism class of a $3$-manifold $Y=S^3_L$ to 
 $\mu(U^+)^{-\sigma_+(L)}\mu(U^-)^{-\sigma_-(L)}\mu(L) \in A$ where $L$  is any surgery presentation of $Y$ and $\sigma_{\pm}$ are as in Definition~\ref{map:pos-neg-eigenvalues-lkmatrix}. In particular, for any $3$-manifold $Y$ and any $(A,\mu),(B,\nu)\in\mathcal{KS}$, we have 
\begin{equation}\label{equ:monoidalityofthefunctorinv}
Z_{((A,\mu) \otimes (B,\nu))}(Y)=Z_{(A,\mu)}(Y) \otimes Z_{(B,\nu)}(Y),
\end{equation}
the equality \eqref{equ:monoidalityofthefunctorinv} is in $A\otimes B$.  Moreover, if  $\phi: (A,\mu)\to (B,\nu)$ is a morphism of Kirby structures and $Y\in 3\text{-}\mathrm{Mfds}$, then $Z_{(B,\nu)}(Y)=\phi(Z_{(A,\mu)}(Y))$ as elements of $B$.    
\end{proposition}

\begin{proof} Let $(A,\mu)$ be a Kirby structure. Then the monoid morphism $Z_{(A,\mu)}$ is well-defined by Proposition~\ref{cor:invt:3-vars:alg-version1}. The assignment $(A,\mu)\to (A,Z_{(A,\mu)})$ therefore takes objects of $\mathcal{KS}$ to objects of $\mathcal{I}nv$. A morphism $\phi: (A,\mu)\to (B,\nu)$ in $\mathcal{KS}$ is an algebra homomorphism $\phi : A \to B$ such that $\nu=\phi \circ \mu$. One then checks that $\phi$ sets up a morphism $(A,Z_{(A,\mu)}) \to (B,Z_{(B,\nu)})$ in $\mathcal{I}nv$, which is defined to be the image of $(A,\mu)\to(B,\nu)$. This defines the announced functor $\mathcal{KS} \to \mathcal{I}nv$ and one then checks that it has the announced monoidal functor property. 
\end{proof}

\begin{example} Let $(\mathbb{C}, \mu_r)$ be the Kirby structure from Example~\ref{example2}. Then the image under the functor $\mathcal{KS}\to \mathcal Inv$  (from Proposition~\ref{cor:invt:3-vars:alg}) is the Reshetikhin-Turaev invariant $\mathrm{RT}_r$ from \eqref{RT}, that is, 
$$Z_{(\mathbb{C},\mu_r)} = \mathrm{RT}_r.$$
\end{example}

\subsection{Kirby structures arising from pre-LMO structures over symmetric monoidal categories}\label{sec:3.8}

In this subsection, we first introduce the notion of pre-LMO structures (\S\ref{sec:1:8:1}). Then  we show that such a structure gives rise to a semi-Kirby structure (\S~\ref{sec:1:8:2}).

\subsubsection{Definition of pre-LMO structure}\label{sec:1:8:1}

Let $\mathcal Set_f$\index[notation]{Set_f@$\mathcal Set_f$}  (resp. $\mathcal Set_f^*$\index[notation]{Set_f^*@$\mathcal Set_f^*$}) be the category whose objects are (resp. pointed) finite sets and set of morphisms given by (resp. pointed) bijections.

\begin{notation}\label{not:twistinX} Let $\tau^{\mathcal Set_f}_{\bullet,\diamond}:\bullet\sqcup \diamond \to \diamond\sqcup \bullet$\index[notation]{tau^{\mathcal Set_f}@$\tau^{\mathcal Set_f}_{\bullet,\diamond}$} denote the natural isomorphism between the functors $\bullet\sqcup\diamond, \diamond\sqcup\bullet:(\mathcal Set_f)^2\to \mathcal Set_f$. in particular, for any finite sets $S$ and $S'$, the map $\tau^{\mathcal Set_f}_{S,S'}:S\sqcup S'\to S'\sqcup S$ is the canonical bijection between $S\sqcup S'$ and $S'\sqcup S$.
\end{notation}

\begin{definition}\label{def:preLMOstructure} A \emph{pre-LMO structure} 
\begin{equation*}
\begin{split}
\mathbf{X}=\Big(\big\{\mathcal X(P,\bullet)\big\}_{P \in \vec{\mathcal  Br}}, \big\{\mathrm{Tens}^{\mathcal X}(P,P')_{\bullet,\bullet'}\big\}_{P,P' \in \vec{\mathcal Br}},  \big\{\mathrm{Comp}^{\mathcal X}(P,P')_{\bullet,\bullet'}\big\}_{P,P' \text{ composable in } \vec{\mathcal Br}},\\
 \big\{\mathrm{co}^{\mathcal X}(P,\bullet)\big\}_{P \in \vec{\mathcal Br}}, \big\{\mathrm{dbl}_{\mathcal X}((P,A),\bullet)\big\}_{P \in \vec{\mathcal Br}, A \subset \pi_0(P)}, Z_{\mathcal X}(\bullet)\Big)
\end{split}
\end{equation*}
consists  of the following. 

\begin{enumerate}
\item[(LMO1)] The data: For each $P\in\vec{\mathcal Br}$ a functor $\mathcal X(P,\bullet):\mathcal Set_f\to \mathcal Vect$.\index[notation]{X(P,\bullet)@$\mathcal X(P,\bullet)$} For each $w\in\{+,-\}^*$ an element $\mathrm{Id}^{\mathcal X}_w\in\mathcal X(\mathrm{Id}_w,\emptyset)$.\index[notation]{Id^{\mathcal X}_w@$\mathrm{Id}^{\mathcal X}_w$} For each pair $P,P'\in\vec{\mathcal Br}$ a natural transformation $\mathrm{Tens}^{\mathcal X}(P,P')_{\bullet, \diamond}: \mathcal X(P,\bullet)\otimes \mathcal X(P',\diamond)\to \mathcal X(P\otimes P',\bullet\sqcup \diamond)$\index[notation]{Tens^{\mathcal X}(P,P')_{\bullet, \diamond}@$\mathrm{Tens}^{\mathcal X}(P,P')_{\bullet, \diamond}$} from the functor $\mathcal X(P,\bullet)\otimes \mathcal X(P',\diamond): \mathcal Set_f\times \mathcal Set_f \to \mathcal Vect$ to the functor $\mathcal X(P\otimes P',\bullet\sqcup \diamond): \mathcal Set_f\times \mathcal Set_f \to \mathcal Vect$. And for every composable  $P,P'\in\vec{\mathcal Br}$  a natural transformation $\mathrm{Comp}^{\mathcal X}(P,P')_{\bullet,\diamond}:\mathcal X(P,\bullet)\otimes \mathcal X(P',\diamond)\to \mathcal X(P'\circ P, \bullet\sqcup \diamond\sqcup \mathrm{Cir}(P',P))$\index[notation]{Comp^{\mathcal X}(P,P')_{\bullet, \diamond}@$\mathrm{Comp}^{\mathcal X}(P,P')_{\bullet, \diamond}$}  from the functor $\mathcal X(P,\bullet)\otimes \mathcal X(P',\diamond): \mathcal Set_f \to \mathcal Vect$ to the functor $\mathcal X(P'\circ P, \bullet\sqcup \diamond\sqcup \mathrm{Cir}(P',P)):\mathcal Set_f \to \mathcal Vect$.

\noindent The conditions:

\begin{itemize}
\item[$(a)$] For any $P\in\vec{\mathcal Br}$, any $S\in\mathcal Set_f$ and any $x\in\mathcal X(P,S)$ we have
$$\mathrm{Comp}^{\mathcal X}(P,\mathrm{Id}_{\mathrm{t}(P)})_{S,\emptyset}(x\otimes \mathrm{Id}_{\mathrm{t}(P)}^{\mathcal X}) = x \qquad \text{and} \qquad \mathrm{Comp}^{\mathcal X}(\mathrm{Id}_{\mathrm{s}(P)}, P)_{\emptyset,S}(\mathrm{Id}_{\mathrm{s}(P)}^{\mathcal X}\otimes x) = x.$$

\item[$(b)$] For any $P,P',P''\in\vec{\mathcal Br}$, we have
\begin{equation*}
\begin{split}
\mathrm{Tens}^{\mathcal X}(P\otimes P',P''&)_{\bullet\sqcup \diamond, \star}\circ \big(\mathrm{Tens}^{\mathcal X}(P,P')_{\bullet, \diamond}\otimes \mathrm{Id}_{\mathcal X(P'',\star)}\big) \\
& = \mathrm{Tens}^{\mathcal X}(P, P'\otimes P'')_{\bullet, \diamond\sqcup \star}\circ \big(\mathrm{Id}_{\mathcal X(P,\bullet)} \otimes\mathrm{Tens}^{\mathcal X}(P',P'')_{\diamond, \star}\big)
\end{split}
\end{equation*}
and
\begin{equation*}
\mathcal X(P,\tau^{\mathcal Set_f}_{\bullet,\diamond})\circ\mathrm{Tens}^{\mathcal X}(\vec{\varnothing},P)_{\bullet,\diamond} =\mathrm{Tens}^{\mathcal X}(P,\vec{\varnothing})_{\diamond,\bullet}\circ\tau^{\mathcal Vect}(\vec{\varnothing}, P)_{\bullet,\diamond}
\end{equation*}
where $\tau^{\mathcal Vect}(P, P')_{\bullet,\diamond}:\mathcal X(P,\bullet) \otimes \mathcal X(P',\diamond) \to   \mathcal X(P',\diamond) \otimes\mathcal X(P,\bullet)$\index[notation]{tau^{\mathcal Vect}(P, P')_{\bullet,\diamond}@$\tau^{\mathcal Vect}(P, P')_{\bullet,\diamond}$} is the natural isomorphism induced by switching components in the tensor product of $\mathcal Vect$.
\item[$(c)$] For  any $P,P',P''\in\vec{\mathcal Br}$ composable, we have
\begin{equation*}
\begin{split}
\mathcal X(P''\circ P'\circ P, \mathrm{Id}_{\bullet\sqcup\diamond\sqcup\star}\sqcup \mathrm{bij}_{P'',P',P})\circ \mathrm{Comp}^{\mathcal X}(P,P''\circ P')_{\bullet,\dagger}\circ\big(\mathrm{Id}_{\mathcal X(P,\bullet)}\otimes \mathrm{Comp}^{\mathcal X}(P',P'')_{\diamond,\star}\big)\\
= \theta\circ \big(\mathrm{Comp}^{\mathcal X}(P'\circ P, P'')_{\bullet\sqcup \diamond\sqcup\mathrm{Cir}(P',P),\star}\big)\circ\big(\mathrm{Comp}^{\mathcal X}(P,P')_{\bullet,\diamond}\otimes \mathrm{Id}_{\mathcal X(P'',\star)}\big)
\end{split}
\end{equation*}
where $\dagger:=\diamond\sqcup \star\sqcup \mathrm{Cir}(P'',P')$ and   $\theta := \mathcal X(P''\circ P'\circ P,\mathrm{Id}_{\bullet\sqcup\diamond}\sqcup \tau^{\mathcal Set_f}_{\mathrm{Cir}(P',P),\star}\sqcup \mathrm{Id}_{\mathrm{Cir}(P'',P'\circ P)})$ (see Notation~\ref{not:twistinX}) and $$\mathrm{bij}_{P'',P',P}: \mathrm{Cir}(P'',P')\sqcup \mathrm{Cir}(P''\circ P',P)\longrightarrow \mathrm{Cir}(P',P)\sqcup \mathrm{Cir}(P'',P'\circ P)$$ is the bijection from Proposition~\ref{r:2022-09-13-bijofcirclesPart}.

\item[$(d)$]  For any $P_i,P_i'\in\vec{\mathcal Br}$ ($i=1,2$) composable, we have
\begin{equation*}
\begin{split}
\mathrm{Comp}^{\mathcal X}(P'_1\otimes P'_2, P_1\otimes P_2)_{\bullet_1\sqcup\bullet_2,\bullet'_1\sqcup\bullet'_2} 
\circ\big(\mathrm{Tens}^{\mathcal X}(P_1,P_2)_{\bullet_1,\bullet_2}\otimes \mathrm{Tens}^{\mathcal X}(P'_1,P'_2)_{\bullet'_1,\bullet'_2}\big)\\
=
\mathcal X((P'_1\otimes P'_2)\circ (P_1\otimes P_2), \mathrm{Id}_{\bullet_1}\sqcup \tau^{\mathcal Set_f}_{\bullet'_1, \bullet_2}\sqcup\mathrm{Id}_{\bullet'_2}\sqcup \mathrm{bij}_{P'_1,P_1,P'_2,P_2})\\
 \circ \mathcal X\big((P'_1\circ P_1)\otimes(P'_2\circ P_2),\mathrm{Id}_{\bullet_1\sqcup\bullet'_1}\sqcup \tau^{\mathcal Set_f}_{\mathrm{Cir}(P'_1,P_1), \bullet_2\sqcup \bullet'_2}\sqcup \mathrm{Id}_{\mathrm{Cir}(P'_2, P_2)}\big)\\
 \circ\big(\mathrm{Tens}^{\mathcal X}(P'_1\circ P_1, P'_2\circ P_2)_{\bullet_1\sqcup\bullet'_1\sqcup\mathrm{Cir}(P'_1,P_1), \bullet_2\sqcup\bullet'_2\sqcup\mathrm{Cir}(P'_2,P_2)}\big)\\ \circ\big(\mathrm{Comp}^{\mathcal X}(P_1,P'_1)_{\bullet_1, \bullet'_1} \otimes \mathrm{Comp}^{\mathcal X}(P_2,P'_2)_{\bullet_2, \bullet'_2}  \big)\circ\big(\mathrm{Id}_{\mathcal X(P_1,\bullet_1)}\otimes \tau^{\mathcal Vect}(P_2,P'_1)_{\bullet_2,\bullet'_1}\otimes \mathrm{Id}_{\mathcal X(P'_2,\bullet'_2)}\big)
\end{split}
\end{equation*}
where $$\mathrm{bij}_{P'_1,P_1,P'_2, P_2}:\mathrm{Cir}(P'_1,P_1)\sqcup \mathrm{Cir}(P'_2,P_2) \longrightarrow \mathrm{Cir}(P'_1\otimes P'_2,P_1\otimes P_2)$$ is the bijection from Lemma~\ref{r:bijectioncomptensorandcompinOBr}. 
\end{itemize}

\item[(LMO2)] The data: For any $P\in\vec{\mathcal Br}$  a natural involution $\mathrm{co}^{\mathcal X}(P,\bullet): \mathcal X(P,\bullet)\to \mathcal X(P,\bullet)$\index[notation]{co^{\mathcal X}(P,\bullet)@$\mathrm{co}^{\mathcal X}(P,\bullet)$}  of the functor $\mathcal X(P,\bullet):\mathcal Set_f^*\to \mathcal Vect$ taking the pointed finite set $(S,s)$ to the vector space $\mathcal X(P,S)$. In particular, for any  $S$ finite set and $s\in S$ there is a linear involution $\mathrm{co}^{\mathcal X}(P,(S,s)): \mathcal X(P,S)\to \mathcal X(P,S)$. 

\noindent  The conditions:

\begin{itemize}
\item[$(a)$] For any $s,s'\in S$, we have $$\mathrm{co}^{\mathcal X}(P,(S,s))\circ \mathrm{co}^{\mathcal X}(P,(S,s')) = \mathrm{co}^{\mathcal X}(P,(S,s'))\circ \mathrm{co}^{\mathcal X}(P,(S,s)).$$

\item[$(b)$] For any finite sets $S$ and $S'$ and $s\in S$, the diagram 
\begin{equation} \label{LMO3:26022020}
\xymatrix{
\mathcal{X}(\vec{\varnothing}, S)\times \mathcal{X}(\vec{\varnothing}, S')\ar_{\mathrm{co}^{\mathcal X}(\vec{\varnothing}, (S,s)) \times\mathrm{Id}_{\mathcal X(\vec{\varnothing}, S')}}[d]
\ar^-{\mathrm{Tens}^{\mathcal X}(\vec{\varnothing},\vec{\varnothing})_{S,S'}}[rrr]& & &\mathcal{X}(\vec{\varnothing}, S\sqcup S')
\ar^{\mathrm{co}^{\mathcal X}(\vec{\varnothing}, (S\sqcup S',s))}[d]\\ 
\mathcal{X}(\vec{\varnothing}, S)\times {\mathcal X}(\vec{\varnothing},S')\ar_-{\mathrm{Tens}^{\mathcal X}(\vec{\varnothing},\vec{\varnothing})_{S,S'}}[rrr]& & &
{\mathcal X}(\vec{\varnothing}, S\sqcup S')}
\end{equation}
commutes, where in the right vertical arrow we identify $s$ with an element of $S \sqcup S'$ via the inclusion  $S\subset S\sqcup S'$.
\end{itemize}

\item[(LMO3)] The data: for any $P\in\vec{\mathcal Br}$ and $A\subset\pi_0(P)$  a natural transformation $$\mathrm{dbl}_{\mathcal X}(P,A)(\bullet): \mathcal X(P,\bullet)\to \mathcal X(\mathrm{dbl}_{\vec{\mathcal Br}}(P,A),\bullet)$$\index[notation]{dbl_{\mathcal X}(P,A)(\bullet)@$\mathrm{dbl}_{\mathcal X}(P,A)(\bullet)$} between the functors $\mathcal X (P,\bullet):\mathcal Set_f\to\mathcal Vect$ and $\mathcal X (\mathrm{dbl}_{\vec{\mathcal Br}}(P,A),\bullet):\mathcal Set_f\to\mathcal Vect$, in particular, for any finite set $S$ there is a linear map $\mathrm{dbl}_{\mathcal X}(P,A)(S): \mathcal X(P,S)\to \mathcal X(\mathrm{dbl}_{\vec{\mathcal Br}}(P,A),S)$. 

\noindent The condition: for any $P,P'\in\vec{\mathcal Br}$ and $A\subset\pi_0(P)$, the diagram of natural transformations
 \begin{equation}\label{diags***23-01-2021}
\xymatrix{
{\mathcal X}(P,\bullet)\times {\mathcal X}(P',\diamond)\ar_{\mathrm{dbl}_{\mathcal X}(P,A)(\bullet)\times\mathrm{Id}_{\mathcal X(P',\diamond)}}[d]
\ar^{\mathrm{Tens}^{\mathcal X}(P,P')_{\bullet,\diamond}}[rr]& &{\mathcal X}(P\otimes P', \bullet\sqcup\diamond)
\ar^{\mathrm{dbl}_{\mathcal X}(P\otimes P', A)(\bullet\sqcup\diamond)}[d]
\\ 
{\mathcal X}(\mathrm{dbl}_{\vec{\mathcal Br}}(P,A),\bullet)\times {\mathcal X}(P',\diamond)\ar_-{\mathrm{Tens}^{\mathcal X}(\mathrm{dbl}_{\vec{\mathcal Br}}(P,A),P')_{\bullet,\diamond}\ \ \ \ \ }[dr]& & {\mathcal X}(\mathrm{dbl}_{\vec{\mathcal Br}}(P\otimes P',A),\bullet\sqcup\diamond)\ar@{=}_{}[dl]\\
 &{\mathcal X}(\mathrm{dbl}_{\vec{\mathcal Br}}(P,A)\otimes P',\bullet\sqcup\diamond) & 
}
\end{equation}
commutes. In the right vertical arrow we identify  $A\subset\pi_0(P)$ as a subset  of $\pi_0(P\otimes P')$, still denoted by $A$, via the inclusion $\pi_0(P)\subset \pi_0(P\otimes P')$. The equality in Diagram~\eqref{diags***23-01-2021}   follows from the equality in Lemma~\ref{lemma:1:18:03032020}$(c)$. 

\item[(LMO4)] The data: an assignment $\mathcal Pa\vec{\mathcal T}\ni\mathbf{T}\mapsto Z_{\mathcal X}(\mathbf{T})\in \mathcal X(\vec{\mathrm{br}}(\mathbf{T}), \pi_0(\mathbf{T})_{\mathrm{circ}})$.

\noindent The conditions: 
\begin{itemize}

\item[$(a)$] For any $w\in\{+,-\}^{(*)}$ we have $Z_{\mathcal X}(\mathrm{Id}_w)= \mathrm{Id}_{f(w)}^{\mathcal X}$, where $f:\{+,-\}^{(*)}\to \{+,-\}^{*}$ is the forgetful parenthesization map. 

\item[$(b)$] For any $\mathbf{T}, \mathbf{T}'\in\mathcal Pa\vec{\mathcal T}$,
$$\mathrm{Tens}^{\mathcal X}(\vec{\mathrm{br}}(\mathbf{T}), \vec{\mathrm{br}}(\mathbf{T}'))_{\pi_0(\mathbf{T})_{\mathrm{cir}}, \pi_0(\mathbf{T}')_{\mathrm{cir}}}(Z_{\mathcal X}(\mathbf{T})\otimes Z_{\mathcal X}(\mathbf{T}')) = Z_{\mathcal X}(\mathbf{T}\otimes \mathbf{T'}),$$
where the ambient vector spaces of both sides are identified using Lemma~\ref{r:2022-10-24-03}$(a)$.

\item[$(c)$] For any composable $\mathbf{T}, \mathbf{T}'\in\mathcal Pa\vec{\mathcal T}$,
$$\mathrm{Comp}^{\mathcal X}(\vec{\mathrm{br}}(\mathbf{T}), \vec{\mathrm{br}}(\mathbf{T}'))_{\pi_0(\mathbf{T})_{\mathrm{cir}}, \pi_0(\mathbf{T}')_{\mathrm{cir}}}(Z_{\mathcal X}(\mathbf{T})\otimes Z_{\mathcal X}(\mathbf{T}')) = Z_{\mathcal X}(\mathbf{T}'\circ \mathbf{T}),$$
where the ambient vector spaces of both sides are identified using Lemma~\ref{r:2022-10-24-03}$(b)$.

\item[$(d)$]  Let $\mathbf{T}\in\mathcal Pa\vec{\mathcal T}$ and  $s\in\pi_0(\mathbf{T})_{\mathrm{cir}}$. Then
$$Z_{\mathcal X}(\mathrm{co}_{\vec{\mathcal T}}(\mathbf{T},s))=\mathrm{co}^{\mathcal X}\big(\vec{\mathrm{br}}(\mathbf{T}),(\pi_0(\mathbf{T})_{\mathrm{cir}},s)\big)(Z_{\mathcal X}(\mathbf{T}))$$
where the ambient vector spaces of both sides in the equality are identified using Lemma~\ref{r:2022-12-05-01}.

\item[$(e)$] Let $\mathbf{T}\in\mathcal Pa \vec{\mathcal T}$ and $A\subset\pi_0(\mathbf{T})_{\mathrm{seg}}\simeq \pi_0(\vec{\mathrm{br}}(\mathbf{T}))$. Then
$$Z_{\mathcal X}(\mathrm{dbl}_{\vec{\mathcal T}}(\mathbf{T},A))=\mathrm{dbl}_{\mathcal X}(\vec{\mathrm{br}}(\mathbf{T}),A)(\pi_0(\mathbf{T})_{\mathrm{cir}})(Z_{\mathcal X}(\mathbf{T})),$$
where the ambient vector spaces of both sides are identified using Lemma~\ref{r:2022-10-31-01}.
\end{itemize}

\end{enumerate}

\end{definition}

\subsubsection{The coinvariant space functor applied to a pre-LMO structure}\label{sec:2-9-2}
 
In this subsection, we fix a pre-LMO structure $\mathbf{X}$.

\begin{definition} Define the category of \emph{vector spaces endowed with a group action}, denoted by $\mathcal Vect\mathcal Gp$, \index[notation]{VectGp@$\mathcal Vect\mathcal Gp$} as follows. The objects are pairs  $(G,V)$ where $G$ is a group and V is a $G$-module. Let $(G,V)$ and $(G',V')$ be objects of  $\mathcal Vect\mathcal Gp$. A morphism  $(G,V) \to (G',V')$ is a pair $(\varphi, f)$. where $\varphi:G \to G'$ is a group homomorphism and $f:V \to V'$ is a linear map, satisfying $f(g\cdot v) = \varphi(g)\cdot f(v)$ for any $g\in G$ and $v\in V$.
\end{definition}

Recall that for $G$ a group acting on a vector space $V$, i.e.   for an object $(G,V)$  of $\mathcal Vect\mathcal Gp$, the space of coinvariants $V_G$ is the quotient of $V$ by the subspace generated by $g\cdot v-v$ where $v \in V$ and $g \in G$. 

\begin{lemma}\label{def:coinvariantspacefunctor}  There is a functor $\mathcal Vect\mathcal Gp \to \mathcal Vect$, which we call the \emph{coinvarant space functor}, defined by $(G,V) \mapsto V_G$ at the level of objects and taking a morphism $(\varphi,f) : (G,V)\to (G',V')$ in $\mathcal Vect \mathcal Gp$ to the unique linear map $f_{\varphi} : V_G \to V'_{G'}$ such that $f_{\varphi} \circ \mathrm{proj}_{(G,V)}=\mathrm{proj}_{(G',V')} \circ f$, where $\mathrm{proj}_{(G,V)} : V \to V_G$ is the canonical  projection. 
\end{lemma}
\begin{proof}
The proof is straightforward. 
\end{proof}

\begin{definition}\label{def:actionPermSonXS} Let $k\in\mathbb{Z}_{\geq 0}$ and $P\in\vec{\mathcal Br}$. The action of $\mathfrak{S}_k$ on $[\![1,k]\!]$  induces an action of $\mathfrak{S}_k$  on the vector space $\mathcal X(P,[\![1,k]\!])$. We set $\mathcal{X}(P)_k:=\mathcal X(P,[\![1,k]\!])_{\mathfrak{S}_k}$ \index[notation]{X(P)_k@$\mathcal{X}(P)_k$} for $k \geq 1$ and $\mathcal X(P)_0:=\mathcal X(P,\emptyset)$. By convention, we denote    $\mathcal X(\vec{\varnothing})_{k}$ by~$\mathcal X_k$. \index[notation]{X_k@$\mathcal{X}_k$}
\end{definition}

\begin{lemma}\label{r:lemmaproj12-29} Let $P\in\vec{\mathcal Br}$ and $S$ be a finite set. The map $\mathrm{proj}^{\mathcal X}_{P,S} : \mathcal X(P,S)\to \mathcal X(P)_{|S|}$\index[notation]{proj^{\mathcal X}_{P,S} @$\mathrm{proj}^{\mathcal X}_{P,S} $} obtained by the composition of the isomorphism $\mathcal X(P,S)\to \mathcal X(P,[\![1,|S|]\!])$ induced by the choice of a bijection $S\simeq [\![1,|S|]\!]$ and of the canonical projection $\mathcal X(P,[\![1,|S|]\!]) \to \mathcal X(P)_{|S|}$ is independent of the choice of such a bijection.
\end{lemma}
\begin{proof}
Direct verification. 
\end{proof}

\begin{lemma}\label{r:morphismsinduitsurlescoinvariants}
Let $P,P' \in \vec{\mathcal Br}$ and $k,k' \in \mathbb{Z}_{\geq 0}$. 
\begin{itemize}
\item[$(a)$]  The pair of the vector space morphism $$\mathrm{Tens}^{\mathcal X}(P,P')_{[\![1,k]\!],[\![1,k']\!]} :\mathcal X(P,[\![1,k]\!]) \otimes \mathcal X(P',[\![1,k']\!])\longrightarrow \mathcal X(P \otimes P',[\![1,k]\!] \sqcup [\![1,k']\!])$$ and of the group morphism $\mathfrak{S}_k \times \mathfrak{S}_{k'} \to \mathfrak{S}_{k+k'}$ induces a linear map $$\mathcal X(P)_k \otimes \mathcal X(P')_{k'} \longrightarrow \mathcal X(P \otimes P')_{k+k'}$$ which will be denoted $x \otimes x' \mapsto x \otimes_{\mathcal X} x'$ for any $x\in \mathcal X(P)_k$ and $x'\in \mathcal X(P')_{k'}$.

\medskip

\item[$(b)$] Assume that $P,P'$ are composable. The pair of the vector space morphism $$\mathrm{Comp}^{\mathcal X}(P,P')_{[\![1,k]\!],[\![1,k']\!]} : \mathcal X(P,[\![1,k]\!]) \otimes \mathcal X(P',[\![1,k']\!]) \longrightarrow \mathcal X(P \otimes P', [\![1,k]\!] \sqcup [\![1,k']\!] \sqcup \mathrm{Cir}(P',P))$$ and of the group morphism $\mathfrak{S}_{[\![1,k]\!]}\times \mathfrak{S}_{[\![1,k']\!]} \to  \mathfrak{S}_{[\![1,k']\!] \sqcup [\![1,k']\!] \sqcup \mathrm{Cir}(P',P)}$ induces a linear map $$\mathcal X(P)_k \otimes \mathcal X(P')_{k'} \longrightarrow \mathcal X(P' \circ P)_{k+k'+|\mathrm{Cir}(P',P)|}$$ which will be denoted $x \otimes x'\mapsto x' \circ_{\mathcal X} x$ for any $x\in \mathcal X(P)_k$ and $x'\in \mathcal X(P')_{k'}$. 

\end{itemize}
\end{lemma}

\begin{proof} The fact the mentioned pairs of vector space and group morphisms are morphisms in the category $\mathcal Vect\mathcal Gp$ follows from the naturality of the transformations $\mathrm{Tens}^{\mathcal X}(P,P')_{\bullet,\diamond} $ and 
$\mathrm{Comp}^{\mathcal X}(P,P')_{\bullet,\diamond}$. 
\end{proof}

\begin{lemma}\label{r:tenscompcoinvprop}  The operations $\otimes_{\mathcal X}$\index[notation]{\otimes_{\mathcal X}@$\otimes_{\mathcal X}$} and $\circ_{\mathcal X}$\index[notation]{\circ_{\mathcal X}@$\circ_{\mathcal X}$} have the following properties: 

\begin{itemize}
\item[$(a)$]  For any $P\in\vec{\mathcal Br}$, any $k\in\mathbb{Z}_{\geq 0}$ and any $x\in\mathcal X(P)_k$ we have 
$$x \circ_{\mathcal X} \mathrm{Id}^{\mathcal X}_{\mathrm{t}(P)}=x=\mathrm{Id}^{\mathcal X}_{\mathrm{s}(P)} \circ_{\mathcal X} x.$$
The equality is in $\mathcal X(P)_k$, where $\mathrm{Id}^{\mathcal X}_{\mathrm{s}(P)} \in \mathcal X(\mathrm{Id}_{\mathrm{s}(P)})_0$ and $\mathrm{Id}^{\mathcal X}_{\mathrm{t}(P)} \in \mathcal X(\mathrm{Id}_{\mathrm{t}(P)})_0$.
\item[$(b)$]  For any $P,P',P''\in\vec{\mathcal Br}$, any $k,k',k''\in\mathbb{Z}_{\geq 0}$ and $x \in \mathcal X(P)_k, x' \in \mathcal X(P')_{k'}, x'' \in \mathcal X(P'')_{k''}$ and $y \in \mathcal X_{k'}$ we have

$$(x \otimes_{\mathcal X} x') \otimes_{\mathcal X} x''=x \otimes_{\mathcal X} (x' \otimes_{\mathcal X} x'') \qquad \quad \text{and} \qquad \quad x \otimes_{\mathcal X}y=y \otimes_{\mathcal X}x.$$

The first equality is  in $\mathcal X(P \otimes P' \otimes P'')_{k+k'+k''}$ and the second one  in $\mathcal X(P)_{k+k'}$.

\item[$(c)$]   For  any $P,P',P''\in\vec{\mathcal Br}$ composable, any $k,k',k''\in\mathbb{Z}_{\geq 0}$ and $x \in \mathcal X(P)_k, x' \in \mathcal X(P')_{k'}$ and $x'' \in \mathcal X(P'')_{k''}$ we have
$$ (x'' \circ_{\mathcal X}x') \circ_{\mathcal X} x=x'' \circ_{\mathcal X}(x' \circ_{\mathcal X} x).$$

The equality is  in $\mathcal X(P'' \circ P' \circ P)_{k+k'+k''+|\mathrm{Cir}(P',P)|+|\mathrm{Cir}(P'',P' \circ P)|}$ (note that $$|\mathrm{Cir}(P',P)|+|\mathrm{Cir}(P'',P' \circ P)|=|\mathrm{Cir}(P'',P')|+|\mathrm{Cir}(P'' \circ P',P)|$$ by Proposition~\ref{r:2022-09-13-bijofcirclesPart}).

\item[$(d)$]    For any $P_i,P_i'\in\vec{\mathcal Br}$ ($i=1,2$) composable, any $k_i, k_i'\in\mathbb{Z}_{\geq 0}$ and any  $x_i \in \mathcal X(P_i)_{k_i}, x'_i \in \mathcal X(P'_i)_{k'_i}$  $(i=1,2) $ we have
$$(x'_1 \otimes_{\mathcal X} x'_2) \circ_{\mathcal X} (x_1 \otimes_{\mathcal X} x_2)=(x'_1 \circ_{\mathcal X}x_1) \otimes_{\mathcal X} (x'_2 \circ_{\mathcal X}x_2).$$

The equality is in $\mathcal X((P'_1 \circ P_1) \otimes (P'_2 \circ P_2))_{k_1+k'_1+|\mathrm{Cir}(P'_1,P_1)|+k_2+k'_2+|\mathrm{Cir}(P'_2,P_2)|}$.

\end{itemize}

\end{lemma}

\begin{proof} The result follows by applying the coinvariant space functor  to the identities from the axiom (LMO1) in Definition~\ref{def:preLMOstructure}.    
\end{proof}

\begin{lemma}\label{r:doublingcoinv}  Let $P \in \vec{\mathcal Br}$ and $A \subset \pi_0(P)$, and $k \in\mathbb{Z}_{\geq 0}$, the map $\mathrm{dbl}_{\mathcal X}(P,A)([\![1,k]\!]) : \mathcal X(P,[\![1,k]\!]) \to \mathcal X(\mathrm{dbl}_{\vec{\mathcal Br}}(P,A),[\![1,k]\!])$ (obtained from axiom (LMO3) in Definition~\ref{def:preLMOstructure}) is $\mathfrak{S}_k$-equivariant; denote by $$\mathrm{dbl}_{\mathcal X}(P,A)_k : \mathcal X(P)_k \longrightarrow \mathcal X(\mathrm{dbl}_{\vec{Br}}(P,A))_k$$ the induced linear map.
\end{lemma}
\begin{proof} The result follows from the naturality of the transformation $\mathrm{dbl}_{\mathcal X}(P,A)(\bullet)$.
\end{proof}

\begin{lemma}\label{r:doublingcoinvprop} Let $P,P' \in \vec{\mathcal Br}$, $A \subset \pi_0(P)$, $k,k' \in\mathbb{Z}_{\geq 0}$ and $x\in \mathcal X(P)_k$, $x' \in \mathcal X(P')_k'$. Then

$$\mathrm{dbl}_{\mathcal X}(P \otimes P',A)_{k+k'}(x \otimes_{\mathcal X}x')
=\mathrm{dbl}_{\mathcal X}(P,A)_{k}(x) \otimes_{\mathcal X}x'.$$

On the left-hand side of the equality we see $A\subset \pi_0(P\otimes P')$ through the canonical inclusion $\pi_0(P)\subset \pi_0(P\otimes P')$. The equality is  in $\mathcal X(\mathrm{dbl}_{\vec{\mathcal Br}}(P,A) \otimes P')_{k+k'} =\mathcal X(\mathrm{dbl}_{\vec{\mathcal Br}}(P\otimes P',A))_{k+k'}$ (notice that $\mathrm{dbl}_{\vec{\mathcal Br}}(P,A) \otimes P' =\mathrm{dbl}_{\vec{\mathcal Br}}(P\otimes P',A)$   by Lemma~\ref{lemma:1:18:03032020}$(c)$).
\end{lemma}

\begin{proof}
The result follows applying the coinvariant space functor to the equality  arising from the commutative diagram~\eqref{diags***23-01-2021} in  (LMO3) with $\bullet=[\![1,k]\!]$ and $\diamond=[\![1,k']\!]$. 
\end{proof}

\subsubsection{Semi-Kirby structure associated to a pre-LMO structure}\label{sec:1:8:2}

In this subsection, we fix a pre-LMO structure $\mathbf{X}$. From now on, If $\mathbf{T}$ is a parenthesized framed oriented  tangle such that $\pi_0(\mathbf{T})_{\mathrm{cir}}=\emptyset$, we see $Z_{\mathcal X}(\mathbf{T})$ as an element of $\mathcal X(\vec{\mathrm{br}}(\mathbf{T}))_0=\mathcal X(\vec{\mathrm{br}}(\mathbf{T}),\emptyset)$, see Definition~\ref{def:actionPermSonXS}.

\begin{definition}\label{r:2022.12.27KIIMaps1} Let $k\in\mathbb{Z}_{\geq 2}$. Define the map $\mathrm{KIIMap}^1_{\mathcal X,k} : \mathcal X(\downarrow \uparrow)_{k-2}\to \mathcal X_k$\index[notation]{KIIMap^1_{\mathcal X,k}@$\mathrm{KIIMap}^1_{\mathcal X,k}$}
by 
\begin{equation}\label{KIIMap1forX} \mathrm{KIIMap}^1_{\mathcal X,k} (x) := Z_{\mathcal X}(\mathrm{cap}^{\mathcal T}_{{(+-)(+-)}}) \circ_{\mathcal X} (\mathrm{Id}_{+-}^{\mathcal X}\otimes_{\mathcal X} x) \circ Z_{\mathcal X}(\mathrm{cup}^{\mathcal T}_{{(+-)(+-)}})
\end{equation}
for any $x\in \mathcal{X}(\downarrow\uparrow)_{k-2}$. The elements $\mathrm{cap}^{\mathcal T}_{(+-)(+-)}$ and $\mathrm{cup}^{\mathcal T}_{(+-)(+-)}$ in $\mathcal Pa\vec{\mathcal T}$ are as  in Definition~\ref{gencupcap_1}.
\end{definition}
Notice that the above map is well-defined by Lemma~\ref{r:morphismsinduitsurlescoinvariants} and the fact that $$|\mathrm{Cir}(\vec{\mathrm{br}}(\mathrm{cap}^{\mathcal T}_{(+-)(+-)}), \vec{\mathrm{br}}(\mathrm{cup}^{\mathcal T}_{(+-)(+-)}))|= 2.$$

\begin{definition}\label{lemma1.80-20210731} Let $k\in\mathbb{Z}_{\geq 2}$. Define the map $\mathrm{KIIMap}^2_{\mathcal X,k} : \mathcal X(\downarrow \uparrow)_{k-2}\to \mathcal X_k$\index[notation]{KIIMap^2_{\mathcal X,k}@$\mathrm{KIIMap}^2_{\mathcal X,k}$}
by 
\begin{equation}\label{KIIMap2forX} \mathrm{KIIMap}^2_{\mathcal X,k} (x) := Z_{\mathcal X}(\mathrm{cap}^{\mathcal T}_{-((++)-))} \circ_{\mathcal X} (\mathrm{Id}_-^{\mathcal X}\otimes_{\mathcal X}\mathrm{dbl}_{\mathcal X}(\downarrow\uparrow,\downarrow)_{k-2}(x) \circ_{\mathcal X} Z_{\mathcal X}(\mathrm{cup}^{\mathcal T}_{-((++)-))})
\end{equation}
for any $x\in \mathcal{X}(\downarrow\uparrow)_{k-2}$.
The elements   $\mathrm{cap}^{\mathcal T}_{-((++)-)}$ and $\mathrm{cup}^{\mathcal T}_{-((++)-)}$ are as in Definition~\ref{gencupcap_2}. 
\end{definition}

\begin{lemma}\label{def:mapCotoXk}  Let  $k\in\mathbb{Z}_{\geq 1}$.  The vector space $\mathcal{X}(\vec{\varnothing},[\![1,k]\!])$  is equipped both with the action of  $\mathfrak{S}_{[\![2,k]\!]}\simeq \mathfrak{S}_{k-1}$ and  $\mathfrak{S}_k$ as in Definition~\ref{def:actionPermSonXS}.  The linear map 
$$\mathrm{co}^{\mathcal X}(\vec{\varnothing},([\![1,k]\!],1)):\mathcal{X}(\vec{\varnothing},[\![1,k]\!])\longrightarrow 
\mathcal{X}(\vec{\varnothing}, [\![1,k]\!])$$
given by axiom \emph{(LMO2)} together with the group  homomorphism $\mathfrak{S}_{k-1} \simeq \mathfrak{S}_{[\![2,k]\!]}\to \mathfrak{S}_k$ determine a morphism in $\mathcal Vect \mathcal Gp$. Define the  linear map  $$\mathrm{co}^{\mathcal X}_k : \mathcal X(\vec{\varnothing},[\![1,k]\!])_{\mathfrak{S}_{k-1}}\to  \mathcal X_k$$\index[notation]{co^{\mathcal X}_k@$\mathrm{co}^{\mathcal X}_k$} to be its image by the coinvariant space functor.
\end{lemma}
\begin{proof}
The result follows from the functor status of the assignment  $\mathcal X(P,\bullet)$ for any $P\in~\vec{\mathcal Br}$ and of the natural transformation status of $\mathrm{co}^{\mathcal X}$.
\end{proof}

\begin{definition}\label{def:projtoXk}  Let $k\in\mathbb{Z}_{\geq 1}$. We denote by 
$\mathrm{proj}^{\mathcal X}_{1,k-1} :  \mathcal X(\vec{\varnothing},[\![1,k]\!])_{\mathfrak{S}_{k-1}} \to \mathcal X_k $\index[notation]{proj^{\mathcal X}_{1,k-1}@$\mathrm{proj}^{\mathcal X}_{1,k-1} $} the projection induced by the injective homomorphism $\mathfrak{S}_{k-1}\simeq\mathfrak{S}_{[\![2,k]\!]}\hookrightarrow\mathfrak{S}_{k}$ determined by the inclusion $[\![2,k]\!]\subset [\![1,k]\!]$.
\end{definition}

\begin{definition}\label{def:spaceKIIabs}\label{spaceCOabs} 
\begin{enumerate}
\item[$(a)$] For $k\geq 2$, define the subspaces $(\mathrm{KII}\mathcal{X}_k)$\index[notation]{KII\mathcal{X}_k@$(\mathrm{KII}\mathcal{X}_k)$}  of $\mathcal{X}_k$ and $(\mathrm{KII}\mathcal{X})$\index[notation]{KII\mathcal{X}@$(\mathrm{KII}\mathcal{X})$} of $\bigoplus_{k\geq 0}\mathcal{X}_k$  by
\begin{equation}\label{spaceKIIabs}
(\mathrm{KII}\mathcal{X}_k) :=\mathrm{Im}\big({\mathrm{KIIMap}}_{{\mathcal{X}},k}^2 - {\mathrm{KIIMap}}_{{\mathcal{X}},k}^1\big) \subset\mathcal{X}_k
\end{equation}
and
\begin{equation}
(\mathrm{KII}\mathcal{X}) := \bigoplus_{k\geq 2} (\mathrm{KII}\mathcal{X}_k) \subset \bigoplus_{k\geq 0}\mathcal{X}_k. 
\end{equation}

\item[$(b)$]   Let  $k\in\mathbb{Z}_{\geq 1}$. Define  the subspaces $({\mathrm{CO}\mathcal{X}_k})$\index[notation]{CO\mathcal{X}_k@$({\mathrm{CO}\mathcal{X}_k})$} of $\mathcal{X}_k$ and $({\mathrm{CO}\mathcal{X}})$\index[notation]{CO\mathcal{X}@$({\mathrm{CO}\mathcal{X}})$} of $\bigoplus_{k\geq 0}\mathcal{X}_k$  by

$${(\mathrm{CO}\mathcal{X}_k)}:= \mathrm{Im}\big(\mathrm{co}^{\mathcal X}_k -  \mathrm{proj}^{\mathcal X}_{1,k-1}\big)
$$
and
$${(\mathrm{CO}\mathcal{X})}:=\bigoplus_{k\geq 1} (\mathrm{CO}\mathcal{X}_k)\subset \bigoplus_{k\geq 0}\mathcal{X}_k.$$

\end{enumerate}
We set $(\mathrm{KII}\mathcal{X}_0) = (\mathrm{KII}\mathcal{X}_1) = (\mathrm{CO}\mathcal{X}_0) = \{0\}$.
\end{definition}

\begin{lemma}\label{lemma:1207}   Let $A:=\bigoplus_{k \geq 0} \mathcal{X}_k$, $M:=\bigoplus_{k \geq 2} \mathcal{X}(\downarrow \uparrow)_{k-2}$,  and let
$f : M\to A$ be  the linear map defined by
$$f:=\bigoplus_{k\geq 2}\big({\mathrm{KIIMap}}_{{\mathcal{X}},k}^2 - {\mathrm{KIIMap}}_{{\mathcal{X}},k}^1\big).$$
Then $A$ is a commutative algebra, $M$ is a right $A$-modules, and  $f$ is a  morphism of right $A$-modules, the algebra $A$ being equipped with its regular module structure over itself. 
\end{lemma}


\begin{proof}  The direct sum over $k,k'\geq 0$ of the linear maps $\mathcal X_k\otimes \mathcal X_{k'}\to \mathcal X_{k+k'}$ (given as in  Lemma~\ref{r:morphismsinduitsurlescoinvariants}$(a)$ with $P=P'=\vec{\varnothing}$)  induces a graded product on $\bigoplus_{k\geq 0}\mathcal X_k$ which  is associative and commutative by Lemma~\ref{r:tenscompcoinvprop}$(b)$. One could also check that the same product structure on $\bigoplus_{k\geq 0}\mathcal X_k$ is induced by the maps $\circ_{\mathcal X}$ given as in  Lemma~\ref{r:morphismsinduitsurlescoinvariants}$(b)$ with $P=P'=\vec{\varnothing}$. 

 The direct sum over $k\geq 2$ and $k'\geq 0$ of the linear maps $\mathcal X(\downarrow \uparrow)_{k-2} \otimes \mathcal X_{k'} \to  \mathcal X(\downarrow \uparrow)_{k+k'-2}$ (given as in Lemma~\ref{r:morphismsinduitsurlescoinvariants}$(a)$ with $P=\downarrow\uparrow$ and $P'=\vec{\varnothing}$) induces a linear map $M \otimes A \to M$. The fact that it defines a right $A$-module structure over $M$ is a consequence of Lemma~\ref{r:tenscompcoinvprop}$(b)$. Let us show that~$f:M\to A$ is a morphism of right $A$-modules.

Let $k\geq 2$, $l\geq 0$, 
$m \in \mathcal X(\downarrow \uparrow)_{k-2}$  and $a \in \mathcal X_l$. By  the definition of $\mathrm{KIIMap}^1_{\mathcal X, k-2+l}$ and Lemma~\ref{r:tenscompcoinvprop}$(b)$  and $(d)$ we have  
\begin{equation*}
\mathrm{KIIMap}_{{\mathcal X},k-2+l}^1(m\otimes_{\mathcal X} a) = \mathrm{KIIMap}_{{\mathcal X},k-2+l}^1(m) \otimes_{\mathcal X} a 
\end{equation*}

Using in addition Lemma~\ref{r:doublingcoinvprop}, we obtain 
\begin{equation*}
\mathrm{KIIMap}_{{\mathcal X},k-2+l}^2(m\otimes_{\mathcal X} a)
=  \mathrm{KIIMap}_{{\mathcal X},k-2}^2(m)\otimes_{\mathcal X} a.
\end{equation*}
This implies  that for $i=1,2$, the map $\bigoplus_k {\mathrm{KIImap}}^i_{\mathcal X, k}$
defines a right $A$-module homomorphism  $M \to A$. One has $f=\bigoplus_k {\mathrm{KIImap}}^1_{\mathcal X, k}  - \bigoplus_k {\mathrm{KIImap}}^2_{\mathcal X, k}$, therefore $f : M \to A$ is a right $A$-module homomorphism.   
\end{proof}

\begin{lemma}\label{1207b}  Consider the  commutative $\mathbb{Z}_{\geq 0}$-graded algebra $A:=\bigoplus_{k \geq 0} \mathcal{X}_k$ (see Lemma~\ref{lemma:1207}). Then
\begin{itemize}
\item[$(a)$] $N:=\bigoplus_k \mathcal X(\vec{\varnothing},[\![1,k]\!])_{\mathfrak{S}_{k-1}}$ is a right $A$-module.

\item[$(b)$]  Recall the maps $\mathrm{proj}^{\mathcal X}_{1,k-1} :  \mathcal X(\vec{\varnothing},[\![1,k]\!])_{\mathfrak{S}_{k-1}}\to \mathcal X_k$ and $\mathrm{co}^{\mathcal X}_{k} :  \mathcal X(\vec{\varnothing},[\![1,k]\!])_{\mathfrak{S}_{k-1}}\to \mathcal X_k$  from Definition~\ref{def:projtoXk} and Lemma~\ref{def:mapCotoXk}, respectively. Then $\bigoplus_{k\geq 1} \mathrm{proj}^{\mathcal X}_{1,k-1} : N \to A$ and  $\bigoplus_k \mathrm{co}^{\mathcal X}_k : N \to A$ are homomorphisms of  $A$-modules.

\item[$(c)$] $\bigoplus_{k\geq 1}  \big(\mathrm{proj}^{\mathcal X}_k - \mathrm{co}^{\mathcal X}_k\big) : N \to A$ is a homomorphism of $A$-modules.
\end{itemize}
\end{lemma}

\begin{proof} 
$(a)$ Let  $k,k'\in\mathbb{Z}_{\geq 0}$. The map $\mathrm{Tens}^{\mathcal X}(\vec{\varnothing},\vec{\varnothing})_{[\![1,k]\!],[\![1,k']\!]}$ is compatible with the  group homomorphism $\mathfrak{S}_{k-1}\times \mathfrak S_{k'} \to \mathfrak S_{k+k'}$ defined by the composition of the inclusion $\mathfrak S_{k-1}\times \mathfrak S_{k'} \subset \mathfrak S_k \times \mathfrak S_{k'}$ with the concatenation group homomorphism $\mathfrak S_k \times \mathfrak S_{k'} \to \mathfrak S_{k+k'}$ and defines therefore a morphism in the category $\mathcal Vect\mathcal Gp$. Taking its image by the coinvariant space functor, we obtain  a linear map $\mathcal X(\vec{\varnothing},[\![1,k]\!])_{\mathfrak{S}_{k-1}}\times \mathcal X_{k'} \to \mathcal X(\vec{\varnothing},[\![1,k]\!])_{\mathfrak{S}_{k+k'-1}}$. The direct sum over $k$ and $k'$ of these maps defines a linear map $N \otimes A \to A$.   The module property of this map follows from the associativity axioms for the natural transformation $\mathrm{Tens}^{\mathcal X}(\vec{\varnothing}, \vec{\varnothing})_{\bullet,\diamond}$, see (LMO1)~$(b)$.

$(b)$ The fact that the map $\bigoplus_{k\geq 1} \mathrm{proj}^{\mathcal X}_{1,k-1} : N \to A$  is a homomorphism of $A$-modules follows from the commutative diagram obtained as the direct sum over $k$ and $k'$ of the images by the coinvariant space functor of the commutative diagram in the category  $\mathcal Vect\mathcal Gp$ with underlying vector spaces and groups diagrams 
\begin{equation*}
\xymatrix{ \mathcal{X}(\vec{\varnothing},[\![1,k]\!])\otimes \mathcal  X(\vec{\varnothing},[\![1,k']\!]) \ar_{\mathrm{Id}\otimes \mathrm{Id}}[d] \ar^-{\mathrm{Tens}^{\mathcal X}(\vec{\varnothing}, \vec{\varnothing})_{[\![1,k]\!],[\![1,k']\!]}}[rrrr]&&&&\mathcal{X}(\vec{\varnothing},[\![1,k+k']\!])\ar^-{\mathrm{Id}}[d]\\
  \mathcal{X}(\vec{\varnothing},[\![1,k]\!])\otimes \mathcal  X(\vec{\varnothing},[\![1,k']\!])\ar_-{\mathrm{Tens}^{\mathcal X}(\vec{\varnothing}, \vec{\varnothing})_{[\![1,k]\!],[\![1,k']\!]}}[rrrr]&&&& \mathcal{X}(\vec{\varnothing},[\![1,k+k']\!])}
\end{equation*}
and
\begin{equation}\label{eq:diagramofgroups12-29}
\xymatrix{ \mathfrak{S}_{k-1}\times \mathfrak{S}_{k'}   \ar^{\mathrm{conc}_{k-1,k'}}[rr] \ar_-{}[d]&&\mathfrak{S}_{k+k'-1}\ar^-{}[d]\\
  \mathfrak{S}_k \times \mathfrak{S}_{k'}\ar_{\mathrm{conc}_{k,k'}}[rr]&& \mathfrak{S}_{k+k'}}
\end{equation}
where   $\mathrm{conc}_{k,l} : \mathfrak{S}_k \times \mathfrak{S}_l \to \mathfrak{S}_{k+l}$\index[notation]{conc_{k,l}@$\mathrm{conc}_{k,l}$} is the concatenation group homomorphism $(\sigma,\tau)\mapsto \sigma*\tau \in \mathfrak{S}_{[\![1,k]\!] \sqcup [\![1,l]\!]}\simeq \mathfrak{S}_{k+l}$, see  Definition~\ref{def:concatenationofperm}. The fact that  $\bigoplus_{k\geq 1} \mathrm{co}^{\mathcal X}_{k} : N \to A$ is a homomorphism of $A$-modules follows from the commutative diagram obtained as the direct sum over $k$ and $k'$ of the images by the coinvariant space functor of the commutative diagrams  in $\mathcal Vect\mathcal Gp$ with underlying groups diagram~\eqref{eq:diagramofgroups12-29} and vector spaces diagram
\begin{equation*}
\xymatrix{ \mathcal{X}(\vec{\varnothing},[\![1,k]\!])\otimes \mathcal  X(\vec{\varnothing},[\![1,k']\!]) \ar_{\mathrm{co}^{\mathcal X}(\vec{\varnothing},([\![1,k]\!],1)) \otimes \mathrm{Id}}[d] \ar^-{\mathrm{Tens}^{\mathcal X}(\vec{\varnothing}, \vec{\varnothing})_{[\![1,k]\!],[\![1,k']\!]}}[rrrr]&&&&\mathcal{X}(\vec{\varnothing},[\![1,k+k']\!])\ar^-{\mathrm{\mathrm{co}^{\mathcal X}(\vec{\varnothing},([\![1,k+k']\!],1))}}[d]\\
  \mathcal{X}(\vec{\varnothing},[\![1,k]\!])\otimes \mathcal  X(\vec{\varnothing},[\![1,k']\!])\ar_-{\mathrm{Tens}^{\mathcal X}(\vec{\varnothing}, \vec{\varnothing})_{[\![1,k]\!],[\![1,k']\!]}}[rrrr]&&&& \mathcal{X}(\vec{\varnothing},[\![1,k+k']\!]).}
\end{equation*}
The latter diagram is commutative because of  axioms (LMO1)~$(b)$ and (LMO2)~$(d)$.

$(c)$ Follows from $(b)$ and the $A$-module homomorphism property of the difference map $A\oplus A\to A$ taking $(x,y)$ to $x-y$ for any $x,y\in A$.

\end{proof}

\begin{proposition}\label{l:LMOinvt} Consider the commutative $\mathbb{Z}_{\geq 0}$-graded algebra $\bigoplus_{k \geq 0} \mathcal{X}_k$ (see Lemma~\ref{lemma:1207}). Then
\begin{itemize} 
\item[$(a)$] The subspaces $(\mathrm{KII}\mathcal{X})$, $(\mathrm{CO}\mathcal{X})$ and $(\mathrm{KII}\mathcal{X}) + (\mathrm{CO}\mathcal{X})$ are $\mathbb{Z}_{\geq 0}$-graded ideals of $\bigoplus_{k \geq 0}\mathcal{X}_k$.
The quotient 
\begin{equation}\label{abstarget} 
\mathfrak{s}(\mathbf{X})=\frac{\bigoplus_{k \geq 0}\mathcal{X}_k}{(\mathrm{KII}\mathcal{X})+(\mathrm{CO}\mathcal{X})}
\end{equation} \index[notation]{s(\mathbf{X})@$\mathfrak{s}(\mathbf{X})$} 
is therefore a $\mathbb{Z}_{\geq 0}$-graded algebra. Explicitly, $\mathfrak{s}(\mathbf{X})=\bigoplus_{k\geq 0} \mathfrak{s}(\mathbf{X})_k$ where
\begin{equation}\label{form:Vk}
\mathfrak{s}(\mathbf{X})_k=\frac{\mathcal{X}_k}{(\mathrm{KII}\mathcal{X}_k) + (\mathrm{CO}\mathcal{X}_k)}.
\end{equation} \index[notation]{s(\mathbf{X})_k@$\mathfrak{s}(\mathbf{X})_k$}

\item[$(b)$] Let $k \in\mathbb{Z}_{\geq 0}$. Let $Z_{\mathcal X,k}: \underline{\vec{\mathcal T}}(\vec{\varnothing},\vec{\varnothing})_k \to \mathcal X_k$ \index[notation]{Z_{\mathcal X,k}@$Z_{\mathcal X,k}$}  be the map such that for any $\mathbf{T} \in \underline{\vec{\mathcal T}}(\vec{\varnothing},\vec{\varnothing})_k$ one has $Z_{\mathcal X,k}(\mathbf{T}):=\mathrm{proj}^{\mathcal X}_{\vec{\varnothing},\pi_0(\mathbf{T})_{\mathrm{cir}}}(Z_{\mathcal X}(\mathbf{T})) $ where $\mathrm{proj}^{\mathcal X}_{\vec{\varnothing},\pi_0(\mathbf{T})_{\mathrm{cir}}} : \mathcal X(\vec{\varnothing},\pi_0(\mathbf{T})_{\mathrm{cir}}) \to \mathcal X_{|\pi_0(\mathbf{T})_{\mathrm{cir}}|}=\mathcal X_k$ is as in Lemma~\ref{r:lemmaproj12-29}. There is a unique linear map $\mu_k: \mathfrak{Kir}_k\to \mathfrak{s}(\mathbf{X})_k$, such that the following diagram of sets commutes,
\begin{equation*}
\xymatrix{
\underline{\vec{\mathcal T}}(\emptyset,\emptyset)_k \ar[d]\ar^{Z_{{\mathcal{X}},k}}[rr] & & \mathcal{X}_k\ar[d]\\
\mathfrak{Kir}_k \ar^{\mu_k}[rr] & &\mathfrak{s}(\mathbf{X})_k,
} 
\end{equation*}
where $\underline{\vec{\mathcal T}}(\emptyset,\emptyset)_k$ is as in Lemma~\ref{lemma:identif}$(c)$. The direct sum over $k \geq 0$ of these maps is an algebra morphism $\mu:\mathfrak{Kir}\to \mathfrak{s}(\mathbf{X})$, i.e, the pair $(\mathfrak{s}(\mathbf{X}),\mu)$ \index[notation]{s(\mathbf{X}),\mu@$(\mathfrak{s}(\mathbf{X}),\mu)$} is a semi-Kirby structure, we call it the \emph{semi-Kirby structure associated to} $\mathbf{X}$. Here $\mathfrak{Kir}$ and $\mathfrak{Kir}_k$ are as in Definition~\ref{def:Kir-and-Kir-k}.
\end{itemize}
\end{proposition}

\begin{proof} 
$(a)$ Let $A:=\bigoplus_{k \geq 0} \mathcal{X}_k$.  By Definition~\ref{def:spaceKIIabs}, $(\mathrm{KII}\mathcal{X})$ and $(\mathrm{CO}\mathcal{X})$ are $\mathbb{Z}_{\geq 0}$-graded subspaces of $A$. Both subspaces  are defined as images of linear maps with target $A$. By Lemmas~\ref{lemma:1207} and~\ref{1207b}, their sources are equipped with $A$-module structures, and the linear maps are $A$-module morphisms, the target being equipped with the regular $A$-module structure. This implies that $(\mathrm{KII}\mathcal{X})$ and $(\mathrm{CO}\mathcal{X})$ are ideals of~$A$.

$(b)$ The map $Z_{\mathcal X,k}$ extends to a morphism of abelian groups $Z_{\mathcal X,k} : \mathbb{Z}\underline{\vec{\mathcal T}}(\emptyset,\emptyset)_k \to \mathcal X_k $. Let us show that $Z_{\mathcal X,k}((\mathrm{KII}\vec{\mathcal T}_k)) \subset (\mathrm{KII}\mathcal{X}_k)$ and $Z_{\mathcal X,k}((\mathrm{CO}\vec{\mathcal T}_k)) \subset (\mathrm{CO}\mathcal X_k)$, which by Definition~\ref{def:Kir-and-Kir-k} implies the existence of a map $\mu_k:\mathfrak{Kir}_k\to \mathfrak{s}(\mathbf{X})_k$.

For $P \in \vec{\mathcal Br}$, there is a map $Z_{\mathcal X, P, k}: \mathcal{P}a\vec{\mathcal T}(P)_k \to \mathcal X(P)_k$ such that for any $\mathbf{T}\in  \mathcal Pa\vec{\mathcal T}(P)_k$ one has $Z_{\mathcal X,k}(\mathbf{T}):=\mathrm{proj}^{\mathcal X}_{P,\pi_0(\mathbf{T})_{\mathrm{cir}}}(Z_{\mathcal X}(\mathbf{T}))$ where $\mathrm{proj}^{\mathcal X}_{P, \pi_0(\mathbf{T})_{\mathrm{cir}}} : \mathcal X(P,\pi_0(\mathbf{T})_{\mathrm{cir}}) \to \mathcal X(P)_{|\pi_0(\mathbf{T})_{\mathrm{cir}}|}=\mathcal X(P)_k$ is as in Lemma~\ref{r:lemmaproj12-29}. The commutativity  of the diagrams
\begin{equation*}
\xymatrix{
\mathcal{P}a\vec{\mathcal T}(\downarrow \uparrow)_{k-2} \ar^{\ \ \ \widetilde{\mathrm{KIImap}}_{\underline{\mathcal{P}a\vec{\mathcal T}},k}^i}[rrr]
\ar_{Z_{{\mathcal{X}},\downarrow\uparrow,k-2}}[d]& & & \underline{\vec{\mathcal T}}(\emptyset,\emptyset)_k
\ar^{Z_{\mathcal{X},k}}[d]\\ 
\mathcal{X}(\downarrow \uparrow)_{k-2}\ar_-{{\mathrm{KIImap}}_{{\mathcal{X}},k}^i}[rrr]
& & &\mathcal{X}_k} \text{\quad\quad and \quad\quad} \xymatrix{
\mathcal{P}a\vec{\mathcal T}(\uparrow \downarrow)_{k-2} \ar^-{\widetilde{\mathrm{KIImap}}_{\underline{\mathcal{P}a\vec{\mathcal T}},k}^i}[rrr]
\ar_{Z_{\mathcal{X},\uparrow\downarrow,k-2}}[d]& & & \vec{\mathcal T}(\emptyset,\emptyset)_k
\ar^{Z_{\mathcal{X},k}}[d]\\ 
\mathcal{X}(\uparrow \downarrow)_{k-2}\ar_{\ \ \ {\mathrm{KIImap}}_{{\mathcal{X}},k}^i}[rrr]
& & &\mathcal{X}_{k}}
\end{equation*}
follows from axioms (LMO4)~$(a)$, $(b)$ and $(c)$ for $i=1$; and (LMO4) $(a)$, $(b)$, $(c)$ and $(e)$ for $i=2$. Here ${\mathrm{KIImap}}_{{\mathcal{X}},k}^i$  for $i=1,2$ are as in Definitions~\ref{r:2022.12.27KIIMaps1} and~\ref{lemma1.80-20210731}.  The above implies the commutativity of the diagram 
\begin{equation}
\xymatrix{
\mathcal{P}a\vec{\mathcal T}(\downarrow \uparrow)_{k-2}  \ar^-{ f_{k}}[rrr]
\ar_{Z_{\mathcal{X},\downarrow\uparrow,k-2}}[d]& & & \mathbb{Z}\vec{\underline{\mathcal T}}(\emptyset,\emptyset)_{k}
\ar^{Z_{\mathcal{X},k}}[d]\\ 
\mathcal{X}(\downarrow \uparrow)_{k-2}\ar_-{g_{k}}[rrr]
& & &\mathcal{X}_k,}
\end{equation}
where
$$f_{k} = \widetilde{\mathrm{KIImap}}_{\underline{\mathcal{P}a\vec{\mathcal T}},k}^2 - \widetilde{\mathrm{KIImap}}_{\underline{\mathcal{P}a\vec{\mathcal T}},k}^1,\quad \quad \text{and}\quad \quad  g_{k} ={\mathrm{KIImap}}_{{\mathcal{X}},k}^2 - {\mathrm{KIImap}}_{{\mathcal{X}},k}^1.$$

By Proposition~\ref{secondequivKIIT}, $(\mathrm{KII}\vec{\mathcal T}_k) =\mathrm{Im}(f_{k})$. The diagram then implies that $Z_{{\mathcal X},k}((\mathrm{KII}\vec{\mathcal T}_k))$ is contained in $\mathrm{Im}(g_{k})$, which is $(\mathrm{KII}\mathcal{X}_k)$.

Besides, for any $\mathbf{T}\in \underline{\vec{\mathcal T}}(\emptyset,\emptyset)_k$ and $c \in \pi_0(\mathbf{T})_{\mathrm{cir}}$, 
one has 
$$Z_{\mathcal X,k}(\mathrm{co}_{\vec{\mathcal T}}(\mathbf{T},c))-Z_{\mathcal X,k}(\mathbf{T}) 
=(\mathrm{co}^{\mathcal X}(\vec{\varnothing},\pi_0(\mathbf{T})_{\mathrm{cir}})-\mathrm{Id})(Z_{\mathcal X,k}(\mathbf{T}))$$
by (LMO4)~$(d)$, therefore $Z_{\mathcal X,k}(\mathrm{co}_{\vec{\mathcal T}}(\mathbf{T},c))-Z_{\mathcal X,k}(\mathbf{T}) \in \mathrm{CO}\mathcal X_k$. 
It follows that $Z_{\mathcal X,k}((\mathrm{CO}\vec{\mathcal T}_k)) \subset (\mathrm{CO}\mathcal X_k)$. This completes the proof of the first statement. The second statement then follows from (LMO1)~$(b)$. 
\end{proof}

\section{Construction of a pre-LMO structure based on the Kontsevich integral}\label{sec:2}

The purpose of this section is to introduce a pre-LMO structure
\begin{equation*}
\begin{split}
\mathring{\mathbf{A}}=\Big(\big\{\mathring{\mathcal A}^\wedge(P,\bullet)\big\}_{P \in \vec{\mathcal  Br}}, \big\{\mathrm{Tens}^{{\mathcal A}}(P,P')_{\bullet,\bullet'}\big\}_{P,P' \in \vec{\mathcal Br}},  \big\{\mathrm{Comp}^{{\mathcal A}}(P,P')_{\bullet,\bullet'}\big\}_{P,P' \text{ composable in } \vec{\mathcal Br}},\\
 \big\{\mathrm{co}^{{\mathcal A}}(P,\bullet)\big\}_{P \in \vec{\mathcal Br}},\big\{\mathrm{dbl}_{{\mathcal A}}((P,A),\bullet)\big\}_{P \in \vec{\mathcal Br}, A \subset \pi_0(P)}, Z_{{\mathcal A}}(\bullet)\Big)
\end{split}
\end{equation*}
see Definition~\ref{def:preLMOstructure}. We start \S\ref{sec:4-0} with some preliminaries about linear and cyclic orders. In \S\ref{sec:4-1} (resp. \S\ref{sec:4-2}, \S\ref{sec:4-3}, \S\ref{sec:4-4} and \S\ref{sec:4-5}), we introduce the functor $\mathring{\mathcal A}(P,\bullet)$ of \emph{Jacobi diagrams} with underlying oriented Brauer diagram $P$ (resp.  $\mathrm{Tens}^{{\mathcal A}}(P,P')_{\bullet,\bullet'}$, $\mathrm{Comp}^{{\mathcal A}}(P,P')_{\bullet,\bullet'}$, $\mathrm{co}^{{\mathcal A}}(P,\bullet)$, $\mathrm{dbl}_{{\mathcal A}}((P,A),\bullet)$).
In \S\ref{sec:4-6}, we verify that these data satisfy the pre-LMO axioms axioms (LMO1), (LMO2) and (LMO3), see Theorem~\ref{r:thmsummarizingLMO123}. In \S\ref{sec:4-7} we introduce a natural transformation  $\big\{\mathrm{co}^{{\mathcal A}}((P,a),\bullet)\big\}_{P \in \vec{\mathcal Br}, a\in\pi_0(P)}$ and define the assignment $Z_{{\mathcal A}}(\bullet)$ (\emph{Kontsevich integral}). We conclude with \S\ref{sec:4-8} proving axiom (LMO4) and therefore that $\mathring{\mathbf{A}}$ is a pre-LMO structure, see Theorem~\ref{thm:prelmokont}.  Further properties of $\mathring{\mathbf{A}}$ are discussed in \S\ref{sec:4-9}.

\subsection{Preliminaries on linear and cyclic orders}\label{sec:4-0}

Recall that a \emph{linear order} (or \emph{total order}) on a set $X$ 
is a subset $L \subset X^2$ (one writes $a<b$ if and only if $(a,b) \in L$), such that: for any $a$, $a<a$ is not true (irreflexivity); for any $a$, $b$, $c$, $a<b$ and $b<c$ implies $a<c$ (transitivity); and for any distinct $a$, $b$, one has $a<b$ or $b<a$.   A set endowed with a linear order is called \emph{linearly ordered set}.

If $(X,<)$ is a linearly ordered set and $Y\subset X$, then the restriction of  $<$ to $Y$ endows it with a linear order. 

We will need the notion of \emph{cyclic order} on a set $X$. When $X$ is finite, a cyclic order on $X$ is the same as: $(a)$ an orbit of the action of the group of cyclic permutations on the set of linear orders on $X$ or $(b)$ a transitive action of $\mathbb{Z}$ on $X$ by permutations, (see \cite[Chapter 1]{viro2008elementary}). More generally, we have the following. 
 
\begin{definition}{\cite[Chapter 1]{viro2008elementary}}\label{def:definitioncyclicorder} \footnote{Condition $(d)$ in Definition~\ref{def:definitioncyclicorder} is present in the Russian edition, but has been mistakenly erased from the English translation.} Let  $X$ be a set. A \emph{cyclic order} on $X$ is a subset $C\subset X^3$, we write $[a<b<c]$ instead of $(a,b,c)\in C$, such that
\begin{itemize}
\item[$(a)$] for any pairwise distinct $a,b,c\in X$, we have either $[a<b<c]$ or $[b<a<c]$, but not both;
\item[$(b)$] for any $a,b,c\in X$ we have $[a<b<c]$ if and only if $[b<c<a]$;
\item[$(c)$] if $a,b,c,d\in X$ are such that $[a<b<c]$ and $[a<c<d]$ then $[a<b<d]$.
\item[$(d)$] if $a,b,c \in X$ are not pairwise distinct, then $[a<b<c]$ does not hold.
\end{itemize}
\end{definition}

Notice that if $|X| \leq 2$, the only cyclic order on $X$ is $\emptyset \subset X^3$. A set $X$ endowed with a cyclic order is called \emph{cyclically ordered set}.

\begin{lemma}\label{r:2023-01-11-r2}
 \begin{itemize}
 
\item[$(a)$] Let $(I,<_{I})$ be a linearly ordered set and let $\{(A_i,<_i)\}_{i \in I}$ be a family of linearly ordered sets indexed by $I$. Then there is a unique linear order  $\prec$ on $\bigsqcup_{i \in I}A_i$, such that for each $i \in I$, its restriction to $A_i$ coincides with $<_i$, and for any $i,j\in I$ with $i<_{I} j$ and any $x\in A_i$ and $y\in A_j$, one has $x\prec y$.
 
\item[$(b)$] Let $(I,[\cdot <\cdot <\cdot])$ be a cyclically ordered set  and let $\{(A_i,<_i)\}_{i \in I}$ be a family of linearly ordered sets indexed by $I$. Let $A:=\bigsqcup_{i \in I}A_i$ and $\pi:A\to I$ the canonical projection ($\pi(a)= i$ if and only if $a\in A_i$). Then there is a unique cyclic order $[\cdot <\cdot <\cdot]_A$ on $A$ such that for any $a,b,c \in A$, one has $[a<b<c]_A$ if and only if one of the following conditions is satisfied: 
\begin{itemize}
\item[$(1)$] $[\pi(a)<\pi(b)<\pi(c)]$, 
\item[$(2)$] $\pi(a)=\pi(b) \neq \pi(c)$ and $a<_{\pi(a)}b$,
\item[$(3)$] $\pi(b)=\pi(c) \neq \pi(a)$ and $b<_{\pi(b)}c$ 
\item[$(4)$] $\pi(c)=\pi(a) \neq \pi(b)$ and $c<_{\pi(c)}a$ 
\item[$(5)$] $\pi(a)=\pi(b)=\pi(c)$  and one of the following requirements is verified $a<_{\pi(a)}b<_{\pi(a)}c$ or $c<_{\pi(a)}a<_{\pi(a)}b$ or $b<_{\pi(a)}c<_{\pi(a)}a$. 
 \end{itemize}
\end{itemize}
\end{lemma}
\begin{proof}
$(a)$ The binary relation $\prec$ on $\bigsqcup_{i \in I}A_i$ is uniquely determined by the conditions of the statement, i.e., if $x,y\in \bigsqcup_{i \in I}A_i$ are such that $x,y\in A_i$ for some $i\in I$, then $x\prec y$ if and only if $x<_i y$ and if $x\in A_i$ and $y\in A_j$ for $i\neq j$ then $x\prec y$ if and only if $i<_I j$. It is straightforward to check that $\prec$ defines a linear order on $\bigsqcup_{i \in I}A_i$.

$(b)$ The ternary relation $[\cdot <\cdot <\cdot]_A$ on $A:=\bigsqcup_{i \in I}A_i$ is uniquely determined by the conditions of the statement. It is easy to check that it defines a cyclic order on $A$. 
\end{proof}

\begin{definition}\label{def:PDISM}
\begin{itemize}
\item[($a$)] A \emph{partially defined injective self-map} (PDISM for short)  \index[notation]{PDISM@PDISM} on a finite set $X$ is a diagram $X \supset D_{\mathrm{succ}}\xrightarrow{\mathrm{succ}} X$, denoted shortly by $\mathrm{succ}$, where $D_{\mathrm{succ}}$ is a subset of $X$ (the \emph{domain of definition} of  $\mathrm{succ}$) and $\mathrm{succ}: D_{\mathrm{succ}} \to X$ is an injective map. 
\item[($b$)] Let  $\mathrm{succ}$ be a PDISM on a finite set $X$. We set $X_0:=X\setminus D_{\mathrm{succ}}$, that is $X_0$ is the subset of $X$ of elements for which the image of the map $\mathrm{succ}$ is not defined. 
\end{itemize}
\end{definition}

\begin{lemma}\label{r:2023-02-14-r1}
Let $X$ be a finite set and $\mathrm{succ}$ be a \emph{PDISM} on $X$ and $X_0\subset X$ as in Definition~\ref{def:PDISM}. For $x \in X$, let $N(x)$ be the set $\{n \in\mathbb{Z}_{\geq 0} \ | \ \mathrm{succ}^n(x) \text{ is defined} \}$.  Define $\mathrm{orb}(x) \subset X$ as $$\mathrm{orb}(x):=\{\mathrm{succ}^n(x)\ | \ n \in N(x)\} \cup \{ y \in X \ | \  \text{there exists } k\in \mathbb{Z}_{\geq 0} \text{ such that }\mathrm{succ}^k(y)=x\},$$
so $\mathrm{orb}(x) \in \mathcal P(X)$. This defines a map $\mathrm{orb} : X \to \mathcal P(X)$. \index[notation]{orb@$\mathrm{orb}$} We have 
\begin{itemize}
\item[$(a)$] $\mathrm{orb}(X)\subset \mathcal{P}(X)$ is a partition of $X$. 

\item[$(b)$] Let $\mathrm{orb}_{\mathrm{lin}}(X):=\{\omega \in \mathrm{orb}(X)\ | \ \omega \cap X_{0} \neq \emptyset\}$\index[notation]{orb_lin@$\mathrm{orb}_{\mathrm{lin}}$} and  $\mathrm{orb}_{\mathrm{cyc}}(X):=\{\omega \in \mathrm{orb}(X) \ | \ \omega \cap X_{0}=\emptyset\}$.\index[notation]{orb_cyc@$\mathrm{orb}_{\mathrm{cyc}}$} Then  $\mathrm{orb}(X)=\mathrm{orb}(X)_{\mathrm{lin}} \sqcup \mathrm{orb}_{\mathrm{cyc}}(X)$. One also has $\cup_{\omega \in \mathrm{orb}_{\mathrm{lin}}(X)}\omega=\{x \in X\ | \ N(x) \text{ is finite}\}$ and $\cup_{\omega \in \mathrm{orb}_{\mathrm{cyc}}(X)}\omega=\{x \in X\ | \ N(x) \text{ is infinite}\}$.

\item[$(c)$] Let $\omega \in \mathrm{orb}_{\mathrm{lin}}(X)$, then $|\omega\cap X_{0}|=1$, we denote by $\omega_{0}$ the only element in $\omega\cap X_{0}$. The set $\omega$ (viewed as a subset of $X$) is equipped with a linear order $<_{\omega}$, such that for each $x \in \omega\setminus X_{0}$ one has $x <_{\omega} \mathrm{succ}(x)$ and $x<_{\omega}\omega_{0}$.
 
\item[$(d)$] For $\omega \in \mathrm{orb}_{\mathrm{cyc}}(X)$, the restriction to $\omega$ (viewed as a subset of $X$) of the map $\mathrm{succ}$ defines a transitive action of $\mathbb{Z}$ and therefore equips $\omega$ with a cyclic order.
\end{itemize}
\end{lemma}

\begin{proof}
$(a)$ Clearly we have $X=\cup_{x}\mathrm{orb}(x)$. One checks that  if $\mathrm{orb}(x)\cap \mathrm{orb}(x')\not = \emptyset$ then $\mathrm{orb}(x)=\mathrm{orb}(x')$.

$(b)$ Follows directly from the definitions of $\mathrm{orb}_{\mathrm{lin}}(X)$ and $\mathrm{orb}_{\mathrm{cyc}}(X)$.

$(c)$ The fact that $|\omega\cap X_{0}|=1$ follows from the injectivity of $\mathrm{succ}$. One then has $\omega =\{\mathrm{succ}^{-i}(\omega_{0}) \ | \ i\in[\![0,|\omega|-1]\!]\}$ which is clearly equipped with a linear order as stated in the lemma.

$(d)$ Follows straightforwardly. 
\end{proof}

\begin{lemma}\label{r:2023-02-15r1}  Let $X$ be a finite set , let $\mathrm{succ}$ be a \emph{PDISM} on $X$ and let  $X_0\subset X$ be as in Definition~\ref{def:PDISM}. Set $Y:=X\setminus X_0$, $Y_0:=\mathrm{succ}^{-1}(X_0)$.
\begin{itemize}
\item[$(a)$] There is a unique \emph{PDISM} $\mathrm{succ}_Y$ on $Y$ with domain of definition $Y \setminus Y_0$, given by   $\mathrm{succ}_Y(y):=\mathrm{succ}(y)$ for all $y\in Y\setminus Y_0$. 

\item[$(b)$]  Let $\mathrm{orb}(X)$ (resp. $\mathrm{orb}(Y)$) be the partition of $X$ (resp. $Y$) and $\mathrm{orb}(Y)=\mathrm{orb}_{\mathrm{lin}}(Y) \sqcup \mathrm{orb}_{\mathrm{cyc}}(Y)$  (resp. $\mathrm{orb}(X)=\mathrm{orb}_{\mathrm{lin}}(X) \sqcup \mathrm{orb}_{\mathrm{cyc}}(X)$) be its decomposition obtained by applying Lemma~\ref{r:2023-02-14-r1} to $(X,X_0,\mathrm{succ})$ (resp. $(Y,Y_0,\mathrm{succ}_Y)$).  There are canonical bijections $\mathrm{orb}_{\mathrm{cyc}}(X) \simeq \mathrm{orb}_{\mathrm{cyc}}(Y)$ and $$\mathrm{orb}_{\mathrm{lin}}(X) \simeq \mathrm{orb}_{\mathrm{lin}}(Y) \sqcup \{x \in X_0\ | \ \mathrm{succ}^{-1}(x)=\emptyset\}.$$    
 
\item[$(c)$] There is a canonical  bijection between the underlying set of each element of $\mathrm{orb}_{\mathrm{cyc}}(Y)$ and its image in $\mathrm{orb}_{\mathrm{cyc}}(X)$, compatible with the cyclic orders; 
and there is an injection of the underlying set of each element of $\mathrm{orb}_{\mathrm{lin}}(Y)$ to that of its image in $\mathrm{orb}_{\mathrm{lin}}(X)$, compatible with the linear orders. 
\end{itemize}
\end{lemma}
\begin{proof}
$(a)$ follows straightforwardly. $(b)$ We can show that $\mathrm{orb}_{\mathrm{cyc}}(X) = \mathrm{orb}_{\mathrm{cyc}}(Y)$. Indeed, if $\omega\in\mathrm{orb}_{\mathrm{cyc}}(X) $, then $\omega =\mathrm{orb}(x)$ for some $x\in X$, since $\omega\cap X_0 =\mathrm{orb}(x)\cap X_0=\emptyset$ then $x\in Y$ and clearly $\omega \cap Y_0=\mathrm{orb}(x)\cap Y_0=\emptyset$. Therefore  $\omega =\mathrm{orb}(x)\in \mathrm{orb}_{\mathrm{cyc}}(Y)$. The converse inclusion is evident. Let us show the second statement in $(b)$, if $\omega\in \mathrm{orb}_{\mathrm{lin}}(Y)$, then $\omega =\mathrm{orb}(y)$ with $y\in Y$. Since $\omega\cap Y_0 =\mathrm{orb}(y)\cap Y_0\not = \emptyset$ , then there exist $y_0\in Y_0$ such that  $\omega =\mathrm{orb}(y_0)$. It follows that $\omega\cup\{\mathrm{succ}(y_0)\}\in \mathrm{orb}(X)$ and $(\omega\cup\{\mathrm{succ}(y_0)\})\cap X_0\not =\emptyset$, that is, $\omega\cup\{\mathrm{succ}(y_0)\}\in \mathrm{orb}_{\mathrm{lin}}(X)$. This defines an injective map  $\mathrm{orb}_{\mathrm{lin}}(Y)\to \mathrm{orb}_{\mathrm{lin}}(X)$ with image $\mathrm{orb}_{\mathrm{lin}}(X)\setminus \{x\in X_0 \ | \ \mathrm{succ}^{-1}(x)=\emptyset\}$.  $(c)$ follows from $(b)$.
\end{proof}

\subsection{The functor \texorpdfstring{$\mathring{\mathcal A}^\wedge(P,\bullet)$}{A(P,.)} of Jacobi diagrams}\label{sec:4-1}

\begin{definition}
 A \emph{vertex-oriented unitrivalent graph} is a finite graph whose vertices are univalent or trivalent, and such that for each trivalent vertex the set of half-edges incident to it is cyclically ordered. {We allow looped edges (without vertices).} For $D$ is a vertex-oriented unitrivalent graph, we denote by $\partial D$ its set of univalent vertices. 
\end{definition}

\begin{definition}\label{def:thesetJacPS} Let $P\in\vec{\mathcal Br}$ be an oriented Brauer diagram  (see Definition~\ref{def:vecBr}) and let $S$ be a finite set. A \emph{Jacobi diagram on $(P,S)$} is a class of a tuple $\underline{D}=(D,\varphi,\{\mathrm{lin}_l\}_{l \in \pi_0(P)}, \{\mathrm{cyc}_s\}_{s \in S})$  where $D$ is a vertex-oriented unitrivalent graph ({loops without vertices are allowed}), $\varphi : \partial D\to \pi_0(P) \sqcup S$ is a map,  and for each $l \in \pi_0(P)$, $\mathrm{lin}_l$ is a linear order on $\varphi^{-1}(l)$ and for each $s \in S$, $\mathrm{cyc}_s$ is a cyclic order on $\varphi^{-1}(s)$.  Two such tuples $\underline{D}=(D,\varphi,\{\mathrm{lin}_l\}_{l \in \pi_0(P)}, \{\mathrm{cyc}_s\}_{s \in S})$ and $\widetilde{\underline{D}}=(\widetilde{D},\widetilde{\varphi},\{\widetilde{\mathrm{lin}}_l\}_{l \in \pi_0(P)}, \{\widetilde{\mathrm{cyc}}_s\}_{s \in S})$ are considered equivalent if there exists a graph isomorphism $\psi: D\to \widetilde{D}$ such that $\widetilde{\varphi}\circ \psi_{|\partial D} = \varphi$ and inducing isomorphisms of linear and cyclic ordered sets  $(\varphi^{-1}(l),\mathrm{lin}_l)\cong  (\widetilde{\varphi}^{-1}(l),\widetilde{\mathrm{lin}}_l)$ and $(\varphi^{-1}(s),\mathrm{cyc}_s)\cong  (\widetilde{\varphi}^{-1}(s),\widetilde{\mathrm{cyc}}_s)$ for any $l\in\pi_0(P)$   and any $s\in S$. We denote by ${\mathring{\mathrm{Jac}}}(P,S)$ the set of Jacobi diagrams on $(P,S)$.  \index[notation]{Jac(P,S)@${\mathring{\mathrm{Jac}}}(P,S)$}
\end{definition}

\begin{definition}\label{def:emptyJacdiagram2} Let $P\in\vec{\mathcal Br}$ be an oriented Brauer diagram  and let $S$ be a finite set.  We denote by $\emptyset_{(P,S)}$ the Jacobi diagram on $(P,S)$ given by the  tuple  with ${D}=\emptyset$. \index[notation]{\emptyset_{(P,S)}@$\emptyset_{(P,S)}$}
\end{definition}

In order to represent a Jacobi diagram $\underline{D}=(D,\varphi,\{\mathrm{lin}_l\}_{l \in \pi_0(P)}, \{\mathrm{cyc}_s\}_{s \in S})$ on $(P,S)$ by a drawing, we use the schematic representation of $P$  (see \S\ref{sec:2-8}) using solid lines,  we draw with solid lines an oriented circle labelled by $s$ for each element $s\in S$, we use dashed lines to represent $D$, and the map $\varphi:\partial D\to \pi_0(P)\sqcup S$ is indicated by attaching the univalent vertices to the solid lines so that the linear orders and cyclic orders coincide with the orders induced by the orientations of the solid lines. The vertex orientation of trivalent vertices of $D$ is counterclockwise by convention. In particular, the schematic representation of $\emptyset_{(P,S)}$ coincides with the schematic representation of the oriented Brauer diagram $P$. See Figure~\ref{figuraJD1_ell} for more examples.
\begin{figure}[ht!]
										\centering
                        \includegraphics[scale=0.8]{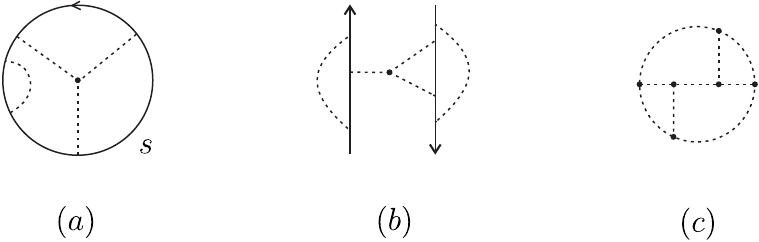}
												\caption{Jacobi diagrams with $(P, S)=( \vec{\varnothing}, \{s\})$ in $(a)$, $(P, S)=( \uparrow \downarrow, \emptyset)$ in $(b)$ and $(P, S)=(\vec{\varnothing}, \emptyset)$ in $(c)$.}
\label{figuraJD1_ell} 										
\end{figure}

\begin{definition}\label{themapm}  Let $P$ be an oriented Brauer diagram and $S$  a finite set. For $x\in\pi_0(P)\sqcup S$, define the map  
$$\mathrm{m}_x : \mathring{\mathrm{Jac}}(P,S)\longrightarrow \mathbb{Z}_{\geq 0}, \quad \quad \quad \underline{D}\mapsto \mathrm{m}_x(\underline{D})$$ \index[notation]{m_x@$\mathrm{m}_x$}
where $\mathrm{m}_x(\underline{D}):=|\varphi^{-1}(x)|$ for any $\underline{D}=(D,\varphi,\{\mathrm{lin}_l\}_{l \in \pi_0(P)}, \{\mathrm{cyc}_s\}_{s \in S})\in \mathring{\mathrm{Jac}}(P,S)$.
\end{definition}

\begin{definition}\label{thespaceAPS}
Let ${P}$ be an oriented Brauer diagram and $S$ a finite set. The \emph{space of Jacobi diagrams on} $({P},S)$ is the $\mathbb{C}$-vector space generated by the set ${\mathring{\mathrm{Jac}}}(P,S)$ of Jacobi diagrams on $(P,S)$ up to the STU, AS, and IHX relations:
$$
\mathring{\mathcal{A}}({P},S):=\frac{{\mathbb{C}}\mathring{\mathrm{Jac}}({P},S)}{\text{(STU, AS, IHX)}}
$$ \index[notation]{A({P},S)@$\mathring{\mathcal{A}}({P},S)$}
where the relations STU,  AS,  IHX  are  the local relations shown in Figure \ref{figuraJD2_ell} using the schematic representation of Jacobi diagrams. In particular, we remark that in the STU relation the solid line could be a portion of the schematic representation of an element of $\pi_0(P)$ or a portion of a circle representing an element of $S$.
\end{definition}

\begin{figure}[ht!]
										\centering
                        \includegraphics[scale=0.8]{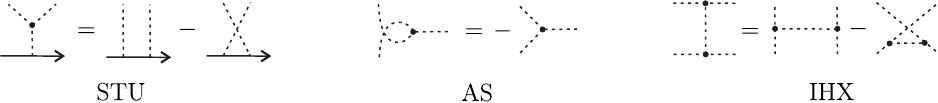}
												
\caption{Relations in $\mathring{\mathcal{A}}(P, S)$. The solid line in the STU relation can be a portion of the schematic representation of an element of $\pi_0(P)$ or a portion of a circle representing an element of $S$.}
\label{figuraJD2_ell} 												
\end{figure}

We define the \emph{degree} of a Jacobi diagram $\underline{D}$ by 
\begin{align}\label{degree-on-APS}
\mathrm{deg}(\underline{D}):= &\  \frac{1}{2}\big((\text{number of trivalent vertices of }D)  + | \partial D|\big)\in\mathbb Z_{\geq0}.
\end{align}\index[notation]{deg@$\mathrm{deg}(\underline{D})$}
One checks that this induces a degree on the space $\mathring{\mathcal{A}}({P},S)$.  We denote by $\mathring{\mathcal{A}}^\wedge({P},S)$\index[notation]{A^\wedge({P},S)@$\mathring{\mathcal{A}}^\wedge({P},S)$} its degree 
completion. From now on, for  a Jacobi diagram $\underline{D}$ on  $({P},S)$, we denote by $[\underline{D}]$ its class in $\mathring{\mathcal{A}}(P,S)\subset \mathring{\mathcal{A}}^\wedge({P},S)$.

\begin{definition}\label{def:emptyJacdiagram} Let $P\in\vec{\mathcal Br}$ be an oriented Brauer diagram  and let $S$ be a finite set.  We denote by $(P,S)_{{\mathcal A}}$  \index[notation]{(P,S)_{{\mathcal A}}@$(P,S)_{{\mathcal A}}$} the class in $\mathring{\mathcal A}(P,S)\subset \mathring{\mathcal A}^\wedge(P,S)$ of  the Jacobi diagram $\emptyset_{(P,S)}$, see Definition~\ref{def:emptyJacdiagram2}. In the case $S=\emptyset$, we just write $P_{{\mathcal A}}$  \index[notation]{P_{{\mathcal A}}@$P_{{\mathcal A}}$} for this class.
\end{definition}

Denote by $\mathcal Compl\mathcal Gr\mathcal Vect$  \index[notation]{Compl\mathcal Gr\mathcal Vect@$\mathcal Compl\mathcal Gr\mathcal Vect$} the category whose objects are complete graded vector spaces  and with morphisms continuous,  grading-preserving, linear maps. 

\begin{lemma} Let $P$ be an oriented Brauer diagram. Then
\begin{itemize}
\item[$(a)$]  The  assignment $S \mapsto \mathring{\mathrm{Jac}}(P,S)$ for $S$ a finite set, and for $f:S\to S'$ a bijection between $S$, $S'$ the map $ \mathring{\mathrm{Jac}}(P,S)\to   \mathring{\mathrm{Jac}}(P,S')$ defined by $$(D,\varphi, \{\mathrm{lin}_{l}\}_{l \in\pi_0(P)}, \{\mathrm{cyc}_{s}\}_{s \in S})\longmapsto (D,(\mathrm{Id}_{\pi_0(P)}\sqcup f)\circ\varphi, \{\mathrm{lin}_{l}\}_{l \in\pi_0(P)}, \{\mathrm{cyc}_{f^{-1}(s')}\}_{s' \in S'})$$ for any $(D,\varphi, \{\mathrm{lin}_{l}\}_{l \in\pi_0(P)}, \{\mathrm{cyc}_{s}\}_{s \in S})\in\mathring{\mathrm{Jac}}(P,S)$ defines  a functor  $F_P:\mathcal Set_f \to \mathcal Set$.

\item[$(b)$] The assignment in $(a)$ composed with the canonical map $\mathring{\mathrm{Jac}}(P,S)\to \mathring{\mathcal A}(P,S)$ (resp. $\mathring{\mathrm{Jac}}(P,S)\to \mathring{\mathcal A}^\wedge(P,S)$) induces a functor $\mathcal Set_f \to  \mathcal Vect$ (resp. $\mathcal Set_f \to  \mathcal Compl\mathcal Gr\mathcal Vect$).
\end{itemize}
\end{lemma}

\begin{proof}  Let $S,S', S''$ be finite sets and $f:S\to S'$, $g:S'\to S''$ bijections. 

$(a)$ Clearly we have $F_P(\mathrm{Id}_S)= \mathrm{Id}_{\mathring{\mathrm{Jac}}(P,S)}$ and $F_P(f)$ is well-defined because $f$ is a bijection. Let $(D,\varphi, \{\mathrm{lin}_{l}\}_{l \in\pi_0(P)}, \{\mathrm{cyc}_{s}\}_{s \in S})\in\mathring{\mathrm{Jac}}(P,S)$. The equality $(\mathrm{Id}_{\pi_0(P)}\sqcup (g\circ f))\circ \varphi = (\mathrm{Id}_{\pi_0(P)}\sqcup g)\circ \big((\mathrm{Id}_{\pi_0(P)}\sqcup f)\circ \varphi\big)$ and the fact that $f$ and $g$ are bijections imply $F_P(g\circ f) = F_P(g)\circ   F_P(f)$.

$(b)$ One checks that the $\mathbb{C}$-linear extension of the map $F_P(f):\mathring{\mathrm{Jac}}(P,S)\to \mathring{\mathrm{Jac}}(P,S')$ is compatible with the STU, IHX and AS relations which together with $(a)$ imply the statement. 
\end{proof}

\subsection{The natural transformation  \texorpdfstring{$\mathrm{Tens}^{\mathcal A}(P,P')_{\bullet,\bullet'}$}{Tens(P,P')..}}\label{sec:4-2}

\begin{definition}\label{deftensorinJac}  For $P,P'$ oriented Brauer diagrams and $S,S'$ finite sets, define a  map 
\begin{equation}\label{tensor:map:A:PP'Jac}
\mathrm{Tens}^{\mathring{\mathrm{Jac}}}(P,P')_{S,S'}:\mathring{\mathrm{Jac}}(P,S)\times \mathring{\mathrm{Jac}}(P', S')\longrightarrow \mathring{\mathrm{Jac}}(P\otimes P', S\sqcup S'), \quad 
(\underline{D}, \underline{D}')\longmapsto \underline{D}\otimes \underline{D}' 
\end{equation}   \index[notation]{Tens^{\mathring{\mathrm{Jac}}}(P,P')_{S,S'}@$\mathrm{Tens}^{\mathring{\mathrm{Jac}}}(P,P')_{S,S'}$}
where for $\underline{D}=(D,\varphi,\{\mathrm{lin}_l\}_{l \in \pi_0(P)}, \{\mathrm{cyc}_s\}_{s \in S})\in\mathring{\mathrm{Jac}}(P,S)$ and  $\underline{D}'=(D',\varphi',\{\mathrm{lin}'_l\}_{l \in \pi_0(P')}, \{\mathrm{cyc}'_s\}_{s \in S'})\in\mathring{\mathrm{Jac}}(P',S')$ the expression 
$\underline{D}\otimes \underline{D}' $ denotes the Jacobi diagram on ${P} \otimes {P}'$ 
given by $$\big(D\sqcup D',\varphi\sqcup\varphi', \{\mathrm{lin}_l\}_{l \in \pi_0(P)}\sqcup \{\mathrm{lin}'_l\}_{l \in \pi_0(P')}, \{\mathrm{cyc}_s\}_{s \in S}\sqcup \{\mathrm{cyc}'_s\}_{s \in S'}\big).$$ Here we use the canonical identifications $\pi_0(P\otimes P')\simeq \pi_0(P)\sqcup \pi_0(P')$,  and the equality $\partial (D\sqcup D') = \partial D\sqcup \partial D'$.
\end{definition}

\begin{lemma}
The map~\eqref{tensor:map:A:PP'Jac} induces a continuous linear map 
\begin{equation}\label{tensor:map:A:PP'}
\mathrm{Tens}^{\mathcal A}(P,P')_{S,S'} : \mathring{\mathcal A}^{\wedge}(P,S) \otimes \mathring{\mathcal A}^{\wedge}(P',S')\longrightarrow \mathring{\mathcal A}^{\wedge}(P \otimes P',S \sqcup S')
\end{equation}  \index[notation]{Tens^{\mathcal A}(P,P')_{S,S'}@$\mathrm{Tens}^{\mathcal A}(P,P')_{S,S'} $}
for any oriented Brauer diagrams ${P}$ and  ${P}'$ and any finite sets $S$ and $S'$. The assignment $(S,S')\mapsto \mathrm{Tens}^{\mathcal A}(P,P')_{S,S'}$ is a natural transformation relating the functors $\mathcal Set_f^2 \to \mathcal Vect$ given by $(S,S')\mapsto \mathring{\mathcal A}^{\wedge}(P,S) \otimes \mathring{\mathcal A}^{\wedge}(P',S')$ and $(S,S')\mapsto \mathring{\mathcal A}^{\wedge}(P\otimes P', S\sqcup S')$.

\end{lemma}

\begin{proof}
One checks that the bilinear extension of the map \eqref{tensor:map:A:PP'Jac} is compatible with the STU, AS and IHX relations. One also checks that the assignment  $(S,S')\mapsto \mathrm{Tens}^{\mathcal A}(P,P')_{S,S'}$ is compatible with pairs  $(f,g):(S,S')\to (E,E')$ of bijections $f:S\to E$ and $f':S'\to E'$ which implies the  second statement.
\end{proof}

\begin{notation}\label{notationfortens} We will write $x\otimes_{\mathcal A} x'$ \index[notation]{\otimes_{\mathcal A}@$\otimes_{\mathcal A}$} instead of $\mathrm{Tens}^{\mathcal A}(P,P')_{S,S'}(x\otimes x')$ for any  $x \in \mathring{\mathcal A}^{\wedge}(P,S)$ and $x'\in\mathring{\mathcal A}^{\wedge}(P',S')$ and we refer to such an element as the tensor product of $x$ and $x'$. In drawings, the tensor product $x\otimes_{\mathcal A} x'$ is described by the horizontal concatenation of $x$ and $x'$, see Figure~\ref{figuraJD_tensorexample} for an example.
\end{notation}

\begin{figure}[ht!]
										\centering
                        \includegraphics[scale=0.8]{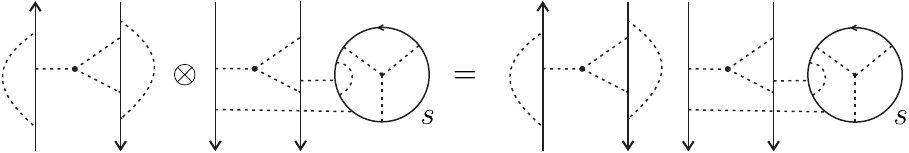}
												\caption{Schematic representation of the tensor product of an element in $\mathring{\mathcal{A}}^{\wedge}(\uparrow\downarrow, \emptyset)$ and an element in $\mathring{\mathcal{A}}^{\wedge}(\downarrow\downarrow, \{s\})$ which gives rise to an element in  $\mathring{\mathcal{A}}^{\wedge}(\uparrow\downarrow\downarrow\downarrow, \{s\})$.}
\label{figuraJD_tensorexample} 										
\end{figure}

\subsection{The natural transformation  \texorpdfstring{$\mathrm{Comp}^{\mathcal A}(P,P')_{\bullet,\bullet'}$}{Comp(P,P')..}}\label{sec:4-3}

Let  $P,P'\in\underline{\vec{\mathcal Br}}$ be a  composable  pair of oriented Brauer diagrams and $S,S'$ finite sets. Recall from Lemma~\ref{r:relPi0}$(c)$ that there is a map  $\mathrm{proj}^{\pi_0}_{P,P'}:\pi_0(P)\sqcup \pi_0(P')\to \pi_0(P'\circ P)\sqcup \mathrm{Cir}(P',P)$. Its disjoint union with the identity maps of $S$ and $S'$ is a map  $\mathrm{proj}_{P,P'}^{S,S'}:\pi_0(P)\sqcup S\sqcup \pi_0(P')\sqcup S'\to \pi_0(P'\circ P)\sqcup  S\sqcup S'\sqcup \mathrm{Cir}(P',P)$. We have  $(\mathrm{proj}^{\pi_0}_{P,P'})^{-1}(x)=(\mathrm{proj}^{S,S'}_{P,P'})^{-1}(x)$ for any $x \in  \pi_0(P'\circ P)\sqcup \mathrm{Cir}(P',P)$.

\begin{lemma}\label{r:2023-01-11-r1} Let $(P,P')$ be a composable pair in $\underline{\vec{\mathcal Br}}$. One then has $P=((p,q,\sigma),B)$, $P'=((q,r,\sigma'),B')$, and there exists a subset $L \subset [\![1,q]\!]$ such that $\mathrm{can}_2(L)=E \cap [\![1,q]\!]$ and $\mathrm{can}_1(L)=B'\cap [\![1,q]\!]$, where $E:=\sigma(B) \subset [\![1,p]\!] \sqcup [\![1,q]\!]$ (see Definition~1.30) and  $\mathrm{can}_1:[\![1,q]\!] \to [\![1,q]\!]\sqcup [\![1,r]\!]$ and $\mathrm{can}_2:[\![1,q]\!] \to [\![1,p]\!]\sqcup [\![1,q]\!]$ denote the canonical inclusions. We also set $E':=\sigma'(B') \subset [\![1,p']\!] \sqcup [\![1,q']\!]$. We have
\begin{itemize}
\item[$(a)$] For any $x\in\pi_0(P'\circ P)$ the set $(\mathrm{proj}^{\pi_0}_{P,P'})^{-1}(x)\subset \pi_0(P)\sqcup\pi_0(P')$ admits a
 unique linear order~$<$ with the following properties:
 \begin{itemize}
 \item[$(i)$] If $\alpha \in [\![1,q]\!]$ is such that $\mathrm{can}_2(\alpha) \in E \subset [\![1,p]\!] \sqcup [\![1,q]\!]$ and $\mathrm{can}_1(\alpha) \in B'\subset [\![1,q]\!] \sqcup [\![1,r]\!]$, then
$\{\sigma(\mathrm{can}_2(\alpha)), \mathrm{can}_2(\alpha)\} \in \pi_0(P)$ and $\{\mathrm{can}_1(\alpha),\sigma'(\mathrm{can}_1(\alpha))\} \in \pi_0(P')$ are such that $$\{\sigma(\mathrm{can}_2(\alpha)), \mathrm{can}_2(\alpha)\} < \{\mathrm{can}_1(\alpha),\sigma'(\mathrm{can}_1(\alpha))\}.$$

\item[$(ii)$] If $\alpha \in [\![1,q]\!]$ is such that $\mathrm{can}_2(\alpha) \in B \subset [\![1,p]\!] \sqcup [\![1,q]\!]$ and $\mathrm{can}_1(\alpha) \in E' \subset [\![1,q]\!] \sqcup [\![1,r]\!]$, then
$\{\sigma(\mathrm{can}_2(\alpha)),\mathrm{can}_2(\alpha)\} \in \pi_0(P)$ and $\{\mathrm{can}_1(\alpha),\sigma'(\mathrm{can}_1(\alpha))\} \in \pi_0(P')$ are such that $$\{\mathrm{can}_1(\alpha),\sigma'(\mathrm{can}_1(\alpha))\}< \{\sigma(\mathrm{can}_2(\alpha)),\mathrm{can}_2(\alpha)\}.$$ 
\end{itemize}

\item[$(b)$] Let $x\in\mathrm{Cir}(P', P)$. Let $X:=(\mathrm{proj}_{P,P'}^{\pi_0})^{-1}(x)$. If $y\in X\cap \pi_0(P)$, then there exists a unique $\alpha \in [\![1,q]\!]$ such that $\mathrm{e}_{P}(y)=\mathrm{can}_2(\alpha)$, and there exists a unique element $y'\in\pi_0(P')$ such that $\mathrm{b}_{P'}(y')=\mathrm{can}_1(\alpha)$. 
If $y\in X\cap \pi_0(P')$, then there exist a unique $\alpha \in [\![1,q]\!]$ such that  $\mathrm{e}_P(y)=\mathrm{can}_1(\alpha)$, and there exists a unique element $y' \in \pi_0(P)$ such that $\mathrm{b}_P(y)=\mathrm{can}_2(\alpha)$. 
The assignment $y \mapsto y'$ defines a self-map $\mathrm{succ}_X$ of $X$. The map $\mathrm{succ}_X$ is a permutation, therefore defines an action of $\mathbb{Z}$ on $X$. This action is transitive, therefore defines a cyclic order on $X$.
\end{itemize}
\end{lemma}

\begin{proof}
Recall the canonical injections $\mathrm{can}_{12} :  [\![1,p]\!] \sqcup [\![1,q]\!] \to [\![1,p]\!] \sqcup [\![1,q]\!] \sqcup [\![1,r]\!]$ and $\mathrm{can}_{23} : [\![1,q]\!] \sqcup [\![1,r]\!] \to [\![1,p]\!] \sqcup [\![1,q]\!] \sqcup [\![1,r]\!]$. Set $V:=[\![1,p]\!] \sqcup [\![1,q]\!] \sqcup [\![1,r]\!]$,  it follows from the assumption that the subsets $\mathrm{can}_{12}(B)$, $\mathrm{can}_{23}(B')$, $\mathrm{can}_{12}(E) \cap [\![1p]\!]$, $\mathrm{can}_{23}(E) \cap [\![1,r]\!]$ of $[\![1,p]\!] \sqcup [\![1,q]\!] \sqcup [\![1,r]\!]$ form a partition of $V$. Set $V_{0}:=(\mathrm{can}_{12}(E) \cap [\![1,p]\!]) \sqcup (\mathrm{can}_{23}(E') \cap [\![1,r]\!])$ and define $\mathrm{succ} : V\setminus V_0 \to V$ by $\mathrm{succ}(\mathrm{can}_{12}(b)):=\mathrm{can}_{12}(\sigma(b))$ and $\mathrm{succ}(\mathrm{can}_{23}(b')):=\mathrm{can}_{23}(\sigma'(b'))$ for $b \in B$, $b' \in B'$.  The injectivity of $\mathrm{succ}$ follows from the partition status of $\big(\mathrm{can}_{12}(B)$,$\mathrm{can}_{23}(B'), \mathrm{can}_{12}(E) \cap [\![1,p]\!], \mathrm{can}_{23}(E) \cap [\![1,r]\!]\big)$, the injectivity of $\mathrm{can}_{12}$ and $\mathrm{can}_{23}$, and the permutation status of $\sigma$ and $\sigma'$.  By Lemma~\ref{r:2023-02-14-r1}, one attaches to $(V,V_{0},\mathrm{succ})$ a set $\mathrm{orb}(V)$ equipped with a partition $(\mathrm{orb}_{\mathrm{lin}}(V),\mathrm{orb}_{\mathrm{cyc}}(V))$. 

There is a canonical bijection $\mathrm{orb}(V) \simeq \big([\![1,p]\!] \sqcup [\![1,q]\!] \sqcup [\![1,r]\!]\big)/\langle \alpha,\beta \rangle$ inducing a bijection between the partitions $(\mathrm{orb}_{\mathrm{lin}}(V),\mathrm{orb}_{\mathrm{cyc}}(V))$ and $(\pi_0(P' \circ P),{\mathrm{Cir}}(P',P))$, see Equation \eqref{eq:toto}. 

Besides, $\{v \in V_{0}\ | \ \mathrm{succ}^{-1}(v)=\emptyset\}=\emptyset$. Indeed, if $v \in V_{0}$, then either there exists $e \in E$ with $\mathrm{can}_{12}(e) \in [\![1,p]\!]$ and $v=\mathrm{can}_{12}(e)$; in that case
$\mathrm{can}_{12}(\sigma(e)) \in \mathrm{succ}^{-1}(v)$;  or there exists $e' \in E'$ with $\mathrm{can}_{23}(e') \in [\![1,r]\!]$ and $v'=\mathrm{can}_{23}(e')$; in that case $\mathrm{can}_{23}(\sigma'(e')) \in\mathrm{succ}^{-1}(v)$. By Lemma~\ref{r:2023-02-15r1}, one derives a bijection $\mathrm{orb}(V\setminus V_0) \simeq \mathrm{orb}(V)$ inducing a bijection between the partitions $(\mathrm{orb}_{\mathrm{lin}}(V\setminus V_0),\mathrm{orb}_{\mathrm{cyc}}(V\setminus V_0))$ and $(\mathrm{orb}_{\mathrm{lin}}(V),\mathrm{orb}_{\mathrm{cyc}}(V))$. Now $V\setminus V_0 \simeq B \sqcup B' \simeq \pi_0(P) \sqcup \pi_0(P')$. Transporting structures, we obtain a PDISM on $\pi_0(P) \sqcup \pi_0(P')$, see Definition~\ref{def:PDISM}. Its associated partition is in bijection with $(\mathrm{orb}_{\mathrm{lin}}(V\setminus V_0),\mathrm{orb}_{\mathrm{cyc}}(V\setminus V_0))$, and therefore with $(\pi_0(P' \circ P),{\mathrm{Cir}}(P',P))$, and whose underlying cyclic and linear orders are as announced. 
\end{proof}

\begin{definition}
Let  $P,P'\in\underline{\vec{\mathcal Br}}$ be a composable pair of oriented Brauer diagrams and $S,S'$ finite sets. Define a map 
\begin{equation}\label{Compmap:A:PP'Jac}
\mathrm{Comp}^{\mathring{\mathrm{Jac}}}(P,P')_{S,S'}:\mathring{\mathrm{Jac}}({P},S)\times \mathring{\mathrm{Jac}}({P}',S')\longrightarrow\mathring{\mathrm{Jac}}({P}'\circ {P}, S\sqcup S'\sqcup \mathrm{Cir}(P',P)), \quad (\underline{D}, \underline{D}')\longmapsto \underline{D}'\circ \underline{D}
\end{equation}  \index[notation]{Comp^{\mathring{\mathrm{Jac}}}(P,P')_{S,S'}@$\mathrm{Comp}^{\mathring{\mathrm{Jac}}}(P,P')_{S,S'}$}
where for $\underline{D}=(D,\varphi,\{\mathrm{lin}_l\}_{x \in \pi_0(P)}, \{\mathrm{cyc}_s\}_{s \in S})\in\mathring{\mathrm{Jac}}(P,S)$ and  $\underline{D}'=(D',\varphi',\{\mathrm{lin}'_x\}_{x \in \pi_0(P')}, \{\mathrm{cyc}'_s\}_{s \in S'})\in\mathring{\mathrm{Jac}}(P',S')$ the expression $\underline{D}'\circ \underline{D}$  denotes the Jacobi diagram on ${P}' \circ {P}'$ given by
 $$\big(D\sqcup D',\mathrm{proj}_{P,P'}^{S,S'}\circ(\varphi\sqcup\varphi'), \{\underline{\mathrm{lin}}_l\}_{x \in \pi_0(P'\circ P)}, \{\mathrm{cyc}_s\}_{s \in S}\sqcup \{\mathrm{cyc}'_s\}_{s \in S'}\sqcup \{\underline{\mathrm{cyc}}_x\}_{x \in \mathrm{Cir}(P',P)} \big)$$
with $\mathrm{proj}_{P,P'}^{S,S'}$ as in Lemma~\ref{r:2023-01-11-r1}$(b)$ and orders $\underline{\mathrm{lin}}_x$ and $\underline{\mathrm{cyc}}_x$ defined as follows. Let $x\in\pi_0(P'\circ P)$, by Lemma~\ref{r:2023-01-11-r1}$(a)$, the set $(\mathrm{proj}_{P,P}^{S,S'})^{-1}(x)\subset \pi_0(P)\sqcup \pi_0(P')$ is linearly ordered, let us say it is given by $x_1< x_2 < \cdots < x_k$. This order induces the linear order $$(\varphi\sqcup \varphi')^{-1}(x_1) <(\varphi\sqcup \varphi')^{-1}(x_2) < \cdots < (\varphi\sqcup \varphi')^{-1}(x_k)$$ on the family of sets $\{(\varphi\sqcup \varphi')^{-1}(x_i)\}_{i\in[\![1,k]\!]}$. By hypothesis each set   $(\varphi\sqcup \varphi')^{-1}(x_i)$ is linearly ordered, then by Lemma~\ref{r:2023-01-11-r2}$(a)$, there is a unique linear  order on 
$$\big(\mathrm{proj}_{P,P'}^{S,S'}\circ (\varphi\sqcup \varphi')\big)^{-1}(x)=\sqcup_{i=1}^{k}(\varphi\sqcup \varphi')^{-1}(x_i)$$ 
satisfying the conditions stated in Lemma~\ref{r:2023-01-11-r2}$(a)$, we denote this order by $\underline{\mathrm{lin}}_x$. The cyclic order $\underline{\mathrm{cyc}}_x$ for $x\in\mathrm{Circ}(P'\circ P)$ is defined similarly using Lemma~\ref{r:2023-01-11-r1}$(b)$ and Lemma~\ref{r:2023-01-11-r2}$(b)$.
\end{definition}

\begin{lemma}
The composition  map \eqref{Compmap:A:PP'Jac}  induces a continuous bilinear map
\begin{equation}\label{map:A:PP'2}
\mathrm{Comp}^{\mathcal{A}}(P,P')_{S,S'}:\mathring{\mathcal{A}}^\wedge({P},S)\times\mathring{\mathcal{A}}^\wedge({P}',S')\longrightarrow\mathring{\mathcal{A}}^\wedge(P'\circ {P}, S\sqcup S'\sqcup \mathrm{Cir}(P',P)), 
\end{equation} \index[notation]{Comp^{\mathcal{A}}(P,P')_{S,S'}@$\mathrm{Comp}^{\mathcal{A}}(P,P')_{S,S'}$}
for any composable Brauer diagrams ${P}$ and ${P}'$ and any finite sets $S$ and $S'$. The assignment $(S,S')\mapsto \mathrm{Comp}^{\mathcal A}(P,P')_{S,S'}$ is a natural transformation relating the functors $\mathcal Set_f^2 \to \mathcal Vect$ given by $(S,S')\mapsto \mathring{\mathcal A}^{\wedge}(P,S) \otimes \mathring{\mathcal A}^{\wedge}(P',S')$ and $(S,S')\mapsto \mathring{\mathcal A}^{\wedge}(P'\circ P, S\sqcup S'\sqcup \mathrm{Cir}(P',P))$.
\end{lemma}

\begin{proof}
One checks that the bilinear extension of the map \eqref{Compmap:A:PP'Jac} is compatible with the STU, AS and IHX relations.  One also checks that the assignment  $(S,S')\mapsto \mathrm{Comp}^{\mathcal A}(P,P')_{S,S'}$ is compatible with pairs  $(f,g):(S,S')\to (E,E')$ of bijections $f:S\to E$ and $f':S'\to E'$ which implies the  second statement.
\end{proof}

\begin{notation}\label{notationforcomp}Let $P,P'\in\vec{\mathcal Br}$ be composable oriented Brauer diagrams and $S$ and $S'$ finite sets. We will write $x'\circ_{\mathcal A} x$  \index[notation]{\circ_{\mathcal A}@$\circ_{\mathcal A}$} instead of $\mathrm{Comp}^{\mathcal A}(P,P')_{S,S'}(x\otimes x')$ for any  $x \in \mathring{\mathcal A}^{\wedge}(P,S)$ and $x'\in\mathring{\mathcal A}^{\wedge}(P',S')$  and we refer to such an element as the composition of $x$ and $x'$. In drawings, the composition $x'\circ_{\mathcal A} x$ is described by the vertical concatenation of $x'$ above $x$, see Figure~\ref{figuraJD_compositionexample} for an example.
\end{notation}

\begin{figure}[ht!]
										\centering
                        \includegraphics[scale=0.65]{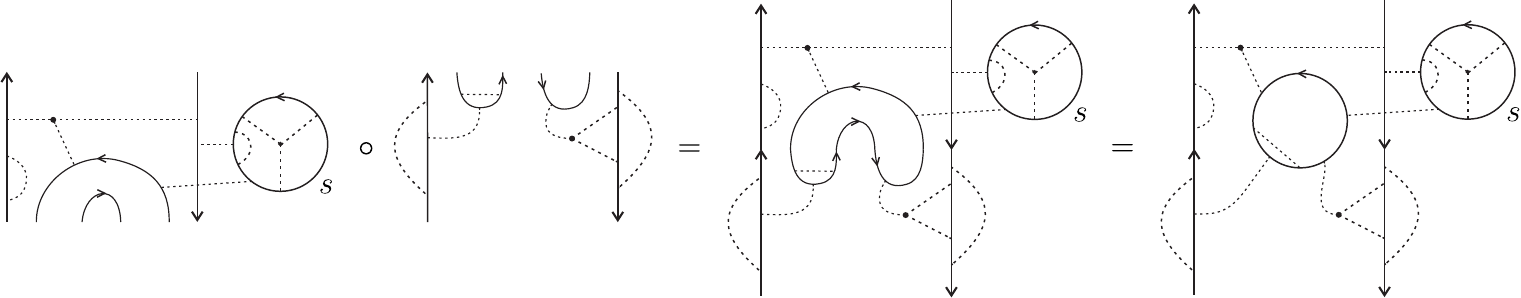}
												\caption{Schematic representation of the composition of Jacobi diagrams.}
\label{figuraJD_compositionexample} 										
\end{figure}

\subsection{The natural transformation  \texorpdfstring{$\mathrm{co}^{\mathcal A}(P,\bullet)$}{co(P,.)}}\label{sec:4-4}

\begin{definition}\label{def:changeoforientationJac} Let $P$ be an oriented Brauer diagram, $S$ a finite set and $s_0\in S$. Define a map
\begin{equation}\label{eq:2023-01-11-coJac}
\mathrm{co}^{\mathring{\mathrm{Jac}}}(P, (S,s_0)):\mathring{\mathrm{Jac}}(P,S) \longrightarrow \mathbb{C}\mathring{\mathrm{Jac}}(P,S)
\end{equation}  \index[notation]{co^{\mathring{\mathrm{Jac}}}(P, (S,s_0))@$\mathrm{co}^{\mathring{\mathrm{Jac}}}(P, (S,s_0))$}
by $$\mathrm{co}^{\mathring{\mathrm{Jac}}}(P, (S,s_0))(D,\varphi, \{\mathrm{lin}_l\}_{l\in\pi_0(P)},\{\mathrm{cyc}_s\}_{s\in S}) =(-1)^{|\varphi^{-1}(s_0)|} (D,\varphi, \{\mathrm{lin}_l\}_{l\in\pi_0(P)},\{\mathrm{cyc}_s\}_{s\in S\setminus\{s_0\}}\sqcup \{\overline{\mathrm{cyc}}_{s_0}\} )$$ for any $(D,\varphi, \{\mathrm{lin}_l\}_{l\in\pi_0(P)},\{\mathrm{cyc}_s\}_{s\in S})\in\mathring{\mathrm{Jac}}(P,S)$, where $\overline{\mathrm{cyc}}_{s_0}$ denotes the opposite cyclic order of ${\mathrm{cyc}}_{s_0}$ on the set $\varphi^{-1}(s_0)\subset \partial D$.
\end{definition}

\begin{lemma}\label{r:2023-02-26r1} Let $P$ be an oriented Brauer diagram, $S$ a finite set and $s_0\in S$. The map \eqref{eq:2023-01-11-coJac}  induces a continuous linear map
\begin{equation}\label{map:A:PP'}
\mathrm{co}^{\mathcal A}(P, (S,s)): \mathring{\mathcal{A}}^\wedge({P},S) \longrightarrow\mathring{\mathcal{A}}^\wedge(P, S).
\end{equation} \index[notation]{co^{\mathcal A}(P, (S,s))@$\mathrm{co}^{\mathcal A}(P, (S,s))$}
The assignment $(S,s)\mapsto \mathrm{co}^{\mathcal A}(P,(S,s))$ is a natural self-transformation of the functor $\mathcal Set_f^*\to \mathcal Vec_{\mathbb{C}}$ given by  $(S,s)\mapsto  \mathring{\mathcal A}^{\wedge}(P,S)$.
\end{lemma}

\begin{proof}
One checks that the linear extension of the map \eqref{eq:2023-01-11-coJac} is compatible with the STU, AS and IHX relations. One also checks that the assignment $(S,s)\mapsto \mathrm{co}^{\mathcal A}(P.(S,s))$ is compatible with pointed bijections $f:(S,s_0)\to (E,e_0)$ which implies the second statement. 
\end{proof}

In drawings, the element $\mathrm{co}^{\mathcal A}(P.(S,s))(x)$ is obtained from $x\in\mathring{\mathcal A}^{\wedge}(P,S)$ by changing the orientation of the circle in the schematic representation of $x$ labelled by $s_0$ and multiplying it by $(-1)$ power to the number of univalent vertices attached to this circle, see  Figure~\ref{figuraJD_changeorientationcircle} for an example. 
\begin{figure}[ht!]
										\centering
                        \includegraphics[scale=0.9]{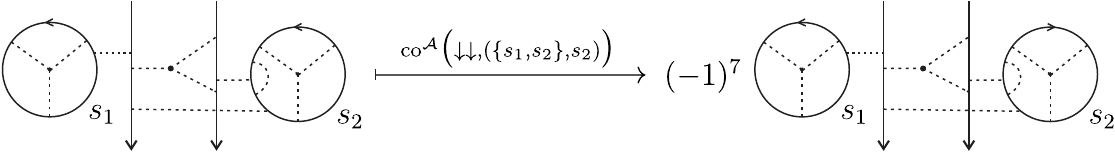}
												\caption{Schematic representation of the change of orientation operation for Jacobi diagrams.}
\label{figuraJD_changeorientationcircle} 										
\end{figure}

\subsection{The natural transformation  \texorpdfstring{$\mathrm{dbl}_{\mathcal A}((P,A),\bullet)$}{dbl((P,A),.)}}\label{sec:4-5}

\begin{definition}
For $P \in \vec{\mathcal Br}$ and $A \subset \pi_0(P)$ and $S$ a finite set. Use the simplified notation $P^A$ for $\mathrm{dbl}_{\vec{\mathcal Br}}(P,A)$. By Lemma~\ref{r:2023-01-11-r2}$(c)$, there is a map $\mathrm{proj}^{\pi_0}_{P,A} : \pi_0(P^A) \to \pi_0(P)$. Let $$\big(D,\varphi,\{\mathrm{lin}_x\}_{x \in \pi_0(P)},\{\mathrm{cyc}_s\}_{s \in S}\big)\in\mathring{\mathrm{Jac}}(P,S).$$ For any $\tilde{\varphi} : \partial D \to \pi_0(P^A) \sqcup S$ such that $\mathrm{proj}^{\pi_0}_{P,A} \circ \tilde{\varphi}=\varphi$ and $x' \in \pi_0(P^A)$, one has 
$\tilde{\varphi}^{-1}(x') \subset \varphi^{-1}(x)$, where  $x:=\mathrm{proj}^{\pi_0}_{P,A}(x')$. Define  $\mathrm{lin}^{\tilde{\varphi}}_{x'}$ as the linear order on $\tilde{\varphi}^{-1}(x')$ given by the restriction of $\mathrm{lin}_x$, see \S\ref{sec:4-0}. 

Define a map 
\begin{equation}\label{defeq:doublingmapinJac}
\mathrm{dbl}_{\mathring{\mathrm{Jac}}}((P,A),S) : \mathring{\mathrm{Jac}}(P,S)\longrightarrow \mathbb{C}\mathring{\mathrm{Jac}}(P^A,S)
\end{equation}  \index[notation]{dbl_{\mathring{\mathrm{Jac}}}((P,A),S) @$\mathrm{dbl}_{\mathring{\mathrm{Jac}}}((P,A),S)$}  by 
\begin{equation*}
\big(D,\varphi,\{\mathrm{lin}_x\}_{x \in \pi_0(P)},\{\mathrm{cyc}_s\}_{s \in S}\big) \longmapsto \sum_{\tilde{\varphi} : \partial D \to \pi_0(P^A) \sqcup S \text{ with } \mathrm{proj}^{\pi_0}_{P,A} \circ \tilde{\varphi}=\varphi}
\big(D,\tilde{\varphi},\{\mathrm{lin}^{\tilde{\varphi}}_{x'}\}_{x' \in \pi_0(P^A)},\{\mathrm{cyc}_s\}_{s \in S}\big).
\end{equation*} 
\end{definition}

\begin{lemma}\label{r:2023-03-04r1} Let $P \in \vec{\mathcal Br}$ and $A \subset \pi_0(P)$ and $S$ be a finite set.The map \eqref{defeq:doublingmapinJac} induces a continuous linear map
\begin{equation}\label{defeq:doublingmapinA}
\mathrm{dbl}_{\mathcal{A}}((P,A),S) : \mathring{\mathcal{A}}^{\wedge}(P,S)\longrightarrow \mathring{\mathcal{A}}^{\wedge}(P^A,S).
\end{equation} \index[notation]{dbl_{\mathcal{A}}((P,A),S) @$\mathrm{dbl}_{\mathcal{A}}((P,A),S) $}
The assignment $S\mapsto \mathrm{dbl}_{\mathcal{A}}((P,A),S)$ defines a natural transformation $\mathrm{dbl}_{\mathcal{A}}((P,A),\bullet) : \mathring{\mathcal{A}}^{\wedge}(P,\bullet)\longrightarrow \mathring{\mathcal{A}}^{\wedge}(P^A,\bullet)$ between the functors $\mathring{\mathcal{A}}^{\wedge}(P,\bullet): \mathcal Set_f\to\mathcal Vect$ and  $\mathring{\mathcal{A}}^{\wedge}(P^A,\bullet): \mathcal Set_f\to\mathcal Vect$.
\end{lemma}
\begin{proof}
One checks that the linear extension of the map \eqref{defeq:doublingmapinJac} is compatible with the STU, AS and IHX relations. One also checks that the assignment $S\mapsto \mathrm{dbl}_{\mathcal{A}}((P,A),S)$ is compatible with  bijections $S\to E$ which implies the second statement. 
\end{proof}

{\begin{remark} In drawings, we use the \emph{box notation} to denote the sum over all the possible ways of gluing the univalent vertices of a diagram $D$ attached to the grey box to the  segments involved in it, so that if there are $k$ univalent vertices attached to it and two segments, the sum consists of $2^k$ terms. See Figure \ref{fig:doubling:A}.
\end{remark}}

\begin{figure}[ht!]
\centering
\includegraphics[scale=0.8]{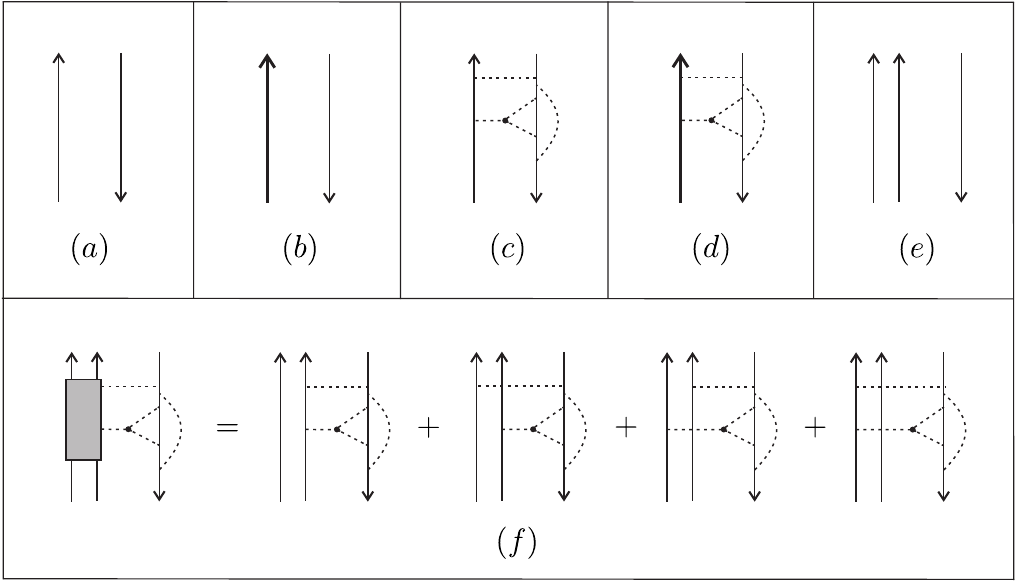}
\caption{Schematic representation of $\mathrm{dbl}_{\mathcal{A}}((\uparrow\downarrow,\{\uparrow\}),\emptyset) : \mathring{\mathcal{A}}^{\wedge}(\uparrow\downarrow,\emptyset)\longrightarrow \mathring{\mathcal{A}}^{\wedge}(\uparrow\uparrow\downarrow,\emptyset)$. In~$(a)$ we show ${P}=\uparrow\downarrow\in \vec{\mathcal Br}$, in $(b)$ we show  $(P,A) = (\uparrow\downarrow,\{\uparrow\})$ using thick lines for representing $A\subset\pi_0(P)$. In  $(c)$ we show an element  $\underline{D}\in\mathring{\mathrm{Jac}}(\uparrow\downarrow,\emptyset)$. In $(e)$ we show ${P}^A=\uparrow\downarrow^{\{\uparrow\}}=\mathrm{dbl}_{\vec{\mathcal Br}}(P,A)=\uparrow\uparrow\downarrow$ and in $(f)$ we show a schematic representation of $\mathrm{dbl}_{\mathring{\mathrm{Jac}}}((\uparrow\downarrow,\{\uparrow\}),\emptyset)(\underline{D})$ using the box notation. Its class in $\mathring{\mathcal{A}}^{\wedge}(\uparrow\uparrow\downarrow,\emptyset)$ is $\mathrm{dbl}_{\mathcal{A}}((\uparrow\downarrow,\{\uparrow\}),\emptyset)([\underline{D}])$.}
\label{fig:doubling:A}
\end{figure}

\begin{convention}\label{convention-doublingblueboxes} Consider a pair $({P},A)$ with $P\in\vec{\mathcal Br}$ and $A\subset \pi_0(P)$ and $S$ a finite set. For $[\underline{D}]\in\mathring{\mathcal{A}}^\wedge({P},S)$, we also represent $[\underline{D}]^A := {\mathrm{dbl}_{\mathcal A}((P,A),S)}([\underline{D}])\in\mathring{\mathcal{A}}^\wedge({P}^A,S)$ schematically by using two nested blue boxes around the schematic representation of $\underline{D}$ and $A$ is indicated by using $\vee$'s  (resp. $\wedge$'s ) between the two blue boxes on each top (resp. bottom) boundary point  $a\in \bigsqcup_{p\in A}p\cap [\![1,\mathrm{t}(P)]\!]$ (resp. $a\in \bigsqcup_{p\in A}p\cap [\![1,\mathrm{t}(P)]\!]$). In Figure~\ref{fig:doubling:A-nc} we shown an example of this notation. 
\begin{figure}[ht!]
\centering
\includegraphics[scale=0.8]{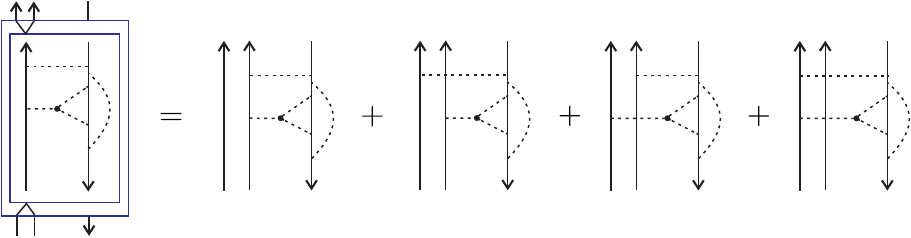}
\caption{The element ${\mathrm{dbl}}_{1}([\underline{D}])\in\mathring{\mathcal{A}}^\wedge({P}^A,\emptyset)$ from Figure~\ref{fig:doubling:A} $(f)$ using the convention  of  nested blue boxes for doubling.}
\label{fig:doubling:A-nc}
\end{figure}
\end{convention}

\begin{convention}\label{convention:co_i-dbl_i} If ${P}=\mathrm{Id}_w\in\underline{\vec{\mathcal Br}}(w,w)$ for $w\in\{+,-\}^*$, then for $i\in[\![1,|w|]\!]$, we denote by   $\mathrm{dbl}_i$ \index[notation]{dbl_i@$\mathrm{dbl}_i$} the map  $\mathrm{dbl}_{\mathcal{A}}{({P}, \{i,i\})}$. Moreover, if $|w|=1$ we write $\mathrm{dbl}$ instead of $\mathrm{dbl}_1$. 
\end{convention}

\subsection{Proof of the pre-LMO axioms (LMO1), (LMO2) and (LMO3)}\label{sec:4-6}

\begin{lemma}\label{r:2023-02-24r1} Let $X$, $Y$ and $Z$ be finite sets and let $\mathrm{succ}_{XY}$, $\mathrm{succ}_{YZ}$ be \emph{PDISM}s (see Definition~\ref{def:PDISM}) on $X \sqcup Y$ and $Y\sqcup Z$. Let $\mathrm{orb}(X\sqcup Y) = \mathrm{orb}_{\mathrm{lin}}(X\sqcup Y)\sqcup \mathrm{orb}_{\mathrm{cyc}}(X\sqcup Y)$ and $\mathrm{orb}(Y\sqcup Z)=\mathrm{orb}_{\mathrm{lin}}(Y\sqcup Z)\sqcup \mathrm{orb}_{\mathrm{cyc}}(Y\sqcup Z)$ as in Lemma~\ref{r:2023-02-14-r1}. Suppose that $\mathrm{succ}_{XY}(X) \subset Y$, $\mathrm{succ}_{XY}(X) \subset Y$, and $\mathrm{succ}_{YZ}(Y) \subset Z$, 
$\mathrm{succ}_{YZ}(Z) \subset Y$. Assume also that $\mathrm{succ}_{XY}(X) \cap \mathrm{succ}_{YZ}(Z)=\emptyset$ and that for $y \in Y$, $\mathrm{succ}_{XY}(y)$ is defined if and only if $\mathrm{succ}_{YZ}(y)$ is undefined.
\begin{itemize}
\item[$(a)$] A \emph{PDISM} $\mathrm{succ}_{XY,Z}$ on $\mathrm{orb}_{\mathrm{lin}}(X \sqcup Y) \sqcup Z$ is defined as follows: Let $\omega \in \mathrm{orb}_{\mathrm{lin}}(X \sqcup Y)$ and $m_{\omega}$ be the maximal element of the underlying set of $\omega$; 

$(i)$ if $m_{\omega} \in X$, then $\mathrm{succ}_{XY,Z}(\omega)$ is undefined; 

$(ii)$ if $m_{\omega} \in Y$, then $\mathrm{succ}_{XY}(m_{\omega})$ is undefined, therefore $\mathrm{succ}_{YZ}(m_{\omega})$ is defined and belongs to~$Z$, and one sets $\mathrm{succ}_{XY,Z}(\omega):= \mathrm{succ}_{YZ}(m_{\omega})$.      
If $z \in Z$, then if $\mathrm{succ}_{YZ}(z)$ is undefined, then so is $\mathrm{succ}_{XY,Z}(z)$; and if $\mathrm{succ}_{YZ}(z)$ is defined, then it belongs to $Y$ and if one denotes it $y$, then $\mathrm{succ}_{XY,Z}(z):=(y,\mathrm{succ}_{XY}(y),\mathrm{succ}_{XY}^2(y),\ldots)\in\mathrm{orb}_{\mathrm{lin}}(X\sqcup Y)$.

\item[$(b)$]  A \emph{PDISM} $\mathrm{succ}_{X,YZ}$ on $X \sqcup \mathrm{orb}_{\mathrm{lin}}(Y\sqcup Z)$ is similarly defined. 

\item[$(c)$]  A \emph{PDISM} $\mathrm{succ}$ on $X \sqcup Y \sqcup Z$ is defined as follows: for $x \in X$, if $\mathrm{succ}_{XY}(x)$ is undefined, then $\mathrm{succ}(x)$ is undefined; if $\mathrm{succ}_{XY}(x)$ is defined, then $\mathrm{succ}(x):=\mathrm{succ}_{XY}(x)$; 
for $z \in Z$, if $\mathrm{succ}_{YZ}(z)$ is undefined, then $\mathrm{succ}(z)$ is undefined; if $\mathrm{succ}_{YZ}(z)$ is defined, then $\mathrm{succ}(z):=\mathrm{succ}_{YZ}(z)$; if $y \in Y$, then either $\mathrm{succ}_{XY}(y)$ is defined, then $\mathrm{succ}(y):=\mathrm{succ}_{XY}(y)$; or $\mathrm{succ}_{YZ}(y)$ is defined, then $\mathrm{succ}(y):=\mathrm{succ}_{YZ}(y)$. 

\item[$(d)$]  The  \emph{PDSIM} $\mathrm{succ}$ coincides with the \emph{PDISM} on $X \sqcup Y \sqcup Z$ obtained by combining the \emph{PDISM} $\mathrm{succ}_{XY,Z}$ on $\mathrm{orb}_{\mathrm{lin}}(X \sqcup Y) \sqcup Z$, the map $X \sqcup Y \sqcup Z \to (\mathrm{orb}_{\mathrm{lin}}(X \sqcup Y) \sqcup Z) \sqcup  \mathrm{orb}_{\mathrm{cyc}}(X \sqcup Y)$, the linear (resp. cyclic) order in the fibers of $\mathrm{orb}_{\mathrm{lin}}(X \sqcup Y) \sqcup Z$ (resp. $\mathrm{orb}_{\mathrm{cyc}}(X \sqcup Y)$) for this map, and Lemma~\ref{r:2023-01-11-r2}, as well as with the \emph{PDISM} on $X \sqcup Y \sqcup Z$ obtained by combining the \emph{PDISM} $\mathrm{succ}_{X,YZ}$  the map $X \sqcup Y \sqcup Z \to (X \sqcup \mathrm{orb}_{\mathrm{lin}}(Y \sqcup Z)) \sqcup  \mathrm{orb}_{\mathrm{cyc}}(Y \sqcup Z)$, the linear (resp. cyclic) order in the fibers of $X \sqcup \mathrm{orb}_{\mathrm{lin}}(Y \sqcup Z)$ (resp. $\mathrm{orb}_{\mathrm{cyc}}(Y \sqcup Z)$) for this map, and Lemma~\ref{r:2023-01-11-r2}. 
\end{itemize}
\end{lemma}

\begin{proof}
The result follows straightforwardly from the definitions.
\end{proof}

\begin{lemma}\label{r:lemmaforLMO1}  The collection of  assignments taking:   
\begin{itemize}
\item[$(a)$]   $P\in\vec{\mathcal Br}$ to the functor $\mathring{\mathcal A}^{\wedge}(P,\bullet):\mathcal Set_f\to \mathcal Vect$, see \S\ref{sec:4-1};
\item[$(b)$] $w\in\{+,-\}^*$ to the  element $\mathrm{Id}^{\mathcal A}_w:=(\mathrm{Id}_w)_{\mathcal A}\in\mathring{\mathcal A}^{\wedge}(\mathrm{Id}_w,\emptyset)$, see Definition~\ref{def:emptyJacdiagram};   \index[notation]{Id^{\mathcal A}_w@$\mathrm{Id}^{\mathcal A}_w$}
\item[$(c)$]   each pair $P,P'\in\vec{\mathcal Br}$ to the natural transformation $\mathrm{Tens}^{\mathcal A}(P,P')_{\bullet, \diamond}:\mathring{\mathcal A}^{\wedge}(P,\bullet)\otimes \mathring{\mathcal A}^{\wedge}(P',\diamond)\to \mathring{\mathcal A}^{\wedge}(P\otimes P',\bullet\sqcup \diamond)$, see \S\ref{sec:4-2};
\item[$(d)$]   every composable pair $P,P'\in\vec{\mathcal Br}$ to  the natural transformation $\mathrm{Comp}^{\mathcal A}(P,P')_{\bullet,\diamond}:\mathring{\mathcal A}^{\wedge}(P,\bullet)\otimes \mathring{\mathcal A}^{\wedge}(P',\diamond)\to \mathring{\mathcal A}^{\wedge}(P'\circ P, \bullet\sqcup \diamond\sqcup \mathrm{Cir}(P',P))$, see \S\ref{sec:4-3};
\end{itemize}
satisfies the axiom (LMO1) in Definition~\ref{def:preLMOstructure}.
\end{lemma}

\begin{proof}
The only non-immediate statement is $(c)$ from (LMO1). It relies on the folowing statements: $(a)$ for $(P,P')$ and  $(P',P'')$ two composable pairs in $\vec {\mathcal Br}$, the diagram 
\begin{equation}\label{eq:diagram2023-02-25}
\xymatrix{
\pi_0(P'') \sqcup \pi_0(P') \sqcup \pi_0(P)   
\ar_{\mathrm{proj}_{P'',P'}^{\pi_0}\sqcup\mathrm{Id}_{\pi_0(P)}}[d]  \ar^{\mathrm{Id}}[r] &   \pi_0(P'') \sqcup \pi_0(P') \sqcup \pi_0(P)   
\ar^{\mathrm{Id}_{\pi_0(P'') }\sqcup \mathrm{proj}_{P',P}^{\pi_0}}[d]      \\
 \pi_0(P'' \circ P') \sqcup \mathrm{Cir}(P'',P') \sqcup \pi_0(P) \ar_{\simeq}[d]  &   \pi_0(P'')\sqcup \pi_0(P' \circ P) \sqcup \mathrm{Cir}(P',P)  \ar^{\mathrm{Id}}[d] \\
 \pi_0(P'' \circ P') \sqcup \pi_0(P) \sqcup \mathrm{Cir}(P'',P')  \ar_{\mathrm{proj}_{P''\circ P',P}^{\pi_0}\sqcup\mathrm{Id}_{\mathrm{Cir}(P'',P')}}[d] &    \pi_0(P'')\sqcup \pi_0(P' \circ P) \sqcup \mathrm{Cir}(P',P)  \ar^{ \mathrm{proj}_{P''\circ P',P}^{\pi_0}\sqcup \mathrm{Id}_{\mathrm{Cir}(P',P)}}[d]\\
 \pi_0((P'' \circ P') \circ P) \sqcup \mathrm{Cir}(P''\circ P',P) \sqcup \mathrm{Cir}(P'',P') \ar^{\varphi}[r]& \pi_0(P'' \circ (P' \circ P)) \sqcup \mathrm{Cir}(P'',  P'\circ  P) \sqcup \mathrm{Cir}(P',P)
 }
\end{equation}
commutes, where $\varphi:=\mathrm{Id}\sqcup \mathrm{bij}_{P'',P',P}$ (see Proposition~\ref{r:2022-09-13-bijofcirclesPart}) and the notation  $\mathrm{proj}_{\bullet,\diamond}^{\pi_0}$ is as in Lemma~\ref{r:relPi0};  and $(b)$ the linear (resp.  cyclic) orders on the fibers of $\pi_0((P'' \circ P') \circ P)$ and $\pi_0(P'' \circ (P' \circ P))$ (resp. $\mathrm{Cir}(P''\circ P',P) \sqcup \mathrm{Cir}(P'',P')$ and $\mathrm{Cir}(P'',P'\circ P) \sqcup \mathrm{Cir}(P',P))$ for the composed map corresponding to the left (resp. right) column of \eqref{eq:diagram2023-02-25} (obtained using Lemma~\ref{r:2023-01-11-r2}) coincide. This follows from Lemma~\ref{r:2023-02-24r1} with $X:=\pi_0(P)$, $Y:=\pi_0(P')$ and  $Z:=\pi_0(P'')$ and $\mathrm{succ}_{XY}$ and $\mathrm{succ}_{YZ}$  given by Lemma~\ref{r:2023-01-11-r1}. 
\end{proof}

\begin{lemma}\label{lemma:CO:lmo}

The collection of  assignments: $(a)$ and $(c)$ as in Lemma~\ref{r:lemmaforLMO1} and  
\begin{itemize}
\item[$(e)$] the assignment taking $P\in\vec{\mathcal Br}$  to the natural involution $\mathrm{co}^{\mathcal A}(P,\bullet):\mathring{\mathcal A}^{\wedge}(P,\bullet)\to \mathring{\mathcal A}^{\wedge}(P,\bullet)$ of the functor $\mathring{\mathcal A}^{\wedge}(P,\bullet):\mathcal Set_f^*\to \mathcal Vect$ taking the pointed finite set $(S,s)$ to the vector space $\mathring{\mathcal A}^{\wedge}(P,S)$, see Lemma~\ref{r:2023-02-26r1}; 
\end{itemize}
satisfies the axiom (LMO2) in Definition~\ref{def:preLMOstructure}.
\end{lemma}

\begin{proof}
It follows directly from Definition~\ref{def:changeoforientationJac} that the analogous  statement to  $(a)$  in (LMO2) in the setting of Jacobi diagrams holds. By considering the corresponding linear extensions and quotient by the STU, AS and IHX relations we then obtain the desired result. Similarly, for finite sets $S$ and $S'$  and $s_0\in S$, the diagram  
\begin{equation}\label{LMO3:16022021}
\xymatrix{
\mathring{\mathrm{Jac}}(\vec{\varnothing}, S)\times \mathring{\mathrm{Jac}}(\vec{\varnothing}, S')\ar_{\mathrm{co}^{\mathring{\mathrm{Jac}}}(\vec{\varnothing}, (S,s_0)) \times\mathrm{Id}_{\mathring{\mathrm{Jac}}(\vec{\varnothing}, S')}}[d]
\ar^-{\mathrm{Tens}^{\mathring{\mathrm{Jac}}}(\vec{\varnothing},\vec{\varnothing})_{S,S'}}[rrr]& & &\mathring{\mathrm{Jac}}(\vec{\varnothing}, S\sqcup S')
\ar^{\mathrm{co}^{\mathring{\mathrm{Jac}}}(\vec{\varnothing}, (S\sqcup S',s_0))}[d]\\ 
\mathbb{C}\mathring{\mathrm{Jac}}(\vec{\varnothing}, S)\times {\mathring{\mathrm{Jac}}}(\vec{\varnothing},S')\ar_-{\mathbb{C}\mathrm{Tens}^{\mathring{\mathrm{Jac}}}(\vec{\varnothing},\vec{\varnothing})_{S,S'}}[rrr]& & &
\mathbb{C}{\mathring{\mathrm{Jac}}}(\vec{\varnothing}, S\sqcup S')}
\end{equation}
commutes. Indeed, let $\underline{D}=(D,\varphi, \emptyset, \{\mathrm{cyc}_s\}_{s\in S})\in \mathring{\mathrm{Jac}}(\vec{\varnothing}, S)$ (resp.  $\underline{D}'=(D',\varphi', \emptyset, \{\mathrm{cyc}_{s'}\}_{s'\in S'})\in \mathring{\mathrm{Jac}}(\vec{\varnothing}, S')$). The image of 
$(\underline{D},\underline{D}')$ by $\mathrm{co}^{\mathring{\mathrm{Jac}}}(\vec{\varnothing}, (S,s_0)) \times\mathrm{Id}_{\mathring{\mathrm{Jac}}(\vec{\varnothing}, S')}$ is $\big((-1)^{|\varphi^{-1}(s_0)|}\underline{D}^{s_0},\underline{D}'\big)$, where $$ \underline{D}^{s_0}:=(D,\varphi, \emptyset, \{\mathrm{cyc}_s\}_{s\in S\setminus \{s_0\}}\sqcup \{\overline{\mathrm{cyc}}_{s_0}\}),$$
 see Definition~\ref{def:changeoforientationJac}. The image of the 
latter element by $\mathbb{C}\mathrm{Tens}^{\mathring{\mathrm{Jac}}}(\vec{\varnothing},\vec{\varnothing})_{S,S'}$ is 
\begin{equation}\label{eq:2023-02-26equ1}
(-1)^{|\varphi^{-1}(s_0)|}\big(D\sqcup D', \varphi\sqcup\varphi', \emptyset,
\{\mathrm{cyc}_s\}_{s\in S\setminus \{s_0\}}\sqcup \{\overline{\mathrm{cyc}}_{s_0}\}\sqcup  \{\mathrm{cyc}_s\}_{s'\in S'}\big),
\end{equation}
see Definition~\ref{deftensorinJac}. The image of $(\underline{D},\underline{D}')$ under the composition  $ \mathrm{co}^{\mathring{\mathrm{Jac}}}(\vec{\varnothing}, (S\sqcup S',s_0))\circ\mathrm{Tens}^{\mathring{\mathrm{Jac}}}(\vec{\varnothing},\vec{\varnothing})_{S,S'}$  is 
\begin{equation}\label{eq:2023-02-26equ2}
(-1)^{|(\varphi\sqcup \varphi')^{-1}(s_0)|}\big(D\sqcup D', \varphi\sqcup\varphi', \emptyset,
\{\mathrm{cyc}_s\}_{s\in S\setminus \{s_0\}}\sqcup \{\overline{\mathrm{cyc}}_{s_0}\}\sqcup  \{\mathrm{cyc}_s\}_{s'\in S'}\big).
\end{equation}
The equality of the expresions \eqref{eq:2023-02-26equ1} and \eqref{eq:2023-02-26equ2} follows from the equality between $(\varphi\sqcup\varphi')^{-1}(s_0)$ 
and $\varphi^{-1}(s_0)$, which shows the commutativity of  diagram \eqref{LMO3:16022021}. The  statement $(b)$ in (LMO2)  follows from the commutativity of this diagram. 
\end{proof}

\begin{lemma}\label{lmo2:2:A}  The collection of assignments $(a)$ and $(c)$ from Lemma~\ref{r:lemmaforLMO1} together with 
\begin{itemize}
\item[$(f)$]  the assignment taking a pair $(P,A)$ with $P \in \vec{\mathcal Br}$ and $A \subset \pi_0(P)$ to the natural transformation $\mathrm{dbl}_{\mathcal A}((P,A),\bullet) : \mathring{\mathcal A}^\wedge(P,\bullet)\to \mathring{\mathcal A}^\wedge(\mathrm{dbl}_{\vec{\mathcal Br}}(P,A),\bullet)$ (see Lemma~\ref{r:2023-03-04r1})
\end{itemize}
satisfies axiom (LMO3) from Definition~\ref{def:preLMOstructure}.
\end{lemma}

\begin{proof}
This is a consequence of the commutativity of the analogue of diagram \eqref{diags***23-01-2021} in (LMO3) of Definition~\ref{def:preLMOstructure} for spaces of Jacobi diagrams, which we now prove. Let $P,P' \in \vec{\mathcal Br}$, let $A \subset \pi_0(P)$ and $S$ and $S'$ be finite sets. We denote by $\mathrm{pr}:=\mathrm{proj}^{\pi_0}_{P,A} : \pi_0(\mathrm{dbl}_{\vec{\mathcal Br}}(P,A)) \to \pi_0(P)$ the natural projection map, see Lemma~\ref{r:relPi0}$(b)$. Let $\underline{D}=(D,\varphi,\{\mathrm{lin}_l\}_{l \in \pi_0(P)},\{\mathrm{cyc}_s\}_{s \in S}) \in \mathring{\mathrm{Jac}}(P,S)$ and $\underline{D}'=(D',\varphi',\{\mathrm{lin}_{l'}\}_{l' \in \pi_0(P')},\{\mathrm{cyc}_{s'}\}_{s' \in S'}) \in \mathring{\mathrm{Jac}}(P',S')$. Observe that for $\psi : \partial D \sqcup \partial D' \to \pi_0(\mathrm{dbl}_{\vec{\mathcal Br}}(P,A)) \sqcup S \sqcup \pi_0(P') \sqcup S'$ a map such that $(\mathrm{pr} \sqcup \mathrm{Id}_{S \sqcup \pi_0(P') \sqcup S'}) \circ \psi=\varphi \sqcup \varphi'$, one has: for all $l \in \pi_0(\mathrm{dbl}_{\vec{\mathcal Br}}(P,A))$, $\psi^{-1}(l)$ is contained in $\varphi^{-1}(\mathrm{pr}(l))$, which is equipped with the linear order $\mathrm{lin}_{\varphi^{-1}(\mathrm{pr}(l))}$, therefore $\psi^{-1}(l)$ is equipped with the restriction of this linear order, we denote it by  $\mathrm{lin}_{\varphi^{-1}(\mathrm{pr}(l))}|_{\psi^{-1}(l)}$; for $l' \in \pi_0(P')$, $\psi^{-1}(l')$ is equal to $(\varphi')^{-1}(l')$, which is equipped with the linear order $\mathrm{lin}_{l'}$; for $s \in S$ (resp. $s' \in S'$), $\psi^{-1}(s)$ (resp. $\psi^{-1}(s')$) is equal to $\varphi^{-1}(s)$ (resp. $(\varphi')^{-1}(s')$), which is equipped with the cyclic order $\mathrm{cyc}_s$ (resp. $\mathrm{cyc}_{s'}$).      
Then one checks that the common value of the images of $\underline{D }\otimes \underline {D}'$ by the two involved maps is equal to 
$$\sum_{\psi} 
\big(D \sqcup D',\psi,\{\mathrm{lin}_{\varphi^{-1}(\mathrm{pr}(l))}|_{\psi^{-1}(l)}\}_{l \in \pi_0(\mathrm{dbl}_{\vec{\mathcal Br}}(P,A))}, \{\mathrm{lin}_{l'}\}_{l' \in \pi_0(P')}, \{\mathrm{cyc}_s\}_{s \in S}, \{\mathrm{cyc}_s'\}_{s' \in S'}\big)$$ where the sum runs over the set of maps $\psi : \partial D \sqcup \partial D'\to \pi_0(\mathrm{dbl}_{\vec{\mathcal Br}}(P,A)) \sqcup S \sqcup \pi_0(P') \sqcup S'$   such that  $(\mathrm{pr} \sqcup \mathrm{Id}_{S \sqcup \pi_0(P') \sqcup S'}) \circ \psi=\varphi \sqcup \varphi'$.
 
\end{proof}

We summarize the results of this section in the following result.

\begin{theorem}\label{r:thmsummarizingLMO123} The collection of assignments $(a)$--$(d)$ in Lemma~\ref{r:lemmaforLMO1},   $(e)$ in Lemma~\ref{lemma:CO:lmo} and $(f)$ in Lemma~\ref{lmo2:2:A}, satisfies axioms (LMO1), (LMO2) and (LMO3) from Definition~\ref{def:preLMOstructure}.
\end{theorem}

\subsection{The natural transformation  \texorpdfstring{$\mathrm{co}^{\mathcal A}((P,a),\bullet)$}{co((P,a),.)} and the assignment \texorpdfstring{$Z_{\mathcal A}(\bullet)$}{Z_A(.)}}\label{sec:4-7}

\subsubsection{The natural transformation  \texorpdfstring{$\mathrm{co}^{\mathcal A}((P,a),\bullet)$}{co( ( P , a ), .)}} Let $P \in \vec{\mathcal Br}$, let $a \in \pi_0(P)$. Recall that there is a canonical bijection $\pi_0(\mathrm{co}_{\vec{\mathcal Br}}(P,a)) \simeq \pi_0(P)$, see Lemma~\ref{lemma:1:18:03032020}.  Let $S$ be a finite set.  For $$\big(D,\varphi,\{\mathrm{lin}_l\}_{l \in \pi_0(P)},\{\mathrm{cyc}_s\}_{s \in S}\big) \in \mathring{\mathrm{Jac}}(P,S),$$ the tuple $\big(D,\tilde{\varphi},\{\widetilde{\mathrm{lin}}_l\}_{l \in \pi_0(\mathrm{co}_{\vec{\mathcal Br}}(P,a))},\{\mathrm{cyc}_s\}_{s \in S}\big)$, where 
$\tilde{\varphi}$ is the composition of $\partial D\xrightarrow{\varphi}\pi_0(P)$ with the bijection $\pi_0(P)\simeq \pi_0(\mathrm{co}_{\vec{\mathcal Br}}(P,a))$ and where  for $l \in \pi_0(\mathrm{co}_{\vec{\mathcal Br}}(P,a)) \simeq \pi_0(P)$, $\widetilde{\mathrm{lin}}_l$ is equal to $\mathrm{lin}_l$ if $l \neq a$ and is equal to $\overline{\mathrm{lin}}_a$ (the opposite linear order of $\mathrm{lin}_a$) otherwise, belongs to $\mathring{\mathrm{Jac}}(\mathrm{co}_{\vec{\mathcal Br}}(P,a),S)$.

\begin{definition} Let $P$ be an oriented Brauer diagram and $S$ be a finite set. Define a map
\begin{equation}\label{r:2023-03-04COforsegment}
\mathrm{co}^{\mathring{\mathrm{Jac}}}((P,a),S): \mathring{\mathrm{Jac}}(P,S) \longrightarrow \mathbb{C} \mathring{\mathrm{Jac}}(\mathrm{co}_{\vec{\mathcal Br}}(P,a),S)
\end{equation}  \index[notation]{co^{\mathring{\mathrm{Jac}}}((P,a),S)@$\mathrm{co}^{\mathring{\mathrm{Jac}}}((P,a),S)$}
 by  $\big(D,\varphi,\{\mathrm{lin}_l\}_{l \in \pi_0(P)},\{\mathrm{cyc}_s\}_{s \in S}\big) \mapsto  (-1)^{|\varphi^{-1}(a)|}\big(D,\tilde{\varphi},\{\widetilde{\mathrm{lin}}_l\}_{l \in \pi_0(\mathrm{co}_{\vec{\mathcal Br}}(P,a))},\{\mathrm{cyc}_s\}_{s \in S}\big)$. 
\end{definition}

This map is illustrated in Figure \ref{fig:co:A}. 

\begin{figure}[ht!]
\centering
\includegraphics[scale=0.8]{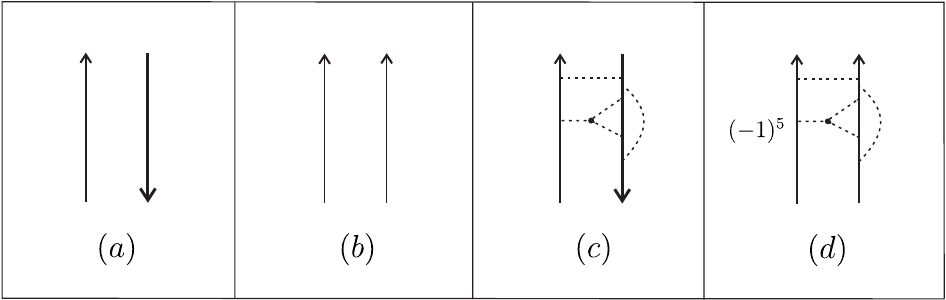}
\caption{Schematic illustration of the change of orientation map $\mathrm{co}^{\mathring{\mathrm{Jac}}}((\uparrow\downarrow,\downarrow),\emptyset): \mathring{\mathrm{Jac}}(\uparrow\downarrow,\emptyset) \longrightarrow \mathbb{C} \mathring{\mathrm{Jac}}(\mathrm{co}_{\vec{\mathcal Br}}(\uparrow\downarrow,\downarrow),\emptyset) = \mathring{\mathrm{Jac}}(\uparrow\uparrow,\emptyset)$.  In $(a)$ we show the schematic representation of the oriented Brauer diagram $\uparrow\downarrow$  with $a\in\pi_0(\uparrow\downarrow)$ indicated by using thick lines. In~$(b)$ we show the schematic representation of $\mathrm{co}_{\vec{\mathcal Br}}(\uparrow\downarrow,\downarrow) = \uparrow \uparrow$.  In $(c)$  we show a schematic representation of a Jacobi diagram on $(\uparrow\downarrow,\emptyset)$ and in $(d)$ it image under the map $\mathrm{co}^{\mathring{\mathrm{Jac}}}((\uparrow\downarrow,\downarrow),\emptyset)$.}
\label{fig:co:A}
\end{figure}

\begin{lemma}\label{r:2022-06-13changeoforientation} Let $P$ be an oriented Brauer diagram and $S$ be a finite set. The map~\eqref{r:2023-03-04COforsegment} induces a continuous linear isomorphism
\begin{equation}
\mathrm{co}^{\mathcal A}((P,a),S): \mathring{\mathcal A}^\wedge(P,S) \to  \mathring{\mathcal A}^\wedge(\mathrm{co}_{\vec{\mathcal Br}}(P,a),S).
\end{equation} \index[notation]{co^{\mathcal A}((P,a),S)@$\mathrm{co}^{\mathcal A}((P,a),S)$}
The assignment $S\mapsto \mathrm{co}^{\mathcal A}((P,a),\bullet)$  defines a natural isomorphism of the functors $\mathring{\mathcal A}^\wedge(P,\bullet):\mathcal Set_f \to \mathcal Vect$ and $\mathring{\mathcal A}^\wedge(\mathrm{co}_{\vec{\mathcal Br}}(P,a),\bullet):\mathcal Set_f \to \mathcal Vect$ .
\end{lemma}

 \begin{proof} It follows easily from the definition that the linear extension of the map~\eqref{r:2023-03-04COforsegment} is compatible with the STU, AS and IHX relations. One also checks that the assignment  $S\mapsto \mathrm{co}^{\mathcal A}((P,a),\bullet) $ is compatible with bijections $f:S\to S'$. 
 \end{proof}

\begin{convention}\label{convention:co_i-dbl_i2} If ${P}=\mathrm{Id}_w\in\underline{\vec{\mathcal Br}}(w,w)$ for $w\in\{+,-\}^*$, then for $i\in[\![1,|w|]\!]$, we denote by   $\mathrm{co}_i$   \index[notation]{co_i@$\mathrm{co}_i$} the map $\mathrm{co}_{\mathcal{A}}{(P,\tilde{w}(i)})$, where $\tilde{w}(i)$ denotes the element in $\pi_0(\mathrm{Id}_w)$ corresponding to $i$. Moreover, if $|w|=1$ we write $\mathrm{co}$ instead of $\mathrm{co}_1$. 
\end{convention}

\subsubsection{The assignment \texorpdfstring{$Z_{\mathcal A}(\bullet)$}{Z_A(.)}}\label{sec:3-8-2} In this subsection we associate to a parenthesized tangle $T$ an element  $Z_{\mathcal A}(T)\in  \mathring{\mathcal{A}}^{\wedge}(\vec{\mathrm{br}}(T),\pi_0(T))$ called the Kontsevich integral of $T$, which is an isotopy invariant of parenthesized tangles, and therefore defines an element $Z_{\mathcal A}(\mathbf{T})\in  \mathring{\mathcal{A}}^{\wedge}(\vec{\mathrm{br}}(\mathbf{T}),\pi_0(\mathbf{T}))$ for any $\mathbf{T}\in \mathcal Pa\vec{\mathcal T}$. We  follow the combinatorial approach  \cite{LM2,JacksonMoffat,Ohts}.

Recall that a family of morphisms in a monoidal category $\underline{\mathcal{C}}$ is a \emph{generating family} if every morphism of $\underline{\mathcal{C}}$ can be obtained from its morphisms and the morphisms $\mathrm{Id}_x$ with $x\in\mathrm{Ob}(\underline{\mathcal{C}})$ by a finite number of applications of the tensor product and composition.  

The following can be deduced from ~\cite[Thm.~3.2]{Tur1}), see also \cite[Thm.~1]{BN2} or \cite[Thm.~6.5]{Ohts}.

\begin{theorem}[Turaev] Let $\mathcal E$ be the family consisting of the set of parenthesized tangles  
 shown in Figure~\ref{figure_gen},  where the thick lines in the figures in position three and four (from left to right)  represent a trivial tangle (with arbitrary orientation) and the black dots some non-associative words on $\{+,-\}$.
 \begin{figure}[ht!]
										\centering
                        \includegraphics[scale=1.4]{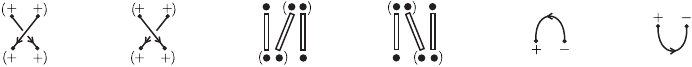}
												
\caption{Elements of the family $\mathcal{E}$.}
\label{figure_gen} 								
\end{figure}
Then the family~$\mathcal E$ is a generating  family of the monoidal category  $\underline{\mathcal Pa\vec{\mathcal T}}$.
\end{theorem}

\begin{lemma} The parenthesized tangles shown in Figure~\ref{figure_genpar}  
 \begin{figure}[ht!]
										\centering
                        \includegraphics[scale=1.4]{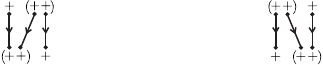}
			\caption{Particular elements of $\mathcal{E}$.}
\label{figure_genpar} 								
\end{figure}
 are particular  elements of $\mathcal E$ and all  the elements in position three and four (from left to right)   in Figure~\ref{figure_gen}    can be obtained from these particular elements  by applying to them a finite number of doubling operations $\mathrm{dbl}_{\mathcal Pa\vec{\mathcal T}}$   and change of orientation operations  $\mathrm{co}_{\mathcal Pa\vec{\mathcal T}}$ (see Definition~\ref{def:DPaT}).
\end{lemma} 
\begin{proof} The result follows straightforwardly. 
\end{proof}

To define $Z_{\mathcal A}$ we need and \emph{associator}   $\Phi$. More precisely, let $\mathbb{C}\langle\langle A,B\rangle\rangle$ be the power series ring on the non-commutative free variables $A$ and $B$.  We fix an \emph{even Drinfeld associator} $\varphi(A,B) \in \mathbb{C}\langle\langle A,B \rangle\rangle$. This is an exponential series satisfying several conditions, see \cite{BN3, Dri} or \cite[Appendix D]{Ohts} for the precise definitions.

There is an algebra morphism
\begin{equation}\label{equ:morphcanC-ABtoA}
\mathrm{can}: \mathbb{C}\langle\langle A,B\rangle\rangle\longrightarrow  \mathring{\mathcal A}^\wedge(\downarrow\downarrow\downarrow, \emptyset)
\end{equation}
given by
\begin{align*}
\includegraphics[scale=1]{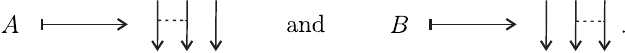}
\end{align*}
Set $\Phi:=\mathrm{can}(\varphi(A,B))\in\mathring{\mathcal A}^\wedge(\downarrow\downarrow\downarrow, \emptyset)$. Therefore $\Phi$ is a exponential series of Jacobi diagrams satisfying several properties, see~\cite[(6.11)--(6.13)]{Ohts}. The element $\Phi$ is also called an \emph{even Drinfeld associator}. In low degree we have:

\begin{align*}
\includegraphics[scale=1]{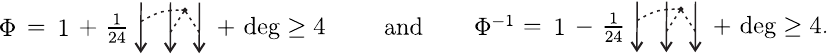}
\end{align*}

\noindent Here $1$ means $\downarrow\downarrow\downarrow$. Set
\begin{align*}
\includegraphics[scale=1]{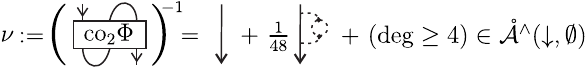}
\end{align*}
where $\mathrm{co}_2$ is as in Convention~\ref{convention:co_i-dbl_i}. One can show that the element $\nu\in\mathring{\mathcal{A}}^\wedge(\downarrow,\emptyset)$  \index[notation]{\nu@$\nu$} does not depend on the choice of the associator $\Phi$, see for instance \cite[Thm. 6.7]{Ohts}. Its explicit value  was conjectured in~\cite{BGRTw}, this is known as the \emph{wheels conjecture}, and it was proved  in~\cite{BLT}.

We also need the square root of $\nu$:
\begin{align*}
\includegraphics[scale=1.1]{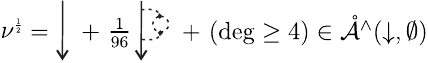}
\end{align*}

The following is the main theorem of this subsection, it was obtained by Le--Murakami~\cite{LM1,LM2}, see  also \cite{BN2,Ca, Piu}. We refer to  \cite[\S17.2]{JacksonMoffat} for a detailed explanation.

\begin{theorem}[Le--Murakami]\label{thmkontsevichintegral} There is a unique assignment $$\mathcal Pa\mathcal T \ni \mathbf{T} \longmapsto Z_{\mathcal A}(\mathbf{T}) \in \mathring{\mathcal A}^{\wedge}(\vec{\mathrm{br}}(\mathbf{T}),\pi_0(\mathbf{T})_{\mathrm{cir}}),$$   
  \index[notation]{Z_{\mathcal A}(\bullet)@$Z_{\mathcal A}(\bullet)$} such that 
\begin{itemize}
\item[$(a)$]  ${Z}_{\mathcal A}(\mathrm{Id}_w) = \mathrm{Id}_{w'}\in {\mathring{\mathcal{A}}}^\wedge(\downarrow_{|w'|},\emptyset)$ for any $w\in\{+\}^{(*)}$, where $w'\in\{+\}^*$ is the image under the natural forget-parenthesization map $\{+\}^{(*)}\to \{+\}^{*}$ and $\downarrow_{|w'|}=\mathrm{Id}_{w'}\in\underline{\vec{\mathcal Br}}(w',w')$ and
\begin{align*}
\includegraphics[scale=1]{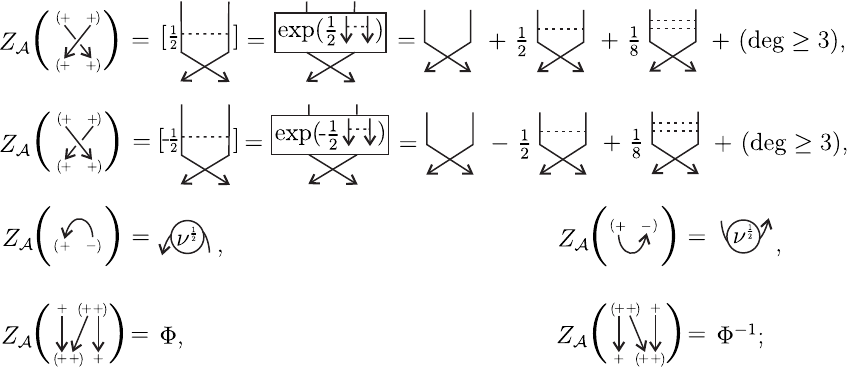}
\end{align*}
and on parenthesized tangles shown in position three and four (from left to right)   in Figure~\ref{figure_gen} 
${Z}_{\mathcal A}$ is defined using  the corresponding  doubling and change of orientation operations applied to the parenthesized tangles in  Figure~\ref{figure_genpar}, for instance 
\begin{align*}
\includegraphics[scale=1]{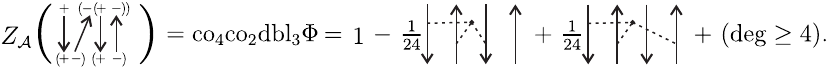}
\end{align*}
Here we are using Convention~\ref{convention:co_i-dbl_i}.
\item[$(b)$] $Z_{\mathcal A}$ satisfies axiom \emph{(LMO4)}$(b)$ in Definition~\ref{def:preLMOstructure}.

\item[$(c)$] $Z_{\mathcal A}$ satisfies axiom \emph{(LMO4)}$(c)$ in Definition~\ref{def:preLMOstructure}.

\item[$(d)$] $Z_{\mathcal A}$ satisfies axiom \emph{(LMO4)}$(d)$ from Definition~\ref{def:preLMOstructure} and   for any $\mathbf{T}\in \underline{\mathcal Pa\vec{\mathcal T}}$ and $a\in\pi_0(\mathbf{T})_{\mathrm{seg}}\simeq\pi_0(\vec{\mathrm{br}}(\mathbf{T}))$ we have
$$Z_{\mathcal A}(\mathrm{co}_{\mathcal Pa\vec{\mathcal T}}(\mathbf{T},a))=\mathrm{co}^{\mathcal A}\big((\vec{\mathrm{br}}(\mathbf{T}),a),\pi_0(\mathbf{T})_{\mathrm{cir}}\big)(Z_{\mathcal A}(\mathbf{T})).$$  

\end{itemize}
\end{theorem}

\begin{proof} Denote by $Z_{\mathcal A}$ the assignment $\check Z$ from \cite[\S17.2]{JacksonMoffat}. Then $Z_{\mathcal A}$  satisfies property  $(a)$ by  \cite[Def. 17.7, 17.5 and $(17.3)$]{JacksonMoffat}, properties $(b)$, $(c)$, $(d)$, by \cite[Def. 17.5 and 17.7]{JacksonMoffat}. One can check that these conditions define $Z_{\mathcal A}$ uniquely. 
\end{proof}

 It is important to notice that the value  of ${Z}_{\mathcal A}$ on a framed oriented link does not depend on the choice of an associator~\cite[Thm.~8]{LM2}. 

By definition, If $U$ denotes the unknot with trivial framing, we have ${Z}_{\mathcal A}(U) =\nu\in\mathring{\mathcal{A}}^\wedge(\downarrow, \emptyset)\cong\mathring{\mathcal{A}}^\wedge(\vec{\varnothing}, \pi_0(U)_{\mathrm{cir}})$, see \cite[Proposition~6.3]{Ohts} for the isomorphism  $\mathring{\mathcal{A}}^\wedge(\downarrow, \emptyset)\cong\mathring{\mathcal{A}}^\wedge(\vec{\varnothing},\{s\})$.

\begin{convention}\label{convention:red-box} For an element $\mathbf{T}\in\mathcal{P}a\vec{\mathcal T}(u,v)$ with  parenthesized tangle representative $T$, we denote by 
$${\underset{u'}{\overset{v'}{\begin{tikzpicture}[squarednode/.style={rectangle, draw=red!, very thick, minimum size=5mm,}
]
\node[squarednode] (1){$T$};
\end{tikzpicture}}}}$$
the element ${Z}_{\mathcal A}(\mathbf{T})\in{\mathring{\mathcal{A}}}^{\wedge}(\vec{\mathrm{br}}(\mathbf{T}), \pi_0(\mathbf{T}_{\mathrm{circ}}))$, where $\vec{\mathrm{br}}(\mathbf{T})\in\vec{\mathcal Br}(u',v')$ and $u',v'\in\{+\}^*$ are the images under the natural forget-parenthesization map $\{+,-\}^{(*)}\to \{+,-\}^{*}$. Moreover, sometimes we will omit the symbols $+$ and $-$ from the notation since these can be recovered from the orientation of $T$ or from the orientation of its underlying Brauer diagram. Since ${Z}_{\mathcal A}$ defines an isotopy invariant of parenthesized tangles, we will write sometimes ${Z}_{\mathcal A}(T)$ instead of ${Z}_{\mathcal A}(\mathbf{T})$ for any $\mathbf{T}\in \mathcal Pa\vec{\mathcal T}(u,v)$ with tangle representative~$T$.
\end{convention}

\begin{example}\label{ejemplo1KI} Here we use Conventions~\ref{convention:red-box} and~\ref{convention:co_i-dbl_i}.
\begin{align*}
\includegraphics[scale=1]{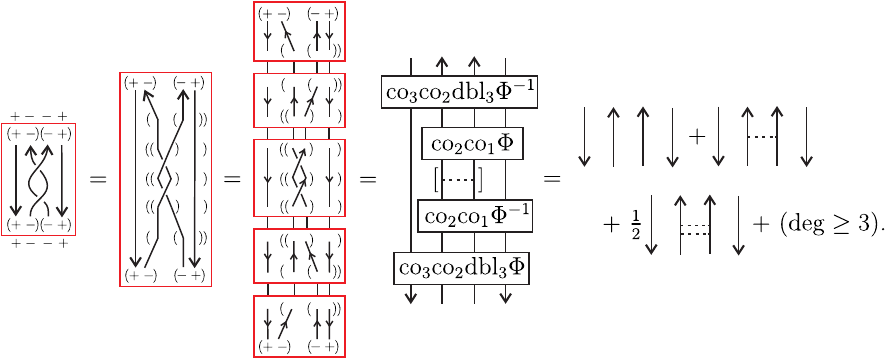}
\end{align*}
\end{example}

\subsection{Proof of the pre-LMO axiom (LMO4) and  the pre-LMO structure  \texorpdfstring{$\mathring{\mathbf{A}}$}{AA}}\label{sec:4-8}

\begin{lemma}\label{lemmaZanddoubling}  The collection of assignments $(a)$--$(d)$ in Lemma~\ref{r:lemmaforLMO1},   $(e)$ in Lemma~\ref{lemma:CO:lmo} and $(f)$ in Lemma~\ref{lmo2:2:A} together with
\begin{itemize}
\item[$(g)$]  the assignment $\mathcal Pa\vec{\mathcal T} \ni \mathbf{T}\mapsto Z_{\mathcal A}(\mathbf{T})\in  {\mathring{\mathcal{A}}}^{\wedge}(\vec{\mathrm{br}}(\mathbf{T}), \pi_0(\mathbf{T}_{\mathrm{circ}}))$
\end{itemize}
satisfies axiom \emph{(LMO4)} from Definition~\ref{def:preLMOstructure}.
\end{lemma}

\begin{proof} Items $(a)$--$(d)$ in (LMO4) follow from Theorem~\ref{thmkontsevichintegral}. Item $(e)$ in (LMO4) is a reformulation of~\cite[Theorem~4.2]{LeMu97}. See also~\cite[\S 6.5]{Ohts}. 
\end{proof}

We can summarize the results in  \S\ref{sec:4-1}~--~\S\ref{sec:4-7} in the following theorem.

\begin{theorem}\label{thm:prelmokont}The tuple  \begin{equation*}
\begin{split}
\mathring{\mathbf{A}}:=\Big(\big\{\mathring{\mathcal A}^\wedge(P,\bullet)\big\}_{P \in \vec{\mathcal  Br}}, \big\{\mathrm{Tens}^{{\mathcal A}}(P,P')_{\bullet,\bullet'}\big\}_{P,P' \in \vec{\mathcal Br}},  \big\{\mathrm{Comp}^{{\mathcal A}}(P,P')_{\bullet,\bullet'}\big\}_{P,P' \text{ composable in } \vec{\mathcal Br}},\\
 \big\{\mathrm{co}^{{\mathcal A}}(P,\bullet)\big\}_{P \in \vec{\mathcal Br}},\big\{\mathrm{dbl}_{{\mathcal A}}((P,A),\bullet)\big\}_{P \in \vec{\mathcal Br}, A \subset \pi_0(P)}, Z_{{\mathcal A}}(\bullet)\Big)
\end{split}
\end{equation*}
is  a pre-LMO structure in the sense of Definition~\ref{def:preLMOstructure}.
\end{theorem}

Since $\mathring{\mathbf{A}}$ is a pre-LMO structure, one may apply Proposition~\ref{l:LMOinvt}, so that one may attach to it the semi-Kirby structure $(\mathfrak{s}(\mathbf{A}),\mu)$.

\subsection{Additional material}\label{sec:4-9}

\subsubsection{Some relations among Jacobi diagrams}

\begin{definition}\label{def:cup:capinBR} Define the oriented Brauer diagrams 
 $\mathrm{cup}_{+-}:=\big((0,2,(12)),\{1\}\big)\in\underline{\vec{\mathcal Br}}(\emptyset,+-)$, where $(12)$ is the only element in $\mathrm{FPFI}(\emptyset \sqcup [\![1,2]\!])$ and $\{1\} \subset \emptyset \sqcup [\![1,2]\!]$ is the set beginning points, see  Definition~\ref{def:vecBr}. Similarly, define  $\mathrm{cup}_{-+}:=\big((0,2,(12)),\{2\}\big)\in\underline{\vec{\mathcal Br}}(\emptyset,-+)$,    $\mathrm{cap}_{-+}:=\big((2,0,(12)),\{1\}\big)\in\underline{\vec{\mathcal Br}}(-+,\emptyset)$ and $\mathrm{cap}_{+-}:=\big((2,0,(12)),\{2\}\big)\in\underline{\vec{\mathcal Br}}(+-,\emptyset)$. Finally, define $\tau_{++}:=((2,2, (1\bar{2})(2\bar{1})), \{\bar{1},\bar{2}\})\in\underline{\vec{\mathcal Br}}(++,++)$, we use Notation~\ref{notation1}.  The schematic representations of $\mathrm{cup}_{+-}$, $\mathrm{cup}_{-+}$, $\mathrm{cap}_{-+}$ and $\mathrm{cap}_{+-}$ coincide with the schematic representation of the oriented tangles shown in Figure~\ref{figure_cup_cap} and that of $\tau_{++}$ is shown in Figure~\ref{figura2024-01-17}.
\end{definition}
\begin{figure}[ht!]
\centering
\includegraphics[scale=1]{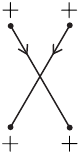}
\caption{Schematic representation of the oriented Brauer diagram $\tau_{++}\in\underline{\vec{\mathcal Br}}(++,++)$.}
\label{figura2024-01-17}
\end{figure}

\begin{definition}\label{themapz} Define the map
$$z:\mathring{\mathcal{A}}^{\wedge}(\downarrow,\emptyset)\longrightarrow \mathring{\mathcal{A}}^{\wedge}(\uparrow,\emptyset)$$
by $$z(x) = (\mathrm{Id}_-^{\mathcal A}\otimes_{\mathcal A} (\mathrm{cap}_{+-})_{\mathcal A})\circ_{\mathcal A} (\mathrm{Id}_-^{\mathcal A}\otimes_{\mathcal A} x\otimes_{\mathcal A} \mathrm{Id}_-^{\mathcal A}) \circ_{\mathcal A} ((\mathrm{cup}_{-+})_{\mathcal A}\otimes_{\mathcal A} \mathrm{Id}_-^{\mathcal A})$$ for any $x\in \mathring{\mathcal{A}}^{\wedge}(\downarrow, \emptyset)$.Here $\mathrm{Id}_-^{\mathcal A}=(\mathrm{Id}_-)_{\mathcal A}\in \mathring{\mathcal{A}}^\wedge(\mathrm{Id}_{-},\emptyset)$ and $(\mathrm{cup}_{-+})_{\mathcal A}\in \mathring{\mathcal{A}}^\wedge(\mathrm{cup}_{-+},\emptyset)$ are as in Definition~\ref{def:emptyJacdiagram},   where $\mathrm{Id}_-\in\underline{\vec{\mathcal Br}}(-,-)$ is the identity morphism and $\mathrm{cup}_{-+}\in \underline{\vec{\mathcal Br}}(\emptyset,-+) $ as in Definition~\ref{def:cup:capinBR}. In Figure~\ref{figuraKI17-mapz} we show a schematic representation of this map.
\begin{figure}[ht!]
\centering
\includegraphics[scale=1]{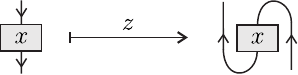}
\caption{Schematic representation of the map $z:{\mathring{\mathcal{A}}}^{\wedge}(\downarrow, \emptyset)\to {\mathring{\mathcal{A}}}^{\wedge}(\uparrow, \emptyset)$. We represent $x\in{\mathring{\mathcal{A}}}^{\wedge}(\downarrow, \emptyset)$  by the grey-colored box.}
\label{figuraKI17-mapz}
\end{figure}
\end{definition}

\begin{lemma}\label{themapzalgebr} The map $z:\mathring{\mathcal{A}}^{\wedge}(\downarrow, \emptyset)\to \mathring{\mathcal{A}}^{\wedge}(\uparrow, \emptyset)$ from Definition~\ref{themapz} is an algebra anti-homomorphism, i.e, $z(x\circ y) = z(y)\circ z(x)$ for any $x,y\in\mathring{\mathcal{A}}^\wedge(\downarrow, \emptyset)$.
\end{lemma}

\begin{proof}
Let $x,y\in\mathring{\mathcal{A}}^\wedge(\downarrow, \emptyset)$. In Figure~\ref{figuraKI21-mapzalgebrahom} we give a schematic proof of the equality $z(x\circ y) = z(y)\circ z(x)$.
\begin{figure}[ht!]
\centering
\includegraphics[scale=1]{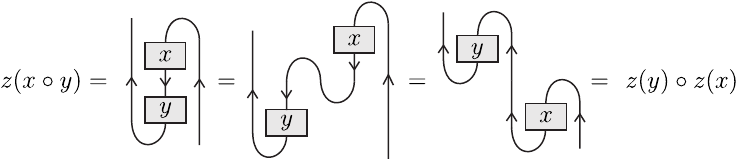}
\caption{The map $z:{\mathring{\mathcal{A}}}^{\wedge}(\downarrow, \emptyset)\to {\mathring{\mathcal{A}}}^{\wedge}(\uparrow, \emptyset)$ is an algebra anti-homomorphism.}
\label{figuraKI21-mapzalgebrahom}
\end{figure}
\end{proof}

\begin{lemma}\label{slide-nu} For any $x\in\mathring{\mathcal{A}}^{\wedge}(\downarrow, \emptyset)$, we have 
\begin{align*}
\includegraphics[scale=1]{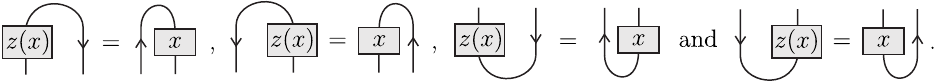}
\end{align*}
where $z:\mathring{\mathcal{A}}^{\wedge}(\downarrow, \emptyset)\to \mathring{\mathcal{A}}^{\wedge}(\uparrow, \emptyset)$ is map from Definition~\ref{themapz}.
\end{lemma}
\begin{proof} We prove the first equality, the other ones follow similarly. We have
\begin{align*}
\includegraphics[scale=1]{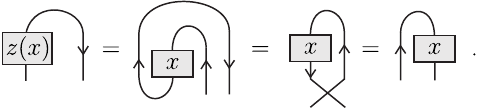}
\end{align*}
\end{proof}

\begin{lemma}\label{commut-prop} Let $w_1,w_2\in\{+,-\}^*$, $w=w_1\otimes +\otimes w_2$, $w'=w_1\otimes -\otimes w_2$. Recall the notation  $\mathrm{Id}_w \in \vec{\underline{\mathcal Br}}(w,w)$ and $\mathrm{Id}_{w'}  \in \vec{\underline{\mathcal Br}}(w',w')$. Let $D\in\mathring{\mathcal{A}}^\wedge(\mathrm{Id}_w,\emptyset)$, $E\in\mathring{\mathcal{A}}^\wedge(\mathrm{Id}_{w'},\emptyset)$, $y\in\mathring{\mathcal{A}}^\wedge(\downarrow,\emptyset)$ and $x'\in\mathring{\mathcal{A}}^\wedge(\uparrow, \emptyset)$. Then
$$D\circ(\mathrm{Id}_{w_1}\otimes x \otimes \mathrm{Id}_{w_2}) = (\mathrm{Id}_{w_1}\otimes x \otimes \mathrm{Id}_{w_2})\circ D \quad \text{and} \quad E\circ(\mathrm{Id}_{w_1}\otimes y \otimes \mathrm{Id}_{w_2}) = (\mathrm{Id}_{w_1}\otimes y \otimes \mathrm{Id}_{w_2})\circ E.$$
Schematically
\begin{align*}
\includegraphics[scale=1]{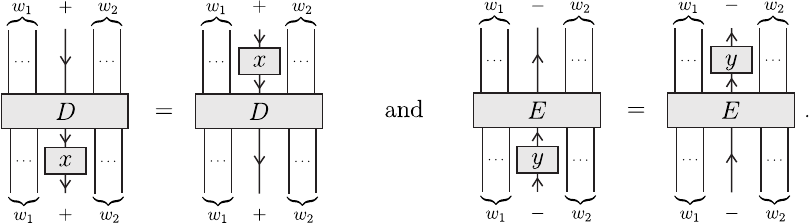}
\end{align*}
\end{lemma}
\begin{proof}
This is  a well-known result, see for instance \citep[Lemma~16.13]{JacksonMoffat} for the case $w_1=\emptyset$. 
\end{proof}

\begin{lemma}\label{commut-prop-doubl} Let $w_1,w_2\in\{+,-\}^*$ with $|w_1|=i-1\geq 1$, $w=w_1\otimes +\otimes w_2$, $w'=w_1\otimes -\otimes w_2$. Let $D\in\mathring{\mathcal{A}}^\wedge(\mathrm{Id}_w,\emptyset)$, $E\in\mathring{\mathcal{A}}^\wedge(\mathrm{Id}_{w'},\emptyset)$, $y\in\mathring{\mathcal{A}}^\wedge(\downarrow^n,\emptyset)$ and $x'\in\mathring{\mathcal{A}}^\wedge(\uparrow^n,\emptyset)$ where $n\geq 1$. Then
$$\mathrm{dbl}_i^{(n-1)}(D)\circ(\mathrm{Id}_{w_1}\otimes x \otimes \mathrm{Id}_{w_2}) = (\mathrm{Id}_{w_1}\otimes x \otimes \mathrm{Id}_{w_2})\circ \mathrm{dbl}_i^{(n-1)}(D)$$
{and}
$$\mathrm{dbl}_i^{(n-1)}(E)\circ(\mathrm{Id}_{w_1}\otimes y \otimes \mathrm{Id}_{w_2}) = (\mathrm{Id}_{w_1}\otimes y \otimes \mathrm{Id}_{w_2})\circ \mathrm{dbl}_i^{(n-1)}(E).$$
Schematically
\begin{align*}
\includegraphics[scale=1]{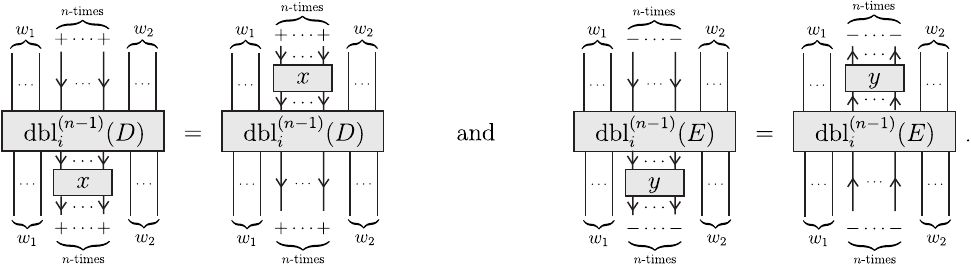}
\end{align*}
\end{lemma}
\begin{proof}
This is  a well-known result, see for instance \citep[Lemma~16.14]{JacksonMoffat} for the case $w_1=\emptyset$. 
\end{proof}

\subsubsection{The functor $\mathring{\mathcal A}(P,\bullet,\diamond)$ of Jacobi diagrams with free vertices}\label{sec:310-2}

We introduce a slight generalization of Definition~\ref{def:thesetJacPS}. 
\begin{definition}\label{generalJDwithfree} Let $P$ be an oriented Brauer diagram and let $S$ and $\Sigma$ be finite sets. Let $\mathring{\mathrm{Jac}}(P,S,\Sigma)$   \index[notation]{Jac(P,S,\Sigma)@$\mathring{\mathrm{Jac}}(P,S,\Sigma)$}  be the set consisting  of classes of tuples
$\underline D = (D,\varphi,\{\mathrm{lin}_l\}_{l \in \pi_0(P)},\{\mathrm{cyc}_s\}_{s \in S})$ where 
\begin{itemize}
\item $D$ is a vertex-oriented unitrivalent graph, 
\item $\varphi : \partial D \to \pi_0(P) \sqcup S \sqcup \Sigma$ is a map such that its image  contains $\Sigma$ and the preimage of any element of $\Sigma$ has cardinality $1$,
\item   $\mathrm{lin}_l$ is a linear order on $\varphi^{-1}(l)$ for each $l \in \pi_0(P)$, 
\item $\mathrm{cyc}_s$ is a cyclic order on $\varphi^{-1}(s)$ for each $s \in S$.
\end{itemize}
Two such tuples $\underline D$ and $\widetilde{\underline D}$ being considered equivalents if there exists a graph isomorphism $\psi: D\to \widetilde{D}$ compatible with all the structures (see  Definition~~\ref{def:thesetJacPS}). 
\end{definition}

In particular,  $\mathring{\mathrm{Jac}}(P,S,\emptyset)$ coincides with the set $\mathring{\mathrm{Jac}}(P,S)$  as in Definition~\ref{def:thesetJacPS}.

\begin{definition}\label{operationVOUTG} Let $D$ and $D'$ be vertex-oriented unitrivalent graphs and $(d,d',\theta)$ be a triple where $d \subset \partial D$, $d' \subset \partial D'$ and $\theta : d \to d'$ is a bijection. Define \emph{the glueing of $D$ and $D'$ along $(d,d',\theta)$}, denoted by $D \sqcup_{(d,d',\theta)}D'$, as the vertex-oriented unitrivalent graph obtained from $D$ and $D'$ by glueing together the univalent vertices with endpoints in $d$ and $d'$ along $\theta$. Notice that  $\partial(D \sqcup_{(d,d',\theta)}D')=(\partial D \setminus d) \sqcup (\partial D'\setminus d')$.   
\end{definition}

\begin{definition}\label{themappairing} Let ${P}$ and ${P}'$ be oriented Brauer diagrams and $S,S',\Sigma$  finite sets. Define the map
\begin{equation}\label{thepairing}
\langle \ ,\  \rangle : \mathring{\mathrm{Jac}}({P},S, \Sigma)\times \mathring{\mathrm{Jac}}({P}',S',\Sigma) \longrightarrow \mathring{\mathrm{Jac}}({P}\otimes {P}', S\sqcup S') 
\end{equation}
as follows. Let  $(\underline{D},\underline{D}')\in \mathring{\mathrm{Jac}}({P},S, \Sigma)\times \mathring{\mathrm{Jac}}({P}',S', \Sigma)$ where
$$\underline D=(D,\varphi : \partial D\to \pi_0(P) \sqcup S \sqcup \Sigma,\{\mathrm{lin}_l\}_{l \in \pi_0(P)}, \{\mathrm{cyc}_s\}_{s \in S})$$
and
$$\underline D'=(D',\varphi' : \partial D'\to \pi_0(P') \sqcup S' \sqcup \Sigma,\{\mathrm{lin}_{l'}\}_{l' \in \pi_0(P')}, \{\mathrm{cyc}_{s'}\}_{s' \in S'}).$$
Define  $\langle \underline{D} , \underline{D}'  \rangle\in \mathring{\mathrm{Jac}}({P}\otimes {P}', S\sqcup S')$  by 
$$\langle \underline D,\underline D'\rangle :=\big(D \sqcup_{(d,d',\theta)} D',\varphi_{|\partial D\setminus \varphi^{-1}(\Sigma)}
\sqcup \varphi'_{|\partial D'\setminus \varphi'^{-1}(\Sigma)},\{\mathrm{lin}_l\}_{l \in \pi_0(P)}\sqcup\{\mathrm{lin}_{l'}\}_{l' \in \pi_0(P')}, \{\mathrm{cyc}_s\}_{s \in S}\sqcup\{\mathrm{cyc}_{s'}\}_{s' \in S'}\big)$$
where $ d=\varphi^{-1}(\Sigma)$,  $d'=(\varphi')^{-1}(\Sigma)$ and $\theta : \varphi^{-1}(\Sigma) \to (\varphi')^{-1}(\Sigma)$ is the composition of the bijections  $\varphi^{-1}(\Sigma)\to \Sigma$ and $(\varphi')^{-1}(\Sigma)\to \Sigma$.
\end{definition}

For any oriented Brauer diagrams  $P,P'$, and finite sets $S$, $S'$, $\Sigma$, $\Sigma'$ define the tensor product map
\begin{equation}\label{tensor:map:A:PP'Jacgen} 
\mathring{\mathrm{Jac}}({P},S,\Sigma)\times \mathring{\mathrm{Jac}}({P}',S',\Sigma')\longrightarrow \mathring{\mathrm{Jac}}({P}\otimes {P}', S\sqcup S', \Sigma \sqcup \Sigma')
\end{equation}
taking  the pair $(\underline{D},\underline{D}')\in \mathring{\mathrm{Jac}}({P},S, \Sigma)\times \mathring{\mathrm{Jac}}({P}',S, \Sigma')$ where $\underline D=(D,\varphi ,\{\mathrm{lin}_l\}_{l \in \pi_0(P)}, \{\mathrm{cyc}_s\}_{s \in S})$ and $\underline D'=(D',\varphi' ,\{\mathrm{lin}_{l'}\}_{l' \in \pi_0(P')}, \{\mathrm{cyc}_{s'}\}_{s' \in S'})$ to the Jacobi diagram $\underline{D}\otimes \underline D'\in \mathring{\mathrm{Jac}}({P}\otimes {P}', S\sqcup S', \Sigma \sqcup \Sigma')$ given by 
$$\underline{D}\otimes \underline D' = \big(D \sqcup D',\varphi \sqcup \varphi',\{\mathrm{lin}_l\}_{l \in \pi_0(P)}\sqcup \{\mathrm{lin}_{l'}\}_{l' \in \pi_0(P')}, \{\mathrm{cyc}_{s}\}_{s \in S}\sqcup\{\mathrm{cyc}_{s'}\}_{s' \in S'}\big).$$
The above definition coincides with the definition tensor product map \eqref{tensor:map:A:PP'Jac} when $\Sigma=\Sigma'=\emptyset$.

\begin{definition}\label{spacewithfreelegs} Let $P$ be an oriented Brauer diagram and let $S$ and $\Sigma$ be finite sets. The \emph{space of Jacobi diagrams on} $({P},S, \Sigma)$ is the $\mathbb{C}$-vector space generated by the set $\mathring{\mathrm{Jac}}(P,S,\Sigma)$ of Jacobi diagrams on $(P,S,\Sigma)$ up to the STU, AS, and IHX relations:
$$
\mathring{\mathcal{A}}({P},S,\Sigma):=\frac{{\mathbb{C}}\mathring{\mathrm{Jac}}({P},S,\Sigma)}{\text{(STU, AS, IHX)}}
$$ \index[notation]{A({P},S,\Sigma)@$\mathring{\mathcal{A}}({P},S,\Sigma)$}
where the relations STU, AS, and IHX are as in Definition~\ref{thespaceAPS}. 
\end{definition}

For $P$ an oriented Brauer diagram and $S$, $\Sigma$ finite sets, the formula \eqref{degree-on-APS} defines a degree $\mathrm{deg}$ on $\mathring{\mathcal{A}}({P},S,\Sigma)$.   We denote by $\mathring{\mathcal{A}}^\wedge({P},S,\Sigma)$   \index[notation]{A^\wedge({P},S,\Sigma)@$\mathring{\mathcal{A}}^\wedge({P},S,\Sigma)$} its completion with respect to this degree. The class of $\underline{D}\in \mathring{\mathrm{Jac}}(P,S,\Sigma)$ in $\mathring{\mathcal{A}}({P},S,\Sigma)\subset \mathring{\mathcal{A}}^\wedge({P},S,\Sigma)$ is denoted by $[\underline{D}]$.

Assume that $P=\vec{\varnothing}$ and $S=\emptyset$. Define a map $\mathrm{deg}_{\mathrm{cyc}} : \mathrm{Jac}(\vec{\varnothing},\emptyset,\Sigma) \to \mathbb{Z}_{\geq 0}$ by 
\begin{equation}\label{degcircles}
\mathrm{deg}_{\mathrm{cyc}}(\underline{D}):=\frac{1}{2}\big((\text{number of trivalent vertices of } D) -|\partial D|\big)+|\pi_0(D)|.
\end{equation}
Then $(|\pi_0|, \mathrm{deg}_\mathrm{cyc}, \mathrm{deg})$ defines a tridegree on $\mathbb{C}\mathring{\mathrm{Jac}}(\vec{\varnothing},\emptyset, \Sigma)$. The relations AS and IHX relate diagrams with identical triples 
$(\# \text{univalent vertices},\#  \text{trivalent vertices}, |\pi_0|)$, hence $(|\pi_0|, \mathrm{deg}_\mathrm{cyc}, \mathrm{deg})$ defines a tridegree on $\mathring{\mathcal A}(\vec{\varnothing},\emptyset,\Sigma)$. Then the completion $\mathring{\mathcal A}^\wedge(\vec{\varnothing},\emptyset,\Sigma)$ of  $\mathring{\mathcal A}(\vec{\varnothing},\emptyset,\Sigma)$ with respect to $\mathrm{deg}$ is equipped with the bidegree $(|\pi_0|,\mathrm{deg}_{\mathrm{cyc}})$. For $c,d \in \mathbb{Z}_{\geq 0}$, denote by $\mathring{\mathcal{A}}^{\wedge}(\vec{\varnothing},\emptyset,\Sigma)[c,d]$ the part of $\mathring{\mathcal{A}}^{\wedge}(\vec{\varnothing},\emptyset,\Sigma)$ of bidegree $(c,d)$.

Let  $P$ be an oriented Brauer diagram and  $S$, $\Sigma$, $\Sigma'$  be finite sets. One checks that any map $f:\Sigma \to \Sigma'$ induces  maps $f_*:\mathring{\mathrm{Jac}}(P,S,\Sigma) \to \mathring{\mathrm{Jac}}(P,S,\Sigma')$ and $\mathring{\mathcal A}(P,S,\Sigma)\to \mathring{\mathcal A}(P,S,\Sigma')$. Moreover, if $\Sigma''$ is other finite set and $g:\Sigma' \to \Sigma'' $ is a map,  one has $(g \circ f)_*=g_* \circ f_*$.

The pairing \eqref{thepairing} induces a well-defined continuous bilinear map
\begin{equation}\label{thepairinginA}
\langle \ ,\  \rangle : \mathring{\mathcal{A}}^\wedge({P},S,\Sigma)\times \mathring{\mathcal{A}}^\wedge({P}',S',\Sigma) \longrightarrow \mathring{\mathcal{A}}^\wedge({P}\otimes{P}', S\sqcup S').
\end{equation}

\begin{remark} The map \eqref{thepairinginA} is a particular case of the map introduced in~\cite[Definition 2.9]{BNGRTI}.
\end{remark}

The tensor map \eqref{tensor:map:A:PP'Jacgen} induces a well-defined  continuous bilinear map
\begin{equation}\label{tensor:map:A:PP'gen}
\otimes_{\mathcal A} : \mathring{\mathcal{A}}^\wedge({P}, S,\Sigma)\times\mathring{\mathcal{A}}^\wedge({P}', S', \Sigma')\longrightarrow\mathring{\mathcal{A}}^\wedge({P}\otimes {P}', S\sqcup S', \Sigma\sqcup\Sigma').
\end{equation}
Similarly, for any composable pair $(P,P')$ of oriented Brauer diagrams, one  defines a composition map 
\begin{equation}\label{COMP:MAP:SIGMA}
\circ_{\mathcal A} : \mathring{\mathcal{A}}^\wedge({P}, S,\Sigma)\times\mathring{\mathcal{A}}^\wedge({P}', S', \Sigma')\longrightarrow\mathring{\mathcal{A}}^\wedge({P'}\circ {P}, S\sqcup S'\sqcup\mathrm{Cir}(P',P), \Sigma\sqcup\Sigma').
\end{equation}
The above map coincides with the map \eqref{map:A:PP'2} when $\Sigma=\Sigma'=\emptyset$.

\section{Construction of two families of invariants \texorpdfstring{$\{\Omega^{\mathfrak{c}}_n\}_{n\geq 1}$}{O'n} and \texorpdfstring{$\{\Omega^{\mathfrak{d}}_n\}_{n\geq 1}$}{On} of 3-manifolds}\label{ch4}

In \S\ref{sec:4-8},  we  constructed a semi-Kirby structure $(\mathfrak{s}(\mathbf{A}),\mu)$ where the algebra $\mathfrak{s}(\mathbf{A})$ is a quotient of $\bigoplus_{k \geq 0}\mathring{\mathcal A}^\wedge_k$, where $\mathring{\mathcal A}^\wedge_k=\mathring{\mathcal A}^\wedge(\vec{\varnothing},[\![1,k]\!])_{\mathfrak{S}_k}$. In~\S\ref{sec:3:1}, we review the definition of an algebra automorphism $\mathrm{cs}^\nu$ of $\bigoplus_{k \geq 0}\mathring{\mathcal A}^{\wedge}_k$. In~\S\ref{sec:3:2}, we construct an algebra $\mathfrak{a}$ and an algebra isomorphism  $\overline{\mathrm{cs}}^\nu: \mathfrak{s}(\mathbf{A})\to \mathfrak{a}$. Fix an integer $n>0$. In~\S\ref{sec:3:3}, we construct an algebra $\mathfrak{b}(n)$ and an algebra morphism $\varphi_n : \mathfrak{a} \to \mathfrak{b}(n)$. In \S\ref{sec:tracesandLMOtrace} we define a \emph{space of diagrammatic traces} and the \emph{LMO diagrammatic trace}.  In~\S\ref{sec:3:4}   we construct  an algebra $\mathfrak{c}(n)$ and use the LMO diagrammatic trace to define an algebra morphism $\overline{j}_n:\mathfrak{b}(n)\to \mathfrak{c}(n)$. In~\S\ref{sec:3:5-2022-04-18}, we define algebras $\mathfrak{d}(n)$ and algebra morphisms $\mathrm{pr}_n:\mathfrak{c}(n)\to \mathfrak{d}(n)$ which give rise to the push-forward semi-Kirby structures  $(\mathfrak{b}(n), \varphi_n \circ \overline{\mathrm{cs}}^{\nu}\circ {Z}_{{\mathcal A}})$, 
$(\mathfrak{c}(n),\overline{j}_n \circ \varphi_n \circ \overline{\mathrm{cs}}^{\nu}\circ {Z}_{{\mathcal A}})$ and $(\mathfrak{d}(n),\mathrm{pr}_n \circ \overline{j}_n \circ \varphi_n \circ \overline{\mathrm{cs}}^{\nu}\circ {Z}_{{\mathcal A}})$. We then prove that the semi-Kirby structures $(\mathfrak{c}(n),\overline{j}_n \circ \varphi_n \circ \overline{\mathrm{cs}}^{\nu}\circ {Z}_{{\mathcal A}})$ and $(\mathfrak{d}(n),\mathrm{pr}_n \circ \overline{j}_n \circ \varphi_n \circ \overline{\mathrm{cs}}^{\nu}\circ {Z}_{{\mathcal A}})$  are Kirby structures for any $n\geq 1$ (Theorem~\ref{mainr-2022-02-14} and Proposition~\ref{r2022-04-12-dn}); moreover in \S\ref{sec:5:7}, we show the equality   $\mathfrak{c}(1)=\mathfrak{d}(1)$\footnote{In \cite{LMMO99} it is shown, in our terminology, that $(\mathfrak{b}(1), \varphi_1 \circ \overline{\mathrm{cs}}^\nu\circ Z_{\mathcal A})$ is a Kirby structure; it seems to follow from \cite{LMO98} that there is an isomorphism of Kirby structures  $\mathfrak{b}(1)\simeq \mathfrak{e}(1)$, where $\mathfrak{e}(1)$ is as in Definition~\ref{def:definitionofen}, which given Lemma~\ref{r2022-04-12-ohtlemma}, implies the isomorphism  $\mathfrak{b}(1)\simeq\mathfrak{d}(1)$.}.

The relations between the involved algebras can be descried by the following overall diagram of algebra morphisms, where $(\mathrm{KII}^s \mathring{\mathcal{A}}^{\wedge})$, $P^{n+1}$ and $L^{<2n}$  are suitable ideals of $\bigoplus_{k \geq 0}\mathring{\mathcal{A}}^{\wedge}_k$, and $P^{n+1}_0$, $O^n$ and $(\mathrm{deg}>n)$ ideals of $\mathring{\mathcal{A}}^{\wedge}_0$; and  $\mathrm{cs}^\nu$ and $j_n$ are suitable algebra morphisms.

$$
\xymatrixcolsep{0.95pc}\xymatrix{
\mathbb{Z}\underline{\vec{\mathcal T}}(\emptyset,\emptyset)\ar^-{{Z}_{{\mathcal A}}}[r]\ar[d]& \bigoplus_{k\geq0} \mathring{\mathcal{A}}^\wedge_k \ar^-{\mathrm{cs}^\nu}[r]\ar[d] & \bigoplus_{k\geq0} \mathring{\mathcal{A}}^\wedge_k\ar^-{\mathrm{Id}}[r] \ar[d]& \bigoplus_{k\geq0} \mathring{\mathcal{A}}^\wedge_k\ar^-{j_n}[r] \ar[d]& \mathring{\mathcal{A}}^\wedge_0\ar[d]\ar^-{\mathrm{Id}}[r] & \mathring{\mathcal{A}}^\wedge_0\ar[d]\\
\frac{\mathbb{Z}\underline{\vec{\mathcal T}}(\emptyset,\emptyset)}{(\mathrm{KII}\vec{\mathcal T}\cup\mathrm{CO}\vec{\mathcal T})}\ar^-{{Z}_{{\mathcal A}}}[r]\ar@{=}[d] & \frac{\bigoplus_{k\geq 0} \mathring{\mathcal{A}}^\wedge_k}{(\mathrm{KII}\mathring{\mathcal{A}}^{\wedge})+(\mathrm{CO}\mathring{\mathcal{A}}^{\wedge})}\ar^-{\overline{\mathrm{cs}}^\nu}[r] \ar@{=}[d]& \frac{\bigoplus_{k\geq 0} \mathring{\mathcal{A}}^\wedge_k}{(\mathrm{KII}^s\mathring{\mathcal{A}}^{\wedge})+(\mathrm{CO}\mathring{\mathcal{A}}^{\wedge})} \ar^-{\varphi_n}[r]\ar@{=}[d] & \frac{\bigoplus_{k\geq 0} \mathring{\mathcal{A}}_k}{P^{n+1}+ L^{<2n}+ (\mathrm{CO}\mathring{\mathcal{A}}^{\wedge})}  \ar^-{\overline{j}_n}[r]\ar@{=}[d] & \frac{\mathring{\mathcal{A}}^\wedge_0}{P^{n+1}_0+ O^n}\ar@{=}[d]\ar^-{\mathrm{pr}_n}[r]  & \frac{\mathring{\mathcal{A}}^\wedge_0}{P^{n+1}_0 + O^n +(\mathrm{deg}>n)} \ar@{=}[d]\\
 \mathfrak{Kir} \ar^-{{Z}_{{\mathcal A}}}[r]& \mathfrak{s}(\mathbf{A})\ar^-{\overline{\mathrm{cs}}^\nu}[r] &  \mathfrak{a} \ar^-{\varphi_n}[r]& \mathfrak{b}(n)\ar^-{\overline{j}_n}[r] &  \mathfrak{c}(n)\ar^-{\mathrm{pr}_n}[r] & \mathfrak{d}(n).
} 
$$

In~\S\ref{sec:3:6-2022-04-18} we derive from the Kirby structures $(\mathfrak{c}(n),\overline{j}_n \circ \varphi_n \circ \overline{\mathrm{cs}}^{\nu}\circ {Z}_{{\mathcal A}})$  and $(\mathfrak{d}(n),\mathrm{pr}_n \circ \overline{j}_n \circ \varphi_n \circ \overline{\mathrm{cs}}^{\nu}\circ {Z}_{{\mathcal A}})$  the $3$-manifold invariants ${\Omega}^{\mathfrak{c}}_n$ and ${\Omega}^{\mathfrak{d}}_n$ with values in $\mathfrak{c}(n)$ and $\mathfrak{d}(n)$ and related by $\Omega^{\mathfrak{d}}_n(Y) = \mathrm{pr}_n(\Omega^{\mathfrak{c}}_n(Y))$ for any 3-manifold $Y$  (Theorem~\ref{r2022-04-12theinvariants}). In \S\ref{sec::equalityomega1comega1d} we deduce the equality ${\Omega}^{\mathfrak{c}}_1 = {\Omega}^{\mathfrak{d}}_1$. The construction of $\Omega^{\mathfrak{d}}_n$ was first obtained in \cite{LMO98}. In~\S\ref{sec:3:7}, we show the universality of the morphism $\overline{j}_n:\mathfrak{b}(n)\to \mathfrak{c}(n)$ in the class of surjective graded algebra morphisms $\mathfrak{b}(n) \to A$ giving rise to Kirby structures.

\subsection{The automorphism \texorpdfstring{$\mathrm{cs}^{\nu}$}{cs-nu} of the algebra \texorpdfstring{$\bigoplus_{k \geq 0}\mathring{\mathcal A}^{\wedge}_k$}{A}}\label{sec:3:1}

Recall that $\bigoplus_{k \geq 0}\mathring{\mathcal A}^{\wedge}_k$ is a commutative algebra equipped with the product induced by the  tensor product (Lemma~\ref{lemma:1207}). One checks that this product  coincides with the one induced by composition. In this subsection, we associate to each $\alpha\in\mathring{\mathcal{A}}^{\wedge}(\downarrow,\emptyset)$ an endomorphism $\mathrm{cs}^{\alpha}$ of this algebra, which is an automorphism if $\alpha$ is invertible.

Let us start by recalling some notations.  Let $k \in\mathbb{Z}_{\geq 0}$. Let $P \in \vec{\mathcal Br}$ be an oriented Brauer diagram, from Definition~\ref{def:actionPermSonXS},     $\mathcal X(P)_k$ denotes the coinvariant space $\mathcal X(P,[\![1,k]\!])_{\mathfrak S_k}$. From Notation~\ref{notationarrows},  $\downarrow^k:=\downarrow\cdots\downarrow =\mathrm{Id}_{+^{\otimes k}} \in \underline{\vec{\mathcal Br}}(+^{\otimes k},+^{\otimes k})$ and $\uparrow^k:=\uparrow\cdots\uparrow = \mathrm{Id}_{-^{\otimes k}} \in \underline{\vec{\mathcal Br}}(-^{\otimes k},-^{\otimes k})$. 

Define the oriented Brauer diagrams $\mathrm{cup}^k \in \underline{\vec{\mathcal Br}}(\emptyset,-^{\otimes k} \otimes +^{\otimes k})$  \index[notation]{cup^k@$\mathrm{cup}^k$}  and  $\mathrm{cap}^k \in \underline{\vec{\mathcal Br}}(-^{\otimes k} \otimes +^{\otimes k},\emptyset)$ \index[notation]{cap^k@$\mathrm{cap}^k$} by 
$$\mathrm{cup}^k:=((0,2k,\sigma_{\mathrm{cup}^k}), B:=[\![k+1,2k]\!]) \quad \text{and} \quad \mathrm{cap}^k:=((2k,0,\sigma_{\mathrm{cap}^k}), B':=[\![1,k]\!])$$
where $\sigma_{\mathrm{cup}^k}$ and $\sigma_{\mathrm{cap}^k}$ are the $\mathrm{FPFI}$ of $[\![1,2k]\!]$ defined by $\sigma_{\mathrm{cup}^k}(i)=\sigma_{\mathrm{cap}^k}(i) = 2k-i+1$ for every $i\in [\![1,k]\!]$. One easily checks that $|\mathrm{Cir}(\mathrm{cup}^k,\mathrm{cap}^k)|=k$.

Denote by  $\mathrm{cup}^k_{\mathcal A}$ \index[notation]{cup^k_{\mathcal A}@$\mathrm{cup}^k_{\mathcal A}$}  (resp. $\mathrm{cap}^k_{\mathcal A}$) \index[notation]{cap^k_{\mathcal A}@$\mathrm{cap}^k_{\mathcal A}$} the element  in $\mathring{\mathcal A}^{\wedge}(\mathrm{cup}^k,\emptyset)$ (resp. $\mathring{\mathcal A}^{\wedge}(\mathrm{cap}^k,\emptyset)$) corresponding to  $\mathrm{cup}^k $ (resp. $\mathrm{cap}^k)$, see Definition~\ref{def:emptyJacdiagram}.  Let  $\mathrm{Id}^{\mathcal A}_{(-)^{\otimes k}} \in \mathring{\mathcal A}^{\wedge}(\uparrow^k, \emptyset)$ as in Lemma~\ref{r:lemmaforLMO1}$(b)$.

\begin{definition}\label{def:closuremap} Define the linear continous map 
$$\mathrm{clos}_k : \mathring{\mathcal A}^{\wedge}(\downarrow^k \otimes P,\emptyset)\longrightarrow  \mathring{\mathcal A}^{\wedge}(P,\mathrm{Cir}(\mathrm{cup}^k,\mathrm{cap}^k))$$ \index[notation]{clos_k@$\mathrm{clos}_k$}
 by 
$\mathrm{clos}_k(x) := \mathrm{cap}^k_{\mathcal A} \circ (\mathrm{Id}^{\mathcal A}_{-^{\otimes k}} \otimes x) \circ \mathrm{cup}^k_{\mathcal A}$  for any $x\in  \mathring{\mathcal A}^{\wedge}(\downarrow^k \otimes P,\emptyset)$.
\end{definition}

\begin{lemma}\label{themapcsalpha}   Let $k \in\mathbb{Z}_{\geq 0}$ and let $P \in \vec{\mathcal Br}$ be an oriented Brauer diagram.
\begin{itemize}
\item[$(a)$] For any $\alpha \in \mathring{\mathcal A}^{\wedge}(\downarrow,\emptyset)$ there exists a  unique linear continuous endomorphism $$\tilde{\mathrm{cs}}^{\alpha}_{P,k} : \mathring{\mathcal A}^{\wedge}(P,\mathrm{Cir}(\mathrm{cup}^k,\mathrm{cap}^k)) \longrightarrow \mathring{\mathcal A}^{\wedge}(P,\mathrm{Cir}(\mathrm{cup}^k,\mathrm{cap}^k))$$  \index[notation]{cs^{\alpha}_{P,k}tilde@$\tilde{\mathrm{cs}}^{\alpha}_{P,k}$}
such that the diagram
\begin{equation*}
\xymatrix{\mathring{\mathcal{A}}^\wedge(\downarrow^k \otimes P, \emptyset )\ar_{ (\alpha^{\otimes k}\circ \bullet) \otimes \mathrm{Id}_{\mathrm{t}(P)}}[d]
\ar^{\mathrm{clos}_k}[rrr] & &  & \mathring{\mathcal{A}}^\wedge(P, \mathrm{Cir}(\mathrm{cup}^k,\mathrm{cap}^k)) \ar^{\tilde{\mathrm{cs}}^{\alpha}_{P,k}}[d]\\
\mathring{\mathcal{A}}^\wedge(\downarrow^k \otimes P,\emptyset)\ar^{\mathrm{clos}_k}[rrr]  & & &\mathring{\mathcal{A}}^\wedge(P, \mathrm{Cir}(\mathrm{cup}^k,\mathrm{cap}^k))}
\end{equation*}
is commutative (where $\mathrm{t}(P)$ denotes the target of $P$ in the category $\underline{\vec{\mathcal Br}}$).

\item[$(b)$] Let $\alpha \in \mathring{\mathcal A}^{\wedge}(\downarrow, \emptyset)$. For $S$ a finite set, the endomorphism of $\mathring{\mathcal{A}}^\wedge(P, S)$ defined by $\mathring{\mathcal A}^\wedge(P,\phi) \circ  \tilde{\mathrm{cs}}^{\alpha}_{P,|S|} \circ \mathring{\mathcal A}^\wedge(P,\phi)^{-1}$, is independent of the choice of a bijection $\phi  : \mathrm{Cir}(\mathrm{cup}^{|S|},\mathrm{cap}^{|S|}) \to S$, where $\mathring{\mathcal A}^\wedge(P,\bullet)$ is the functor from Section~\ref{sec:4-1}. This endomorphism will be denoted by $\tilde{\mathrm{cs}}^{\alpha}_{P,S}$. When $S=[\![1,k]\!]$, it induces a linear endomorphism 
\begin{equation} \mathrm{cs}^{\alpha}_{P,k}: \mathring{\mathcal A}^{\wedge}(P)_k \longrightarrow \mathring{\mathcal A}^{\wedge}(P)_k
\end{equation}   \index[notation]{cs^{\alpha}_{P,k}@$\mathrm{cs}^{\alpha}_{P,k}$}
called \emph{$\alpha$-connected sum map}.

\item[$(c)$]   Let $\mathrm{cs}^{\alpha}$   \index[notation]{cs^{\alpha}@$\mathrm{cs}^{\alpha}$} be the endomorphism of $\bigoplus_{k \geq 0}\mathring{\mathcal A}_k^{\wedge}$ given by $\bigoplus_{k\geq 0}\mathrm{cs}^{\alpha}_{\vec{\varnothing},k}$ then  $\mathrm{cs}^{\alpha}$ is an algebra endomorphism.  

\item[$(d)$]  If $\alpha,\beta \in \mathring{\mathcal A}^{\wedge}(\downarrow,\emptyset)$ then $\mathrm{cs}^{\beta \circ \alpha} = \mathrm{cs}^{\beta} \circ \mathrm{cs}^{\alpha}$.

\item[$(e)$] Moreover, if $\alpha \in \mathring{\mathcal A}^{\wedge}(\downarrow,\emptyset)$ is invertible, then  $\mathrm{cs}^{\alpha}$ is an algebra automorphism.

\end{itemize}
\end{lemma}

\begin{proof} $(a)$ is a formulation of well-known results, see for instance~\cite[Thm. 16.20]{JacksonMoffat}. These results imply that for fixed $P \in \vec{\mathcal{B}r}$ and $\alpha \in \mathring{\mathcal A}^{\wedge}(\downarrow, \emptyset)$, the map $\tilde {\mathrm{cs}}^{\alpha}_{P,k}$ is equivariant under the action of $\mathfrak{S}_{\mathrm{Cir}(\mathrm{cup}^k, \mathrm{cap}^k)}$, which implies the independence stated in $(b)$. The association $S\mapsto \tilde{\mathrm{cs}}^{\alpha}_{P,S}$ for any finite set $S$, is then a natural equivalence of the functor $S\mapsto \mathring{\mathcal A}^{\wedge}(P,S)$ with itself, which implies that for $S=[\![1,k]\!]$ the map $\tilde{\mathrm{cs}}^{\alpha}_{P,[\![1,k]\!]}$ is $\mathfrak{S}_k$-equivariant. The endomorphism $\tilde{\mathrm{cs}}^{\alpha}_{P,k}$ is then defined by applying the functor of coinvariants, this finishes the proof of $(b)$. Statement  $(c)$ is a consequence of the following more general claim: let $P$ and $P'$ be oriented Brauer diagrams,  if $x \in \mathring{\mathcal A}^\wedge(P)_k$ and $x'\in \mathring{\mathcal A}^\wedge(P')_k$ then $\mathrm{cs}^{\alpha}_{P \otimes P',k+k'}(x\otimes x')=\mathrm{cs}^{\alpha}_{P,k}(x) \otimes \mathrm{cs}^{\alpha}_{P',k'}(x')$ which can be shown from the definitions of the maps and the tensor product. $(d)$ follows from the fact that the map taking $\alpha$ to the endomorphism of $\mathring{\mathcal{A}}(\downarrow^k \otimes P,\emptyset)$ in the commutative diagram in $(a)$ is an algebra morphism. $(e)$ follows from $(d)$.

\end{proof}

\begin{remark} One can prove that the ring $( \mathring{\mathcal A}^{\wedge}(\downarrow,\emptyset), \circ)$ is commutative, see for instance~\cite[Proposition~6.1]{Ohts}.
\end{remark}

 Intuitively, the operation $\mathrm{cs}^{\alpha}_{P,k}$ corresponds to inserting the element $\alpha$ (or $z(\alpha)$) to each circle component of the schematic representation of  $\underline{D}$, see Figure~\ref{figura_alpha_sum} for a schematic representation.

\begin{figure}[ht!]
										\centering
                        \includegraphics[scale=0.8]{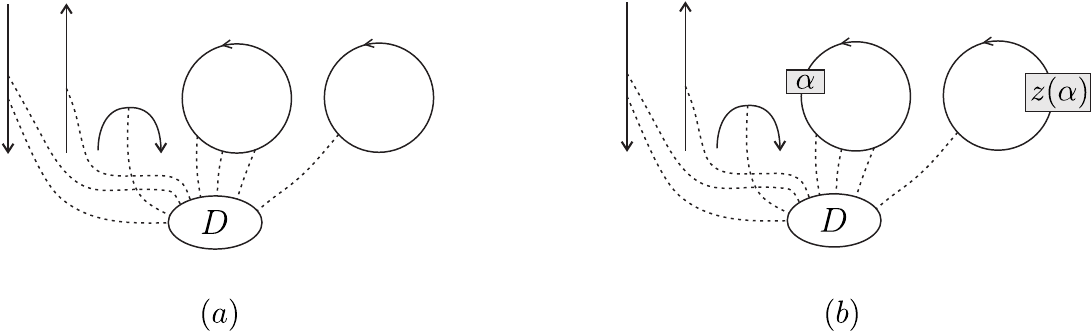}
												\caption{$(a)$ schematic representation of a Jacobi diagram $\underline D$ on $(P,[\![1,2]\!])$ and $(b)$ its image under the map $\mathrm{cs}^{\alpha}_{P,2}$.}
\label{figura_alpha_sum} 										
\end{figure}

We specialize the map in Lemma~\ref{themapcsalpha}$(c)$ to the invertible element $\nu$ defined in subsection~\ref{sec:4-8}, to obtain an algebra automorphism:
\begin{equation}\label{cs-nu-total}
\mathrm{cs}^{\nu}:\bigoplus_{k\geq 0}\mathring{\mathcal{A}}^\wedge_k\longrightarrow \bigoplus_{k\geq 0}\mathring{\mathcal{A}}^\wedge_k.
\end{equation}
  \index[notation]{cs^{\nu}@$\mathrm{cs}^{\nu}$}

\begin{remark}  In~\cite{LMO98}, a normalization $\check{Z}$   \index[notation]{Z@$\check{Z}$} of the Kontsevich integral ${Z}_{{\mathcal A}}$ restricted to links is defined as the composed map $\mathrm{cs}^{\nu} \circ {Z}_{{\mathcal A}}$. 
\end{remark}

\subsection{The algebras \texorpdfstring{$\mathfrak{s}(\mathbf{A})$}{SA}, \texorpdfstring{$\mathfrak{a}$}{m} and the  algebra isomorphism \texorpdfstring{$\overline{\mathrm{cs}}^{\nu}: \mathfrak{s}(\mathbf{A})\to \mathfrak{a}$}{cs-nu-m-to-sA}}\label{sec:3:2}

Let $\mathbf{A}$ be the pre-LMO structure from Theorem~\ref{thm:prelmokont}. By Proposition~\ref{l:LMOinvt}, this pre-LMO structure gives rise to a semi-Kirby structure $(\mathfrak{s}(\mathbf{A}),\mu)$. In \S\ref{sec:3:2:1} we recall the definition of the algebra $\mathfrak{s}(\mathbf{A})$   as the quotient of $\bigoplus_{k \geq 0}\mathring{\mathcal{A}}^\wedge_k$ by the sum of the ideals  $(\mathrm{KII}\mathring{\mathcal{A}}^\wedge)$ and~$(\mathrm{CO}\mathring{\mathcal{A}}^\wedge)$. In \S\ref{sec:3:2:2} we define the ideal $(\mathrm{KII}^s\mathring{\mathcal{A}}^\wedge)$ of $\bigoplus_{k \geq 0}\mathring{\mathcal{A}}^\wedge_k$ and the algebra $\mathfrak{a}$  as the quotient  of~$\bigoplus_{k \geq 0}\mathring{\mathcal{A}}^\wedge_k$ by the sum of the ideals  $(\mathrm{KII}^s\mathring{\mathcal{A}}^\wedge)$ and $(\mathrm{CO}\mathring{\mathcal{A}}^\wedge)$. In \S\ref{sec:3:2:3} and \S\ref{sec:3:2:4} we show that $\mathrm{cs}^\nu((\mathrm{KII}\mathring{\mathcal{A}}^\wedge))=(\mathrm{KII}^s\mathring{\mathcal{A}}^\wedge)$  and $\mathrm{cs}^\nu((\mathrm{CO}\mathring{\mathcal{A}}^\wedge))=(\mathrm{CO}\mathring{\mathcal{A}}^\wedge)$, respectively. This enables us to construct the isomorphism of algebras $\overline{\mathrm{cs}}^{\nu}: \mathrm{s}(\mathbf{A}) \to \mathfrak{a}$ in \S\ref{sec:3:2:4}.

\subsubsection{The algebra $\mathfrak{s}(\mathbf{A})$}\label{sec:3:2:1}

By specializing Definition~\ref{def:spaceKIIabs} to the pre-LMO structure $\mathbf{A}$, we have  $(\mathrm{KII}\mathring{\mathcal A}^{\wedge}) =\bigoplus_{k\geq 2}(\mathrm{KII}\mathring{\mathcal A}^{\wedge}_k)$   \index[notation]{KII\mathring{\mathcal A}^{\wedge}@$(\mathrm{KII}\mathring{\mathcal A}^{\wedge})$} where $(\mathrm{KII}\mathring{\mathcal A}^{\wedge}_k) \subset \mathring{\mathcal A}^{\wedge}_k$ is given by
\begin{equation}\label{spaceKIIforA}
(\mathrm{KII}\mathring{\mathcal A}^{\wedge}_k) =\mathrm{Im}\Big(\mathrm{KIIMap}_{{\mathring{\mathcal A}}^{\wedge},k}^2 - \mathrm{KIIMap}_{{\mathring{\mathcal  A}}^{\wedge},k}^1\Big)
\end{equation} \index[notation]{KII\mathring{\mathcal A}^{\wedge}_k@$(\mathrm{KII}\mathring{\mathcal A}^{\wedge}_k)$}
where the maps $\mathrm{KIIMap}_{{\mathring{\mathcal A}}^{\wedge},k}^1$ \index[notation]{KIIMap_{{\mathring{\mathcal A}}^{\wedge},k}^1@$\mathrm{KIIMap}_{{\mathring{\mathcal A}}^{\wedge},k}^1$} and $\mathrm{KIIMap}_{{\mathring{\mathcal A}}^{\wedge},k}^2$ \index[notation]{KIIMap_{{\mathring{\mathcal A}}^{\wedge},k}^2@$\mathrm{KIIMap}_{{\mathring{\mathcal A}}^{\wedge},k}^2$} are the specializations of \eqref{KIIMap1forX} and \eqref{KIIMap2forX} to  $\mathbf{A}$. In Figure~\ref{figureEx1_52_abstract} we give a schematic representation of these maps using Conventions~\ref{convention:red-box} and~\ref{convention-doublingblueboxes}.


\begin{figure}[ht!]
			\centering
            \includegraphics[scale=0.7]{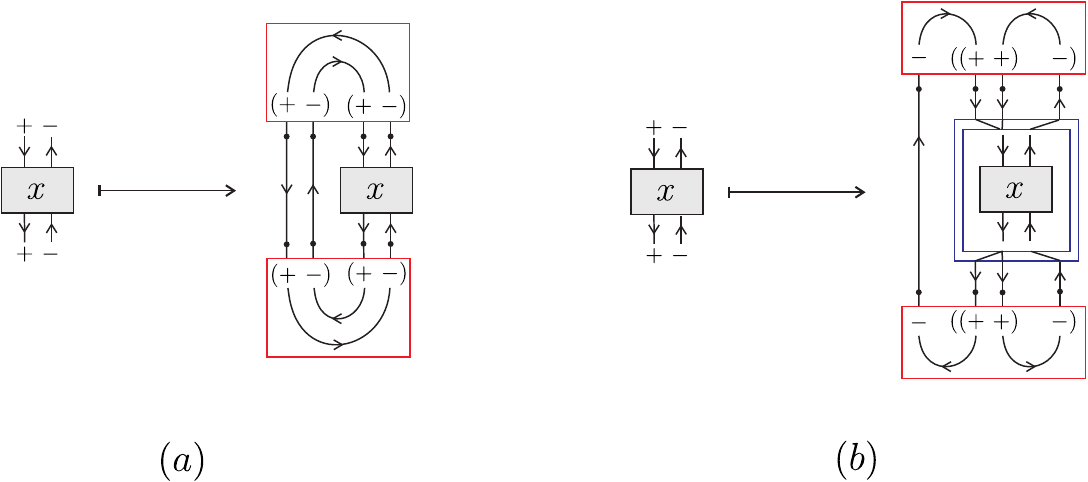}					\caption{$(a)$ The map $\mathrm{KIIMap}^1_{{\mathring{\mathcal A}}^\wedge,k}$, $(b)$ the map $\mathrm{KIIMap}^2_{{\mathring{\mathcal A}}^\wedge,k}$. We represent $x\in{\mathring{\mathcal{A}}}^{\wedge}(\downarrow\uparrow)_{k-2}$ by the grey-colored box. Notice that there are ${k-2}$ circles  in the schematic representation  of $x$ which are not shown in the figure.}
\label{figureEx1_52_abstract} 								
\end{figure}

Similarly, we have $(\mathrm{CO}\mathring{\mathcal{A}}^\wedge)= \bigoplus_{k\geq 1} (\mathrm{CO}\mathring{\mathcal{A}}_k^\wedge) \subset \bigoplus_{k\geq 0}\mathring{\mathcal{A}}^\wedge_k$  \index[notation]{CO\mathring{\mathcal{A}}^\wedge)@$(\mathrm{CO}\mathring{\mathcal{A}}^\wedge)$} where $(\mathrm{CO}\mathring{\mathcal{A}}_k^\wedge) \subset \mathring{\mathcal{A}}^\wedge_k$  \index[notation]{CO\mathring{\mathcal{A}}^\wedge_k@$(\mathrm{CO}\mathring{\mathcal{A}}^\wedge_k)$} is given by
\begin{equation*}
(\mathrm{CO}\mathring{\mathcal{A}}^\wedge_k)=\mathrm{Im}\big(\mathrm{co}^{{\mathcal{A}}}_k  - \mathrm{proj} ^{{{\mathcal{A}}}}_{1,k-1}\big) \subset \mathring{\mathcal{A}}^\wedge_k,
\end{equation*}
\noindent see Definition~\ref{def:projtoXk}. By Proposition~\ref{l:LMOinvt}, $(\mathrm{KII}\mathring{\mathcal{A}}^\wedge)$ and $(\mathrm{CO}\mathring{\mathcal{A}}^\wedge)$ are both ideals of the commutative algebra $\bigoplus_{k \geq 0}\mathring{\mathcal A}^{\wedge}_k$ and its associated  semi-Kirby structure $\mathfrak{s}(\mathbf{A})$ is given by
\begin{equation}\label{modSA}
\mathfrak{s}(\mathbf{A}) = \frac{\bigoplus_{k \geq 0}\mathring{\mathcal A}^{\wedge}_k}{(\mathrm{KII}\mathring{\mathcal{A}}^\wedge) + (\mathrm{CO}\mathring{\mathcal{A}}^\wedge)},
\end{equation}   \index[notation]{s(\mathbf{A})@$\mathfrak{s}(\mathbf{A})$}
together with the algebra morphism $\mu:\mathfrak{Kir}\to \mathfrak{s}(\mathbf{A})$ induced by the Kontsevich integral.

\subsubsection{The algebra $\mathfrak{a}$}\label{sec:3:2:2}

We introduce a variant of the KII maps $\mathrm{KIIMap}_{{\mathring{\mathcal A}}^{\wedge},k}^1$ and $\mathrm{KIIMap}_{{\mathring{\mathcal A}}^{\wedge},k}^2$  which are simpler.  Recall that $\downarrow \uparrow \in \vec{\mathcal Br}$ denotes the identity morphism in $\underline{\vec{\mathcal Br}}(+-,+-)$.

Define  $\mathrm{cup}_{+-+-}$,  \index[notation]{cup_{+-+-}@$\mathrm{cup}_{+-+-}$} $\mathrm{cup}_{-+-+}$,  \index[notation]{cup_{-+-+}@$\mathrm{cup}_{-+-+}$} $\mathrm{cap}_{+-+-}$ \index[notation]{cap_{+-+-}@$\mathrm{cap}_{+-+-}$} and $\mathrm{cap}_{-+-+}$ \index[notation]{cap_{-+-+}@$\mathrm{cap}_{-+-+}$} to be the underlying oriented Brauer diagrams of the tangles  $\mathrm{cup}^{\mathcal T}_{+-+-}$, $\mathrm{cup}^{\mathcal T}_{-+-+}$, $\mathrm{cap}^{\mathcal T}_{+-+-}$ and $\mathrm{cap}^{\mathcal T}_{-+-+}$ from Definition~\ref{gencupcap_1}.  Let $(\mathrm{cup}_{+-+-})_{\mathcal A}$, \index[notation]{cup_{+-+-}_{\mathcal A}@$(\mathrm{cup}_{+-+-})_{\mathcal A}$} $(\mathrm{cup}_{-+-+})_{\mathcal A}$, \index[notation]{cup_{-+-+}_{\mathcal A}@$(\mathrm{cup}_{-+-+})_{\mathcal A}$}   $(\mathrm{cap}_{+-+-})_{\mathcal A}$  \index[notation]{cap_{+-+-}_{\mathcal A}@$(\mathrm{cap}_{+-+-})_{\mathcal A}$}  and  $(\mathrm{cap}_{-+-+})_{\mathcal A}$  \index[notation]{cap_{-+-+}_{\mathcal A}@$(\mathrm{cap}_{-+-+})_{\mathcal A}$}  be the corresponding elements in the respective spaces $\mathring{\mathcal{A}}^\wedge(\mathrm{cup}_{+-+-},\emptyset)$,  $\mathring{\mathcal{A}}^\wedge(\mathrm{cup}_{-+-+},\emptyset)$, $\mathring{\mathcal{A}}^\wedge(\mathrm{cap}_{+-+-},\emptyset)$ and $\mathring{\mathcal{A}}^\wedge(\mathrm{cap}_{-+-+},\emptyset)$, see Definition~\ref{def:emptyJacdiagram}.

\begin{lemma}\label{themapsKIIsimple}  Let $k\geq 2$ be an integer. There are well-defined continuous linear maps
$$
\mathrm{KIIMap}_{{\mathring{\mathcal A}}^\wedge,k}^{s,1}:{\mathring{\mathcal A}}^\wedge(\downarrow\uparrow)_{k-2}\longrightarrow {\mathring{\mathcal A}}^\wedge_k, \quad \quad  \quad \quad \mathrm{KIIMap}_{{\mathring{\mathcal A}}^\wedge,k}^{s,2}:{\mathring{\mathcal A}}^\wedge(\downarrow\uparrow)_{k-2}\longrightarrow {\mathring{\mathcal A}}^\wedge_k
$$ \index[notation]{KIIMap_{{\mathring{\mathcal A}}^\wedge,k}^{s,1}@$\mathrm{KIIMap}_{{\mathring{\mathcal A}}^\wedge,k}^{s,1}$}  \index[notation]{KIIMap_{{\mathring{\mathcal A}}^\wedge,k}^{s,2}@$\mathrm{KIIMap}_{{\mathring{\mathcal A}}^\wedge,k}^{s,2}$} 
given by
\begin{equation}\label{KIIMap1forA-2}
\mathrm{KIIMap}^{s,1}_{{\mathring{\mathcal A}}^\wedge,k}(x):=(\mathrm{cap}_{+-+-})_{{\mathcal{A}}}\circ_{\mathcal A} ((\mathrm{Id}_{+-})_{\mathcal A}\otimes_{\mathcal A} x)\circ_{\mathcal A} (\mathrm{cup}_{+-+-})_{{\mathcal{A}}}
\end{equation}
and
\begin{equation}\label{KIIMap2forA-2}
\mathrm{KIIMap}^{s,2}_{{\mathring{\mathcal A}}^\wedge,k}(x):=(\mathrm{cap}_{-++-})_{{\mathcal{A}}}\circ_{\mathcal A}((\mathrm{Id}_{-})_{\mathcal A}\otimes_{\mathcal A} \mathrm{dbl}_1(x))\circ_{\mathcal A} (\mathrm{cup}_{-++-})_{{\mathcal{A}}},
\end{equation}
for any  $x\in {\mathring{\mathcal A}}^\wedge(\downarrow\uparrow)_{k-2}$. In \eqref{KIIMap2forA-2}, we are using Convention~\ref{convention:co_i-dbl_i}. See Figure~\ref{figureEx1_53_abstract} for a schematic representation of these maps using Conventions~\ref{convention:red-box} and \ref{convention-doublingblueboxes}. The superscript $s$ refers to \emph{simple}.


\end{lemma}

\begin{proof} Let us do the proof for the map $\mathrm{KIIMap}^{s,1}_{{\mathring{\mathcal A}}^\wedge,k}$, the other cases are similar. We have
 $(\mathrm{cup}_{+-+-})_{{\mathcal{A}}} \in \mathring{\mathcal A}^\wedge(\mathrm{cup}_{+-+-})_0$ and  $(\mathrm{Id}_{+-})_{\mathcal A}\otimes_{\mathcal A} x \in \mathring{\mathcal A}^\wedge(\uparrow\downarrow)_{k-2}$  hence  $((\mathrm{Id}_{+-})_{\mathcal A}\otimes_{\mathcal A} x)\circ_{\mathcal A} (\mathrm{cup}_{+-+-})_{{\mathcal{A}}} \in \mathring{\mathcal A}^\wedge(\mathrm{cap}_{+-+-})_{k-2}$. Moreover,  
$(\mathrm{cap}_{+-+-})_{{\mathcal{A}}} \in \mathring{\mathcal A}^\wedge(\mathrm{cap}_{+-+-})_0$ therefore the final triple composition~\eqref{KIIMap1forA-2} belongs to $$\mathring{\mathcal A}^\wedge(\vec{\varnothing})_{k-2+|\mathrm{Cir}(\mathrm{cap}_{+-+-}, \mathrm{cup}_{+-+-})|}=\mathring{\mathcal A}^\wedge_k.$$
\end{proof}

\begin{figure}[ht!]
			\centering
            \includegraphics[scale=0.7]{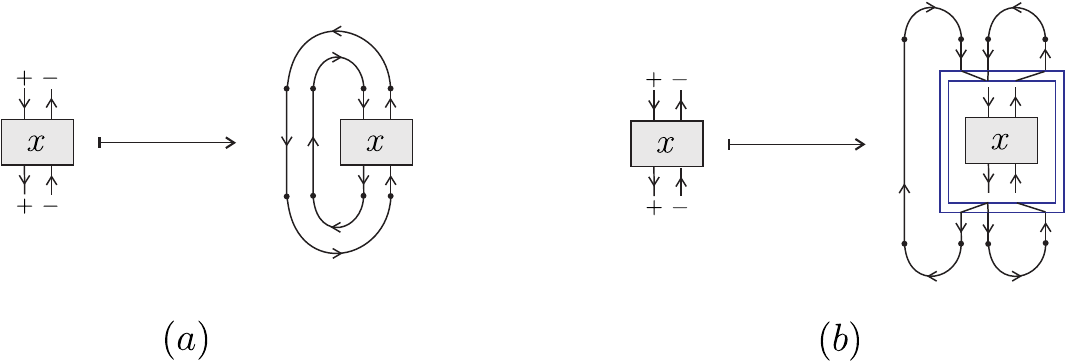}					\caption{$(a)$ The map $\mathrm{KIIMap}^{s,1}_{{\mathring{\mathcal A}}^\wedge,k}$, $(b)$ the map $\mathrm{KIIMap}^{s,2}_{{\mathring{\mathcal A}}^\wedge,k}$. We represent $x\in{\mathring{\mathcal{A}}}^{\wedge}(\downarrow\uparrow)_{k-2}$ by the grey-colored box and we use Convention~\ref{convention-doublingblueboxes}  for representing $\mathrm{dbl}_1(x)\in{\mathring{\mathcal{A}}}^{\wedge}(\downarrow\downarrow\uparrow)_{k-2}$.}
\label{figureEx1_53_abstract} 								
\end{figure}

\begin{definition}\label{spaceKIIsimpleforA}
Let $k\geq 2$, define the subgroups $(\mathrm{KII}^s{\mathring{\mathcal A}}^\wedge_k)$  \index[notation]{KII^s{\mathring{\mathcal A}}^\wedge_k@$(\mathrm{KII}^s{\mathring{\mathcal A}}^\wedge_k)$} of ${\mathring{\mathcal A}}^\wedge_k$ and $(\mathrm{KII}^s{\mathring{\mathcal A}}^\wedge)$ \index[notation]{KII^s{\mathring{\mathcal A}}^\wedge@$(\mathrm{KII}^s{\mathring{\mathcal A}}^\wedge)$} of $\bigoplus_{k\geq 0}\mathring{\mathcal A}^\wedge_k$ by
\begin{equation}\label{spaceKIIsimpleforA-1}
(\mathrm{KII}^s{\mathring{\mathcal A}}^\wedge_k) :=\mathrm{Im}\Big(\mathrm{KIIMap}_{{\mathring{\mathcal A}}^\wedge,k}^{s,2} - \mathrm{KIIMap}_{{\mathring{\mathcal A}}^\wedge,k}^{s,1}\Big) 
\end{equation}
and
\begin{equation}\label{spaceKIIsimpleforA-2}
(\mathrm{KII}^s{\mathring{\mathcal A}}^\wedge) := \bigoplus_{k\geq 0} (\mathrm{KII}^s{\mathring{\mathcal A}}^\wedge_k) \subset \bigoplus_{k\geq 0}\mathring{\mathcal A}^\wedge_k,
\end{equation}
where we set $(\mathrm{KII}^s{\mathring{\mathcal A}}^\wedge_0) = (\mathrm{KII}^s{\mathring{\mathcal A}}^\wedge_1) = \{0\}$.
\end{definition}

\begin{lemma}\label{r:lemma1-20210820} The subgroup $(\mathrm{KII}^s\mathring{\mathcal A}^\wedge)$ is a $\mathbb{Z}_{\geq 0}$-graded ideal of $\bigoplus_{k\geq 0} \mathring{\mathcal{A}}^\wedge_k$.
\end{lemma}

\begin{proof} The reasoning is the same as in the proof of Proposition~\ref{l:LMOinvt}$(a)$.
\end{proof}

Since $(\mathrm{CO}\mathring{\mathcal A}^\wedge)$ is also an ideal of $\bigoplus_{k\geq 0} \mathring{\mathcal{A}}^\wedge_k$, so is $(\mathrm{KII}^s\mathring{\mathcal A}^\wedge) + (\mathrm{CO}\mathring{\mathcal A}^\wedge)$.

\begin{definition} Define the commutative algebra $\mathfrak{a}$ \index[notation]{a@$\mathfrak{a}$} as the quotient algebra
\begin{equation}\label{Z2geq0-module-m}
\mathfrak{a}:=\frac{\bigoplus_{k\geq 0} \mathring{\mathcal{A}}^\wedge_k}{(\mathrm{KII}^s\mathring{\mathcal A}^\wedge) + (\mathrm{CO}\mathring{\mathcal A}^\wedge)} = \bigoplus_{k\geq 0}\frac{\mathring{\mathcal{A}}^\wedge_k}{(\mathrm{KII}^s\mathring{\mathcal A}^\wedge_k) + (\mathrm{CO}\mathring{\mathcal A}^\wedge_k)}.
\end{equation}
\end{definition}

\subsubsection{The linear isomorphism $\mathrm{cs}^{\nu}_{+-}$ }

Recall that for a Jacobi diagram $x$, the element $\mathrm{cs}^{\nu}(x)$ is obtained from $x$ by inserting $\nu$ or $z(\nu)$ in each circle component, see \S\ref{sec:3:1}. In this subsection we introduce a map $\mathrm{cs}^{\nu}_{+-}$  which also insert $\nu$ or $z(\nu)$ in the segment components.

\begin{definition} For any integer $k\geq 2$, define the linear map 
\begin{equation}\label{equ:mapcs-nuopen-def}
\mathrm{cs}^{\nu}_{+-}:\mathring{\mathcal{A}}^\wedge(\downarrow \uparrow)_{k-2}\longrightarrow \mathring{\mathcal{A}}^\wedge(\downarrow \uparrow)_{k-2} 
\end{equation}  \index[notation]{cs^{\nu}_{+-}@$\mathrm{cs}^{\nu}_{+-}$} 
by
\begin{equation}\label{equ:mapcs-nu+-}
\begin{split}
\mathrm{cs}_{+-}^{\nu}(x)= \big(\nu\otimes_{\mathcal A} z(\nu)\otimes_{\mathcal A} {Z}_{\mathcal A}(\mathrm{cap}^{\mathcal T}_{(+-)(+-)})\big)\circ_{\mathcal A} & \big(((\mathrm{cup}_{+-+-})_{{\mathcal A}}\circ_{\mathcal A} (\mathrm{cap}_{+-+-})_{{\mathcal A}})\otimes_{\mathcal A} \mathrm{cs}^{\nu}(x)\big) 
 \\
 & \circ_{\mathcal A} \big(\mathrm{Id}_{+-}\otimes_{\mathcal A} {Z}_{\mathcal A}(\mathrm{cup}^{\mathcal T}_{(+-)(+-)})\big)
\end{split}
\end{equation}
for any $x\in\mathring{\mathcal{A}}(\downarrow\uparrow)_{k-2}$, where $z(\nu)\in\mathring{\mathcal A}^{\wedge}(\uparrow)_0$ denotes the image of $\nu\in\mathring{\mathcal{A}}^\wedge(\downarrow)_0$ under the map $z:{\mathring{\mathcal A}}^\wedge(\downarrow)_0\to {\mathring{\mathcal A}}^\wedge(\uparrow)_0$ from Definition~\ref{themapz}. See Figure~\ref{figureEx1_49_abstract} for a schematic representation.

\begin{figure}[ht!]
			\centering
            \includegraphics[scale=0.7]{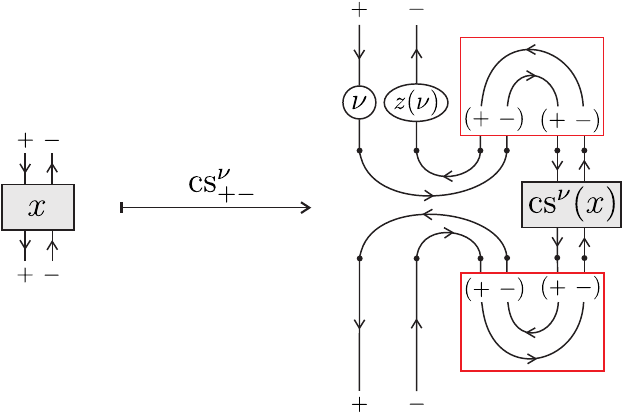}					\caption{Schematic representation of the map $\mathrm{cs}^{\nu}_{+-}:\mathring{\mathcal{A}}(\downarrow\uparrow)_{k-2}\to \mathring{\mathcal{A}}(\downarrow\uparrow)_{k-2}$. We represent $x\in \mathring{\mathcal{A}}(\downarrow\uparrow)_{k-2}$ by the grey-colored box. }
\label{figureEx1_49_abstract} 								
\end{figure}

\end{definition}

\begin{lemma}\label{the-map-cs-+-} Let $k\in\mathbb{Z}_{\geq 2}$. The linear  map
\begin{equation}\label{equ:mapcs-nuopen}
\mathrm{cs}^{\nu}_{+-}:\mathring{\mathcal{A}}^\wedge(\downarrow \uparrow)_{k-2}\longrightarrow \mathring{\mathcal{A}}^\wedge(\downarrow \uparrow)_{k-2} 
\end{equation}
is an isomorphism.
\end{lemma}

\begin{proof}
Define the map 
$$\widetilde{\mathrm{cs}}^{\nu}_{+-}: \mathring{\mathcal{A}}^\wedge(\downarrow \uparrow)_{k-2}\longrightarrow \mathring{\mathcal{A}}^\wedge(\downarrow \uparrow)_{k-2}$$
by
\begin{equation*}\label{equ:mapcs-nu-inverse}
\begin{split}
\widetilde{\mathrm{cs}}_{+-}^{\nu}(x)=\Big(\nu^{-1}\otimes_{\mathcal A} z(\nu)^{-1}\otimes_{\mathcal A} (\mathrm{cap}_{+-+-})_{{\mathcal A}} \Big) \circ_{\mathcal A} &\Big(({Z}_{{\mathcal A}}(\mathrm{cup}^{\mathcal T}_{(+-)(+-)})\circ_{\mathcal A} {Z}_{{\mathcal A}}(\mathrm{cup}^{\mathcal T}_{(+-)(+-)})) \otimes_{\mathcal A} \mathrm{cs}^{\nu^{-1}}(x)\Big)\\
& \qquad \circ_{\mathcal A}\Big(\mathrm{Id}_{+-}\otimes_{\mathcal A}  
(\mathrm{cup}_{+-+-})_{{\mathcal A}}\Big),
\end{split}
\end{equation*}
for $x\in \mathring{\mathcal{A}}^\wedge(\downarrow \uparrow)_{k-2}$. Schematically:
\begin{align*}
\includegraphics[scale=0.7]{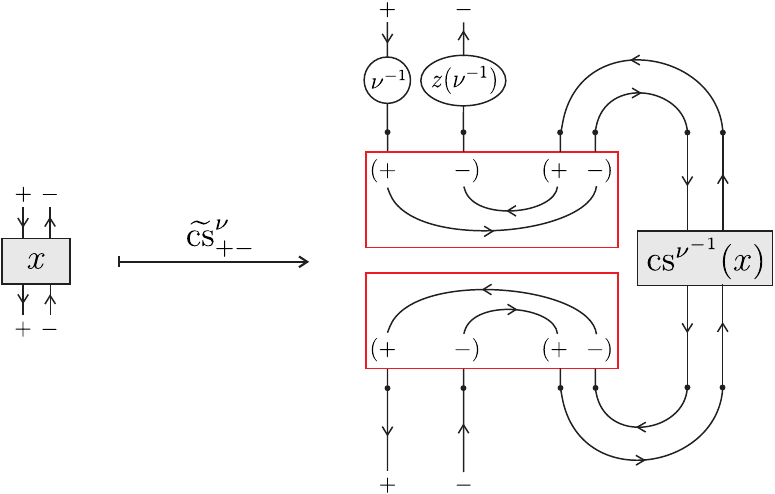}
\end{align*}

Let $x\in{\mathring{\mathcal{A}}}^{\wedge}(\downarrow\uparrow)_{k-2}$, in Figure~\ref{figureEx1_51_abstract} we show a schematic argument for the equality $\widetilde{\mathrm{cs}}_{+-}^{\nu}\circ\mathrm{cs}_{+-}^{\nu}(x) = x$. In the second equality in Figure~\ref{figureEx1_51_abstract} we have used the compatibility of  ${Z}_{\mathcal A}$ with the tensor product and composition and the fact the $\nu\otimes z(\nu)$ is central in the algebra $\mathring{\mathcal{A}}^{\wedge}(\downarrow\uparrow)_{k-2}$, see Lemma~\ref{commut-prop}. In the third equality we have used Lemma~\ref{themapzalgebr} and a combination of Lemmas~\ref{slide-nu} and~\ref{commut-prop}.

\begin{figure}[ht!]
			\centering
            \includegraphics[scale=0.7]{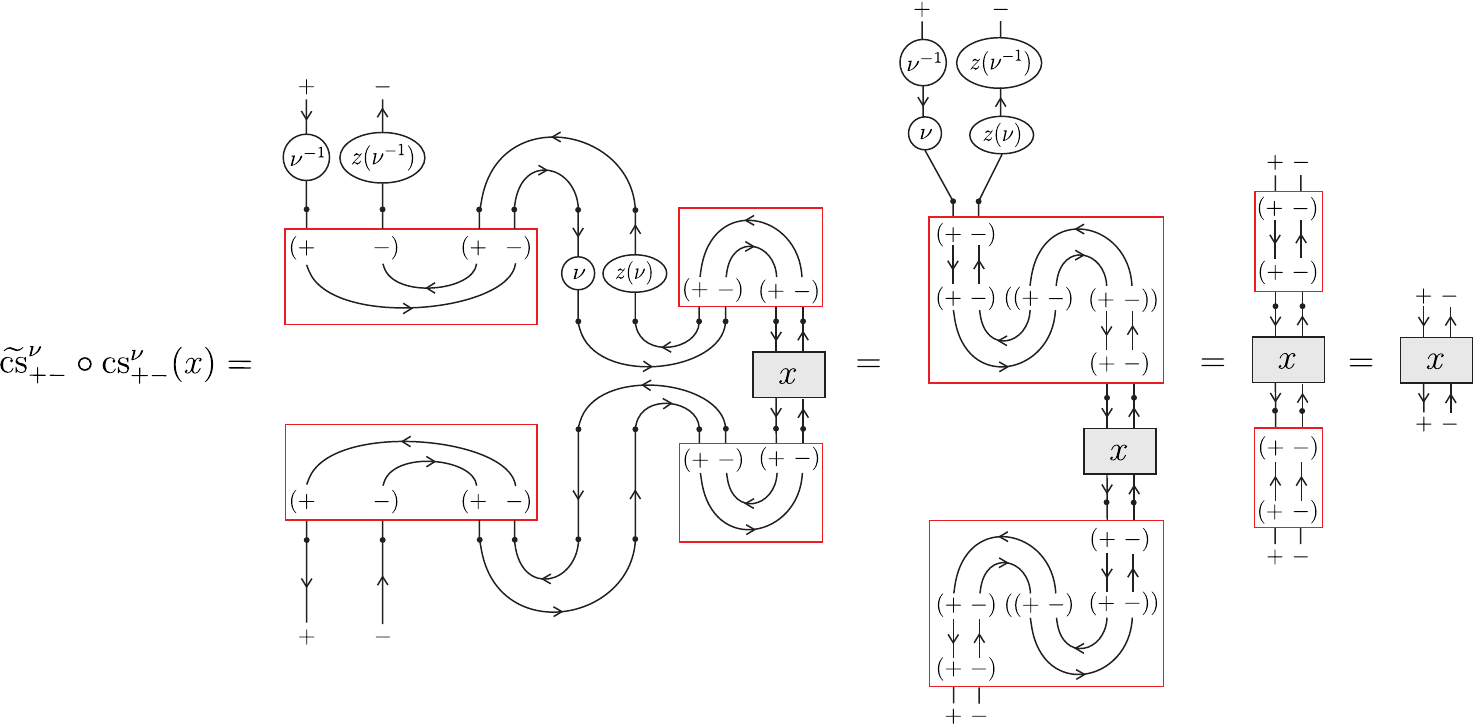}					\caption{The equality $\widetilde{\mathrm{cs}}_{+-}^{\nu}\circ\mathrm{cs}_{+-}^{\nu}(x) = x$. We represent $x\in{\mathring{\mathcal{A}}}^{\wedge}(\downarrow\uparrow)_{k-2}$ by the grey-colored box.}
\label{figureEx1_51_abstract} 								
\end{figure}

Therefore 
$$\widetilde{\mathrm{cs}}^{\nu}_{+-}\circ {\mathrm{cs}}^{\nu}_{+-}= \mathrm{Id}_{\mathring{\mathcal{A}}^\wedge(\downarrow\uparrow)_{k-2}}.$$
Similarly, one checks ${\mathrm{cs}}^{\nu}_{+-} \circ \widetilde{\mathrm{cs}}^{\nu}_{+-} = \mathrm{Id}_{\mathring{\mathcal{A}}^\wedge(\downarrow\uparrow)_{k-2}}$. 
\end{proof}

\subsubsection{The equality $\mathrm{cs}^\nu((\mathrm{KII}\mathring{\mathcal{A}}^\wedge))=(\mathrm{KII}^s\mathring{\mathcal{A}}^\wedge)$}\label{sec:3:2:3}
The automorphism $\mathrm{cs}^\nu$ of $\bigoplus_{k\geq 0} \mathring{\mathcal{A}}^\wedge_k$  given in~\eqref{cs-nu-total} allows us to relate the two ideals $(\mathrm{KII}\mathring{\mathcal{A}}^\wedge)$ and $(\mathrm{KII}^s\mathring{\mathcal{A}}^\wedge)$ given in \eqref{spaceKIIforA} and \eqref{spaceKIIsimpleforA-2}, respectively.

\begin{lemma}\label{20231225-auxiliarlemma}  Let 
$$
A:= \mathring{\mathcal{A}}^\wedge(\mathrm{cap}_{+-+-},\emptyset)\otimes \mathring{\mathcal{A}}^\wedge(\downarrow \uparrow\downarrow \uparrow,\mathrm{Cir}(\mathrm{cup}^{k-2}, \mathrm{cap}^{k-2})) \otimes \mathring{\mathcal{A}}^\wedge(\mathrm{cup}_{+-+-},\emptyset)$$
and
$$
B:= \mathring{\mathcal{A}}^\wedge(\mathrm{cap}_{+-+-},\emptyset)\otimes \mathring{\mathcal{A}}^\wedge(\downarrow^{k-2}\otimes\downarrow \uparrow\downarrow \uparrow,\emptyset) \otimes \mathring{\mathcal{A}}^\wedge(\mathrm{cup}_{+-+-},\emptyset)$$

Recall the  anti-homomorphism $z:\mathring{\mathcal{A}}^{\wedge}(\downarrow,\emptyset)\longrightarrow \mathring{\mathcal{A}}^{\wedge}(\uparrow,\emptyset)$ from Definition~\ref{themapz}. Let $h: B\to B$ be the linear map defined by 
\begin{equation}
h(x\otimes y \otimes w) = x \otimes \Big(y\circ_{\mathcal A} \big(\nu^{\otimes k-2}\otimes_{\mathcal A}\nu \otimes_{\mathcal A}z(\nu)\otimes_{\mathcal A}\mathrm{Id}_{+-}^{\mathcal A}\big)\Big) \otimes w
\end{equation}
for any $x\in \mathring{\mathcal{A}}^\wedge(\mathrm{cap}_{+-+-},\emptyset)$, $y\in \mathring{\mathcal{A}}^\wedge(\downarrow^{k-2}\otimes\downarrow \uparrow\downarrow \uparrow,\emptyset)$ and $w\in \mathring{\mathcal{A}}^\wedge(\mathrm{cup}_{+-+-},\emptyset)$.

Let $g$ be the endomorphism of $\mathring{\mathcal{A}}^\wedge(\downarrow^k,\emptyset)$ defined by 
\begin{equation}
g(x)= \big(\mathrm{Id}_{+^{\otimes (k-2)}}^{\mathcal A}\otimes_{\mathcal A} \nu \otimes_{\mathcal A}\mathrm{Id}_{+}^{\mathcal A}\big) \circ_{\mathcal A} x \circ_{\mathcal A} \big( \nu^{\otimes (k-2)}\otimes_{\mathcal A} \mathrm{Id}^{\mathcal A}_+\otimes_{\mathcal A} \nu\big)
\end{equation}
for any $x\in \mathring{\mathcal{A}}^\wedge(\downarrow^k,\emptyset)$.

Let $\tilde{\lambda}_k^\nu : \mathring{\mathcal{A}}^\wedge(\downarrow \uparrow\downarrow \uparrow,\mathrm{Cir}(\mathrm{cup}^{k-2}, \mathrm{cap}^{k-2}))\to \mathring{\mathcal{A}}^\wedge(\downarrow \uparrow\downarrow \uparrow, \mathrm{Cir}(\mathrm{cup}^{k-2}, \mathrm{cap}^{k-2}))$ be the linear map given by 
\begin{equation}\tilde{\lambda}^{\nu}_k(x) =(\nu\otimes  z(\nu) \otimes  \mathrm{Id}_{+-})\circ \widetilde{\mathrm{cs}}^{\nu}_{\downarrow\uparrow\downarrow\uparrow, \mathrm{Cir}(\mathrm{cup}^{k-2}, \mathrm{cap}^{k-2})}(x)
\end{equation}
for any $x\in\mathring{\mathcal{A}}^\wedge(\downarrow \uparrow\downarrow \uparrow, \mathrm{Cir}(\mathrm{cup}^{k-2}, \mathrm{cap}^{k-2}))$,

Let $c: B\to A$ be the linear map given by $c=\mathrm{Id}\otimes \mathrm{clos}_{k-2}\otimes \mathrm{Id}$, where $\mathrm{clos}_{k-2}$ is as in Definition~\ref{def:closuremap} with $P=\downarrow\uparrow\downarrow\uparrow$.

Let $f: B\to  \mathring{\mathcal{A}}^\wedge(\downarrow^k,\emptyset)$ be the linear  map defined schematically in Figure~\ref{figure20231224}.
\begin{figure}[ht!]
			\centering
            \includegraphics[scale=0.7]{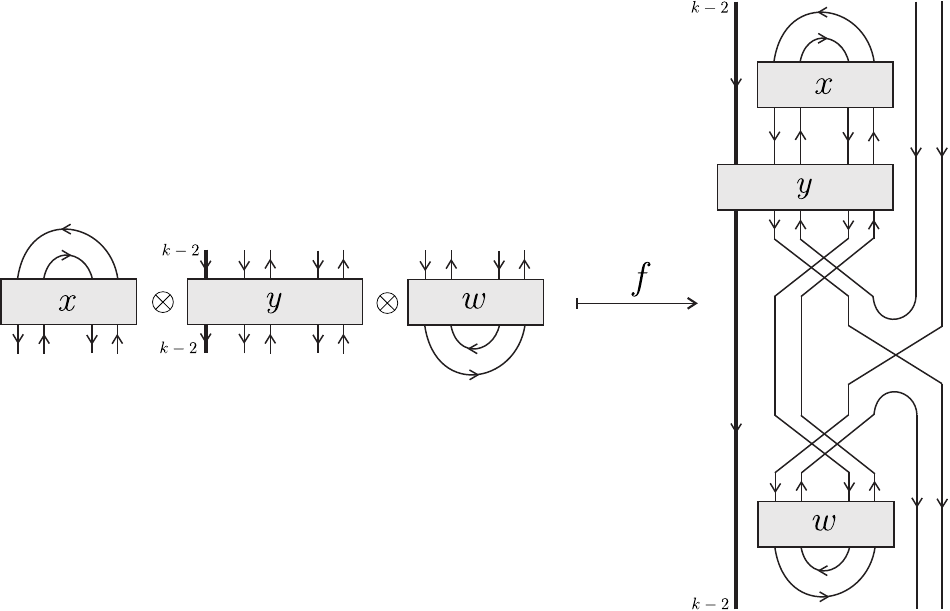}					\caption{The linear map $f$. The thick line with the label $k-2$ means $k-2$ vertical segments oriented downward.}
\label{figure20231224} 								
\end{figure}

Then the diagrams in  \eqref{eq:thenewdiagram2} and \eqref{eq:thenewdiagram3} are commutative.  
\begin{equation}\label{eq:thenewdiagram2}
\xymatrix{B \ar_{h}[d]
\ar^-{c}[r] & A \ar^{\mathrm{Id}\otimes \tilde{\lambda}_k^\nu\otimes\mathrm{Id}}[d]\\
B\ar^-{c}[r]  &A} 
\qquad \qquad \qquad \xymatrix{B \ar_{h}[d]
\ar^-{f}[r] & \mathring{\mathcal{A}}^\wedge(\downarrow^k, \emptyset) \ar^{g}[d]\\
B\ar^-{f}[r]  &\mathring{\mathcal{A}}^\wedge(\downarrow^k,\emptyset)}
\end{equation}

\begin{equation}\label{eq:thenewdiagram3}
\xymatrix{B \ar_{c}[d]
\ar^-{f}[r] &  \mathring{\mathcal A}^{\wedge}(\downarrow^k,\emptyset) \ar^-{\mathrm{clos}_k}[r]& \mathring{\mathcal A}^{\wedge}(\vec{\varnothing},\mathrm{Cir}(\mathrm{cup}^k,\mathrm{cap}^k)) \ar^{\simeq}[d]\\
A\ar^-{}[rr] & & \mathring{\mathcal A}^{\wedge}(\vec{\varnothing},\mathrm{Cir}(\mathrm{cap}^{k-2}, \mathrm{cap}^{k-2})\sqcup \mathrm{Cir}(\mathrm{cap}_{+-+-},\mathrm{cup}_{+-+-}))}
\end{equation}
The bottom map in \eqref{eq:thenewdiagram3} is the composition map and the right vertical isomorphism is induced by the bijection $\mathrm{Cir}(\mathrm{cup}^k,\mathrm{cap}^k) \xrightarrow{\simeq} \mathrm{Cir}(\mathrm{cap}^{k-2}, \mathrm{cap}^{k-2})\sqcup \mathrm{Cir}(\mathrm{cap}_{+-+-},\mathrm{cup}_{+-+-})$ which identifies the most outer circles in $\mathrm{cup}^k\circ \mathrm{cap}^k$ with $\mathrm{Cir}(\mathrm{cap}_{+-+-},\mathrm{cup}_{+-+-})$.
\end{lemma}

\begin{proof}
Let $R:= \mathrm{Cir}(\mathrm{cap}^{k-2}, \mathrm{cup}^{k-2})$. Define the linear maps  
$$a: \mathring{\mathcal{A}}^{\wedge}(\downarrow^{k-2}\otimes \downarrow\uparrow\downarrow\uparrow,\emptyset) \to \mathring{\mathcal{A}}^{\wedge}(\downarrow^{k-2}\otimes \downarrow\uparrow\downarrow\uparrow,\emptyset)\quad \text{and} \quad b: \mathring{\mathcal{A}}^{\wedge}( \downarrow\uparrow\downarrow\uparrow,R) \to \mathring{\mathcal{A}}^{\wedge}( \downarrow\uparrow\downarrow\uparrow,R)$$
by $a(y)=(\mathrm{Id}_{(+)^{\otimes (k-2)}}^{\mathcal A}\otimes_{\mathcal A}\nu\otimes_{\mathcal A}z(\nu)\otimes_{\mathcal A}\mathrm{Id}^{\mathcal A}_{+-})\circ_{\mathcal A} y$ for any $y\in \mathring{\mathcal{A}}^{\wedge}(\downarrow^{k-2}\otimes \downarrow\uparrow\downarrow\uparrow,\emptyset)$ and  by  $b(y)= (\nu\otimes_{\mathcal A} z(\nu)\otimes_{\mathcal A}\mathrm{Id}^{\mathcal A}_{+-})\circ_{\mathcal A}y$ for any $y\in  \mathring{\mathcal{A}}^{\wedge}( \downarrow\uparrow\downarrow\uparrow,R)$. In diagram~\eqref{eq:diagramauxiliar2024}, the left triangle commutes because of associativity of the composition, the right triangle commutes  by the definition of  the involved maps, the top square commutes by Lemma~\ref{themapcsalpha}$(a)$ and the bottom square commutes by the identity represented schematically in Figure~\ref{figure20231229-1}. Therefore the outer square, which is the left  diagram on the left of \eqref{eq:thenewdiagram2}, commutes. 
\begin{figure}[ht!]
			\centering
            \includegraphics[scale=0.8]{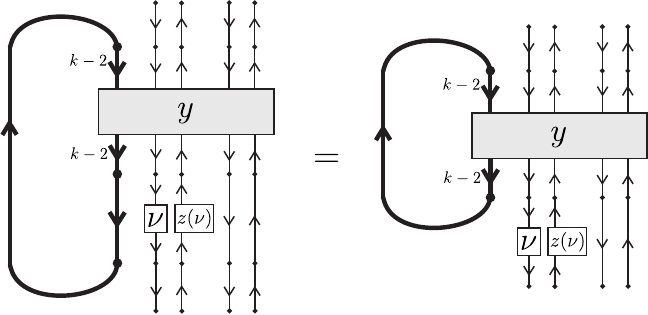}	\caption{Schematic proof of the equality $\mathrm{clos}_{k-2}\circ a (y) = \mathrm{clos}_{k-2} (y) \circ_{\mathcal A} (\nu\otimes_{\mathcal A} z(\nu) \otimes_{\mathcal A} \mathrm{Id}_{+-}^{\mathcal A})$ in $ \mathring{\mathcal{A}}^{\wedge}(\downarrow\uparrow\downarrow\uparrow, \mathrm{Cir}(\mathrm{cap}^{k-2}, \mathrm{cup}^{k-2}))$ for any $y\in  \mathring{\mathcal{A}}^{\wedge}(\downarrow^{k-2}\otimes \downarrow\uparrow\downarrow\uparrow,\emptyset)$.}
\label{figure20231229-1} 								
\end{figure}

\begin{equation}\label{eq:diagramauxiliar2024}
\xymatrix{
  \mathring{\mathcal{A}}^{\wedge}(\downarrow^{k-2}\otimes \downarrow\uparrow\downarrow\uparrow,\emptyset) \ar^-{\mathrm{clos}_{k-2}}[rrr] \ar_{(\nu^{\otimes (k-2)}\otimes_{\mathcal A} \nu \otimes_{\mathcal A} z(\nu)\otimes_{\mathcal A} \mathrm{Id}_{+-}^{\mathcal A})\circ_{\mathcal A} \bullet}[dd]\ar^-{\quad \qquad (\nu^{\otimes (k-2)}\otimes_{\mathcal A}  \mathrm{Id}_{+-+-}^{\mathcal A})\circ_{\mathcal A} \bullet}[dr]&    & & \mathring{\mathcal{A}}^{\wedge}(\downarrow\uparrow\downarrow\uparrow,R) \ar^{\widetilde{\lambda}^{\nu}_{k}}[dd]\ar_{\widetilde{\mathrm{cs}}^{\nu}_{\downarrow\uparrow\downarrow\uparrow, R}}[dl]\\
&  \mathring{\mathcal{A}}^{\wedge}(\downarrow^{k-2}\otimes \downarrow\uparrow\downarrow\uparrow,\emptyset) \ar^-{\mathrm{clos}_{k-2}}[r] \ar_-{a}[dl]& \mathring{\mathcal{A}}^{\wedge}(\downarrow\uparrow\downarrow\uparrow,R) \ar^{b}[rd] & \\
 \mathring{\mathcal{A}}^{\wedge}(\downarrow^{k-2}\otimes \downarrow\uparrow\downarrow\uparrow,\emptyset)  \ar^-{\mathrm{clos}_{k-2}}[rrr]&    & & \mathring{\mathcal{A}}^{\wedge}(\downarrow\uparrow\downarrow\uparrow,R)}
\end{equation}

Let $y\in \mathring{\mathcal{A}}^\wedge(\downarrow^{k-2}\otimes\downarrow \uparrow\downarrow \uparrow,\emptyset)$. Multiple applications of Lemma~\ref{commut-prop} implies 
\begin{equation}\label{eq:20231225-eq1}
y \circ_{\mathcal A}\big(\nu^{\otimes(k-2)}\otimes_{\mathcal A}\nu\otimes_{\mathcal A} z(\nu)\otimes_{\mathcal A} \mathrm{Id}^{\mathcal A}_{+-}\big) = \big(\nu^{\otimes(k-2)}\otimes_{\mathcal A}\nu\otimes_{\mathcal A} z(\nu)\otimes_{\mathcal A} \mathrm{Id}^{\mathcal A}_{+-}\big)  \circ_{\mathcal A} y,
\end{equation}
equality in  $\mathring{\mathcal{A}}^\wedge(\downarrow^{k-2}\otimes\downarrow \uparrow\downarrow \uparrow,\emptyset)$.  A schematic representation of it is shown in Figure~\ref{figure20231225-1}.
\begin{figure}[ht!]
			\centering
            \includegraphics[scale=0.9]{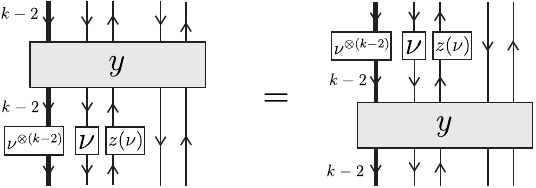}					\caption{Schematic representation of the equality \eqref{eq:20231225-eq1}.}
\label{figure20231225-1} 								
\end{figure}
Equality \eqref{eq:20231225-eq1} implies the commutativity of the diagram on the left in \eqref{eq:thenewdiagram2}.

Let $x\in \mathring{\mathcal{A}}^\wedge(\mathrm{cap}_{+-+-},\emptyset)$, $y\in \mathring{\mathcal{A}}^\wedge(\downarrow^{k-2}\otimes\downarrow \uparrow\downarrow \uparrow,\emptyset)$ and $w\in \mathring{\mathcal{A}}^\wedge(\mathrm{cup}_{+-+-},\emptyset)$. We give  schematic representation of $f \circ h (x\otimes y\otimes w)$ and  $g \circ f (x\otimes y\otimes w)$ in Figure~\ref{figure20231225-2}$(a)$ and $(b)$, respectively.

\begin{figure}[ht!]
			\centering
            \includegraphics[scale=0.9]{
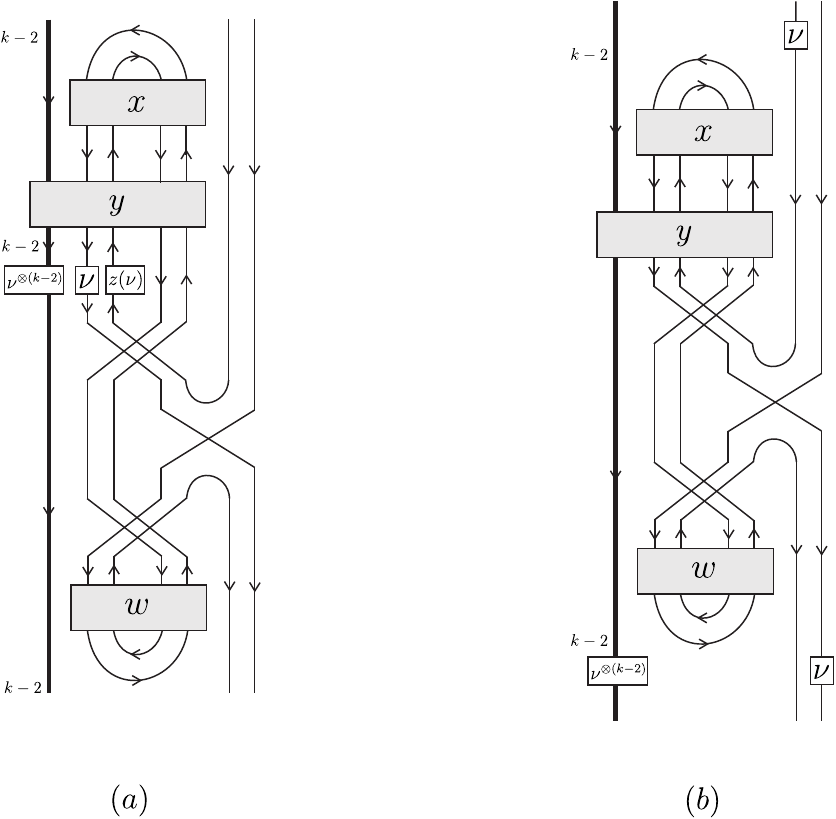}					\caption{$(a)$ Schematic representation of $f \circ h (x\otimes y\otimes w)$. $(b)$ Schematic representation of $g \circ f (x\otimes y\otimes w)$.}
\label{figure20231225-2} 								
\end{figure}

Lemma~\ref{slide-nu} implies that these two elements are equal in  $ \mathring{\mathcal{A}}^\wedge(\downarrow^k,\emptyset)$, which is exactly the commutativity of the diagram on the right in  \eqref{eq:thenewdiagram2}.

The commutativity of the diagram  \eqref{eq:thenewdiagram3} follows from the equality shown in Figure~\ref{figure20231225-3}.

\begin{figure}[ht!]
			\centering
            \includegraphics[scale=0.7]{
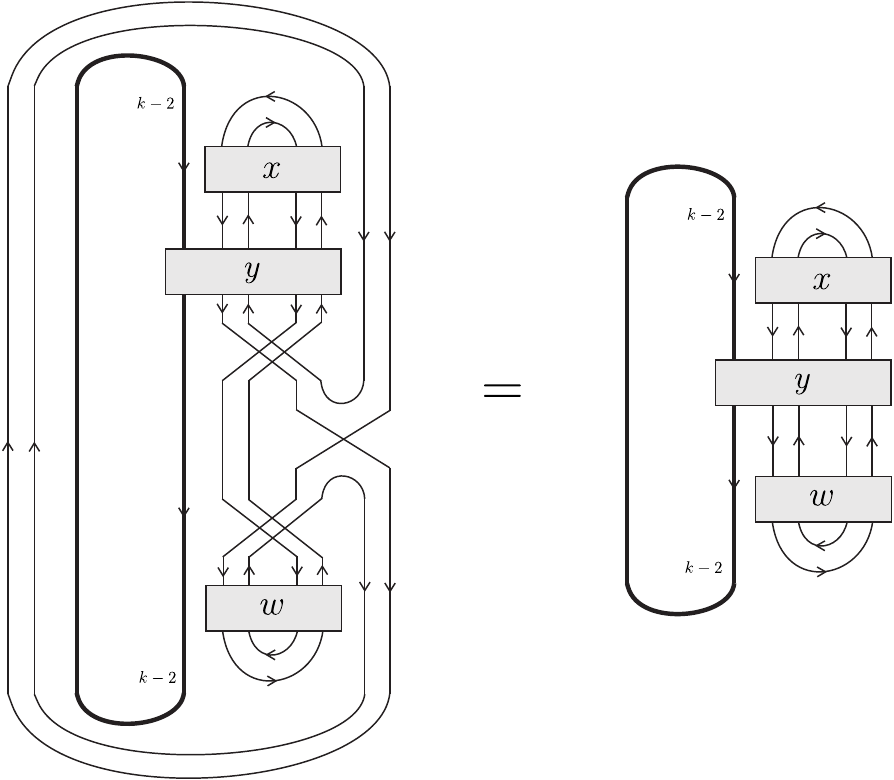}					\caption{Equality implying the commutativity of the diagram  \eqref{eq:thenewdiagram3}.}
\label{figure20231225-3} 								
\end{figure}

\end{proof}

\begin{lemma}\label{rlemma20211031} Let $k\in\mathbb{Z}_{\geq 2}$. Let 
$$\lambda^{\nu}_k:\mathring{\mathcal{A}}^\wedge(\downarrow \uparrow\downarrow \uparrow)_{k-2}\longrightarrow \mathring{\mathcal{A}}^\wedge(\downarrow \uparrow\downarrow \uparrow)_{k-2}$$
be the map defined by
$\lambda^{\nu}_k(x) :=\big(\nu\otimes_{\mathcal A} z(\nu)\otimes_{\mathcal A} (\mathrm{Id}_{+-})_{\mathcal A}\big)\circ_{\mathcal A}\mathrm{cs}^{\nu}_{\downarrow \uparrow \downarrow \uparrow,k-2}(x)$ for any $x\in\mathring{\mathcal{A}}^\wedge(\downarrow \uparrow\downarrow \uparrow)_{k-2}$.  The composition operation $\circ_{\mathcal A}$ gives rise to a continuous  linear map 
$$\mathring{\mathcal{A}}^\wedge(\mathrm{cap}_{+-+-})_{0}\otimes \mathring{\mathcal{A}}^\wedge(\downarrow \uparrow\downarrow \uparrow)_{k-2} \otimes \mathring{\mathcal{A}}^\wedge(\mathrm{cup}_{+-+-})_{0}\longrightarrow 
\mathring{\mathcal{A}}^\wedge_{k-2+|\mathrm{Cir}(\mathrm{cap}_{+-+-}, \mathrm{cup}_{+-+-})|}=\mathring{\mathcal{A}}^\wedge_k. $$
Then the  diagram
\begin{equation}\label{eq:thediagram}
\xymatrix{\mathring{\mathcal{A}}^\wedge(\mathrm{cap}_{+-+-})_{0}\otimes \mathring{\mathcal{A}}^\wedge(\downarrow \uparrow\downarrow \uparrow)_{k-2} \otimes \mathring{\mathcal{A}}^\wedge(\mathrm{cup}_{+-+-})_{0}\ar_{\mathrm{Id}\otimes \lambda^{\nu}_{k}\otimes\mathrm{Id}}[d]
\ar^{}[rrr] & &  & \mathring{\mathcal{A}}^\wedge_k \ar^{\mathrm{cs}^{\nu}}[d]\\
\mathring{\mathcal{A}}^\wedge(\mathrm{cap}_{+-+-})_{0}\otimes \mathcal{A}^\wedge(\downarrow \uparrow\downarrow \uparrow)_{k-2}\otimes \mathring{\mathcal{A}}^\wedge(\mathrm{cup}_{+-+-})_{0}\ar^{}[rrr]  & & &\mathring{\mathcal{A}}^\wedge_k}
\end{equation}
is commutative.
\end{lemma}
\begin{proof}
Let $A$, $B$ and $\tilde{\lambda}_k^{\nu}$ be as in Lemma~\ref{20231225-auxiliarlemma}.  Let $S:=  \mathrm{Cir}(\mathrm{cap}^{k-2}, \mathrm{cup}^{k-2})\sqcup \mathrm{Cir}(\mathrm{cap}_{+-+-},\mathrm{cup}_{+-+-})$. 
The composition map $A \longrightarrow \mathring{\mathcal A}^{\wedge}(\vec{\varnothing}, S)$ and  the maps $\mathrm{Id} \otimes \tilde{\lambda}_k^\nu \otimes \mathrm{Id}: A \to A$ and $\tilde{\mathrm{cs}}^{\nu}_{\vec{\varnothing}, S}: \mathring{\mathcal A}^{\wedge}(\vec{\varnothing}, S)  \to \mathring{\mathcal A}^{\wedge}(\vec{\varnothing}, S)$ are respectively compatible with the evident group homomorphism $\mathfrak{S}_{\mathrm{Cir}(\mathrm{cap}^{k-2}, \mathrm{cup}^{k-2})}\to \mathfrak{S}_{S}$ and
with the actions of these groups. Therefore, they give rise to morphisms in the category $\mathcal Vect\mathcal Gp$. By their definition and since  $|S|=k$ and  $|\mathrm{Cir}(\mathrm{cap}^{k-2}, \mathrm{cup}^{k-2})|=k-2$, the images of these morphisms under the coinvariant space functor are the maps involved in Diagram~\eqref{eq:thediagram}. The announced commutativity is then a consequence of the commutativity of Diagram~\eqref{eq:thenewdiagram},  which we now prove.   
\begin{equation}\label{eq:thenewdiagram}
\xymatrix{ A  \ar_{\mathrm{Id}\otimes \tilde{\lambda}^{\nu}_{k}\otimes\mathrm{Id}}[d]
\ar^{}[r] & \mathring{\mathcal{A}}^\wedge(\vec{\varnothing},S) \ar^{\widetilde{\mathrm{cs}}^{\nu}_{\vec{\varnothing},S}}[d]\\
 A  \ar^{}[r]  &\mathring{\mathcal{A}}^\wedge(\vec{\varnothing},S)}
\end{equation}
We will decompose Diagram~\eqref{eq:thenewdiagram} in several diagrams. Let $c$, $f$, $g$, and $h$ as in Lemma~\ref{20231225-auxiliarlemma} and consider the diagram  \eqref{eq:diagramauxiliar}.
\begin{equation}\label{eq:diagramauxiliar}
\xymatrix{
 A \ar^{}[rrrr] \ar_{\mathrm{Id}\otimes \tilde{\lambda}^{\nu}_{k}\otimes\mathrm{Id}}[ddd]&  &  & & \mathring{\mathcal{A}}^{\wedge}(\vec{\varnothing},S) \ar^{\widetilde{\mathrm{cs}}^{\nu}_{\vec{\varnothing},S}}[ddd]\\
  & B \ar^-{f}[r]\ar_-{c}[lu]  \ar_{h}[d]&  \mathring{\mathcal{A}}^{\wedge}(\downarrow^k,\emptyset) \ar@/^1pc/^{\bullet\circ_{\mathcal A} \nu^{\otimes k}}[d] \ar@/_1pc/_{g}[d]\ar^-{\mathrm{clos}_k}[r]& \mathring{\mathcal{A}}^{\wedge}(\vec{\varnothing},\mathrm{Cir}(\mathrm{cap}^{k}, \mathrm{cup}^{k}))\ar^{\tilde{\mathrm{cs}}^{\nu}_{\vec{\varnothing},\mathrm{Cir}(\mathrm{cap}^{k}, \mathrm{cup}^{k})}}[d]\ar^{\simeq}[ru] & \\
  & B \ar^-{f}[r]\ar_-{c}[ld]&  \mathring{\mathcal{A}}^{\wedge}(\downarrow^k,\emptyset) \ar^-{\mathrm{clos}_k}[r]& \mathring{\mathcal{A}}^{\wedge}(\vec{\varnothing},\mathrm{Cir}(\mathrm{cap}^{k}, \mathrm{cup}^{k})) \ar^{\simeq}[rd] & \\
 A \ar^{}[rrrr]&  &  & & \mathring{\mathcal{A}}^{\wedge}(\vec{\varnothing},S)}
\end{equation}
The top and bottom pentagons are the commutative diagram \eqref{eq:thenewdiagram3}, the square on the left is the commutative diagram on the left of \eqref{eq:thenewdiagram2}, the inner square involving $f$, $g$ and $h$ is the commutative diagram on the right of \eqref{eq:thenewdiagram2}.  By Lemma~\ref{commut-prop} we have $g= \bullet\circ_{\mathcal A} \nu^{\otimes k}$. The inner square involving $\mathrm{clos}_k$ is commutative  by Lemma~\ref{themapcsalpha}$(a)$. The square on the right is commutative by Lemma~\ref{themapcsalpha}$(b)$. All this implies the commutativity of the outer square which is exactly \eqref{eq:thenewdiagram}. This finishes the proof.
\end{proof}

\begin{lemma}\label{r20210921-1} The diagram
\begin{equation}
\xymatrix{\mathring{\mathcal{A}}^\wedge(\downarrow \uparrow)_{k-2}\ar_{\mathrm{cs}^{\nu}_{+-}}[d]
\ar^{\ \ \mathrm{KIIMap}_{{\mathring{\mathcal A}}^\wedge,k}^{1}}[rrr]  & & &\mathring{\mathcal{A}}^\wedge_k\ar^{\mathrm{cs}^{\nu}}[d]\\
\mathring{\mathcal{A}}^\wedge(\downarrow \uparrow)_{k-2}\ar^{\ \ \mathrm{KIIMap}_{{\mathring{\mathcal A}}^\wedge,k}^{s,1}}[rrr] &  & &\mathring{\mathcal{A}}^\wedge_k}.
\end{equation}
commutes, where the maps $\mathrm{cs}^{\nu}_{+-}$ and $\mathrm{cs}^{\nu}$ are given in~\eqref{equ:mapcs-nuopen} and~\eqref{cs-nu-total} respectively. 
\end{lemma}
\begin{proof}
We use the schematic representations of the maps $\mathrm{cs}^{\nu}_{+-}$, $\mathrm{KIIMap}^1_{{\mathring{\mathcal A}}^\wedge,k}$ and $\mathrm{KIIMap}^{s,1}_{{\mathring{\mathcal A}}^\wedge,k}$ given in Figures~\ref{figureEx1_49_abstract}, \ref{figureEx1_52_abstract} and \ref{figureEx1_53_abstract} to give a schematic proof in Figure~\ref{figureEx1_54_abstract}. In the second equality we have used Lemma~\ref{slide-nu} and in the third equality Lemma~\ref{rlemma20211031}.
\begin{figure}[ht!]
			\centering
            \includegraphics[scale=0.7]{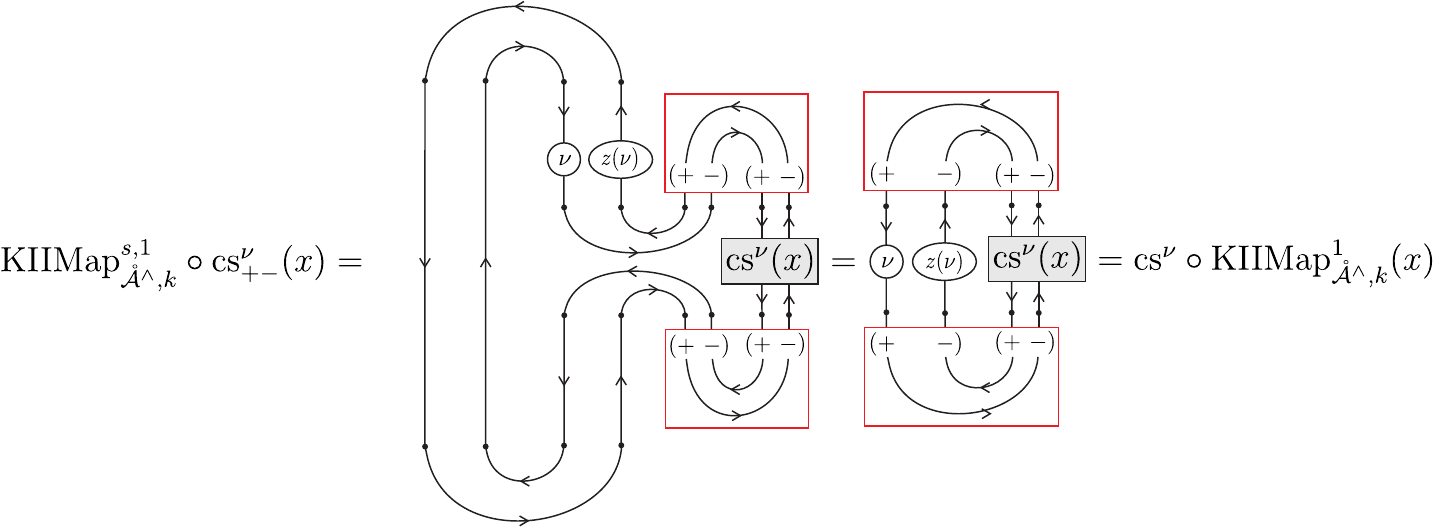}					\caption{The equality $\mathrm{KIIMap}^{s,1}_{{\mathring{\mathcal{A}}}^\wedge,k}\circ \mathrm{cs}^\nu_{+-}(x)=\mathrm{cs}^\nu\circ \mathrm{KIIMap}^{1}_{{\mathring{\mathcal{A}}}^\wedge,k}(x)$. We represent $x\in{\mathring{\mathcal{A}}}^{\wedge}(\downarrow\uparrow)_{k-2}$  by the grey-colored box.}
\label{figureEx1_54_abstract} 								
\end{figure}

\end{proof}

Consider the oriented Brauer diagrams $\mathrm{cap}_{-++-}:= \mathrm{cap}_{-+}\otimes \mathrm{cap}_{+-}\in\underline{\vec{\mathcal Br}}(-++-,\emptyset)$ and  $\mathrm{cup}_{-++-}:= \mathrm{cup}_{-+}\otimes \mathrm{cup}_{+-}\in\underline{\vec{\mathcal Br}}(\emptyset, -++-)$. Here the oriented Brauer diagrams  $\mathrm{cap}_{u}$ and $\mathrm{cup}_{u}$ for $u\in\{-+,+-\}$ are as in Definition~\ref{def:cup:capinBR}.

\begin{definition} Define the multilinear map
$$\langle \ ,\ , \  \rangle_{k}: \mathring{\mathcal{A}}^{\wedge}(\mathrm{cap}_{-++-})_{0}\otimes  \mathring{\mathcal{A}}^{\wedge}(\mathrm{cup}_{-++-})_{0}\otimes  \mathring{\mathcal{A}}^{\wedge}(\downarrow\downarrow\uparrow)_{k-2}\longrightarrow  \mathring{\mathcal{A}}^{\wedge}_k$$

by $$\langle x,y,z\rangle_{k}= x\circ_{\mathcal A} ((\mathrm{Id}_-)_{\mathcal A}\otimes_{\mathcal A} z)\circ_{\mathcal A} y $$
for any $x\in \mathring{\mathcal{A}}^{\wedge}(\mathrm{cap}_{-++-})_{0}$, $y\in \mathring{\mathcal{A}}^{\wedge}(\mathrm{cup}_{-++-})_{0}$, $z\in \mathring{\mathcal{A}}^{\wedge}(\downarrow\downarrow\uparrow)_{k-2}$.  See Figure~\ref{figureEx1_60_abstract} for a schematic representation of this map.
\begin{figure}[ht!]
			\centering
            \includegraphics[scale=0.8]{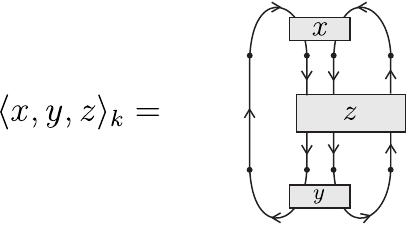}					\caption{Schematic representations of the map $\langle \ ,\ , \  \rangle_{k}$. Notice that in the grey-colored box labelled by $z$  there is implicitly $k-2$ circles which are not shown in the figure.}
\label{figureEx1_60_abstract} 								
\end{figure}
\end{definition}

\begin{lemma}\label{r20211102lemequalityofpairing} Let $x\in \mathring{\mathcal{A}}^{\wedge}(\mathrm{cap}_{-++-})_{0}$, $y\in \mathring{\mathcal{A}}^{\wedge}(\mathrm{cup}_{-++-})_{0}$, $z\in \mathring{\mathcal{A}}^{\wedge}(\downarrow\uparrow)_{k-2}$ and $v\in\mathring{\mathcal{A}}^{\wedge}(\downarrow\downarrow)_{0}$,  with $v$  invertible. Set $${u}:= (\mathrm{Id}_-)_{\mathcal A}\otimes_{\mathcal A} v\otimes_{\mathcal A}(\mathrm{Id}_-)_{\mathcal A}\in\mathring{\mathcal{A}}^{\wedge}(\uparrow\downarrow\downarrow\uparrow)_{0}.$$
Then
\begin{equation}\label{r20211102equ1}
\big\langle x\circ_{\mathcal A} u, u^{-1}\circ_{\mathcal A} y, \mathrm{dbl}_1(z)\rangle = \langle x, y, \mathrm{dbl}_1(z)\big\rangle_{k} 
\end{equation}

\end{lemma}

\begin{proof} In Figure~\ref{figureEx1_68_abstract} we give a schematic proof of the  identity~\eqref{r20211102equ1}. 
\begin{figure}[ht!]
			\centering
            \includegraphics[scale=0.8]{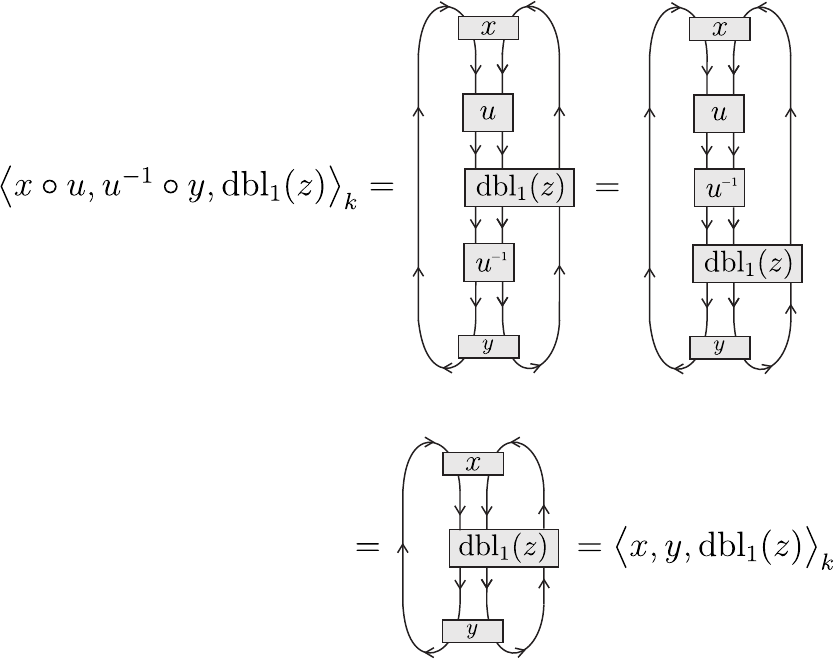}					\caption{Schematic proof of the identity~\eqref{r20211102equ1}.}
\label{figureEx1_68_abstract} 								
\end{figure}
\end{proof}

\begin{lemma}\label{r20211021lem1} Let 
$$\alpha_{t}, \beta_{t}\in \mathring{\mathcal{A}}^{\wedge}(\mathrm{cap}_{-++-})_{0}, \quad \quad  \alpha_{b}, \beta_{b}\in \mathring{\mathcal{A}}^{\wedge}(\mathrm{cup}_{-++-})_{0},$$
be the elements given by
\begin{gather*}
\begin{split}
\alpha_{t} &={Z}_{\mathcal A}(\mathrm{cap}^{\mathcal T}_{-((++)-)})\circ_{\mathcal A} ((\mathrm{Id}_{-})_{\mathcal A}\otimes_{\mathcal A} \nu\otimes_{\mathcal A} \nu\otimes_{\mathcal A}(\mathrm{Id}_-)_{\mathcal A}), \\
\alpha_{b} &= {Z}_{\mathcal A}(\mathrm{cup}^{\mathcal T}_{-((++)-)}),\\
\beta_{t} &= \big((\mathrm{cap}_{-+})_{\mathcal A}\otimes_{\mathcal A} (\mathrm{cap}_{+-})_{\mathcal A}\big)\circ_{\mathcal A}\Big((\mathrm{Id}_-)_{\mathcal A}\otimes_{\mathcal A} \mathrm{dbl}_1\big((\nu\otimes_{\mathcal A} z(\nu)\otimes_{\mathcal A}{Z}_{\mathcal A}(\mathrm{cap}^{\mathcal T}_{(+-)(+-)})) \\
& \qquad\qquad \qquad \qquad\qquad \qquad  \qquad\qquad \qquad \qquad \qquad \qquad \circ_{\mathcal A} ((\mathrm{cup}_{+-+-})_{\mathcal A}\otimes (\mathrm{Id}_{+-})_{\mathcal A})\big)\Big),\\
\beta_{b}& =\Big((\mathrm{Id}_{-})_{\mathcal A}\otimes_{\mathcal A} \mathrm{dbl}_1\big(((\mathrm{cap}_{+-+-})_{\mathcal A}\otimes_{\mathcal A}(\mathrm{Id}_{+-})_{\mathcal A})\circ_{\mathcal A}((\mathrm{Id}_{+-})_{\mathcal A}\otimes_{\mathcal A}{Z}_{\mathcal A}(\mathrm{cup}^{\mathcal T}_{(+-)(+-)}))\big)\Big) \\
&\qquad\qquad \qquad \qquad\qquad \qquad  \qquad\qquad \qquad \qquad \qquad \qquad \circ_{\mathcal A} \big((\mathrm{cup}_{-+})_{\mathcal A}\otimes_{\mathcal A} (\mathrm{cup}_{+-})_{\mathcal A}\big).
\end{split}
\end{gather*}
In Figures~\ref{figureEx1_61_abstract} and~\ref{figureEx1_62_abstract}  we give a schematic representations of these elements using Conventions~\ref{convention:red-box} and~\ref{convention-doublingblueboxes}. 
\begin{figure}[ht!]
			\centering
            \includegraphics[scale=0.7]{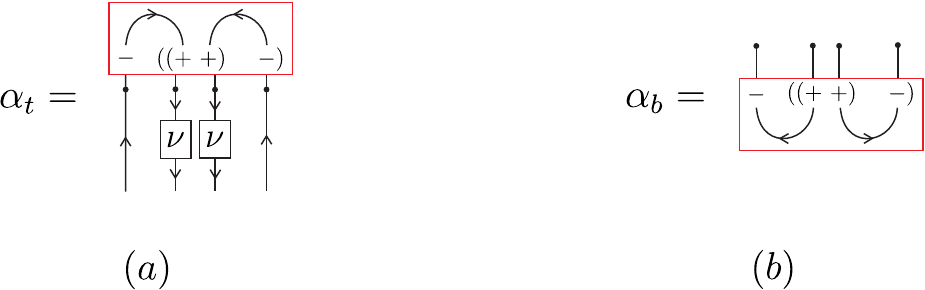}					\caption{Schematic representation of the elements  $\alpha_{t}$ and $\alpha_{b}$.}
\label{figureEx1_61_abstract} 								
\end{figure}

\begin{figure}[ht!]
			\centering
            \includegraphics[scale=0.7]{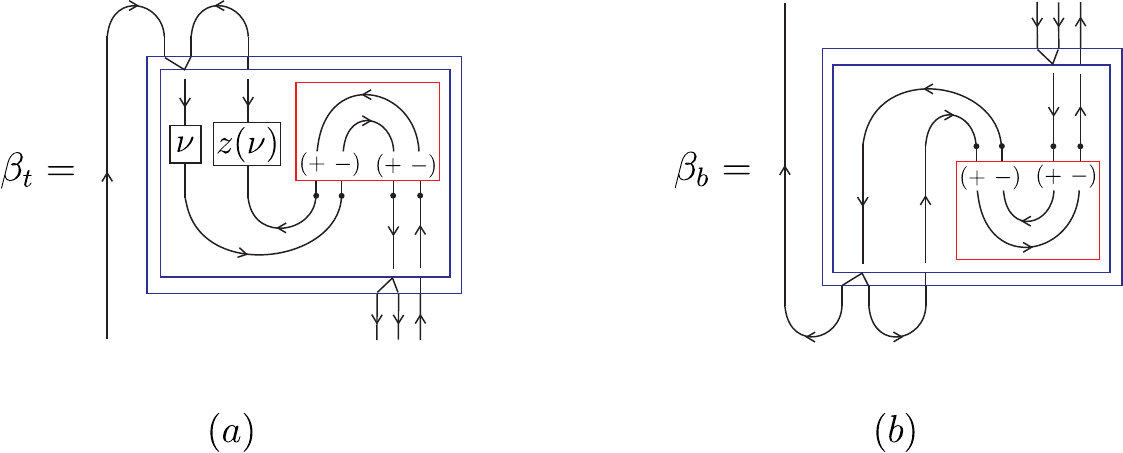}					\caption{Schematic representation of the elements  $\beta_{t}$ and $\beta_{b}$. }
\label{figureEx1_62_abstract} 								
\end{figure}
Then for any $k\geq 2$ and any $x\in\mathring{\mathcal{A}}^\wedge(\downarrow\uparrow)_{k-2}$ we have the  following equalities in $\mathring{\mathcal A}^\wedge_k$. 
\begin{equation}\label{r20211021lem1equ1}
\big\langle \alpha_{t}, \alpha_{b}, \mathrm{dbl}_1(\mathrm{cs}^{\nu}(x)) \big\rangle_{k} = \mathrm{cs}^\nu\circ \mathrm{KIIMap}^2_{{\mathring{\mathcal{A}}}^\wedge,k}(x)
\end{equation}
\begin{equation}\label{r20211021lem1equ2}
\big\langle \beta_{t}, \beta_{b}, \mathrm{dbl}_1(\mathrm{cs}^{\nu}(x)) \big\rangle_{k} = \mathrm{KIIMap}^{s,2}_{{\mathring{\mathcal{A}}}^\wedge,k}\circ \mathrm{cs}^\nu_{+-}(x).
\end{equation}
Here we are using Convention~\ref{convention:co_i-dbl_i} for the doubling map.
\end{lemma}

\begin{proof}
The two equalities  follow directly by the definitions of the elements and maps involved together with Lemma~\ref{rlemma20211031}, Lemma~\ref{slide-nu}, the commutativity of $\mathrm{cs}^\nu$ with the doubling operation and the compatibility of the composition with the doubling operation. Indeed, in Figure~\ref{figureEx1_63_abstract}  we show a schematic version of the proof of~\eqref{r20211021lem1equ1}.  
\begin{figure}[ht!]
			\centering
            \includegraphics[scale=0.75]{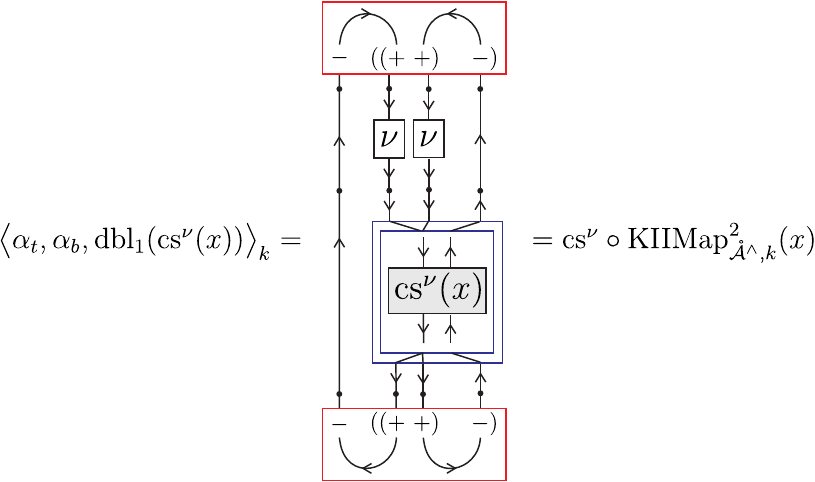}					\caption{Schematic representation of the proof of \eqref{r20211021lem1equ1}. In the second equality we use the schematic representation of $\mathrm{KIIMap}^2_{{\mathring{\mathcal A}}^\wedge,k}$ from Figure~\ref{figureEx1_52_abstract}.}
\label{figureEx1_63_abstract} 								
\end{figure}

In Figure~\ref{figureEx1_64_abstract} we show a schematic proof of \eqref{r20211021lem1equ2}, where in the second equality we use the compatibility of the doubling map $\mathrm{dbl}_1:\mathring{\mathcal{A}}^\wedge(\downarrow\uparrow)\to \mathring{\mathcal{A}}^\wedge(\downarrow\downarrow\uparrow)$ with the composition in  $\mathring{\mathcal{A}}^\wedge(\downarrow\uparrow)$.
\begin{figure}[ht!]
			\centering
            \includegraphics[scale=0.75]{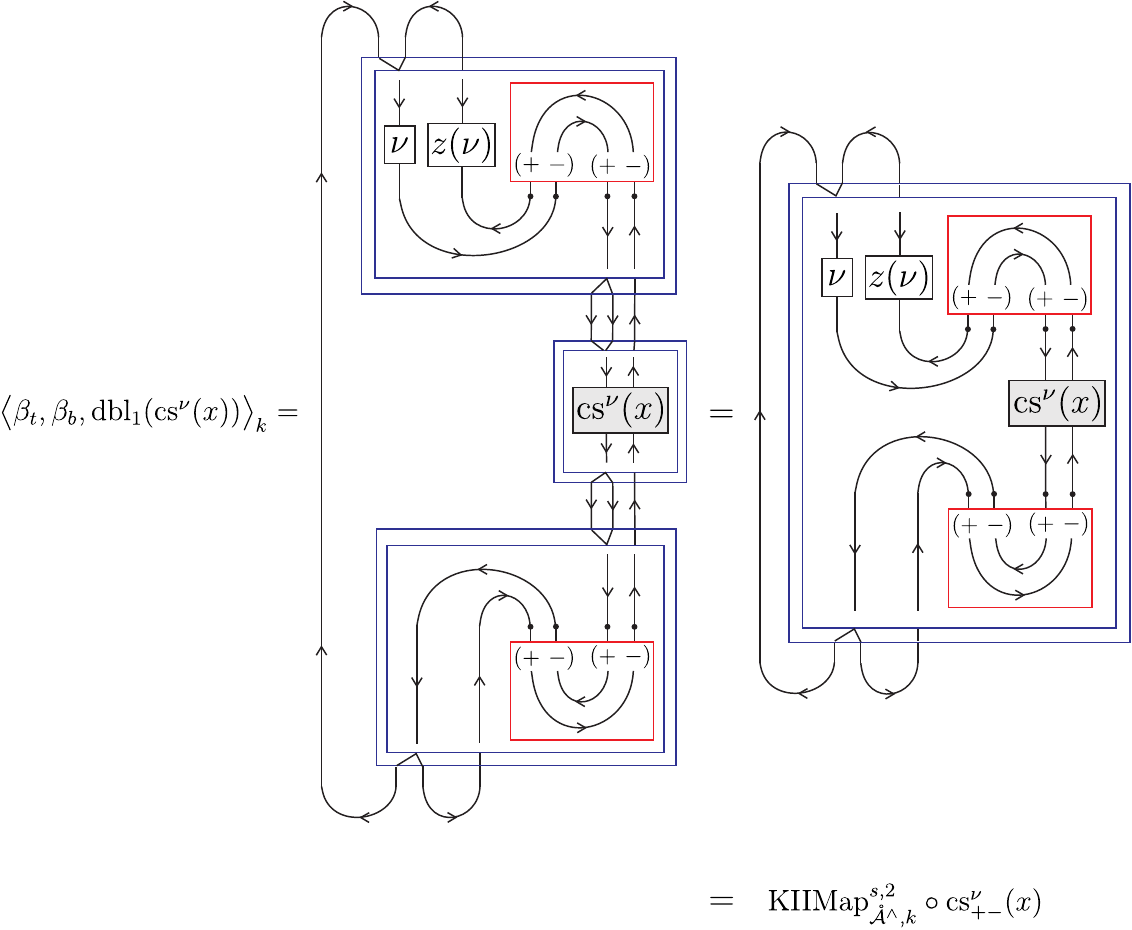}					\caption{Schematic representation of the proof of \eqref{r20211021lem1equ2}.}
\label{figureEx1_64_abstract} 								
\end{figure}

\end{proof}

\begin{lemma}\label{r20211019lem1} Let ${P}$ be the oriented Brauer diagram $(\mathrm{Id}_+\otimes \mathrm{cup}_{+-}\otimes \mathrm{Id}_-)\circ\mathrm{cup}_{+-}$. Consider the action of $\mathring{\mathcal{A}}^{\wedge}(\downarrow\downarrow\uparrow\uparrow)_0$ on $\mathring{\mathcal{A}}^{\wedge}({P})_0$  given by composition. Let $\gamma\in\mathring{\mathcal{A}}^{\wedge}(\downarrow\downarrow\uparrow\uparrow)_0$ be the element 
\begin{align*}
\includegraphics[scale=0.9]{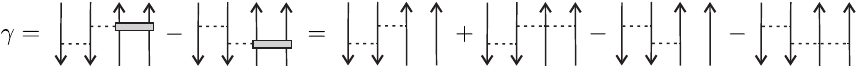}
\end{align*}
then $\gamma\cdot m =0$ for any $m\in\mathring{\mathcal{A}}^{\wedge}({P})_0$.
\end{lemma}

\begin{proof} Let $m\in\mathring{\mathcal{A}}^{\wedge}({P})_0$, which we represent with a grey-colored box labelled with $m$ over the schematic representation of the oriented Brauer  ${P}$. Then we have
\begin{align*}
\includegraphics[scale=1]{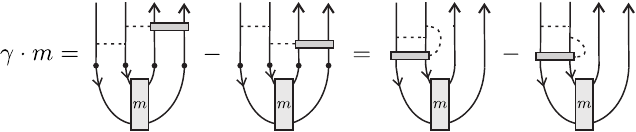}
\end{align*}
In the second equality we have used Lemma~\ref{commut-prop-doubl}. Besides, we have
\begin{align*}
\includegraphics[scale=1]{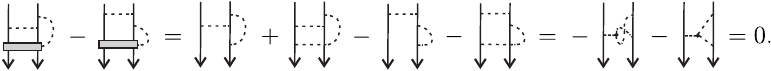}
\end{align*}
In the second equality we have used a STU relation in the first $\downarrow$ involving the first and third term and a STU relation in the second $\downarrow$ involving the second and fourth term. Last equality follows by the AS relation. Therefore, $\gamma\cdot m = 0$.
\end{proof}

Notice that the linear maps $\mathrm{co}_3:\mathring{\mathcal{A}}^{\wedge}(\downarrow\downarrow\downarrow)_0\to \mathring{\mathcal{A}}^{\wedge}(\downarrow\downarrow\uparrow)_0$ and  $\mathrm{dbl}_3:\mathring{\mathcal{A}}^{\wedge}(\downarrow\downarrow\uparrow)_0\to \mathring{\mathcal{A}}^{\wedge}(\downarrow\downarrow\uparrow\uparrow)_0$ are algebra morphisms. Therefore, the composition map
\begin{equation}\label{equ:mor}
\delta:=\mathrm{dbl}_3\circ \mathrm{co}_3\circ \mathrm{can}: \mathbb{C}\langle\langle A, B\rangle\rangle\longrightarrow \mathring{\mathcal{A}}^{\wedge}(\downarrow\downarrow\uparrow\uparrow)
\end{equation}
is also an algebra morphism, where $\mathrm{can}:\mathbb{C}\langle\langle A, B\rangle\rangle \to \mathring{\mathcal{A}}^{\wedge}(\downarrow\downarrow\downarrow)_0$
is the algebra morphism given in \eqref{equ:morphcanC-ABtoA}. Explicitly, the morphism $\delta$ is given by
\begin{align*}
\includegraphics[scale=0.9]{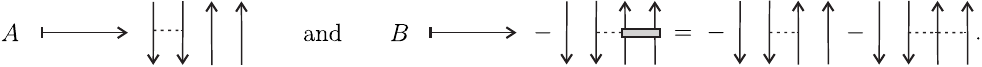}
\end{align*}

\begin{corollary}\label{r20211019cor1} Let ${P}$ be as in Lemma~\ref{r20211019lem1}. Then for any $x\in\mathbb{C}\langle\langle A,B\rangle\rangle [A,B] \mathbb{C}\langle\langle A,B\rangle\rangle$ and any $m\in \mathring{\mathcal{A}}^{\wedge}({P})_0$ we have $\delta(x)\cdot m = 0$. 
\end{corollary}

\begin{proof}
Notice that $\delta([A,B])=\gamma$, where $\gamma$ is the element described in Lemma~\ref{r20211019lem1}. Therefore, by this lemma, $\delta([A,B])\cdot m = 0$ for any $m\in   \mathring{\mathcal{A}}^{\wedge}(P)_0$, therefore $[A,B] \in \mathrm{ker}(\delta)$. Since $\delta$ is an algebra morphism, its kernel is a two-sided ideal, therefore it contains   $\mathbb{C}\langle\langle A,B\rangle\rangle [A,B] \mathbb{C}\langle\langle A,B\rangle\rangle$.
\end{proof}

\begin{corollary}\label{r20211020corasso} Let ${P}$ be as in Lemma~\ref{r20211019lem1}. Then for any $m\in\mathring{\mathcal{A}}^{\wedge}({P})_0$ we have
\begin{align*}
\includegraphics[scale=0.75]{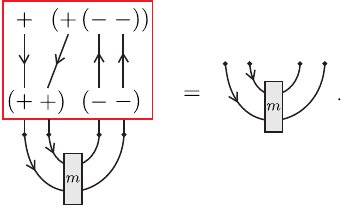}		
\end{align*}
Here we are using Convention~\ref{convention:red-box} and represent $m$ as in the proof of Lemma~\ref{r20211019lem1}.
\end{corollary}
\begin{proof}
By definition of ${Z}_{\mathcal A}$, the left-hand side is given by
\begin{equation}\label{equ1:lemma20211020}
\mathrm{LHS} = \big(\mathrm{dbl}_3\circ \mathrm{co}_3(\Phi)\big)\circ m = \big(\mathrm{dbl}_3\circ \mathrm{co}_3\circ \mathrm{can}(\varphi(A,B))\big)\circ m = \delta(\varphi(A,B))\cdot m,
\end{equation}
where $\delta$ is the algebra morphism \eqref{equ:mor}. Besides, we can write the even Drinfeld associator $\varphi(A,B)$ as 
$$\varphi(A,B) = 1 + \theta(A,B)$$
with $\theta(A,B)\in \mathbb{C}\langle\langle A,B\rangle\rangle [A,B] \mathbb{C}\langle\langle A,B\rangle\rangle$.  Replacing in \eqref{equ1:lemma20211020} and using Corollary~\ref{r20211019cor1}, we have
$$\mathrm{LHS} = \delta(1 +\theta(A,B))\cdot m = m = \mathrm{RHS}.$$
\end{proof}

\begin{lemma}\label{somecomputationsofZ} We have
\begin{align*}
\includegraphics[scale=1]{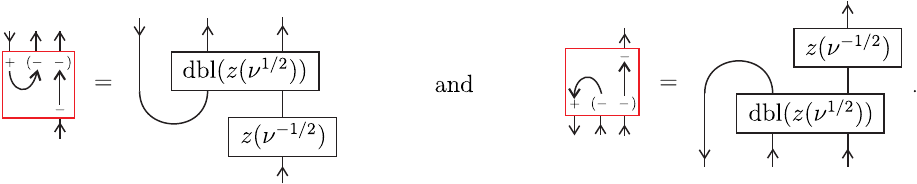}
\end{align*}
where $z:\mathring{\mathcal{A}}^\wedge(\downarrow)\to \mathring{\mathcal{A}}^\wedge(\uparrow)$ is the map given in \eqref{themapz}. We have similar identities obtained by changing all the orientations and replacing $z(\nu^{\pm 1/2})$ with $\nu^{\pm 1/2}$.
\end{lemma}
\begin{proof}
Let us prove the first identity. Using the definition of ${Z}_{\mathcal A}$ and its compatibility  with composition and tensor product, we have

\begin{align*}
\includegraphics[scale=1]{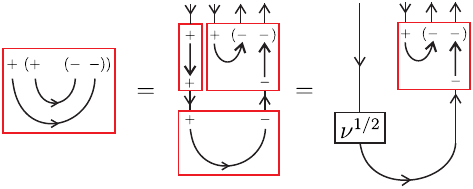}
\end{align*}

\noindent Besides

\begin{align*}
\includegraphics[scale=1]{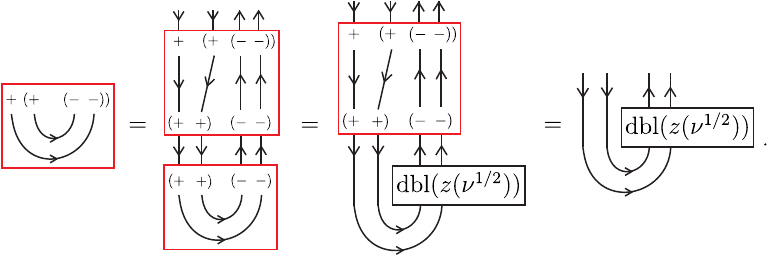}
\end{align*}

\noindent In the first and second equalities we have used the definition of ${Z}_{\mathcal A}$, its compatibility   with the composition, tensor product and doubling operation, see Theorem~\ref{thmkontsevichintegral}.  The third  equality follows from Corollary~\ref{r20211020corasso}.

Thus
\begin{align*}
\includegraphics[scale=1]{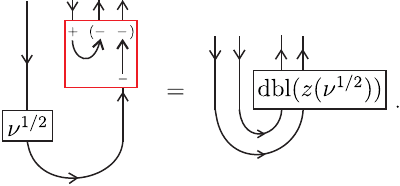}
\end{align*}
Therefore
\begin{align*}
\includegraphics[scale=1]{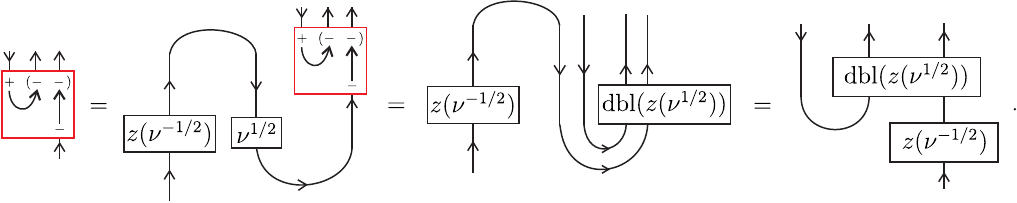}
\end{align*}

\noindent The second identity can be proved similarly.
\end{proof}

\begin{lemma}\label{r20211029lienbetaalpha}
Let
$$u:=(\mathrm{Id}_{-})_{\mathcal A}\otimes_{\mathcal A} \big(\mathrm{dbl}(\nu^{-1/2})\circ_{\mathcal A} (\nu^{1/2}\otimes_{\mathcal A} (\mathrm{Id}_{+})_{\mathcal A})\big) \otimes_{\mathcal A}(\mathrm{Id}_{-})_{\mathcal A}\in \mathring{\mathcal{A}}^\wedge(\uparrow\downarrow\downarrow\uparrow).$$
Then
\begin{equation}\label{equ1_20211027}
\alpha_{t} = \beta_{t} \circ_{\mathcal A} u \quad \quad  \text{and} \quad \quad \alpha_{b} =u^{-1}\circ_{\mathcal A} \beta_{b}.
\end{equation}
\end{lemma}
\begin{proof}
 In Figure~\ref{figureEx1_65_abstract} we show the identity $\beta_{t}=\alpha_{t}\circ_{\mathcal A} (u)^{-1}$. We use:  the compatibly of the doubling map with the composition and tensor product in the second equality; Lemma~\ref{lemmaZanddoubling} in third equality; Lemma~\ref{slide-nu} in the fourth equality; Lemmas~\ref{slide-nu} and \ref{commut-prop} in the fifth equality and Lemma~\ref{somecomputationsofZ} in the sixth equality.

\begin{figure}[ht!]
			\centering
            \includegraphics[scale=0.7]{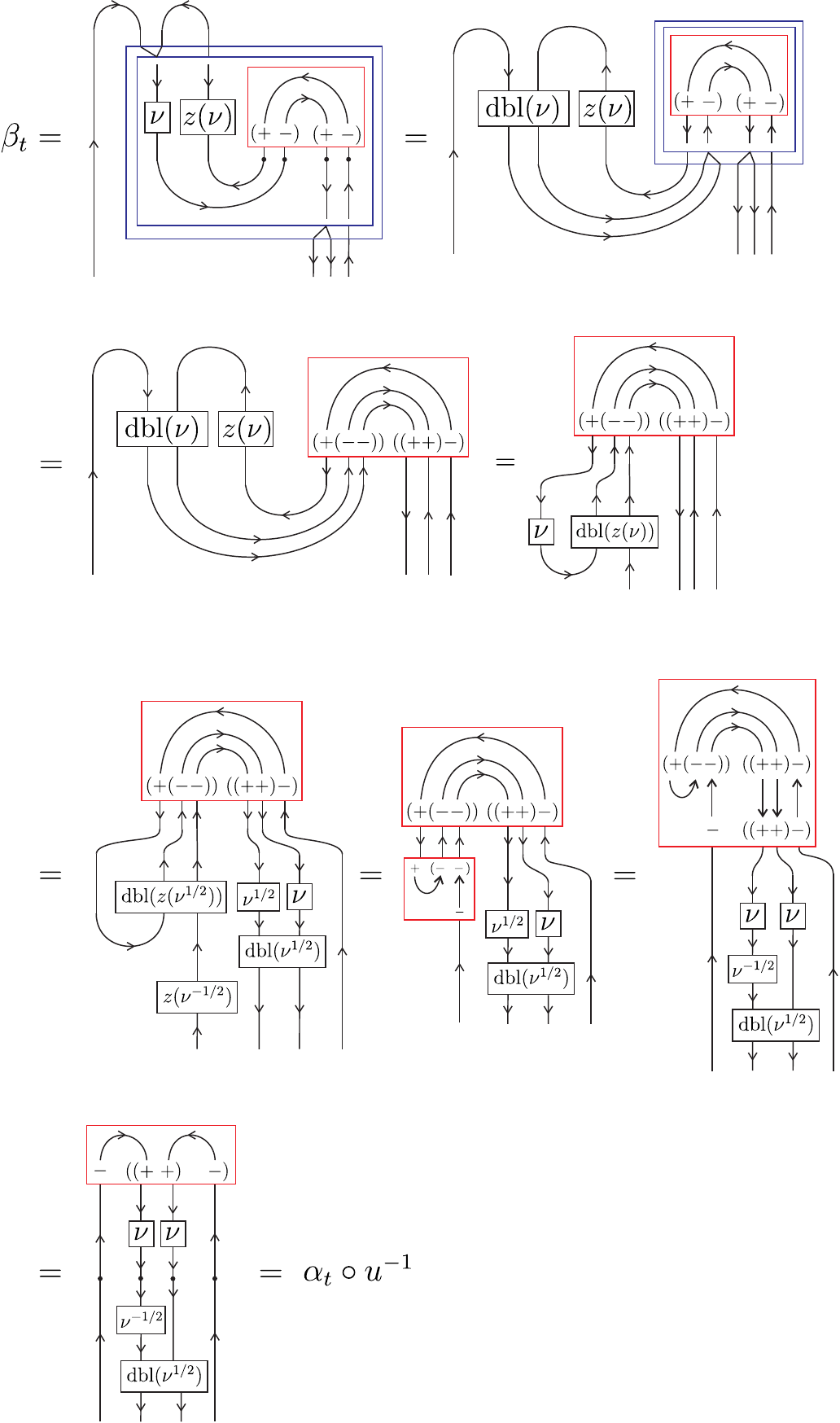}					\caption{Proof of the equality $\beta_{t}=\alpha_{t}\circ_{\mathcal A} u^{-1}$ from Lemma~\ref{r20211029lienbetaalpha}.}
\label{figureEx1_65_abstract} 								
\end{figure}

Let us now prove the second identity in~\eqref{equ1_20211027}. 
Using the compatibility of the doubling map with the composition and tensor product and Lemma~\ref{lemmaZanddoubling}, we get the equalities shown in Figure~\ref{auxequality}.

\begin{figure}[ht!]
			\centering
\includegraphics[scale=0.75]{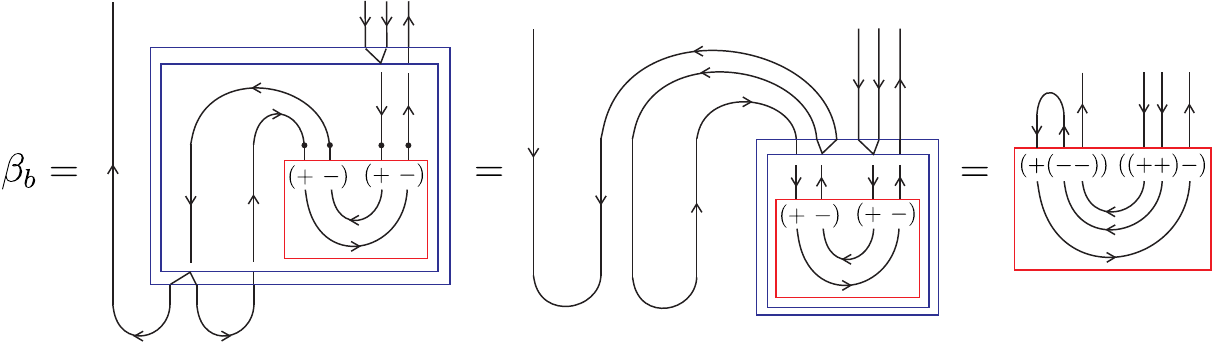}
\caption{An equivalent expression of $\beta_{b}$.}
\label{auxequality} 								
\end{figure}

Using the identity  from Figure~\ref{auxequality}, we show in Figure~\ref{figureEx1_67_abstract} the equality $u^{-1}\circ_{\mathcal A}\beta_{b}=\alpha_{t}$. We use: Lemmas~\ref{slide-nu} and~\ref{commut-prop}  in the second equality; Lemma~\ref{somecomputationsofZ} in the third equality and the compatibility  of ${Z}_{\mathcal A}$ with tensor product and composition in the fourth equality. 
\begin{figure}[ht!]
			\centering
            \includegraphics[scale=0.7]{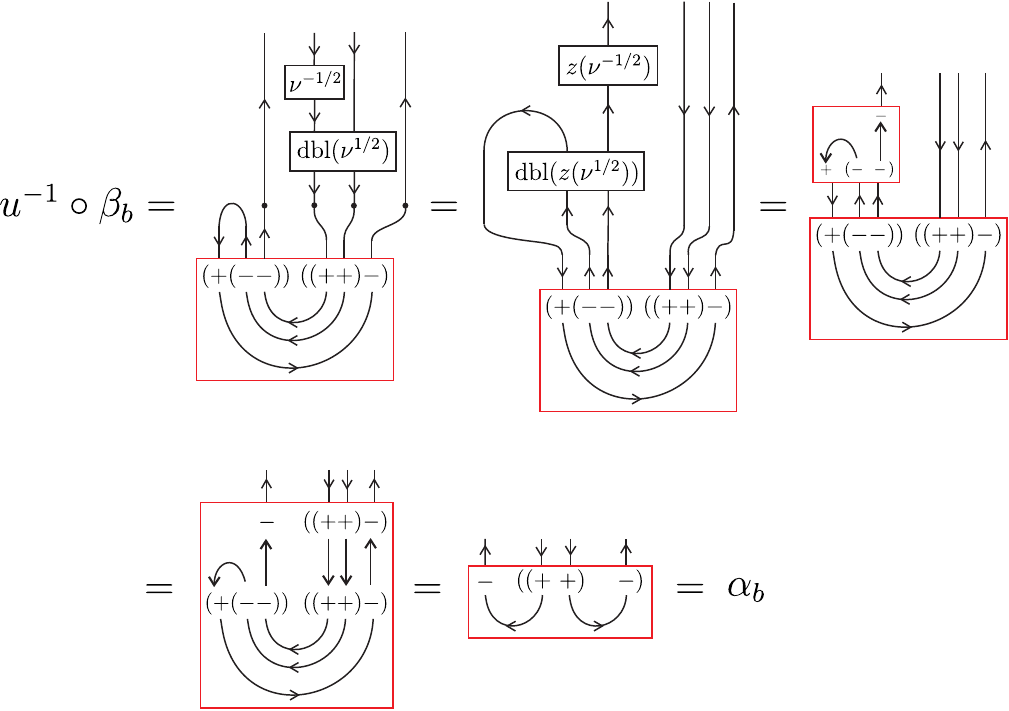}					\caption{Proof of the equality $u^{-1}\circ_{\mathcal A}\beta_{b}=\alpha_{b}$ from Lemma~\ref{r20211029lienbetaalpha}.}
\label{figureEx1_67_abstract} 								
\end{figure}

\end{proof}

\begin{proposition}\label{r20210921-2} Let $k\in\mathbb{Z}_{\geq 2}$. The diagram
\begin{equation}
\xymatrix{\mathring{\mathcal{A}}^\wedge(\downarrow \uparrow)_{k-2}\ar_{\mathrm{cs}^{\nu}_{+-}}[d]
\ar^{\ \ \mathrm{KIIMap}_{{\mathring{\mathcal A}}^\wedge,k}^{2}}[rrr] & &  &\mathring{\mathcal{A}}^\wedge_k\ar^{\mathrm{cs}^{\nu}}[d]\\
\mathring{\mathcal{A}}^\wedge(\downarrow \uparrow)_{k-2}\ar^{\ \ \mathrm{KIIMap}_{{\mathring{\mathcal A}}^\wedge,k}^{s,2}}[rrr]  & & &\mathring{\mathcal{A}}^\wedge_k}.
\end{equation}
commutes, where the maps $\mathrm{cs}^{\nu}_{+-}$ and $\mathrm{cs}^{\nu}$ are given in~\eqref{equ:mapcs-nuopen} and~\eqref{cs-nu-total} respectively. 
\end{proposition}

\begin{proof}
Let $x\in\mathring{\mathcal{A}}^\wedge(\downarrow \uparrow)_{k-2}$, then
\begin{gather*}
\begin{split}
 \mathrm{cs}^\nu\circ \mathrm{KIIMap}^2_{{\mathring{\mathcal{A}}}^\wedge,k}(x) & =\big\langle \alpha_{t}, \alpha_{b}, \mathrm{dbl}_1(\mathrm{cs}^{\nu}(x)) \big\rangle_{k} \\
  & = \big\langle \beta_{t}\circ_{\mathcal A} u, u^{-1}\circ_{\mathcal A}\beta_{b}, \mathrm{dbl}_1(\mathrm{cs}^{\nu}(x)) \big\rangle_{k}\\
  & = \big\langle \beta_{t}, \beta_{b}, \mathrm{dbl}_1(\mathrm{cs}^{\nu}(x)) \big\rangle_{k}\\
  & = \mathrm{KIIMap}^{s,2}_{{\mathring{\mathcal{A}}}^\wedge,k}\circ \mathrm{cs}^{\nu}_{+-}(x).
\end{split}
\end{gather*}
We have used equation~\eqref{r20211021lem1equ1} from Lemma~\ref{r20211021lem1} in the first equality, Lemma~\ref{r20211029lienbetaalpha} in the second equality, Lemma~\ref{r20211102lemequalityofpairing} in the third equality and equation~\eqref{r20211021lem1equ2} from Lemma~\ref{r20211021lem1} in the last equality.

\end{proof}

\begin{corollary}\label{r20210809} Let $k\in\mathbb{Z}_{\geq 2}$. The diagram
\begin{equation}
\xymatrix{\mathring{\mathcal{A}}^\wedge(\downarrow \uparrow)_{k-2}\ar_{\mathrm{cs}^{\nu}_{+-}}[d]
\ar^{\ \ \mathrm{KIIMap}_{{\mathring{\mathcal A}}^\wedge,k}^{2} - \mathrm{KIIMap}_{{\mathring{\mathcal A}}^\wedge,k}^{1}}[rrrrr] & & & & &\mathring{\mathcal{A}}^\wedge_k \ar^{\mathrm{cs}^{\nu}}[d]\\
\mathring{\mathcal{A}}^\wedge(\downarrow \uparrow)_{k-2}\ar^{\ \ \mathrm{KIIMap}_{{\mathring{\mathcal A}}^\wedge,k}^{s,2} - \mathrm{KIIMap}_{{\mathring{\mathcal A}}^\wedge,k}^{s,1}}[rrrrr] & & & & &\mathring{\mathcal{A}}^\wedge_k}.
\end{equation}
commutes, where the maps $\mathrm{cs}^{\nu}_{+-}$ and $\mathrm{cs}^{\nu}$ are given in~\eqref{equ:mapcs-nuopen} and~\eqref{cs-nu-total} respectively. 
\end{corollary}
\begin{proof}
The result follows directly from Lemma~\ref{r20210921-1} and Proposition~\ref{r20210921-2}.
\end{proof}

\begin{corollary}\label{r2-20210809}
The algebra automorphism $\mathrm{cs}^{\nu}$ of $\bigoplus_{k \geq 0}\mathring{\mathcal{A}}^\wedge_k$ is such that $\mathrm{cs}^{\nu}((\mathrm{KII}\mathring{\mathcal{A}}^\wedge)) = (\mathrm{KII}^s\mathring{\mathcal{A}}^\wedge)$.
\end{corollary}
\begin{proof}
The result follows from Corollary~\ref{r20210809}, from the fact that the maps $\mathrm{cs}^\nu_{+-}$ and $\mathrm{cs}^\nu$ are bijective (Lemma~\ref{the-map-cs-+-} and Lemma~\ref{themapcsalpha}$(e)$) and from   the definitions of $(\mathrm{KII}\mathring{\mathcal{A}}^\wedge)$, $\mathrm{KII}^s\mathring{\mathcal{A}}^\wedge$ (see \eqref{spaceKIIforA} and \eqref{spaceKIIsimpleforA-1}).
\end{proof}

\subsubsection{The equality $\mathrm{cs}^\nu((\mathrm{CO}\mathring{\mathcal{A}}^\wedge))=(\mathrm{CO}\mathring{\mathcal{A}}^\wedge)$}\label{sec:3:2:4}

\begin{lemma}\label{lemma1-20210805}
The algebra automorphism $\mathrm{cs}^{\nu}=\bigoplus_{k\geq 0}\mathrm{cs}^{\nu}_{\vec{\varnothing},k}$ of $\bigoplus_{k \geq 0}\mathring{\mathcal{A}}^\wedge_k$ is such that $\mathrm{cs}^{\nu}((\mathrm{CO}\mathring{\mathcal{A}}^\wedge)) = (\mathrm{CO}\mathring{\mathcal{A}}^\wedge)$. 
\end{lemma}

\begin{proof}

Let $k \geq 0$ be an integer and let us identify $\mathrm{Cir}(\mathrm{cup}^k, \mathrm{cap}^k)$ with the set $[\![1,k]\!]$ by associating increasing numbers from the innermost circle to the outermost circle in $\mathrm{cap}^k \circ \mathrm{cup}^k$ and, by simplicity, continue to denote by $\mathrm{clos}_k$ the composite map $ \mathring{\mathcal{A}}^{\wedge}(\downarrow^k, \emptyset) \xrightarrow{\mathrm{clos}_k} \mathring{\mathcal{A}}^{\wedge}(\vec{\varnothing},\mathrm{Cir}(\mathrm{cup}^k, \mathrm{cap}^k)) \simeq  \mathring{\mathcal{A}}^{\wedge}(\vec{\varnothing},[\![1,k]\!])$.

Consider the diagram \eqref{auxiliardiagram2024-01-04}.

\begin{equation}\label{auxiliardiagram2024-01-04}
\xymatrix{
 \mathring{\mathcal{A}}^{\wedge}(\vec{\varnothing},[\![1,k]\!]) \ar^{\tilde{\mathrm{cs}}^{\nu}_{\vec{\varnothing},[\![1,k]\!]}}[rrrr] \ar_{\mathrm{co}^{{\mathcal A}}(\vec{\varnothing},([\![1,k]\!],1))}[ddd]&  &  & & \mathring{\mathcal{A}}^{\wedge}(\vec{\varnothing},[\![1,k]\!]) \ar^{\mathrm{co}^{{\mathcal A}}(\vec{\varnothing},([\![1,k]\!],1))}[ddd]\\
  & \mathring{\mathcal{A}}^{\wedge}(\downarrow^k, \emptyset) \ar^-{\bullet  \circ_{\mathcal A}\nu^{\otimes k}}[rr]\ar_-{\mathrm{clos}_k}[lu]  \ar_{\mathrm{co}^{{\mathcal A}}((\downarrow^k, \downarrow_1),\emptyset)}[d]&  & \mathring{\mathcal{A}}^{\wedge}(\downarrow^k, \emptyset)\ar^{\mathrm{co}^{{\mathcal A}}((\downarrow^k, \downarrow_1),\emptyset)}[d]\ar^{\mathrm{clos}_k}[ru] & \\
  & \mathring{\mathcal{A}}^{\wedge}(\uparrow\downarrow^{k-1}, \emptyset)  \ar^-{a}[rr]\ar_-{\widetilde{\mathrm{clos}}_k}[ld]& & \mathring{\mathcal{A}}^{\wedge}(\uparrow\downarrow^{k-1}, \emptyset) \ar^{\widetilde{\mathrm{clos}}_k}[rd] & \\
  \mathring{\mathcal{A}}^{\wedge}(\vec{\varnothing},[\![1,k]\!])  \ar^{\tilde{\mathrm{cs}}^{\nu}_{\vec{\varnothing},[\![1,k]\!]}}[rrrr]&  &  & &  \mathring{\mathcal{A}}^{\wedge}(\vec{\varnothing},[\![1,k]\!]) }
\end{equation}
where $a(y):= y \circ_{\mathcal A} \big(z(\nu) \otimes_{\mathcal A} \nu^{\otimes k-1}\big)$ for any $y\in  \mathring{\mathcal{A}}^{\wedge}(\uparrow\downarrow^{k-1}, \emptyset)$ and $\mathrm{co}^{{\mathcal A}}((\downarrow^k, \downarrow_1),\emptyset): \mathring{\mathcal{A}}^{\wedge}(\downarrow^k, \emptyset)  \to \mathring{\mathcal{A}}^{\wedge}(\uparrow\downarrow^{k-1}, \emptyset)$ is as in Lemma~\ref{r:2022-06-13changeoforientation} (the subindexed arrow $\downarrow_1$ indicates the first arrow, from left to right, in~$\downarrow^k$). The map  $\widetilde{\mathrm{clos}}_k: \mathring{\mathcal{A}}^{\wedge}(\uparrow\downarrow^{k-1}, \emptyset) \to  \mathring{\mathcal{A}}^{\wedge}(\vec{\varnothing},[\![1,k]\!])$ is defined as the composition of $\mathrm{clos}_k$ with the map $\zeta: \mathring{\mathcal A}(\uparrow\downarrow^{k-1},\emptyset) \to \mathring {\mathcal A}(\downarrow^{k},\emptyset) $ given by
$$x\xmapsto{\ \zeta\ } \big(  \mathrm{Id}_+^{\mathcal A}  \otimes_{\mathcal A} (\mathrm{cap}_{+-})_{\mathcal A}   \otimes_{\mathcal A}  \mathrm{Id}_{+^{\otimes k-1}}^{\mathcal A}\big) \circ_{\mathcal A} \big( (\tau_{++})_{\mathcal A}  \otimes_{\mathcal A}  x \big) \circ_{\mathcal A}\big(\mathrm{Id}_+^{\mathcal A} \otimes_{\mathcal A} (\mathrm{cup}_{+-})_{\mathcal A} \otimes_{\mathcal A}\mathrm{Id}_{+^{\otimes k-1}}^{\mathcal A} \big)$$
for any $x\in \mathring{\mathcal A}(\uparrow\downarrow^{k-1},\emptyset)$, where $\mathrm{cap}_{+-}$, $\mathrm{cup}_{+-}$ and $\tau_{++}$ are the oriented Brauer diagrams as in Definition~\ref{def:cup:capinBR}, and the notation $(P)_{\mathcal A}$, for an oriented Brauer diagram $P$,  is as in Definition~\ref{def:emptyJacdiagram}.

The top square in diagram  \eqref{auxiliardiagram2024-01-04} is commutative by definition of the map $\tilde{\mathrm{cs}}^{\nu}_{\vec{\varnothing},[\![1,k]\!]}$, see Lemma~\ref{themapcsalpha}$(a)$. The bottom square in \eqref{auxiliardiagram2024-01-04} can be decomposed as shown in diagram \eqref{auxiliardiagram2024-01-12-3}.

\begin{equation}\label{auxiliardiagram2024-01-12-3}
\xymatrix{
 \mathring{\mathcal{A}}^{\wedge}(\uparrow\downarrow^{k-1}, \emptyset) \ar^-{a}[rrrr] \ar_{\widetilde{\mathrm{clos}}_k}[dd] \ar^{\zeta}[dr]&  &  & & \mathring{\mathcal{A}}^{\wedge}(\uparrow\downarrow^{k-1}, \emptyset) \ar^{\widetilde{\mathrm{clos}}_k}[dd]\ar_{\zeta}[dl]\\
  & \mathring{\mathcal{A}}^{\wedge}(\downarrow^{k}, \emptyset)  \ar^-{\bullet \circ_{\mathcal A} \nu^{\otimes k }}[rr]\ar_-{{\mathrm{clos}}_k}[ld]& & \mathring{\mathcal{A}}^{\wedge}(\downarrow^{k}, \emptyset) \ar^{{\mathrm{clos}}_k}[rd] & \\
  \mathring{\mathcal{A}}^{\wedge}(\vec{\varnothing},[\![1,k]\!])  \ar^{\tilde{\mathrm{cs}}^{\nu}_{\vec{\varnothing},[\![1,k]\!]}}[rrrr]&  &  & &  \mathring{\mathcal{A}}^{\wedge}(\vec{\varnothing},[\![1,k]\!]) }
\end{equation}
The bottom square in \eqref{auxiliardiagram2024-01-12-3} is commutative by definition of  $\tilde{\mathrm{cs}}^{\nu}_{\vec{\varnothing},[\![1,k]\!]}$, see Lemma~\ref{themapcsalpha}$(a)$. The triangles on the left and right in  \eqref{auxiliardiagram2024-01-12-3} are the same diagram and it commutes by definition. The top square in \eqref{auxiliardiagram2024-01-12-3} commutes by Lemmas~\ref{slide-nu} and~\ref{commut-prop}.  Therefore, the outer square in  \eqref{auxiliardiagram2024-01-12-3}, which is the bottom square in \eqref{auxiliardiagram2024-01-04}, is commutative.

The left and right squares in  \eqref{auxiliardiagram2024-01-04} are identical diagrams. Their commutativity is a consequence of that of the diagram 

\begin{equation}\label{auxiliardiagram2024-01-12}
\xymatrix{ \mathbb{C}\mathring{\mathrm{Jac}}(\downarrow^k,\emptyset) \ar_{\mathrm{co}^{\mathring{\mathrm{Jac}}}((\downarrow^k,\downarrow_1),\emptyset)}[d]  \ar^{\mathrm{clos}^{\mathring{\mathrm{Jac}}}_k}[rr] & & \mathbb{C}\mathring{\mathrm{Jac}}(\vec{\varnothing},[\![1,k]\!]) \ar^{\mathrm{co}^{\mathring{\mathrm{Jac}}}(\vec{\varnothing},([\![1,k]\!],1))}[d]\\
 \mathbb{C}\mathring{\mathrm{Jac}}(\uparrow \downarrow^{k-1},\emptyset) \ar^{\widetilde{\mathrm{clos}}^{\mathring{\mathrm{Jac}}}_k}[rr] & & \mathbb{C}\mathring{\mathrm{Jac}}(\vec{\varnothing},[\![1,k]\!])
}
\end{equation}
where $\mathrm{clos}^{\mathring{\mathrm{Jac}}}_k$, $\widetilde{\mathrm{clos}}^{\mathring{\mathrm{Jac}}}_k$ are the obvious lifts of their $\mathring{\mathcal A}^{\wedge}$-counterparts. Commutativity of \eqref{auxiliardiagram2024-01-12} follows from the commutativity of the following diagram 
\begin{equation*}\label{auxiliardiagram2024-01-12-2}
\xymatrix{\mathbb{Z}L\ar_{\mathrm{co}^L}[d] \ar^{f}[r] & \mathbb{Z}C\ar^{\mathrm{co}^C}[d]\\
\mathbb{Z}L\ar^{f}[r]& \mathbb{Z}C
}
\end{equation*}
where $L$ (resp. $C$) is the set of classes of tuples $(X,f,\{\mathrm{lin}_i\}_{i\in[\![1,k]\!]})$ (resp. $(X,f,\{\mathrm{cyc}_i\}_{i\in[\![1,k]\!]})$) where $X$ is a finite set, $f : X \to [\![1,k]\!]$ is a map, and for any $i \in [\![1,k]\!]$, $\mathrm{lin}_i$ (resp. $\mathrm{cyc}_i$) is a linear (resp. cyclic) order on $f^{-1}(i)$. The linear map $f : \mathbb{Z}L \to \mathbb{Z}C$ is induced by the set map $L \to C$ defined on representatives of classes by $(X,f,\{\mathrm{lin}_i\}_{i\in[\![1,k]\!]})\mapsto (X,f,\{\mathrm{cyc}(\mathrm{lin}_i)\}_{i\in[\![1,k]\!]})$ where 
$\mathrm{cyc}(\mathrm{lin})$ is the cyclic order associated to a linear order $\mathrm{lin}$.  The map $\mathrm{co}^L: \mathbb{Z}L \to \mathbb{Z}L$ (resp. $\mathrm{co}^C: \mathbb{Z}C \to \mathbb{Z}C$) is induced by the association $(X,f,\{\mathrm{lin}_i\}_{i\in[\![1,k]\!]}) \mapsto (-1)^{|f^{-1}(1)|}(X,f,\{\widetilde{\mathrm{lin}}_1,\mathrm{lin}_2,\ldots,\mathrm{lin}_k\})$ where $\widetilde{\mathrm{lin}}$ denotes the linear order opposite to the linear order $\mathrm{lin}$ (resp. $(X,f,\{\mathrm{cyc}_i\}_{i\in[\![1,k]\!]}) \mapsto 
(-1)^{|f^{-1}(1)|}(X,f,\{\widetilde{\mathrm{cyc}}_1,\mathrm{cyc}_2, \ldots,\mathrm{cyc}_k\})$ where $\widetilde{\mathrm{cyc}}$ is the cyclic order opposite to the cyclic order $\mathrm{cyc}$). 

Let us show the commutativity of the inner square. We can check that for any $x,y\in\mathring{\mathcal{A}}^{\wedge}(\downarrow^k, \emptyset)$ we have 
\begin{equation}\label{eq:2024-01-16-1}
\mathrm{co}^{\mathcal A}((\downarrow^k,\downarrow_1),\emptyset)\big(x\circ_{\mathcal A} y) = \mathrm{co}^{\mathcal A}((\downarrow^k,\downarrow_1),\emptyset)\big(x\big) \circ_{\mathcal A} \mathrm{co}^{\mathcal A}((\downarrow^k,\downarrow_1),\emptyset)\big(y\big).
\end{equation}

Besides, all  the terms appearing in the invertible element $\nu\in\mathring{\mathcal{A}}^\wedge(\downarrow)$ (see \S\ref{sec:3-8-2})   have an even number of univalent vertices attached to $\downarrow$ (see \cite{BLT} for an explicit formula for $\nu$). Therefore,  $\mathrm{co}^{\mathcal A}((\downarrow^k,\downarrow_1),\emptyset)\big(\nu^{\otimes k}\big) = z(\nu)\otimes_{\mathcal A}\nu^{\otimes k-1}$. Using this and \eqref{eq:2024-01-16-1} we obtain
$$\mathrm{co}^{\mathcal A}((\downarrow^k,\downarrow_1),\emptyset)\big(x\circ_{\mathcal A} \nu^{\otimes k}\big) = \mathrm{co}^{\mathcal A}((\downarrow^k,\downarrow_1),\emptyset)\big(x\big)\circ_{\mathcal A} (z(\nu)\otimes_{\mathcal A}\nu^{\otimes k-1}),$$
for any $x \in \mathring{\mathcal{A}}^{\wedge}(\downarrow^k, \emptyset)$, which is precisely the stated commutativity.  

We then have that the external square diagram of \eqref{auxiliardiagram2024-01-04} is commutative. The square  diagram obtained from  the external square diagram of \eqref{auxiliardiagram2024-01-04}  by replacing the vertical maps by the identity is also trivially commutative. Both diagrams are compatible with commutative group diagram
\begin{equation*}\label{auxiliardiagram2024-01-09}
\xymatrix{ \mathfrak{S}_{k-1}\ar^{}[d] \ar^{\mathrm{Id}}[r] & \mathfrak{S}_{k-1} \ar^{}[d]\\
\mathfrak{S}_{k} \ar^{\mathrm{Id}}[r]& \mathfrak{S}_k
}
\end{equation*}
where the vertical arrows are induced by the obvious identification $[\![1,k-1]\!] \simeq [\![2,k]\!] \subset [\![1,k]\!]$.
This result in two commutative  diagrams in $\mathcal{V}ec\mathcal{G}p$.  Applying the coinvariant space functor (see Lemma~\ref{def:coinvariantspacefunctor}) to these commutative diagrams, we obtain that the two  diagrams in  \eqref{auxiliardiagram2024-01-09-2} are  commutative.
\begin{equation}\label{auxiliardiagram2024-01-09-2}
\xymatrix{
 \mathring{\mathcal{A}}^{\wedge}(\vec{\varnothing},[\![1,k]\!])_{\mathfrak{S}_{k-1}} \ar^{{\mathrm{cs}}^{\nu}_{\vec{\varnothing},1,k-1}}[rr] \ar_{\mathrm{proj}^{{\mathcal A}}_{1,k-1}}[d]  & & \mathring{\mathcal{A}}^{\wedge}(\vec{\varnothing},[\![1,k]\!])_{\mathfrak{S}_{k-1}} \ar^{\mathrm{proj}^{{\mathcal A}}_{1,k-1}}[d]\\
  \mathring{\mathcal{A}}^{\wedge}_k  \ar^{{\mathrm{cs}}^{\nu}_{\vec{\varnothing},k}}[rr]  & &  \mathring{\mathcal{A}}^{\wedge}_k} 
 \quad 
\xymatrix{
 \mathring{\mathcal{A}}^{\wedge}(\vec{\varnothing},[\![1,k]\!])_{\mathfrak{S}_{k-1}} \ar^{{\mathrm{cs}}^{\nu}_{\vec{\varnothing},1,k-1}}[rr] \ar_{\mathrm{co}^{{\mathcal A}}_{k}}[d]  & & \mathring{\mathcal{A}}^{\wedge}(\vec{\varnothing},[\![1,k]\!])_{\mathfrak{S}_{k-1}} \ar^{\mathrm{co}^{{\mathcal A}}_{k}}[d]\\
  \mathring{\mathcal{A}}^{\wedge}_k  \ar^{{\mathrm{cs}}^{\nu}_{\vec{\varnothing},k}}[rr]  & &  \mathring{\mathcal{A}}^{\wedge}_k} 
\end{equation}
where the endomorphism  $\mathrm{cs}^{\nu}_{\vec{\varnothing},1,k-1}$ of $\mathring{\mathcal{A}}^{\wedge}(\vec{\varnothing},[\![1,k]\!])_{\mathfrak{S}_{k-1}}$ is defined to be the image by the  coinvariant space functor of the endomorphism $\tilde{\mathrm{cs}}^{\nu}_{\vec{\varnothing},([\![1,k]\!]}$ of $\mathring{\mathcal A}(\vec{\varnothing},[\![1,k])$ equipped with the action of $\mathfrak{S}_{k-1}$.  This implies that the diagram obtained from the diagrams in \eqref{auxiliardiagram2024-01-09-2} by replacing the vertical maps by $\mathrm{co}^{{\mathcal A}}_k-\mathrm{proj}^{{\mathcal A}}_{1,k-1}$ and $\mathrm{co}^{{\mathcal A}}_k-\mathrm{proj}^{{\mathcal A}}_{1,k-1}$ is comutative.  Since $(\mathrm{CO}\mathring{\mathcal{A}}^\wedge_k)=\mathrm{Im}\big(\mathrm{co}^{{\mathcal{A}}}_k  - \mathrm{proj} ^{{{\mathcal{A}}}}_{1,k-1}: \mathring{\mathcal{A}}^{\wedge}(\vec{\varnothing},[\![1,k]\!])_{\mathfrak{S}_{k-1}} \to \mathring{\mathcal{A}}^\wedge_k \big)$, the commutativity of the latter   implies 
$\mathrm{cs}_{\vec{\varnothing},k}^{\nu}\big((\mathrm{CO}\mathring{\mathcal A}^\wedge_k)\big) \subset (\mathrm{CO}\mathring{\mathcal A}^\wedge_k)$.  Since $\mathrm{cs}^\nu_{\vec{\varnothing},k}$ is an automorphism (see Lemma~\ref{themapcsalpha}),  this implies the equality $\mathrm{cs}_{\vec{\varnothing},k}^{\nu}\big((\mathrm{CO}\mathring{\mathcal A}^\wedge_k)\big) = (\mathrm{CO}\mathring{\mathcal A}^\wedge_k)$. The result then follows from the equalities  $\mathrm{cs}^{\nu}=\bigoplus_{k\geq 0}\mathrm{cs}^{\nu}_{\vec{\varnothing},k}$  (see Lemma~\ref{themapcsalpha}$(c)$) and $(\mathrm{CO}\mathring{\mathcal A}^\wedge)=\bigoplus_{k \geq 1}(\mathrm{CO}\mathring{\mathcal A}^\wedge_k)$  (see \S\ref{sec:3:2:1}).

\end{proof}

\subsubsection{The algebra isomorphism $\overline{\mathrm{cs}}^{\nu}:\mathfrak{s}(\mathbf{A}) \to \mathfrak{a}$}\label{sec:3:2:5}

\begin{proposition} 
The algebra automorphism  $\mathrm{cs}^\nu$ of $\bigoplus_{k \geq 0}\mathring{\mathcal{A}}^\wedge_k$ (see Equation\eqref{cs-nu-total}) induces an algebra isomorphism 
\begin{equation}\label{iso-theta}
\overline{\mathrm{cs}}^\nu: \mathfrak{s}(\mathbf{A})\longrightarrow \mathfrak{a},
\end{equation} \index[notation]{cs^\nu@$\overline{\mathrm{cs}}^\nu$} 
where $\mathfrak{s}(\mathbf{A})$ is given in~\eqref{modSA} and $\mathfrak{a}$ in~\eqref{Z2geq0-module-m}.
\end{proposition}
\begin{proof}

By Corollary~\ref{r2-20210809} and Lemma~\ref{lemma1-20210805}  the algebra automorphism $\mathrm{cs}^\nu$ of $\bigoplus_{k \geq 0}\mathring{\mathcal{A}}^\wedge_k$ satisfies $$\mathrm{cs}^\nu((\mathrm{KII}\mathring{\mathcal{A}}^\wedge) + (\mathrm{CO}\mathring{\mathcal{A}}^\wedge)) = (\mathrm{KII}^s\mathring{\mathcal{A}}^\wedge) + (\mathrm{CO}\mathring{\mathcal{A}}^\wedge).$$ It then follows from the definitions $\mathfrak{s}(\mathbf{A})$ and $\mathfrak{a}$ that  $\mathrm{cs}^\nu$ induces an algebra isomorphism $\overline{\mathrm{cs}}^\nu:\mathfrak{s}(\mathbf{A})\to\mathfrak{a}$.
\end{proof}

\begin{remark} The  algebra isomorphism \eqref{iso-theta} is compatible with the $\mathbb{Z}_{\geq 0}$-gradings of its source and target. 
\end{remark}

\subsection{The algebras \texorpdfstring{$\mathfrak{b}(n)$}{a(n)} and the morphisms \texorpdfstring{$\varphi_n : \mathfrak{a}\to \mathfrak{b}(n)$}{varphi-m-to-a(n)}}\label{sec:3:3}

Throughout  this subsection let $n\geq 1$ be an integer. In this subsection we use the notations from subsection~\ref{sec:4-1}. In particular we use the map $\mathrm{m}_C:\mathring{\mathrm{Jac}}({P},S)\to \mathbb{Z}_{\geq 0}$ from Definition~\ref{themapm}, the space $\mathring{\mathcal{A}}^\wedge({P},S)$ from Definition~\ref{thespaceAPS} (for $S\not=\emptyset$) and the pairing $\langle\ , \  \rangle$ from ~\eqref{thepairinginA}.

\subsubsection{The algebra \texorpdfstring{$\mathfrak{b}(n)$}{a(n)}}

\begin{definition}\label{thespaceL2n} For any integers $k,r\geq 0$ define the subspace $L^{<r}_k$  \index[notation]{L^{<r}_k@$L^{<r}_k$}  of $\mathring{\mathcal{A}}^\wedge_k$ by
$${L}^{<r}_k:=\widehat{\mathrm{Vect}}_{\mathbb{C}}\big\{\ [\underline{D}]\in\mathring{\mathcal{A}}^\wedge_k \ | \ \underline{D}\in\mathring{\mathrm{Jac}}(\vec{\varnothing},[\![1,k]\!]) \text{ and } \exists i\in[\![1,k]\!] \text{ with } \mathrm{m}_{i}(\underline{D})<r\ \big\}.$$
That is, $L^{<r}_k$ consists of linear combinations of classes of Jacobi diagrams in which there is at least one circle component with less than $r$ attached univalent vertices to it.  See Figure~\ref{figureL2n} for an example.

\begin{figure}[ht!]
										\centering
                        \includegraphics[scale=0.8]{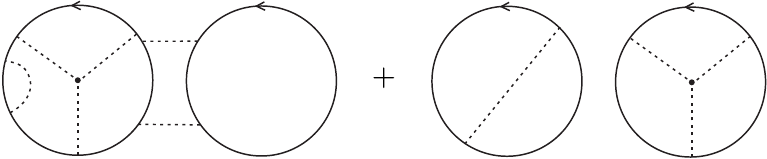}
												\caption{ An element of $L^{<3}_2$.}
\label{figureL2n} 										
\end{figure}

We set
\begin{equation}
L^{<2n}:=\bigoplus_{k\geq 0} L^{<2n}_k \subset \bigoplus_{k \geq 0}\mathring {\mathcal A}_k^\wedge.
\end{equation}  \index[notation]{L^{<2n}_k@$L^{<2n}$}
\end{definition}

Let $\Sigma$ be a finite set. Recall that we denote by  $\mathrm{FPFI}(\Sigma)$  the set of fixed-point  free involutions of $\Sigma$. This set is of cardinality $0$ if $|\Sigma|$ is odd and  $(2m-1)!/(2^{m-1}(m-1)!)$ if $|\Sigma|=2m$ for some $m \geq 1$.   There is a map 
\begin{equation}\label{equ:JacandFPFI}
\mathring{\mathrm{Jac}}:\mathrm{FPFI}(\Sigma)  \longrightarrow  \mathring{\mathrm{Jac}}(\vec{\varnothing},\emptyset,\Sigma),
\end{equation}  \index[notation]{Jac@$\mathring{\mathrm{Jac}}$}
(see Definition~\ref{generalJDwithfree}) associating to each $\tau\in\mathrm{FPFI}(\Sigma)$  (for $|\Sigma|$ even) the Jacobi diagram $\mathring{\mathrm{Jac}}(\tau)$ on $(\vec{\varnothing},\emptyset, \Sigma)$ whose vertex-oriented unitrivalent  graph consists of the disjoint union of $|\Sigma|/2$ segments each one of this having  endpoints colored by $s$ and $\tau(s)$ for $s\in \Sigma$. 
Such kind of Jacobi diagrams are called \emph{struts}. 

\begin{definition}\label{the-element-varsigma-k}
For $k \geq 1$, define $\varsigma_k\in{\mathring{\mathcal A}}^\wedge(\vec{\varnothing},\emptyset,[\![1,2k]\!])$ by $$ \varsigma_k :=\sum_{\tau \in\mathrm{FPFI}([\![1,2k]\!])}[\mathring{\mathrm{Jac}}(\tau)],$$  \index[notation]{\varsigma_k@$\varsigma_k$}
where $[ \ ]: \mathbb{C}\mathring{\mathrm{Jac}}(\vec{\varnothing},\emptyset, [\![1,2k]\!])\to {\mathring{\mathcal A}}^\wedge(\vec{\varnothing},\emptyset,[\![1,2k]\!])$ is the canonical projection.  
In Figure~\ref{figureEx1_69} we show the element $\varsigma_2$ and some of the terms of $\varsigma_3$.

\begin{figure}[ht!]
										\centering
                        \includegraphics[scale=0.8]{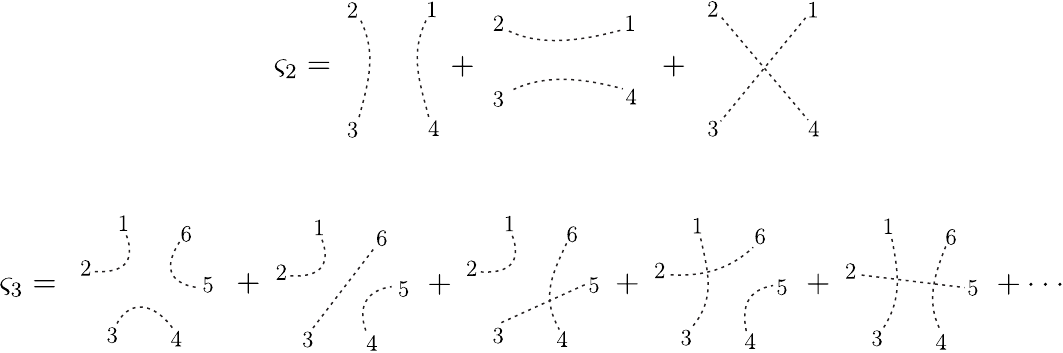}
												\caption{The elements $\varsigma_2$ and $\varsigma_3$.}
\label{figureEx1_69} 										
\end{figure}
\end{definition}

\begin{definition}\label{space-P-n-general}
Let $P$ an oriented Brauer diagram and $n \geq 0$.
\begin{itemize} 
\item[$(a)$] For any finite set $S$, define the subspace $$P^{n+1}(P,S):=\left\{\langle x, \varsigma_{n+1}\rangle\in\mathring{\mathcal{A}}^\wedge(P,S) \ | \ x\in \mathring{\mathcal{A}}^\wedge(P,S, [\![1,2(n+1)]\!]) \right\}\subset\mathring{\mathcal A}^{\wedge}(P,S),$$  \index[notation]{P^{n+1}(P,S)@$P^{n+1}(P,S)$} the pairing being as in \eqref{thepairinginA} with $\Sigma=[\![1,2(n+1)]\!]$.  See Figure~\ref{figureEx1_70} for an example.

\item[$(b)$] For any $k \geq 0$, define $P^{n+1}(P)_k$ \index[notation]{P^{n+1}(P)_k@$P^{n+1}(P)_k$} to be the image of $P^{n+1}(P,[\![1,k]\!]) \subset \mathring{\mathcal A}^\wedge(P,[\![1,k]\!])$ by the projection $\mathring{\mathcal A}^\wedge(P,[\![1,k]\!])\to \mathring{\mathcal A}^\wedge(P)_k$.  
\item[$(c)$] For $k \geq 0$, we set $P^{n+1}_k:=P^{n+1}(\vec{\varnothing})_k\subset \mathring{\mathcal A}^\wedge_k$.  \index[notation]{P^{n+1}_k@$P^{n+1}_k$}
\end{itemize}

We set
\begin{equation}
P^{n+1}:=\bigoplus_{k\geq 0}P^{n+1}_k\subset \bigoplus_{k\geq 0}\mathring{\mathcal A}^\wedge_k.
\end{equation} \index[notation]{P^{n+1}@$P^{n+1}$}

\begin{figure}[ht!]
										\centering
                        \includegraphics[scale=0.7]{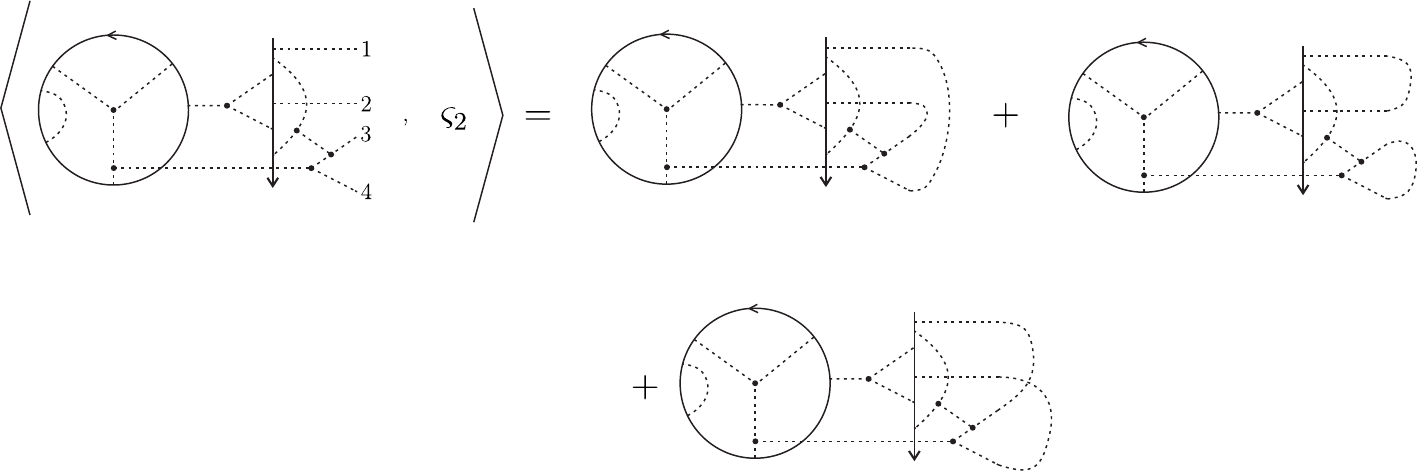}
												\caption{An element of $P^2(\downarrow, \{1\})$.}
\label{figureEx1_70} 										
\end{figure}
\end{definition}

\begin{lemma}\label{rlemma1_20211111} The subspaces $L^{<2n}$ and $P^{n+1}$ are $\mathbb{Z}_{\geq 0}$-graded ideals of $\bigoplus_{k\geq 0}\mathring{\mathcal{A}}^{\wedge}_k$.
\end{lemma}
\begin{proof}

Let $\underline{E} \in \mathring{\mathrm{Jac}}(\vec{\varnothing},[\![1,k]\!])$,  $\underline{D} \in \mathring{\mathrm{Jac}}(\vec{\varnothing},[\![1,l]\!])$ and $i \in [\![1,l]\!]$ such that $\mathrm{m}_{i}(\underline{D})<2n$, then $\mathrm{m}_{i}(\underline{D}\otimes_{\mathcal A} \underline{E})<2n$. This implies  $\mathring{\mathcal{A}}^{\wedge}_k\cdot L^{<2n}_l\subset L^{<2n}_{k+l}$.  Hence $L^{<2n}$ is an ideal of $\bigoplus_{k\geq 0}\mathring{\mathcal{A}}^{\wedge}_k$.

Let $k,l \geq 0$, one checks the identity $ \langle y\otimes_{\mathcal A} x, s\rangle = y\otimes_{\mathcal A}\langle x, s\rangle$ (identity in $\mathring{\mathcal{A}}^\wedge(\vec{\varnothing},[\![1,k+l]\!])$) for any
$y\in \mathring{\mathcal{A}}^{\wedge}(\vec{\varnothing},[\![1,l]\!])$,  $x\in\mathring{\mathcal{A}}^{\wedge}(\vec{\varnothing},[\![1, k]\!], [\![1, 2(n+1)]\!])$ and 
$s \in  \mathring{\mathcal A}^\wedge(\vec{\varnothing},\emptyset,[\![1,2(n+1)]\!])$, where $\otimes_{\mathcal A}$ is as in~\eqref{tensor:map:A:PP'gen} in the LHS and as in Notation \ref{notationfortens} in the RHS. 

It follows from the specialization of this identity to $s= \varsigma_{n+1}$ that the linear map 
$\otimes_{\mathcal A} : \mathring{\mathcal A}^\wedge(\vec{\varnothing},[\![1,k]\!]) \otimes \mathring{\mathcal A}^\wedge(\vec{\varnothing},[\![1,l]\!])\to
\mathring{\mathcal A}^\wedge(\vec{\varnothing},[\![1,k+l]\!])$
induces a map
$\mathring{\mathcal A}^\wedge(\vec{\varnothing},[\![1,k]\!]) \otimes P^{n+1}_l\to P^{n+1}_{k+l}$
between the subspaces of the source and target of this map. Taking coinvariants, one derives the inclusion  $\mathring{\mathcal{A}}^{\wedge}_k\cdot P^{n+1}_l\subset P^{n+1}_{k+l}$. Hence $P^{n+1}$ is an ideal of $\bigoplus_{k\geq 0}\mathring{\mathcal{A}}^{\wedge}_k$.
\end{proof}
By Proposition~\ref{l:LMOinvt}$(a)$, $(\mathrm{CO}\mathring{\mathcal A}^\wedge)$ is an ideal of $\bigoplus_{k\geq 0}\mathring{\mathcal{A}}^{\wedge}_k$, therefore so is $L^{<2n} + P^{n+1} + (\mathrm{CO}\mathring{\mathcal A}^\wedge)$.

\begin{definition}\label{def:space-a-n} Define the commutative algebra $\mathfrak{b}(n)$ as the quotient algebra
\begin{equation}\label{space-a-n}
\mathfrak{b}(n)=\frac{\bigoplus_{k \geq 0}\mathring{\mathcal A}^\wedge_{k}}{L^{<2n}+P^{n+1}+(\mathrm{CO}\mathring{\mathcal{A}}^{\wedge})} = \bigoplus_{k\geq 0} \frac{\mathring{\mathcal A}^\wedge_k}{L^{<2n}_k+P^{n+1}_k + (\mathrm{CO}\mathring{\mathcal{A}}_k^{\wedge})}.
\end{equation}  \index[notation]{b(n)@$\mathfrak{b}(n)$}
\end{definition}

\subsubsection{The morphism \texorpdfstring{$\varphi_n : \mathfrak{a}\to \mathfrak{b}(n)$}{varphi-m-to-a(n)}}

\begin{proposition}[See {\cite[Lemma 10.4]{Ohts}}]\label{rprop1_20211111} For any integer $k\geq 2$ we have
$$\mathring{\mathcal{A}}^\wedge(\downarrow\uparrow) _{k-2}= P^{n+1}(\downarrow\uparrow)_{k-2} + \widehat{\mathrm{Vect}}_{\mathbb{C}}\left\{[\underline{D}]\in\mathring{\mathcal{A}}^\wedge(\downarrow\uparrow)_{k-2} \ | \ \underline{D}\in\mathring{\mathrm{Jac}}(\downarrow\uparrow,[\![1,k-2]\!]) \text{ with } \mathrm{m}_{\downarrow}(\underline{D})\leq 2n\right\}.$$
\end{proposition}
\begin{proof}
Lemma 3.1 in \cite{LMO98}  (see also \cite[Lemma 10.4]{Ohts}) proves  for any finite set $S$, any oriented Brauer diagram $P$  and $C \in \pi_0(P)$, the equality
$$\mathring{\mathcal A}^\wedge(P,S)=P^{n+1}(P,S) + \widehat{\mathrm{Vect}}_{\mathbb{C}}\left\{[\underline{D}]\in\mathring{\mathcal{A}}^\wedge(P,S) \ | \ \underline{D}\in\mathring{\mathrm{Jac}}(P,S) \text{ with } \mathrm{m}_{C}(\underline{D})\leq 2n\right\}.$$
The result follows upon taking coinvariants from the specialization of this equality to 
$S=[\![1,k-2]\!]$ , $P=\downarrow \uparrow$, and $C=\downarrow \in \pi_0(\downarrow \uparrow)$. 
\end{proof}

\begin{lemma}\label{r2-20211111} For any $k\geq 2$ we have
$$\left(\mathrm{KIIMap}^{s,2}_{{\mathring{\mathcal{A}}}^\wedge,k}-\mathrm{KIIMap}^{s,1}_{{\mathring{\mathcal{A}}}^\wedge,k}\right)\big(P^{n+1}(\downarrow\uparrow)_{k-2}\big)\subset P^{n+1}_k. 
$$
\end{lemma}
\begin{proof}
For $P$ an oriented Brauer diagram, $k\geq 0$ and $\Sigma$ a finite set, define $\mathring{\mathcal A}^\wedge(P,\Sigma)_k$ to be the space $\mathring{\mathcal A}^\wedge(P,[\![1,k]\!],\Sigma)_{\mathfrak{S}_k}$. For $i=1,2$, one can define the maps
\begin{equation}
\mathrm{KIIMap}^{s,i}_{{\mathring{\mathcal{A}}}^\wedge,k,[\![1,2(n+1)]\!]}: \mathring{\mathcal{A}}^\wedge(\downarrow\uparrow,[\![1,2(n+1)]\!])_{k-2}\longrightarrow \mathring{\mathcal{A}}^\wedge(\vec{\varnothing}, [\![1,2(n+1)]\!])_{k}
\end{equation}
by the same formulas as those given in \eqref{KIIMap1forA-2} and \eqref{KIIMap2forA-2} in Lemma~\ref{themapsKIIsimple}. By definition of these maps the diagram
\begin{equation*}
\xymatrix{\mathring{\mathcal{A}}^\wedge(\downarrow\uparrow,[\![1,2(n+1)]\!])_{k-2}  \ar_{\mathrm{KIIMap}^{s,i}_{{\mathring{\mathcal{A}}}^\wedge,k,[\![1,2(n+1)]\!]}}[dd]
\ar^{\ \ \ \ \ \ \langle \ \cdot \ , \varsigma_{n+1} \rangle}[rrr] & & &\mathring{\mathcal{A}}^\wedge(\downarrow\uparrow)_{k-2}\ar^{\mathrm{KIIMap}^{s,i}_{{\mathring{\mathcal{A}}}^\wedge,k}}[dd]\\
 & & & \\
\mathring{\mathcal{A}}^\wedge(\vec{\varnothing}, [\![1,2(n+1)]\!])_k\ar^{\ \ \ \ \ \ \langle \ \cdot \ , \varsigma_{n+1} \rangle}[rrr] & & &\mathring{\mathcal{A}}^\wedge_k}
\end{equation*}
commutes. Then the result follows. 
\end{proof}

\begin{lemma}\label{r20211117} Let $k\geq 2$ and $\underline D=\big(D,\varphi,\{\mathrm{lin}_l\}_{l \in \pi_0(\downarrow \uparrow)},\{\mathrm{cyc}_s\}_{s \in [\![1,k-2]\!]}\big) \in\mathring{\mathrm{Jac}}(\downarrow\uparrow, [\![1,k-2]\!])$. Then
\begin{itemize}
\item[$(a)$] There exist 
$\big(\tilde\varphi,\{\widetilde{\mathrm{cyc}}_s\}_{s \in [\![1,k]\!]}\big)$ and 
$\big(\varphi_i,\{{\mathrm{cyc}}_s^i\}_{s \in[\![1,k]\!]}\big)_i$, indexed by the set of maps $i : \varphi^{-1}(\downarrow) \to \{1,2\}$, where $\tilde{\varphi},\varphi_i : \partial D\to [\![1,k]\!]$ are maps and for each $s \in [\![1,k]\!]$,  $\widetilde{\mathrm{cyc}}_s$ and ${\mathrm{cyc}}_s^i$ are cyclic orders on $\tilde{\varphi}^{-1}(s)$ and $\varphi_i^{-1}(s)$ respectively;  such that 
$$\mathrm{KIIMap}_{{\mathring{\mathcal A}}^\wedge,k}^{s,1}([\underline{D}]) = [\underline{\tilde{D}}] \quad \quad \text{and} \quad \quad \mathrm{KIIMap}_{{\mathring{\mathcal A}}^\wedge,k}^{s,2}([\underline{D}])= \sum_i[\underline{D}_i],$$ 
 where $\tilde{\underline D}$ and $\tilde{\underline D}_i$ are the elements of $\mathring{\mathrm{Jac}}(\vec{\varnothing}, [\![1,k]\!])$ defined by $\tilde{\underline{D}}:=\big(D,\tilde{\varphi},\{\widetilde{\mathrm{cyc}}_s\}_{s \in [\![1,k]\!]}\big)$ and $\underline{D}_i:=\big(D,\varphi_i,\{\widetilde{\mathrm{cyc}}_s^i\}_{s \in[\![1,k]\!]}\big)$.

\item[$(b)$] Let $i_{1}:\varphi^{-1}(\downarrow) \to \{1,2\} $ be the constant map with value $1$. Then $$\big(\varphi_{i_1},\{\mathrm{cyc}_s^{i_1}\}_{s \in [\![1,k]\!]}\big)=\big(\tilde{\varphi},\{\widetilde{\mathrm{cyc}}_s\}_{s \in [\![1,k]\!]}\big).$$

\item[$(c)$] For any map $i:\varphi^{-1}(1)\to \{1,2\}$   with $i\not = i_1$, we have $\mathrm{m}_{1}(\underline{D}_i) < \mathrm{m}_{\downarrow}(\underline{D})$  (in these expressions, the indices 1 and $\downarrow$ respectively refer to elements of $[\![1,k]\!]$ and of $\pi_0(\downarrow \uparrow)$).

\item[$(d)$] $$\big(\mathrm{KIIMap}_{{\mathring{\mathcal A}}^\wedge,k}^{s,2} -\mathrm{KIIMap}_{{\mathring{\mathcal A}}^\wedge,k}^{s,1}\big)([\underline{D}])\in L^{<\mathrm{m}_{\downarrow}(\underline{D})}_k.$$

\end{itemize}
\end{lemma}
\begin{proof}

$(a)$ It follows from the definition  of the maps $\mathrm{KIIMap}_{\mathring{\mathcal A}^\wedge,k}^{s,j}$ for $j=1,2$, (see Lemma 4.5) that the first equality holds with $\tilde{\varphi}$ being the composition of $\varphi$ with the bijection $\pi_0(\downarrow \uparrow) \sqcup [\![1,k-2]\!]\to [\![1,k]\!]$ given by $\downarrow \mapsto 1$, $\uparrow \mapsto 2$, $s \mapsto s+2$ for $s \in [\![1,k-2]\!]$, and for $s \in [\![3,k]\!]$, $\widetilde{\mathrm{cyc}}_s:=\mathrm{cyc}_{s-2}$ and for $s \in \{1,2\}$, $\widetilde{\mathrm{cyc}}_s$ is the cyclic order derived from the linear order $\mathrm{lin}_{s'}$ where $s'$ is the preimage of $s$ under the bijection $\pi_0(\downarrow\uparrow)\simeq \{1,2\}$ given by $\downarrow\mapsto 1$ and $\uparrow\mapsto 2$, as in Lemma~\ref{r:2023-01-11-r2}$(b)$  (with $I=\{\downarrow\}$ or $I=\{\uparrow\}$). The second equality holds if for each map $i : \varphi^{-1}(1) \to \{1,2\}$, the map $\varphi_i : \partial D \to [\![1,k]\!]$ is defined by the conditions that $(\varphi_i)_{|\varphi^{-1}(\downarrow)}=i$ and $(\varphi_i)_{|\partial D\setminus \varphi^{-1}(\downarrow)}$ is the composition of $\varphi_{|\partial D\setminus \varphi^{-1}(\downarrow)} : \partial D\setminus \varphi^{-1}(\downarrow) \to \{\uparrow\} \sqcup [\![1,k-2]\!]$ with the map $\{\uparrow\} \sqcup [\![1,k-2]\!] \to [\![1,k]\!]$ given by $\uparrow \mapsto 2$ and $s \mapsto s+2$ for $s \in [\![1,k-2]\!]$, and the cyclic orders are defined by $\mathrm{cyc}^i_s:=\mathrm{cyc}_{s-2}$ for $s \in [\![3,k]\!]$ and  $\mathrm{cyc}^i_1$ is the cyclic order derived from the restriction of the linear order $\mathrm{lin}_{\downarrow}$ to $i^{-1}(1)$ as in Lemma~\ref{r:2023-01-11-r2}$(b)$  (with $I=\{\downarrow\}$), and $\mathrm{cyc}^i_2$  is the cyclic order obtained from the linear orders $\mathrm{lin}_{\uparrow}$ and the restriction of the linear order $\mathrm{lin}_{\downarrow}$ to $i^{-1}(2)$ as in  Lemma~\ref{r:2023-01-11-r2}$(b)$ (with $I=\{\downarrow,\uparrow\}$ equipped with its unique cyclic order).

$(b)$ Follows directly from the definition of $\big(\varphi_{i_1},\{\mathrm{cyc}_s^{i_1}\}_{s \in [\![1,k]\!]}\big)$ and $\big(\tilde{\varphi},\{\widetilde{\mathrm{cyc}}_s\}_{s \in [\![1,k]\!]}\big)$ in $(a)$.

$(c)$ By the definitions given in $(a)$, for any map $i:\varphi^{-1}(\downarrow)\to \{1,2\}$ we have 
\begin{equation}\label{eq:equalitiesofm}
\mathrm{m}_1(\underline{D}_i) + \mathrm{m}_2(\underline{D}_i) = \mathrm{m}_{\downarrow}(\underline{D}) + \mathrm{m}_{\uparrow}(\underline{D}) \qquad \text{and} \qquad \mathrm{m}_2(\underline{D}_i) = |i^{-1}(2)| +  \mathrm{m}_{\uparrow}(\underline{D}).
\end{equation}
If $i\not= i_1$, then $|i^{-1}(2)| > 0$ and therefore, by \eqref{eq:equalitiesofm}, we get $\mathrm{m}_{1}(\underline{D}_i) < \mathrm{m}_{\downarrow}(\underline{D})$.

$(d)$ By $(a)$ and $(b)$ we have
$$\Big(\mathrm{KIIMap}_{{\mathring{\mathcal A}}^\wedge,k}^{s,2} -\mathrm{KIIMap}_{{\mathring{\mathcal A}}^\wedge,k}^{s,1}\Big)([\underline{D}]) = \sum_{i\not = i_1}[\underline{D}_i]$$
and by $(c)$ each  $[\underline{D}_i]$ with $i\not=i_1$ belongs to $L^{<{\mathrm m}_{\downarrow}(\underline{D})}$, which finishes the proof.
\end{proof}

\begin{proposition}[See {\cite[Prop. 10.5]{Ohts}}]\label{r220211117} For any $k\geq 2$ we have
\begin{equation}
(\mathrm{KII}^s\mathring{\mathcal{A}}^\wedge_k) \subset L^{<2n}_k + P^{n+1}_k.
\end{equation}
Therefore $(\mathrm{KII}^s\mathring{\mathcal{A}}^\wedge) \subset L^{<2n} + P^{n+1}.$
\end{proposition}
\begin{proof}
Recall from Definition~\ref{spaceKIIsimpleforA} that 
\begin{equation*}
(\mathrm{KII}^s{\mathring{\mathcal A}}^\wedge_k) =\mathrm{Im}\Big(\mathrm{KIIMap}_{{\mathring{\mathcal A}}^\wedge,k}^{s,2} - \mathrm{KIIMap}_{{\mathring{\mathcal A}}^\wedge,k}^{s,1}\Big).
\end{equation*}

\noindent Besides
\begin{gather*}
\begin{split}
\mathrm{Im}&\Big(\mathrm{KIIMap}_{{\mathring{\mathcal A}}^\wedge,k}^{s,2} - \mathrm{KIIMap}_{{\mathring{\mathcal A}}^\wedge,k}^{s,1}\Big)  = \Big(\mathrm{KIIMap}_{{\mathring{\mathcal A}}^\wedge,k}^{s,2} - \mathrm{KIIMap}_{{\mathring{\mathcal A}}^\wedge,k}^{s,1}\Big)(\mathring{\mathcal{A}}^\wedge(\downarrow\uparrow)_{k-2}) \\
& = \Big(\mathrm{KIIMap}_{{\mathring{\mathcal A}}^\wedge,k}^{s,2} - \mathrm{KIIMap}_{{\mathring{\mathcal A}}^\wedge,k}^{s,1}\Big)(P^{n+1}(\downarrow\uparrow)_{k-2})\\
&\ \ \ \  + \widehat{\mathrm{Vect}}_{\mathbb{C}}\left\{\Big(\mathrm{KIIMap}_{{\mathring{\mathcal A}}^\wedge,k}^{s,2} - \mathrm{KIIMap}_{{\mathring{\mathcal A}}^\wedge,k}^{s,1}\Big)([\underline{D}])\ | \ \underline{D}\in\mathring{\mathrm{Jac}}(\downarrow\uparrow, [\![1,k-2]\!]) \text{ with } \mathrm{m}_{\downarrow}(\underline{D})\leq 2n\right\} \\
& \ \ \ \ \subset P^{n+1}_k + L^{<2n}_k.
\end{split}
\end{gather*}
We have used Proposition~\ref{rprop1_20211111} in the second equality; and Lemma~\ref{r2-20211111} and Lemma~\ref{r20211117}$(d)$ in the last inclusion. Therefore, we conclude $(\mathrm{KII}^s\mathring{\mathcal{A}}^\wedge) \subset L^{<2n} + P^{n+1}.$
\end{proof}

\begin{corollary} The identity map of $\bigoplus_{k\geq 0} \mathring{\mathcal A}_k^\wedge$ induces an algebra morphism $\varphi_n : \mathfrak{a}\to \mathfrak{b}(n)$, where $\mathfrak{a}$ is as in~\eqref{Z2geq0-module-m} and $\mathfrak{b}(n)$ as in~\eqref{space-a-n}.
\end{corollary}

\begin{proof}
The result follows directly from Proposition~\ref{r220211117}. 
\end{proof}

\begin{remark} The algebra morphism~$\varphi_n$ is compatible with the $\mathbb{Z}_{\geq 0}$-gradings of its source and target. 
\end{remark}

\subsection{The spaces of diagrammatic traces and the LMO diagrammatic trace}\label{sec:tracesandLMOtrace}

\subsubsection{The space of diagrammatic traces}
For a finite group $G$, a pair $(P,S)$ of an oriented Brauer diagram $P$ and a finite set~$S$, and a  finite $G$-set $\Sigma$, there is  an action of $G$ on $\mathring{\mathrm{Jac}}(P,S,\Sigma)$, $\mathbb{C}\mathring{\mathrm{Jac}}(P,S,\Sigma)$,  $\mathring{\mathcal A}(P,S,\Sigma)$ and  $\mathring{\mathcal A}^\wedge(P,S,\Sigma)$ induced by the action of $G$ on $\Sigma$. Recall that for $m\geq 0$ we denote by $\mathbb{Z}_m$  the cyclic group of order $m$ (this is $\emptyset$ for $m=0$). In particular, for any integer  $n \geq 1$ we have actions of $\mathfrak{S}_n$ (using the identification $\mathbb{Z}_n \simeq [\![0,n-1]\!]$) and~$\mathbb{Z}_n$ (by addition) on   $\mathring{\mathcal A}(\vec{\varnothing},\emptyset,\mathbb{Z}_n)$.

\begin{definition}\label{def:diagrtrace}  A \emph{diagrammatic trace} is an element $T_{\bullet}:=(T_n)_{n \geq 0} \in\prod_{n \geq 0} \mathring{\mathcal A}(\vec{\varnothing},\emptyset,\mathbb{Z}_n)$ satisfying:
\begin{enumerate}
\item $T_0=T_1= 0$,
\item for $n \geq 2$, the element $T_n$ is invariant under the action of $\mathbb{Z}_n$, i.e. $(1\cdots n)T_n=T_n$,
\item (\emph{STU-compatibility}) $T_m - (i \ i+1)T_{m} = T_{m-1}*_{i} \Yup$ for any {$m\geq 3$} and $i\in\mathbb{Z}_m$, 
\end{enumerate}
where $T_{m-1} \mapsto T_{m-1}*_{i} \Yup$  is the linear map $\mathring{\mathcal A}(\vec{\varnothing},\emptyset,\mathbb{Z}_{m-1}) \to \mathring{\mathcal A}(\vec{\varnothing},\emptyset,\mathbb{Z}_m)$  derived from the map 
$$\mathring{\mathrm{Jac}}(\vec{\varnothing}, \emptyset, \mathbb{Z}_{m-1}) \longrightarrow \mathring{\mathrm{Jac}}(\vec{\varnothing}, \emptyset, \mathbb{Z}_{m})$$ which takes $\underline{D}\in\mathring{\mathrm{Jac}}(\vec{\varnothing},\emptyset,\mathbb{Z}_{n-1})$ to the Jacobi diagram $\underline{D}*_{i} \Yup$  obtained from $\underline{D}$ by gluing the $\Yup$-shaped graph to the univalent vertex  labelled $i$ and re-labelling the vertices so that the remaining vertices of the $\Yup$-shaped graph become $i$ and $i+1$ with the cyclic order in the new trivalent vertex given by $[a<i+1 <i]$, where $a$ is the edge connecting the $\Yup$-shaped graph with $\underline{D}$, see Figure~\ref{fig:STU-compatibility}.

Denote by $\mathbb{D}\mathrm{iag}\mathbb{T}\mathrm{r} \subset \prod_{n \geq 0}\mathring{\mathcal A}(\vec{\varnothing}, \emptyset, \mathbb{Z}_n)$  \index[notation]{D\mathrm{iag}\mathbb{T}\mathrm{r}@$\mathbb{D}\mathrm{iag}\mathbb{T}\mathrm{r} $}  the subset of all diagrammatic traces. 
\end{definition}

\begin{figure}[ht]
		\centering
         \includegraphics[scale=1]{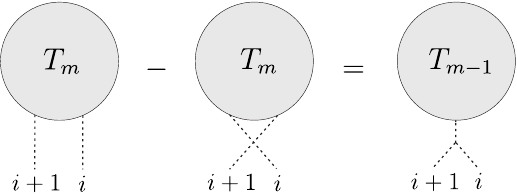}
\caption{Schematic description of the STU-compatibility condition of a diagrammatic  trace $T_{\bullet}=(T_n)_{n\geq 0}$.}
\label{fig:STU-compatibility} 								
\end{figure}

\begin{lemma}\label{def:bidegree}
Equip $\prod_{n \geq 0} \mathring{\mathcal A}(\vec{\varnothing},\emptyset,\mathbb{Z}_n)$ with the bigrading: For $c,d \in\mathbb{Z}_{\geq 0}$
 $$\big(\prod_{n \geq 0} \mathring{\mathcal A}(\vec{\varnothing},\emptyset,\mathbb{Z}_n)\big)[c,d]:=
\prod_{n \geq 0} \mathring{\mathcal A}(\vec{\varnothing},\emptyset,\mathbb{Z}_n)[c,d]$$ induced by $(|\pi_0|,\mathrm{deg}_\mathrm{cyc})$ (see  \eqref{degcircles}). Then $\mathbb{D}\mathrm{iag}\mathbb{T}\mathrm{r}$ is a bigraded subspace of $\prod_{n \geq 0}\mathring{\mathcal A}(\vec{\varnothing},\emptyset,\mathbb{Z}_n)$. We denote by $\mathbb{D}\mathrm{iag}\mathbb{T}\mathrm{r}[c,d]$ the space of diagrammatic traces of bidegree $(c,d)$. 
\end{lemma}

\begin{proof}
The self-maps of $\mathring{\mathcal A}(\vec{\varnothing},\emptyset,\mathbb{Z}_m)$ defined by  $T \mapsto (1\cdots n)  T$ and $T \mapsto (i,i+1) T$, as well as the map $\mathring{\mathcal A}(\vec{\varnothing},\emptyset,\mathbb{Z}_{m-1})\to \mathring{\mathcal A}(\vec{\varnothing},\emptyset,\mathbb{Z}_m)$ defined by $T \mapsto T*_{i} \Yup$ are compatible with the bidegrees. It follows that $\mathbb{D}\mathrm{iag}\mathbb{T}\mathrm{r}$ is the kernel of a linear map of bidegree $(0,0)$ with source $\prod_{n \geq 0}\mathring{\mathcal A}(\vec{\varnothing},\emptyset,\mathbb{Z}_n)$ and with target $\mathring{\mathcal A}(\vec{\varnothing},\emptyset,\mathbb{Z}_0) \times 
\mathring{\mathcal A}(\vec{\varnothing},\emptyset,\mathbb{Z}_1) \times \prod_{n \geq 2} \mathring{\mathcal A}(\vec{\varnothing},\emptyset,\mathbb{Z}_n) \times \prod_{n \geq 2} \mathring{\mathcal A}(\vec{\varnothing},\emptyset,\mathbb{Z}_n)$, which implies the result. 
\end{proof}

One checks that if $T_{\bullet}=(T_n)_{n \geq 0} \in \mathbb{D}\mathrm{iag}\mathbb{T}\mathrm{r}[c,d]$, then $\mathrm{deg}(T_n)=n+d-c$ for any $n \geq 2$.

\begin{lemma}[{\cite[Proposition 2.3]{LMO98}}]\label{def:productoftraces} Let $T_{\bullet}=(T_n)_{n \geq 0}$ and $T'_{\bullet}=(T'_n)_{n \geq 0}$ be diagrammatic traces. For $n \geq 0$, set $T''_n:=\sum_{p,q \ | \ p+q=n} \sum_{\sigma \in \mathfrak{S}_{p,q}}
\sigma_*(T_p \otimes_{\mathcal A} T'_q)$ where $\otimes_{\mathcal A}$ is as in \eqref{tensor:map:A:PP'gen}, $\mathfrak{S}_{p,q}$ is the set of bijections $\mathbb{Z}_p \sqcup \mathbb{Z}_q \to \mathbb{Z}_{p+q}$ whose restrictions to $\mathbb{Z}_p$ and $\mathbb{Z}_p$ are both increasing (using the identifications  of $\mathbb{Z}_p$, $\mathbb{Z}_q$ and $\mathbb{Z}_{p+q}$ with $[\![0,p-1]\!]$, $[\![0,q-1]\!]$ and $[\![0,p+q-1]\!]$ respectively) and $\sigma_*$ is defined as in \S\ref{sec:310-2}. Then $T_{\bullet}*T'_{\bullet}:=(T''_n)_{n \geq 0}$ is a diagrammatic trace.
\end{lemma}

\begin{lemma} Equipped with the product from Lemma~\ref{def:productoftraces}, and with the bidegree from Lemma~\ref{def:bidegree}, the vector space $\mathbb{D}\mathrm{iag}\mathbb{T}\mathrm{r}$ as a $\mathbb{Z}_{\geq 0}^2$-graded commutative  algebra.
\end{lemma}
\begin{proof} The result is a direct verification.
\end{proof}

The commutative algebra $\mathring{\mathcal A}(\vec{\varnothing},\emptyset,\emptyset)$ is bigraded by $(|\pi_0|, (\mathrm{deg}_{\mathrm{cyc}})$. It acts on $\mathbb{D}\mathrm{iag}\mathbb{T}\mathrm{r}$ by $a*(T_n)_{n \geq 0}:=(a \otimes_{\mathcal A} T_n)_{n \geq 0}$. Then $\mathbb{D}\mathrm{iag}\mathbb{T}\mathrm{r}$ is naturally a bigraded  $\mathring{\mathcal A}(\vec{\varnothing},\emptyset,\emptyset)$-bimodule, and the algebra structure of $\mathbb{D}\mathrm{iag}\mathbb{T}\mathrm{r}$ gives rise to a structure of algebra of $\mathbb{D}\mathrm{iag}\mathbb{T}\mathrm{r}$ in the category of $\mathring{\mathcal A}(\vec{\varnothing},\emptyset,\emptyset)$-bimodules.

\begin{remark}\label{exampleUU}
Consider the family $U_{\bullet}:={(U_n)}_{n\geq 0} \in \prod_{n\geq 0}\mathring{\mathcal A}(\vec{\varnothing}, \emptyset, \mathbb{Z}_n)$, where $U_0 = U_1 = 0$ and for $n\geq 2$, $U_n$ consists of a dashed loop with $n$ legs attached to it, see Figure~\ref{elements-U-m} for the schematic representation of $U_8$. One checks that $U_{
\bullet}\in\mathbb{D}\mathrm{iag}\mathbb{T}\mathrm{r}[1,1]$ .

\begin{figure}[ht]
		\centering
         \includegraphics[scale=0.8]{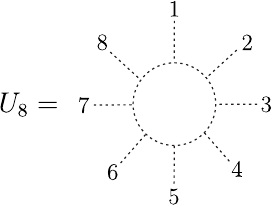}
\caption{The element $U_8\in\mathring{\mathcal A}(\vec{\varnothing}, \emptyset, \mathbb{Z}_8)$.}
\label{elements-U-m} 								
\end{figure}

\end{remark}

\subsubsection{The LMO diagrammatic trace}

\begin{lemma}[{\cite[Proposition 2.2]{LMO98}}]\label{r:elementsTmLMO} The space $\mathbb{D}\mathrm{iag}\mathbb{T}\mathrm{r}[1,0]$ is $1$-dimensional. A generator $T^{\mathrm{LMO}}_{\bullet}=(T_n^{\mathrm{LMO}})_{n\geq 0}$  \index[notation]{T^{\mathrm{LMO}}_{\bullet}@$T^{\mathrm{LMO}}_{\bullet}$} is such that $T_n^{\mathrm{LMO}}$ is given by Figure~\ref{elements-T-m} for $n=2,3,4$.  For any $n \geq2$, $T_n^{\text{LMO}}$ is a sum of connected {tree} diagrams, hence $\mathrm{deg}(T_n^{\mathrm{LMO}})=n-1$ for $\mathrm{deg}$ given by \eqref{degree-on-APS}. We call $T^{\mathrm{LMO}}$ the \emph{LMO diagrammatic trace}. 
\end{lemma}

\begin{figure}[ht]
		\centering
         \includegraphics[scale=0.8]{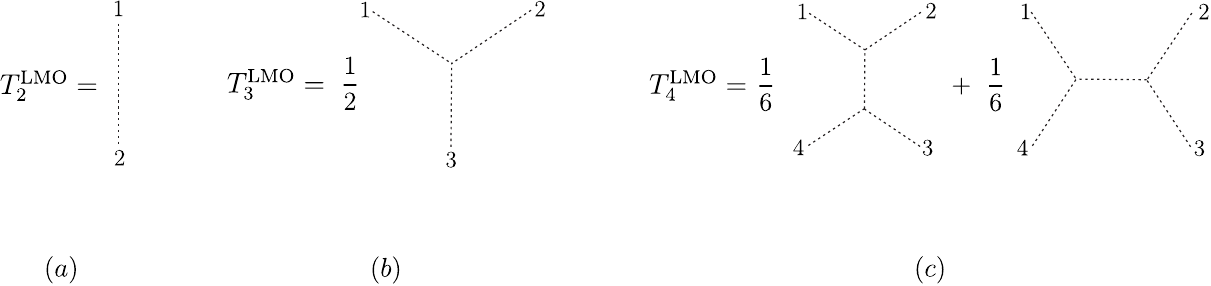}
\caption{The elements $T_2^{\mathrm{LMO}}\in\mathring{\mathcal A}(\vec{\varnothing}, \emptyset, \mathbb{Z}_2)$, $T_3^{\mathrm{LMO}}\in \mathring{\mathcal A}(\vec{\varnothing}, \emptyset, \mathbb{Z}_3)$ and  $T_4^{\mathrm{LMO}}\in\mathring{\mathcal A}(\vec{\varnothing}, \emptyset, \mathbb{Z}_4)$.}
\label{elements-T-m} 								
\end{figure}

\subsubsection{Lie algebras and the space of diagrammatic traces}

\begin{definition} Recall that a trace on an associative algebra $A$ is a linear map $T : A\to \mathbb{C}$ such that $T(ab) = T(ba)$ for any $a,b\in A$. Denote by $\mathbb{T}(A)$  \index[notation]{T(A)@$\mathbb{T}(A)$} the set of traces of $A$; this is a vector subspace of $A^*$.
\end{definition}

 \begin{lemma}\label{productindiagrtraces} Let $\mathfrak{g}$ be a Lie algebra. Let $U\mathfrak{g}$ denote the enveloping algebra of $\mathfrak{g}$ and $\Delta: U\mathfrak{g}\to U\mathfrak{g}\otimes U\mathfrak{g}$ the canonical coproduct. We use the notation $\Delta(x)=\sum x^{(1)}\otimes x^{(2)}$.
\begin{itemize}
\item[$(a)$]  Let $\tau_1,\tau_2\in \mathbb{T}(U\mathfrak{g})$. The map $\tau_1\otimes \tau_2: U\mathfrak{g}\to\mathbb{C}$ defined by $(\tau_1*\tau_2)(x) = \sum \tau_1(x^{(1)})\otimes \tau_2(x^{(2)})$  for any $x\in U\mathfrak{g}$, is a trace of $U\mathfrak{g}$.

\item[$(b)$] The operation in $(a)$ endows $\mathbb{T}(U\mathfrak{g})$ with a structure of commutative algebra. 

\end{itemize}

\end{lemma}

\begin{proof}
The result is a direct verification.
\end{proof}

Similar to the definition of a weight system, see \cite[Chapter 14]{JacksonMoffat},  any pair $(\mathfrak{g},\langle\ , \ \rangle_{\mathfrak{g}})$, where $\mathfrak{g}$ is a finite dimensional Lie algebra  and $\langle\ ,\  \rangle_{\mathfrak{g}}$ is an non-degenerated invariant pairing on $\mathfrak{g}$, gives rise to a collection of linear maps 
  \begin{equation}\label{eq:mapsAtoU}
 \mathring{\mathcal A}(\mathrm{Id}_{[\![1,m]\!]},\emptyset, [\![1,n]\!]) \longrightarrow \mathrm{Hom}(\mathfrak{g}^{\otimes n},U\mathfrak{g}^{\otimes m}) 
 \end{equation}
 which moreover intertwine the composition 
$$\circ_{{\mathcal A}} :  \mathring{\mathcal A}(\mathrm{Id}_{[\![1,m]\!]},\emptyset,[\![1,n]\!]) \otimes  \mathring{\mathcal A}(\mathrm{Id}_{[\![1,m]\!]},\emptyset,[\![1,n']\!])\longrightarrow \mathring{\mathcal A}(\mathrm{Id}_{[\![1,m]\!]},\emptyset,[\![1,n+n']\!])$$
(see\eqref{COMP:MAP:SIGMA}) with the map induced by the product in  $U\mathfrak{g}^{\otimes m}$ and the concatenation $\mathfrak{g}^{\otimes n} \otimes \mathfrak{g}^{\otimes n'}\to \mathfrak{g}^{\otimes n+n'}$; the tensor product 
$$\otimes_{{\mathcal A}}:  \mathring{\mathcal A}(\mathrm{Id}_{[\![1,m]\!]},\emptyset,[\![1,n]\!]) \otimes  \mathring{\mathcal A}(\mathrm{Id}_{[\![1,m']\!]},\emptyset,[\![1,n']\!])\longrightarrow \mathring{\mathcal A}(\mathrm{Id}_{[\![1,m+m']\!]},\emptyset,[\![1,n+n']\!])$$ 
(see \eqref{tensor:map:A:PP'gen}) with the operation defined by the concatenations  $\mathfrak{g}^{\otimes n} \otimes \mathfrak{g}^{\otimes n'}\to \mathfrak{g}^{\otimes n+n'}$   and  $U\mathfrak{g}^{\otimes m} \otimes U\mathfrak{g}^{\otimes m'}\to U\mathfrak{g}^{\otimes m+m'}$; and the connecting operation  $\mathring{\mathcal A}(\mathrm{Id}_{[\![1,m]\!]},\emptyset,[\![1,n]\!]) \otimes  \mathring{\mathcal A}(\mathrm{Id}_{[\![1,p]\!]},\emptyset,[\![1,n]\!]) \to \mathring{\mathcal A}(\mathrm{Id}_{[\![1,m+p]\!]},\emptyset, \emptyset)$ (see \eqref{thepairinginA}) with the map arising from the 
concatenation of $U\mathfrak{g}^{\otimes m}$ and $U\mathfrak{g}^{\otimes p}$ and the composition of the $n$-th power of the element of $\mathfrak{g} \otimes \mathfrak{g}$ arising from the pairing $\langle\ , \ \rangle_{\mathfrak{g}}$. 
 
For $n \geq 0$, the bijection $\mathbb{Z}_n \simeq [\![0,n-1]\!]$ and the map \eqref{eq:mapsAtoU} give rise to a linear map  $\mathring{\mathcal A}(\vec{\varnothing},\emptyset,\mathbb{Z}_n) \to \mathrm{Hom}(\mathfrak{g}^{\otimes n},\mathbb{C})$. 
The product over $n \geq 0$ of these maps is a linear map 
\begin{equation}\label{eq:mapev}
\mathrm{ev}_{\mathfrak{g}} : \prod_{n \geq 0}\mathring{\mathcal A}(\vec{\varnothing},\emptyset, \mathbb{Z}_n) \longrightarrow \mathrm{Hom}(T\mathfrak{g},\mathbb{C}),
\end{equation}
 where $T\mathfrak{g}$ is the tensor algebra of $\mathfrak{g}$.

\begin{lemma}
The image of $\mathbb{D}\mathrm{iag}\mathbb{T}\mathrm{r}$ by the map \eqref{eq:mapev}  is contained in  $\mathbb{T}(U\mathfrak{g})$. The resulting map $\mathbb{D}\mathrm{iag}\mathbb{T}\mathrm{r}\to \mathbb{T}(U\mathfrak{g})$ is an algebra morphism.
\end{lemma}

\begin{proof}
Recall that $\mathbb{T}(U\mathfrak{g}) \subset (U\mathfrak{g})^* \subset (T\mathfrak{g})^*$. Let $T_{\bullet} \in \mathbb{D}\mathrm{iag}\mathbb{T}\mathrm{r}$, by the condition $(3)$ in Definition~\ref{def:diagrtrace}, the map $\mathrm{ev}_{\mathfrak{g}}(T_{\bullet})$ vanishes  on the two-sided ideal of $T\mathfrak{g}$ generated by the elements
$x \otimes y - y \otimes x -[x,y]$ for $x,y \in \mathfrak{g}$. It follows that $\mathrm{ev}_{\mathfrak{g}}(T_{\bullet}) \in U\mathfrak{g}^*$. The relation $(2)$  in Definition~\ref{def:diagrtrace} then implies that $\mathrm{ev}_{\mathfrak{g}}(T_{\bullet}) \in \mathbb{T}(U\mathfrak{g})$. The algebra morphism property follows from the definition of the product structure on $\mathbb{T}(U\mathfrak{g})$ (see Lemma~\ref{productindiagrtraces}$(a)$) and the expression of the coproduct of a pure tensor in $T\mathfrak{g}$ based on pure shuffles, namely  $\Delta(x_1\cdots x_n)=\sum_{a=0}^n \sum_{\sigma \in \mathfrak{S}_{a,n-a}}x_{\sigma(1)}\cdots x_{\sigma(a)} \otimes x_{\sigma(a+1)}\cdots x_{\sigma(n)}$, where the sets $\mathfrak{S}_{a,n-a}$ are as in Lemma~\ref{def:productoftraces}. 
\end{proof}

\begin{remark} One checks that the image of the element $U_{\bullet}$ from Remark~\ref{exampleUU}  is the trace on $U\mathfrak{g}$ induced by the adjoint representation  of $U\mathfrak{g}$ on $\mathfrak{g}$, namely 
$\mathrm{ev}_{\mathfrak{g}}(U_{\bullet})=\mathrm{tr}_{\mathfrak{g}}$, where $\mathrm{tr}_{\mathfrak{g}} \in \mathbb{T}(U\mathfrak{g})$ is given by $U\mathfrak{g}\ni a\mapsto \mathrm{tr}_{\mathfrak{g}}(\mathrm{ad}_a)$, and $U\mathfrak{g}\ni a \mapsto \mathrm{ad}_a$ is the algebra morphism $U\mathfrak{g} \to \mathrm{End}(\mathfrak{g})$ whose restriction  to $\mathfrak{g}$ is the map  $x\mapsto  \mathrm{ad}_x=(y\mapsto [x,y])$ for any $x,y\in\mathfrak{g}$.  
\end{remark}

\subsection{The algebras \texorpdfstring{$\mathfrak{c}(n)$}{c(n)} and the morphisms \texorpdfstring{$\overline{j}_n:\mathfrak{b}(n) \to \mathfrak{c}(n)$}{j-n}}\label{sec:3:4}
Throughout this subsection let $n \geq 1$ be an integer.

\subsubsection{Gradings on $\bigoplus_{k \geq 0} \mathring{\mathcal A}^{\wedge}_k$ and $\mathring{\mathcal A}^{\wedge}_0$}

Recall that for any pair $(P,S)$ of an oriented Brauer diagram $P$ and a finite set $S$, the space $\mathring{\mathcal{A}}(P,S)$ is equipped with the grading $\mathrm{deg}$ given in~\eqref{degree-on-APS}, for which the composition (see \S\ref{sec:4-3}) and tensor product (see \S\ref{sec:4-2}) operations are homogeneous. These operations coincide in the case of a pair of the form $(\vec{\varnothing},[\![1,k]\!])$ for $k \geq 0$. They induce graded maps $\mathring{\mathcal{A}}_k \otimes \mathring{\mathcal{A}}_l \to \mathring{\mathcal{A}}_{k+l}$ for $k,l\geq 0$, which define an algebra structure on the  vector space $\bigoplus_{k\geq 0} \mathring{\mathcal{A}}_k$. This algebra is equipped with a bigrading $(\mathrm{deg}, \mathrm{deg}_{\mathrm{cir}})$  \index[notation]{deg_cir@$\mathrm{deg}_{\mathrm{cir}}$} where $\mathring{\mathcal{A}}_k$ is the homogeneous component of degree $k$ for $\mathrm{deg}_{\mathrm{cir}}$. The algebra $\bigoplus_{k\geq 0} \mathring{\mathcal{A}}^\wedge_k$ is the completion of the latter algebra for $\mathrm{deg}$.

The subspace $\mathring{\mathcal{A}}_0$ \index[notation]{A_0@$\mathring{\mathcal{A}}_0$} of $\bigoplus_{k\geq 0} \mathring{\mathcal{A}}_k$ is a subalgebra, equipped with the degree $\mathrm{deg}$. We denote by $\mathring{\mathcal{A}}^\wedge_0$  \index[notation]{A^\wedge_0@$\mathring{\mathcal{A}}^\wedge_0$} its degree completion.  

\subsubsection{The algebra morphisms $j^{T_{\bullet}}:\bigoplus_{k\geq 0} \mathring{\mathcal{A}}^\wedge_k\to \mathring{\mathcal{A}}^\wedge_0$ associated to a diagrammatic trace $T_{\bullet}$}

In this subsection we define a map   $j^{T_{\bullet}}:\bigoplus_{k\geq 0} \mathring{\mathcal{A}}^\wedge_k\to \mathring{\mathcal{A}}^\wedge_0$ associated to a diagrammatic trace $T_{\bullet}$. If $T_{\bullet}$ is the $n$-th power (using the product $*$ from Lemma~\ref{productindiagrtraces}) of the diagrammatic LMO trace (see Lemma~\ref{r:elementsTmLMO}), then $j^{T_{\bullet}}$ is the map  introduced in \cite[Section~2]{LMO98} (see also  \cite[Chapter 10]{Ohts}).

\begin{definition}\label{def:absSTUcyc} For $k \geq 1$ and $n_1,...,n_k \geq 0$ and $i \in [\![1,k]\!]$. 
Set $$\mathcal A := \bigoplus_{m_1,\ldots,m_k \geq 0}\mathring{\mathcal A}^{\wedge}(\vec{\varnothing},\emptyset,\mathbb{Z}_{m_1} \sqcup\cdots\sqcup \mathbb{Z}_{m_k}).$$
\begin{itemize}

\item[$(a)$]  Define the linear map
$$\varphi_{(n_1,\ldots,n_k,i)} : \mathring{\mathcal A}^{\wedge}(\vec{\varnothing},\emptyset,\mathbb{Z}_{n_1} \sqcup\cdots\sqcup \mathbb{Z}_{n_k})\longrightarrow \mathring{\mathcal A}^{\wedge}(\vec{\varnothing},\emptyset,\mathbb{Z}_{n_1} \sqcup\cdots\sqcup \mathbb{Z}_{n_k})\subset \mathcal A$$ by $T\mapsto \tilde{T}-T$, where $\tilde T$ is the action of the permutation of  $\mathbb{Z}_{n_1} \sqcup \cdots \sqcup \mathbb{Z}_{n_k}$ given by $x\mapsto x+1$ for $i$-th component. 
Let $$\varphi^{\mathrm{cyc}}: \bigoplus_{i \in [\![1,k]\!],n_1,\ldots,n_k \in \mathbb{Z}_{\geq 0}}  \mathring{\mathcal A}^{\wedge}(\vec{\varnothing},\emptyset,\mathbb{Z}_{n_1} \sqcup\cdots\sqcup \mathbb{Z}_{n_k}) \longrightarrow  \mathcal A$$ be defined as
$$\varphi^{\mathrm{cyc}}:=\sum_{i=1}^k\sum_{n_1,...,n_k \geq 0}\varphi_{(n_1,\ldots,n_k,i)}.$$

\item[$(b)$] For  $i \in [\![1,k]\!]$ such that $n_i \geq 2$,  define the linear map
 $$\psi^1_{(n_1,\ldots,n_k,i)} : \mathring{\mathcal A}^{\wedge}(\vec{\varnothing},  \emptyset ,\mathbb{Z}_{n_1} \sqcup\cdots \sqcup \mathbb{Z}_{n_k})\longrightarrow  \mathring{\mathcal A}^{\wedge}(\vec{\varnothing}, \emptyset,\mathbb{Z}_{n_1}\sqcup \cdots \sqcup \mathbb{Z}_{n_i-1}\sqcup\cdots\sqcup \mathbb{Z}_{n_k})\subset  \mathcal A$$
as the map derived from the map 
$$\mathring{\mathrm{Jac}}(\vec{\varnothing},  \emptyset ,\mathbb{Z}_{n_1} \sqcup\cdots \sqcup \mathbb{Z}_{n_k})\longrightarrow  \mathring{\mathrm{Jac}}(\vec{\varnothing}, \emptyset,\mathbb{Z}_{n_1}\sqcup \cdots \sqcup \mathbb{Z}_{n_i-1}\sqcup\cdots\sqcup \mathbb{Z}_{n_k})$$
which takes $\underline{D}\in\mathring{\mathrm{Jac}}(\vec{\varnothing},  \emptyset ,\mathbb{Z}_{n_1} \sqcup\cdots \sqcup \mathbb{Z}_{n_k})$ to the Jacobi diagram $\underline{D}*_{(n_i-2,n_i-1)} \Yup$  obtained from $\underline{D}$ by gluing the two univalent vertices (corresponding to $\mathbb{Z}_{n_i}$) labelled by $n_i-2$ and $n_i-1$ to two of the univalent vertices  of the $\Yup$-shaped graph and labelling the remaining vertex of the $\Yup$-shaped graph by $n_i-2\in\mathbb{Z}_{n_i-1}\subset \mathbb{Z}_{n_1}\sqcup \cdots \sqcup \mathbb{Z}_{n_i-1}\sqcup\cdots\sqcup \mathbb{Z}_{n_k}$. The cyclic order of the new trivalent vertex is $[\mathrm{old}(n_i-1)<\mathrm{new}(n_i-2)<\mathrm{old}(n_i-2)]$.

\item[$(c)$] For  $i \in [\![1,k]\!]$ such that $n_i \geq 2$, define the linear map
$$\psi^2_{(n_1,\ldots,n_k,i)} : \mathring{\mathcal A}^{\wedge}(\vec{\varnothing},\emptyset,\mathbb{Z}_{n_1} \sqcup\cdots\sqcup \mathbb{Z}_{n_k})\longrightarrow \mathring{\mathcal A}^{\wedge}(\vec{\varnothing},\emptyset,\mathbb{Z}_{n_1} \sqcup\cdots\sqcup \mathbb{Z}_{n_k})\subset  \mathcal A$$ by $T \mapsto \tilde{T} - T$, where $\tilde T$ is obtained from $T$ by the action of the transposition of $\mathbb{Z}_{n_1} \sqcup \cdots \sqcup \mathbb{Z}_{n_k}$ given by the exchange of the elements $n_i-2$ and $n_i-1$ in position $i$.

\item[$(d)$] For  $i \in [\![1,k]\!]$ such that $n_i \geq 2$,  let 
$$\psi_{(n_1,\ldots,n_k,i)} : \mathring{\mathcal A}^{\wedge}(\vec{\varnothing}, \emptyset, \mathbb{Z}_{n_1} \sqcup\cdots\sqcup \mathbb{Z}_{n_k}) \longrightarrow 
 \mathcal A$$ 
given by $\psi_{(n_1,\ldots,n_k,i)}:= \psi^1_{(n_1,\ldots,n_k,i)} + \psi^2_{(n_1,\ldots,n_k,i)}$.
Let 
$$\psi^{\mathrm{STU}}: \bigoplus_{(n_1,\ldots,n_k,i)\ | \ (n_1,\ldots,n_k) \in \mathbb{Z}_{\geq0}^k \text{ and } n_i\geq 2} \mathring{\mathcal A}^{\wedge}(\vec{\varnothing},\emptyset,\mathbb{Z}_{n_1} \sqcup\cdots\sqcup \mathbb{Z}_{n_k}) \longrightarrow  \mathcal A$$ the map defined by 
$$\psi^{\mathrm{STU}}:=\sum_{i=1}^k\sum_{n_1,...,n_k \geq 0, \text{ with } n_i \geq 2}\psi_{(n_1,\ldots,n_k,i)}.$$
\end{itemize}
\end{definition}

Recall that if $G$ is a group acting on a set $\Sigma$ then $\Sigma^G:=\{s\in \Sigma \ | \ g\cdot s = s\}$ is called the set of \emph{invariant elements} and  $\Sigma/G$, i.e., the set of all orbits of $\Sigma$ under the action of $G$, is  called the \emph{quotient set} of the action. If $G$ acts on a $\mathbb{C}$-vector space $V$, then $V^G$ is a vector space. If $\Sigma$ is a $G$-space, one has ($\mathbb{C} \Sigma)^G=\mathbb{C}(\Sigma^G)$ and $(\mathbb{C} \Sigma)_G=\mathbb{C}(\Sigma/G)$, where the $(\mathbb{C} \Sigma)_G$ denotes the space of coinvariants of the pair $(G,\mathbb{C}\Sigma)$ (see \S\ref{sec:2-9-2}) .

 Then for any family $n_1$, $\ldots$, $n_k$ of non-negative  integers, the group $\mathbb{Z}_{n_1}\times \cdots \times \mathbb{Z}_{n_k}$ acts on the disjoint union $\mathbb{Z}_{n_1}\sqcup \cdots \sqcup \mathbb{Z}_{n_k}$, the action of $\mathbb{Z}_{n_i}$ on itself being by the product and the action of $\mathbb{Z}_{n_i}$ on $\mathbb{Z}_{n_j}$ being trivial if $i \neq j$.

\begin{lemma}\label{r:2022-06-14lemmabijJac}
For any $k\geq 1$ and any family  $n_1,\ldots, n_k$ of  non-negative integers, there is a bijection 
\begin{equation}\label{eq:2022-06-06bij}\bigsqcup_{n_1, \ldots, n_k \geq 0}\mathring{\mathrm{Jac}}(\vec{\varnothing},\emptyset, \mathbb{Z}_{n_1}\sqcup \cdots \sqcup \mathbb{Z}_{n_k})/\mathbb{Z}_{n_1}\times\cdots\times\mathbb{Z}_{n_k}\longrightarrow\mathring{\mathrm{Jac}}(\vec{\varnothing},[\![1,k]\!]),
\end{equation}
which induces an isomorphism
\begin{equation}\label{eq:2023-08-17iso}
\frac{\bigoplus_{n_1,\ldots, n_k \geq 0}
\mathring{\mathcal A}^{\wedge}(\vec{\varnothing},\emptyset,\mathbb{Z}_{n_1} \sqcup \cdots\sqcup \mathbb{Z}_{n_k})_{\mathbb{Z}_{n_1}\times\cdots\times \mathbb{Z}_{n_k}}}{[\mathrm{Im}(\psi^{\mathrm{STU}})]} \simeq \frac{\bigoplus_{n_1,\ldots, n_k \geq 0}
\mathring{\mathcal A}^{\wedge}(\vec{\varnothing},\emptyset,\mathbb{Z}_{n_1} \sqcup \cdots\sqcup \mathbb{Z}_{n_k})}{\mathrm{Im}(\psi^{\mathrm{STU}})+ \mathrm{Im}(\varphi^{\mathrm{cyc}})} \simeq  \mathring{\mathcal A}^{\wedge}(\vec{\varnothing},[\![1,k]\!]),
\end{equation}
where $\psi^{\mathrm{STU}}$ and $\varphi^{\mathrm{cyc}}$ are as in Definition~\ref{def:absSTUcyc} and where  $[\mathrm{Im}(\psi^{\mathrm{STU}})]$ is the image of $\mathrm{Im}(\psi^{\mathrm{STU}})$ under the canonical projection $V\to V_G$ for $V=\bigoplus_{n_1,\ldots, n_k \geq 0} \mathring{\mathcal A}^{\wedge}(\vec{\varnothing},\emptyset,\mathbb{Z}_{n_1} \sqcup \cdots\sqcup \mathbb{Z}_{n_k})$ and $G=\bigoplus_{n_1,\ldots, n_k \geq 0}\mathbb{Z}_{n_1}\times\cdots\times \mathbb{Z}_{n_k}$.

\end{lemma}

\begin{proof}
An element of $\mathring{\mathrm{Jac}}(\vec{\varnothing},[\![1,k]\!]$ is a triple $\big(D,\varphi, \{\mathrm{cyc}_s\}_{s\in[\![1,k]\!]}\big)$ where $D$ is a vertex-oriented unitrivalent graph, $\varphi:\partial{D}\to [\![1,k]\!]$ is a map  and for $s \in [\![1,k]\!]$, $\mathrm{cyc}_s$ is a cyclic order on $\varphi^{-1}(s)$. The assignment $\big(D,\varphi, \{\mathrm{cyc}_s\}_{s\in[\![1,k]\!]}\big)\mapsto \big(|\varphi^{-1}(1)|, \ldots, |\varphi^{-1}(k)|\big)$ defines a map $\mathring{\mathrm{Jac}}(\vec{\varnothing},[\![1,k]\!])\to \mathbb{Z}_{\geq 0}^k$ and therefore a partition $\mathring{\mathrm{Jac}}(\vec{\varnothing},[\![1,k]\!]) = \bigsqcup_{n_1,\ldots, n_k \geq 0}\mathring{\mathrm{Jac}}(\vec{\varnothing},[\![1,k]\!])_{(n_1,\ldots, n_k)}$; where an element of $\mathring{\mathrm{Jac}}(\vec{\varnothing},[\![1,k]\!])_{(n_1,...n_k)}$ is identified with the triple of a vertex-oriented unitrivalent graph $D$, a partition of $\partial D$ in $k$ components of cardinalities $n_1$, $\ldots$, $n_k$ and a cyclic order on each of these components. 

Notice that a cyclic order on a finite set $S$ is the same as an element of the quotient set $\mathrm{Bij}(S,\mathbb{Z}_{|S|})/\mathbb{Z}_{|S|}$, where $\mathrm{Bij}(S,\mathbb{Z}_{|S|})$ denotes the set of bijections from $S$ to $\mathbb{Z}_{|S|}$. Therefore each set $\mathring{\mathrm{Jac}}(\vec{\varnothing},[\![1,k]\!])_{(n_1,\ldots, n_k)}$  can be identified with $\mathring{\mathrm{Jac}}(\vec{\varnothing},\emptyset,\mathbb{Z}_{n_1}\sqcup \cdots \sqcup \mathbb{Z}_{n_k})/\mathbb{Z}_{n_1}\times\cdots\times\mathbb{Z}_{n_k}$, which gives the stated bijection.

The bijection~\eqref{eq:2022-06-06bij}  induces a linear isomorphism 
$$\bigoplus_{n_1, \ldots, n_k \geq 0}\mathbb{C}\mathring{\mathrm{Jac}}(\vec{\varnothing},\emptyset,\mathbb{Z}_{n_1} \sqcup \cdots \sqcup \mathbb{Z}_{n_k})_{\mathbb{Z}_{n_1}\times \cdots \times \mathbb{Z}_{n_k}}\longrightarrow \mathbb{C}\mathring{\mathrm{Jac}}(\vec{\varnothing},[\![1,k]\!]),$$ 
which gives rise to a morphism
$$\frac{\bigoplus_{n_1,\ldots,n_k\geq 0}\mathbb{C}\mathring{\mathrm{Jac}}(\vec{\varnothing},\emptyset,\mathbb{Z}_{n_1} \sqcup \cdots\sqcup \mathbb{Z}_{n_k})}{\mathrm{Im}(\varphi^{\mathrm{cyc}}_{\mathring{\mathrm{Jac}}})}\longrightarrow \mathring{\mathcal A}(\vec{\varnothing},[\![1,k]\!]),$$
where $\varphi^{\mathrm{cyc}}_{\mathring{\mathrm{Jac}}}$  is the analogue of the map $\varphi^{\mathrm{cyc}}$ between sets of Jacobi diagrams. One checks that it induces  a map
$$\frac{\bigoplus_{n_1,\ldots, n_k \geq 0}
\mathring{\mathcal A}^{\wedge}(\vec{\varnothing},\emptyset,\mathbb{Z}_{n_1} \sqcup \cdots\sqcup \mathbb{Z}_{n_k})}{\mathrm{Im}(\psi^{\mathrm{STU}})+ \mathrm{Im}(\varphi^{\mathrm{cyc}})} \longrightarrow  \mathring{\mathcal A}^{\wedge}(\vec{\varnothing},[\![1,k]\!])$$
and that this map is an isomorphism.  The first isomorphism in \eqref{eq:2023-08-17iso} follows since taking coinvariants with respect to the action of $\mathbb{Z}_{n_1}\times \cdots \times \mathbb{Z}_{n_k}$ is equivalent to taking the quotient by  $[\mathrm{Im}(\varphi^{\mathrm{cyc}})]$.

\end{proof}

\begin{proposition}\label{r:diagtracesimpliestrace} Let $T_{\bullet}=(T_m)_{m\geq 0}$ be a diagrammatic trace of bidegree $(c,d)$ (see Lemma~\ref{def:bidegree}). Set $j^{T_{\bullet}}[0]:=\mathrm{Id}: \mathring{\mathcal A}_0^\wedge \to \mathring{\mathcal A}_0^\wedge$ be the identity map. For any integer  $k\geq 1$ we can associate to $T_{\bullet}$ a $\mathfrak{S}_k$-equivariant continuous linear map $j^{T_{\bullet}}[k]: \mathring{\mathcal A}^{\wedge}(\vec\varnothing,[\![1,k]\!]) \to \mathring{\mathcal A}^{\wedge}(\vec{\varnothing}, \emptyset)$ and therefore a continuous linear map  $j^{T_{\bullet}}[k]: \mathring{\mathcal A}^{\wedge}_k \to \mathring{\mathcal A}^{\wedge}_0$. The assignment $T_{\bullet} \mapsto j^{T_{\bullet}}[k]$ is polynomial of degree $k$ with respect to $T_{\bullet}$; in particular, for any $\lambda\in\mathbb{C}$ and $k\geq 1$, we have $j^{\lambda T_{\bullet}}[k]=\lambda^k j^{T_{\bullet}}[k]$ and $T_{\bullet} \mapsto j^{T_{\bullet}}[1]$ is linear. Moreover, the direct sum 
\begin{equation}\label{equ:mapj1}
j^{T_{\bullet}}=\bigoplus_{k\geq 0}j^{T_{\bullet}}[k]:\bigoplus_{k\geq 0} \mathring{\mathcal{A}}^\wedge_k\longrightarrow \mathring{\mathcal{A}}^\wedge_0
\end{equation}  \index[notation]{j^{T_{\bullet}}@$j^{T_{\bullet}}$}
is an algebra homomorphism and for a homogeneous $x \in \mathring{\mathcal A}^{\wedge}_k$ one has 
\begin{equation}\label{equdegreeandj}
\mathrm{deg}(j^{T_{\bullet}}(x))=\mathrm{deg}(x)+(d-c)k,
\end{equation}
 with $\mathrm{deg}$ being defined on $\mathring{\mathcal A}_k^{\wedge}$ and $\mathring{\mathcal A}_0^{\wedge}$ as in \eqref{degree-on-APS}.
\end{proposition}

\begin{proof}

Let $k\geq 1$ be an integer.  For  $n_1,\ldots,n_k\geq 0$, the map \eqref{tensor:map:A:PP'gen} gives rise to a linear map
$$\mathring{\mathcal A}^{\wedge}(\vec{\varnothing}, \emptyset,\mathbb{Z}_{n_1}) \otimes\cdots \otimes \mathring{\mathcal A}^{\wedge}(\vec{\varnothing}, \emptyset, \mathbb{Z}_{n_k})\longrightarrow 
\mathring{\mathcal A}(\vec{\varnothing}, \emptyset, \mathbb{Z}_{n_1}\sqcup\cdots \sqcup \mathbb{Z}_{n_k}),$$
which restricts to a linear map 
$$\mathring{\mathcal A}^{\wedge}(\vec{\varnothing}, \emptyset, \mathbb{Z}_{n_1})^{\mathbb{Z}_{n_1}} \otimes\cdots \otimes \mathring{\mathcal A}^{\wedge}(\vec{\varnothing}, \emptyset, \mathbb{Z}_{n_k})^{\mathbb{Z}_{n_k}}\longrightarrow
\mathring{\mathcal A}^{\wedge}(\vec{\varnothing}, \emptyset, \mathbb{Z}_{n_1}\sqcup\cdots \sqcup \mathbb{Z}_{n_k})^{\mathbb{Z}_{n_1}\times\cdots\times\mathbb{Z}_{n_k}}.$$

If $G$ is a group, if $V,W$ are $G$-modules and $Z$ is a vector space, then any $G$-invariant linear map $V \otimes W \to Z$ induces a linear map $V^G \otimes W_G \to Z$, obtained by factorization of its restriction to $V^G \otimes W$. It follows that in the situation of the pairing \eqref{thepairinginA}, if $\Sigma$ is equipped with an action of a group $G$, it induces a pairing 
$$\mathring{\mathcal A}^{\wedge}(P,S,\Sigma)^G \otimes \mathring{\mathcal A}^{\wedge}(P',S',\Sigma)_G\longrightarrow\mathring{\mathcal A}(P \otimes P',S \sqcup S').$$ 
In our case, we obtain the pairing
$$\mathring{\mathcal A}^{\wedge}(\vec{\varnothing},\emptyset, \mathbb{Z}_{n_1}\sqcup\cdots \sqcup \mathbb{Z}_{n_k})^{\mathbb{Z}_{n_1}\times\cdots\times\mathbb{Z}_{n_k}} \otimes \mathring{\mathcal A}^{\wedge}(\vec{\varnothing},\emptyset,\mathbb{Z}_{n_1}\sqcup \cdots \sqcup \mathbb{Z}_{n_k})_{\mathbb{Z}_{n_1}\times\cdots\times \mathbb{Z}_{n_k}}\longrightarrow\mathring{\mathcal A}^{\wedge}(\vec{\varnothing},\emptyset).$$
In particular, the pairing with  $T_{n_1} \otimes\cdots\otimes T_{n_k}\in \mathring{\mathcal A}(\vec{\varnothing},\emptyset, \mathbb{Z}_{n_1}\sqcup\cdots \sqcup \mathbb{Z}_{n_k})^{\mathbb{Z}_{n_1}\times\cdots\times\mathbb{Z}_{n_k}}$, induces a linear map 
\begin{equation}\label{eq:2022-06-06-equ3}
\langle T_{n_1}\otimes \cdots\otimes T_{n_k}, \bullet \rangle : \mathring{\mathcal A}^{\wedge}(\vec{\varnothing},\emptyset,\mathbb{Z}_{n_1}\sqcup \cdots \sqcup \mathbb{Z}_{n_k})_{\mathbb{Z}_{n_1}\times\cdots\times \mathbb{Z}_{n_k}} \longrightarrow\mathring{\mathcal A}^{\wedge}(\vec{\varnothing},\emptyset),
\end{equation}
the direct sum over $n_1,\ldots,n_k\geq 0$ gives rise to the linear map
\begin{equation}\label{eq:2022-06-06-equ3SUM}
\bigoplus_{n_1,\ldots,n_k \geq 0}\langle T_{n_1}\otimes\cdots\otimes T_{n_k},\bullet \rangle\ : \bigoplus_{n_1,\ldots,n_k \geq 0}\mathring{\mathcal A}^{\wedge}(\vec{\varnothing},\emptyset,\mathbb{Z}_{n_1} \sqcup \cdots \sqcup \mathbb{Z}_{n_k})_{\mathbb{Z}_{n_1}\times \cdots \times \mathbb{Z}_{n_k}} \longrightarrow \mathring{\mathcal A}^{\wedge}(\vec{\varnothing},\emptyset).
\end{equation}
By the STU-compatibility of $T_{\bullet}$ (see Definition~\ref{def:diagrtrace}), the map \eqref{eq:2022-06-06-equ3SUM} gives rise, by factorization, to a map 
  \begin{equation}\label{eq:2022-06-06-equ3SUMquotient}
\bigoplus_{n_1,\ldots,n_k \geq 0}\langle T_{n_1}\otimes\cdots\otimes T_{n_k},\bullet \rangle\ : \frac{\bigoplus_{n_1,\ldots,n_k \geq 0}\mathring{\mathcal A}^{\wedge}(\vec{\varnothing},\emptyset,\mathbb{Z}_{n_1} \sqcup \cdots \sqcup \mathbb{Z}_{n_k})_{\mathbb{Z}_{n_1}\times \cdots \times \mathbb{Z}_{n_k}}}{\mathrm{Im}(\psi^{\mathrm{STU}})} \longrightarrow \mathring{\mathcal A}^{\wedge}(\vec{\varnothing},\emptyset),
\end{equation}
where $\psi^{\mathrm{STU}}$ is as in Definition~\ref{def:absSTUcyc}$(d)$. Using Lemma~\ref{r:2022-06-14lemmabijJac}, one then defines a linear map 
\begin{equation}\label{eq:2022-06-06equ2}
j^{T_{\bullet}}[k] : \mathring{\mathcal A}^{\wedge}(\vec{\varnothing},[\![1,k]\!]) \stackrel{\eqref{eq:2023-08-17iso}}{\ \simeq\ } \frac{\bigoplus_{n_1,\ldots,n_k \geq 0}\mathring{\mathcal A}^{\wedge}(\vec{\varnothing},\emptyset,\mathbb{Z}_{n_1} \sqcup \cdots \sqcup \mathbb{Z}_{n_k})_{\mathbb{Z}_{n_1}\times \cdots \times \mathbb{Z}_{n_k}}}{\mathrm{Im}(\psi^{\mathrm{STU}})}\xrightarrow{\ \eqref{eq:2022-06-06-equ3SUMquotient}\ } \mathring{\mathcal A}^{\wedge}(\vec{\varnothing},\emptyset)
\end{equation}

The map~\eqref{eq:2022-06-06equ2} is  described schematically in Figure~\ref{the-map-j-1}.
\begin{figure}[ht]
		\centering
         \includegraphics[scale=1]{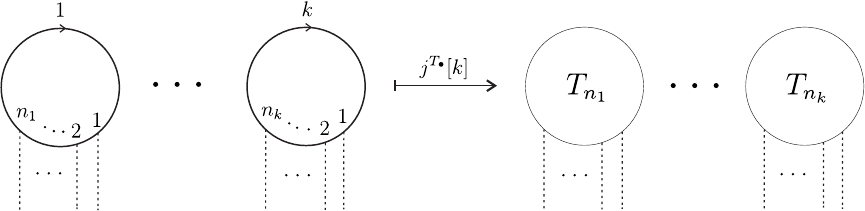}
\caption{Schematic representation of the map \eqref{eq:2022-06-06equ2}.}
\label{the-map-j-1} 								
\end{figure}

This  map is $\mathfrak{S}_k$-invariant,  therefore it induces a  linear map $j^{T_{\bullet}}[k]: \mathring{\mathcal{A}}^{\wedge}_k \to \mathring{\mathcal{A}}^{\wedge}_0$. We then check that the direct sum  $j^{T_{\bullet}}=\bigoplus_{k\geq 0}j^{T_{\bullet}}[k] : \bigoplus_{k\geq 0} \mathring{\mathcal{A}}^\wedge_k\to \mathring{\mathcal{A}}^\wedge_0$   is an algebra homomorphism and that for $x \in \mathring{\mathcal A_k}^{\wedge}$ one has $\mathrm{deg}(j^{T_{\bullet}}(x))=\mathrm{deg}(x)+(d-c)k$.
\end{proof}

By Lemma~\ref{def:productoftraces}, for integer $n\geq 1$ the $n$-th power  $(T_{\bullet}^{\mathrm{LMO}})^{*n}$ of the $\mathrm{LMO}$ diagrammatic trace $T_{\bullet}^{\mathrm{LMO}}$  (see Lemma~\ref{r:elementsTmLMO}) is a diagrammatic trace.

\begin{definition}\label{def:the-map-j-n} Let  $n\geq 1$ be an integer.  Define the algebra homomorphism
\begin{equation}\label{equ:the-map-j-n}
j_n^{\mathrm{LMO}}:\bigoplus_{k\geq 0} \mathring{\mathcal{A}}^\wedge_k\longrightarrow \mathring{\mathcal{A}}^\wedge_0
\end{equation}  \index[notation]{j_n^{\mathrm{LMO}}@$j_n^{\mathrm{LMO}}$}
as the morphism associated by Proposition~\ref{r:diagtracesimpliestrace} to the  $n$-th power  $(T_{\bullet}^{\mathrm{LMO}})^{*n}$ of the $\mathrm{LMO}$ diagrammatic trace $T_{\bullet}^{\mathrm{LMO}}$. 
\end{definition}

Specializing \eqref{equdegreeandj} to the LMO diagrammatic trace,  for $x \in \mathring{\mathcal{A}}^\wedge_k$ one has
\begin{equation}\label{equ:2023-10-10}
\mathrm{deg}(j_n^{\mathrm{LMO}}(x))=\mathrm{deg}(x)-kn,
\end{equation} 
where $\mathrm{deg}$  is as in \eqref{degree-on-APS}.

\begin{definition}\label{def:doublingfree}Let $n \geq 1$ be an integer and $S$ a finite set. Define a map $\mathbb{C}\mathring{\mathrm{Jac}}(\vec{\varnothing},S) \to \mathbb{C}\mathring{\mathrm{Jac}}(\vec{\varnothing},S \times [\![1,n]\!])$ by
\begin{equation}\label{equ:doublingfree}
\underline{D}=\big(D,\varphi : \partial D\to S, \{\mathrm{cyc}_s\}_{s \in S}\big)\longmapsto \underline {D}^{(n)}:=\sum_{\psi : \partial D \to S\times [\![1,n]\!] \ | \  \mathrm{pr}_1 \circ \psi=\varphi}\big(D,\psi,\{\mathrm{cyc}_{(s,i)}^\psi\}_{(s,i) \in S \times [\![1,n]\!]}\big)\end{equation}
where $\mathrm{pr}_1: S\times[\![1,n]\!]\to S$ is the projection on the first component and for any $\psi :\partial D \to S\times [\![1,n]\!]$ such that $\mathrm{pr}_1 \circ \psi=\varphi$ and any 
$(s,i) \in S \times [\![1,n]\!]$, we denote by $\mathrm{cyc}_{(s,i)}^{\psi}$  the cyclic order on $\psi^{-1}(s,i) \subset \varphi^{-1}(s)$ induced by the restriction of the cyclic order $\mathrm{cyc}_s$ on $\varphi^{-1}(s)$.
\end{definition}

\begin{lemma}\label{r:jLMOanddoubling} Let $n, k \geq 1$ be integers. For any diagrammatic  trace $T_{\bullet}$ and $\underline{D} \in \mathring{\mathrm{Jac}}(\vec{\varnothing},[\![1,k]\!])$, one has  $$j^{(T_{\bullet})^{*n}}([\underline{D}])=\frac{1}{(n!)^{k}}j^{T_{\bullet}}\big([\underline{D}^{(n)}]\big).$$
In particular, specializing to the LMO diagrammatic trace, one has
$$j^{\mathrm{LMO}}_n([\underline{D}])=\frac{1}{(n!)^{k}}j^{\mathrm{LMO}}_1\big([\underline{D}^{(n)}]\big).\footnote{In \cite[pag. 281]{Ohts}, there is typo, the factor $1/n!$ should be $1/(n!)^l$.}$$
\end{lemma}

\begin{proof}For $E$ a finite set, denote by $\mathrm{Part}_n(E)$ the set of ordered $n$-partitions of $E$, and by $\mathrm{UPart}_n(E):=\mathrm{Part}_n(E)/\mathfrak{S}_n$ the set of its unordered $n$-partitions. Let $\underline{D}=\big(D,\varphi : \partial D\to [\![1,k]\!], \{\mathrm{cyc}_s\}_{s \in [\![1,k]\!]}\big)\in \mathring{\mathrm{Jac}}(\vec{\varnothing},[\![1,k]\!])$. To each $\big((A_1^1,\ldots,A_n^1),\ldots,(A_1^k,\ldots,A_n^k)\big) \in \mathrm{Part}_n(\varphi^{-1}(1))\times \cdots \times \mathrm{Part}_n(\varphi^{-1}(k))$ attach the map $\psi : \partial D \to [\![1,k]\!]\times [\![1,n]\!]$ such that $\psi(A_i^s):=(s,i)$.  Then the assignment $$\big((A_1^1,\ldots,A_n^1),\ldots,(A_1^k,\ldots,A_n^k)) \longmapsto j^{T_{\bullet}}\big(D,\psi,\{\mathrm{cyc}^{\psi}_{(s,i)}\}_{(s,i) \in [\![1,k]\!]\times [\![1,n]\!]}\big),$$ 
where $\mathrm{cyc}^{\psi}_{(s,i)}$ is as in Definition~\ref{def:doublingfree}, is a map 
 $$\mathrm{Part}_n(\varphi^{-1}(1))\times \cdots \times \mathrm{Part}_n(\varphi^{-1}(k)) \longrightarrow \mathbb{C}\mathring{\mathrm{Jac}}(\vec{\varnothing}, \emptyset).$$
It factors through a map 
$$\mathrm{UPart}_n(\varphi^{-1}(1))\times\cdots \times \mathrm{UPart}_n(\varphi^{-1}(k))\longrightarrow \mathbb{C}\mathring{\mathrm{Jac}}(\vec{\varnothing}, \emptyset).$$
The element $j^{T_{\bullet}}([\underline{D}^{(n)}])$  (resp. $j^{(T_{\bullet})^{*n}}([\underline{D}])$) 
is equal to the sum of the classes in $\mathring{\mathcal A}^\wedge(\vec{\varnothing},\emptyset)$ of the images of the former (resp. latter) map. The result then follows from  the fact that the  fibers of the projection  
$\mathrm{Part}_n(\varphi^{-1}(1))\times \cdots \times \mathrm{Part}_n(\varphi^{-1}(k)) \to \mathrm{UPart}_n(\varphi^{-1}(1))\times\cdots\times \mathrm{UPart}_n(\varphi^{-1}(k))$ are all of cardinality $(n!)^k$. 
\end{proof}

\subsubsection{The algebra $\mathfrak{c}(n)$}\label{sec:4-5-3}

\begin{definition}\label{def:theelementXaloop}
Let $X$ be the element in $\mathring{\mathcal{A}}^\wedge_0$ defined as the class of  the connected Jacobi diagram without trivalent vertices (dashed loop).  Notice that it has degree  $0$.  
\end{definition}

\begin{definition}\label{def:ideal-O-n} Define  $O^n$   \index[notation]{O^n@$O^n$} to be the closed ideal of  the  commutative algebra $\mathring{\mathcal{A}}^\wedge_0$ generated by $X+2n$. Since $X+2n$ is homogeneous  of degree $0$, then $O^n$ is a complete graded ideal of $\mathring{\mathcal{A}}^\wedge_0$.
\end{definition}

Recall that $P^{n+1}_0$ is a complete graded ideal of $\mathring{\mathcal{A}}^\wedge_0$  (see Definition~\ref{space-P-n-general}$(c)$). Hence $P^{n+1}_0 + O^n$ is also a complete graded ideal of  $\mathring{\mathcal{A}}^\wedge_0$.

\begin{definition}\label{def:b-n} Define the complete graded commutative algebra $\mathfrak{c}(n)$ as the quotient algebra
\begin{equation}\label{eq:def:b-n}
\mathfrak{c}(n)= \frac{\mathring{\mathcal{A}}^\wedge_0}{P^{n+1}_0 + O^n}.
\end{equation}  \index[notation]{c(n)@$\mathfrak{c}(n)$} 

\end{definition}

\subsubsection{The algebra morphism $\overline{j}_n:\mathfrak{b}(n)\to \mathfrak{c}(n)$}

\begin{lemma}\label{jandco} Let $k \geq 1$ be an integer. Then $j^{\mathrm{LMO}}_1=j^{\mathrm{LMO}}_1\circ \mathrm{co}^{\mathcal{A}}_k$. Here $\mathrm{co}^{\mathcal{A}}_k: \mathring{\mathcal{A}}^\wedge(\vec{\varnothing},[\![1,k]\!])_{\mathfrak{S}_{k-1}}\to \mathring{\mathcal{A}}^\wedge_k$ is the change of orientation map from Lemma~\ref{def:mapCotoXk} specialized to the pre-LMO structure $\mathbf{A}$.
\end{lemma}
\begin{proof}
The space $\mathring{\mathcal{A}}^\wedge_k$ is topologically spanned by the classes of elements in $\bigsqcup_{n_1,\ldots, n_k\geq 0}\mathring{\mathrm{Jac}}(\vec{\varnothing},\emptyset,\mathbb{Z}_{n_1}\times \cdots\times \mathbb{Z}_{n_k})/\mathbb{Z}_{n_1}\times \cdots\times \mathbb{Z}_{n_k}$, see Lemma~\ref{r:2022-06-14lemmabijJac}. For $\underline{D}\in \mathring{\mathrm{Jac}}(\vec{\varnothing},\emptyset,\mathbb{Z}_{n_1}\times \cdots\times \mathbb{Z}_{n_k})$, the image of the corresponding element in $ \mathring{\mathcal{A}}^\wedge(\vec{\varnothing},[\![1,k]\!])_{\mathfrak{S}_{k-1}}$ under $\mathrm{co}^{\mathcal{A}}_k$ is $(-1)^{n_1}$ times class of the element obtained from $\underline{D}$ by the automorphism of $\mathbb{Z}_{n_1}\times \cdots \times \mathbb{Z}_{n_k}$ given by the action of $-1$ on the first component. The image of the latter element by $j^{\mathrm{LMO}}_1$ is therefore the pairing with $(-1)^{n_1}$ times the effect of the action of same automorphism of $\mathbb{Z}_{n_1}\times \cdots \times \mathbb{Z}_{n_k}$  on the disjoint union of $T_{n_1}^{\mathrm{LMO}}$, $\ldots$, $T_{n_k}^{\mathrm{LMO}}$. By \cite[Lemma~2.5]{LMO98}, this element coincides with the disjoint union of $T_{n_1}^{\mathrm{LMO}}$, $\ldots$, $T_{n_k}^{\mathrm{LMO}}$ therefore the image by $j^{\mathrm{LMO}}_1\circ \mathrm{co}^{\mathcal{A}}_k$ of  the corresponding element associated to $\underline{D}$ coincides with its image under $j^{\mathrm{LMO}}_1$.
\end{proof}

\begin{proposition}\label{r2022-04-08-jn} We have
$$j_n^{\mathrm{LMO}}(L^{<2n}+ P^{n+1} + (\mathrm{CO}\mathring{\mathcal{A}}^\wedge))\subset P^{n+1}_0 \subset P^{n+1}_0 + O^n.$$
Therefore, $j_n^{\mathrm{LMO}}$ induces an algebra homomorphism
\begin{equation}
\overline{j}_n:\mathfrak{b}(n)\longrightarrow \mathfrak{c}(n),
\end{equation}  \index[notation]{j_n@$\overline{j}_n$}
where $\mathfrak{b}(n)$ is as in \eqref{space-a-n} and $\mathfrak{c}(n)$ as in~\eqref{eq:def:b-n}.
\end{proposition}

\begin{proof}
Let $k \geq 0$ be an integer and $\underline{D}=(D,\varphi,\{\mathrm{cyc}_s\}_{s \in [\![1,k]\!]}) \in \mathring{\mathrm{Jac}}(\vec{\varnothing},[\![1,k]\!])$ and $i \in [\![1,k]\!]$ be such that $|\varphi^{-1}(i)|<2n$, that is $[\underline{D}]\in L^{<2n}(k)$. In what follows, we use  the notation from the proof of Lemma~\ref{r:jLMOanddoubling}. Any $(A^i_1, \ldots, A^i_n) \in \mathrm{Part}_n(\varphi^{-1}(i))$ is such that for some $s \in [\![1,n]\!]$ we have $|A^i_s|<2$. This implies $j_1^{\mathrm{LMO}}([\underline{D}^{(n)}])=0$, which  by 
Lemma 4.53 implies $j_n^{\mathrm{LMO}}([\underline{D}])=0$. Therefore, $j_n^{\mathrm{LMO}}(L^{<2n}_k)=0$, and hence $j_n^{\mathrm{LMO}}(L^{<2n})=0$.

For $\underline{D}=(D,\varphi,\{\mathrm{cyc}_s\}_{s \in [\![1,k]\!]}) \in \mathring{\mathrm{Jac}}(\vec{\varnothing},[\![1,k]\!])$ and $X\subset [\![1,k]\!]$, let 
$$\mathrm{co}_X(\underline{D}):=(-1)^{|\varphi^{-1}(X)|}\big(D,\varphi, \{\mathrm{cyc}_s\}^X_{s \in [\![1,k]\!]}\big),$$  where $\mathrm{cyc}_s^X:=\mathrm{cyc}_s$ if $s \notin X$  and $\mathrm{cyc}_s^X:=\overline{\mathrm{cyc}}_s$ if $s \in X$.  Then for $\underline{D }\in \mathring{\mathrm{Jac}}(\vec{\varnothing},[\![1,k]\!])$  one has
\begin{equation*}
\begin{split}
 j_n^{\mathrm{LMO}}\big([\mathrm{co}_1(\underline D)]\big) & =\frac{1}{(n!)^{k}}j_1^{\mathrm{LMO}}\big(([{\mathrm{co}_1(\underline D)}^{(n)}]\big)
=\frac{1}{(n!)^{k}}j_1^{\mathrm{LMO}}\big([\mathrm{co}_{{1} \times [\![1,k]\!]}(\underline{D}^{(n)})]\big)\\
&=\frac{1}{(n!)^{k}}j_1^{\mathrm{LMO}}\big([\mathrm{co}_{(1,1)} \circ\cdots \circ \mathrm{co}_{(1,k)}(\underline{D}^{(n)})]\big)
=\frac{1}{(n!)^{k}}j_1^{\mathrm{LMO}}\big([\underline{D}^{(n)}]\big)=j_n^{\mathrm{LMO}}\big([\underline{D}]\big)
\end{split}
\end{equation*} 
(using Lemma~\ref{jandco} and Lemma~\ref{r:jLMOanddoubling}), which implies that $j_n^{\mathrm{LMO}}\big((\mathrm{CO}\mathring{\mathcal A}^\wedge)\big)=0$.

For $T_{\bullet}$ a diagrammatic trace and $k,l \geq 1$, one defines similarly to the map $j^{T_{\bullet}}[k]$ from \eqref{equ:mapj1} a map
\begin{equation}\label{equ:mapj1moregeneral}
j^{T_{\bullet}} [k,l]: \mathring{\mathcal{A}}^\wedge(\vec{\varnothing}, [\![1,k]\!], [\![1,l]\!])\longrightarrow \mathring{\mathcal{A}}^\wedge(\vec{\varnothing},\emptyset, [\![1,l]\!])
\end{equation}
such that for any $\sigma \in  \mathring{\mathcal{A}}^\wedge(\vec{\varnothing},\emptyset, [\![1,l]\!])$, one has $j^{T_{\bullet}}[k] \circ \langle \ ,\sigma\rangle =\langle \ ,\sigma\rangle \circ j^{T_{\bullet}}[k,l]$ (equality of maps $\mathring{\mathcal{A}}^\wedge(\vec{\varnothing}, [\![1,k]\!], [\![1,l]\!])\longrightarrow \mathring{\mathcal{A}}^\wedge(\vec{\varnothing},\emptyset, \emptyset)$), where $\langle \ , \ \rangle$ is as in \eqref{thepairinginA}. Specializing to $T_{\bullet}=(T^{\mathrm{LMO}})^{*n}$, $l:=2(n+1)$ and $\sigma:=\varsigma_{n+1}$ (from Definition~\ref{the-element-varsigma-k}), one obtains the second equality in
\begin{equation*}
\begin{split}
j_n^{\mathrm{LMO}}[k]\big(P^{n+1}(\vec{\varnothing},[\![1,k]\!])\big) 
&=j_n^{\mathrm{LMO}}[k]\left(\big\langle \mathring{\mathcal{A}}^\wedge(\vec{\varnothing}, [\![1,k]\!], [\![1,2(n+1)]\!]), \varsigma_{n+1}\big\rangle\right)\\ 
&=\left\langle j^{(T^{\mathrm{LMO}})^{*n}}[k,2n]\big(\mathring{\mathcal{A}}^\wedge(\vec{\varnothing}, [\![1,k]\!], [\![1,2(n+1)]\!])\big), \varsigma_{n+1}\right\rangle \\
& \subset \left\langle \mathring{\mathcal{A}}^\wedge(\vec{\varnothing}, \emptyset, [\![1,2(n+1)]\!]), \varsigma_{n+1}\right\rangle 
\\
& =P^{n+1}(\vec{\varnothing})_0=P^{n+1}_0
\end{split}
\end{equation*}
where the first and last equalities follow from Definition~\ref{space-P-n-general}. Therefore $j_n^{\mathrm{LMO}}[k](P^{n+1}_k) \subset P^{n+1}_0$ and hence $j_n^{\mathrm{LMO}}(P^{n+1}) \subset P^{n+1}_0$ .
 
\end{proof}

Since $(\mathfrak{b}(n), \varphi_n\circ\overline{\mathrm{cs}}^{\nu}\circ \hat{Z})$ is a semi-Kirby structure, then by Lemma~\ref{sec:1.8lemma3bis}, the pair $(\mathfrak{c}(n), \overline{j}_n\circ\varphi_n \circ \overline{\mathrm{cs}}^{\nu}\circ \hat{Z})$ is also a semi-Kirby structure.

\subsubsection{The Kirby structure $(\mathfrak{c}(n), \overline{j}_n \circ\varphi_n\circ\overline{\mathrm{cs}}^{\nu}\circ \hat{Z})$} In this subsection we prove that the semi-Kirby structure $(\mathfrak{c}(n), \overline{j}_n\circ\varphi_n\circ \overline{\mathrm{cs}}^{\nu}\circ \hat{Z})$ is in fact a Kirby structure (Theorem~\ref{mainr-2022-02-14}). In order to do this, we need to prove that the elements $\overline{j}_n\circ \varphi_n \circ \overline{\mathrm{cs}}^\nu\circ \hat{Z}(U^\pm)\in\mathring{\mathcal{A}}^\wedge_0/(P^{n+1}_0+ O^n)$  are invertible.

In view of this we will compute explicitly the degree~$0$ part  $$\big(\mathring{\mathcal{A}}^{\wedge}_0/P^{n+1}_0\big)[0] = \mathring{\mathcal{A}}^{\wedge}_0[0]/(\mathring{\mathcal{A}}^{\wedge}_0[0]\cap P^{n+1}_0)$$
of the quotient algebra $\mathring{\mathcal{A}}^{\wedge}_0/P^{n+1}_0$. This will be done in Proposition~\ref{degree-0-of-Pn} based on a result (Lemma~\ref{equinFPFI}) on fixed-point free involutions.  Recall that  we denote by $\mathrm{FPFI}(S)$ the set of fixed-point free involutions of a set $S$.

\begin{definition}\label{def:pairingparpar} Let $S$ be a set. 

\begin{itemize}

\item [$(a)$] Define the map
\begin{equation}\label{defequ:pairingparpar}
(\!( \ ,\ )\!): \mathrm{FPFI}(S)\times \mathrm{FPFI}(S) \longrightarrow \mathbb{Z}_{>0} 
\end{equation}
by $(\!( \sigma, \tau )\!) =\left|S/\langle\sigma,\tau\rangle\right|$, i.e., the number of orbits of the action of the subgroup $\langle\sigma,\tau\rangle\leq \mathfrak{S}_{S}$ generated by $\sigma,\tau\in\mathrm{FPFI}(S)$ on the set $S$.

\item [$(b)$] Define the bilinear map
\begin{equation}\label{thepairinginFPFI}
\langle\!\langle \  ,\ \rangle\!\rangle: \mathbb{Z}\mathrm{FPFI}(S)\times \mathbb{Z}\mathrm{FPFI}(S) \longrightarrow \mathbb{Z}[X]
\end{equation}
 by $\langle\!\langle \sigma, \tau \rangle\!\rangle=X^{(\!( \sigma, \tau )\!)}$ for any $\sigma,\tau\in\mathrm{FPFI}(S)$. 
\end{itemize}
\end{definition}

\begin{lemma}\label{equinFPFI} Let $S$ be  a set with $|S|$ even and $\sigma_0 \in \mathrm{FPFI}(S)$. Then
\begin{equation}
\sum_{\sigma\in\mathrm{FPFI}(S)}\langle\!\langle \sigma_0 , \sigma\rangle\!\rangle= X(X+2)\cdots (X+2(|S|-1)).
\end{equation}
\end{lemma}

\begin{proof}
For a polynomial $Q(X)$ we denote by $c_{X^k}(Q(X))$ the coefficient of $X^k$ in $Q(X)$. Set $m:=|S|/2$, we may assume that $S=[\![1,2m]\!]$ and $\sigma_0=(12)\cdots(2m-1,2m)\in\mathrm{FPFI}([\![1,2m]\!])$. We have
\begin{gather*}
\sum_{\sigma\in\mathrm{FPFI}([\![1,2m]\!])}\langle\!\langle \sigma_0 , \sigma\rangle\!\rangle
  =\sum_{k\geq 0}\Big|\big\{\sigma\in\mathrm{FPFI}([\![1,2m]\!])\ | \ \langle\!\langle\sigma_0, \sigma\rangle\!\rangle = \left|[\![1,2m]\!]/\langle\sigma,\tau\rangle\right|=k\big\}\Big|X^k.\\
\end{gather*}

Besides, it is well know that
\begin{equation}\label{stanley}
\sum_{k\geq 0}\Big|\big\{\sigma\in \mathfrak{S}_m \ | \ \sigma \text{ has } k \text{ cycles}\big\}\Big|X^k = X(X+1)\cdots (X+(m-1)),
\end{equation}
see for instance~\cite[Proposition~1.3.4]{Stanley}.  Hence,
\begin{gather*}
\begin{split}
c_{X^k}\big(X(X+2)\cdots(X+2(m-1))\big)&=2^{m-k}c_{X^k}\big(X(X+1)\cdots (X+(m-1))\big)\\
& =2^{n-k}\Big|\big\{\sigma\in \mathfrak{S}_m \ |\ \sigma \text{ has } k \text{ cycles}\big\}\Big|.
\end{split}
\end{gather*}
Therefore, we need to show the equality
\begin{equation}\label{final-equality}
\Big|\big\{\sigma\in\mathrm{FPFI}([\![1,2m]\!])\ | \ \langle\!\langle\sigma_0, \sigma\rangle\!\rangle = \left|[\![1,2m]\!]/\langle\sigma_0,\sigma\rangle\right|=k\big\}\Big| = 2^{m-k}\Big|\big\{\sigma\in \mathfrak{S}_m \ |\ \sigma \text{ has } k \text{ cycles}\big\}\Big|.
\end{equation}

We set
\begin{equation}
A=\Big\{\big(\sigma,(x_{\omega})_{\omega\in [\![1,2m]\!]/\langle\sigma,\sigma_0\rangle}\big) \ \big| \ \sigma\in\mathrm{FPFI}([\![1,2m]\!]) \text{ and } \forall\omega\in [\![1,2m]\!]/\langle\sigma,\sigma_0\rangle : x_{\omega}\in \omega/\langle\sigma\sigma_0\rangle  \Big\},
\end{equation}
and
\begin{equation}
B=\left\{\begin{array}{l|l}
         &\tau\in\mathfrak{S}_{[\![1,2m]\!]/\langle\sigma_0\rangle} \text{ and } \\
(\tau,\lambda)          &  \lambda: [\![1,2m]\!]/\langle\sigma_0\rangle\to [\![1,2m]\!] \text{ is a section}\\
 & \text{of the projection }\pi:[\![1,2m]\!]\to [\![1,2m]\!]/\langle\sigma_0\rangle
        \end{array}\right\}.
\end{equation}

Let us construct a map $f:A\to B$. Let $\big(\sigma, (x_\omega)_{\omega\in [\![1,2m]\!]/\langle\sigma,\sigma_0\rangle }\big)\in A$. Each $\omega\in [\![1,2m]\!]/\langle\sigma,\sigma_0\rangle$ is an orbit of $\langle\sigma,\sigma_0\rangle$ in $[\![1,2m]\!]$, therefore is a $\langle\sigma,\sigma_0\rangle$-invariant subset of $[\![1,2m]\!]$. For each $\omega \in [\![1,2m]\!]/\langle\sigma,\sigma_0\rangle$, $x_{\omega}$ is an orbit of the action of $\langle \sigma\sigma_0\rangle$ on $\omega$ therefore it is a $\langle \sigma\sigma_0\rangle$-invariant subset of $\omega$. The various $\omega$ in $[\![1,2m]\!]/\langle\sigma,\sigma_0\rangle$ are disjoint subsets of $[\![1,2m]\!]$ therefore the various $x_{\omega}$, $\omega\in [\![1,2m]\!]/\langle\sigma,\sigma_0\rangle$ are disjoint $\langle \sigma\sigma_0\rangle$-invariant subsets of $[\![1,2m]\!]$. Their union $\cup_{\omega \in [\![1,2m]\!]/\langle\sigma,\sigma_0\rangle}x_{\omega}$ is therefore a $\langle \sigma\sigma_0\rangle$-invariant subset of $[\![1,2m]\!]$. Hence, the inclusion map  $$\iota:\bigsqcup_{\omega\in [\![1,2m]\!]/\langle\sigma,\sigma_0\rangle}x_{\omega}\hooklongrightarrow [\![1,2m]\!]$$ is injective and equivariant under the action of $\sigma\sigma_0$. 

Since each $x_{\omega}$ is $\langle \sigma\sigma_0\rangle$-invariant, then the union $\sqcup_{\omega\in[\![1,2m]\!]/\langle\sigma,\sigma_0\rangle}x_{\omega}$ is equipped with a permutation denoted $(\sigma\sigma_0)_{|\sqcup_{\omega}x_{\omega}}$.

Consider the map
\begin{equation}\label{equ:0315themapb}
b:=\pi \circ \iota: \bigsqcup_{\omega\in([\![1,2m]\!]/\langle\sigma,\sigma_0\rangle)}x_{\omega}\hooklongrightarrow [\![1,2m]\!]\longrightarrow [\![1,2m]\!]/\langle\sigma_0\rangle.
\end{equation}
Notice that each $\omega$ is a finite set equipped with a transitive action of the group $\mathbb{Z}\rtimes \mathbb{Z}_2$, where the generator of $\mathbb{Z}$ acts by $\sigma\sigma_0$ and the generator of $\mathbb{Z}_2$ acts by $\sigma_0$. One can show that such $\mathbb{Z}\rtimes \mathbb{Z}_2$-sets correspond bijectively to $\mathbb{Z}_{\geq 0}$, the $\mathbb{Z}\rtimes \mathbb{Z}_2$-set corresponding to $k\in\mathbb{Z}_{\geq 0}$ being $\mathbb{Z}_k\times\{1,-1\}$, the action of $\sigma_0$ being $\mathrm{Id}\times (\epsilon\mapsto -\epsilon)$ and the action of $\sigma\sigma_0$ being given by $(a,\epsilon)\mapsto (a+\epsilon,\epsilon)$. It follows from this description that such a $\mathbb{Z}\rtimes \mathbb{Z}_2$-set $E$ is such that $E/\langle\sigma\sigma_0\rangle$ has cardinality $2$ and that for each $x_E \in E/\langle \sigma\sigma_0\rangle$, the composition $x_E\hookrightarrow E\to E/\langle \sigma_0\rangle$ is bijective. In particular, for each $\omega \in [\![1,2m]\!]/\langle\sigma,\sigma_0\rangle$, the composition $x_\omega\hookrightarrow\omega\to \omega/\langle \sigma_0\rangle$ is bijective.  The disjoint union over $\omega \in  [\![1,2m]\!]/\langle\sigma,\sigma_0\rangle$, of these maps is therefore a bijection, and as it identifies with~\eqref{equ:0315themapb}, the latter map is a bijection.

Let $\tau=b\circ (\sigma\sigma_0)_{|\sqcup_{\omega}x_{\omega}}\circ b^{-1}\in \mathfrak{S}_{[\![1,2m]\!]/\langle \sigma_0\rangle}$ and $\lambda=\iota\circ b^{-1}:[\![1,2m]\!]/\langle \sigma_0\rangle \to [\![1,2m]\!]$. Since $b=\pi \circ \iota$ and $b$ is bijective, one has $\pi \circ (\iota \circ b^{-1})=\pi \circ \lambda=\mathrm{Id}_{[\![1,2m]\!]/\langle \sigma,\sigma_0\rangle}$, therefore $\lambda$ is a section of $\pi:[\![1,2m]\!]\to[\![1,2m]\!]/\langle\sigma_0\rangle$. It follows that $(\tau,\lambda) \in B$. One can then define $f: A\to B$ by 
\begin{equation}
f\big(\sigma, (x_\omega)_{\omega\in([\![1,2m]\!]/\langle\sigma,\sigma_0\rangle)}\big)=(\tau,\lambda)\in B.
\end{equation}

Define a map $g:B\to A$ as follows. Let $(\tau,\lambda)\in B$, then $\lambda:[\![1,2m]\!]/\langle\sigma_0\rangle \to [\![1,2m]\!]$ is an injection. Moreover, it follows from $\langle\sigma_0\rangle\simeq \mathbb{Z}_2$ and from the freeness of the action of $\langle\sigma_0\rangle$ on $[\![1,2m]\!]$ that
$$[\![1,2m]\!]=\lambda\left([\![1,2m]\!]/\langle\sigma_0\rangle\right)\sqcup \sigma_0\big(\lambda\left([\![1,2m]\!]/\langle\sigma_0\rangle\right)\big).$$

Therefore, we have a bijection
\begin{equation}
\tilde{b}:\big([\![1,2m]\!]/\langle\sigma_0\rangle\big)\times \{0,1\}\longrightarrow [\![1,2m]\!]
\end{equation}
given by $\tilde{b}(\underline{a},0)=\lambda(\underline{a})$ and $\tilde{b}(\underline{a},1)=\sigma_0\big(\lambda(\underline{a})\big)$ for any $\underline{a}\in [\![1,2m]\!]/\langle\sigma_0\rangle$. 

Besides, consider the permutation $\tilde{\sigma}$ of $([\![1,2m]\!]/\langle\sigma_0\rangle) \times \{0,1\}$ given by  
\begin{equation}
\tilde{\sigma}:\big([\![1,2m]\!]/\langle\sigma_0\rangle\big)\times \{0,1\}\longrightarrow \big([\![1,2m]\!]/\langle\sigma_0\rangle\big)\times \{0,1\}
\end{equation}
given by $\tilde{\sigma}(\underline{a},0)=(\tau^{-1}(\underline{a}),1)$ and $\tilde{\sigma}(\underline{a},1)=(\tau(\underline{a}),0)$ for any $\underline{a}\in [\![1,2m]\!]/\langle\sigma_0\rangle$. Clearly, $\tilde{\sigma}$ is a fixed-point free involution of $\big([\![1,2m]\!]/\langle\sigma_0\rangle\big)\times \{0,1\}$.

Set $\hat{\sigma}:=\tilde{b}\circ \tilde{\sigma}\circ \tilde{b}^{-1}\in\mathrm{FPFI}([\![1,2m]\!])$. For $\omega\in [\![1,2m]\!]/\langle \hat{\sigma},\sigma_0\rangle$, define $x_{\omega}:= \omega\cap\lambda\big([\![1,2m]\!]/\langle\sigma_0\rangle\big)$, i.e., one of the two orbits of the action of $\langle\hat{\sigma}\sigma_{0}\rangle$ on $\omega$. Thus, $x_{\omega}\in \omega/\langle \hat{\sigma}\sigma_0\rangle$. We set
\begin{equation}
g(\tau,\lambda)= \big(\hat{\sigma}, (x_{\omega})_{\omega\in  [\![1,2m]\!]/\langle \hat{\sigma},\sigma_0\rangle}\big)\in A.
\end{equation}

One checks that $f\circ g=\mathrm{Id}_B$ and $g\circ f=\mathrm{Id}_A$, that is, $f$ is bijective.

Consider the maps
\begin{equation}
p_{\mbox{\tiny $A$}}:A\longrightarrow \mathbb{Z}_{\geq 0} \quad \quad \quad \text{ and } \quad \quad \quad p_{\mbox{\tiny $B$}}:B\longrightarrow \mathbb{Z}_{\geq 0}
\end{equation}
given by
$$p_{\mbox{\tiny $A$}}\big(\sigma, (x_\omega)_{\omega\in[\![1,2m]\!]/\langle\sigma,\sigma_0\rangle}\big)=\big|[\![1,2m]\!]/\langle\sigma,\sigma_0\rangle\big|\quad\quad\quad \text{ and } \quad\quad\quad  p_{\mbox{\tiny $B$}}(\tau,\lambda)=\big|([\![1,2m]\!]/\langle\sigma_0\rangle)/\langle \tau\rangle\big|,$$
for any  $\big(\sigma, (x_\omega)_{\omega\in[\![1,2m]\!]/\langle\sigma,\sigma_0\rangle}\big)\in A$ and any $(\tau,\lambda)\in B$.
One checks that the diagram
$$
\xymatrix{
A\ar^{f}[r]
\ar_{p_{\mbox{\tiny $A$}}}[dr] & B\ar^{p_{\mbox{\tiny $B$}}}[d]\\ 
& \mathbb{Z}_{\geq 0}}
$$
commutes. As a consequence, we have 
\begin{equation}\label{A-k=B-k}
|p_{\mbox{\tiny $A$}}^{-1}(k)| = |p_{\mbox{\tiny $B$}}^{-1}(k)|\end{equation}
for any $k\in\mathbb{Z}_{\geq 0}$. Let us compute these cardinalities.

There is  a map $\pi:A \to \mathrm{FPFI}([\![1,2m]\!])$ given by $\pi\big(\sigma, (x_\omega)_{\omega\in [\![1,2m]\!]/\langle\sigma,\sigma_0\rangle }\big)=\sigma$. It follows from the classification of transitive finite $\mathbb{Z}\rtimes \mathbb{Z}_2$-sets that if $\omega$ is an orbit of $\sigma$ in $[\![1,2m]\!]$, then $|\omega/\langle \sigma\sigma_0\rangle|=2$. For $\sigma \in \mathrm{FPFI}([\![1,2m]\!])$, the fiber over $\sigma$ of $A\to \mathrm{FPFI}([\![1,2m]\!])$ is in bijection with $\prod_{\omega \in [\![1,2m]\!]/\langle \sigma,\sigma_0\rangle} \omega/\langle \sigma\sigma_0\rangle$, therefore 
\begin{equation}\label{equ:thecardinality}
\pi^{-1}(\sigma) = 2^{|[\![1,2m]\!]/\langle \sigma,\sigma_0\rangle|}.
\end{equation} 

The map $\pi: A\to\mathrm{FPFI}([\![1,2m]\!])$ induces a map 
\begin{equation}\label{equ:theinducedmap}
p_{\mbox{\tiny $A$}}^{-1}(k)\longrightarrow  \big\{\sigma\in\mathrm{FPFI}([\![1,2m]\!])\ | \ \langle\!\langle\sigma_0, \sigma\rangle\!\rangle = \left|[\![1,2m]\!]/\langle\sigma_0,\sigma\rangle\right|=k\big\}
\end{equation}
Then by \eqref{equ:thecardinality}, the cardinality of the fiber of the map~\eqref{equ:theinducedmap} is $2^k$. Hence,
\begin{equation}\label{card-of-A-k}
|p_{\mbox{\tiny $A$}}^{-1}(k)| = 2^k\Big|\big\{\sigma\in\mathrm{FPFI}([\![1,2m]\!])\ | \ \langle\!\langle\sigma_0, \sigma\rangle\!\rangle = \left|[\![1,2m]\!]/\langle\sigma_0,\sigma\rangle\right|=k\big\}\Big|.
\end{equation}

Besides,  there are exactly $2^{\big|([\![1,2m]\!]/\langle \sigma_0\rangle)\big|}=2^n$  sections of the projection $\pi:[\![1,2m]\!]\to [\![1,2m]\!]/\langle\sigma_0\rangle$. Since
$$p_B^{-1}(k)=\big\{\tau\in\mathfrak{S}_{[\![1,2m]\!]/\langle\sigma_0\rangle} \ | \ \big|([\![1,2m]\!]/\langle\sigma_0\rangle)/\langle \tau\rangle\big|= k\big\} \times \{\text{sections of } \pi : [\![1,2m]\!]\to [\![1,2m]\!]/\langle \sigma_0\rangle\},$$
one has  
\begin{equation}\label{card-of-B-k}
\begin{split}
|p_{\mbox{\tiny $B$}}^{-1}(k)| &= 2^m\Big|\big\{\tau\in\mathfrak{S}_{[\![1,2m]\!]/\langle\sigma_0\rangle}\ | \ \big|([\![1,2m]\!]/\langle\sigma_0\rangle)/\langle \tau\rangle\big|= k\big\}\Big|\\
 &= 2^m\Big|\big\{\tau\in \mathfrak{S}_m\ | \ \tau \text{ has } k \text{ cycles }\big\}\Big|.
\end{split}
\end{equation}

From \eqref{A-k=B-k}, \eqref{card-of-A-k} and \eqref{card-of-B-k} we deduce \eqref{final-equality}, which completes the proof.
\end{proof}

\begin{proposition}\label{degree-0-of-Pn}
Let $X$ denote  the connected Jacobi diagram on $\vec{\varnothing}$ without trivalent vertices (i.e. a dashed loop). Then

$(a)$ for $n\geq 2$, the ideal $\mathring{\mathcal{A}}^\wedge_0[0]\cap P^n_0$  of $\mathring{\mathcal{A}}^\wedge_0[0]$ is generated by $X(X+2)\cdots (X+2(n-1))$. Hence,  there is an algebra isomorphism
\begin{equation}\label{eq:degree-0-of-Pn}
\left(\frac{\mathring{\mathcal{A}}^\wedge_0}{P^{n}_0}\right)[0]= \frac{\mathring{\mathcal{A}}^\wedge_0[0]}{\mathring{\mathcal{A}}^\wedge_0[0]\cap P^{n}_0} \simeq \frac{\mathbb{C}[X]}{( X(X+2)\cdots (X+2(n-1)))},
\end{equation}

$(b)$ for $n\geq 1$, there are algebra isomorphisms 
\begin{equation}\label{eq:degree-0-of-Pn-On}
\left(\frac{\mathring{\mathcal{A}}^\wedge_0}{P^{n+1}_0 + O^n}\right)[0] \simeq \frac{\mathbb{C}[X]}{(X+2n)} \simeq \mathbb{C},
\end{equation}
where the first isomorphism is induced by \eqref{eq:degree-0-of-Pn} and the second isomorphism is given by evaluation at~$-2n$. Here~$O^n$ is as in Definition~\ref{def:ideal-O-n}.

\end{proposition}

\begin{proof}
$(a)$ Recall that 
$$P^n_0 =\mathrm{Im}\big(\mathring{\mathcal{A}}^{\wedge}(\vec{\varnothing}, \emptyset, [\![1,2n]\!])\xrightarrow{\ \langle\ ,\varsigma_n \rangle\ }\mathring{\mathcal{A}}^{\wedge}_0\big),$$
where $\langle\ ,\ \rangle$ is the graded pairing~ \eqref{thepairinginA} and $\varsigma_n\in\mathring{\mathcal{A}}^\wedge(\vec{\varnothing},\emptyset, [\![1,2n]\!])$ is as in Definition~\ref{the-element-varsigma-k}. Since $\varsigma_n$ has degree~$0$, then the map $\langle\ ,\varsigma_n\rangle$ is of degree~$0$. Hence,
\begin{equation}\label{2022-04-07-equ1}
P^n_0\cap\mathring{\mathcal{A}}^{\wedge}_0[0] =\mathrm{Im}\big(\mathring{\mathcal{A}}^{\wedge}(\vec{\varnothing},\emptyset,[\![1,2n]\!])[0]\xrightarrow{\ \langle\ ,\varsigma_n \rangle\ }\mathring{\mathcal{A}}^{\wedge}_0[0]\big)= \langle \mathring{\mathcal{A}}^\wedge(\vec{\varnothing}, \emptyset, [\![1,2n]\!])[0], \varsigma_n\rangle.
\end{equation}

Besides,
\begin{equation}\label{2022-04-07-equ2}
\mathring{\mathcal{A}}^{\wedge}(\vec{\varnothing},\emptyset,[\![1,2n]\!])[0]=\bigoplus_{\sigma\in\mathrm{FPFI}([\![1,2n]\!])}\mathbb{C}[X][\mathring{\mathrm{Jac}}(\sigma)],
\end{equation}
where $\mathring{\mathrm{Jac}}$ is as in \eqref{equ:JacandFPFI} and  $[\ ]:\mathring{\mathrm{Jac}}(\vec{\varnothing},\emptyset,[\![1,2n]\!])\to \mathring{\mathcal A}^\wedge(\vec{\varnothing}, \emptyset,[\![1,2n]\!])$ denotes the projection. We continue to denote  by $\mathring{\mathrm{Jac}}$ the linear extension of the  map  from \eqref{equ:JacandFPFI}. Notice that $\mathring{\mathrm{Jac}}(\varsigma_n)=\tilde{\varsigma_n}$, where
\begin{equation*}\label{the-element-tildevarsigma}
\tilde{\varsigma}_n:=\sum_{\sigma\in\mathrm{FPFI}([\![1,2n]\!])}\sigma \in \mathbb{Z}\mathrm{FPFI}([\![1,2n]\!]).
\end{equation*}
Thus, for any $Q(X)\in\mathbb{C}[X]=\mathring{\mathcal{A}}^\wedge_0[0]$ and $\sigma\in\mathrm{FPFI}([\![1,2n]\!])$, we have
\begin{equation}\label{2022-04-07-equ3}
\langle Q(X)\mathring{\mathrm{Jac}}(\sigma),\varsigma_n\rangle = Q(X)\langle \mathring{\mathrm{Jac}}(\sigma),\mathring{\mathrm{Jac}}(\tilde{\varsigma}_n)\rangle =  Q(X)\langle\!\langle \sigma,\tilde{\varsigma}_n\rangle\!\rangle, 
\end{equation}
where $\langle\!\langle \  , \ \rangle\!\rangle$ is the pairing~\eqref{thepairinginFPFI} and  $\langle \ , \  \rangle$ is the pairing~\eqref{thepairinginA}.

For $g\in \mathfrak{S}_{2n}$ define the map $\mathrm{Ad}_g: \mathrm{FPFI}([\![1,2n]\!]) \to \mathrm{FPFI}([\![1,2n]\!])$ by $\mathrm{Ad}_g(\sigma):=g\sigma g^{-1}$ for any $\sigma\in \mathrm{FPFI}([\![1,2n]\!])$. Then, if $\sigma,\tau\in\mathrm{FPFI}([\![1,2n]\!])$ and $g\in \mathfrak{S}_{2n}$, we have $\langle\!\langle \mathrm{Ad}_g(\sigma),\mathrm{Ad}_g(\tau) \rangle\!\rangle= \langle\!\langle  \sigma,\tau \rangle\!\rangle$.

Fix $\sigma_0\in\mathrm{FPFI}([\![1,2n]\!])$, then
\begin{equation}\label{equ:tildevarsigma202203}
\tilde{\varsigma}_n =\sum_{\sigma\in\mathrm{FPFI}([\![1,2n]\!])}\sigma = \frac{1}{2^n n!}\sum_{g\in \mathfrak{S}_{2n}}\mathrm{Ad}_g(\sigma_0). 
\end{equation}
The second equality in \eqref{equ:tildevarsigma202203} follows from the fact that the conjugation action of $\mathfrak{S}_{2n}$ on $\mathrm{FPFI}([\![1,2n]\!])$ is transitive, the orbit of each point being isomorphic to the semidirect product  $\mathfrak{S}_n \ltimes \mathbb{Z}_2^n$.

If $\sigma\in \mathrm{FPFI}([\![1,2n]\!])$, then there exists $h\in \mathfrak{S}_{2n}$ such that $\sigma=\mathrm{Ad}_h(\sigma_0)$. Therefore,
\begin{equation}\label{2022-04-07-equ4}
\begin{split}
\langle\!\langle \sigma,\tilde{\varsigma}_n\rangle\!\rangle = \left\langle\!\left\langle \mathrm{Ad}_h(\sigma_0),\frac{1}{2^n n!}\sum_{g\in \mathfrak{S}_{2n}}\mathrm{Ad}_g(\sigma_0)\right\rangle\!\right\rangle = \left\langle\!\left\langle \sigma_0,\frac{1}{2^n n!}\sum_{g\in \mathfrak{S}_{2n}}\mathrm{Ad}_{h^{-1}g}\sigma_0\right\rangle\!\right\rangle \\
=\sum_{\sigma\in\mathrm{FPFI}([\![1,2n]\!])}\langle\!\langle \sigma_0,\sigma \rangle\!\rangle =X(X+2)\cdots (X+2(n-1)),\\
\end{split}
\end{equation}
where the last equality  follows from Lemma~\ref{equinFPFI}. Combining \eqref{2022-04-07-equ1}, \eqref{2022-04-07-equ3} and \eqref{2022-04-07-equ4}, one obtains 
$$P^n_0\cap\mathring{\mathcal{A}}^\wedge_0[0]=\langle \mathring{\mathcal{A}}^\wedge(\vec{\varnothing},\emptyset,[\![1,2n]\!])[0], \varsigma_n\rangle=(X(X+2)\cdots(X+2(n-1))).$$

$(b)$ Since $O^n$ is generated by $X+2n$ (see Definition~\ref{def:ideal-O-n}), the isomorphism \eqref{eq:degree-0-of-Pn} induces an isomorphism of $\left(\mathring{\mathcal A}^\wedge_0/(P^{n+1}_0 + O^n)\right)[0]$  with the quotient of $\mathbb{C}[X]$ by the sum of its ideals $(X+2n)$ and $(X(X+2)\cdots(X+2n))$, which is equal to the ideal $(X+2n)$.

\end{proof}

\begin{lemma}\label{lemma:deg-of-j-n} For any integer $r\geq 1$,  the algebra homomorphism $j^{\mathrm{LMO}}_r:\bigoplus_{k\geq 0} \mathring{\mathcal{A}}^\wedge_k\to \mathring{\mathcal{A}}^\wedge_0$ is homogeneous for the degrees $\mathrm{deg}-r\mathrm{deg}_{\mathrm{cir}}$ for the source space and $\mathrm{deg}$ for the target space.
\end{lemma}

\begin{proof}
The results follows from \eqref{equ:2023-10-10}.
\end{proof}

\begin{lemma}\label{reduction-calcul-j-n-cs-Z-U} For any integer $r\geq 1$ let $\theta_r\in\mathring{\mathcal{A}}^\wedge_1 =\mathring{\mathcal{A}}^\wedge(\vec{\varnothing},\{1\})$ denote the class of the Jacobi diagram  consisting of $r$ ``parallel'' chords, see Figure~\ref{figure-theta-n} for some examples. 
\begin{figure}[ht]
		\centering
         \includegraphics[scale=1]{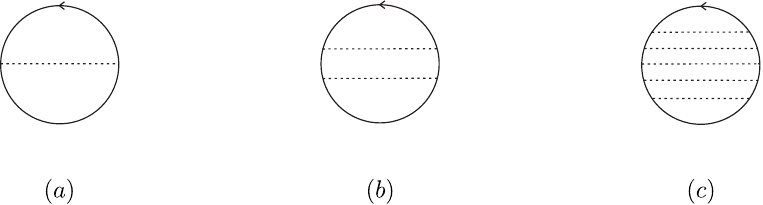}
\caption{The elements $\theta_1, \theta_2, \theta_5\in \mathring{\mathcal{A}}^\wedge_1$.}
\label{figure-theta-n} 								
\end{figure}

Then 
\begin{equation}
j^{\mathrm{LMO}}_r(\mathrm{cs}^{\nu}\circ {Z}_{\mathcal A}(U^{\pm}))[0] =\frac{1}{r!}\left(\pm\frac{1}{2}\right)^r j^{\mathrm{LMO}}_r(\theta_r)\in \mathring{\mathcal{A}}^\wedge_0[0],
\end{equation}
where ${Z}_{\mathcal A}$ is the Kontsevich integral (see Theorem~\ref{thmkontsevichintegral}).
\end{lemma}

\begin{proof} Notice that  $\theta_r \in\mathring{\mathcal{A}}^\wedge_1[r]$, therefore, by Lemma~\ref{lemma:deg-of-j-n}, $j^{\mathrm{LMO}}_r(\theta_r)\in \mathring{\mathcal{A}}^\wedge_0[0]$. Let $\mathbb{C}\mathring{\mathrm{Jac}}(\vec{\varnothing},\{1\})[r]$ be the subspace of $\mathbb{C}\mathring{\mathrm{Jac}}(\vec{\varnothing},\{1\})$ generated by  pairs $\underline{D}=(D,\varphi, \mathrm{cyc}_1)$ (see Defintion \ref{def:thesetJacPS}) such that $\mathrm{deg}(\underline{D})=\frac{1}{2}(|\mathrm{triv}({D})| + |\partial {D}|)=r$, where $\mathrm{triv}({D})$ (resp. $\partial D$) denotes the set of trivalent (resp. univalent) vertices of $D$, see  \eqref{degree-on-APS}. Then $\mathring{\mathcal{A}}^\wedge_1[r]$ is a quotient of $\mathbb{C}\mathring{\mathrm{Jac}}(\vec{\varnothing},\{1\})[r]$. 

If $|\mathrm{triv}(D)| \geq 1$, then $j_r^{\mathrm{LMO}}([\underline{D}])=0$. Indeed, by Lemma~\ref{r:jLMOanddoubling}, $j_r^{\mathrm{LMO}}([\underline{D}])=\frac{1}{r!}j_1^{\mathrm{LMO}}([\underline{D}^{(r)}])$ and,  by its definition, $\underline{D}^{(r)}$ is a sum of elements $\underline{E}=(E,\psi,\{\mathrm{cyc}_{(1,i)}\}_{i\in [\![1,r]\!]})$ of $\mathring{\mathrm{Jac}}(\vec{\varnothing},[\![1,r]\!])$ with $\partial E \simeq \partial D$, see Definition~\ref{def:doublingfree}. Since $|\partial D|<2r$, each one of the maps $\psi : \partial E\to [\![1,r]\!]$ is such that that the fiber of at least one element has cardinality less or equal than $1$, which implies $j_1^{\mathrm{LMO}}(\underline{E})=0$ as $T_0=T_1=1$. Therefore, $j_r^{\mathrm{LMO}}([\underline{D}])=0$.  Besides, we have 
\begin{equation}
\begin{split}
\mathrm{cs}^\nu\circ {Z}_{\mathcal A}(U^{\pm}) &= \mathrm{cs}^{\nu}\left(\mathrm{cs}^\nu\left(\mathrm{exp}_{\#}\left(\pm\frac{1}{2}\theta_1\right)\right)\right)\\
& = \mathrm{exp}_{\#}\left(\pm\frac{1}{2} \theta_1 + \big(\text{sum of connected diagrams with } \geq 1 \text{ trivalent vertices} \big)\right),
\end{split}
\end{equation}
where $\#$ denotes the product (connected sum)\footnote{Notice that the product $\#$ in $\mathring{\mathcal{A}}^\wedge_1=\mathring{\mathcal{A}}^\wedge(\vec{\varnothing},\{1\})$ is the product induced by the composition in $\mathring{\mathcal{A}}^\wedge(\uparrow,\emptyset)$ using the  continuous linear isomorphism $\mathring{\mathcal{A}}^\wedge(\uparrow, \emptyset) \simeq \mathring{\mathcal{A}}^\wedge_1$, see \cite[Proposition~6.3]{Ohts}.} in $\mathring{\mathcal{A}}^\wedge_1$. We then have
\begin{equation*}
\begin{split}
 j_r^{\mathrm{LMO}}\left(\mathrm{cs}^\nu\circ {Z}_{\mathcal A}(U^{\pm})\right) &=  \mathrm{exp}_{\#}\left(\pm\frac{1}{2} \theta_1 + \big(\text{sum of connected diagrams with } \geq 1 \text{ trivalent vertices} \big)\right)\\
 & = j_r^{\mathrm{LMO}}\left(\mathrm{exp}_{\#}\left(\pm\frac{1}{2} \theta_1\right)\right)
\end{split}
\end{equation*}
Hence, by taking the degree $0$ part in both sides of the above equation, using Lemma~\ref{lemma:deg-of-j-n} and the fact that $\theta_1^{\#r}=\theta_r$, we obtain
\begin{equation*}
\begin{split}
j_r^{\mathrm{LMO}}\left(\mathrm{cs}^{\nu} \circ {Z}_{\mathcal A}(U^{\pm})\right)[0]
& =j_r^{\mathrm{LMO}}\left(\mathrm{exp}_{\#}\left(\pm\frac{1}{2}\theta_1\right)\right)[0] = \sum_{k\geq 0}\frac{1}{k!}\left(\pm\frac{1}{2}\right)^k j_r^{\mathrm{LMO}}(\theta_1^{\# k})[0]\\
&=\frac{1}{r!}\left(\pm\frac{1}{2}\right)^r j_r^{\mathrm{LMO}}(\theta_1^{\# r}) = \frac{1}{r!}\left(\pm\frac{1}{2}\right)^rj_r^{\mathrm{LMO}}(\theta_r)\\
\end{split}
\end{equation*}
which finishes the proof.
\end{proof}

\begin{lemma}\label{calcul-of-j-n-theta-n} For any integer $r\geq 1$ we have
\begin{equation}\label{eq:j-n-theta-n}
j^{\mathrm{LMO}}_r(\theta_r) = X(X+2)\cdots (X+2(r-1))\in \mathring{\mathcal{A}}^\wedge_0[0],
\end{equation}
where $X$ denotes the dashed loop (as in Proposition~\ref{degree-0-of-Pn}).
\end{lemma}
\begin{proof}
We will use Lemma~\ref{r:jLMOanddoubling}. Let us describe $\theta_r^{(r)}$ explicitly. Let us label the chords of $\theta_r$ from $1$
 to $r$ and for each chord we label its endpoints with $s$ and $t$.

 Then we have a bijection $\{\text{endpoints of the } r \text{ chords}\}\simeq [\![1,r]\!]\times\{s, t\}$. Let $\mathrm{cyc}$ be the induced cyclic order on  $[\![1,r]\!]\times\{s, t\}$ by this bijection and the cyclic order deduced from the diagrammatic representation of $\theta_r$, for instance, if in Figure~\ref{figure-theta-n}$(b)$ we label the chords from top to bottom with $1$ and $2$ and from left to right with $s$ and $t$, then the cyclic order $\mathrm{cyc}$    on $[\![1,2]\!]\times\{s, t\}$ is given by $(1,s)\leq (2,s)\leq (2,t)\leq (1,t)\leq (1,s)$. Let $\mathrm{Map}([\![1,r]\!]\times\{s, t\}, [\![1,r]\!])$ be the sets of maps $\varphi:[\![1,r]\!]\times\{s, t\} \to [\![1,r]\!]$. For such a $\varphi$ let $\mathrm{diag}(\varphi)$ be  the class in $\mathring{\mathcal{A}}^{\wedge}(\vec{\varnothing}, [\![1,r]\!])$ of the Jacobi diagram $(\theta_r, \varphi, \{\mathrm{cyc}_i\}_{i\in[\![1,r]\!]})$ where $\mathrm{cyc}_i$ is the cyclic order on $\varphi^{-1}(i)\subset [\![1,r]\!]\times\{s, t\}$ induced by $\mathrm{cyc}$. Then we have
\begin{equation}\label{equ:2023-10-12}
\theta_r^{(r)}= \sum_{\varphi} \mathrm{diag}(\varphi)\in\mathring{\mathcal{A}}^{\wedge}(\vec{\varnothing}, [\![1,r]\!]),
\end{equation}
where the sum runs over the set $\mathrm{Map}([\![1,r]\!]\times\{s, t\}, [\![1,r]\!])$. 

We know that if $\varphi$ is such that there is $c\in [\![1,r]\!]$ with $|\varphi^{-1}(c)|\leq 1$, then $j_1^{\mathrm{LMO}}(\mathrm{diag}(\varphi))=0$ (because $T_0=T_1=0$).
Set 
$$\mathrm{Map}'([\![1,r]\!]\times\{s, t\}, [\![1,r]\!]):=\{\varphi\in \mathrm{Map}([\![1,r]\!]\times\{s, t\}, [\![1,r]\!]) \ | \ \forall c\in [\![1,r]\!]: |\varphi^{-1}(c)|\geq 2\}.$$
Then by Lemma~\ref{r:jLMOanddoubling} and \eqref{equ:2023-10-12}, we have
\begin{equation}\label{2022-04-03-equ0}
j^{\mathrm{LMO}}_r(\theta_r) = \frac{1}{r!}j^{\mathrm{LMO}}_1(\theta_r^{(r)})  = \sum_{\varphi\in \mathrm{Map}'([\![1,r]\!]\times\{s, t\}, [\![1,r]\!])} j^{\mathrm{LMO}}_1(\mathrm{diag}(\varphi)).
\end{equation}

Consider the map 
\begin{equation}\label{2022-04-05-equ0}
\Phi:\mathrm{Map}'([\![1,r]\!]\times\{s, t\}, [\![1,r]\!])\longrightarrow \mathrm{FPFI}([\![1,r]\!]\times\{s, t\})
\end{equation}
given as follows. For $\varphi\in \mathrm{Map}'([\![1,r]\!]\times\{s, t\}, [\![1,r]\!])$, we have $\sum_{c\in [\![1,r]\!]}|\varphi^{-1}(c)|=2r$, which implies that $|\varphi^{-1}(c)|=2$ for any $c\in [\![1,r]\!]$. Define $\Phi(\varphi):[\![1,r]\!]\times\{s, t\}\to [\![1,r]\!]\times\{s, t\}$ by the condition $\varphi^{-1}(\varphi(x))=\{x,\Phi(\varphi)(x)\}$ for all $x\in [\![1,r]\!]\times\{s, t\}$.

Let $\sigma\in\mathrm{FPFI}([\![1,r]\!]\times\{s, t\})$, let us compute the cardinality $\big|\Phi^{-1}(\sigma)\big|$. The quotient set $([\![1,r]\!]\times\{s, t\})/\langle \sigma \rangle$ has cardinality $r$. Hence
\begin{equation}\label{2022-04-03-equ1}
\big|\mathrm{Bij}(([\![1,r]\!]\times\{s, t\})/\langle \sigma \rangle, [\![1,r]\!])\big| =r!,
\end{equation}
where $\mathrm{Bij}(([\![1,r]\!]\times\{s, t\})/\langle \sigma \rangle, [\![1,r]\!])$ is the set of bijective maps from $([\![1,r]\!]\times\{s, t\})/\langle \sigma \rangle$ to  $[\![1,r]\!]$. If $\varphi\in\Phi^{-1}(\sigma)$, then we have a commutative diagram
$$
\xymatrix{
[\![1,r]\!]\times\{s, t\}\ar^{}[r]
\ar_{\varphi}[dr] & \frac{[\![1,r]\!]\times\{s, t\}}{\langle \sigma \rangle}\ar^{\overline{\varphi}}[d]\\ 
& [\![1,r]\!],}
$$
where $\overline{\varphi}$ is a bijection. One checks that the map $\varphi\mapsto \overline{\varphi}$ defines a bijective map 
$\Phi^{-1}(\sigma)\to \mathrm{Bij}(([\![1,r]\!]\times\{s, t\})/\langle \sigma \rangle, [\![1,r]\!])$. Therefore, by \eqref{2022-04-03-equ1} we have
\begin{equation}\label{2022-04-03-equ2}
\big|\Phi^{-1}(\sigma)\big| =r!.
\end{equation} 

Let $\sigma_0\in\mathrm{FPFI}([\![1,r]\!]\times\{s, t\})$ defined by $\sigma_0(x,s)=(x,t)$ and $\sigma_0(x,t)=(x,s)$ for any $x\in [\![1,r]\!]$. Recall from \eqref{defequ:pairingparpar} in Definition~\ref{def:pairingparpar} that for $\sigma\in\mathrm{FPFI}([\![1,r]\!]\times\{s, t\})$
 $$(\!(\sigma_0,\sigma)\!)=\left|\frac{[\![1,r]\!]\times\{s, t\}}{\langle \sigma_0, \sigma \rangle}\right|.$$

Let  $\varphi\in \mathrm{Map}'([\![1,r]\!]\times\{s, t\}, [\![1,r]\!])$ and let us compute $j^{\mathrm{LMO}}_1(\mathrm{diag}(\varphi))$. The chords in $\mathrm{diag}(\varphi)$ connect the pairs $(x,\sigma_0(x))$ for $x\in [\![1,r]\!]\times\{s, t\}$. Besides, $j^{\mathrm{LMO}}_1(\mathrm{diag}(\varphi))$ is obtained from $\mathrm{diag}(\varphi)$ by connecting with chords (i.e., with $T_2$) pairs of univalent vertices lying on a circle and then erasing the circle (see Figure~\ref{the-map-j-1}). Therefore, we are connecting the pairs $(x,\sigma(x))$ for $x\in [\![1,r]\!]\times\{s, t\}$ where $\sigma=\Phi(\varphi)$. Since $\sigma=\Phi(\varphi)$ and $\sigma_0$ are involutions we have a bijection between the set of connected components of $j^{\mathrm{LMO}}_1(\mathrm{diag}(\varphi))$ and the quotient set $([\![1,r]\!]\times\{s, t\})/\langle \sigma, \sigma_0\rangle$. Such connected components are dashed circles, hence 
\begin{equation}\label{2022-04-03-equ3}
j^{\mathrm{LMO}}_1(\mathrm{diag}(\varphi)) = X^{(\!(\sigma_0,\Phi(\varphi))\!)}=\langle\!\langle \sigma_0,\Phi(\varphi)\rangle\!\rangle,
\end{equation}
where $\langle\!\langle \ , \ \rangle\!\rangle$ is as in Definition~\ref{def:pairingparpar}$(b)$. We conclude that
\begin{equation*}
\begin{split}
j^{\mathrm{LMO}}_r(\theta_r) = \frac{1}{r!}\sum_{\varphi\in \mathrm{Map}'([\![1,r]\!]\times\{s, t\}, [\![1,r]\!])} j^{\mathrm{LMO}}_1(\mathrm{diag}(\varphi))
 = \frac{1}{r!}\sum_{\varphi\in \mathrm{Map}'([\![1,r]\!]\times\{s, t\}, [\![1,r]\!])} \langle\!\langle\sigma_0, \Phi(\varphi)\rangle\!\rangle\\
  = \frac{1}{r!}\sum_{ \sigma\in\mathrm{FPFI}( [\![1,r]\!]\times\{s, t\} ) }|\Phi^{-1}(\sigma)|\langle\!\langle\sigma_0, \Phi(\varphi)\rangle\!\rangle
 =  \sum_{\sigma\in\mathrm{FPFI}([\![1,r]\!]\times\{s, t\})}\langle\!\langle\sigma,\sigma_0\rangle\!\rangle\\
 = X(X+2)\cdots (X+2(r-1)).\\
 \end{split}
\end{equation*}
We have used~\eqref{2022-04-03-equ0} in the first equality, \eqref{2022-04-03-equ3} in the second equality, the map \eqref{2022-04-05-equ0} in the third equality, \eqref{2022-04-03-equ2} in the fourth equality, and Lemma~\eqref{equinFPFI} in the last equality.
\end{proof}

\begin{proposition}\label{calcul-j-n-cs-Z-U-pm-deg0} For any integer $r\geq 1$ we have
\begin{equation}
j^{\mathrm{LMO}}_r(\mathrm{cs}^{\nu}\circ {Z}_{\mathcal A}(U^{\pm}))[0] =\frac{1}{r!}\left(\pm\frac{1}{2}\right)^rX(X+2)\cdots (X+2(r-1)),
\end{equation}
where $X$ denotes the dashed loop (as in Proposition~\ref{degree-0-of-Pn}).
\end{proposition}

\begin{proof}
The result follows from Lemmas~\ref{reduction-calcul-j-n-cs-Z-U} and~\ref{calcul-of-j-n-theta-n}.
\end{proof}

\begin{theorem}\label{mainr-2022-02-14} The pair $(\mathfrak{c}(n), \overline{j}_n\circ \varphi_n\circ\overline{\mathrm{cs}}^{\nu}\circ {Z}_{\mathcal A})$ is a Kirby structure.
\end{theorem}

\begin{proof}
We need to show that the elements  $\overline{j}_n \circ \varphi_n\circ\overline{\mathrm{cs}}^\nu \circ {Z}_{\mathcal A}(U^\pm)$ of the algebra  $\mathfrak{c}(n)$ are invertible (see Definition~\ref{def:kirby}). This algebra is a complete $\mathbb{Z}_{\geq 0}$-graded algebra, with its degree zero component $\mathfrak{c}(n)[0]$ being isomorphic to $\mathbb{C}$ by Proposition~\ref{degree-0-of-Pn}$(b)$. The said invertibilities are then equivalent to that of the images of the elements $(\overline{j}_n \circ \varphi_n \circ \overline{\mathrm{cs}}^\nu \circ {Z}_{\mathcal A}(U^\pm))[0]$ in $\mathbb{C}$. Their images are, by Proposition~\ref{calcul-j-n-cs-Z-U-pm-deg0}, the classes of the polynomials  $\frac{1}{n!(\pm 2)^n}X(X+2)\cdots(X+2(n-1))$ in $\mathbb{C}[X]/(X+2n)$ which are equal to $(\mp 1)^n$ and therefore invertible.
\end{proof}

\subsection{The algebras \texorpdfstring{$\mathfrak{d}(n)$}{dn} and the morphisms \texorpdfstring{$\mathrm{pr}_n : \mathfrak{c}(n) \to \mathfrak{d}(n)$}{cn-to-dn}}\label{sec:3:5-2022-04-18}

We fix an integer $n\geq 1$.  Let $\mathring{\mathcal{A}}^\wedge_0[n]$  \index[notation]{A^\wedge_0[n]@$\mathring{\mathcal{A}}^\wedge_0[n]$} denote the subspace of ${\mathcal{A}}^\wedge_0$ consisting of elements of degree $n$ and  let $\mathring{\mathcal{A}}^\wedge_0[>n]$  \index[notation]{A^\wedge_0[>n]@$\mathring{\mathcal{A}}^\wedge_0[>n]$} denote the complete direct sum of all the components of all the $\mathring{\mathcal{A}}^\wedge_0[k]$ with $k>n$. Since $\mathring{\mathcal{A}}^\wedge_0$ is a complete graded algebra, this is a complete graded ideal  of $\mathring{\mathcal{A}}^\wedge_0$. 

\begin{definition}\label{def:alg-dn} Define the complete graded commutative algebra $\mathfrak{d}(n)$ as the quotient algebra
\begin{equation}
\mathfrak{d}(n)=\frac{\mathring{\mathcal{A}}^\wedge_0}{P^{n+1}_0 + O^n + {\mathring{\mathcal{A}}^\wedge_0[>n]}}.
\end{equation} \index[notation]{d(n)@$\mathfrak{d}(n)$}
\end{definition}

One checks that $\mathfrak{d}(n)$ is the direct sum of its components of degrees~$\leq n$, and is therefore a graded commutative algebra.

\begin{proposition}\label{r2022-04-12-dn} Let $\mathrm{pr}_n:\mathfrak{c}(n)\to \mathfrak{d}(n)$ \index[notation]{pr_n@$\mathrm{pr}_n$} denote the canonical projection. Then the pair  $(\mathfrak{d}(n),\mathrm{pr}_n\circ\overline{j}_n\circ\varphi_n\circ \overline{\mathrm{cs}}^\nu\circ {Z}_{\mathcal A})$ is a Kirby structure.
\end{proposition}
\begin{proof} The result follows from Theorem~\ref{mainr-2022-02-14} and Lemma~\ref{sec:1.8lemma1}$(b)$. 
\end{proof}

\subsection{The equality \texorpdfstring{$\mathfrak{c}(1)= \mathfrak{d}(1)$}{c(1) = d(1)}}\label{sec:5:7}

In this subsection we show that the algebra homomorphism $\mathrm{pr}_1:\mathfrak{c}(1) \to \mathfrak{d}(1)$ is an isomorphism, which implies $\mathfrak{c}(1)= \mathfrak{d}(1)$. 

\begin{lemma}\label{r-2022-05-03-02} Let $x$ be a connected element in $\mathring{\mathrm{Jac}}(\vec{\varnothing},\emptyset)$ of degree $m\geq 3$. Then, modulo AS, IHX and~$P^2_0$ relations, the diagram $x$ is equal to a linear combination $\sum_i x_i$ of elements $x_i\in\mathring{\mathrm{Jac}}(\vec{\varnothing},\emptyset)$ of degree~$m$ such that each $x_i$ has a bigon as a subdiagram, see Figure~\ref{bigon}, in other words each $x_i$ is a graph having a cycle of length~$2$.
\begin{figure}[ht]
		\centering
         \includegraphics[scale=1]{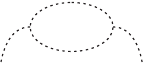}
\caption{Local view of a bigon subdiagram, equivalently a cycle of length $2$, in a Jacobi diagram.}
\label{bigon} 								
\end{figure}
\end{lemma}

\begin{proof} Notice that $x$ has a $k$-cycle with $2\leq k \leq 2m$. That is, a closed path containing $k$ trivalent vertices (and therefore $k$ edges). Up to the AS relation we can suppose that $x$ has the form shown in Figure~\ref{elementxwithkcycle}.
\begin{figure}[ht]
		\centering
         \includegraphics[scale=1]{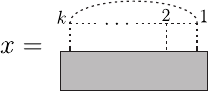}
\caption{The diagram $x$ has a $k$-cycle.}
\label{elementxwithkcycle} 								
\end{figure}

We prove by induction on $k$ that, modulo AS, IHX and $P^2_0$ relations, the diagram $x$ can be written as stated in the lemma. If $k=2$, there is nothing to prove. Suppose $k=3$, consider the instance of the $P^2_0$ relation show in Figure~\ref{instanceofP-2-4}.

\begin{figure}[ht]
		\centering
         \includegraphics[scale=1]{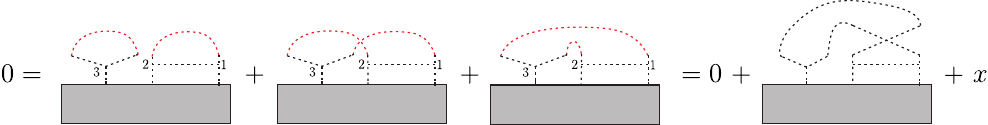}
\caption{Instance of the $P^2_0$ relation. The grey part is as in Figure~\ref{elementxwithkcycle}. The first diagram in the the second part of the equality is equal to~$0$ by the AS relation.}
\label{instanceofP-2-4} 								
\end{figure}

Then, using the IHX relation we have the equality shown in Figure~\ref{elementxafterIHX}.

\begin{figure}[ht]
		\centering
         \includegraphics[scale=1]{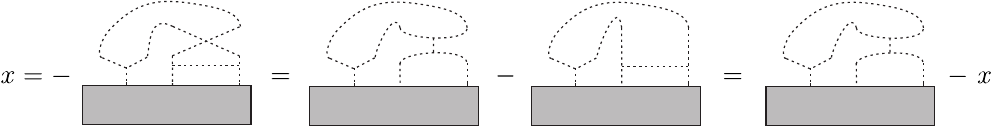}
\caption{Using the IHX relation.}
\label{elementxafterIHX} 								
\end{figure}

The equality in Figure~\ref{elementxafterIHX} implies that, modulo AS, IHX and $P^2_0$ relations, the diagram $x$ can be written in terms of a diagram with a $2$-cycle.

Let us assume $k\geq 4$. Consider the instance of the $P^2_0$ relation shown in Figure~\ref{instanceoP-2-5}.
  
\begin{figure}[ht]
		\centering
         \includegraphics[scale=1]{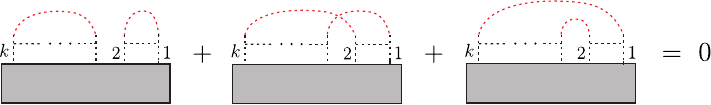}
\caption{Instance of $P^2_0$. Here the grey part is as in Figure~\ref{elementxwithkcycle}. Hence, the last diagram in the left-hand of the equation is the diagram $x$.}
\label{instanceoP-2-5} 								
\end{figure}

We then obtain the equations  shown Figure~\ref{elementxafterIHXgeneralcase}.

\begin{figure}[ht]
		\centering
         \includegraphics[scale=1]{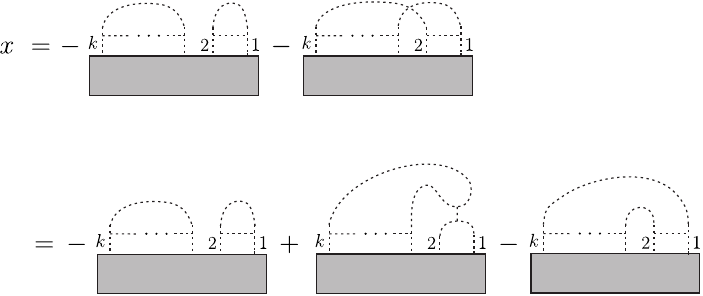}
\caption{In the second equality we have used the IHX relation.}
\label{elementxafterIHXgeneralcase} 								
\end{figure}

The first diagram in the last part of the equations in Figure~\ref{elementxafterIHXgeneralcase} has a $2$-cycle while the second diagram  has a $(k-1)$-cycle and the last diagram is $x$. It follows from this that the diagram $x$, modulo AS, IHX and $P^2_0$ relations, can be written as the linear combination of a diagram with a $2$-cycle and a diagram with a $(k-1)$-cycle, we can apply the induction hypothesis to the latter. This completes the proof. 

\end{proof}

\begin{lemma}\label{r-2022-05-02-lemma1} For all $m\geq 2$ the quotient space $\frac{\mathring{\mathcal{A}}^\wedge_0}{P^2_0}[m]$ is equal to $\{0\}$. 
\end{lemma}

\begin{proof} We proceed by induction on $m\geq 2$. Let us consider the case $m=2$. Let $\Theta$, 
\myencircle{$=$} and \myencircle{\rotatebox[origin=c]{180}{\textsf{Y}}} denote the classes in $\mathring{\mathcal{A}}^\wedge_0$ of the Jacobi diagrams shown in Figure~\ref{figure-theta-dashed-pl-dashed} $(a)$, $(b)$ and $(c)$ respectively. Let $X$ denote the dashed loop as in Proposition~\ref{degree-0-of-Pn}. 

\begin{figure}[ht]
		\centering
         \includegraphics[scale=1]{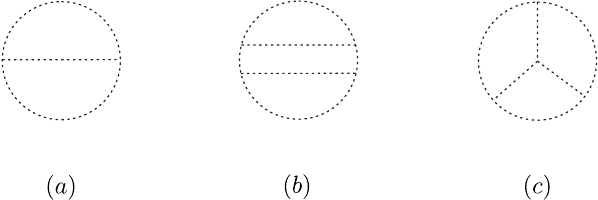}
\caption{Some elements in
  $\mathring{\mathcal{A}}^\wedge_0$. }
\label{figure-theta-dashed-pl-dashed} 								
\end{figure}

By the IHX and AS relations we have 
\begin{equation}\label{eq:2022-05-03-01}
\text{\myencircle{$=$}} {}=2 \text{ \myencircle{\rotatebox[origin=c]{180}{\textsf{Y}}}}.
\end{equation} 
One checks that $$\mathring{\mathcal{A}}^\wedge_0[2] = \mathbb{C}[X]\cdot \text{Vect}_{\mathbb{C}}\big\{\Theta\cdot\Theta, \text{\myencircle{\rotatebox[origin=c]{180}{\textsf{Y}}}}\big\}.$$

An instance of the relation $P^2_0$ and the AS relation imply the equation shown in Figure~\ref{instanceofP-2-2}. 
\begin{figure}[ht]
		\centering
         \includegraphics[scale=1]{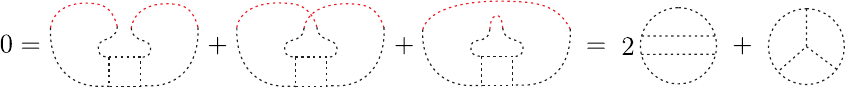}
\caption{Instance of $P^2_0$. In the second equality we have used the AS relation.}
\label{instanceofP-2-2} 								
\end{figure}

From the equation in Figure~\ref{instanceofP-2-2} we deduce $\text{\myencircle{$=$}} {}=-\frac{1}{2} \text{ \myencircle{\rotatebox[origin=c]{180}{\textsf{Y}}}}$ in $\mathring{\mathcal{A}}^\wedge_0/P^2_0$. Using this and \eqref{eq:2022-05-03-01} we obtain $\text{ \myencircle{\rotatebox[origin=c]{180}{\textsf{Y}}}} {}= 0$ in $\mathring{\mathcal{A}}^\wedge_0/P^2_0$.

\begin{figure}[ht]
		\centering
         \includegraphics[scale=1]{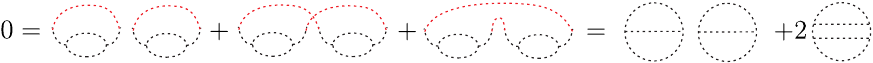}
\caption{Instance of $P^2_0$. In the second equality we have used the AS relation.}
\label{instanceofP-2} 								
\end{figure}

The instance of the relation $P^2_0$ shown in Figure~\ref{instanceofP-2} together with \eqref{eq:2022-05-03-01} imply $\Theta\cdot \Theta = -4 \text{ \myencircle{\rotatebox[origin=c]{180}{\textsf{Y}}}}$ in $\mathring{\mathcal{A}}^\wedge_0/P^2_0$ and therefore   $\Theta\cdot \Theta = 0$ in $\mathring{\mathcal{A}}^\wedge_0/P^2_0$. We deduce that $\frac{\mathring{\mathcal{A}}^\wedge_0}{P^2_0}[2]=\{0\}$. 

Let us consider now the case $m\geq 3$. Let $x\in\mathring{\mathrm{Jac}}(\vec{\varnothing},\emptyset)$  of degree $m$ and which does not contain the dashed component $X$. We consider the two cases:

\medskip

$(i)$ The diagram $x$ has at least two connected components. In this case we can write $x=x'\sqcup x''$ where $x'$ is a connected component of $x$ and $x''=x\setminus x'$. Thus, $1\leq \mathrm{deg}(x'), \mathrm{deg}(x'')\leq m-1$ and $\mathrm{deg}(x') + \mathrm{deg}(x'') = m$. If $\mathrm{deg}(x')\geq 2$, then by inductive hypothesis $x'=0$ in $\mathring{\mathcal{A}}^\wedge_0/P^2_0$, therefore $x=0$ in $\mathring{\mathcal{A}}^\wedge_0/P^2_0$. If $\mathrm{deg}(x')=1$, then $\mathrm{deg}(x'')=m-1\geq 2$. Thus $x''=0$ in $\mathring{\mathcal{A}}^\wedge_0/P^2_0$ and therefore $x=0$ in $\mathring{\mathcal{A}}^\wedge_0/P^2_0$.

\medskip

$(ii)$ The diagram $x$ has only one connected component. In this case by Lemma~\ref{r-2022-05-03-02}, we can suppose that $x$ has the form shown  in Figure~\ref{elementxwithbigon}. 
\begin{figure}[ht]
		\centering
         \includegraphics[scale=1]{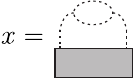}
\caption{Form of $x$.}
\label{elementxwithbigon} 
\end{figure}

Consider the instance of the relation $P^2_0$ shown in Figure~\ref{instanceofP-2-3}. 

\begin{figure}[ht]
		\centering
         \includegraphics[scale=1]{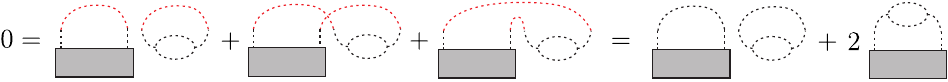}
\caption{Instance of the relation $P^2_0$. The grey part is the same as in Figure~\ref{elementxwithbigon}.}
\label{instanceofP-2-3} 
\end{figure}
Let $D$ denote the diagram with the $\Theta$ component in Figure~\ref{instanceofP-2-3}. We have  $\mathrm{deg}(D\setminus \Theta)=m-1\geq 2$, then by the inductive hypothesis $D\setminus \Theta =0$ in $\mathring{\mathcal{A}}^\wedge_0/P^2_0$. It follows that $D=0$ in $\mathring{\mathcal{A}}^\wedge_0/P^2_0$. Using this and the equation from Figure~\ref{instanceofP-2-3}, we obtain $2x=0$ and therefore $x=0$ in $\mathring{\mathcal{A}}^\wedge_0/P^2_0$.

We conclude that $\frac{\mathring{\mathcal{A}}^\wedge_0}{P^2_0}[m]=\{0\}$ for any $m\geq 2$.
\end{proof}

\begin{proposition}\label{r-2022-05-05-1} The map $\mathrm{pr}_1: \mathfrak{c}(1)\to \mathfrak{d}(1)$ is an algebra isomorphism and therefore $\mathfrak{c}(1) = \mathfrak{d}(1)$.
\end{proposition}
\begin{proof}

Consider the surjective algebra homomorphism $\widetilde{\mathrm{pr}}_1: \mathfrak{c}(1)\to \mathfrak{e}(1)$ given by the composition
$$ \mathfrak{c}(1)=\frac{\mathring{\mathcal{A}}_0}{P^{2}_0 + O^1 }\xrightarrow{\ \mathrm{pr}_1 \ } \mathfrak{d}(1)=\frac{\mathring{\mathcal{A}}_0}{P^{2}_0 + O^1 + (\mathring{\mathrm{deg}}>1)}\simeq\frac{{\mathcal{A}}_0}{(\mathrm{deg}>1)}=\mathfrak{e}(1),$$
where the last isomorphism is given by Lemma~\ref{r2022-04-12-ohtlemma}. Recall that $\mathfrak{e}(1)=\mathfrak{e}(1)[0]\oplus \mathfrak{e}(1)[1]$ is a two dimensional $\mathbb{C}$-vector space where $\mathfrak{e}(1)[0]$ is spanned by the class in $\mathfrak{e}(1)$ of the empty Jacobi diagram (of degree $0$) and $\mathfrak{e}(1)[1]$ is spanned by the class in $\mathfrak{e}(1)$ of  the theta Jacobi diagram  (of degree $1$) $\Theta$ shown in Figure~\ref{figure-theta-dashed-pl-dashed}$(a)$. Let us study the restriction of the map $\widetilde{\mathrm{pr}}_1$ to the degree components 
$\mathfrak{c}(1)[m]=\frac{\mathring{\mathcal{A}}^\wedge_0[m]}{P^{2}_0[m] + O^1[m]}$ of $\mathfrak{c}(1)$ where the degree is given by half of the number of trivalent vertices. By Proposition~\ref{degree-0-of-Pn}, we have $\mathfrak{c}(1)[0]\simeq \mathbb{C}$ spanned by the class in $\mathfrak{c}(1)$ of the empty Jacobi diagram. We also have that $\mathfrak{c}(1)[1]\simeq \mathbb{C}$ is spanned by the class in $\mathfrak{c}(1)$ of the theta Jacobi diagram $\Theta$. Thus, $\widetilde{\mathrm{pr}}_1: \mathfrak{c}(1)\to \mathfrak{e}(1)$ induces  isomorphisms $\mathfrak{c}(1)[0]\simeq \mathfrak{e}(1)[0]$ and $\mathfrak{c}(1)[1]\simeq \mathfrak{e}(1)[1]$. Besides, by Lemma~\ref{r-2022-05-02-lemma1}, $\frac{\mathring{\mathcal{A}}^\wedge_0[m]}{P^{2}_0[m]} =\{0\}$ for any $m\geq 2$, hence $\mathfrak{c}(1)[m]=\{0\}$ for any $m\geq 2$. We conclude that the map  $\widetilde{\mathrm{pr}}_1: \mathfrak{c}(1)\to \mathfrak{e}(1)$ is an isomorphism and therefore ${\mathrm{pr}}_1: \mathfrak{c}(1)\to \mathfrak{d}(1)$ is too. 

\end{proof}

\subsection{The families of \texorpdfstring{$3$}{3}-manifold invariants \texorpdfstring{$\{\Omega^{\mathfrak{c}}_n\}_{n\geq 1}$}{tilde-Omega-n} and \texorpdfstring{$\{\Omega^{\mathfrak{d}}_n\}_{n\geq 1}$}{Omega-n}}\label{sec:3:6-2022-04-18}

\begin{theorem}\label{r2022-04-12theinvariants} For any integer $n\geq 1$, there exist invariants of closed oriented $3$-manifolds
\begin{equation}
\Omega^{\mathfrak{c}}_n: 3\text{-}\mathrm{Mfds}\longrightarrow \mathfrak{c}(n)\quad \quad \text{and} \quad \quad  \Omega^{\mathfrak{d}}_n: 3\text{-}\mathrm{Mfds}\longrightarrow \mathfrak{d}(n)
\end{equation}  \index[notation]{\Omega^{\mathfrak{c}}_n@$\Omega^{\mathfrak{c}}_n$}  \index[notation]{\Omega^{\mathfrak{d}}_n@$\Omega^{\mathfrak{d}}_n$}
with values in the algebras $\mathfrak{c}(n)$ and $\mathfrak{d}(n)$, explicitly given by $\Omega^{\mathfrak{c}}_n:=Z_{(\mathfrak{c}(n),\overline{j}_n\circ  \varphi_n\circ\overline{\mathrm{cs}}^\nu\circ {Z}_{\mathcal A})}$ and $\Omega^{\mathfrak{d}}_n:=Z_{(\mathfrak{d}(n),\mathrm{pr}_n\circ\overline{j}_n\circ \varphi_n\circ \overline{\mathrm{cs}}^\nu\circ {Z}_{\mathcal A})}$.
Moreover, these invariants are monoid homomorphisms and for any closed oriented $3$-manifold $Y$ we have $\mathrm{pr}_n(\Omega^{\mathfrak{c}}_n(Y)) = \Omega^{\mathfrak{d}}_n(Y)$, where $\mathrm{pr}_n:\mathfrak{c}(n)\to \mathfrak{d}(n)$ is the canonical projection.
\end{theorem}
\begin{proof}
By Theorem~\ref{mainr-2022-02-14} and Proposition~\ref{r2022-04-12-dn}, for any $n\geq 1$,  the pairs $(\mathfrak{c}(n),\overline{j}_n\circ  \varphi_n\circ\overline{\mathrm{cs}}^\nu\circ {Z}_{\mathcal A})$ and $(\mathfrak{d}(n),\mathrm{pr}_n\circ\overline{j}_n\circ \varphi_n\circ \overline{\mathrm{cs}}^\nu\circ {Z}_{\mathcal A})$ are Kirby structures, that is, they belong to the category~$\mathcal{KS}$ (see Definition~\ref{def:catSKSandKS}). We can then apply the functor $\mathcal{KS}\to \mathcal{I}nv$ from Proposition~\ref{cor:invt:3-vars:alg} to obtain invariants of $3$-manifolds $\Omega^{\mathfrak{c}}_n := Z_{(\mathfrak{c}(n),\overline{j}_n\circ \varphi_n\circ\overline{cs}^\nu\circ {Z}_{\mathcal A})}$ and $\Omega^{\mathfrak{d}}_n:= Z_{(\mathfrak{d}(n),\mathrm{pr}_n\circ\overline{j}_n\circ \varphi_n\circ\overline{cs}^\nu\circ {Z}_{\mathcal A})}$~:
\begin{equation}
\Omega^{\mathfrak{c}}_n: 3\text{-}\mathrm{Mfds}\longrightarrow \mathfrak{c}(n)\quad \quad \text{and} \quad \quad  \Omega^{\mathfrak{d}}_n: 3\text{-}\mathrm{Mfds}\longrightarrow \mathfrak{d}(n)
\end{equation}
with values in the algebras $\mathfrak{c}(n)$ and $\mathfrak{d}(n)$, respectively. The last statement follows  from Proposition~\ref{cor:invt:3-vars:alg}.
\end{proof}

\begin{remark} The invariant $\Omega^{\mathfrak{d}}_n: 3\text{-}\mathrm{Mfds}\longrightarrow \mathfrak{d}(n)\simeq\mathfrak{e}(n)$, is the one considered in \cite{LMO98,Ohts}. The algebra $\mathfrak{e}(n)$ will be introduced in Definition~\ref{def:definitionofen} and the isomorphism $\mathfrak{d}(n)\simeq\mathfrak{e}(n)$ is the one given in Lemma~\ref{r2022-04-12-ohtlemma}.
\end{remark}

\subsection{Equality of \texorpdfstring{$\Omega^{\mathfrak{c}}_1$}{tildeOmega-1} and \texorpdfstring{$\Omega^{\mathfrak{d}}_1$}{Omega-1}}\label{sec::equalityomega1comega1d}

\begin{corollary}  For any $Y\in 3\text{-}\mathrm{Mfds}$, we have $\Omega^{\mathfrak{c}}_1(Y)=\Omega^{\mathfrak{d}}_1(Y)$.
\end{corollary}
\begin{proof} By Theorem~\ref{r2022-04-12theinvariants} we have $\mathrm{pr}_1(\Omega_1^{\mathfrak{c}}(Y)) = \Omega^{\mathfrak{d}}_1(Y)$ and by Proposition~\ref{r-2022-05-05-1}  $\mathfrak{c}(1) =  \mathfrak{d}(1)$. Therefore $\Omega^{\mathfrak{c}}_1(Y)=\Omega^{\mathfrak{d}}_1(Y)$.
\end{proof}

\subsection{On the construction of the algebras \texorpdfstring{$\mathfrak{c}(n)$}{bn}}\label{sec:3:7}

In this subsection we consider integers $m,n\geq 1$. By Lemma~\ref{lemma:deg-of-j-n} the algebra homomorphism $j_n^{\mathrm{LMO}}:\bigoplus_{k\geq 0}\mathring{\mathcal{A}}^\wedge_k\to \mathring{\mathcal{A}}^\wedge_0$ is homogeneous, therefore $j_n^{\mathrm{LMO}}\big(P^{m+1}+L^{<2m}+ (\mathrm{CO}\mathring{\mathcal{A}}^\wedge)\big)$ is a graded subspace of $\mathring{\mathcal{A}}^\wedge_0$.

If $J\subset \mathring{\mathcal{A}}^\wedge_0$ is a graded ideal containing $j_n^{\mathrm{LMO}}\big(P^{m+1}+L^{<2m}+ (\mathrm{CO}\mathring{\mathcal{A}}^\wedge)\big)$, then the map $j_n^{\mathrm{LMO}}$ induces an algebra homomorphism
\begin{equation}
\overline{j}_{n,m,J}: \mathfrak{b}(m)\longrightarrow \frac{\mathring{\mathcal{A}}^\wedge_0}{J},
\end{equation}  \index[notation]{j_{n,m,J}@$\overline{j}_{n,m,J}$}
where $\mathfrak{b}(m)$ is as in Definition~\ref{def:space-a-n}. Hence, by Lemma~\ref{sec:1.8lemma1}, the pair $\big(\mathring{\mathcal{A}}^\wedge_0/J, \overline{j}_{n,m,J}\circ\varphi_n\circ \overline{\mathrm{cs}}^\nu\circ {Z}_{\mathcal A}\big)$ is a semi-Kirby structure. In particular, if $n=m$ and $J=P^{n+1}_0 + O^n$ then $\mathring{\mathcal{A}}^\wedge_0/J = \mathfrak{c}(n)$ (see Definition~\ref{def:b-n}) and the algebra homomorphism $\overline{j}_{n,n,J}$ is the algebra  homomorphism $\overline{j}_n$ in Proposition~\ref{r2022-04-08-jn}. Moreover, by Theorem~\ref{mainr-2022-02-14}, the pair $\big(\mathfrak{c}(n),\overline{j}_{n}\circ \varphi_n \circ \overline{\mathrm{cs}}^\nu\circ{Z}_{\mathcal A}\big)$ is a Kirby structure. We have the following.

\begin{proposition} For $m\not = n$ there is no graded ideal $J$ of  $\mathring{\mathcal{A}}^\wedge(\emptyset)$ containing  $j_n^{\mathrm{LMO}}\big(P^{m+1}+L^{2m}+ (\mathrm{CO}\mathring{\mathcal{A}}^\wedge)\big)$ and such that the pair $\big(\mathring{\mathcal{A}}^\wedge_0/J, \overline{j}_{n,m,J}\circ \varphi_n\circ \overline{\mathrm{cs}}^\nu\circ{Z}_{\mathcal A}\big)$ is a Kirby structure.
\end{proposition}
\begin{proof}
Suppose, contrary to our claim, that such a graded ideal $J$ of $\mathring{\mathcal{A}}^\wedge_0$ exists. We consider two cases.

If $n<m$, the element $\theta_n$ (described in Lemma~\ref{reduction-calcul-j-n-cs-Z-U}) belongs the ideal $L^{<2m}$.  Using Lemma~\ref{reduction-calcul-j-n-cs-Z-U}, we have
\begin{equation*}
\begin{split}
j_n^{\mathrm{LMO}}\big(\mathrm{cs}^{\nu}\circ{Z}_{\mathcal A}(U^\pm)\big)[0]&=\frac{1}{n!}\left(\pm\frac{1}{2}\right)^n j_1^{\mathrm{LMO}}(\theta_n^{(n)})\in j_n^{\mathrm{LMO}}(L^{<2m})[0]\\
& \subset j_n^{\mathrm{LMO}}\big(P^{m+1} + L^{<2m}+ (\mathrm{CO}\mathring{\mathcal{A}}^\wedge)\big)[0] \subset J[0].
\end{split}
\end{equation*}
Hence, the elements $\overline{j}_{n,m,J}\circ \varphi_n\circ \overline{\mathrm{cs}}^\nu\circ {Z}_{\mathcal A}(U^\pm)[0]$, which are the classes of $j_n^{\mathrm{LMO}}\big(\mathrm{cs}^{\nu}\circ{Z}_{\mathcal A}(U^\pm)\big)[0]$ in $\mathring{\mathcal{A}}^\wedge_0[0]/J[0]$, are zero. Therefore, the elements $\overline{j}_{n,m,J}\circ\varphi_n\circ \overline{\mathrm{cs}}^\nu\circ {Z}_{\mathcal A}(U^\pm)$ are not invertible in $\mathring{\mathcal{A}}^\wedge_0/J$, that is, the pair $\big(\mathring{\mathcal{A}}^\wedge(\emptyset)/J,\overline{j}_{n,m,J}\circ\varphi_n\circ\overline{\mathrm{cs}}^\nu\circ {Z}_{\mathcal A}\big)$ is not a Kirby structure.

Now, let us assume $n>m$. In the proof of Proposition~\ref{r2022-04-08-jn} we have seen that $j_n^{\mathrm{LMO}}(L^{<2n})=j_n^{\mathrm{LMO}}((\mathrm{CO}\mathring{\mathcal{A}}^\wedge))= 0$ and the same argument used to show  $j_n^{\mathrm{LMO}}(P^{n+1})\subset P^{n+1}_0$ works to show $j_n^{\mathrm{LMO}}(P^{m+1})$ is a subset of  $P^{m+1}_0$. Since $L^{<2n}\subset L^{<2m}$ (because $n>m$), then $j_n^{\mathrm{LMO}}(L^{<2m})=0$. Besides, $P^{m+1}_0\subset P^{m+1}$ and the restriction of $j_n^{\mathrm{LMO}}:\bigoplus_{k\geq 0}\mathring{\mathcal{A}}^\wedge_k\to \mathring{\mathcal{A}}^\wedge_0$ to $\mathring{\mathcal{A}}^\wedge_0$ is the identity, therefore $P^{m+1}_0 = j_n^{\mathrm{LMO}}(P^{m+1}_0)\subset j^{\mathrm{LMO}}_n(P^{m+1})$. Thus, we have $j^{\mathrm{LMO}}_n(P^{m+1}) = P^{m+1}_0$. 
To sum up, 
$$P^{m+1}_0 = j^{\mathrm{LMO}}_n\big(P^{m+1} + L^{2m} + (\mathrm{CO}\mathring{\mathcal{A}}^\wedge)\big) \subset J.$$
From this, it follows that $\mathring{\mathcal{A}}^\wedge_0[0]\cap P^{m+1}_0 = P^{m+1}_0[0] \subset J[0]$. By Proposition~\ref{degree-0-of-Pn}$(a)$, the ideal $P^{m+1}_0[0]$ is  generated by $X(X+2)\cdots (X+2m)$. Consequently, the element $X(X+2)\cdots (X+2(n-1))$ belongs to $P^{m+1}_0[0] \subset J[0]$ (because $n>m$). Hence, by Proposition~\ref{calcul-j-n-cs-Z-U-pm-deg0},  the elements $j^{\mathrm{LMO}}_n\big(\mathrm{cs}^\nu\circ {Z}_{\mathcal A}(U^\pm)\big)[0]$ belong to~$J[0]$. Therefore, similarly to the case $n<m$, we conclude that the pair $\big(\mathring{\mathcal{A}}^\wedge(\emptyset)/J,\overline{j}_{n,m,J}\circ\varphi_n\circ\overline{\mathrm{cs}}^\nu\circ {Z}_{\mathcal A}\big)$ is not a Kirby structure.
\end{proof}

\section{Elimination of redundant information and the LMO invariant}\label{sec::5}

We start this section by recalling in \S\ref{sec:5:1} the well-known coproduct map $\Delta_{\bigoplus_{k\geq 0}\mathring{\mathcal{A}}^\wedge_k}$ for the algebra $\bigoplus_{k\geq 0}\mathring{\mathcal{A}}^\wedge_k$, which endows it with a bialgebra structure (Proposition \ref{r:2022-06-15Deltaisalgebramorpshim}). Recall that the algebras $\mathfrak{b}(n)$, $\mathfrak{c}(n)$ and $\mathfrak{d}(n)$ (from \S\ref{sec:3:3}, \S\ref{sec:3:4} and  \S\ref{sec:3:5-2022-04-18}, respectively) are quotients of  $\bigoplus_{k\geq 0}\mathring{\mathcal{A}}^\wedge_k$, we show that the above coproduct map induces coproduct maps $\Delta^{\mathfrak{b}}_{n,p}:\mathfrak{b}(n)\to \mathfrak{b}(p) \otimes \mathfrak{b}(n-p)$,  $\Delta^{\mathfrak{c}}_{n,p}:\mathfrak{c}(n)\to \mathfrak{c}(p) \otimes \mathfrak{b}(n-p)$ and $\Delta^{\mathfrak{d}}_{n,p}:\mathfrak{d}(n)\to \mathfrak{d}(p) \otimes \mathfrak{d}(n-p)$   in \S\ref{sec:5:2},  \S\ref{sec:5:3} and  \S\ref{sec:5:4}, respectively. We show in the same sections  that $\Delta^{\mathfrak{b}}_{n,p}$ is a morphism of semi-Kirby structures (Proposition \ref{r:2022-06-15commdiag1}) and that $\Delta^{\mathfrak{c}}_{n,p}$ and $\Delta^{\mathfrak{d}}_{n,p}$ are morphisms of Kirby structures (Proposition \ref{r:2022-06-15inducedalgebramorphismonmathfrakc} and Proposition \ref{r:2024-01-23-1}, respectively). This allows us to show  a compatibility between the family of invariants of 3-manifolds $\{\Omega_n^{\mathfrak{c}}\}_{n\geq 1}$   (resp. $\{\Omega_n^{\mathfrak{d}}\}_{n\geq 1}$) obtained in \S\ref{ch4} (see Theorem~\ref{r2022-04-12theinvariants}) and the coproduct maps $\Delta^{\mathfrak{c}}_{n,p}$ (resp. $\Delta^{\mathfrak{d}}_{n,p}$), see Proposition \ref{r:2023-12-06r2}. In \S\ref{nsec:5:6}, we consider a family of Kirby structures $\{\mathfrak{e}(n)\}_{n\geq 1}$ and Kirby morphisms $\Delta^{\mathfrak{e}}_{n,p}:\mathfrak{e}(n)\to \mathfrak{e}(p) \otimes \mathfrak{e}(n-p)$ (Proposition \ref{r:2022-06-15inducedalgebramorphismonmathfrake}) and show it to be isomorphic to the family $\{\mathfrak{d}(n)\}_{n\geq 1}$  of Kirby structures  (Lemma~\ref{r2022-04-12-ohtlemma}). To the family $\{\mathfrak{e}(n)\}_{n\geq 1}$  we associate  the family of invariants $\{\Omega_n^{\mathfrak{e}}\}_{n\geq 1}$ of oriented closed 3-manifolds. One can prove that the new information given by $\Omega_n^{\mathfrak{e}}$ with respect to the previous $\Omega_k^{\mathfrak{e}}$ ($k<n$) is contained in the degree $n$ part of $\Omega_n^{\mathfrak{e}}$, see Corollary~\ref{r:2023-02-05r4}. This allows us to unify the family of invariants $\{\Omega_n^{\mathfrak{e}}\}_{n\geq 1}$ into a single invariant: \emph{the LMO invariant} $Z^{\mathrm{LMO}}$, see \S\ref{sec:5:6}.  In \S\ref{sec:5:7}, we prove that $\mathrm{pr}_1:\mathfrak{c}(1)\to\mathfrak{d}(1)$ is an isomorphism (see \S\ref{sec:3:5-2022-04-18} for the definition of $\mathrm{pr}_n:\mathfrak{c}(n)\to\mathfrak{d}(n)$), from which we derive the equality of invariants $\Omega^{\mathfrak{c}}_1 = \Omega^{\mathfrak{d}}_1$. For $n \geq 2$, the map $\mathrm{pr}_n: \mathfrak{c}(n)\to\mathfrak{d}(n)$ is surjective, but it is not known whether it is injective; in the case it is not, the invariant $\Omega^{\mathfrak{c}}_n$ may contain additional information with respect to $\Omega^{\mathfrak{d}}_n$.

\subsection{A bialgebra structure on \texorpdfstring{$\bigoplus_{k\geq 0}\mathring{\mathcal{A}}^{\wedge}_k$}{oplus A(circle k)}}\label{sec:5:1}

\subsubsection{Coproduct on \texorpdfstring{$\mathring{\mathcal{A}}^{\wedge}_k$}{A(circle k)}}

Let $k \geq 0$. Recall that the set $\mathring{\mathrm{Jac}}(\vec{\varnothing}, [\![1,k]\!])$ consists of tuples of the form  $\underline{D}= \big(D,\varphi,\{\mathrm{cyc}_i\}_{i\in[\![1,k]\!]}\big)$ where $D$ is a vertex-oriented unitrivalent graph, $ \varphi : \partial D \to [\![1,k]\!]$ is a map and $\mathrm{cyc}_i$ is a cyclic order on $\varphi^{-1}(i)$ for each $i\in[\![1,k]\!]$, see Definition~\ref{def:thesetJacPS}.

\begin{definition}\label{def:coproductinA-k}
Let $\underline{D}= \big(D,\varphi,\{\mathrm{cyc}_i\}_{i\in[\![1,k]\!]}\big) \in \mathring{\mathrm{Jac}}(\vec{\varnothing}, [\![1,k]\!])$. 

$(a)$ For $A \subset \pi_0(D)$ we set
$$\underline{D}_A:=\left(\bigsqcup_{x \in A}x, \varphi_{|\bigsqcup_{x\in A} \partial x}, \{\widetilde{\mathrm{cyc}}_i\}_{i\in [\![1,k]\!]}\right)\in \mathring{\mathrm{Jac}}(\vec{\varnothing}, [\![1,k]\!]),$$
where   $\widetilde{\mathrm{cyc}}_i$ is the restriction of $\mathrm{cyc}_i$ to $\varphi^{-1}(i) \cap \bigsqcup_{x \in A}\partial x$ for all $i\in [\![1,k]\!]$.

$(b)$ Define the linear map 
\begin{equation}\label{eq:coproductinJac}
\Delta_{\mathring{\mathrm{Jac}}(\vec{\varnothing}, [\![1,k]\!])} : \mathbb{C}\mathring{\mathrm{Jac}}(\vec{\varnothing}, [\![1,k]\!])\longrightarrow \mathbb{C}\mathring{\mathrm{Jac}}(\vec{\varnothing}, [\![1,k]\!]) \otimes \mathbb{C}\mathring{\mathrm{Jac}}(\vec{\varnothing}, [\![1,k]\!]),
\end{equation} \index[notation]{\Delta_{\mathring{\mathrm{Jac}}(\vec{\varnothing}, [\![1,k]\!])}@$\Delta_{\mathring{\mathrm{Jac}}(\vec{\varnothing}, [\![1,k]\!])}$}
 by  $$\underline{D}\xmapsto{\ \Delta_{\mathring{\mathrm{Jac}}(\vec{\varnothing}, [\![1,k]\!])} \ } \sum_{A \subset \pi_0(D)}\underline{D}_A \otimes \underline{D}_{\pi_0(D)\setminus A}$$
for any $\underline{D}\in\mathring{\mathrm{Jac}}(\vec{\varnothing}, [\![1,k]\!])$.  In other words $\Delta_{\mathring{\mathrm{Jac}}(\vec{\varnothing}, [\![1,k]\!])}(\underline{D})$ is the linear combination in $\mathbb{C}\mathring{\mathrm{Jac}}(\vec{\varnothing}, [\![1,k]\!])\otimes \mathbb{C}\mathring{\mathrm{Jac}}(\vec{\varnothing}, [\![1,k]\!])$ obtained by considering all the possible ways of splitting the connected components of $D$ into the two tensor factors. See Figure~\ref{figure-example-coproduit-Jac} for an example. 
\end{definition}

\begin{figure}[ht]
		\centering
         \includegraphics[scale=0.9]{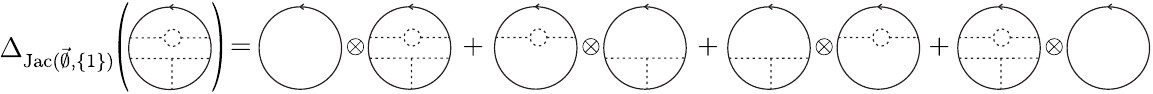}
\caption{Image of an element in $\mathring{\mathrm{Jac}}(\vec{\varnothing}, \{1\})$ under the map $\Delta_{\mathring{\mathrm{Jac}}(\vec{\varnothing}, \{1\})}$.}
\label{figure-example-coproduit-Jac} 								
\end{figure}

One checks that the map \eqref{eq:coproductinJac} equips  $\mathbb{C}\mathring{\mathrm{Jac}}(\vec{\varnothing}, [\![1,k]\!])$ with a cocommutative coassociative coalgebra structure and that there is a unique continuous linear map $\Delta_{\mathring{\mathcal A}^\wedge_k} : \mathring{\mathcal A}^\wedge_k\to \mathring{\mathcal A}^\wedge_k \hat{\otimes} \mathring{\mathcal A}^\wedge_k$, \index[notation]{\Delta_{\mathring{\mathcal A}^\wedge_k}@$\Delta_{\mathring{\mathcal A}^\wedge_k}$} which makes $\mathring{\mathcal A}^\wedge_k$ into a coassociative coalgebra such that the projection $\mathbb{C}\mathring{\mathrm{Jac}}(\vec{\varnothing}, [\![1,k]\!])\to \mathring{\mathcal A}^\wedge_k$ is a coalgebra homomorphism.

\begin{lemma}[{\cite[Proposition 6.9]{Ohts}}]\label{r:imageofhatZcontainedinGL} Let $k\geq 0$. The image of ${Z}_{\mathcal A} : \underline{\vec{\mathcal T}}(\emptyset,\emptyset)_k \to \mathring{\mathcal{A}}^\wedge_k$ (see Theorem~\ref{thmkontsevichintegral}) is contained in the set of group-like elements. That is, for any $L\in  \underline{\vec{\mathcal T}}(\emptyset,\emptyset)_k $ we have
$$\Delta_{\mathring{\mathcal{A}}^\wedge_k}({Z}_{\mathcal A}(L)) = {Z}_{\mathcal A}(L) \hat{\otimes} {Z}_{\mathcal A}(L) \in \mathring{\mathcal{A}}^\wedge_k\hat{\otimes} \mathring{\mathcal{A}}^\wedge_k.$$
\end{lemma}

\subsubsection{Coproduct on $\bigoplus_{k\geq 0}\mathring{\mathcal{A}}^\wedge_k$}

\begin{lemma}  There is a unique cocommutative  coassociative coalgebra structure on $\bigoplus_{k\geq 0}\mathring{\mathcal{A}}^\wedge_k$ with coproduct 
\begin{equation}\label{eq:coproductin-oplus-A-l}
\Delta_{\bigoplus_{k\geq 0}\mathring{\mathcal{A}}^\wedge_k} : \bigoplus_{k\geq 0}\mathring{\mathcal{A}}^\wedge_k \longrightarrow 
\bigoplus_{k\geq 0}\mathring{\mathcal{A}}^\wedge_k\hat{\otimes} \bigoplus_{k\geq 0}\mathring{\mathcal{A}}^\wedge_k
\end{equation} \index[notation]{\Delta_{\bigoplus_{k\geq 0}\mathring{\mathcal{A}}^\wedge_k}@$\Delta_{\bigoplus_{k\geq 0}\mathring{\mathcal{A}}^\wedge_k}$} 
such that for any $k \geq 0$, the injection $\mathring{\mathcal{A}}^\wedge_k\subset \bigoplus_{k'\geq 0}\mathring{\mathcal{A}}^\wedge_{k'}$ is a coalgebra homomorphism. 
\end{lemma}

\begin{proof} The proof is immediate.
\end{proof}

\subsubsection{Compatibility of coproduct and product of \texorpdfstring{$\bigoplus_{k\geq 0}\mathring{\mathcal{A}}^{\wedge}_k$}{Al}}

\begin{proposition}\label{r:2022-06-15Deltaisalgebramorpshim} The map $$\Delta_{\bigoplus_{k\geq 0}\mathring{\mathcal{A}}^\wedge_k} : \bigoplus_{k\geq 0}\mathring{\mathcal{A}}^\wedge_k\longrightarrow 
\bigoplus_{k\geq 0}\mathring{\mathcal{A}}^\wedge_k \hat{\otimes} \bigoplus_{k\geq 0}\mathring{\mathcal{A}}^\wedge_k$$  from \eqref{eq:coproductin-oplus-A-l} is an algebra homomorphism. Therefore, $\bigoplus_{k\geq 0}\mathring{\mathcal{A}}^\wedge_k$ is a bialgebra.
\end{proposition}

\begin{proof}
One checks that the space $\bigoplus_{k \geq 0}\mathbb{C}\mathring{\mathrm{Jac}}(\vec{\varnothing},[\![1,k]\!])$, equipped with the product 
$$\otimes_{\mathring{\mathrm{Jac}}} : \mathbb{C}\mathring{\mathrm{Jac}}(\vec{\varnothing},[\![1,k]\!]) \otimes \mathbb{C}\mathring{\mathrm{Jac}}(\vec{\varnothing},[\![1,l]\!]) \longrightarrow \mathbb{C}\mathring{\mathrm{Jac}}(\vec{\varnothing},[\![1,k]\!]\sqcup [\![1,l]\!]) \simeq\mathbb{C}\mathring{\mathrm{Jac}}(\vec{\varnothing},[\![1,k+l]\!]),$$ 
where the first map is as in Definition~\ref{deftensorinJac} and the second map is induced by the bijection  $[\![1,k]\!] \sqcup [\![1,l]\!]\simeq [\![1,k+l]\!]$ given by $x \mapsto x$ for $x \in [\![1,k]\!]$ and $y\mapsto k+y$ for $y \in [\![1,l]\!]$, is an associative (non-commutative) algebra. The fact that $\Delta_{\mathring{\mathrm{Jac}}}$ is a morphism of algebras from this algebra to its tensor square follows from the equality $$(\underline{D} \otimes_{\mathring{\mathrm{Jac}}} \underline{E})_A=\left(\underline{D}_{A \cap \pi_0(D)}\right) \otimes_{\mathring{\mathrm{Jac}}} \left(\underline{E}_{A \cap \pi_0(E)}\right)$$
for any $A \subset \pi_0(D) \sqcup \pi_0(E)$. The projection $\bigoplus_{k \geq 0}\mathbb{C}\mathring{\mathrm{Jac}}(\vec{\varnothing},[\![1,k]\!]) \to \bigoplus_k \mathring{\mathcal A}^\wedge_k$ is compatible with the products and coproducts. The compatibility of the product and coproduct in its source then imply the same for the product and coproduct in the target.       
\end{proof}

\begin{proposition}[{\cite[Proposition~10.3]{Ohts}}]\label{2023-11-15prop} The algebra automorphism $\mathrm{cs}^\nu$ of the algebra $\bigoplus_{k\geq 0}\mathring{\mathcal{A}}^\wedge_k$ (see \S\ref{sec:3:1}) is a bialgebra automorphism.
\end{proposition}

\begin{remark}\label{remarkaboutcoproduit} 
Recall that the algebras $\mathfrak{s}(\mathbf{A})$  and $\mathfrak{a}$ are defined as the quotient algebras of $\bigoplus_{k\geq 0}\mathring{\mathcal{A}}^\wedge_k$ by the ideals  $(\mathrm{KII}\mathring{\mathcal{A}}^\wedge) + (\mathrm{CO}\mathring{\mathcal{A}}^\wedge)$ and $(\mathrm{KII}^s\mathring{\mathcal{A}}^\wedge) + (\mathrm{CO}\mathring{\mathcal{A}}^\wedge)$, respectively. We expect these ideals to be  bialgebra ideals, so that $\mathfrak{s}(\mathbf{A})$ and $\mathfrak{a}$ are bialgebras. Moreover, we expect that the algebra homomorphisms $\mu:\mathfrak{Kir} \to \mathfrak{s}(\mathbf{A})$  and $\overline{\mathrm{cs}}^\nu:\mathfrak{s}(\mathbf{A})\to\mathfrak{a}$ to be  bialgebra homomorphisms.
\end{remark}

\subsection{A semi-Kirby structure morphism \texorpdfstring{$\Delta^{\mathfrak{b}}_{n,p}:\mathfrak{b}(n)\to \mathfrak{b}(p) \otimes \mathfrak{b}(n-p)$}{bn-to-bp-otimes-bn-p}}\label{sec:5:2}

Consider integers $n \geq 1$ and $p\in [\![1,n-1]\!]$.

\subsubsection{Preliminaries on coproducts}

Let $k,m \geq 0$ be  integers. Recall from Definition~\ref{generalJDwithfree} that the set $\mathring{\mathrm{Jac}}(\vec{\varnothing},[\![1,k]\!],[\![1,m]\!])$  consists of classes of tuples $\big(D,\varphi : \partial D\to [\![1,k]\!] \sqcup [\![1,m]\!], \{\mathrm{cyc}_s\}_{s \in [\![1,k]\!]}\big)$ where the image of $\varphi$ contains $[\![1,m]\!]$, the preimage of any element of $[\![1,m]\!]$ has cardinality $1$, and  $\mathrm{cyc}_s$ is a cyclic order on $\varphi^{-1}(s)$ for any $s \in [\![1,k]\!]$.  There is an action of the symmetric group $\mathfrak{S}_{m}$ on $m$ letters  on $\mathring{\mathrm{Jac}}(\vec{\varnothing},[\![1,k]\!],[\![1,m]\!])$ given by 
$$\sigma \cdot \big(D,\varphi, \{\mathrm{cyc}_s\}_{s \in [\![1,k]\!]}\big):=  \big(D,(\mathrm{Id}_{[\![1,k]\!]} \sqcup\sigma)\circ\varphi , \{\mathrm{cyc}_s\}_{s \in [\![1,k]\!]}\big)$$
for any $\sigma\in \mathfrak{S}_{m}$ and $\big(D,\varphi, \{\mathrm{cyc}_s\}_{s \in [\![1,k]\!]}\big)\in \mathring{\mathrm{Jac}}(\vec{\varnothing},[\![1,k]\!],[\![1,m]\!])$.

\begin{definition}\label{def:definition1sec4.1} We set
\begin{equation}
\mathring{\mathrm{Jac}}(\vec{\varnothing},[\![1,k]\!],m):=\frac{\mathring{\mathrm{Jac}}(\vec{\varnothing},[\![1,k]\!],[\![1,m]\!])}{\mathfrak{S}_{m}} \quad \quad \text{ and } \quad \quad \mathring{\mathrm{Jac}}(\vec{\varnothing},[\![1,k]\!],\bullet) := \bigsqcup_{m\geq 0} \mathring{\mathrm{Jac}}(\vec{\varnothing},[\![1,k]\!],m).
\end{equation} \index[notation]{Jac(\vec{\varnothing},[\![1,k]\!],m)@$\mathring{\mathrm{Jac}}(\vec{\varnothing},[\![1,k]\!],m)$} \index[notation]{Jac(\vec{\varnothing},[\![1,k]\!],\bullet)@$\mathring{\mathrm{Jac}}(\vec{\varnothing},[\![1,k]\!],\bullet)$} 
\end{definition}

Notice that  $\mathring{\mathrm{Jac}}(\vec{\varnothing},[\![1,k]\!],0) = \mathring{\mathrm{Jac}}(\vec{\varnothing},[\![1,k]\!],\emptyset) = \mathring{\mathrm{Jac}}(\vec{\varnothing},[\![1,k]\!])$, from now on we will use this identification. Moreover, we have a grading  on $\mathbb{C}\mathring{\mathrm{Jac}}(\vec{\varnothing},[\![1,k]\!],\bullet)$ induced by $\bullet$.

\begin{lemma} The set $\mathring{\mathrm{Jac}}(\vec{\varnothing},[\![1,k]\!],\bullet)$ consists of tuples $\big(D,{\partial D}_a, {\partial D}_f, \varphi : \partial D_a \to [\![1,k]\!], \{\mathrm{cyc}_s\}_{s \in [\![1,k]\!]}\big)$ where   $D$ is a vertex-oriented unitrivalent graph, $\partial D = {\partial D}_a\sqcup {\partial D}_f$ is a partition, $\varphi : {\partial D}_a \to [\![1,k]\!]$ is a map and $\mathrm{cyc}_s$ is a cyclic order on $\varphi^{-1}(s)$ for any $s \in [\![1,k]\!]$.
\end{lemma}
\begin{proof}
This is just a restatement of Definition~\ref{def:definition1sec4.1}.
\end{proof}

From now on, for $\underline{D}=\big(D,{\partial D}_a, {\partial D}_f, \varphi : \partial D_a \to [\![1,k]\!], \{\mathrm{cyc}_s\}_{s \in [\![1,k]\!]}\big)\in \mathring{\mathrm{Jac}}(\vec{\varnothing},[\![1,k]\!],\bullet)$ we set $\pi_0(\underline{D}):=\pi_0(D)$ and ${\partial\underline{D}}_f:={\partial{D}}_f$.

\begin{lemma} For $\underline{D}\in \mathring{\mathrm{Jac}}(\vec{\varnothing},[\![1,k]\!],\bullet)$  and $A \subset \pi_0(\underline{D})$ we set 
$$\underline{D}_A:=\left(\bigsqcup_{x\in A} x, ({\partial D}_a)\cap \bigsqcup_{x\in A} x, ({\partial D}_f) \cap \bigsqcup_{x\in A} x, \varphi_{|({\partial D}_a) \cap \bigsqcup_{x\in A} x}, \{\mathrm{cyc}_s^A\}_{s \in [\![1,k]\!]} \right)$$
where $\mathrm{cyc}_s^A$ is the restriction to $\varphi^{-1}(s)\cap \bigsqcup_{x\in A} x$ of the cyclic order $\mathrm{cyc}_s$. Then $\underline{D}_A\in \mathring{\mathrm{Jac}}(\vec{\varnothing},[\![1,k]\!],\bullet)$. If moreover $\underline{D} \in \mathring{\mathrm{Jac}}(\vec{\varnothing},[\![1,k]\!],0)$ and $A \subset \pi_0(\underline{D})$, we have  $\underline{D}_A \in \mathring{\mathrm{Jac}}(\vec{\varnothing},[\![1,k]\!],0)$. 
\end{lemma}

\begin{proof}
The result follows directly from the definitions.
\end{proof}

Let $\underline{D} \in \mathring{\mathrm{Jac}}(\vec{\varnothing},[\![1,k]\!],\bullet)$,  $A \subset \pi_0(\underline{D})$ and $\pi_0(\underline{D})\setminus A$ the complement of $A$ in $\pi_0(\underline{D})$.  Then $\sum_{A \subset \pi_0(\underline{D})} \underline{D}_A \otimes \underline {D}_{\pi_0(\underline{D})\setminus A} \in \mathbb{C} \mathring{\mathrm{Jac}}(\vec{\varnothing},[\![1,k]\!],\bullet) \otimes \mathbb{C}\mathring{\mathrm{Jac}}(\vec{\varnothing},[\![1,k]\!],\bullet)$.

\begin{definition}  Define the linear map 
\begin{equation}\label{eq:coproductinJacbullet}
\Delta_{\mathring{\mathrm{Jac}}(\vec{\varnothing},[\![1,k]\!],\bullet)} : \mathbb{C}\mathring{\mathrm{Jac}}(\vec{\varnothing},[\![1,k]\!],\bullet)\longrightarrow \mathbb{C}\mathring{\mathrm{Jac}}(\vec{\varnothing},[\![1,k]\!],\bullet) \otimes \mathbb{C}\mathring{\mathrm{Jac}}(\vec{\varnothing},[\![1,k]\!],\bullet),
\end{equation} \index[notation]{\Delta_{\mathring{\mathrm{Jac}}(\vec{\varnothing},[\![1,k]\!],\bullet)}@$\Delta_{\mathring{\mathrm{Jac}}(\vec{\varnothing},[\![1,k]\!],\bullet)}$} 
 by  $$\underline{D}\xmapsto{\ \Delta_{\mathring{\mathrm{Jac}}(\vec{\varnothing},[\![1,k]\!],\bullet)} \ } \sum_{A \subset \pi_0(\underline{D})}\underline{D}_A \otimes \underline{D}_{\pi_0(D)\setminus A}$$
for any $\underline{D}\in\mathring{\mathrm{Jac}}(\vec{\varnothing},[\![1,k]\!],\bullet)$. 
\end{definition}

\begin{lemma}\label{r:2022-05-23lem:grading} The map  $\Delta_{\mathring{\mathrm{Jac}}(\vec{\varnothing},[\![1,k]\!],\bullet)}$ from \eqref{eq:coproductinJacbullet} is compatible with grading of $\mathbb{C}\mathring{\mathrm{Jac}}(\vec{\varnothing},[\![1,k]\!],\bullet)$. In particular,  it restricts to a map 
\begin{equation}\label{eq:coproductinJaczero}
\Delta_{\mathring{\mathrm{Jac}}(\vec{\varnothing},[\![1,k]\!],0)} : \mathbb{C}\mathring{\mathrm{Jac}}(\vec{\varnothing},[\![1,k]\!],0)\longrightarrow \mathbb{C}\mathring{\mathrm{Jac}}(\vec{\varnothing},[\![1,k]\!],0) \otimes \mathbb{C}\mathring{\mathrm{Jac}}(\vec{\varnothing},[\![1,k]\!],0).
\end{equation}
\end{lemma}
\begin{proof}
The proof is straightforward.
\end{proof}

Under the identification  $\mathring{\mathrm{Jac}}(\vec{\varnothing},[\![1,k]\!],0)=\mathring{\mathrm{Jac}}(\vec{\varnothing},[\![1,k]\!],\emptyset) = \mathring{\mathrm{Jac}}(\vec{\varnothing},[\![1,k]\!])$,  the two coproducts maps $\Delta_{\mathring{\mathrm{Jac}}(\vec{\varnothing},[\![1,k]\!],0)}$ and $\Delta_{\mathring{\mathrm{Jac}}(\vec{\varnothing},[\![1,k]\!])}$ (from Definition~~\ref{def:coproductinA-k}$~(b)$) coincide.

Let $\underline{D}\in \mathring{\mathrm{Jac}}(\vec{\varnothing},[\![1,k]\!],\bullet)$. Recall  from~\eqref{equ:JacandFPFI} that there is a map $\mathring{\mathrm{Jac}}: \mathrm{FPFI}(\partial{\underline{D}}_f) \to \mathring{\mathrm{Jac}}(\vec{\varnothing}, \emptyset, \partial{\underline{D}}_f)$ (the set $\mathrm{FPFI}(\partial{\underline{D}}_f)$ is equal to $\emptyset$ if $|\partial{\underline{D}}_f|$ is odd). Hence, given $\sigma\in \mathrm{FPFI}(\partial{\underline{D}}_f)$ we obtain an element 
\begin{equation}\label{eq:2022-05-27eqpairing}
\langle \underline{D}, \mathring{\mathrm{Jac}}(\sigma)\rangle \in \mathring{\mathrm{Jac}}(\vec{\varnothing},[\![1,k]\!],0) =  \mathring{\mathrm{Jac}}(\vec{\varnothing},[\![1,k]\!]), 
\end{equation}
where $\langle \ , \ \rangle$ is the pairing~\eqref{thepairing} 
 in Definition~\ref{themappairing}.

\begin{definition}\label{def:2022-05-23-def1} Let $\underline{D}\in \mathring{\mathrm{Jac}}(\vec{\varnothing},[\![1,k]\!],\bullet)$. 

$(a)$ Define the sets
\begin{equation*}
\mathrm{Set}_{\underline{D}}:=\left\{ (A,\sigma,\sigma') \ \Big| \ A\subset \pi_0(\underline{D}), \ \sigma\in\mathrm{FPFI}\left({\partial{\underline{D}}}_f\cap \bigsqcup_{x\in A}x \right), \ \sigma'\in\mathrm{FPFI}\left({\partial{\underline{D}}}_f\cap \bigsqcup_{x\in(\pi_0(\underline{D})\setminus A)}x \right)  \right\}
\end{equation*}
and
\begin{equation*}
\widetilde{\mathrm{Set}}_{\underline{D}}:=\Big\{ (\widetilde{\sigma}, \widetilde{A}) \ \Big| \    \widetilde{\sigma}\in\mathrm{FPFI}({\partial{\underline{D}}}_f), \ \widetilde{A}\subset \pi_0\big(\big\langle \underline{D}, \mathring{\mathrm{Jac}}(\widetilde{\sigma}) \big\rangle \big)  \Big\}.
\end{equation*}

$(b)$ Define the map
\begin{equation} (l_{\underline{D}}, r_{\underline{D}}): \mathrm{Set}_{\underline{D}} \longrightarrow \mathring{\mathrm{Jac}}(\vec{\varnothing},[\![1,k]\!]) \times \mathring{\mathrm{Jac}}(\vec{\varnothing},[\![1,k]\!]), \quad \quad (A,\sigma,\sigma')\longmapsto \big(l_{\underline{D}}(A,\sigma,\sigma'), r_{\underline{D}}(A,\sigma,\sigma')\big),
\end{equation}
where $l_{\underline{D}}(A,\sigma,\sigma') := \big\langle \underline{D}_A, \mathring{\mathrm{Jac}}(\sigma)\big\rangle$ and $r_{\underline{D}}(A,\sigma,\sigma') := \big\langle \underline{D}_{\pi_0(\underline{D})\setminus A}, \mathring{\mathrm{Jac}}(\sigma')\big\rangle$ for any $(A,\sigma,\sigma')\in \mathrm{Set}_{\underline{D}}$.

$(c)$ Define the map
\begin{equation} (\widetilde{l}_{\underline{D}}, \widetilde{r}_{\underline{D}}): \widetilde{\mathrm{Set}}_{\underline{D}} \longrightarrow \mathring{\mathrm{Jac}}(\vec{\varnothing},[\![1,k]\!]) \times \mathring{\mathrm{Jac}}(\vec{\varnothing},[\![1,k]\!]), \quad \quad (\widetilde{\sigma}, \widetilde{A})\longmapsto \big(\widetilde{l}_{\underline{D}}(A,\sigma,\sigma'), \widetilde{r}_{\underline{D}}(A,\sigma,\sigma')\big),
\end{equation}
where $\widetilde{l}_{\underline{D}}(\widetilde{\sigma}, \widetilde{A}) := \big\langle \underline{D}, \mathring{\mathrm{Jac}}(\widetilde{\sigma})\big\rangle_{\widetilde{A}}$ and $\widetilde{r}_{\underline{D}}(\widetilde{\sigma},\widetilde{A}) := \big\langle \underline{D}, \mathring{\mathrm{Jac}}(\widetilde{\sigma})\big\rangle_{\pi_0\big(\langle \underline{D}, \mathring{\mathrm{Jac}}(\widetilde{\sigma})\rangle\big)\setminus \widetilde{A}}$ for any $(\widetilde{\sigma},\widetilde{A})\in \widetilde{\mathrm{Set}}_{\underline{D}}$.
\end{definition} 

Notice that if in Definition~\ref{def:2022-05-23-def1} the number $|\partial{\underline{D}}_f|$ is odd,  then $\mathrm{Set}_{\underline D} = \widetilde{\mathrm{Set}}_{\underline D} = \emptyset.$

\begin{definition}\label{def:2022-05-30} Let $U$ and $V$ be finite sets and $f:U\to V$ be a surjective map.

$(a)$ For $\widetilde{\sigma}\in\mathrm{FPFI}(U)$ let $\sim_{\widetilde{\sigma}}$ be the finest equivalent relation on $V$ such that $f(u)\sim_{\widetilde{\sigma}} f(\widetilde{\sigma}u)$. Explicitly, $v\sim_{\widetilde{\sigma}} v'$ if and only if $v=v'$ or there exists $i\geq 0$ and $u_0,\ldots, u_i\in U$ such that  $v_0:=f(u_0)$, $v_1:=f(\widetilde{\sigma}u_0) = f(u_1)$, $\ldots$, $v_i:=f(\widetilde{\sigma}u_{i-1}) =f(u_i)$ and $v=v_0$ and $v'=v_i$.

$(b)$ Define the sets
$$\mathrm{Set}(f):=\Big\{ (A,\sigma,\sigma') \ \big| \ A\subset V, \ \sigma\in\mathrm{FPFI}(f^{-1}(A)), \ \sigma'\in\mathrm{FPFI}(f^{-1}(V\setminus A)) \Big\}$$
and
$$\widetilde{\mathrm{Set}}(f):=\Big\{ (\widetilde{\sigma}, \widetilde{A}) \ \big| \  \widetilde{\sigma}\in \mathrm{FPFI}(U), \ \widetilde{A}\subset V/\sim_{\widetilde{\sigma}} \Big\}.$$
\end{definition}

Recall that for a set $E$, we denote by $\mathfrak{S}_{E}$ the group of permutations of $E$. For $E$ and $F$ finite sets there exists a unique operation 
\begin{equation}\label{eq:2022-05-20-disjointofperm}
\mathfrak{S}_E\times \mathfrak{S}_F\to \mathfrak{S}_{E\sqcup F}
\end{equation}
taking $(\alpha,\beta)\in \mathfrak{S}_{E}\times \mathfrak{S}_{F}$ to the unique permutation $\gamma$ of $E\sqcup F$ such that $\gamma_{|E}=\alpha$ and $\gamma_{|F}=\beta$. We denote $\gamma:=\alpha\sqcup \beta$ and call it  the \emph{disjoint union} of $\alpha$ and $\beta$.

\begin{lemma}\label{r:2022-05-20lemma2} Let $U$ and $V$ be finite sets and $f:U\to V$ be a surjective map. 

$(a)$ To each $(A,\sigma,\sigma')\in \mathrm{Set}(f)$ we associate the pair $(\widetilde{\sigma}, \widetilde{A})$ where $\widetilde{\sigma}:=\sigma\sqcup \sigma'$ and $\widetilde{A}$ is the image of $A$ in $V/\sim_{\widetilde{\sigma}}$ under the canonical projection $V\to V/\sim_{\widetilde{\sigma}}$. Then  $(\widetilde{\sigma}, \widetilde{A})\in \widetilde{\mathrm{Set}}(f)$ and the map
$$\Phi: \mathrm{Set}(f)\longrightarrow \widetilde{\mathrm{Set}}(f),  \quad \quad (A,\sigma,\sigma')\longmapsto (\widetilde{\sigma}, \widetilde{A})$$ 
is a bijection.

$(b)$ There are involutions $\mathrm{a}$ and $\widetilde{\mathrm{a}}$ of $\mathrm{Set}(f)$ and $\widetilde{\mathrm{Set}}(f)$, respectively defined by $\mathrm{a}(A,\sigma,\sigma'):=(V\setminus A,\sigma',\sigma)$ and $\widetilde{\mathrm{a}}(\widetilde{\sigma},\widetilde{A}):=\big(\widetilde{\sigma},(V/\sim_{\widetilde{\sigma}})\setminus \widetilde{A}\big)$ for any $(A,\sigma,\sigma')\in \mathrm{Set}(f)$ and $(\widetilde{\sigma},\widetilde{A})\in\widetilde{\mathrm{Set}}(f)$. They are such that $\widetilde{\mathrm{a}} \circ {\Phi}=\Phi \circ \mathrm{a}$. 
\end{lemma}

\begin{proof}
$(a)$ Since $f$ is surjective, $U=f^{-1}(A) \sqcup f^{-1}(V\setminus A)$. One checks that the disjoint union operation~\eqref{eq:2022-05-20-disjointofperm} $\mathfrak{S}_{f^{-1}(A)}\times \mathfrak{S}_{f^{-1}(V\setminus A)}\to \mathfrak{S}_{U}$ induces a map
$$\mathrm{FPFI}(f^{-1}(A))\times \mathrm{FPFI}(f^{-1}(V\setminus A))\longrightarrow \mathrm{FPFI}(U).$$
Thus, if $(A,\sigma,\sigma')\in \mathrm{Set}(f)$ then $\sigma\sqcup\sigma'\in \mathrm{FPFI}(U)$ and therefore  $(\widetilde{\sigma}, \widetilde{A})\in \widetilde{\mathrm{Set}}(f)$. To show that $\Phi$ is bijective we define a map $\Psi: \widetilde{\mathrm{Set}}(f)\to \mathrm{Set}(f)$ and show that
$\Phi\circ \Psi = \mathrm{Id}_{\widetilde{\mathrm{Set}}(f)}$ and $\Psi\circ \Phi = \mathrm{Id}_{\mathrm{Set}(f)}$.

For $(\widetilde{\sigma},\widetilde{A})\in\widetilde{\mathrm{Set}}(f)$ we set $A\subset V$ equal to the preimage of $\widetilde{A}$ under the projection $V\to V/\sim_{\widetilde{\sigma}}$ and $\sigma:=\widetilde{\sigma}_{|f^{-1}(A)}$ and $\sigma':=\widetilde{\sigma}_{|f^{-1}(V\setminus A)}$. Notice that $A=\bigcup_{a\in V/\sim_{\widetilde{\sigma}}}\hat{a}$ where $\hat{a}$ is preimage of $a$ under the projection $V\to V/\sim_{\widetilde{\sigma}}$. That is, $A$ is a union of orbits of $\sim_{\widetilde{\sigma}}$. Since the orbits of $\sim_{\widetilde{\sigma}}$ are unions of the images of orbits of $\widetilde{\sigma}$, the preimages of the unions of orbits of $\sim_{\widetilde{\sigma}}$ are unions of orbits of $\widetilde{\sigma}$. Therefore,   $f^{-1}(A)$ is stable under $\widetilde{\sigma}$ which, since $\widetilde{\sigma}$ is a permutation of $U$, implies that  $f^{-1}(V\setminus A)$ is also stable under $\widetilde{\sigma}$. Hence, $\sigma\in\mathrm{FPFI}(f^{-1}(A))$ and $\sigma'\in\mathrm{FPFI}(f^{-1}(V\setminus A))$ and therefore $(A,\sigma,\sigma')\in\mathrm{Set}(f)$. This assignment defines the map $\Psi: \widetilde{\mathrm{Set}}(f)\to \mathrm{Set}(f)$. 

We now proceed to show that $\Phi\circ \Psi = \mathrm{Id}_{\widetilde{\mathrm{Set}}(f)}$ and $\Psi\circ \Phi = \mathrm{Id}_{\mathrm{Set}(f)}$.  Let $(\widetilde{\sigma}, \widetilde{A})$ in $\widetilde{\mathrm{Set}}(f)$. Then $\Psi(\widetilde{\sigma}, \widetilde{A})=(A, \sigma, \sigma')$ given as in the above paragraph so $\Phi \circ \Psi(\widetilde{\sigma}, \widetilde{A}) =  \Phi(A, \sigma, \sigma') = (\widetilde{\widetilde{\sigma}}, \widetilde{\widetilde{A}})$. Here $\widetilde{\widetilde{\sigma}}=\sigma \sqcup \sigma'$, which is equal to $\widetilde{\sigma}$ since $\sigma$ and $\sigma'$ are the restrictions of $\widetilde \sigma$ to complementary subsets of $U$. Besides, $\widetilde{\widetilde{A}}$ is the image  $\mathrm{Im}\big(A \subset V \longrightarrow V/\sim_{\widetilde{\widetilde{\sigma}}}= V/\sim_{\widetilde{\sigma}}\big)$ but $A$ is the preimage of $\widetilde{A}$ under the same map, which is  surjective, and therefore $\widetilde{\widetilde{A}}= \widetilde{A}$. We then conclude $\Phi \circ \Psi=\mathrm{Id}_{\widetilde{\mathrm{Set}}(f)}$.  

Let $(A,\sigma,\sigma') \in \mathrm{Set}(f)$. Then $\Phi(A,\sigma,\sigma')=(\widetilde{\sigma}, \widetilde{A})$ as in the statement of the lemma and $\Psi \circ \Phi(A, \sigma, \sigma')=\Psi(\widetilde{\sigma}, \widetilde{A})= (\underline{A},\underline{\sigma},\underline{\sigma}')$. Here $\underline{A}$ is the preimage of $\widetilde{A} \subset V/\sim_{\widetilde{\sigma}}$ by the projection $V \to V/\sim_{\widetilde{\sigma}}$, while $\widetilde{A}$ is the image of $A$ under the same projection. Since $\widetilde{\sigma} \in \mathrm{Perm}(U)$ preserves $f^{-1}(A)$, then $A$ is a union of equivalence classes under $\sim_{\widetilde{\sigma}}$, which implies that $\underline{A}=A$.  Then $\underline{\sigma}$ and $\underline{\sigma'}$ are respectively the restrictions of $\widetilde{\sigma}$ to  $f^{-1}(\underline{A}) = f^{-1}(A)$ and $f^{-1}(V\setminus \underline{A})=f^{-1}(V\setminus  A)$, which are $\sigma$ and $\sigma'$. Hence $\Psi \circ \Phi= \mathrm{Id}_{\mathrm{Set}(f)}$.  

$(b)$ The first statement is obvious. If $(A,\sigma,\sigma') \in \mathrm{Set}(f)$, then $\Phi(A,\sigma,\sigma')$ is equal to $(\widetilde{\sigma},\widetilde{A})$ given by $(a)$. Then $\widetilde{\mathrm{a}}\circ \Phi(A,\sigma,\sigma')=\widetilde{a}(\widetilde{\sigma}, \widetilde{A}) = \big(\widetilde{\sigma},(V/\sim_{\widetilde{\sigma}})\setminus \widetilde{A}\big)$. On the other hand, $\mathrm{a}(A,\sigma,\sigma')=(V\setminus A, \sigma', \sigma)$. Then $\Phi \circ \mathrm{a}(A,\sigma,\sigma')=\Phi(V\setminus A,\sigma',\sigma)=(\widetilde{\widetilde{\sigma}},{\widetilde{B}})$ where $\widetilde{\widetilde{\sigma}}:=\sigma' \sqcup \sigma=\widetilde{\sigma}$, and ${\widetilde{B}}$ is the image in $ V/\sim_{\widetilde{\widetilde{\sigma}}}= V/\sim_{\widetilde{\sigma}})$ of $V\setminus A \subset V$ under the canonical projection $V \to V/\sim_{\widetilde{\widetilde{\sigma}}}= V/\sim_{\widetilde{\sigma}}$. Since $A$ is a union of classes of $\sim_{\widetilde{\sigma}}$, then ${\widetilde{B}} = (V/\sim_{\widetilde{\sigma}})\setminus \mathrm{Im}(A \subset V \to V/\sim_{\widetilde{\sigma}}) = (V/\sim_{\widetilde{\sigma}})\setminus \widetilde{A}$. Therefore $\Phi \circ \mathrm{a}(A,\sigma,\sigma')= \widetilde{\mathrm{a}} \circ \Phi(A,\sigma,\sigma')$. 

\end{proof}

\begin{lemma}\label{r20220511lemma1} For any $\underline{D}\in \mathring{\mathrm{Jac}}(\vec{\varnothing},[\![1,k]\!],\bullet)$ there is a bijection
\begin{equation}
\mathrm{Bij}_{\underline{D}}: \mathrm{Set}_{\underline{D}} \longrightarrow \widetilde{\mathrm{Set}}_{\underline{D}}
\end{equation}
given by $\mathrm{Bij}_{\underline{D}}(A,\sigma,\sigma') = (\widetilde{\sigma}, \widetilde{A})$, where $\widetilde{\sigma}:=\sigma\sqcup \sigma'$ and $\widetilde{A}$ is the image of $A$ in $\pi_0\big(\big\langle\underline{D}, \mathring{\mathrm{Jac}}(\widetilde{\sigma})\big\rangle\big)$ under the canonical map $\pi_0(\underline{D})\to \pi_0\big(\big\langle\underline{D}, \mathring{\mathrm{Jac}}(\widetilde{\sigma})\big\rangle\big)$.
\end{lemma}
\begin{proof}
Let $\underline{D}\in\mathring{\mathrm{Jac}}(\vec{\varnothing},[\![1,k]\!],\bullet)$. Let $f:\partial{\underline{D}}_f \to \pi_0(\underline{D})$ be the map taking $x\in \partial{\underline{D}}_f$ to the unique $C\in\pi_0(\underline{D})$ such that $x\in C$. Using the notation of Definition~\ref{def:2022-05-30}, one checks that  $\mathrm{Set}_{\underline{D}} = \mathrm{Set}(f)$ and $\widetilde{\mathrm{Set}}_{\underline{D}} = \widetilde{\mathrm{Set}}(f)$ and the map $\mathrm{Bij}$ identifies with the map  $\Phi$ in Lemma~\ref{r:2022-05-20lemma2}. The result then follows from this lemma.
\end{proof}

\begin{lemma}\label{r20220511lemma2}  For any $\underline{D}\in \mathring{\mathrm{Jac}}(\vec{\varnothing},[\![1,k]\!],\bullet)$ the diagram
\begin{equation}
\xymatrix{\mathrm{Set}_{\underline{D}}\ar[rr]^-{(l_{\underline{D}}, r_{\underline{D}})}\ar[d]_-{\mathrm{Bij}_{\underline{D}}} & & \mathring{\mathrm{Jac}}(\vec{\varnothing},[\![1,k]\!]) \times \mathring{\mathrm{Jac}}(\vec{\varnothing},[\![1,k]\!]) \\ 
\widetilde{\mathrm{Set}}_{\underline{D}} \ar[rru]_-{(\widetilde{l}_{\underline{D}}, \widetilde{r}_{\underline{D}})} & }
\end{equation}
commutes.
\end{lemma}

\begin{proof}
Let $\underline{D}=(D,\partial{D}_a, \partial{D}_f,\varphi)\in\mathring{\mathrm{Jac}}(\vec{\varnothing},[\![1,k]\!],\bullet)$. We  show first that the diagram 
\begin{equation}\label{equ:20220511diag1}
\xymatrix{\mathrm{Set}_{\underline{D}}\ar[rr]^-{l_{\underline{D}}}\ar[d]_-{\mathrm{Bij}_{\underline{D}}} & & \mathring{\mathrm{Jac}}(\vec{\varnothing},[\![1,k]\!])\\ 
\widetilde{\mathrm{Set}}_{\underline{D}} \ar[rru]_-{\widetilde{l}_{\underline{D}}} & }
\end{equation}
commutes. Indeed, let $(A,\sigma,\sigma')\in \mathrm{Set}_{\underline{D}}$, $\widetilde{\sigma}:=\sigma\sqcup\sigma'$, $\widetilde{A}:=\mathrm{Im}\big(A\subset \pi_0(\underline{D})\to\pi_0\big(\langle \underline{D}, \mathring{\mathrm{Jac}}(\widetilde{\sigma}) \rangle\big)\big)$ and $\widetilde{\underline{D}}:=\langle \underline{D}, \mathring{\mathrm{Jac}}(\widetilde{\sigma})\rangle$. Then $\widetilde{\underline{D}}=(\widetilde{D},\partial{\widetilde{D}}_a= \partial{D}_a, \partial{\widetilde{D}}_f=\emptyset, \widetilde{\varphi} =\varphi)$,
$$\underline{D}_A = \left(\bigsqcup_{x\in A} x, ({\partial D}_a)\cap \bigsqcup_{x\in A} x, ({\partial D}_f) \cap \bigsqcup_{x\in A} x, \varphi_{|({\partial D}_a) \cap \bigsqcup_{x\in A} x}\right)$$
and 
$$l_{\underline{D}}(A,\sigma,\sigma') = \langle \underline{D}_A,\mathring{\mathrm{Jac}}(\sigma)\rangle = \left(\bigsqcup_{z\in B} z, ({\partial D}_a)\cap \bigsqcup_{z\in B} z, \emptyset, \varphi_{|({\partial D}_a) \cap \bigsqcup_{z\in B} z}\right),$$
where $B:=\mathrm{Im}\big(A\subset\pi_0(\underline{D})\to \pi_0\big(\langle\underline{D}, \mathring{\mathrm{Jac}}(\sigma)\rangle\big)\big)$. Here $$\langle\underline{D}, \mathring{\mathrm{Jac}}(\sigma)\rangle\in \mathring{\mathrm{Jac}}(\vec{\varnothing},[\![1,k]\!],|(\partial{\underline{D}}_f)\cap\bigsqcup_{x\in\pi_0(\underline{D})\setminus A}x|)$$ is the Jacobi diagram obtained from $\underline{D}$ by identifying the univalent vertices in $(\partial{{D}}_f)\cap\bigsqcup_{x\in A}x$ according to $\sigma$. Since $\widetilde{\sigma}=\sigma\sqcup \sigma'$, then $B=\widetilde{A}$. We also have $({\partial D}_a) \cap \bigsqcup_{z\in B} z = ({\partial D}_a) \cap \bigsqcup_{x\in A} x$.
Therefore, 
\begin{equation*}
\begin{split}
\widetilde{l}_{\underline{D}}\circ\mathrm{Bij}_{\underline{D}}(A,\sigma,\sigma')=\widetilde{l}_{\underline{D}}(\widetilde{\sigma},\widetilde{A})=\widetilde{\underline{D}}_{\widetilde{A}} & = \left(\bigsqcup_{z\in \widetilde{A}} z, (\partial{\widetilde{D}}_a)\cap \bigsqcup_{z\in \widetilde{A}} z, ({\partial{\widetilde{D}}}_f) \cap \bigsqcup_{z\in \widetilde{A}} z, \widetilde{\varphi}_{|{\partial {\widetilde{D}}}_a \cap \bigsqcup_{z\in \widetilde{A}} z}\right) \\
& = \left(\bigsqcup_{z\in B} z, (\partial{D}_a)\cap \bigsqcup_{z\in B} z, \emptyset, {\varphi}_{|({\partial{D}}_a)\cap \bigsqcup_{z\in B} z}\right) = l_{\underline{D}}(A,\sigma,\sigma'). 
\end{split}
\end{equation*}

We show now the commutativity of the diagram 
\begin{equation}\label{equ:20220511diag2}
\xymatrix{\mathrm{Set}_{\underline{D}}\ar[rr]^-{r_{\underline{D}}}\ar[d]_-{\mathrm{Bij}_{\underline{D}}} & & \mathring{\mathrm{Jac}}(\vec{\varnothing},[\![1,k]\!])  \\ 
\widetilde{\mathrm{Set}}_{\underline{D}} \ar[rru]_-{\widetilde{r}_{\underline{D}}} & }
\end{equation}

Consider the involutions 
$$\mathrm{aut}_{\underline{D}}: \mathrm{Set}_{\underline{D}}\longrightarrow \mathrm{Set}_{\underline{D}}, \quad \quad (A,\sigma,\sigma')\longmapsto (\pi_0(\underline{D}))\setminus A, \sigma' ,\sigma)
$$
and
$$\widetilde{\mathrm{aut}}_{\underline{D}}: \widetilde{\mathrm{Set}}_{\underline{D}}\longrightarrow \widetilde{\mathrm{Set}}_{\underline{D}}, \quad \quad (\widetilde{\sigma}, \widetilde{A}) \longmapsto \big(\widetilde{\sigma}, \pi_0\big(\big\langle \underline{D},\mathring{\mathrm{Jac}}(\widetilde{\sigma})\big\rangle\big)\setminus \widetilde{A}\big).
$$
One checks that 
\begin{equation}\label{equ:20220511equrl}
r_{\underline{D}} = l_{\underline{D}} \circ \mathrm{aut}_{\underline{D}} \quad \quad \text{ and } \quad \quad  \widetilde{r}_{\underline{D}} = \widetilde{l}_{\underline{D}} \circ \widetilde{\mathrm{aut}}_{\underline{D}}.
\end{equation}
Besides, using the notations in the proof of Lemma~\ref{r20220511lemma1}, the maps $\mathrm{Bij}_{\underline{D}}$, $\mathrm{aut}_{\underline{D}}$ and $\widetilde{\mathrm{aut}}_{\underline{D}}$ are respectively identified with the maps $\Phi$, $\mathrm{a}$ and $\widetilde{\mathrm{a}}$ in Lemma~\ref{r:2022-05-20lemma2}. It then follows from Lemma~\ref{r:2022-05-20lemma2}$(b)$ that the diagram 
\begin{equation}\label{equ:20220511diagsquare}
\xymatrix{\mathrm{Set}_{\underline{D}}\ar[rr]^-{{\mathrm{Bij}}_{\underline{D}}}\ar[d]_-{\mathrm{aut}_{\underline{D}}} & & \widetilde{\mathrm{Set}}_{\underline{D}} \ar[d]^-{\widetilde{\mathrm{aut}}_{\underline{D}}}\\ 
{\mathrm{Set}}_{\underline{D}} \ar[rr]_-{{\mathrm{Bij}}_{\underline{D}}} & &\widetilde{\mathrm{Set}}_{\underline{D}} }
\end{equation}
commutes.

Diagram~\eqref{equ:20220511diag2} corresponds to the exterior  triangle in the diagram
\begin{equation}
\xymatrix{\mathrm{Set}_{\underline{D}}\ar[rrrr]^-{{\mathrm{Bij}}_{\underline{D}}} \ar@/_3pc/[ddrr]_{r_{\underline{D}}} \ar[rd]^-{\mathrm{aut}_{\underline{D}}}&  &  &  & \widetilde{\mathrm{Set}}_{\underline{D}}\ar@/^3pc/[ddll]^{\widetilde{r}_{\underline{D}}}\ar[ld]_-{\widetilde{\mathrm{aut}}_{\underline{D}}} \\
& \mathrm{Set}_{\underline{D}}\ar[rr]^-{{\mathrm{Bij}}_{\underline{D}}}\ar[dr]^-{l_{\underline{D}}} &  & \widetilde{\mathrm{Set}}_{\underline{D}} \ar[dl]_-{\widetilde{l}_{\underline{D}}} &  \\
  &  & \mathring{\mathrm{Jac}}(\vec{\varnothing},[\![1,k]\!]) &  & }
\end{equation}
The inner square and inner triangles commute because of \eqref{equ:20220511diagsquare}, \eqref{equ:20220511diag1} and \eqref{equ:20220511equrl}. We deduce then the commutativity of Diagram~\eqref{equ:20220511diag2}. This completes the proof.
\end{proof}

\begin{definition}\label{def:2022-05-19tpair-k} Define the linear map 
\begin{equation}
\mathrm{Pair}_{(\vec{\varnothing},[\![1,k]\!])}: \mathbb{C}\mathring{\mathrm{Jac}}(\vec{\varnothing},[\![1,k]\!],\bullet)\longrightarrow \mathbb{C}\mathring{\mathrm{Jac}}(\vec{\varnothing},[\![1,k]\!])
\end{equation}  \index[notation]{Pair_{(\vec{\varnothing},[\![1,k]\!])}@$\mathrm{Pair}_{(\vec{\varnothing},[\![1,k]\!])}$}
by  $$\mathrm{Pair}_{(\vec{\varnothing},[\![1,k]\!])}(\underline{D}):=\sum_{\sigma\in\mathrm{FPFI}({\partial{\underline{D}}}_f)}\big\langle \underline{D}, \mathring{\mathrm{Jac}}(\sigma) \big\rangle
$$ for any $\underline{D}\in \mathring{\mathrm{Jac}}(\vec{\varnothing},[\![1,k]\!],\bullet)$. (Recall that by definition an empty sum is equal to $0$.)
\end{definition}

\begin{proposition}\label{r:2022-05-19Copr}The diagram 
\begin{equation}\label{eq:2022-05-19commdiag}
\xymatrix{\mathbb{C}\mathring{\mathrm{Jac}}(\vec{\varnothing},[\![1,k]\!],\bullet)\ar[rr]^-{{\mathrm{Pair}}_{(\vec{\varnothing},[\![1,k]\!])}}\ar[d]_-{\Delta_{\mathring{\mathrm{Jac}}(\vec{\varnothing},[\![1,k]\!],\bullet)}} & & \mathbb{C}\mathring{\mathrm{Jac}}(\vec{\varnothing},[\![1,k]\!])\ar[d]^-{\Delta_{\mathring{\mathrm{Jac}}(\vec{\varnothing},[\![1,k]\!])}}\\ 
\mathbb{C}\mathring{\mathrm{Jac}}(\vec{\varnothing},[\![1,k]\!],\bullet)\otimes \mathbb{C}\mathring{\mathrm{Jac}}(\vec{\varnothing},[\![1,k]\!],\bullet) \ar[rr]_-{{\mathrm{Pair}}_{(\vec{\varnothing},[\![1,k]\!])}^{\otimes 2}} & &\mathbb{C}\mathring{\mathrm{Jac}}(\vec{\varnothing},[\![1,k]\!])\otimes \mathbb{C}\mathring{\mathrm{Jac}}(\vec{\varnothing},[\![1,k]\!]) }
\end{equation}

commutes.
\end{proposition}

\begin{proof}
Let $\underline{D}\in \mathring{\mathrm{Jac}}(\vec{\varnothing},[\![1,k]\!],\bullet)$, then
\begin{align*}
\Delta_{\mathring{\mathrm{Jac}}(\vec{\varnothing},[\![1,k]\!])}\circ \mathrm{Pair}_{(\vec{\varnothing},[\![1,k]\!])}(\underline{D})  = \Delta_{\mathring{\mathrm{Jac}}(\vec{\varnothing},[\![1,k]\!])}\left(\sum_{\tilde{\sigma}\in\mathrm{FPFI}({\partial{\underline{D}}}_f)} \big\langle \underline{D}, \mathring{\mathrm{Jac}}(\widetilde{\sigma})\big\rangle \right) \quad \quad \quad \quad \quad \quad \quad \quad \quad \quad \quad \quad\\
 \quad \quad  = \sum_{(\tilde{\sigma}, \tilde{A})\in \widetilde{\mathrm{Set}}_{\underline{D}}} \big\langle \underline{D}, \mathring{\mathrm{Jac}}(\widetilde{\sigma})\big\rangle_{\widetilde{A}}\otimes \big\langle \underline{D}, \mathring{\mathrm{Jac}}(\widetilde{\sigma})\big\rangle_{\pi_0\big(\big\langle\underline{D}, \mathring{\mathrm{Jac}}(\widetilde{\sigma})\big\rangle\big)\setminus \widetilde{A}} 
  = \sum_{(\tilde{\sigma}, \tilde{A})\in \widetilde{\mathrm{Set}}_{\underline{D}}} \widetilde{l}_{\underline{D}}(\widetilde{\sigma},\widetilde{A}) \otimes \widetilde{r}_{\underline{D}}(\widetilde{\sigma},\widetilde{A}) \\
  = \sum_{(\tilde{\sigma}, \tilde{A})\in \widetilde{\mathrm{Set}}_{\underline{D}}} {l}_{\underline{D}}\circ \mathrm{Bij}^{-1}_{\underline{D}}(\widetilde{\sigma},\widetilde{A}) \otimes {r}_{\underline{D}}\circ \mathrm{Bij}^{-1}_{\underline{D}}(\widetilde{\sigma},\widetilde{A})
   = \sum_{(A,\sigma,\sigma')\in {\mathrm{Set}}_{\underline{D}}} {l}_{\underline{D}}(A,\sigma,\sigma') \otimes {r}_{\underline{D}}(A,\sigma,\sigma')\\
   = \sum_{(A,\sigma,\sigma')\in {\mathrm{Set}}_{\underline{D}}} \big\langle \underline{D}_A, \mathring{\mathrm{Jac}}(\sigma)\big\rangle \otimes \big\langle \underline{D}_{\pi_0(\underline{D})\setminus A}, \mathring{\mathrm{Jac}}(\sigma')\big\rangle
    = \mathrm{Pair}^{\otimes 2}_{(\vec{\varnothing},[\![1,k]\!])}\circ\Delta_{\mathring{\mathrm{Jac}}(\vec{\varnothing},[\![1,k]\!],\bullet)}(\underline{D}).
\end{align*}
We have used Lemma~\ref{r20220511lemma1} in the fourth equality and Lemma~\ref{r20220511lemma2} in the fifth equality. 
\end{proof}

\subsubsection{An algebra homomorphism \texorpdfstring{$\Delta^{\mathfrak{b}}_{n,p}:\mathfrak{b}(n)\to \mathfrak{b}(p) \otimes \mathfrak{b}(n-p)$}{bn-to-bp-otimes-bn-p}}

Recall from Lemma~\ref{rlemma1_20211111} the graded ideal $P^{n+1} = \bigoplus_{k\geq 0} P^{n+1}_k$ of $\bigoplus_{k\geq 0} \mathring{\mathcal A}^\wedge_k$.

\begin{definition} Define the category of \emph{finitely fibered set diagrams} $\mathcal{S}et\mathcal{D}$   \index[notation]{Set\mathcal{D}@$\mathcal{S}et\mathcal{D}$} as follows. The objects are sets. For $A,B$ two objects, $\mathcal{S}et\mathcal{D}(A,B)$ is the collection of set diagrams $A\xleftarrow{ \ f \ } X \xrightarrow{\  f' \ }B$, where the map $f: X\to A$ has finite fibers. For objects $A$, $B$ and $C$, the composition
$\mathcal{S}et\mathcal{D} (A,B)\times\mathcal{S}et\mathcal{D}(B,C) \to \mathcal{S}et\mathcal{D}(A,C)$ takes the pair $\big(A\xleftarrow{ \  f \  } X \xrightarrow{ \  f' \ }B, B\xleftarrow{ \  g  \ } Y \xrightarrow{ \  g' \  } C\big)$ to the diagram $A\xleftarrow{\ \hat{f} \ } X\times_B Y  \xrightarrow{\ \hat{g}'} C$, where $X\times_B Y:=\{(x,y)\in X\times Y \ | \  f'(x)=g(y)\}$ and $\hat{f}:X\times_B Y\to A $ (resp. $\hat{g}':X\times_B Y\to B $) is the composition $X\times_B Y\to X \xrightarrow{\ f\ } A$ (resp. $X\times_B Y\to Y \xrightarrow{\ g'\ } B$) where the first map is the canonical projection. 
\end{definition}

\begin{lemma}\label{r:2022-05-27-r1}  Let $\mathcal Vect$ be the category of $\mathbb{C}$-vector spaces. There is a functor $F:\mathcal{S}et\mathcal{D}  \to \mathcal Vect$ taking an object $A$ in $\mathcal{S}et\mathcal{D}$ to the free vector space $\mathbb{C}A$  with 
basis $A$, and taking a morphism $A\xleftarrow{ \  f \  } X \xrightarrow{ \  f'  \ }B$ in $\mathcal{S}et\mathcal{D}(A,B)$  to the linear map $\mathbb{C}A \to \mathbb{C}B$ given by $a \mapsto \sum_{x \in f^{-1}(a)}f'(x)$ for any $a\in A$.
\end{lemma}
\begin{proof}
The proof is straightforward.
\end{proof}

\begin{lemma}\label{r:2022-05-27-r2}
 If $A,B\in\mathcal{S}et\mathcal{D}$, if $A\xleftarrow{ \  f  \ } X \xrightarrow{ \ f' \ }B$ is in $\mathcal{S}et\mathcal{D}(A,B)$ and if $\varphi:Y \to X$ is a set morphism with finite fibers, then $A\xleftarrow{\  f \circ \varphi \ } Y \xrightarrow{\ f' \circ \varphi \ } B$ belongs to $\mathcal{S}et\mathcal{D}(A,B)$. If moreover $c \in \mathbb{Z}_{\geq 0}$ is such that for all $x \in X$, $|\varphi^{-1}(x)|=c$, then $F(A\xleftarrow{ \ f \circ \varphi \ } Y \xrightarrow{ \ f' \circ \varphi \ } B)=cF(A\xleftarrow{\  f  \ }Y \xrightarrow{ \  f' \  }B)$.
\end{lemma}

\begin{proof}
Since both $f:X\to A$ and $\varphi: Y\to X$ have finite fibers, so does $f\circ \varphi: Y\to X$. This implies the first statement of the lemma. If $|\varphi^{-1}(x)|=c$ for all $x\in X$, then
\begin{equation*}
\begin{split}
F\big(A\xleftarrow{ \ f \circ \varphi \ } Y \xrightarrow{ \ f' \circ \varphi \ } B\big)(a) = \sum_{y\in(f\circ\varphi)^{-1}(a)}(f'\circ\varphi)(y)=\sum_{y\in\varphi^{-1}(f^{-1}(a))}f'(\varphi(y)) = c \sum_{x\in f^{-1}(a)}f'(x) \\ = cF\big(A\xleftarrow{ \ f  \ } X \xrightarrow{  f'  \ } B\big)(a)
\end{split}
\end{equation*}
for any $a\in A$. This proves the second statement of the lemma.
\end{proof}

 We apply the previous lemmas for sets of Jacobi diagrams as follows. For a finite set $S$, we denote by $\mathcal{P}_2(S)$ the set of two-element subsets of $S$. Let $k,m \geq 0$. Consider the set diagram
\begin{equation}\label{eq:2022-05-27-eq1}
\mathring{\mathrm{Jac}}(\vec{\varnothing},[\![1,k]\!],2(m+1))\xleftarrow{\ f_{m} \ } X_{m,m+1}\xrightarrow{\ f'_m \ } \mathring{\mathrm{Jac}}(\vec{\varnothing},[\![1,k]\!],2m),
\end{equation}
where $$X_{m,m+1}:=\big\{(\underline D,\alpha)\ | \ \underline{D}\in\mathring{\mathrm{Jac}}(\vec{\varnothing},[\![1,k]\!],2(m+1)),\  \alpha \in \mathcal{P}_2(\partial \underline D_f)\big\},$$
and  $f_m:X_{m,m+1} \to \mathring{\mathrm{Jac}}(\vec{\varnothing},[\![1,k]\!],2(m+1))$ is the projection $(\underline D,  \alpha) \mapsto \underline D$, while $f'_m:X_{m,m+1}\to\mathring{\mathrm{Jac}}(\vec{\varnothing},[\![1,k]\!],2m)$ is the map taking $(\underline D, \alpha)$ to $\underline{D}_{\alpha}$, which is obtained from $\underline D$ by connecting the pairs of endpoints of $\underline D$ corresponding to $\alpha$. For any $\underline{D} \in \mathring{\mathrm{Jac}}(\vec{\varnothing},[\![1,k]\!],2(m+1))$, the set $\mathcal P_2(\partial \underline{D}_f)$ is finite, therefore the set diagram~\eqref{eq:2022-05-27-eq1} is a morphism in $\mathcal{S}et\mathcal{D}\big(\mathring{\mathrm{Jac}}(\vec{\varnothing},[\![1,k]\!], 2(m+1)),\mathring{\mathrm{Jac}}(\vec{\varnothing},[\![1,k]\!],2m)\big).$
Denote by $\Psi_m$ the image of this morphism by the functor $F$ from Lemma~\ref{r:2022-05-27-r1}, which is a vector space morphism $$\Psi_m:\mathbb{C}\mathring{\mathrm{Jac}}(\vec{\varnothing},[\![1,k]\!],2(m+1)) \longrightarrow \mathbb{C} \mathring{\mathrm{Jac}}(\vec{\varnothing},[\![1,k]\!],2m).$$

Let $\mathrm{Pair}_{(\vec{\varnothing},[\![1,k]\!],m)}: \mathbb{C}\mathring{\mathrm{Jac}}(\vec{\varnothing},[\![1,k]\!],m)\to \mathbb{C}\mathring{\mathrm{Jac}}((\vec{\varnothing},[\![1,k]\!])$ denote the restriction of the map $\mathrm{Pair}_{(\vec{\varnothing},[\![1,k]\!])}$ from Definition~\ref{def:2022-05-19tpair-k} to $\mathring{\mathrm{Jac}}(\vec{\varnothing},[\![1,k]\!],m)$.

\begin{lemma}\label{r:2022-05-27mainlemma} For any $k,m \geq 0$, one has $\mathrm{Pair}_{(\vec{\varnothing},[\![1,k]\!],2m)} \circ \Psi_m=\binom{2(m+1)}{2}\mathrm{Pair}_{(\vec{\varnothing},[\![1,k]\!],2(m+1))}$.
\end{lemma}

\begin{proof}
One defines a morphism in $\mathcal{S}et\mathcal{D}\big(\mathring{\mathrm{Jac}}(\vec{\varnothing},[\![1,k]\!],2m),\mathring{\mathrm{Jac}}(\vec{\varnothing},[\![1,k]\!])\big)$ by the set diagram $$\mathring{\mathrm{Jac}}(\vec{\varnothing},[\![1,k]\!],2m) \xleftarrow{\ g_m \ } X_{m} \xrightarrow{\ g'_m \ } \mathring{\mathrm{Jac}}(\vec{\varnothing},[\![1,k]\!]),$$ where $X_{m}:=\{(\underline D,\sigma) \ | \ 
\underline{D}\in \mathring{\mathrm{Jac}}(\vec{\varnothing},[\![1,k]\!],2m) \text{ and } \sigma \in \mathrm{FPFI}(\partial \underline{D}_f)\}$, the map $$g_m: X_{m} \longrightarrow \mathring{\mathrm{Jac}}(\vec{\varnothing},[\![1,k]\!],2m)$$ is the projection $(\underline{D}, \sigma) \mapsto \underline{D}$ and the map $g'_m: X_{m} \to \mathring{\mathrm{Jac}}(\vec{\varnothing},[\![1,k]\!])$
takes $(\underline{D}, \sigma)$ to $\langle \underline{D},\mathring{\mathrm{Jac}}(\sigma)\rangle$, see \eqref{eq:2022-05-27eqpairing}.

It follows easily that 
\begin{equation}
F\big(\mathring{\mathrm{Jac}}(\vec{\varnothing},[\![1,k]\!],2m) \xleftarrow{ \ g_m \ } X_{m} \xrightarrow{\ g'_m \ } \mathring{\mathrm{Jac}}(\vec{\varnothing},[\![1,k]\!])\big)=\mathrm{Pair}_{(\vec{\varnothing},[\![1,k]\!],2m)},
\end{equation}
where $F$ is the functor from Lemma~\ref{r:2022-05-27-r1}.
Using the functoriality of $F$ we have
\begin{equation*}
\begin{split}
\mathrm{Pair}_{(\vec{\varnothing},[\![1,k]\!],2m)} \circ \Psi_m &= F\big(\mathring{\mathrm{Jac}}(\vec{\varnothing},[\![1,k]\!],2m) \xleftarrow{\ g_m \ } X_m \xrightarrow{ \ g'_m} \mathring{\mathrm{Jac}}(\vec{\varnothing},[\![1,k]\!])\big)\\
 &\qquad \qquad \circ F\big(\mathring{\mathrm{Jac}}(\vec{\varnothing},[\![1,k]\!],2(m+1)) \xleftarrow{\ f_m \ } X_{m,m+1} \xrightarrow{ \ f'_m \ }\mathring{\mathrm{Jac}}(\vec{\varnothing},[\![1,k]\!],2m)\big)\\
&=F\big(\mathring{\mathrm{Jac}}(\vec{\varnothing},[\![1,k]\!],2(m+1)) \xleftarrow{ \ \hat{f}_n } X_{m,m+1} \times_{\mathring{\mathrm{Jac}}(\vec{\varnothing},[\![1,k]\!],2m)}X_{2m}\xrightarrow{ \ \hat{g}'_m} \mathring{\mathrm{Jac}}(\vec{\varnothing},[\![1,k]\!])\big).
\end{split} 
\end{equation*}
Notice that $X_{m,m+1} \times_{\mathring{\mathrm{Jac}}(\vec{\varnothing},[\![1,k]\!],2m)}X_m \to X_{m+1}$ consists of the tuples $\big((\underline{D},\alpha),(\underline{E},\sigma)\big)$ with $$\underline{D}\in\mathring{\mathrm{Jac}}(\vec{\varnothing},[\![1,k]\!],2(m+1)), \ \  \alpha\in\mathcal{P}_2(\partial{\underline{D}}_f), \  \ \underline{E}\in\mathring{\mathrm{Jac}}(\vec{\varnothing},[\![1,k]\!],2m)  \  \text{ and } \  \sigma\in\mathrm{FPFI}(\partial{\underline{E}}_f)$$ such that $\underline{D}_{\alpha} =f'_m((\underline{D},\alpha)) = g_m(\underline{E},\sigma) = \underline{E}$.  Define $\varphi_m : X_{m,m+1} \times_{\mathring{\mathrm{Jac}}(\vec{\varnothing},[\![1,k]\!],2m)}X_m \to X_{m+1}$ by $$\big((\underline{D},\alpha),(\underline{E},\sigma)\big)\longmapsto (\underline{D}, \sigma\sqcup\alpha).$$  
We see at once that for any $x \in X_{m+1}$, we have $|\varphi^{-1}(x)|=\binom{2(m+1)}{2}$. Besides, for any $\big((\underline{D},\alpha),(\underline{E},\sigma)\big)\in X_{m,m+1} \times_{\mathring{\mathrm{Jac}}(\vec{\varnothing},[\![1,k]\!],2m)}X_m$ we have
$$g_{m+1}\circ\varphi_m\big((\underline{D},\alpha),(\underline{E},\sigma)\big) = g_{m+1}(\underline{D}, \sigma\sqcup\alpha) = \underline{D} = f_m(\underline{D},\alpha) = \hat{f}_m\big((\underline{D},\alpha),(\underline{E},\sigma)\big)$$
and
\begin{equation*}
\begin{split}
g'_{m+1}\circ\varphi_m\big((\underline{D},\alpha),(\underline{E},\sigma)\big)=g'_{m+1}(\underline{D},\sigma\sqcup\alpha) = \langle \underline{D},\mathring{\mathrm{Jac}}(\sigma\sqcup\alpha)\rangle = \langle \underline{D}_{\alpha}, \mathring{\mathrm{Jac}}(\sigma)\rangle = \langle \underline{E}, \mathring{\mathrm{Jac}}(\sigma)\rangle \\= g'_m(\underline{E},\sigma) = \hat{g}'_m\big((\underline{D},\alpha),(\underline{E},\sigma)\big),
\end{split}
\end{equation*}
that is,  $g_{m+1}\circ \varphi_m =\hat{f}_n$ and $g'_{m+1}\circ\varphi_m=\hat{g}'_m$. Therefore, using Lemma~\ref{r:2022-05-27-r2}, we obtain
\begin{equation*}
\begin{split}
F\big(\mathring{\mathrm{Jac}}&(\vec{\varnothing},[\![1,k]\!],2(m+1)) \xleftarrow{ \ \hat{f}_m } X_{m,m+1} \times_{\mathring{\mathrm{Jac}}(\vec{\varnothing},[\![1,k]\!],2m)}X_{2m}\xrightarrow{ \ \hat{g}'_m} \mathring{\mathrm{Jac}}(\vec{\varnothing},[\![1,k]\!])\big) \\
&= F\big(\mathring{\mathrm{Jac}}(\vec{\varnothing},[\![1,k]\!],2(m+1)) \xleftarrow{ \ g_{m+1}\circ \varphi_m } X_{m,m+1} \times_{\mathring{\mathrm{Jac}}(\vec{\varnothing},[\![1,k]\!],2m)}X_{2m}\xrightarrow{ \ {g}'_{m+1}\circ\varphi_m} \mathring{\mathrm{Jac}}(\vec{\varnothing},[\![1,k]\!])\big) \\
& = \binom{2(m+1)}{2}F\big(\mathring{\mathrm{Jac}}(\vec{\varnothing},[\![1,k]\!],2(m+1))\xleftarrow{\ g_{m+1} \  } X_{m+1}\xrightarrow{ \ g'_{m+1} \ }\mathring{\mathrm{Jac}}(\vec{\varnothing},[\![1,k]\!])\big)\\
& =\binom{2(m+1)}{2}\mathrm{Pair}_{(\vec{\varnothing},[\![1,k]\!],2(m+1)}. 
\end{split}
\end{equation*}
This finishes the proof.
\end{proof}

\begin{corollary}\label{r:2022-05-23inclusionPn+1inPn} For all $k\geq 0$ and $m\geq 2$ we have $P^{m+1}_k\subset P^{m}_k$ and therefore $P_{m+1}\subset P_{m}$.
\end{corollary}

\begin{proof}
 Let $k\geq 0$.  Let $[\ ]: \mathbb{C}\mathring{\mathrm{Jac}}(\vec{\varnothing},[\![1,k]\!])\to \mathring{\mathcal{A}}^\wedge_k$ be the canonical projection. From the definition of $\mathrm{Pair}_{(\vec{\varnothing},[\![1,k]\!])}$ (see Definition~\ref{def:2022-05-19tpair-k}) we have
\begin{equation}\label{eq:2022-05-23:Pn+1andPair1}
P^{m}_k= \big[\mathrm{Pair}_{(\vec{\varnothing},[\![1,k]\!],2m)}\big(\mathbb{C}\mathring{\mathrm{Jac}}(\vec{\varnothing},[\![1,k]\!],2m)\big)\big].
\end{equation}
Using \eqref{eq:2022-05-23:Pn+1andPair1} and Lemma~\ref{r:2022-05-27mainlemma} we obtain $P^{m+1}_k\subset P^{m}_k$. The sum of these inclusions over $k\geq 0$ then gives $P^{m+1}\subset P^{m}$.
\end{proof}

\begin{proposition}\label{r:2022-05-19coproductandP-n} One has
$$\Delta_{\mathcal{A}}(P^{n+1}) \subset P^{p+1} \otimes \mathcal{A} + \mathcal{A} \otimes P^{n-p+1},$$
where by simplicity of notation we set $\mathcal{A}:= \bigoplus_{k\geq 0} \mathring{\mathcal A}^\wedge_k$ and $P^{n+1}=\bigoplus_{k \geq 0}P^{n+1}_k$ is the graded ideal of $\mathcal{A}$ from Lemma~\ref{rlemma1_20211111}.
\end{proposition}

\begin{proof} Let $k\geq 0$. Recall that the coproducts $\Delta_{\mathring{\mathrm{Jac}}(\vec{\varnothing},[\![1,k]\!],0)}$ (from \eqref{eq:coproductinJaczero}) and $\Delta_{\mathring{\mathrm{Jac}}(\vec{\varnothing},[\![1,k]\!])}$  (from Definition~\ref{def:coproductinA-k}$(b)$) coincide. Besides, 
\begin{equation}\label{eq:2022-05-23:Pn+1andPair}
P^{n+1}_k = \big[\mathrm{Pair}_{(\vec{\varnothing},[\![1,k]\!])}\big(\mathbb{C}\mathring{\mathrm{Jac}}(\vec{\varnothing},[\![1,k]\!],2(n+1))\big)\big],
\end{equation}
where $[\ ]: \mathbb{C}\mathring{\mathrm{Jac}}(\vec{\varnothing},[\![1,k]\!])\to \mathring{\mathcal{A}}^\wedge_k$ is the canonical projection.

Let $m\geq 0$, by Lemma~\ref{r:2022-05-23lem:grading}, we have the inclusion
$$\Delta_{\mathring{\mathrm{Jac}}(\vec{\varnothing},[\![1,k]\!],\bullet)}\big(\mathbb{C} \mathring{\mathrm{Jac}}(\vec{\varnothing},[\![1,k]\!],m)\big) \subset \bigoplus_{0 \leq q \leq  m} \mathbb{C}\mathring{\mathrm{Jac}}(\vec{\varnothing},[\![1,k]\!],q) \otimes \mathbb{C} \mathring{\mathrm{Jac}}(\vec{\varnothing},[\![1,k]\!],m-q).$$ 
By the commutativity of Diagram \eqref{eq:2022-05-19commdiag} we then have
\begin{equation*}
\begin{split}
\Delta_{\mathring{\mathrm{Jac}}(\vec{\varnothing},[\![1,k]\!])}&\big(\mathrm{Pair}_{(\vec{\varnothing},[\![1,k]\!])}(\mathbb{C}\mathring{\mathrm{Jac}}(\vec{\varnothing},[\![1,k]\!],m))\big) \\
&\subset \sum_{0\leq q \leq m} \mathrm{Pair}_{(\vec{\varnothing},[\![1,k]\!])}(\mathbb{C} \mathring{\mathrm{Jac}}(\vec{\varnothing},[\![1,k]\!],q)) \otimes \mathrm{Pair}_{(\vec{\varnothing},[\![1,k]\!])}(\mathbb{C} \mathring{\mathrm{Jac}}(\vec{\varnothing},[\![1,k]\!], m-q)).
\end{split}
\end{equation*}

The compatibility of the map $\Delta_{\mathring{\mathcal{A}}^\wedge_k}$, the map $\Delta_{\mathring{\mathrm{Jac}}(\vec{\varnothing},[\![1,k]\!])}$ and the  projection $[\ ]: \mathbb{C}\mathring{\mathrm{Jac}}(\vec{\varnothing},[\![1,k]\!]) \to \mathring{\mathcal{A}}^\wedge_k$ imply
\begin{equation*}
\begin{split}
\Delta_{\mathring{\mathcal{A}}^\wedge_k}&
\big(\big[\mathrm{Pair}_{(\vec{\varnothing},[\![1,k]\!])}(\mathbb{C}\mathring{\mathrm{Jac}}(\vec{\varnothing},[\![1,k]\!],m))\big]\big)\\
&  \subset \sum_{0\leq q \leq m} [\mathrm{Pair}_{(\vec{\varnothing},[\![1,k]\!])}(\mathbb{C} \mathring{\mathrm{Jac}}(\vec{\varnothing},[\![1,k]\!],q))] \otimes [\mathrm{Pair}_{(\vec{\varnothing},[\![1,k]\!])}(\mathbb{C} \mathring{\mathrm{Jac}}(\vec{\varnothing},[\![1,k]\!], m-q))].
\end{split}
\end{equation*}

Specializing this inclusion for $m=2(n+1)$, using \eqref{eq:2022-05-23:Pn+1andPair} and $[\mathrm{Pair}_{(\vec{\varnothing},[\![1,k]\!])}(\mathbb{C}\mathring{\mathrm{Jac}}(\vec{\varnothing},[\![1,k]\!],u))]=0$ for $u$ odd, one then obtains 
\begin{equation*}
\begin{split}
 \Delta_{\mathring{\mathcal{A}}^{\wedge}_k}(P^{n+1}_k) \subset \sum_{2\leq q\leq n-1} P^q_k \otimes P^{n+1-q}_k. \quad \quad \quad \quad \quad \quad \quad \quad \quad \quad \quad \quad \quad \quad \quad \quad \quad \quad \quad \quad \quad \quad \quad \quad \\
\quad \quad \quad \quad =  P^2_k \otimes P^{n-1}_k + \cdots+ P^p_k \otimes P^{n+1-p}_k + P^{p+1}_k \otimes P^{n-p}_k + \cdots + P^{n-1}_k \otimes P^2_k. 
 \end{split}
\end{equation*}
Besides, by Corollary~\ref{r:2022-05-23inclusionPn+1inPn}, we have $P^{q+1}_k\subset P^q_k$ for any $q\geq 2$. We thus get
\begin{equation}\label{2023-11-15copr}\Delta_{\mathring{\mathcal{A}}^\wedge_k}(P^{n+1}_k) \subset \mathring{\mathcal{A}}^\wedge_k\otimes P^{n+1-p}_k   \ + \  P^{p+1}_k\otimes \mathring{\mathcal{A}}^\wedge_k.
\end{equation}
The sum of these inclusions over $k \geq 0$ then implies
$$\Delta_{\mathcal{A}}(P^{n+1}) \subset P^{p+1}\otimes \mathcal{A} \ + \  \mathcal{A} \otimes P^{n-p+1}$$
which finishes the proof.
\end{proof}

\begin{proposition}\label{r:2022-06-15morphismDelta-n-p}
The coproduct of the algebra $ \bigoplus_{k\geq 0} \mathring{\mathcal A}^\wedge_k$ induces a morphism of algebras $$\Delta^{\mathfrak{b}}_{n,p}:\mathfrak{b}(n)\to \mathfrak{b}(p) \otimes \mathfrak{b}(n-p).$$  \index[notation]{\Delta^{\mathfrak{b}}_{n,p}@$\Delta^{\mathfrak{b}}_{n,p}$}
\end{proposition}

\begin{proof}
Set $J(n):=L^{<2n}+P^{n+1}+ (\mathrm{CO}\mathring{\mathcal A}^\wedge) \subset \bigoplus_{k\geq 0} \mathring{\mathcal A}^\wedge_k$, then $J(n)$ is an ideal of $\bigoplus_{k\geq 0} \mathring{\mathcal A}^\wedge_k$ and $\mathfrak{b}(n)$ is the corresponding quotient. The statement then follows from the inclusion $$\Delta_{\bigoplus_{k\geq 0} \mathring{\mathcal A}^\wedge_k}(J(n)) \subset J(p) \otimes \Big(\bigoplus_{k\geq 0} \mathring{\mathcal A}^\wedge_k\Big) \ + \ \Big(\bigoplus_{k\geq 0} \mathring{\mathcal A}^\wedge_k\Big) \otimes J(n-p).$$

\noindent This follows from the following inclusions 

\noindent $(a)$  $\Delta_{\bigoplus_{k\geq 0} \mathring{\mathcal A}^\wedge_k}(P^{n+1}) \subset P^{p+1} \otimes \Big(\bigoplus_{k\geq 0} \mathring{\mathcal A}^\wedge_k\Big) \  +  \ \Big(\bigoplus_{k\geq 0} \mathring{\mathcal A}^\wedge_k\Big) \otimes P^{n-p+1}$,

\noindent $(b)$ $\Delta_{\bigoplus_{k\geq 0} \mathring{\mathcal A}^\wedge_k}(L^{<2n}) \subset L^{<2p} \otimes \Big(\bigoplus_{k\geq 0} \mathring{\mathcal A}^\wedge_k\Big) \  + \ \Big(\bigoplus_{k\geq 0} \mathring{\mathcal A}^\wedge_k\Big) \otimes L^{<2(n-p)}$,

\noindent $(c)$ $\Delta_{\bigoplus_{k\geq 0} \mathring{\mathcal A}^\wedge_k}((\mathrm{CO}\mathring{\mathcal A}^\wedge)) \subset (\mathrm{CO}\mathring{\mathcal A}^\wedge)\otimes \Big(\bigoplus_{k\geq 0} \mathring{\mathcal A}^\wedge_k\Big) \  + \ \Big(\bigoplus_{k\geq 0} \mathring{\mathcal A}^\wedge_k\Big) \otimes (\mathrm{CO}\mathring{\mathcal A}^\wedge)$.

\medskip

\noindent The inclusion in $(a)$ is given by Proposition~\ref{r:2022-05-19coproductandP-n}. We now prove $(b)$ and $(c)$.

Recall from Definition~\ref{thespaceL2n} that $L^{<2n}=\bigoplus_{k\geq 0} L^{<2n}_k$ where $L^{<2n}_k$ is as in \eqref{thespaceL2n}. Define 
$$\mathring{\mathrm{Jac}}^{<2n}(\vec{\varnothing},[\![1,k]\!]):=\{\underline{D} \in \mathring{\mathrm{Jac}}(\vec{\varnothing},[\![1,k]\!])\ |\ \text{there exists } a \in [\![1,k]\!] \text{ with } \mathrm{m}_{a}(\underline{ D})<2n\},$$
where $\mathrm{m}_{a}: \mathring{\mathrm{Jac}}(\vec{\varnothing},[\![1,k]\!])\to \mathbb{Z}_{\geq 0}$ is the map from Definition~\ref{themapm}.
Let  $[\ ]:\mathbb{C}\mathring{\mathrm{Jac}}(\vec{\varnothing},[\![1,k]\!])\to \mathring{\mathcal{A}}^\wedge_k$ be the canonical projection. Then 
\begin{equation}\label{eq:equivdescripofL2n}
L^{<2n}_k= [\mathbb{C}\mathring{\mathrm{Jac}}^{<2n}(\vec{\varnothing},[\![1,k]\!])]^\wedge,
\end{equation} 
where $[\mathbb{C}\mathring{\mathrm{Jac}}^{<2n}(\vec{\varnothing},[\![1,k]\!])]^\wedge$ is the complete graded subspace of $\mathring{\mathcal{A}}^\wedge_k$ generated by  $[\mathbb{C}\mathring{\mathrm{Jac}}^{<2n}(\vec{\varnothing},[\![1,k]\!])]$.  
 
If $\underline{D} \in \mathring{\mathrm{Jac}}^{<2n}(\vec{\varnothing},[\![1,k]\!])$ and $A \subset \pi_0(\underline{D})$, there exists $a \in [\![1,k]\!]$ such that $\mathrm{m}_{a}(\underline{D})<2n$. Then  $\mathrm{m}_{a}(\underline{D}_A) + \mathrm{m}_{a}(\underline{D}_{\pi_0(\underline{D})\setminus A}) = \mathrm{m}_{a}(\underline{D})<2n$. 
Let now  $p \in [\![1,n-1]\!]$. Then either $\mathrm{m}_{a}(\underline{D}_A)<2p$ which implies $\underline{D}_A \otimes \underline{D}_{\pi_0(\underline{D})\setminus A} \in \mathbb{C}\mathring{\mathrm{Jac}}^{<2p}(\vec{\varnothing},[\![1,k]\!]) \otimes \mathbb{C}\mathring{\mathrm{Jac}}(\vec{\varnothing},[\![1,k]\!])$, or 
$\mathrm{m}_{a}(\underline{D}_A) \geq 2p$ which implies  $\mathrm{m}_{a}(\underline{D}_{\pi_0(\underline{D})\setminus A})<2(n-p)$. Therefore $\underline{D}_A \otimes \underline{D}_{\pi_0(\underline{D})\setminus A} \in \mathbb{C}\mathring{\mathrm{Jac}}(\vec{\varnothing},[\![1,k]\!])\otimes \mathbb{C}\mathring{\mathrm{Jac}}^{<2(n-p)}(\vec{\varnothing},[\![1,k]\!])$. It follows that 
\begin{equation}\label{eq:2022-06-13equ1}
\begin{split}
\Delta_{\mathbb{C}\mathring{\mathrm{Jac}}(\vec{\varnothing},[\![1,k]\!])}&(\mathring{\mathrm{Jac}}^{<2n}(\vec{\varnothing},[\![1,k]\!])) \\
& \subset \mathbb{C}\mathring{\mathrm{Jac}}^{<2p}(\vec{\varnothing},[\![1,k]\!]) \otimes \mathbb{C}\mathring{\mathrm{Jac}}(\vec{\varnothing},[\![1,k]\!]) \ + \ \mathbb{C}\mathring{\mathrm{Jac}}(\vec{\varnothing},[\![1,k]\!])\otimes \mathbb{C}\mathring{\mathrm{Jac}}^{<2(n-p)}(\vec{\varnothing},[\![1,k]\!]). 
\end{split}
\end{equation}
Then one has 
\begin{equation*}
\begin{split}
\Delta_{\mathring{\mathcal{A}}_k}(L^{<2n}_k) &= \Delta_{\mathring{\mathcal{A}}^\wedge_k}\big( [\mathbb{C}\mathring{\mathrm{Jac}}_{<2n}(\vec{\varnothing},[\![1,k]\!])]^\wedge\big)=([\ ]^\wedge \otimes [\ ]^\wedge) \circ \Delta_{\mathring{\mathrm{Jac}}(\vec{\varnothing},[\![1,k]\!])}\big(\mathbb{C}\mathring{\mathrm{Jac}}_{<2n}(\vec{\varnothing},[\![1,k]\!])\big)\\ 
& \subset ([\ ]^\wedge {\otimes} [\ ]^\wedge)\big(\mathbb{C}\mathring{\mathrm{Jac}}_{<2p}(\vec{\varnothing},[\![1,k]\!]) {\otimes} \mathbb{C}\mathring{\mathrm{Jac}}(\vec{\varnothing},[\![1,k]\!]) \ + \  \mathbb{C}\mathring{\mathrm{Jac}}(\vec{\varnothing},[\![1,k]\!]){\otimes} \mathbb{C}\mathring{\mathrm{Jac}}_{<2(n-p)}(\vec{\varnothing},[\![1,k]\!])\big) \\
&=L^{<2p}_k \otimes \mathring{\mathcal{A}}_k \ +  \ \mathring{\mathcal{A}}_k \otimes L^{<2(n-p)}_k,
\end{split}
\end{equation*}
where the first and last equality follows from~\eqref{eq:equivdescripofL2n}, the second equality follows from the fact that $[\ ] : \mathbb{C}\mathring{\mathrm{Jac}}(\vec{\varnothing},[\![1,k]\!])\to \mathring{\mathcal{A}}_k$ is a coalgebra morphism and  the inclusion follows from~\eqref{eq:2022-06-13equ1}.  Therefore, we have 
$$\Delta_{\mathring{\mathcal{A}}_k}(L^{<2n}_k) \subset L^{<2p}_k \otimes \mathring{\mathcal{A}}_k + \mathring{\mathcal{A}}_k \otimes L^{<2(n-p)}_k$$
for any $k \geq 0$, which by taking the direct sum over $k\geq 0$ implies the inclusion in $(b)$.

We proceed to show $(c)$. Recall from \S\ref{sec:3:2:1} that $(\mathrm{CO}\mathring{\mathcal{A}}^\wedge)= \bigoplus_{k\geq 1} (\mathrm{CO}\mathring{\mathcal{A}}_k^\wedge) \subset \bigoplus_{k\geq 0}\mathring{\mathcal{A}}^\wedge_k$ and $(\mathrm{CO}\mathring{\mathcal{A}}_k^\wedge)$ is topologically spanned by elements of the form $\mathrm{co}_k^{{\mathcal A}}([\underline{D}]) - \mathrm{proj}_{1,k-1}^{{\mathcal{A}}}([\underline{D}])$ where $$\underline{D}=\big(D,\varphi:\partial D\to [\![1,k]\!], \{\mathrm{cyc}_i\}_{i \in [\![1,k]\!]}\big)\in \mathring{\mathrm{Jac}}(\vec{\varnothing},[\![1,k]\!]).$$ By definition $\mathrm{co}_k^{{\mathcal A}}([\underline{D}])=(-1)^{|\varphi^{-1}(1)|}[(D,\varphi, \{\mathrm{cyc}_i\}_{i\in [\![2,k]\!]}\sqcup \{\overline{\mathrm{cyc}}_{1}\} )]$, where $\overline{\mathrm{cyc}}_{1}$ denotes the opposite cyclic order of ${\mathrm{cyc}}_{1}$ on the set $\varphi^{-1}(1)\subset \partial D$. Set 
\begin{equation}\label{2023-11-17diagraDprime}
\underline{D}':=(D,\varphi, \{\mathrm{cyc}_i\}_{i\in [\![2,k]\!]}\sqcup \{\overline{\mathrm{cyc}}_{1}\}),
\end{equation}
clearly $\pi_0(\underline{D})=\pi_0(\underline{D}')$.

For each $A\subset \pi_0(\underline{D})$ we have
\begin{equation*}
|\varphi^{-1}(1)| = \big|(\varphi_{|(\partial D)\cap \bigsqcup_{x\in A}x})^{-1}(1) \big|\  + \ \big|(\varphi_{|(\partial D)\cap \bigsqcup_{x\in \pi_0(\underline{D})\setminus A}x})^{-1}(1)\big|. 
\end{equation*}
Consequently,
\begin{equation*}
\begin{split}
\Delta_{\mathring{\mathrm{Jac}}(\vec{\varnothing},[\![1,k]\!])}\big(&(-1)^{|\varphi^{-1}(1)|}\underline{D}'\big)  \\
& = \sum_{A\subset \pi_0(\underline{D}')} \Big((-1)^{\big|(\varphi_{|(\partial D')\cap \bigsqcup_{x\in A}x})^{-1}(1) \big|}\underline{D}'_A \Big) \otimes \Big((-1)^{\big|(\varphi_{|(\partial D')\cap \bigsqcup_{x\in \pi_0(\underline{D}')\setminus A}x})^{-1}(1)\big|}\underline{D}'_{\pi_0(\underline{D}')\setminus A}\Big).
\end{split}
\end{equation*}
where $\underline{D}'_A$ and $\underline{D}'_{\pi_0(\underline{D}')\setminus A}$ are as in 
Definition~\ref{def:coproductinA-k}.

According to the above and to the fact that $[\ ]:\mathbb{C}\mathring{\mathrm{Jac}}(\vec{\varnothing},[\![1,k]\!])\to \mathring{\mathcal{A}}^\wedge_k$ is a coalgebra morphism we then obtain 
$$\Delta_{\mathring{\mathcal{A}}^\wedge_k}\big(\mathrm{co}^{{{\mathcal{A}}}}_{k}([\underline{D}])\big) = \sum_{A\subset \pi_0(\underline{D})} \Big(\mathrm{co}^{{{\mathcal{A}}}}_{k}([\underline{D}_A]) \Big) \otimes \Big(\mathrm{co}^{{{\mathcal{A}}}}_{k}([\underline{D}_{\pi_0(\underline{D})\setminus A}]) \Big).$$
Therefore,

\begin{equation*}
\begin{split}
\Delta_{\mathring{\mathcal{A}}^\wedge_k} \Big(\mathrm{proj}_{1,k-1}^{{\mathcal{A}}}([\underline{D}])-  \mathrm{co}^{{{\mathcal{A}}}}_{k}([\underline{D}])\Big) = \sum_{A\subset \pi_0(\underline{D})} [\underline{D}_A]\otimes [\underline{D}_{\pi_0(\underline{D})\setminus A}] -  \mathrm{co}^{{{\mathcal{A}}}}_{k}([\underline{D}_A]) \otimes [\underline{D}_{\pi_0(\underline{D})\setminus A}] \qquad \qquad\\
 \qquad  \qquad  \qquad\ \ \   + \mathrm{co}^{{{\mathcal{A}}}}_{k}([\underline{D}_A]) \otimes [\underline{D}_{\pi_0(\underline{D})\setminus A}] -   \mathrm{co}^{{{\mathcal{A}}}}_{k}([\underline{D}_A]) \otimes  \mathrm{co}^{{{\mathcal{A}}}}_{k}([\underline{D}_{\pi_0(\underline{D})\setminus A}])\\
=  \sum_{A\subset \pi_0(\underline{D})} \Big([\underline{D}_A] - \mathrm{co}^{{{\mathcal{A}}}}_{k}([\underline{D}_A])\Big)\otimes [\underline{D}_{\pi_0(\underline{D})\setminus A}]  + \sum_{A\subset \pi_0(\underline{D})} \mathrm{co}^{{{\mathcal{A}}}}_{k}([\underline{D}_A]) \otimes \Big([\underline{D}_{\pi_0(\underline{D})\setminus A}] -    \mathrm{co}^{{{\mathcal{A}}}}_{k}([\underline{D}_{\pi_0(\underline{D})\setminus A}]\Big)\\
\in (\mathrm{CO}\mathring{\mathcal{A}}_k^\wedge)\otimes \mathring{\mathcal{A}}^\wedge_k \ + \ \mathring{\mathcal{A}}^\wedge_k\otimes (\mathrm{CO}\mathring{\mathcal{A}}_k^\wedge).
\end{split}
\end{equation*}
Hence,
$$ \Delta_{\mathring{\mathcal{A}}^\wedge_k}\big( (\mathrm{CO}\mathring{\mathcal{A}}_k^\wedge)\big)\subset (\mathrm{CO}\mathring{\mathcal{A}}_k^\wedge)\otimes \mathring{\mathcal{A}}^\wedge_k \ + \ \mathring{\mathcal{A}}^\wedge_k\otimes (\mathrm{CO}\mathring{\mathcal{A}}_k^\wedge)$$
for any $k\geq 1$, which by taking the direct sum over $k\geq 1$ yields the inclusion in $(c)$.

\end{proof}

\subsubsection{The algebra homomorphism \texorpdfstring{$\Delta^{\mathfrak{b}}_{n,p}:\mathfrak{b}(n)\to \mathfrak{b}(p) \otimes \mathfrak{b}(n-p)$}{bn-to-bp-otimes-bn-p} as a morphism of semi-Kirby structures}
 
\begin{proposition}\label{r:2022-06-15commdiag1}  The map $\Delta^{\mathfrak{b}}_{n,p}:\mathfrak{b}(n)\to \mathfrak{b}(p) \otimes \mathfrak{b}(n-p)$ from Proposition~\ref{r:2022-06-15morphismDelta-n-p} is a morphism of semi-Kirby structures.
\end{proposition}

\begin{proof}
It follows from the fact that the map ${Z}_{\mathcal A}$ sends each element of $\vec{\mathcal T}(\emptyset, \emptyset)_k$ to a group-like element of $\mathring{\mathcal A}^\wedge_k$, see Lemma~\ref{r:imageofhatZcontainedinGL}; from Proposition~\ref{2023-11-15prop}, and from the compatibility of $\Delta_{n,p}^{\mathfrak b}$ with $\Delta_{\oplus_{k \geq 0}\mathring{\mathcal A}^\wedge_k}$ (see Proposition~\ref{r:2022-06-15morphismDelta-n-p})  that we have a commutative the diagram 
\begin{equation}\label{eq:20220506-diag1}
\xymatrix{\mathfrak{Kir}\ar[rrrr]^-{\varphi_n\circ{\mathrm{cs}}^\nu\circ{Z}_{\mathcal A}}\ar[d]_{\Delta_{\mathfrak{Kir}}} & & & & \mathfrak{b}(n) \ar[d]^{\Delta^{\mathfrak{b}}_{n,p}} \\ 
\mathfrak{Kir}\otimes \mathfrak{Kir} \ar[rrrr]^-{(\varphi_p\circ{\mathrm{cs}}^\nu\circ{Z}_{\mathcal A})\otimes (\varphi_{n-p}\circ{\mathrm{cs}}^\nu\circ {Z}_{\mathcal A})} & & & &  \mathfrak{b}(p)\otimes \mathfrak{b}(n-p).}
\end{equation}

The facts that
\begin{itemize}
\item[$(a)$] the semi-Kirby structure on
on the algebra $\mathfrak{b}(p) \otimes \mathfrak{b}(n-p)$ is given by the composed algebra homomorphism $$\mathfrak{Kir} \xrightarrow{\ \Delta_{\mathfrak{Kir}} \ } \mathfrak{Kir} \otimes \mathfrak{Kir} \xrightarrow{ \ {(\varphi_p\circ{\mathrm{cs}}^\nu\circ {Z}_{\mathcal A})\otimes (\varphi_{n-p}\circ {\mathrm{cs}}^\nu\circ {Z}_{\mathcal A})}\ } \mathfrak{b}(p) \otimes \mathfrak{b}(n-p)$$
which follows from Lemma~\ref{tensorofSKS},
\item[$(b)$] $\Delta_{n,p}^{\mathfrak{b}}$ is an algebra homomorphism (Proposition~\ref{r:2022-06-15morphismDelta-n-p}),
\item[$(c)$] the diagram \eqref{eq:20220506-diag1} is commutative, 
\end{itemize}
imply that $\Delta_{n,p}^{\mathfrak{b}}$ is a morphism of semi-Kirby structures.
\end{proof}

\subsection{A Kirby structure morphism \texorpdfstring{$\Delta^{\mathfrak{c}}_{n,p}:\mathfrak{c}(n)\to \mathfrak{c}(p) \otimes \mathfrak{c}(n-p)$}{cn-to-cp-otimes-cn-p}} \label{sec:5:3}

\subsubsection{An algebra homomorphism \texorpdfstring{$\Delta^{\mathfrak{c}}_{n,p}:\mathfrak{c}(n)\to \mathfrak{c}(p) \otimes \mathfrak{c}(n-p)$}{cn-to-cp-otimes-cn-p}}

Recall that $\mathfrak{c}(n)= \mathring{\mathcal{A}}^\wedge_0/(P^{n+1}_0 + O^n)$, see Definition~\ref{def:b-n}. Specializing the map \eqref{eq:coproductinJac} to $k=0$ and restricting the map \eqref{eq:coproductin-oplus-A-l} to $k=0$, we obtain maps

\begin{equation}\label{eq:coproductinJacempty}
\Delta_{\mathring{\mathrm{Jac}}(\vec{\varnothing},\emptyset)}:\mathbb{C}\mathring{\mathrm{Jac}}(\vec{\varnothing},\emptyset)\longrightarrow \mathbb{C}\mathring{\mathrm{Jac}}(\vec{\varnothing},\emptyset)\otimes \mathbb{C}\mathring{\mathrm{Jac}}(\vec{\varnothing},\emptyset)
\end{equation} \index[notation]{\Delta_{\mathring{\mathrm{Jac}}(\vec{\varnothing},\emptyset)}@$\Delta_{\mathring{\mathrm{Jac}}(\vec{\varnothing},\emptyset)}$}

and
\begin{equation}\label{eq:coproductinAempty}
\Delta_{\mathring{\mathcal{A}}^\wedge_0}:\mathring{\mathcal{A}}^\wedge_0\longrightarrow \mathring{\mathcal{A}}^\wedge_0\otimes \mathring{\mathcal{A}}^\wedge_0.
\end{equation}  \index[notation]{\Delta_{\mathring{\mathcal{A}}^\wedge_0}@$\Delta_{\mathring{\mathcal{A}}^\wedge_0}$}
These maps endow $\mathbb{C}\mathring{\mathrm{Jac}}(\vec{\varnothing},\emptyset)$ and $\mathring{\mathcal{A}}^\wedge_0$ with cocommutative coassociative coalgebra structures such that the projection $[\ ]:\mathbb{C}\mathring{\mathrm{Jac}}(\vec{\varnothing},\emptyset) \to  \mathring{\mathcal{A}}^\wedge_0$ is a coalgebra morphism. The triple $(\mathring{\mathcal A}^\wedge_0, \otimes, \Delta_{\mathring{\mathcal A}^\wedge_0})$ is a bialgebra as it is both a subalgebra and a subcoalgebra of the  bialgebra from Proposition~\ref{r:2022-06-15Deltaisalgebramorpshim}. Besides, equation \eqref{2023-11-15copr} (restricted to $k=0$) implies that
\begin{equation}\label{eq:2022-06-15compatibilitycoproductandPnempty}
\Delta_{\mathring{\mathcal{A}}^\wedge_0}(P^{n+1}_0) \subset P^{p+1}_0\otimes \mathring{\mathcal{A}}^\wedge_0 \ + \ \mathring{\mathcal{A}}^\wedge_0\otimes P^{n-p+1}_0
\end{equation}
for any $1\leq p\leq n$.

\begin{proposition}\label{r:2022-06-15inducedalgebramorphismonmathfrakc}
The coproduct \eqref{eq:coproductinAempty} of the algebra $\mathring{\mathcal{A}}^\wedge_0$ induces a morphism of algebras 
$$\Delta^{\mathfrak{c}}_{n,p}:\mathfrak{c}(n)\longrightarrow \mathfrak{c}(p)\otimes \mathfrak{c}(n-p).$$  \index[notation]{\Delta^{\mathfrak{c}}_{n,p}@$\Delta^{\mathfrak{c}}_{n,p}$}
\end{proposition}
\begin{proof}
The result follows from the following inclusions

\noindent $(a)$ $\Delta_{\mathring{\mathcal{A}}^\wedge0}(P^{n+1}_0) \subset P^{p+1}_0\otimes \mathring{\mathcal{A}}^\wedge_0 \ + \ \mathring{\mathcal{A}}^\wedge_0\otimes P^{n-p+1}_0$,

\noindent $(b)$ $\Delta_{\mathring{\mathcal{A}}^\wedge_0}(O^n) \subset O^{p}\otimes \mathring{\mathcal{A}}^\wedge_0 \ + \ \mathring{\mathcal{A}}^\wedge_0\otimes O^{n-p}$.

\medskip

\noindent The inclusion in $(a)$ is given by \eqref{eq:2022-06-15compatibilitycoproductandPnempty}. Let us show the inclusion in $(b)$. Recall that $O^n$ is the closed ideal of $\mathring{\mathcal{A}}^\wedge_0$ generated by $X+2n$ where $X$ denotes the dashed loop, see Definition~\ref{def:ideal-O-n}. We have
$$\Delta_{\mathring{\mathcal{A}}^\wedge_0}(X+2n) = X\otimes\emptyset + \emptyset\otimes X + 2n = (X+2p)\otimes \emptyset + \emptyset\otimes (X+2(n-p))\in O^p\otimes \mathring{\mathcal{A}}^\wedge_0 + \mathring{\mathcal{A}}^\wedge_0\otimes O^{n-p},$$
which implies the inclusion in $(b)$.
\end{proof}

\subsubsection{The algebra homomorphism \texorpdfstring{$\mathfrak{c}(n)\to \mathfrak{c}(p) \otimes \mathfrak{c}(n-p)$}{cn-to-cp-otimes-cn-p} as a morphism of Kirby structures} 

By Theorem~\ref{mainr-2022-02-14}, for every integer $n\geq 1$, the pair $(\mathfrak{c}(n),\overline{j}_{n}\circ \varphi_n \circ {\mathrm{cs}}^\nu\circ{Z}_{\mathcal A})$ is a Kirby structure (see Proposition~\ref{r2022-04-08-jn}). By Lemma~\ref{sec:1.8lemma1}, for all integers $n,p\geq 1$ with $n-p\geq 1$, the algebra $\mathfrak{c}(p) \otimes \mathfrak{c}(n-p)$ is endowed canonically  with the semi-Kirby structure  given by the composed algebra homomorphism $$\mathfrak{Kir} \xrightarrow{\ \Delta_{\mathfrak{Kir}} \ } \mathfrak{Kir} \otimes \mathfrak{Kir} \xrightarrow{ \ {(\overline{j}_p\circ\varphi_n\circ {\mathrm{cs}}^\nu\circ{Z}_{\mathcal A})\otimes (\overline{j}_{n-p}\circ\varphi_{n-p}\circ{\mathrm{cs}}^\nu\circ{Z}_{\mathcal A})}\ } \mathfrak{c}(p) \otimes \mathfrak{c}(n-p)$$
and therefore $\mathfrak{c}(p) \otimes \mathfrak{c}(n-p)$ is endowed with a Kirby structure.

\begin{proposition}\label{r:2024-01-23-1} The   algebra morphism $\Delta^{\mathfrak{c}}_{n,p}:\mathfrak{c}(n)\to \mathfrak{c}(p) \otimes \mathfrak{c}(n-p)$ from Proposition~\ref{r:2022-06-15inducedalgebramorphismonmathfrakc} defines  a morphism of Kirby structures
$$\Delta^{\mathfrak{c}}_{n,p}:\big(\mathfrak{c}(n),\overline{j}_{n}\circ \varphi_n \circ {\mathrm{cs}}^\nu\circ{Z}_{\mathcal A}\big)\longrightarrow \big(\mathfrak{c}(p),\overline{j}_{p}\circ \varphi_p \circ {\mathrm{cs}}^\nu\circ{Z}_{\mathcal A}\big)\otimes \big(\mathfrak{c}(n-p),\overline{j}_{n-p}\circ \varphi_{n-p} \circ {\mathrm{cs}}^\nu\circ{Z}_{\mathcal A}\big).$$
\end{proposition}
\begin{proof}
 We need to show that the diagram
\begin{equation}\label{eq:2022-06-15-diag1forc}
\xymatrix{\mathfrak{Kir}\ar[rrrrr]^-{\overline{j}_n\circ\varphi_n\circ {\mathrm{cs}}^\nu\circ {Z}_{\mathcal A}}\ar[d]_{\Delta_{\mathfrak{Kir}}} & & & & & \mathfrak{c}(n) \ar[d]^{\Delta^{\mathfrak{c}}_{n,p}} \\ 
\mathfrak{Kir}\otimes \mathfrak{Kir} \ar[rrrrr]^-{(\overline{j}_p\circ\varphi_p\circ {\mathrm{cs}}^\nu\circ {Z}_{\mathcal A})\otimes (\overline{j}_{n-p}\circ\varphi_{n-p}\circ {\mathrm{cs}}^\nu \circ {Z}_{\mathcal A})} & & & & &  \mathfrak{c}(p)\otimes \mathfrak{c}(n-p)}
\end{equation}
commutes. This diagram decomposes as
\begin{equation*}\label{eq:2022-06-15-diag2}
\xymatrix{\mathfrak{Kir}\ar[rrrrr]^-{\varphi_n\circ{\mathrm{cs}}^\nu\circ{Z}_{\mathcal A}}\ar[d]_{\Delta_{\mathfrak{Kir}}} & & & & & \mathfrak{b}(n)\ar[rr]^-{\overline{j}_n} \ar[d]^{\Delta^{\mathfrak{b}}_{n,p}} & & \mathfrak{c}(n) \ar[d]^{\Delta^{\mathfrak{c}}_{n,p}} \\ 
\mathfrak{Kir}\otimes \mathfrak{Kir} \ar[rrrrr]^-{(\varphi_p\circ {\mathrm{cs}}^\nu\circ {Z}_{\mathcal A})\otimes (\varphi_{n-p}\circ {\mathrm{cs}}^\nu\circ {Z}_{\mathcal A})} & & & & &  \mathfrak{b}(p)\otimes \mathfrak{b}(n-p) \ar[rr]^-{\overline{j}_p\otimes \overline{j}_{n-p}} & & \mathfrak{c}(p)\otimes \mathfrak{c}(n-p).}
\end{equation*}
The square on the left is commutative by Proposition~\ref{r:2022-06-15commdiag1}. The commutativity of  the square on the right is exactly \cite[Lemma 4.2]{LMO98}, which finishes the proof.
\end{proof}

\subsection{A Kirby structure morphism \texorpdfstring{$\Delta^{\mathfrak{d}}_{n,p}: \mathfrak{d}(n)\to \mathfrak{d}(p) \otimes \mathfrak{d}(n-p)$}{dn-to-dp-otimes-dn-p}}\label{sec:5:4}

\subsubsection{An algebra homomorphism \texorpdfstring{$\Delta^{\mathfrak{d}}_{n,p}:\mathfrak{d}(n)\to \mathfrak{d}(p) \otimes \mathfrak{d}(n-p)$}{dn-to-dp-otimes-dn-p}}

Recall that $\mathfrak{d}(n)= \mathring{\mathcal{A}}^\wedge_0/(P^{n+1}_0 + O^n + (\mathring{\mathrm{deg}}>n))$, see Definition~\ref{def:alg-dn}.

\begin{proposition}\label{r:2022-06-15inducedalgebramorphismonmathfrakd}
The coproduct \eqref{eq:coproductinAempty} of the algebra $\mathring{\mathcal{A}}^\wedge_0$ induces a morphism of algebras 
$$\Delta^{\mathfrak{d}}_{n,p}:\mathfrak{d}(n)\longrightarrow \mathfrak{d}(p)\otimes \mathfrak{d}(n-p).$$ \index[notation]{\Delta^{\mathfrak{d}}_{n,p}@$\Delta^{\mathfrak{d}}_{n,p}$}
\end{proposition}
\begin{proof}
The result follows from the following inclusions

\noindent $(a)$ $\Delta_{\mathring{\mathcal{A}}^\wedge_0}(P^{n+1}_0) \subset P^{p+1}_0\otimes \mathring{\mathcal{A}}^\wedge_0 \ + \ \mathring{\mathcal{A}}^\wedge_0\otimes P^{n-p+1}_0$,

\noindent $(b)$ $\Delta_{\mathring{\mathcal{A}}^\wedge_0}(O^n) \subset O^{p}\otimes \mathring{\mathcal{A}}^\wedge_0 \ + \ \mathring{\mathcal{A}}^\wedge_0\otimes O^{n-p}$,

\noindent $(c)$ {$\Delta_{\mathring{\mathcal{A}}^\wedge_0}(\mathring{\mathcal{A}}^\wedge_0[>n]) \subset \mathring{\mathcal{A}}^\wedge_0[>p] \otimes \mathring{\mathcal{A}}^\wedge_0 \ + \ \mathring{\mathcal{A}}^\wedge_0\otimes \mathring{\mathcal{A}}^\wedge_0[>n-p]$.}

\medskip

\noindent The inclusions in $(a)$  and $(b)$ where already proved in the proof of Proposition~\ref{r:2022-06-15inducedalgebramorphismonmathfrakc}.  For each $A\subset\pi_0(\underline{D})$ we have $\mathrm{deg}(\underline{D}) = \mathrm{deg}(\underline{D}_A) + \mathrm{deg}(\underline{D}_{\pi_0(\underline{D})\setminus A})$, which implies that $\Delta_{\mathring{\mathcal{A}}^\wedge_0}$ is graded for the degree on $\mathring{\mathcal{A}}^\wedge_0$, this implies $(c)$. 

\end{proof}

\subsubsection{The algebra homomorphism \texorpdfstring{$\mathfrak{d}(n)\to \mathfrak{d}(p) \otimes \mathfrak{d}(n-p)$}{dn-to-dp-otimes-dn-p} as a morphism of Kirby structures} 

By Proposition~\ref{r2022-04-12-dn}, for every integer $n\geq 1$, the pair $(\mathfrak{d}(n),\mathrm{pr}_n\circ\overline{j}_{n}\circ \varphi_n \circ \overline{\mathrm{cs}}^\nu\circ{Z}_{\mathcal A})$ is a Kirby structure. By Lemma~\ref{sec:1.8lemma1}, for all integers $n,p\geq 1$ with $n-p\geq 1$, the algebra $\mathfrak{d}(p) \otimes \mathfrak{d}(n-p)$ is endowed canonically  with the semi-Kirby structure  given by the composed algebra homomorphism $$\mathfrak{Kir} \xrightarrow{\ \Delta_{\mathfrak{Kir}} \ } \mathfrak{Kir} \otimes \mathfrak{Kir} \xrightarrow{ \ {(\mathrm{pr}_p\circ\overline{j}_p\circ\varphi_n\circ\overline{\mathrm{cs}}^\nu\circ{Z}_{\mathcal A})\otimes (\mathrm{pr}_{n-p}\overline{j}_{n-p}\circ\varphi_{n-p}\circ\overline{\mathrm{cs}}^\nu\circ{Z}_{\mathcal A})}\ } \mathfrak{d}(p) \otimes \mathfrak{d}(n-p)$$
and therefore $\mathfrak{d}(p) \otimes \mathfrak{d}(n-p)$ is endowed with a Kirby structure.

\begin{proposition} The  algebra morphism $\Delta^{\mathfrak{d}}_{n,p}:\mathfrak{d}(n)\to \mathfrak{d}(p) \otimes \mathfrak{d}(n-p)$ from Proposition~\ref{r:2022-06-15inducedalgebramorphismonmathfrakd} defines a morphism of Kirby structures
$$\Delta^{\mathfrak{d}}_{n,p}:\big(\mathfrak{d}(n),\overline{j}_{n}\circ \varphi_n \circ {\mathrm{cs}}^\nu\circ{Z}_{\mathcal A}\big)\longrightarrow \big(\mathfrak{d}(p),\overline{j}_{p}\circ \varphi_p \circ {\mathrm{cs}}^\nu\circ{Z}_{\mathcal A}\big)\otimes \big(\mathfrak{d}(n-p),\overline{j}_{n-p}\circ \varphi_{n-p} \circ {\mathrm{cs}}^\nu\circ{Z}_{\mathcal A}\big).$$
\end{proposition}
\begin{proof}
 We need to show that the diagram
\begin{equation}\label{eq:2022-06-15-diag1ford}
\xymatrix{\mathfrak{Kir}\ar[rrrrrrr]^-{\mathrm{pr}_n\circ\overline{j}_n\circ\varphi_n\circ\overline{\mathrm{cs}}^\nu\circ{Z}_{\mathcal A}}\ar[d]_{\Delta_{\mathfrak{Kir}}} & & && & & & \mathfrak{d}(n) \ar[d]^{\Delta^{\mathfrak{d}}_{n,p}} \\ 
\mathfrak{Kir}\otimes \mathfrak{Kir} \ar[rrrrrrr]^-{(\mathrm{pr}_p\circ\overline{j}_p\circ\varphi_p\circ\overline{\mathrm{cs}}^\nu\circ{Z}_{\mathcal A})\otimes (\mathrm{pr}_{n-p}\circ\overline{j}_{n-p}\circ\varphi_{n-p}\circ\overline{\mathrm{cs}}^\nu\circ{Z}_{\mathcal A})} & & & & & & &  \mathfrak{d}(p)\otimes \mathfrak{d}(n-p)}
\end{equation}
commutes. This diagram decomposes as
\begin{equation*}\label{eq:2022-06-15-diag2ford}
\xymatrix{\mathfrak{Kir}\ar[rrrrr]^-{\overline{j}_n\circ\varphi_n\circ\overline{\mathrm{cs}}^\nu\circ{Z}_{\mathcal A}}\ar[d]_{\Delta_{\mathfrak{Kir}}} & & & & & \mathfrak{c}(n)\ar[rr]^-{\mathrm{pr}_n} \ar[d]^{\Delta^{\mathfrak{c}}_{n,p}} & & \mathfrak{d}(n) \ar[d]^{\Delta^{\mathfrak{d}}_{n,p}} \\ 
\mathfrak{Kir}\otimes \mathfrak{Kir} \ar[rrrrr]^-{(\overline{j}_p\circ\varphi_p\circ\overline{\mathrm{cs}}^\nu\circ{Z}_{\mathcal A})\otimes (\overline{j}_{n-p}\circ\varphi_{n-p}\circ\overline{\mathrm{cs}}^\nu\circ{Z}_{\mathcal A})} & & & & &  \mathfrak{c}(p)\otimes \mathfrak{c}(n-p) \ar[rr]^-{\mathrm{pr}_p\otimes \mathrm{pr}_{n-p}} & & \mathfrak{d}(p)\otimes \mathfrak{d}(n-p).}
\end{equation*}
The square on the left is the commutative diagram \eqref{eq:2022-06-15-diag1forc}. The square on the right inserts in the diagram
\begin{equation}\label{eq:2023-11-22-1}
\xymatrix{ \mathring{\mathcal{A}}^\wedge_0 \ar[rd]^-{}\ar[ddd]_-{\Delta_{\mathring{\mathcal{A}}^\wedge_0}}\ar[rrrd]^-{}&  &   & \\
 &  \mathfrak{c}(n)\ar[rr]_-{\mathrm{pr}_n} \ar[d]^{\Delta^{\mathfrak{c}}_{n,p}}  &  &\mathfrak{d}(n) \ar[d]^{\Delta^{\mathfrak{d}}_{n,p}}\\
 & \mathfrak{c}(p)\otimes \mathfrak{c}(n-p) \ar[rr]^-{\mathrm{pr}_p\otimes \mathrm{pr}_{n-p}} &  &\mathfrak{d}(p)\otimes \mathfrak{d}(n-p) \\
\mathring{\mathcal{A}}^\wedge_0\otimes \mathring{\mathcal{A}}^\wedge_0 \ar[ru]^-{}\ar[rrru]^-{}&  &  &}
\end{equation}
where the diagonal maps are all canonical projections. In this diagram the two triangles, left square and exterior square are commutative by definition. This, together with the surjectivity of the projection $\mathring{\mathcal{A}}^\wedge_0 \to \mathfrak{c}(n)$, imply the commutativity of the inner square, which finishes the proof.
\end{proof}

\begin{proposition}\label{r:2023-12-06r2}  For any $Y \in 3\text{-}\mathrm{Mfds}$ and  integers  $n,p\geq 1$ such that $n-p \geq 1$, one has:
 
$(a)$ $\Omega^{\mathfrak{c}}_p(Y) \otimes \Omega^{\mathfrak{c}}_{n-p}(Y)=\Delta^{\mathfrak{c}}_{n,p}(\Omega^{\mathfrak{c}}_n(Y))$ (equality in $\mathfrak{c}(p) \otimes \mathfrak{c}(n-p)$).

$(b)$  $\Omega^{\mathfrak{d}}_p(Y) \otimes \Omega^{\mathfrak{d}}_{n-p}(Y)=\Delta^{\mathfrak{d}}_{n,p}(\Omega^{\mathfrak{e}}_n(Y))$ (equality in $\mathfrak{d}(p) \otimes \mathfrak{d}(n-p)$).

\end{proposition}

\begin{proof}

$(a)$ (resp. $(b)$) follows by applying the functor $\mathcal{KS}\to \mathcal{I}nv$ from Proposition~\ref{cor:invt:3-vars:alg} to the morphism of Kirby structures  $\Delta^{\mathfrak{c}}_{n,p}$ (resp. $\Delta^{\mathfrak{d}}_{n,p}$).
\end{proof}

\subsection{The Kirby structure \texorpdfstring{$\mathfrak{e}(n)$}{e-n} and the invariant  \texorpdfstring{$\Omega^{\mathfrak{e}}_{n}$}{Omega-e-n}}\label{nsec:5:6}

\begin{definition}  Define the set $\mathrm{Jac}(\vec{\varnothing}, \emptyset)$ \index[notation]{Jac(\vec{\varnothing}, \emptyset)@$\mathrm{Jac}(\vec{\varnothing}, \emptyset)$} and the commutative complete graded algebra ${\mathcal{A}}^\wedge_0$  \index[notation]{A^\wedge_02@${\mathcal{A}}^\wedge_0$} similarly to  $\mathring{\mathrm{Jac}}(\vec{\varnothing}, \emptyset)$ and $\mathring{\mathcal{A}}^\wedge_0$ but without allowing the dashed loop. Let ${\mathcal{A}}^\wedge_0[n]$  \index[notation]{A^\wedge_0[n]noring@${\mathcal{A}}^\wedge_0[n]$} denote the subspace of ${\mathcal{A}}^\wedge_0$ consisting of elements of degree $n$ and  let ${\mathcal{A}}^\wedge_0[>n]$  \index[notation]{A^\wedge_0[>n]noring@${\mathcal{A}}^\wedge_0[>n]$} denote the complete direct sum of all the ${\mathcal{A}}^\wedge_0[k]$ with $k>n$. Since ${\mathcal{A}}^\wedge_0$ is a complete graded algebra, this is a complete graded ideal  of ${\mathcal{A}}^\wedge_0$. 
\end{definition}

\begin{definition} For integers $k,l\geq 0$, denote by $\mathring{\mathcal A}^{\wedge}_0[k]\{l\}$  \index[notation]{A^{\wedge}_0[k]\{l\}@$\mathring{\mathcal A}^{\wedge}_0[k]\{l\}$} the subspace of  $\mathring{\mathcal A}^{\wedge}_0[k]$ spanned by classes of Jacobi diagrams with $l$ dashed loops. This induces a bigrading on $\mathring{\mathcal A}^{\wedge}_0$.
\end{definition} 

\begin{lemma} Let ${\mathcal A}^{\wedge}_0[X]$ be the polynomial algebra on $X$ equipped with the bigrading  such that for integers $k,l\geq 0$ the bidegree $(k,l)$ part is given by  $X^l{\mathcal A}^{\wedge}_0[k]$. Then there is a bigraded  algebra isomorphism ${\mathcal A}^{\wedge}_0[X]\xrightarrow{ \cong } \mathring{\mathcal A}^{\wedge}_0$ such that it
restricts to the inclusion ${\mathcal A}^{\wedge}_0\subset  \mathring{\mathcal A}^{\wedge}_0$ and takes $X$ to the dashed loop. 
\end{lemma}
\begin{proof}
The result is a direct verification.
\end{proof}

\begin{definition}\label{def:definitionofen} Define the complete graded commutative algebra $\mathfrak{e}(n)$ as the quotient algebra 
\begin{equation}
\mathfrak{e}(n)=\frac{{\mathcal{A}}^\wedge_0}{{\mathcal{A}}^\wedge_0[>n]}.
\end{equation} \index[notation]{e(n)@$\mathfrak{e}(n)$}
\end{definition}

Similarly to $\mathfrak{d}(n)$, the algebra $\mathfrak{e}(n)$ is the direct sum of its components {$\mathfrak{e}(n)[k]$ of degrees~$0\leq k\leq n$}, and is therefore a graded commutative algebra.

One checks that the coproduct \eqref{eq:coproductinAempty} induces a coproduct
\begin{equation}\label{eq:inAnocircle}
\Delta_{{\mathcal{A}}^\wedge_0}:{\mathcal{A}}^\wedge_0\longrightarrow {\mathcal{A}}^\wedge_0\otimes {\mathcal{A}}^\wedge_0
\end{equation} \index[notation]{\Delta_{{\mathcal{A}}^\wedge_0}@$\Delta_{{\mathcal{A}}^\wedge_0}$} on the algebra ${\mathcal{A}}^\wedge_0$ and that $\Delta_{{\mathcal{A}}^\wedge_0}$ is an algebra homomorphism.

\begin{proposition}\label{r:2022-06-15inducedalgebramorphismonmathfrake}
The coproduct \eqref{eq:inAnocircle} of the algebra ${\mathcal{A}}^\wedge_0$ induces a morphism of algebras 
$$\Delta^{\mathfrak{e}}_{n,p}:\mathfrak{e}(n)\longrightarrow \mathfrak{e}(p)\otimes \mathfrak{e}(n-p).$$  \index[notation]{\Delta^{\mathfrak{e}}_{n,p}@$\Delta^{\mathfrak{e}}_{n,p}$}
\end{proposition}
\begin{proof}
One checks that $\Delta_{{\mathcal{A}}^\wedge_0}$ is graded for the degree ${\mathrm{deg}}$ on ${\mathcal{A}}^\wedge_0$ (see similar argument in the proof of  Proposition~\ref{r:2022-06-15inducedalgebramorphismonmathfrakd}), which implies 
the  inclusion $\Delta_{{\mathcal{A}}^\wedge_0}({\mathcal{A}}^\wedge_0[>n]) \subset {\mathcal{A}}^\wedge_0[>p]\otimes {\mathcal{A}}^\wedge_0 \ + \ {\mathcal{A}}^\wedge_0\otimes {\mathcal{A}}^\wedge_0[>n-p]$ and therefore the stated result.
\end{proof}

The algebra morphism $\mathcal{A}_0^\wedge \to \mathring{\mathcal A}_0^\wedge$ (given by the inclusion) maps $\mathcal{A}_0^\wedge[>n]$ to $\mathring{\mathcal A}_0^\wedge[>n]$  for any integer  $n \geq 0$, therefore it induces an algebra morphism  $\mathfrak{e}(n) \to \mathfrak{d}(n)$.

\begin{lemma}[{{\cite[Lemma 3.3]{LMO98}}{\cite[Lemma 10.6]{Ohts}}}]\label{r2022-04-12-ohtlemma} For every integer  $n\geq 1$ the algebra morphism $\mathfrak{e}(n) \to \mathfrak{d}(n)$ induced by the inclusion $\mathcal{A}_0^\wedge[>n] \subset \mathring{\mathcal A}_0^\wedge[>n]$ is an isomorphism.
\end{lemma}

\begin{proof}
The induced map $\mathcal{A}_0^\wedge \to \mathring{\mathcal{A}}_0^\wedge / O^n$ is an isomorphism of graded algebras therefore the algebra morphism 
$\mathfrak{e}(n)=\mathcal{A}_0^\wedge / \mathcal{A}_0^\wedge[>n] \to \mathring{\mathcal A}_0^\wedge / (O^n+\mathring{\mathcal A}_0^\wedge[>n])$ is an isomorphism. By~\cite[Lemma 3.3]{LMO98}, $P^{n+1}_0 \subset O^n + \mathring{\mathcal A}_0^\wedge[>n]$, hence the target space of this isomorphism is $\mathfrak{d}(n)$, which finishes the proof.
\end{proof}

\begin{proposition}\label{r:2023:12-06r1} For any integers $n,p\geq 1$ with $n-p\geq 1$, the  diagram 

\begin{equation*}
\xymatrix{ \mathfrak{d}(n) \ar[d]_{\Delta^{\mathfrak{d}}_{n,p}} & & \mathfrak{e}(n) \ar[d]^{\Delta^{\mathfrak{e}}_{n,p}}\ar[ll]^-{} \\ 
\mathfrak{d}(p)\otimes \mathfrak{d}(n-p)& & \mathfrak{e}(p)\otimes \mathfrak{e}(n-p) \ar[ll]^-{} }
\end{equation*}
is commutative, where the horizontal maps are given by the isomorphisms from Lemma~\ref{r2022-04-12-ohtlemma}.
\end{proposition}

\begin{proof}The diagram
\begin{equation*}
\xymatrix{ \mathring{\mathcal{A}}^\wedge_0 \ar[d]_{\Delta_{\mathring{\mathcal{A}}^\wedge_0}} & &{\mathcal{A}}^\wedge_0 \ar[d]^{\Delta_{{\mathcal{A}}^\wedge_0}}\ar[ll]^-{} \\ 
\mathring{\mathcal{A}}^\wedge_0\otimes \mathring{\mathcal{A}}^\wedge_0  & & {\mathcal{A}}^\wedge_0 \otimes {\mathcal{A}}^\wedge_0  \ar[ll]^-{},}
\end{equation*}
where the horizontal maps are inclusions, is commutative, which implies the stated result.
\end{proof}

\begin{definition} For any integer $n\geq 1$ and any $Y \in 3\text{-}\mathrm{Mfds}$ let $\Omega_n^{\mathfrak{e}}(Y)$  \index[notation]{\Omega_n^{\mathfrak{e}}@$\Omega_n^{\mathfrak{e}}$} denote the image of $\Omega^{\mathfrak{d}}_n(Y)$ in $\mathfrak{e}(n)$ under the isomorphism $\mathfrak{d}(n)\simeq \mathfrak{e}(n)$ from Lemma~\ref{r2022-04-12-ohtlemma}. This defines a map  $\Omega_n^{\mathfrak{e}}:3\text{-}\mathrm{Mfds}\to \mathfrak{e}(n)$.
\end{definition}

\begin{corollary}\label{r:2023-12-05r1}  For any $Y \in 3\text{-}\mathrm{Mfds}$ and  integers  $n,p\geq 1$ such that $n-p \geq 1$, one has  $\Omega^{\mathfrak{e}}_p(Y) \otimes \Omega^{\mathfrak{e}}_{n-p}(Y)=\Delta^{\mathfrak{e}}_{n,p}(\Omega^{\mathfrak{e}}_n(Y))$ (equality in $\mathfrak{e}(p) \otimes \mathfrak{e}(n-p)$).

\end{corollary}

\begin{proof}
The result follows from the definition of $\Omega^{\mathfrak{e}}_i$ and Propositions~\ref{r:2023:12-06r1} and \ref{r:2023-12-06r2}$(b)$. 
\end{proof}

 Recall that $\mathcal A^\wedge_0$ is a graded algebra with $\mathcal A^\wedge_0[0]=\mathbb{C}$ (generated by the class of the empty Jacobi diagram). Therefore the same holds for each algebra $\mathfrak{e}(p)$ for  $p \geq 0$.

\begin{definition} $(a)$ We denote by $\varepsilon_p : \mathfrak{e}(p) \to \mathbb{C}$  \index[notation]{\varepsilon_p@$\varepsilon_p$} the projection on the degree 0 component, which is an algebra character.  

$(b)$ For every integer $n\geq 1$, let $\mathrm{pr}_{n,n-1}:\mathfrak{e}(n) \to \mathfrak{e}(n-1)$ \index[notation]{pr_{n,n-1}@$\mathrm{pr}_{n,n-1}$} be the projection induced by the inclusion $\mathcal A^\wedge_0[<n-1] \subset \mathcal A^\wedge_0[<n]$. 
\end{definition}

\begin{lemma}\label{r:2023-12-05r3} For any integer $n\geq 1$, the  diagram 
\begin{equation*}
\xymatrix{ \mathfrak{e}(n)\ar[drr]_-{\mathrm{pr}_{n,n-1}} \ar[rr]^-{\Delta^{\mathfrak{e}}_{n,1}} & & \mathfrak{e}(1)\otimes \mathfrak{e}(n-1) \ar[d]^{\varepsilon_1\otimes\mathrm{Id}_{\mathfrak{e}(n-1)}} \\ 
 & & \mathfrak{e}(n-1)}
\end{equation*}
is commutative.
\end{lemma}

\begin{proof} 
Let $\varepsilon_{\mathcal A_0^\wedge} : \mathcal A^{\wedge}_0 \to \mathbb{C}$ be the projection of $\mathcal A^{\wedge}_0$ on its degree~$0$ component. One easily checks that $ (\varepsilon_{\mathcal A_0^\wedge} \otimes \mathrm{Id}_{\mathcal A_0^\wedge}) \circ \Delta_{\mathcal A_0^\wedge}= \mathrm{Id}_{\mathcal A_0^\wedge}$, which, by projection, implies the stated commutativity.

\end{proof}

\begin{corollary}\label{r:2023-02-05r4old} For any $Y \in 3\text{-}\mathrm{Mfds}$ and  integer  $n\geq 2$, we have
$$\varepsilon_1(\Omega^{\mathfrak{e}}_1(Y))\Omega^{\mathfrak{e}}_{n-1}(Y) = \mathrm{pr}_{n,n-1}(\Omega^{\mathfrak{e}}_n(Y)).$$
\end{corollary}

\begin{proof}
The result follows from Corollary~\ref{r:2023-12-05r1} (with $p=1$) and Lemma~\ref{r:2023-12-05r3}.
\end{proof}

For any $Y \in 3\text{-}\mathrm{Mfds}$, the value $\varepsilon_1(\Omega^{\mathfrak{e}}_1(Y))$ is given  explicitly as follows (see \cite[Lemma 10.11]{Ohts}:
\begin{equation}\label{eq:epsilon1}
\varepsilon_1(\Omega^{\mathfrak{e}}_1(Y)) = \begin{cases}
|H_{1}(Y;\mathbb{Z})| & \text{if } Y \text{ is a } \mathbb{Q}\text{-homology sphere},\\
0 & \text{otherwise}.
\end{cases}
\end{equation}

\begin{corollary}\label{r:2023-02-05r4} Let $Y \in 3\text{-}\mathrm{Mfds}$ and  $n\geq 1$ be an integer. For any positive integer $k$, let $\Omega^{\mathfrak{e}}_n(Y)[k]\in\mathfrak{e}(n)[k]$ denote the degree $k$ part of $\Omega^{\mathfrak{e}}_n(Y)\in\mathfrak{e}(n)$, then for $1\leq k<n$ we have
$$\Omega^{\mathfrak{e}}_n(Y)[k] = \big(\varepsilon_1(\Omega^{\mathfrak{e}}_1(Y))\big)^{n-k}\Omega^{\mathfrak{e}}_{k}(Y)[k].$$
\end{corollary}

\begin{proof}
Corollary~\ref{r:2023-02-05r4old} implies $\Omega^{\mathfrak{e}}_n(Y)[k] = \big(\varepsilon_1(\Omega^{\mathfrak{e}}_1(Y))\big)\Omega^{\mathfrak{e}}_{n-1}(Y)[k]$, then the result follows recursively.
\end{proof}

\subsection{The LMO invariant \texorpdfstring{$Z^{\mathrm{LMO}}$}{ZLMO} }\label{sec:5:6}

Corollary~\ref{r:2023-02-05r4} implies that for any $Y \in 3\text{-}\mathrm{Mfds}$  all the information given by the collection  $\{\Omega^{\mathfrak{e}}_n(Y)\}_{n\geq 1}$  is concentrated in the collection of degree $n$ parts 
$\{\Omega^{\mathfrak{e}}_n(Y)[n]\}_{n\geq 1}$.

\begin{definition}   For any $Y \in 3\text{-}\mathrm{Mfds}$ the \emph{Le-Murakami-Ohtsuki invariant} (or \emph{LMO invariant}) $Z^{\mathrm{LMO}}(Y)$ of $Y$  is given by 
$$Z^{\mathrm{LMO}}(Y) = 1 + \sum_{n\geq 1}  \Omega^{\mathfrak{e}}_n(Y)[n]\in {\mathcal A}_0^\wedge.$$ \index[notation]{Z^{\mathrm{LMO}}@$Z^{\mathrm{LMO}}$}
\end{definition}

By virtue of \eqref{eq:epsilon1} and Corollary~\ref{r:2023-02-05r4}, we can consider the following variant of the LMO invariant for rational homology spheres. 
\begin{definition}\label{LMOforQHS} Let $\mathbb{Q}$-$\mathrm{HS}$ denote the set of homeomorphism classes of  rational homology 3-spheres. Define the map $\tilde{Z}^{\mathrm{LMO}}:\mathbb{Q}\text{-}\mathrm{HS}\to \mathcal{A}^{\wedge}_0$ by 
\begin{equation}
\tilde{Z}^{\mathrm{LMO}}(Y) = 1 + \sum_{n\geq 1}  \frac{\Omega^{\mathfrak{e}}_n(Y)[n]}{|H_1(Y;\mathbb{Z})|^n}\in {\mathcal A}_0^\wedge.
\end{equation} \index[notation]{Z^{\mathrm{LMO}tilde}@$\tilde{Z}^{\mathrm{LMO}}$}
for any $Y\in\mathbb{Q}$-$\mathrm{HS}$.
\end{definition}

It follows from the definition, \eqref{eq:epsilon1} and Corollary~\ref{r:2023-02-05r4} that for any positive integer $n$ and any $Y\in\mathbb{Q}$-$\mathrm{HS}$, we have
\begin{equation}\label{deglegntildeZLMO}
\tilde{Z}^{\mathrm{LMO}}(Y)[\leq n] = |H_1(Y;\mathbb{Z})|^{-n}\Omega^{\mathfrak{e}}_n(Y),
\end{equation} \index[notation]{Z^{\mathrm{LMO}}(Y)[\leq n]@$\tilde{Z}^{\mathrm{LMO}}(Y)[\leq n]$} where $\tilde{Z}^{\mathrm{LMO}}(Y)[\leq n]$ denotes the image of $\tilde{Z}^{\mathrm{LMO}}(Y)$ under the projection $\mathcal{A}^\wedge_0\to \frac{\mathcal{A}^\wedge_0}{(\mathrm{deg}> n)} =\mathfrak{e}(n)$.

\begin{proposition} For any $Y \in 3\text{-}\mathrm{Mfds}$ we have
$$\Delta_{\mathring{\mathcal{A}}_0^{\wedge}}\big(Z^{\mathrm{LMO}}(Y)\big) = Z^{\mathrm{LMO}}(Y) \otimes Z^{\mathrm{LMO}}(Y) \in {\mathcal{A}}_0^{\wedge}\otimes {\mathcal{A}}_0^{\wedge}.$$
\end{proposition}

\begin{proof}
For any $k \geq 0$, there is an isomorphism  $\mathfrak{e}(k)[k] \to \mathcal{A}^{\wedge}_0[k]$. Moreover, the coproduct $\Delta^{\mathfrak{e}}_{n,p} : \mathfrak{e}(n) \to \mathfrak{e}(p) \otimes \mathfrak{e}(n-p)$ induces a  map $\mathfrak{e}(n)[n]\to \mathfrak{e}(p)[p] \otimes \mathfrak{e}(n-p)[n-p]$, which fits in the commutative  diagram

\begin{equation*}
\xymatrix{ \mathfrak{e}(n)[n]\ar[d]_-{\simeq} \ar[rrr]^-{} & && \mathfrak{e}(n)[p]\otimes \mathfrak{e}(n-p)[n-p] \ar[d]^{\simeq} \\ 
 \mathcal{A}^{\wedge}_0[n]\ar[rrr]^-{(\mathrm{pr}^{\mathcal{A}^{\wedge}_0}_p \otimes \mathrm{pr}^{\mathcal{A}^{\wedge}_0}_{n-p})\circ\Delta_{\mathcal{A}^{\wedge}_0}}& & & \mathcal{A}^{\wedge}_0[p]\otimes \mathcal{A}^{\wedge}_0[n-p]}
\end{equation*}
where $\mathrm{pr}^{\mathcal{A}^{\wedge}_0}_i: \mathcal{A}^{\wedge}_0 \to \mathcal{A}^{\wedge}_0[i] $ denotes the projection on the degree $i$ component. By Corollary~\ref{r:2023-12-05r1}, $\Delta^{\mathfrak{e}}_{n,p}$ takes $\Omega^{\mathfrak{e}}_n(Y)[n]$ to $\Omega^{\mathfrak{e}}_p(Y)[p] \otimes \Omega^{\mathfrak{e}}_{n-p}(Y)[n-p]$, therefore $(\mathrm{pr}^{\mathcal{A}^{\wedge}_0}_p \otimes \mathrm{pr}^{\mathcal{A}^{\wedge}_0}_{n-p}) \circ \Delta_{\mathcal A^{\wedge}_0}$ takes $\Omega_n^{\mathfrak{e}}(Y)[n] \in \mathcal{A}^{\wedge}_0[n]$ to $\Omega^{\mathfrak{e}}_p(Y)[p] \otimes \Omega^{\mathfrak{e}}_{n-p}(Y)[n-p] \in \mathcal{A}^{\wedge}_0[p] \otimes \mathcal{A}^{\wedge}_0[n-p]$. Hence, $\Delta_{\mathcal{A}^{\wedge}_0}$  takes $\Omega^{\mathfrak{e}}_n(Y)[n] \in \mathcal{A}^{\wedge}_0[n]$ to $\sum_p \Omega^{\mathfrak{e}}_p(Y)[p] \otimes \Omega^{\mathfrak{e}}_{n-p}(Y)[n-p] \in \mathcal{A}^{\wedge}_0 \otimes \mathcal{A}^{\wedge}_0$, which implies the result.  
\end{proof}

A similar result can be shown for the variant $\tilde{Z}^{\mathrm{LMO}}$.

\section{Relation with the Aarhus integral}

As stated in the introduction, there is an  alternative approach for the definition of the LMO invariant for rational homology spheres based on a diagrammatic Gaussian integration: \emph{Aarhus integral}. This approach was developed in \cite{BNGRTI, BNGRT, BNGRTIII}. The authors of these papers introduced a map  
 $\mathring{A}: \mathbb{Q}\text{-}\mathrm{HS}\to \mathcal{A}^\wedge_0$ and showed  that for every positive integer $n$ and any $Y\in\mathbb{Q}\text{-}\mathrm{HS}$ the equality $\mathring{A}(Y)[\leq n] = \tilde{Z}^{\mathrm{LMO}}(Y)[\leq n]$ holds. In this brief section we sketch this construction and this relation.  Throughout this section let   $n$ be a non-negative integer.

\subsection{The Aarhus integral}
Let  $\underline{\vec{\mathcal T}}(\emptyset,\emptyset)^{\mathrm{reg}}$  \index[notation]{T(\emptyset,\emptyset)^{\mathrm{reg}}@$\underline{\vec{\mathcal T}}(\emptyset,\emptyset)^{\mathrm{reg}}$} denote the set of \emph{regular links}, that is, isotopy classes of framed oriented links whose linking matrix  is invertible. This set decomposes as $\underline{\vec{\mathcal T}}(\emptyset,\emptyset)^{\mathrm{reg}}=\sqcup_{k\geq 0}\underline{\vec{\mathcal T}}(\emptyset,\emptyset)^{\mathrm{reg}}_k$, where $\underline{\vec{\mathcal T}}(\emptyset,\emptyset)^{\mathrm{reg}}_k$  \index[notation]{T(\emptyset,\emptyset)^{\mathrm{reg}}_k@$\underline{\vec{\mathcal T}}(\emptyset,\emptyset)^{\mathrm{reg}}_k$}  consists of regular links with $k$ connected components. Let $Y$ be a 3-manifold with surgery presentation $L$, that is $Y\cong S^3_L$, it is well known that $Y$ is a rational homology sphere if and only if $L$ is a regular link.

Let $S$ be a finite set. Define the complete vector space  $\mathcal B(S):= {\mathcal A}^{\wedge}(\vec{\varnothing},\emptyset, S)$  \index[notation]{B(S)@$\mathcal B(S)$}  similarly to $\mathring{\mathcal A}^{\wedge}(\vec{\varnothing},\emptyset, S)$ (see Definition~\ref{spacewithfreelegs}) but without allowing the dashed loop. Then $\mathcal B(S)$ is a subalgebra  of the commutative $\mathring{\mathcal A}^{\wedge}(\vec{\varnothing},\emptyset, S)$ (with disjoint union as product). Moreover, there is an algebra isomomorphism  $\mathcal B(S)[X]\xrightarrow{\cong}\mathring{\mathcal A}^{\wedge}(\vec{\varnothing},\emptyset, S)$ such that it restricts to the inclusion $\mathcal B(S)\subset \mathring{\mathcal A}^{\wedge}(\vec{\varnothing},\emptyset, S)$ and takes $X$ to the dashed loop.

An \emph{$S$-strut}  is a Jacobi diagram ${}^s\!-{}^{s'}$ without trivalent vertices and with two univalent vertices colored by elements $s,s'\in S$. The number of trivalent vertices defines a degree, denoted by $\mathrm{i\text{-}deg}$, \index[notation]{i-deg@$\mathrm{i\text{-}deg}$}  on the algebra $\mathcal B(S)$. For a non-negative integer $k$, let $\mathcal B(S)[\mathrm{i\text{-}deg}=k]$ \index[notation]{B(S)[\mathrm{i\text{-}deg}=k]@$\mathcal B(S)[\mathrm{i\text{-}deg}=k]$}  denote the $\mathrm{i\text{-}deg}=k$ part of $\mathcal B(S)$. We have that $\mathcal B(S)[\mathrm{i\text{-}deg}=0]$ is the polynomial algebra on $S$-struts. More precisely, the map $S\times S\to \mathcal B(S)[\mathrm{i\text{-}deg}=0]$ defined by $S\times S\ni(s,s')\mapsto \ {}^s\!-{}^{s'}$ is $\mathfrak{S}_2$-invariant and $\mathcal B(S)[\mathrm{i\text{-}deg}=0]= \mathbb{C}[\ {}^s\!-{}^{s'} \ |\ (s,s')\in (S\times S)/\mathfrak{S}_2]$. Denote by $\mathcal{B}(S)^{\mathrm{nostrut}}$ \index[notation]{{B}(S)^{\mathrm{nostrut}}@$\mathcal{B}(S)^{\mathrm{nostrut}}$} the subalgebra of $\mathcal{B}(S)$ consisting of elements with no $S$-struts\footnote{In the literature, usually this kind of elements are called \emph{$S$-substantial}.}. The product map $\mathcal B(S)[\mathrm{i\text{-}deg}=0]\otimes \mathcal{B}(S)^{\mathrm{nostrut}}\to \mathcal B(S)$ is an isomorphism.

Let us recall the diagrammatic version of the Poincar\'e-Birkhoff-Witt isomorphism.  Let $D$ be a Jacobi diagram on $(\vec{\varnothing},\emptyset,C\sqcup \{s\})$ where $C$ is a finite set and $s\not\in C$. Consider the arithmetic mean of all the possible ways of gluing the $s$-colored univalent vertices of $D$ to the interval $\uparrow_s$. This defines a map
\begin{equation}\label{symmap}
\chi_s:\mathcal{A}^\wedge(\vec{\varnothing},\emptyset,C\sqcup \{s\})\longrightarrow \mathcal{A}^\wedge(\uparrow_s, \emptyset,C),
\end{equation} \index[notation]{\chi_s@$\chi_s$} 
called \emph{symmetrization map}. It is not difficult to show that the map (\ref{symmap}) is well defined, but it is more laborious to show that it is a vector space isomorphism, see \cite[Theorem 8]{BN}. Let $\uparrow^{S}$ denote the oriented Brauer diagram consisting of $|S|$ copies of $\uparrow$. If $S=\{s_1,\ldots,s_l\}$, define
\begin{equation*}
\chi_S:\mathcal{B}(S)\longrightarrow \mathcal{A}^\wedge(\uparrow^{S})
\end{equation*} \index[notation]{\chi_S@$\chi_S$} 
by $\chi_S = \chi_{s_l}\circ\cdots\circ \chi_{s_1}$. The map $\chi_S$ is a vector space isomorphism called  \emph{formal PBW}. The inverse map of $\chi_S$ is denoted by  $\sigma_S: \mathcal{A}^{\wedge}(\uparrow^{S}) \to \mathcal B(S)$. \index[notation]{\sigma_S@$\sigma_S$}

\begin{proposition}[{\cite[Thm. 3]{BNGRT}}]
The isomorphism $\chi_S:\mathcal{B}(S)\longrightarrow \mathcal{A}^\wedge(\uparrow^{S})$ descends to a well defined isomorphism $\chi_S: \mathcal B^{\mathrm{links}}(S)\longrightarrow\mathcal{A}^\wedge(\vec{\varnothing},S)$, where  $\mathcal B^{\mathrm{links}}(S)$ \index[notation]{B^{\mathrm{links}}(S)@$\mathcal B^{\mathrm{links}}(S)$}  is a quotient of $ \mathcal B(S)$ by certain relations, called \emph{$S$-flavored link relations}, see \cite[\S 5.2]{BNGRT}. Its inverse is denoted by $\sigma_S:  \mathcal{A}^\wedge(\vec{\varnothing},S)\to   \mathcal B^{\mathrm{links}}(S)$.
\end{proposition}

 Denote by 
$\langle \ , \  \rangle_{S}: \mathcal B(S)\otimes \mathcal B(S)\to \mathring{\mathcal A}^\wedge_0$  \index[notation]{\langle \ , \  \rangle_{S}@$\langle \ , \  \rangle_{S}$}  the specialization of the pairing \eqref{thepairinginA}.

\begin{definition}[{\cite[Def. 2.7, Def 2.8]{BNGRT}}]
\begin{itemize}
\item[$(a)$] An element $G\in \mathcal{B}(S)$ is called \emph{non-degenerate Gaussian in $S$} if it is of the form
\begin{equation}\label{eq:gaussian}
G=P\sqcup \mathrm{exp}_{\sqcup}\left(\frac{1}{2}\sum_{x,y\in S}l_{x,y}\   {}^x\!-{}^y\right)
\end{equation}
where  $P\in \mathcal{B}(S)^{\mathrm{nostrut}}$ and $\Lambda=(l_{x,y})_{x,y\in S}$ is an invertible symmetric matrix. Let $\mathcal{B}(S)^{\mathrm{ndg}}$   \index[notation]{B(S)^{\mathrm{ndg}}@$\mathcal{B}(S)^{\mathrm{ndg}}$}  the set of non-degenerate Gaussian elements in  $\mathcal{B}(S)$.

\item[$(b)$] An element $G\in \mathcal{B}^{\mathrm{links}}(S)$ is called \emph{non-degenerate Gaussian in $S$} if it has a pre-image in $\mathcal{B}(S)$ which is a non-degenerate Gaussian in $S$. Let $\mathcal{B}^{\mathrm{links}}(S)^{\mathrm{ndg}}$  \index[notation]{B^{\mathrm{links}}(S)^{\mathrm{ndg}}@$\mathcal{B}^{\mathrm{links}}(S)^{\mathrm{ndg}}$} the set of non-degenerate Gaussian elements in  $\mathcal{B}^{\mathrm{links}}(S)$.

\item[$(c)$] Let $G\in \mathcal{B}(S)^{\mathrm{ndg}}$ expressed as in \eqref{eq:gaussian} and let $\Lambda^{-1}=(l^{x,y})_{x,y\in S}$ be the inverse matrix of $\Lambda=(l_{x,y})_{x,y\in S}$. The \emph{formal Gaussian integral of $G$} is the element in $\mathcal{A}^{\wedge}_0$ given by
$$\int^{\mathrm{FG}}G dS = \left\langle \mathrm{exp}_{\sqcup}\left(-\frac{1}{2}\sum_{x,y\in S}l^{x,y}\   {}^x\!-{}^y\right), P \right\rangle_S.$$
This defines a map
\begin{equation}
\int^{\mathrm{FG}}: \mathcal{B}(S)^{\mathrm{ndg}}\to \mathcal{A}^{\wedge}_0.
\end{equation} \index[notation]{\int^{\mathrm{FG}}@$\int^{\mathrm{FG}}$}
\end{itemize}
\end{definition}

\begin{proposition}[{\cite[Prop. 2.2]{BNL}}]  Let $G, G'\in \mathcal{B}(S)^{\mathrm{ndg}}$ such that their classes in $\mathcal B^{\mathrm{links}}(S)^{\mathrm{ndg}}$ are equal. Then 
$\int^{\mathrm{FG}}G dS = \int^{\mathrm{FG}}G' dS$. Therefore, there is a well defined map
\begin{equation}
\int^{\mathrm{FG}}: \mathcal{B}^{\mathrm{links}}(S)^{\mathrm{ndg}}\to \mathcal{A}^{\wedge}_0.
\end{equation}
\end{proposition}

It can be shown that for $L\in \underline{\vec{\mathcal T}}(\emptyset,\emptyset)^{\mathrm{reg}}$, the element $\sigma_{\pi_0(L)}\circ\mathrm{cs}^{\nu}\circ Z_{\mathcal A}(L)\in \mathcal B^{\mathrm{links}}(\pi_0(L))$ belongs to $\mathcal{B}^{\mathrm{links}}(\pi_0(L))^{\mathrm{ndg}}$, see \cite[Claim 1.10]{BNGRTI}.

\begin{theorem}[\cite{BNGRT}] 
\begin{itemize}
\item[$(a)$] The map $\mathring{A}_0:\underline{\vec{\mathcal T}}(\emptyset,\emptyset)^{\mathrm{reg}}\to \mathcal{A}^\wedge_0$ \index[notation]{\mathring{A}_0@$\mathring{A}_0$} defined by 
$$L\mapsto \int^{\mathrm{FG}}\big(\sigma_{\pi_0(L)}\circ\mathrm{cs}^{\nu}\circ Z_{\mathcal A}(L)\big)d\pi_0(L)$$
is an isotopy invariant of links which is invariant under change of orientation and the Kirby move~KII.

\item[$(b)$] The map $\mathring{A}:\underline{\vec{\mathcal T}}(\emptyset,\emptyset)^{\mathrm{reg}}\to \mathcal{A}^\wedge_0$ defined by 
$$L\mapsto \mathring{A}_0(U^+)^{-\sigma_+(L)}\mathring{A}_0(U^-)^{-\sigma_-(L)}\mathring{A}_0(L)$$
is invariant under change of orientation and Kirby moves KI and KII. Here $U^{\pm}$  denotes the unknots with $\pm1$ framing and $\sigma_{\pm}(L)$ denotes the number of $\pm$ eigenvalues of the linking matrix of~$L$.

\item[$(c)$] Let $Y$ be a  a rational homology sphere. The value $\mathring{A}(Y):=\mathring{A}(L)\in \mathcal{A}^\wedge_0$, where $L\in\underline{\vec{\mathcal T}}(\emptyset,\emptyset)^{\mathrm{reg}}$ is a surgery presentation of  $Y$, is a homeomorphism invariant, called the \emph{Aarhus integral} of $Y$.  Therefore, there is a well defined map
$$\mathring{A}:\mathbb{Q}\text{-}\mathrm{HS}\to\mathcal A^{\wedge}_0.$$
\end{itemize}

\end{theorem}

\subsubsection{Relation with the LMO invariant}

Let $Y$ be a rational homology sphere and denote by $\mathring{A}(Y)[\leq n]$  \index[notation]{\mathring{A}(Y)[\leq n]@$\mathring{A}(Y)[\leq n]$}  the image of $\mathring{A}(Y)$ under the projection $\mathcal{A}^\wedge_0\to \frac{\mathcal{A}^\wedge_0}{(\mathrm{deg}> n)} =\mathfrak{e}(n)$. The following is one of the main results of \cite{BNGRTIII} which gives the precise relation between the Aarhus integral and the LMO invariant.

\begin{theorem}[{\cite{BNGRTIII}}] We have $\mathring{A}(Y)[\leq n] = |H^1(Y;\mathbb{Z})|^{-n}\Omega_n^{\mathfrak{e}}(Y)$ and therefore
$$\mathring{A}(Y)[\leq n] = \tilde{Z}^{\mathrm{LMO}}(Y)[\leq n].$$
\end{theorem}

\printindex[notation] 
\printindex

\bibliographystyle{plain} 
\bibliography{bibliography_survey} 

\end{document}